\documentclass{amsbook}
\usepackage{graphicx}
\usepackage[english]{babel}
\usepackage[utf8x]{inputenc}% pour les accents
\usepackage{stmaryrd}
\usepackage{hyperref}

\newtheorem{theo}{Theorem}[chapter]
\newtheorem{lem}[theo]{Lemma}
\newtheorem{prop}[theo]{Proposition}
\newtheorem{cor}[theo]{Corollary}

\theoremstyle{definition}
\newtheorem{defi}[theo]{Definition}
\newtheorem{ex}[theo]{Example}

\theoremstyle{remark}
\newtheorem{rem}[theo]{Remark}

\numberwithin{section}{chapter}
\numberwithin{equation}{chapter}
\let\disp\displaystyle
\let\vfi\varphi
\let\eps\varepsilon

\newcommand{\IC}{\mathbb{C}}
\newcommand{\IN}{\mathbb{N}}
\newcommand{\IQ}{\mathbb{Q}}
\newcommand{\IR}{\mathbb{R}}
\newcommand{\IS}{\mathbb{S}}
\newcommand{\IZ}{\mathbb{Z}}

\newcommand{\CA}{{\mathcal A}}
\newcommand{\CB}{{\mathcal B}}
\newcommand{\CC}{{\mathcal C}}
\newcommand{\CD}{{\mathcal D}}
\newcommand{\CE}{{\mathcal E}}
\newcommand{\CI}{{\mathcal I}}
\newcommand{\CJ}{{\mathcal J}}
\newcommand{\CF}{{\mathcal F}}
\newcommand{\CN}{{\mathcal N}}
\newcommand{\CO}{{\mathcal O}}
\newcommand{\CP}{{\mathcal P}}
\newcommand{\CQ}{{\mathcal Q}}

\newcommand{\CU}{{\mathcal U}}
\newcommand{\CV}{{\mathcal V}}

\newcommand{\Int}[1]{{\rm Int}\left(#1\right)}
\newcommand{\diam}{{\rm diam}}
\newcommand{\doubleindice}[2]{\mbox{\scriptsize \begin{tabular}{c} 
$ #1 $\\ $ #2 $\end{tabular}}}
\newcommand{\LY}{{\rm LY}}
\renewcommand{\middle}{{\sf mid}}
\newcommand{\slope}{{\sf slope}}
\newcommand{\Id}{{\rm Id}}
\newcommand{\End}[1]{\partial #1}
\newcommand{\Bd}[1]{{\rm Bd}(#1)}

\let\Lbrack\llbracket  % needs stmaryrd package
\let\Rbrack\rrbracket

\newcommand{\labelarrow}[1]{\stackrel{#1}{\longrightarrow}}
\newcommand{\cover}[1]{\xrightarrow[#1]{}}

\makeindex

\begin{document}

\frontmatter

\begin{titlepage}

\begin{center}
{\LARGE{\bf

CHAOS ON THE INTERVAL

a survey of relationship between the various kinds of chaos for continuous interval maps}

\vspace{10mm}

Sylvie Ruette}

\end{center}

\vspace{20mm}
\centerline{\includegraphics{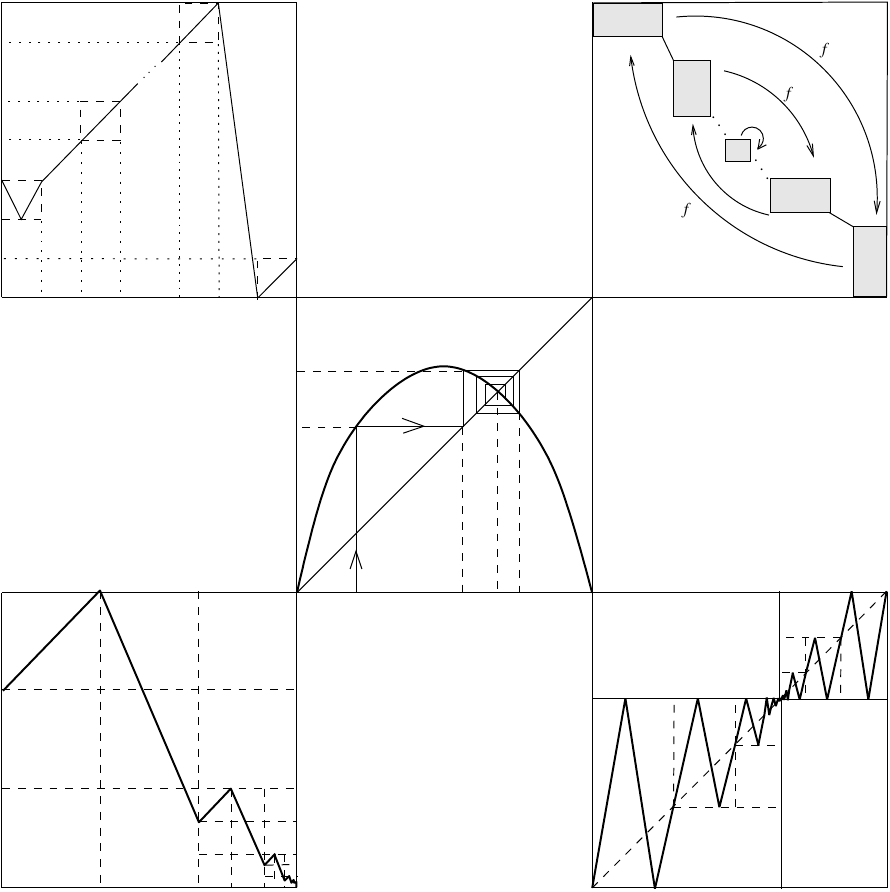}}

\end{titlepage}

\begin{titlepage}

\noindent
Author

\medskip
\noindent Sylvie Ruette\\
Laboratoire de Mathématiques d'Orsay, bâtiment 307\\
Université Paris-Sud 11,
91405 Orsay cedex, France\\
\texttt{sylvie.ruette@math.u-psud.fr}

\vspace{10mm}

\vspace{10mm}
\noindent
Last revision: April 2018
\end{titlepage}
%\pagenumbering{roman}
\tableofcontents

%**********************************************************************
%introduction
\chapter*{Introduction}

Generally speaking, a \emph{dynamical system} is a space in which the
points (which can be viewed as configurations) move along with time
according to a given rule, usually not depending on time. Time can be
either  continuous (the motion of planets, fluid mechanics, etc) or
discrete  (the number of bees each year, etc). In the discrete case,
the system is determined by a map $f\colon X\to X$, where $X$ is the
space, and the evolution is given by successive iterations of the
transformation:  starting from the point $x$ at time $0$, the point
$f(x)$  represents the new position at time $1$ and $f^n(x)=f\circ
f\circ \cdots\circ f (x)$ ($f$ iterated $n$ times) is the position at
time $n$.

A dynamical system ruled by a deterministic law can nevertheless be
unpredictable. In particular, in the early 1960's, Lorenz underlined
this phenomenon after realizing by chance that in his meteorological
model, two very close initial values may lead to totally different
results  \cite{Lor1, Lor2, Lor3}; he discovered the so called
``butterfly effect''.
This kind of behavior has also been exhibited in other dynamical
systems. One of the first to be studied, among the simplest, is given
by the map $f(x)=rx(1-x)$ acting on the interval $[0,1]$, and models
the evolution of a population. If the parameter $r$ is small enough,
then all the trajectories converge to a fixed point -- the population
stabilizes.  However, May showed that for larger values of $r$, the
dynamics may become very complicated \cite{May}.

%%%
This book focuses on dynamical systems given by the iteration of a 
continuous map on an interval.
These systems were broadly studied because they are simple but nevertheless
exhibit complex behaviors. They also allow numerical simulations using
a computer or a mere pocket  calculator, which enabled the discovery
of some chaotic phenomena.  Moreover, the ``most interesting'' part of
some higher-dimensional systems can be of lower dimension, which
allows, in some cases, to boil down to systems in dimension one. This
idea was used for instance to reduce the study of Lorenz flows in dimension 
3 to a class of maps on the interval.
However, continuous interval maps have many properties that are not
generally found in other spaces. As a consequence, the study of one-dimensional
dynamics is very rich but not representative of all systems.

In the 1960's, Sharkovsky began to study the structure of systems
given by a continuous map on an interval, in particular the
co-existence of periodic points of various periods, which is ruled by
\emph{Sharkovsky's order} \cite{Sha}. Non Russian-speaking scientists
were hardly aware of this striking result until a new proof of this
theorem was given in English by \v{S}tefan in 1976 in a preprint
\cite{Ste0} (published one year later in \cite{Ste}).  In
1975, in the paper ``Period three implies chaos'' \cite{LY}, Li and
Yorke proved that a continuous interval map with a periodic point of
period $3$ has periodic points of all periods -- which is actually a
part of Sharkovsky's Theorem eleven years earlier; they also proved
that, for such a map $f$, there exists an uncountable set such that,
if $x,y$ are two distinct points in this set, then $f^n(x)$ and
$f^n(y)$ are arbitrarily close for some $n$ and are further than some
fixed positive distance for other integers $n$ tending to infinity;
the term ``chaos'' was introduced in mathematics in this paper of Li
and Yorke, where it was used in reference to this behavior.

Afterwards, various definitions of chaos were proposed. They do not
coincide in general and none of them can be considered as the unique
``good'' definition of chaos. One may ask ``What is chaos then?''.  It
relies generally on the idea of unpredictability or instability, i.e.,
knowing the trajectory of one point is not enough to know  what
happens elsewhere.  The map $f\colon X\to X$ is said to be
\emph{sensitive to initial conditions} if near every point $x$ there
exists a point $y$ arbitrarily close to $x$ such that the distance
between $f^n(x)$ and $f^n(y)$ is greater than a given $\delta>0$ for
some $n$. \emph{Chaos  in the sense of Li-Yorke} (see above) asks for
more instability, but only on a subset.
For Devaney, chaos is seen as a mixing of unpredictability and regular
behavior: a system is \emph{chaotic in the sense of Devaney} if it is
transitive,
sensitive to initial conditions and has a dense set of periodic points
\cite{Dev}. Others put as a part of their definition that the entropy
should be positive, which means that the number of different
trajectories of length $n$, up to some approximation, grows
exponentially fast.

In order to obtain something uniform, the system is often assumed to
be transitive. Roughly speaking, this means that it cannot be
decomposed into two parts with nonempty interiors that do not
interact under the action of the transformation. This ``basic''
assumption actually has strong consequences for systems on
one-dimensional spaces.  For a continuous interval map, it implies
most of the other notions linked to chaos: sensitivity to initial
conditions, dense set of periodic points, positive entropy, chaos in
the sense of Li-Yorke, etc. This leads us to search for (partial)
converses: for instance, if the interval map $f$ is sensitive to
initial conditions then, for some integer $n$, the map $f^n$ is
transitive on a subinterval.

The study of periodic points has taken an important place in the works
on interval maps. For these systems, chaotic properties not only imply
existence of periodic points, but the possible periods also provide
some information about the system. For instance, for a transitive
interval map, there exist periodic points of all even periods, and an
interval map has positive entropy if and only if there exists a
periodic point whose period is not a power of $2$. This kind of
relationship is very typical of one-dimensional systems.

\medskip 
The aim of this book is not to collect all the results about continuous 
interval maps but to survey the relations 
between the various kinds of chaos and related notions for these systems.
The papers on this topic are numerous but very scattered in the
literature, sometimes little known or difficult to find, sometimes
originally published in Russian (or Ukrainian, or Chinese), and
sometimes without proof. Furthermore some results were  found twice
independently, which was often due to a lack of communication and
language barriers, leading research to develop separately in English
and Russian literature. This has complicated our task when
attributing authorship; we want to apologize for possible errors or
omissions when indicating who first proved the various results.

We adopt a topological point of view, i.e., we do not speak about
invariant measures or ergodic properties. Moreover, we are interested
in the set of continuous interval maps, not  in particular families
such as piecewise monotone, $C^\infty$ or unimodal maps.
We give complete proofs of the results concerning interval maps.

Many results for interval maps have been generalized to other
one-dimensional systems. We briefly describe them in paragraphs called
``Remarks on graph maps'' at the end of the concerned sections. We
indicate some main ideas and give the references. This subject is
still in evolution, and the most recent works and references may be
missing.

This book is addressed to both graduate students and researchers.  We have
tried to keep to the elementary level.  The prerequisites are
basic courses of real analysis and topology, and some linear algebra.

\section*{Contents of the book}
%\addcontentsline{toc}{section}{Contents of the book}

In the \textbf{first Chapter}, we define some elementary notions and
introduce some notation. Throughout this book, a continuous map
$f\colon I\to I$  on a non degenerate compact interval $I$ will be
called an  \emph{interval map}.  We also provide some basic results
about  $\omega$-limit sets and tools to find periodic points.

\medskip In \textbf{Chapter \ref{chap2}}, we study the links between
transitivity, topological mixing and sensitivity to initial
conditions. We first prove that a transitive
interval map has a dense set of periodic points.  Then we show that
transitivity is very close to the notion of topological mixing in the
sense that for a transitive  interval map $f\colon I\to I$, either $f$
is topologically mixing, or the interval $I$ is divided into two
subintervals $J,K$ which are swapped under the action of $f$ and such
that both $f^2|_J$ and $f^2|_K$ are topologically mixing.
Furthermore, the notions of topological mixing, topological weak
mixing and total transitivity are proved to be equivalent for interval
maps.

Next we show that a transitive interval map is sensitive to initial
conditions and, conversely, if the map is sensitive, then there exists
a subinterval $J$ such that $f^n|_J$ is transitive for some
positive integer~$n$.

\medskip \textbf{Chapter \ref{chap:periodic-points}} is devoted to periodic points.  First we
prove that topological mixing is equivalent to the specification
property, which roughly means that any collection of pieces of orbits can be
approximated by the orbit of a periodic point.

Next we show that, if the set of periodic points is dense for the interval
map $f$,  then there exists a non degenerate subinterval $J$  such
that either $f|_J$ or $f^2|_J$ is transitive provided  that $f^2$ is
not equal to the identity map.

Then we present Sharkovsky's Theorem, which says that there is a
total order on $\IN$ -- called \emph{Sharkovsky's order} -- such that,
if an interval map has a periodic point of period $n$, then it also
has periodic points of period $m$ for all integers $m$ greater than
$n$ with respect to this order. The  \emph{type} of a map $f$ is the
minimal integer $n$ for Sharkovsky's order such that $f$ has a
periodic point of period $n$; if there is no such integer $n$, then
the set of periods is exactly $\{2^n\mid n\ge 0\}$ and the type is
$2^{\infty}$.  We build an interval map of type $n$ for every $n\in
\IN\cup\{2^{\infty}\}$.

Next, we study the relation between the type of a map and the
existence of horseshoes. Finally, we compute the type of transitive
and  topologically mixing interval maps.

\medskip 
In \textbf{Chapter \ref{chap4}}, we are concerned with topological entropy. A 
\emph{horseshoe} for the interval map $f$ is a family of two or more closed  
subintervals $J_1,\ldots, J_p$ with disjoint interiors such that $f(J_i)\supset
J_1\cup\cdots\cup J_p$ for all $1\leq i\leq p$. We show that the
existence of a horseshoe implies that the topological entropy is
positive. Reciprocally, Misiurewicz's Theorem states that, if the
entropy of the interval map $f$ is positive, then $f^n$ has a
horseshoe for some positive integer $n$.

Next we show that an interval map has a homoclinic point if and only
if it has positive topological entropy. For an interval map $f$, $x$
is a \emph{homoclinic point} if there exists a periodic point $z$ different
from $x$ such that $x$ is in the unstable manifold of $z$ and $z$ is a limit
point of $(f^{np}(x))_{n\ge 0}$, where $p$ is the period of $z$.

We then give some upper and lower bounds on the entropy, focusing on
lower bounds for transitive and topologically  mixing maps and lower
bounds depending on the periods of periodic points (or, in other words,
on the type of the map for Sharkovsky's order). In particular, an
interval map has positive topological  entropy if and only if it has a
periodic point whose period  is not a power of $2$. The sharpness of
these bounds is illustrated by some examples.

To conclude this chapter, we show that a topologically mixing interval
map has a uniformly positive entropy; that is, every cover by two open
non dense sets has  positive topological entropy. Actually,   this
property is equivalent to topological mixing for interval maps.

\medskip \textbf{Chapter \ref{chap5}} is devoted to chaos in the sense of Li-Yorke.
Two points $x,y$ form a \emph{Li-Yorke pair of modulus $\delta$} for
the map $f$ if
$$
\limsup_{n\to+\infty}|f^n(x)-f^n(y)|\ge\delta\quad\text{and}\quad\liminf_{n\to+\infty}|f^n(x)-f^n(y)|=0.
$$ A \emph{$\delta$-scrambled set} is a set $S$ such that every pair
of distinct points in $S$ is a Li-Yorke pair of modulus $\delta$; the
set $S$ is \emph{scrambled} if for every $x,y\in S$, $x\ne y$,
$(x,y)$ is a Li-Yorke pair (of modulus $\delta$ for some $\delta>0$
depending on $x,y$).  The map is \emph{chaotic in the sense of
  Li-Yorke} if it has an uncountable scrambled set.  We prove that an
interval map of positive topological entropy admits a
$\delta$-scrambled Cantor set for some $\delta>0$, and is thus
chaotic in the sense of Li-Yorke. We also show that a topologically
mixing map has a  dense $\delta$-scrambled set which is a countable
union of Cantor sets.

Next, we study an equivalent condition for zero entropy interval maps
to be chaotic in the sense of Li-Yorke, which implies the existence of
a $\delta$-scrambled Cantor set as in the positive entropy case.  A
zero entropy interval map that is chaotic in the sense of Li-Yorke is
necessarily of type $2^{\infty}$ for Sharkovsky's order, but the
converse is not true; we build two maps of type $2^{\infty}$  having
an infinite $\omega$-limit set, one being chaotic in the sense of
Li-Yorke and the other not.

Then we state that the existence of one Li-Yorke pair for an interval
map is enough to imply chaos in the sense of Li-Yorke.

Finally, we show that an interval map is chaotic in the sense of
Li-Yorke if and only if it has positive topological sequence entropy.

\medskip In \textbf{Chapter \ref{chap6}}, we study some notions related to
Li-Yorke pairs.

\emph{Generic chaos} and \emph{dense chaos} are somehow
two-dimensional notions. A topological system $f\colon X\to X$ is
generically (resp. densely) chaotic if the set of Li-Yorke pairs  is
residual (resp. dense) in $X\times X$.  A transitive
interval map is generically chaotic; conversely, a  generically
chaotic interval map has exactly one or two transitive
subintervals. Dense chaos is strictly weaker than generic chaos: a
densely interval map may have no transitive subinterval, as
illustrated by an example. We show that,  if $f$ is a densely chaotic
interval map, then $f^2$ has a horseshoe, which implies that $f$ has a
periodic point  of period $6$ and the topological entropy of $f$ is at
least $\frac{\log 2}2$.

\emph{Distributional chaos} is based on a refinement of the conditions
defining Li-Yorke pairs. We show that, for interval maps,
distributional chaos is equivalent to positive topological entropy.

\medskip In \textbf{Chapter \ref{chap7}}, we focus on the existence of
some kinds of chaotic subsystems and we relate them to the previous
notions.

A system is said to be \emph{chaotic in the sense of Devaney} if it is
transitive, sensitive to initial conditions and has a dense set of
periodic points. For an interval map,  the existence of an invariant closed
subset that is chaotic in the sense of Devaney is equivalent
to positive topological entropy. We also show that an interval map has an
invariant uncountable closed subset $X$ on which $f^n$ is
topologically mixing for some $n\ge 1$ if and only if $f$ has positive
topological entropy.

Finally, we study the existence of an invariant closed subset on which
the map is transitive and sensitive to initial conditions.  We show
that this property is implied by  positive topological entropy and implies chaos
in the sense of Li-Yorke.  However these notions are distinct: there
exist zero entropy interval maps with a transitive sensitive subsystem
and interval maps with no transitive sensitive subsystem that are
chaotic in the sense of Li-Yorke.

\medskip The \textbf{last chapter} is an appendix
that recalls succinctly some background in topology.

\medskip
The relations between the main notions studied in this book are
summarized by the diagram in Figure~\ref{fig:diagram}.

\begin{figure}
\centerline{\includegraphics{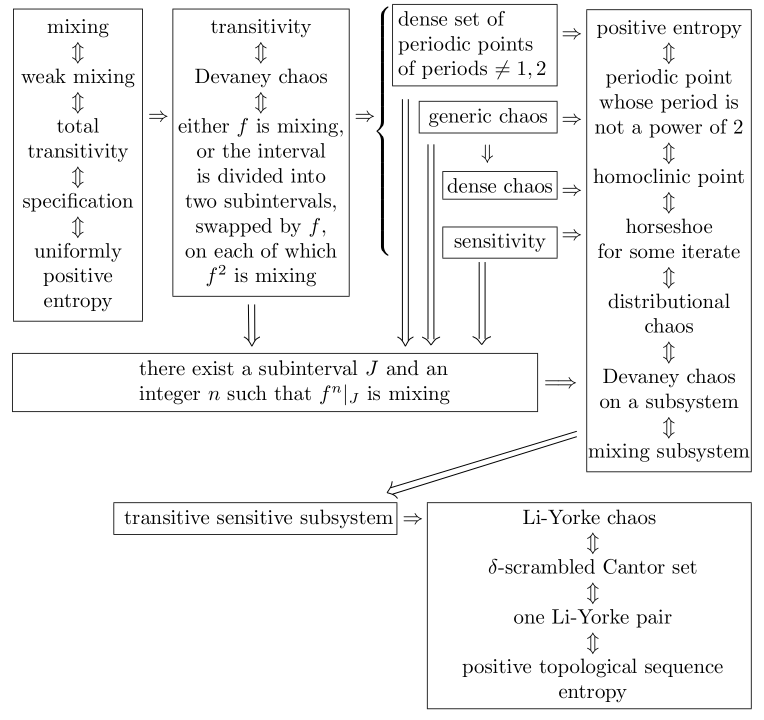}}
\caption{Diagram summarizing the relations between the main notions 
related to chaos for an interval map $f$. 
}
\label{fig:diagram}
\end{figure}

\bigskip
I thank all the people who have contributed to improve this book by
remarks or translations:  Jozef Bobok, Víctor Jiménez López, Sergiy
Kolyada,  Jian Li, Micha{\l} Misiurewicz, T. K. Subrahmonian Moothathu,
L$\!$'ubomír Snoha,  Émilie Tyberghein, Zheng Wei, and Dawei Yang.
I particularly thank Roman Hric who helped me to fill a gap in a proof.

I want to thank CNRS (Centre National de la Recherche  Scientifique)
for giving me time (two sabbatical semesters) to write this book,
which allowed me to finish this long-standing project.

\aufm{Sylvie Ruette}

\mainmatter

%**********************************************************************
%Notations and basic tools
\chapter{Notation and basic tools}\label{chap1}

\section{General notation}

\subsection{Sets of numbers}

The set of natural numbers (that is, positive integers) is denoted by
$\IN$. The symbols $\IZ$, $\IQ$, $\IR$, $\IC$ denote respectively the
set of all integers, rational numbers, real numbers and complex numbers.
The non negative integers and non negative real numbers are denoted
respectively by $\IZ^+$ and $\IR^+$.\index{N@$\IN$}\index{Q@$\IQ$}\index{R@$\IR$}\index{C@$\IC$}
\label{notation:IN}
\label{notation:IQ}
\label{notation:IZ}
\label{notation:IR}
\label{notation:IC}

\subsection{Interval of integers}

The notation $\Lbrack n,m\Rbrack$\index{interval of integers} denotes an interval of integers, that~is,
\label{notation:intervalintegers}
$
\Lbrack n,m\Rbrack:=\{k\in\IZ\mid n\le k\le m\}.
$

\medskip
We shall often deal with sets $X_1,\ldots, X_n$ that are
cyclically permuted. The notation 
$X_{i+1\bmod n}$\index{mod n@$\bmod\ n$} 
\label{notation:integermodn}
means $X_{i+1}$ if
$i\in\Lbrack 1,n-1\Rbrack$ and $X_1$ if $i=n$. 
More generally, if the set of indices $\CI$ under consideration is 
$\Lbrack 1,n\Rbrack$ (resp. $\Lbrack 0,n-1\Rbrack$), then $i\bmod n$
denotes the integer $j\in\CI$ such that $j\equiv i\bmod n$.
%This notation will also be used in a similar way when the set of indices is $\Lbrack 0,n-1\Rbrack$.

\subsection{Cardinality of a set}

If $E$ is a finite set, $\#E$ denotes its cardinality, that is, the number
of elements in $E$.\index{cardinal of a set}

\label{notation:cardE}

A set is \emph{countable}\index{countable set} if it can be written as 
$\{x_n\mid n\in\IN\}$. A finite set is countable.

\subsection{Notation of topology}

The definitions of the topological notions used in this book are recalled in
the appendix. %, Chapter~\ref{appendix}. 
Here we only give some notation.

\medskip

Let $X$ be a metric space and let $Y$ be a subset of $X$. 
Then $\overline{Y}$,
$\Int{Y}$, $\Bd{Y}$ denote respectively the closure, the interior and
the boundary of $Y$.\index{closure of a set}\index{Int@$\Int{Y}$}\index{interior of a set}\index{boundary of a set}\index{Bd(Y)@$\Bd{Y}$}
\label{notation:closure1}
\label{notation:interior1}
\label{notation:boundary1}

\begin{rem}
When talking about topological notions (neighborhood, interior,
etc), we always refer to the induced topology on the ambient space $X$.  For
instance, in Example~\ref{ex:orbit} below, $[0,1/2)$ is an open set since
the ambient space is $[0,1]$.
\end{rem}

The distance on a metric space $X$ is denoted by $d$.\index{d(x,y)@$d(x,y)$}
\label{notation:distance1}
If $x\in X$ and $r>0$, the open ball of center
$x$ and radius $r$ is $B(x,r):=\{y\in X\mid d(x,y)<r\}$\index{B(x,r)@$B(x,r),\overline{B}(x,r)$}\index{open ball}, and the closed ball of center
$x$ and radius $r$ is $\overline{B}(x,r):=\{y\in X\mid d(x,y)\le r\}$\index{closed ball}.
\label{notation:ball1}

The diameter of a set $Y\subset X$ is
$\diam(Y):=\sup\{d(y,y')\mid y,y'\in Y\}$.\index{diam(Y)@$\diam(Y)$}\index{diameter of a set}
\label{notation:diam1}
If $Y$ is compact, then the supremum is reached.

\subsection{Restriction of a map}

Let $f\colon X\to X'$ be a map and $Y\subset X$. The restriction of $f$ to
$Y$, denoted by $f|_Y$\index{restriction of a map}, is the map $f|_Y\colon
\begin{array}[t]{ccc}Y&\to&X'\\x&\mapsto&f(x)\end{array}$.
\label{notation:restriction}

\section{Topological dynamical systems, orbits, $\omega$-limit sets}

Our purpose is to study dynamical systems on intervals. However we
prefer to give the notation in a broader context because most of the
definitions have a meaning for any dynamical system, and a
few properties will not be specific to the interval case.

\subsection{Topological dynamical systems, invariant set}
A \emph{topological dynamical system}\index{topological dynamical system}\index{dynamical system} $(X,f)$ is given by a continuous map $f\colon X\to X$,
where $X$ is a nonempty compact metric space.
The evolution of the system is given by the successive iterations of the
map. If $n\in\IN$, the $n$-th iterate of $f$ is denoted
by $f^n$, that is,
$$
f^n:=\underbrace{f\circ f\circ\cdots\circ f}_{n \text{ \scriptsize
times}}.
$$
By convention, $f^0$ is the identity map on $X$. We can think of $n$ as
time: starting from an initial position $x$ at time $0$,
the point $f^n(x)$ represents the new position at time $n$.

\begin{ex}\label{ex:orbit}
Let $f\colon [0,1]\to [0,1]$ be the map defined by $f(x)=3x(1-x)$.
The successive iterates of $x$ can be plotted on the graph of 
$f$, as illustrated in Figure~\ref{fig:orbit}; the diagonal $y=x$ is utilized to
re-use the result  of an iteration.
\begin{figure}[htb]
\centerline{\includegraphics{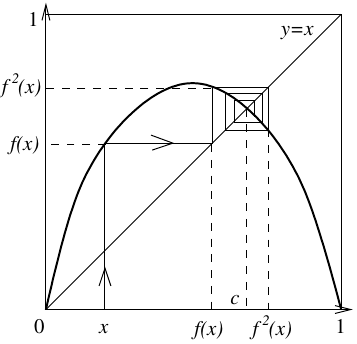}}
\caption{The first iterates of $x$ plotted on the graph of $f$.}
\label{fig:orbit}
\end{figure}
\end{ex}

Let $(X,f)$ be a topological dynamical system. An 
\emph{invariant}\index{invariant set} (or \emph{$f$-invariant}) set
is a  nonempty closed set $Y\subset X$ such that $f(Y)\subset Y$; it
is \emph{strongly   invariant}\index{strongly invariant set} if in addition
$f(Y)=Y$. If $Y$ is an
invariant set, let $f|_Y$ denote the map $f$  restricted to $Y$ and
\emph{arriving in $Y$}, that is, $f|_Y\colon Y\to Y$\index{restriction of a map}.
\label{notation:restrictionbis}
With this slight abuse of notation, $(Y, f|_Y)$ is a topological dynamical
system, called a \emph{subsystem}\index{subsystem} of $(X,f)$,
and we shall speak of the properties of $f|_Y$ (e.g., ``$f|_Y$
is transitive'').

\subsection{Trajectory, orbit, periodic point}

In the literature, the words \emph{trajectory} and \emph{orbit} often have 
the same meaning. However we prefer to follow the
terminology of Block-Coppel \cite{BCop} because it is convenient to make
a distinction between two notions. In this book,  
when $(X,f)$ is a topological dynamical system and $x$ is a point in $X$,
the \emph{trajectory}\index{trajectory} of $x$ is the infinite
sequence $(f^n(x))_{n\ge 0}$ (there may be repetitions in the sequence) and
the \emph{orbit}\index{orbit} of $x$ is the set
$\CO_f(x):=\{f^n(x)\mid n\ge 0\}$\index{Of(x)@$\CO_f(x), \CO_f(E)$}.  Similarly,
if $E$ is a subset of $X$, then $\CO_f(E):=\bigcup_{n\ge 0}f^n(E)$.
\label{notation:orbit}

A point $x$ is \emph{periodic}\index{periodic point} (for the map $f$) if there exists a positive
integer
$n$ such that $f^n(x)=x$. The \emph{period}\index{period of a periodic point/obit} of $x$
is the least positive integer $p$ such that $f^p(x)=x$. It is easy to
see that, if $x$ is periodic of period $p$ and $n\in\IN$, then $f^n(x)=x$ if and
only if $n$ is a multiple of $p$; moreover $\CO_f(x)$ is a finite set of
$p$ distinct points: $\CO_f(x)=\{x,f(x),\ldots, f^{p-1}(x)\}$.
If $x$ is a periodic point of period
$p$, then its orbit  is called a \emph{periodic orbit of period
$p$}\index{periodic orbit}. Each point of a periodic orbit is a
periodic point with the same period and the same orbit.  
If $f(x)=x$, then $x$ is called a
\emph{fixed point}\index{fixed point}.
Let \index{Pn(f)@$P_n(f)$}
$$
P_n(f):=\{x\in X\mid f^n(x)=x\};
$$
\label{notation:Pn}
this is the set of periodic points whose periods divide $n$.

A point $x$ is \emph{eventually periodic}\index{eventually periodic point} 
if there exists an  integer $n\ge 0$ such that $f^n(x)$ is periodic.

\subsection{Omega-limit set}
Let $(X,f)$ be a topological dynamical system.
The \emph{$\omega$-limit set}\index{$\alpha$ z-lp@$\omega$-limit set of a point}\index{$\alpha$ zxf@$\omega(x,f)$} 
\label{notation:omegax}
of a point $x\in X$, denoted by $\omega(x,f)$, is the set of all limit points of the trajectory of $x$, 
that is,
$$
\omega(x,f):=\bigcap_{n\ge 0}\overline{\{f^k(x)\mid k\ge n\}}.
$$
The \emph{$\omega$-limit set of the map $f$}\index{$\alpha$ z-lm@$\omega$-limit set of a map}\index{$\alpha$ zf@$\omega(f)$} is 
\label{notation:omegaf}
$$
\omega(f):=\bigcup_{x\in X}\omega(x,f).
$$

\begin{lem}\label{lem:omega-set}
Let $(X,f)$ be a topological dynamical system, $x\in X$ and $n\ge 1$. Then
\begin{enumerate}
\item $\omega(x,f)$ is a closed set, and it is strongly invariant,
\item $\omega(f^n(x),f)=\omega(x,f)$,
\item $\forall i\ge 0$, $\omega(f^i(x),f^n)=f^i(\omega(x,f^n))$,
\item $\disp\omega(x,f)=\bigcup_{i=0}^{n-1}\omega(f^i(x),f^n)$,
\item if $\omega(x,f)$ is infinite, then
$\omega(f^i(x),f^n)$ is infinite for all $i\ge 0$, 
\item $f(\omega(f))=\omega(f)$,
\item $\omega(f^n)=\omega(f)$.
\end{enumerate}
\end{lem}

\begin{proof}
Assertions (i) to (iv) can be easily deduced from the definition.
Assertion (vi) follows from (i), assertions (v) and (vii)
follow from (iii)-(iv).
\end{proof}

\begin{lem}\label{lem:omega-finite}
Let $(X,f)$ be a topological dynamical system and $x\in X$. 
If $\omega(x,f)$ is finite, then it is a periodic orbit.
\end{lem}

\begin{proof}
Let $F$ be a nonempty subset of $\omega(x,f)$ different from $\omega(x,f)$. 
We set
$F':=\omega(x,f) \setminus F$. Both $F,F'$ are finite and nonempty.
Let $U,U'$ be two open sets such that $F\subset U$, $F'\subset U'$,
$\overline{U}\cap F'=\emptyset$ and  $\overline{U'}\cap F=\emptyset$. 
Thus, for every large enough integer $n$, the point $f^n(x)$ belongs to 
$U\cup U'$. Moreover, there
are infinitely many integers $n$ such that $f^n(x)\in U$ and infinitely many
$n$ such that $f^n(x)\in U'$.  
Therefore, there exists an increasing sequence
$(n_i)_{i\ge 0}$ such that, $\forall i\ge 0$, $f^{n_i}(x)\in U$ and
$f^{n_i+1}(x)\in U'$. By compactness, the sequence $(f^{n_i}(x))_{i\ge 0}$ 
has a limit point  $y\in \overline{U}\cap \omega(x,f)=F$. Since $f$ is 
continuous, $f(y)$ is a limit point of $(f^{n_i+1}(x))_{i\ge 0}$, and
hence $f(y)\in \omega(x,f)\cap \overline{U'}=F'$.  Thus
$f(F)\cap F'\neq \emptyset$, and so $\omega(x,f)$ contains no
invariant subset except itself. 
This implies that $f$  acts as a cyclic permutation on
$\omega(x,f)$, that is, $\omega(x,f)$ is a periodic orbit.
\end{proof}

\subsection{Semi-conjugacy, conjugacy}
Let $(X,f)$ and $(Y,g)$ be two topological dynamical systems. The system 
$(Y,g)$ is said to be 
\emph{(topologically) semi-conjugate}\index{semi-conjugacy} 
to $(X,f)$ if there exists a continuous onto
map $\vfi\colon X\to Y$  such that $\vfi\circ f=g\circ \vfi$.
If in addition the map $\vfi$ is a homeomorphism, 
$(Y,g)$ is \emph{(topologically) conjugate}\index{conjugacy}\index{topological 
conjugacy} to $(X,f)$; conjugacy is an equivalence relation.
Two conjugate dynamical systems share the same dynamical properties as long as 
topology is concerned (differential properties may not be preserved if
$\vfi$ is only assumed to be continuous).

%*********************
\section{Intervals, interval maps}

\subsection{Intervals, endpoints, length, non degenerate interval, inequalities between subsets of $\IR$}

The (real) intervals are exactly the connected sets of $\IR$. An interval $J$
is either the empty set or one of the following forms:
\begin{itemize}
\item $J=[a,b]$ with $a,b\in\IR$, $a\le b$ (if $a=b$, then $J=\{a\}$),
\item $J=(a,b)$ with $a\in\IR\cup\{-\infty\},b\in\IR\cup\{+\infty\}$, $a<b$,
\item $J=[a,b)$ with $a\in\IR,b\in\IR\cup\{+\infty\}$, $a<b$,
\item $J=(a,b]$ with $a\in\IR\cup\{-\infty\},b\in\IR$, $a<b$.
\end{itemize}
Suppose that $J$ is nonempty and bounded  (i.e., when
$a,b\in\IR$). The \emph{endpoints}\index{endpoints of an interval}  of
$J$ are $a$ and $b$; let $\End{J}$ denote the set $\{a,b\}$.
\label{notation:endpointsinterval}
The  \emph{length} of $J$, denoted by  $|J|$\index{length of an interval}, is equal to $b-a$.
\label{notation:lengthJ}

An interval is \emph{degenerate}\index{degenerate interval} 
if it is either empty or reduced to a single point, and it is
\emph{non degenerate}\index{non degenerate interval} otherwise.

If $a,b\in\IR$, let 
$\langle a,b\rangle$
denote the smallest interval containing $\{a,b\}$, that is,
$\langle a,b\rangle=[a,b]$ if $a\le b$ and $\langle a,b\rangle=[b,a]$ if 
$b\le a$. 
\label{notation:convexhull}

\medskip
If $X$ and $Y$ are two nonempty subsets of $\IR$, the notation $X<Y$
\label{notation:XlessY}
means that, $\forall x\in X,\forall y\in
Y$,  $x<y$ (in this case $X$ and $Y$ are disjoint) and  $X\le
Y$ means that, $\forall x\in X,\forall y\in
Y$, $x\le y$ ($X$ and $Y$ may have a common point, equal to $\max X=\min Y$). 
We may also say that $X$ is on the left of $Y$.

\begin{lem}\label{lem:open-finiteCC}
Every open set $U\subset \IR$ can be written as the union of
countably many disjoint open intervals.
\end{lem}

\begin{proof}
The connected components of $U$ are disjoint nonempty
open intervals, and every
non degenerate interval contains a rational number, which implies that the
connected components of $U$ are countable.
\end{proof}

\subsection{Interval maps, monotonicity, critical points, piecewise monotone and piecewise linear maps}

We say that $f\colon I\to I$ is an  \emph{interval map}\index{interval
map} if  $I$ is a non degenerate compact interval and $f$ is a
continuous map.

When dealing with an interval map $f\colon I\to I$, we shall always refer 
to the ambient space. The topology is the induced topology on $I$;
points and sets are implicitly points in $I$ and subsets of $I$, 
and intervals are subintervals of $I$ (and hence are bounded intervals).

\begin{rem}
The fixed points of an interval map $f$ can be easily seen on the graph of $f$.
Indeed, $x$ is a fixed point if and only if $(x,x)$ is in
the intersection of the graph $y=f(x)$ with the diagonal $y=x$. E.g., in 
Example~\ref{ex:orbit}, the map has two fixed points, $0$ and $c$.
Similarly, the points of $P_n(f)$ correspond to 
the intersection of the graph of $f^n$ with $y=x$.
\end{rem}

Let $f\colon I\to \IR$ be a continuous map, where $I$ is an interval, and let 
$J$ be a non degenerate  subinterval of $I$.
\begin{itemize}
\item
The map $f$ is increasing\index{increasing} (resp. decreasing\index{decreasing}) on $J$ if
for all points $x,y\in J$, $x<y\Rightarrow f(x)<f(y)$ (resp. $f(x)> f(y)$).
\item 
The map $f$ is non decreasing\index{non decreasing} (resp. non increasing\index{non decreasing}) on $J$ if
for all $x,y\in J$, $x<y\Rightarrow f(x)\le f(y)$ (resp. $f(x)\ge f(y)$).
\item
The map $f$ is 
\emph{monotone}\index{monotone} (resp. \emph{strictly 
monotone}\index{strictly monotone}) on $J$ if
$f$ is either non decreasing or non increasing
(resp. either increasing or decreasing) on $J$.
\end{itemize}

A \emph{critical point of $f$}\index{critical point}\label{def:Cf} is a point $x\in I$ 
such that there exists no neighborhood of $x$ on which $f$ is strictly monotone.
Notice that if $f$ is differentiable, the set of critical points
is included in the set of zeros of $f'$. %$\{x\in I\mid f'(x)=0\}$.

\medskip
The map $f$ is \emph{piecewise
monotone}\index{piecewise monotone map}\label{def:pm} if the interval $I$ can be
divided into finitely many subintervals on each of which $f$ is
monotone, that is, there exist points 
$a_0=\min I<a_1<\ldots a_{n-1}< a_n=\max I$
such that $f$ is monotone on $[a_i,a_{i+1}]$ for all 
$i\in\Lbrack 0, n-1\Rbrack$.
The set of critical points of $f$ is included in
$\{a_1,\ldots, a_{n-1}\}$. Conversely, if the set of critical points of 
$f$ is finite, then $f$ is piecewise monotone. 

\begin{rem}
The critical points are also called \emph{turning points},
\index{turning points} especially when the map $f$ is piecewise monotone.
\end{rem}

Let $f\colon I\to \IR$ be a continuous map, where $I:=[a,b]$, $a<b$.
The map $f$ is \emph{linear}\index{linear map}
if there exist $\alpha,\beta\in\IR$ such 
that $f(x)=\alpha x+\beta$ for all $x\in [a,b]$. The 
\emph{slope}\index{slope (of a linear map)}\index{slope(f)@$\slope(f)$}
\label{notation:slope}
of $f$ is $\slope(f):=\alpha$. One has $\slope(f)=\frac{f(b)-f(a)}{b-a}$ 
and $|\slope(f)|=\frac{|f(I)|}{|I|}$.

$f$ is \emph{piecewise linear}\index{piecewise linear map} 
if there exist $a_0=\min I<a_1<\ldots <a_{n-1}< a_n=\max I$
such that $f$ is linear on $[a_i,a_{i+1}]$ for all $i\in\Lbrack 0,n-1\Rbrack$.
In particular, a piecewise linear map is piecewise monotone.

Most of our examples will be piecewise linear.

\subsection{Rescaling}

If two interval maps $f$ and $g$ are conjugate by an
increasing linear homeomorphism, they have the same graph up to 
the action of a homothety or a translation. We call this action
a \emph{rescaling}\index{rescaling (of an interval map)}.
If $g$ is conjugate to $f$ by a decreasing linear homeomorphism, the graph 
of $g$ is obtained from the one of $f$ by a half-turn rotation and a rescaling.
Not only are the maps $f$ and $g$ conjugate, but they have exactly the 
same  properties (when the conjugacy is decreasing, it just reverses the 
order when order is involved in a property).
 
\begin{rem}
When dealing with interval maps, one may assume that the interval is
$[0,1]$. Indeed, if  $f\colon [a,b]\to [a,b]$ is an interval map,
let $\vfi\colon [0,1]\to [a,b]$ be the linear homeomorphism defined by
$\vfi(x):=a+(b-a)x$ and let $g:=\vfi^{-1}\circ f\circ \vfi$. 
The maps $f\colon [a,b]\to [a,b]$ and $g\colon [0,1]\to [0,1]$ are conjugate,
and $g$ is a mere rescaling of $f$.
\end{rem}

\subsection{Periodic intervals}
Let $f\colon I\to I$ be an interval map.
If $J_1,\ldots,J_p$ are disjoint non degenerate closed subintervals of
$I$ such that $f(J_i)=J_{i+1\bmod p}$ for all $i\in\Lbrack 1,p\Rbrack$, then
$(J_1,\ldots, J_p)$ (as well as the set $C:=J_1\cup\cdots\cup J_p$)  is called
a \emph{cycle of intervals of period $p$}\index{cycle of intervals}\index{period of a cycle of intervals}. Moreover, $J_1$  is
called a \emph{periodic interval of period $p$}\index{periodic interval}.

\subsection{Intermediate value theorem}
The intermediate value theorem is fundamental and we shall use it constantly.
For a convenience, we give several equivalent statements.

\begin{theo}[intermediate value theorem\label{theo:ivt}]\index{Intermediate value theorem}
Let $f\colon I\to \IR$ be a continuous map, where $I$ is a nonempty interval.
\begin{itemize}
\item
Let $J$ be a nonempty subinterval of $I$.
Then $f(J)$ is also a nonempty interval.
\item
Let $x_1,x_2\in I$ with $x_1\le x_2$.
Then $f([x_1,x_2])\supset \langle f(x_1),f(x_2)\rangle$.
In particular, for every $c$ between $f(x_1)$ and $f(x_2)$, there
exists $x\in [x_1,x_2]$ such that $f(x)=c$.
%% In particular, if there exist $x_1,x_2\in I$ such that $f(x_1)\le 0$ and 
%% $f(x_2)\ge 0$, then there exists $x\in J$ such that $f(x)=0$.
\end{itemize}
\end{theo}

\begin{proof}
The first assertion follows from the fact that the image of a connected
set by a continuous map is connected (Theorem~\ref{theo:f(connected)} in 
Appendix)  and the image of a nonempty set is nonempty.
The second assertion is a straightforward consequence of the first one
with $J=[x_1,x_2]$. %The second assertion is the case $c=0$.
\end{proof}

%************
\subsection*{Definition of graph maps}

A \emph{topological graph}\index{topological graph}\index{graph (topological)}
is a compact connected metric space $G$ containing a finite subset $V$
such that $G\setminus V$ has finitely many connected components
and every connected component of $G\setminus V$ is homeomorphic
to $(0,1)$. A topological graph is \emph{non degenerate}
if it contains more than one point.\index{non degenerate graph/tree} 
A \emph{subgraph}\index{subgraph (of a topological graph)}
of $G$ is a closed connected subset of $G$; a subgraph is a topological graph. 
A \emph{tree}\index{tree} is a topological graph containing
no subset homeomorphic to a circle.
A \emph{branching point}\index{branching point of a topological graph} is
a point having no neighborhood homeomorphic to a real interval. An
\emph{endpoint}\index{endpoint of a topological graph}
\label{notation:endpointsgraph}
is a point having a neighborhood homeomorphic to
the half-closed interval $[0,1)$. The sets of branching points and 
endpoints are finite (they are
included in $V$). If $H$ is a subgraph of $G$, the set of endpoints
of $H$ is denoted by $\End{H}$. A subset of $G$ is called
an \emph{interval}\index{interval in a topological graph} (resp.
a \emph{circle}\index{circle in a topological graph}) if it is
homeomorphic to an interval of the real line (resp. a circle of positive
radius). 

\begin{figure}[htb]
\centerline{\includegraphics{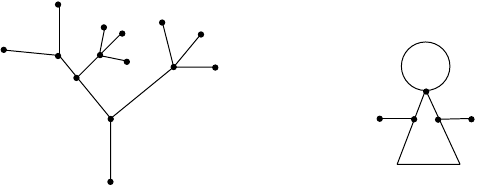}}
\caption{A tree (on the left) and a topological graph (on the right). The branching points and the endpoints are indicated by big dots.}
\end{figure}

A \emph{graph} (resp. \emph{tree}) 
\emph{map}\index{graph map}\index{tree map} is a continuous map 
$f\colon X\to X$, where $X$ is a non degenerate topological 
graph (resp. tree). If $G_1,\ldots, G_p$ are disjoint non degenerate 
subgraphs of $X$ such that $f(G_i)=G_{i+1\bmod p}$
for all $i\in\Lbrack 1,p\Rbrack$, then $(G_1,\ldots, G_p)$
is called a \emph{cycle of graphs of period p}\index{cycle of graphs}\index{period of a cycle of graphs}.

\begin{defi}
Let $f\colon G\to G$ be a graph map.
If $I\subset G$ is either a non degenerate interval or a circle, the map $f|_I$ is said to be
\emph{monotone}\index{monotone (for graph maps)} if it is locally monotone
at every point $x\in I$, that is,
there exists an open neighborhood $U$ of $x$ with respect to the topology of
$I$ such that:
\begin{itemize}
\item $U$ contains $K(x)$, where $K(x)\subset I$ is the largest subinterval
of $I$ containing $x$ on which $f$ is constant,
\item
$U$ and $f(U)$ are homeomorphic to intervals,
\item $f|_U\colon U\to f(U)$,
seen as a map between intervals, is monotone
(more precisely, there exist intervals $J,J'\subset \IR$ and homeomorphisms
$h\colon U\to J$, $h'\colon f(U)\to J'$ such that $h'\circ f|_U\circ h^{-1}
\colon J\to J'$ is monotone). 
\end{itemize} 
\end{defi}

Notice that, when $G$ is a tree, the fact that $f|_I$ is
monotone implies that $f(I)$ is necessarily
an interval, whereas in general $f(I)$ 
may not be an interval (in particular, $f(I)$ may wrap around circles).

%*********************************************************
\section{Chains of intervals and periodic points}

The next lemma is a basic tool to prove the existence of fixed
points. Below,  Lemma~\ref{lem:chain-of-intervals} states the existence
of periodic points when some intervals are nested under
the action of $f$.

\begin{lem}\label{lem:fixed-point}
Let $f\colon [a,b]\to \IR$ be a continuous map. If $f([a,b])\subset
[a,b]$ or $f([a,b])\supset [a,b]$, then $f$ has a fixed point.
\end{lem}

\begin{proof}
Let $g(x):=f(x)-x$. If $f([a,b])\subset [a,b]$, then
$$
g(a)=f(a)-a\ge a-a=0\text{ and } g(b)=f(b)-b\le b-b=0.
$$
By the intermediate value theorem applied to $g$, there exists $c\in [a,b]$ with $g(c)=0$.  If
$f([a,b])\supset [a,b]$, there exist $x,y\in [a,b]$ such that
$f(x)\le a$ and $f(x)\ge b$. We then have
$$
g(x)=f(x)-x\le a-x\le 0\text{ and } g(y)=f(y)-y\ge b-y\ge 0.
$$
Thus there exists $c\in [x,y]$ with $g(c)=0$ by the intermediate value 
theorem. In both cases, $c$ is a fixed point of $f$.
\end{proof}

\begin{defi}[covering, chain of intervals]\label{def:covering}\index{covering}\index{chain of intervals}
\label{notation:coveringintervals}
Let $f$ be an interval map. 
\begin{itemize}
\item
Let $J,K$ be two nonempty closed intervals. Then
$J$ is said to \emph{cover} $K$ (for $f$)
if $K\subset f(J)$. This is denoted by 
$J\cover{f}K$, or simply $J\to K$ if there
is no ambiguity. 
If $k$ is a positive integer, $J$ 
\emph{covers  $K$ $k$ times} if $J$ contains $k$ closed subintervals with
disjoint interiors such that each one covers $K$.
\item
Let $J_0,\ldots, J_n$ be nonempty closed interval such that
$J_{i-1}$ covers $J_i$ for all $i\in\Lbrack 1,n\Rbrack$.
Then $(J_0,J_1,\ldots, J_n)$ is called a 
\emph{chain of intervals (for $f$)}. This is denoted by
$J_0\to J_1\to \ldots\to J_n$.
\end{itemize}
\end{defi}

\begin{lem}\label{lem:chain-of-intervals}
Let $f$ be an interval map and $n\ge 1$.
\begin{enumerate}
\item
Let $J_0,\ldots, J_n$ be nonempty intervals such that
$J_i\subset f(J_{i-1})$ for all $i$ in $\Lbrack 1,n\Rbrack$. Then 
there exists an interval $K\subset J_0$ such that 
$f^n(K)=J_n$, $f^n(\End{K})=\End{J_n}$
and $f^i(K)\subset J_i$
for all $i\in\Lbrack 0,n\Rbrack$. If in addition $J_0,\ldots, J_n$ are closed
(and so $(J_0,\ldots, J_n)$  is a chain of intervals), then $K$ can be
chosen to be closed.
\item 
Let $(J_0,\ldots, J_n)$ be a chain of intervals such that
$J_0\subset J_n$. Then there exists $x\in J_0$ such that
$f^n(x)=x$ and $f^i(x)\in J_i$ for all $i\in\Lbrack 0, n-1\Rbrack$.
\item
Suppose that, for every $i\in \Lbrack 1,p\Rbrack$, $(J_0^i, \ldots, J_n^i)$ is a
chain of intervals and, for every pair $(i,j)$ of distinct indices in 
$\Lbrack 1,p\Rbrack$, there exists $k\in \Lbrack 0,n\Rbrack$
such that $J_k^i$ and $J_k^j$ have disjoint interiors.  Then there
exist closed intervals $K_1,\ldots, K_p$ with pairwise disjoint
interiors such that
\begin{gather*}
\forall i\in\Lbrack 1,p\Rbrack,\ f^n(K_i)=J_n^i,\ 
f^n(\End{K_i})=\End{J_n^i}\\
\text{and}\quad\forall k\in\Lbrack 0,n\Rbrack,\ 
\forall i\in\Lbrack 1,p\Rbrack,\  f^k(K_i)\subset J_k^i.
\end{gather*}
\end{enumerate}
\end{lem}

\begin{proof}
We first prove by induction on $n$ the following:

\medskip\textsc{Fact 1.}
Let $J_0,\ldots, J_n$ be nonempty intervals such that 
$J_i\subset f(J_{i-1})$ for all $i\in\Lbrack 1,n\Rbrack$. Then
there exist intervals $K_n\subset K_{n-1}\subset\cdots\subset K_1
\subset J_0$ 
such that,
for all $k\in\Lbrack 1,n\Rbrack$ and all $i\in\Lbrack 0,k\Rbrack$,
$$
f^i(K_k)\subset J_i, f^k(K_k)=J_k,
f^k(\End{K_k})=\End{J_k} \text{ and }f^k(\Int{K_k})=\Int{J_k}.
$$
Moreover,
if $J_0,\ldots, J_n$ are closed, then $K_1,\ldots, K_n$ can be chosen to 
be closed too.

\medskip
\noindent$\bullet$ Case $n=1$. We write $\overline{J_1}=[a,b]$. There exist 
$x,y\in \overline{J_0}$ such that $f(x)=a$ and $f(y)=b$. 
If $a$ (resp. $b$) belongs to $f(J_0)$, we choose $x$ (resp. $y$) in $J_0$.
If $a$ (resp. $b$) does not belong to $f(J_0)$, then it does not
belong to $J_1$ either, and $x$ (resp. $y$) is necessarily 
an endpoint of $J_0$. %In this situation, we set $x'=x$ (resp. $y'=y$).
With no loss of generality, we may suppose that $x\le y$ (the other case being
symmetric). We define
$$
y':=\min\{z\ge x\mid f(z)=b\},\quad x':=\max\{z\le y'\mid f(z)=a\}
$$
and $K_1':=[x',y']$. Then $f(K_1')=\overline{J_1}$, $f(\{x',y'\})=\{a,b\}$
and no other point in $K_1'$ is mapped to $a$ or $b$ by $f$. If $J_1$
is closed, then $K_1:=K_1'$ is suitable. Otherwise, it is easy to check that
$K_1$ can be chosen among $(x',y'),[x',y'),(x',y']$ in such a way that
$f(K_1)=J_1$ and $K_1\subset J_0$.

\medskip
\noindent$\bullet$ Suppose that Fact~1 holds for $n$ and consider nonempty
intervals $J_0,\ldots, J_{n}$, $J_{n+1}$ such that $J_i\subset f(J_{i-1})$
for all $i\in\Lbrack 1,n+1\Rbrack$. Let $K_1,\ldots, K_n$ be the 
intervals given by Fact~1 applied to $J_0,\ldots, J_{n}$. 
Since $f^{n+1}(K_n) =f(J_n) \supset J_{n+1}$, we can apply the case $n=1$ 
for the map $g:=f^{n+1}$ and the two intervals
$K_n,J_{n+1}$. We deduce that there exists an interval  $K_{n+1}\subset K_n$,
which is closed if $J_0,\ldots, J_{n+1}$ are closed, and such that
\begin{gather*}
f^{n+1}(K_{n+1})=J_{n+1},\\
f^{n+1}(\End{K_{n+1}})=\End{J_{n+1}}\text{ and }
f^{n+1}(\Int{K_{n+1}})=\Int{J_{n+1}}.
\end{gather*}
Moreover, $f^i(K_{n+1})\subset J_i$ for all $i\in\Lbrack 0,n\Rbrack$
because $K_{n+1}\subset K_n$. This ends the proof of Fact~1, which trivially
implies (i).

\medskip
Let $(J_0,\ldots, J_n)$ be a chain of intervals such that $J_0\supset J_n$.
Fact~1 implies that there exists a closed interval 
$K_n\subset J_0$ such that $f^n(K_n)=J_n$ and
$f^i(K_n)\subset J_i$ for all $i\in\Lbrack 0, n\Rbrack$. 
Thus $f^n(K_n)\supset K_n$ and it is sufficient to apply 
Lemma~\ref{lem:fixed-point} to $g:=f^n|_{K_n}$ in order
to find a point $x\in K_n$ such that $f^n(x)=x$. 
For all $i\in\Lbrack 0,n-1\Rbrack$, $f^i(x)$ obviously belongs to $J_i$. 
This proves (ii).

\medskip
Let $(J_0^i,\ldots, J_n^i)_{1\le i\le p}$ be
chains of intervals satisfying the assumptions of (iii). 
For every $i\in\Lbrack 1, p\Rbrack$, 
let $(K^i_1,\ldots, K^i_n)$ be the closed intervals
given by Fact~1 for $(J_0^i,\ldots, J_n^i)$, and set $K_i:=K^i_n$. 
We fix $i\neq j$ in $\Lbrack 1,p\Rbrack$. By assumption, there exists 
$k\in\Lbrack 0, n\Rbrack$ 
such that $J_k^i$ and $J_k^j$ have disjoint interiors. If
$k=0$, then $K_i$ and $K_j$ have trivially disjoint interiors because they
are respectively included in $J_0^i$ and $J_0^j$.
From now on, we assume that $k\ge 1$. Suppose that
$K_i^k\cap K_j^k\neq\emptyset$.
The set $f^k(K_i^k\cap K_j^k)$ is included in $J_k^i\cap J_k^j$ and,
by assumption,  $J_k^i$ and $J_k^j$ have disjoint interiors. Therefore
the intervals $J_k^i$ and $J_k^j$ have a common endpoint, say
$b$, and $f^k(K_i^k\cap K_j^k)=\{b\}$.
By definition of $K_i^k$, there is a unique point $z$ in $K_i^k$
such that $f^k(z)=b$, and the same holds for $K_j^k$.  Hence
$K_i^k\cap K_j^k$ contains at most one point. Since $K_i\subset K_i^k$
and $K_j\subset K_j^k$, the intervals $K_i$ and $K_j$ have disjoint
interiors. This concludes the proof of (iii).
\end{proof}

%************
\subsection*{Definitions for graph maps}

The notion of covering extends to graph maps provided 
Definition~\ref{def:covering} is phrased differently.
A modification is needed for two reasons:
\begin{itemize}
\item
one may want to consider circles as ``intervals'' whose endpoints are equal,
\item
for a graph map $f$,
it may occur that a compact interval $I$ satisfies $f(I)\supset I$ but 
contains no fixed point, as
illustrated in Figure~\ref{fig:contrex-recouvrement}.
\end{itemize}

\begin{figure}[htb]
\centerline{\includegraphics{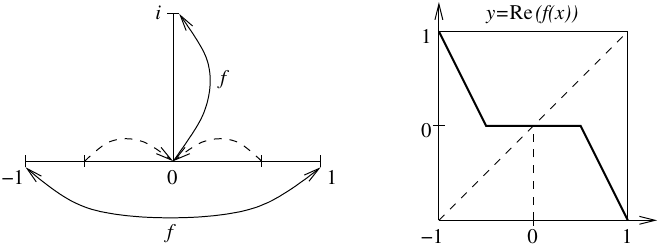}}
\caption{Let $f\colon X\to X$ be a continuous map, where $X$ is the tree 
$[-1,1]\cup i[0,1]\subset \IC$ (on the left), and $f$ is 
such that $f(-1)=1$, $f(1)=-1$,
$f(0)=i$ and  $f$ is one-to-one on $[-1,0]$ and $[0,1]$ (the definition
of $f$ on $i[0,1]$ does not matter). Set $I:=[-1,1]$. It is clear that
$f(I)\supset I$.  Nevertheless $f$ has no fixed point in $I$.  
On the right is represented the real part of $f|_I$; the constant 
interval corresponds to the points $x\in I$ such that $f(x)\in
i[0,1]$. 
%The diagonal $y=x$ is dashed in the picture.
}
\label{fig:contrex-recouvrement}
\end{figure}

%% Next definition is equivalent to Definition~\ref{def:covering} for 
%% intervals maps.  

\begin{defi}\label{def:coveringG}
Let $f\colon G\to G$ be a graph map and let $J,K$ be two non degenerate
intervals in $G$. Then $J$ is said to \emph{cover} $K$\index{covering}
if there exists a subinterval $J'\subset J$ such that $f(J')=K$
and $f(\End{J'})=\End{K}$. 
If $J_0,J_1,\ldots, J_n$ are  intervals in $X$ such
that $J_{i-1}$ covers $J_i$ for all $i\in\Lbrack 1,n\Rbrack$, then
$(\overline{J_0},\ldots, \overline{J_n})$ is a \emph{chain of intervals}\index{chain of intervals} (this is a slight abuse of notation since, if $\overline{J_i}$ is a circle,  it is necessary to remember the endpoint of $J_i$). 
\end{defi}

Using this definition, Lemma~\ref{lem:chain-of-intervals}(ii)-(iii) 
remains valid for graph maps. In particular, if 
$(\overline{J_0},\ldots, \overline{J_{n-1}}, \overline{J_0})$ is a chain of intervals for a graph map $f$,
then there exists a point $x\in \overline{J_0}$ such that $f^n(x)=x$ and
$f^i(x)\in \overline{J_i}$ for all $i\in\Lbrack 1,n-1\Rbrack$.

A variant, called \emph{positive covering}, has been introduced in
\cite{AR}. Positive covering does not imply covering, but implies
the same conclusions concerning periodic points. We do not state
the definition  because it will not be needed in this book. 
See \cite{AR, AR2} for the details.

%*********************************************************
\section{Directed graphs}\label{sec:1-directedgraphs}

A (finite) directed graph\index{directed graph}\index{graph (directed)} 
$G$ is a pair $(V,A)$ where $V,A$ are finite
sets and there exist two maps $i,f\colon A\to V$. 
The elements of $V$ are the \emph{vertices}\index{vertex, vertices (in a directed graph)}
of $G$ and the elements of $A$ are the \emph{arrows} of $G$. 
An arrow $a\in A $ goes from its \emph{initial vertex} $u=i(a)$ to its
\emph{final vertex} $v=f(a)$. The arrow $a$ is also denoted by
$u\labelarrow{a}v$\index{arrow in a directed graph}.
\label{notation:arrow}
A directed graph is often given by a picture, as in 
Example~\ref{ex:directed-graph}.
If $V=\{v_1,\ldots,v_p\}$, the \emph{adjacency matrix}
\index{adjacency matrix} of $G$ is the matrix $M=(m_{ij})_{1\le i,j\le p}$, 
where $m_{ij}$ is equal to the number of arrows from $v_i$ to $v_j$. 
Conversely,
if $M=(m_{ij})_{1\le i,j\le p}$ is a matrix such that $m_{ij}\in\IZ^+$
for all $i,j\in\Lbrack 1, p\Rbrack$,  
one can build a directed graph whose adjacency
matrix is $M$: it has $p$ vertices $\{v_1,\ldots,v_p\}$ and 
there are $m_{ij}$ arrows from $v_i$ to $v_j$
for all $i,j\in\Lbrack 1,p\Rbrack$.

A directed graph is \emph{simple}\index{simple directed graph} 
if, for every pair of vertices $(u,v)$,
there is at most one arrow from $u$ to $v$. In this case, an arrow
$u\labelarrow{a}v$ is simply denoted by $u\to v$ since there is no ambiguity.
A directed graph is simple if and only if all the coefficients of its adjacency
matrix belong to $\{0,1\}$.

There are several, equivalent norms for matrices. We shall use the following 
one: if $M=(m_{ij})_{1\le i,j\le p}$, we set $\|M\|:=\sum_{1\le i,j\le p} 
|m_{ij}|$.\index{norm of a matrix}
\label{notation:normmatrix}

\begin{ex}\label{ex:directed-graph}
Figure~\ref{fig:ex-oriented-graph} represents a directed graph with three
vertices $v_1$, $v_2$, $v_3$. Its adjacency matrix is 
$\left(\begin{array}{ccc}0&2&1\\ 1&1&1\\ 0&0&0\end{array}\right)$.
\begin{figure}[htb]
\centerline{\includegraphics{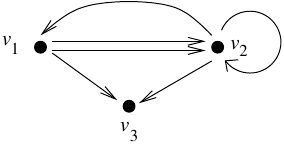}}
\caption{An example of a directed graph.}
\label{fig:ex-oriented-graph}
\end{figure}
\end{ex}

Let $G$ be a directed graph. A \emph{path of length $n$ from $u_0$ to $u_n$} 
\index{path in a directed graph} is a sequence 
$$u_0\labelarrow{a_1} u_1\labelarrow{a_2}u_2\labelarrow{a_3} 
\cdots u_{n-1}\labelarrow{a_n}  u_n,$$ where $u_0,\ldots, u_n$ are
vertices of $G$ and $u_i\labelarrow{a_i}  u_{i+1}$ is an
arrow in $G$ for all $i$ in $\Lbrack 0, n-1\Rbrack$. 
Such a path is called a \emph{cycle}
\index{cycle in a directed graph} if $u_0=u_n$.

If $\CA:=A_0\labelarrow{a_1} A_1\labelarrow{a_2}\cdots \labelarrow{a_n} A_n$ 
and $\CB:=B_0\labelarrow{b_1} B_1\labelarrow{b_2}\cdots
\labelarrow{b_m} B_m$ are two paths such that
$A_n=B_0$, the \emph{concatenation}\index{concatenation of paths in a graph} of $\CA$ and $\CB$, denoted by  $\CA\CB$, is 
the path
$$
A_0\labelarrow{a_1} A_1\labelarrow{a_2}\cdots \labelarrow{a_n} A_n
\labelarrow{b_1} B_1\labelarrow{b_2}\cdots
\labelarrow{b_m} B_m.
$$
If $\CA$, $\CB$ are of respective lengths $n$, $m$, then $\CA\CB$ is of length
$n+m$.

A cycle is \emph{primitive}\index{primitive cycle in a directed graph} 
if it is not the repetition of a shorter cycle, that is, it cannot be
written $\underbrace{\CA\CA\cdots\CA}_{n \text{ \scriptsize times}}$ where $\CA$
is a cycle and $n\ge 2$.

\medskip
A straightforward computation leads to the following result.

\begin{prop}\label{prop:Mn}
Let $G$ be a directed graph and let
$\{v_1,\ldots, v_p\}$ denote its set of vertices. 
Let $M$ be its adjacency matrix. For every
$n\in\IN$, let $M^n=(m_{ij}^n)_{1\le i,j\le p}$.
Then, $\forall n\ge 1$, $\forall i,j\in\Lbrack 1,p\Rbrack$, 
the number of paths of length $n$ from $v_i$ to $v_j$ is
equal to $m_{ij}^n$.
As a consequence, the number of paths of length $n$ in $G$ is equal to 
$$\|M^n\|=\sum_{1\le i,j\le p} 
m^n_{ij}.
$$
\end{prop}

%**************************************************
%Links between transitivity, mixing and sensitivity
\chapter{Links between transitivity, mixing and sensitivity}\label{chap2}

\section{Transitivity and mixing}

We are going to see that, for interval maps, the properties of total
transitivity, topological weak mixing and topological mixing coincide,
contrary to the general case. Moreover the notions of
transitivity and topological mixing are very close. 
Indeed, if $f$ is a transitive interval map which is not topologically 
mixing, then the interval can be divided into two invariant subintervals 
and $f^2$ is topologically mixing on each of 
them. We shall
also give some properties equivalent to topological mixing for interval  maps.
The results of this section are classical (see, e.g., \cite{BCop}).

%**********************************************************************
\subsection{Definitions}

\begin{defi}[transitivity, transitive set]\index{transitive map}
Let $(X,f)$ be a topological dynamical system.
The map  $f$ is \emph{transitive} if, for all
nonempty open sets $U,V$ in $X$, there exists $n\ge 0$ such that $f^n(U)\cap
V\neq\emptyset$ (or, equivalently, $U\cap f^{-n}(V)\neq\emptyset$).

A \emph{transitive set}\index{transitive set} is an invariant
set $E\subset X$  such that $f|_E\colon E\to E$ is transitive.
\end{defi}

\begin{defi}[total transitivity]\index{totally transitive}
Let $(X,f)$ be a topological dynamical system.
The map  $f$ is \emph{totally transitive} if $f^n$ is transitive
for all $n\ge 1$,
\end{defi}

The next result provides an
equivalent definition of transitivity when the space has no isolated
point (see e.g. \cite{DGS}).
Lemma~\ref{lem:transitivity-semi-open}
states two easy properties of transitive interval maps.

\begin{prop}\label{prop:transitive-dense-orbit}
Let $(X,f)$ be a topological dynamical system. 
\begin{enumerate}
\item
If $f$ is transitive, there exists a dense $G_\delta$-set
of points whose orbit is dense in $X$. If a point $x$ has a dense orbit, then
$\omega(x,f)=X$ and the orbit of $f^n(x)$ is dense
in $X$ for all $n\ge 0$. Moreover, either $X$ is finite, or 
$X$ has no isolated point.
\item
If there exists a point whose orbit is dense in $X$
and if $X$ has no isolated point, then $f$ is transitive.
\item
If there exists a point $x$ such that
$\omega(x,f)=X$, then $f$ is transitive. 
\end{enumerate}
In particular, if $X$ has no isolated point, then $f$ is
transitive iff there is a point of dense orbit iff there is a
point $x\in X$ such that $\omega(x,f)=X$.
\end{prop}

\begin{proof}
Assume first that $f$ is transitive. Let $U$ be a nonempty open set. By
transitivity, for every nonempty open set $V$, there exists $n\ge 0$
such that $f^{-n}(U)\cap V\ne\emptyset$. In other words, 
$\bigcup_{n\ge 0}f^{-n}(U)$ is dense in $X$.
Since $X$ is a compact metric space, there exists
a countable basis of nonempty open sets, say $(U_k)_{k\ge 0}$. 
For all $k\ge 0$, the set
$
\bigcup_{n\ge 0}f^{-n}(U_k)
$
is a dense open set by transitivity.
Let
$$
G:=\bigcap_{k\ge 0}\bigcup_{n\ge 0}f^{-n}(U_k).
$$
Then $G$ is a dense $G_\delta$-set and, if
$x\in G$, then $f^n(x)$ enters any set $U_k$ for some $n$,
which means that $\CO_f(x)$ is dense in $X$.

Assume that $x_0$ is an isolated point and set $U:=\{x_0\}$. Since
$U$ is a nonempty open set,  $U\cap G\ne\emptyset$,
that is, $x_0$ has a dense orbit. We set
$V_0:=f^{-1}(U)$; this is an open set. Suppose that $V_0$ is empty.
Then $f^{-n}(U)=\emptyset$ for all $n\ge 1$. The space $X$ is not
reduced to $\{x_0\}$ (otherwise we would have $f(x_0)=x_0$ and
$V_0=\{x_0\}$), and thus there exists a nonempty open set $V$
not containing $x_0$. This implies that $f^{-n}(U)\cap V=\emptyset$
for $n=0$ and for all $n\ge 1$, which contradict the transitivity.
Therefore, $V_0$ is a nonempty open set. By transitivity, there exists
$n\ge 0$ such that $f^n(U)\cap V_0\ne\emptyset$. This implies
that $f^n(x_0)\in V_0$ and $f^{n+1}(x_0)=x_0$. Therefore, the point $x_0$
is periodic. Since $x_0$ has a dense orbit, this implies that $X$ is
finite and equal to $\CO_f(x_0)$. In this case, $f$ acts as a cyclic
permutation on $X$ and it is clear that for every point $x\in X$,
$X=\CO_f(x)=\omega(x_0,f)$.

By refutation, if $f$ is transitive and $X$ is infinite, then $X$ has no isolated point.

\medskip
Assume now that there exists a point $x$ whose orbit is dense 
in $X$ and that $X$ has no isolated point.
Let $U,V$ be two nonempty open sets in $X$. There exists 
an integer $n\ge 0$ such that $f^n(x)\in U$. The set $V\setminus\{x,f(x),
\ldots, f^n(x)\}$ is open, and it is nonempty because $X$ has no
isolated point. Thus there exists $m\ge 0$ such that $f^m(x)\in
V\setminus\{x,f(x),\ldots, f^n(x)\}$. It follows
that $m>n$ and $f^m(x)=f^{m-n}(f^n(x))\in f^{m-n}(U)\cap V$, and 
thus $f^{m-n}U\cap V\neq \emptyset$. We deduce that $f$ is transitive, 
which is (ii).
Moreover, we have proved that,
for all nonempty open sets $V$, for all $n\ge 0$, there exists $m>n$
such that $f^m(x)\in V$. This implies that the orbit of $f^n(x)$ is dense
for all $n\ge 0$, and hence $\omega(x,f)=X$. This ends the proof of (i).

\medskip
Finally, assume that $\omega(x,f)=X$ for some point $x\in X$.
This implies that every nonempty open set contains
some point $f^n(x)$ with $n$ arbitrarily large.
Let $U,V$ be two nonempty open sets in $X$. Then there exist
integers $n_2> n_1\ge 0$ such that $f^{n_1}(x)\in U$ and
$f^{n_2}(x)\in V$. Then $n_2-n_1>0$ and $f^{n_2-n_1}(U)$ contains
the point $f^{n_2-n_1}(f^{n_1}(x))=f^{n_2}(x)$. Thus 
$f^{n_2-n_1}(U)\cap V\ne\emptyset$. This implies that $f$ is transitive,
which is (iii).
\end{proof}

\begin{lem}\label{lem:transitivity-semi-open}
Let $f\colon I\to I$ be a transitive interval map. 
\begin{enumerate}
\item The image of a non degenerate interval is a non degenerate interval.
\item The map $f$ is onto.
\end{enumerate}
\end{lem}

\begin{proof} Let $J$ be a non degenerate interval.
Since $J$ is connected, $f(J)$ is also connected, that is, it is an
interval.  Suppose that $f(J)$ is reduced to a single point; we
write $f(J)=\{y\}$. By
Proposition~\ref{prop:transitive-dense-orbit}, there exists 
a point $x\in J$ whose orbit is dense, and $y=f(x)$ also has a dense
orbit. Thus there exists $n\ge 0$ such that $f^n(y)\in J$. This implies 
that $y=f^{n+1}(y)$, that is, $y$ is a periodic point. But this is impossible 
because the orbit of $y$ is dense in $I$. We deduce that the interval 
$f(J)$ is not degenerate and thus (i) holds.

By Proposition~\ref{prop:transitive-dense-orbit}, there exists
a point $x$ such that $I=\overline{\{f^n(x)\mid n\ge 1\}}$.
Notice that $\{f^n(x)\mid n\ge 1\}\subset f(I)$.
Since $I$ is compact, $f(I)$ is compact too, and hence $I\subset f(I)$,
which implies that $f(I)=I$. This is (ii).
\end{proof}

\begin{defi}[mixing, weak mixing]\index{weak mixing}\index{topologically weakly mixing}\index{mixing}\index{topologically mixing}
Let $(X,f)$ be a topological dynamical system.
The map $f$ is \emph{topologically mixing} if,
for all nonempty open sets $U,V$ in $X$, there exists an integer $N\ge 0$ 
such that, $\forall n\ge N$, $f^{n}(U)\cap V\neq\emptyset$.
The map $f$ is \emph{topologically weakly mixing} if
$f\times f$  is transitive, where $f\times f$ is the map
$$
\begin{array}{ccl}X\times X&\to& X\times X\\
(x,y)&\mapsto&(f(x),f(y))\end{array}
$$
\end{defi}

It is well known that topological mixing implies topological weak mixing (see,
e.g., \cite{DGS}). Moreover, topological weak mixing implies total 
transitivity. This is a folklore result. It can be
proved using the following result, due to 
Furstenberg \cite{Fur1}.

\begin{prop}\label{prop:weakly-mixing-product}
Let $(X,f)$ be a topological dynamical system.
If $f$ is topologically weakly mixing, then the
product system $(X^n,\underbrace{f\times\cdots\times f}_{n\ \rm times})$  is
transitive for all integers $n\ge 1$.
\end{prop}

\begin{proof}
For all open sets $U,V$ in $X$, we define
$$
N(U,V):=\{n\ge 0\mid U\cap f^{-n}(V)\neq\emptyset\}.
$$
Let $U_1, U_2, V_1, V_2$ be nonempty open sets in $X$. Since $f\times f$
is transitive, there exists an integer $n\ge 0$ such that $(U_1\times V_1)\cap
(f\times f)^{-n}(U_2\times V_2)\neq\emptyset$, that is,
$U_1\cap f^{-n}(U_2)\neq \emptyset$ and $V_1\cap f^{-n}(V_2)\neq \emptyset$.
We first remark that this implies
\begin{equation}\label{eq:weakmixing-transitive}
\forall\, U_1, U_2 \text{ nonempty open sets},\ N(U_1,U_2)\neq\emptyset.
\end{equation}
Now we are going to show that there exist nonempty open sets
$U,V$ such that $N(U,V)\subset N(U_1,V_1)\cap N(U_2,V_2)$. We
set $U:=U_1\cap f^{-n}(U_2)$ and $V:=V_1\cap f^{-n}(V_2)$. These sets are open,
and we have shown that they are not empty. Let $k\in N(U,V)$.
This integer exists by \eqref{eq:weakmixing-transitive} and satisfies
$U_1\cap f^{-n}(U_2)\cap f^{-k}(V_1)\cap f^{-n-k}(V_2)
\neq\emptyset$. This implies that $U_1\cap f^{-k}(V_1)\neq \emptyset$ and
$U_2\cap f^{-k}(V_2)\neq \emptyset$, and thus $N(U,V)\subset 
N(U_1,V_1)\cap N(U_2,V_2)$.
Then, by a straightforward induction, we see that, for all nonempty
open sets $U_1,\ldots, U_n,V_1\ldots, V_n$, there exist nonempty
open sets $U,V$ such that
$$
N(U,V)\subset N(U_1,V_1)\cap N(U_2,V_2)\cap\cdots \cap N(U_n,V_n).
$$
Combined with \eqref{eq:weakmixing-transitive}, this implies that
$(X^n,f\times\cdots\times f)$ is transitive.
\end{proof}

\begin{theo}\label{theo:mixing-weak-mixing}
Let $(X,f)$ be a topological dynamical system.
If $f$ is topologically mixing, then it is topologically weakly mixing. 
If $f$ is topologically weakly mixing, then
$f^n$ is topologically weakly mixing for all $n\ge 1$ and 
$f$ is totally transitive.
\end{theo}

\begin{proof}
First we assume that $f$ is topologically mixing. Let $W_1,W_2$ be
two nonempty open sets in $X\times X$. There exist nonempty open sets
$U, U', V,V'$ in $X$ such that $U\times U'\subset W_1$ and $V\times V'\subset
W_2$. Since $f$ is topologically mixing, there exists $N\ge 0$ such that,
$\forall n\ge N$, $f^n(U)\cap V\neq\emptyset$ and
 $f^n(U')\cap V'\neq\emptyset$. Hence $f^N(W_1)\cap W_2\neq\emptyset$.
We deduce that $f$ is topologically weakly mixing.

From now on, we assume that $f$ is topologically weakly mixing and we fix
$n\ge 1$. Let $U,U', V,V'$ be nonempty open sets in $X$. We define
$$
W:=U\times f^{-1} (U)\times\cdots \times f^{-(n-1)}(U)\times
V\times f^{-1} (V)\times\cdots \times f^{-(n-1)}(V)
$$
and
$$
W':=\underbrace{U'\times\cdots\times U'}_{n\ \rm times}\times
\underbrace{V'\times\cdots\times V'}_{n\ \rm  times}.
$$
The sets $W,W'$ are open in $X^{2n}$. Moreover, 
$(X^{2n},f\times\cdots \times f)$ is transitive by 
Proposition~\ref{prop:weakly-mixing-product}. Thus there exists $k\ge 0$
such that $f^{-k}(W)\cap W'\neq\emptyset$. This implies that
$f^{-(k+i)}(U)\cap U'\neq\emptyset$ and 
$f^{-(k+i)}(V)\cap V'\neq\emptyset$ for all $i\in\Lbrack 0,n-1\Rbrack$.
We choose $i\in\Lbrack 0,n-1\Rbrack$ 
such that $k+i$ is a multiple of $n$; we write $k+i=np$. We deduce
that $(f\times f)^{-np}(U\times V)\cap (U'\times V')\neq\emptyset$.
Therefore, $f^n$ is topologically weakly mixing. This trivially implies that 
$f^n$ is transitive. 
\end{proof}

Here is an equivalent definition of mixing for interval maps.

\begin{prop}\label{prop:def-mixing}
An interval map $f\colon [a,b]\to [a,b]$ is topologically mixing if and only 
if for all $\eps>0$ and all non degenerate intervals $J\subset [a,b]$, 
there exists an integer $N$ such that
$f^n(J)\supset [a+\eps, b-\eps]$ for all $n\ge N$.
\end{prop}

\begin{proof}
Suppose first that $f$ is topologically mixing. Let $\eps>0$.
Let $U_1:=(a,a+\eps)$ and
$U_2:=(b-\eps,b)$. If $J$ is a nonempty open interval, there exists
$N_1$ such that $f^n(J)\cap U_1\neq\emptyset$ for all $n\ge N_1$
because $f$ is  topologically mixing. Similarly, there exists $N_2$ such that
$f^n(J)\cap U_2\neq\emptyset$ for all $n\ge N_2$. Therefore, for
all $n\ge \max\{N_1,N_2\}$, $f^n(J)$ meets both $U_1$ and $U_2$,
which implies that $f^n(J)\supset [a+\eps,b-\eps]$ by connectedness.
If $J$ is a non degenerate subinterval, the same result holds by
considering the nonempty open interval $\Int{J}$.

Suppose now that, for every $\eps>0$ and every non degenerate interval
$J\subset [a,b]$, there exists an integer $N$ such that
$f^n(J)\supset [a+\eps,b-\eps]$ for all $n\ge N$ 
Let $U,V$ be two nonempty open sets
in $[a,b]$. We choose two nonempty open subintervals  $J,K$  such that
$J\subset U$, $K\subset V$ and neither $a$ nor $b$ is an endpoint of
$K$.  There exists $\eps>0$ such that $K\subset [a+\eps, b-\eps]$. By
assumption, there exists $N$ such that
$f^n(J)\supset [a+\eps, b-\eps]\supset K$ for all $n\ge N$. 
This implies that $f^n(U)\cap V\neq\emptyset$ for all $n\ge N$. 
We conclude that $f$ is topologically mixing.
\end{proof}

%**********************************************************************
\subsection{A basic example of mixing map}

In the sequel, we shall need to show that several interval maps are transitive
or mixing. In some simple cases, this can be done by using
Lemmas~\ref{lem:N-critical-points} and
\ref{lem:repulsive-fixed-point}, combined together.

Recall that the definition of critical points is given page~\pageref{def:Cf}.

\begin{defi}
Let $f$ be an interval map and $\lambda>1$. 
Suppose that $f$ has finitely or countably many critical points. 
The map $f$ is called \emph{$\lambda$-expanding}\index{expanding} if, 
for every subinterval $[x,y]$ on which $f$ is monotone,  
$|f(y)-f(x)|\ge \lambda |x-y|$.
\end{defi}

\begin{lem}\label{lem:N-critical-points}
Let $f\colon I\to I$ be a $\lambda$-expanding interval map with 
$\lambda>N$, where $N$ is a positive integer. Then, for every non
degenerate  subinterval $J$, there exists an integer $n\ge 0$ such
that $f^n(J)$  contains at least $N$ distinct critical points.
\end{lem}

\begin{proof}
Let $C_f$ be the set of critical points of $f$. We set $\alpha:=\lambda/N>1$.
Consider a nonempty open subinterval $J$. If $J$ contains
exactly $k$ distinct critical points with $k\in\Lbrack 0,N-1\Rbrack$, then
$J\setminus C_f$ has $k+1$ connected components, say
$J_0,\ldots,J_k$, and $|J_0|+\cdots+|J_k|=|J|$. By the pigeonhole principle,
there exists
$i\in\Lbrack 0,k\Rbrack$ such that $|J_i|\ge \frac{|J|}{k+1}\ge \frac{|J|}{N}$.  Since
$J_i$ contains no critical point, the map $f|_{J_i}$ is
monotone. Hence  $|f(J_i)|\ge\lambda|J_i|$ and
\begin{equation}\label{eq:length}
|f(J)|\ge |f(J_i)|\ge \lambda|J_i|\ge \alpha|J|.
\end{equation}
Suppose that, for all
$n\ge 0$, $f^n(J)$ contains strictly less than $N$ distinct critical
points. Then $|f^n(J)|\ge\alpha^n|J|$ for all $n\ge 0$
by \eqref{eq:length}. But this is impossible
because $|f^n(J)|$ is bounded by $|I|$ whereas $\alpha^n|J|$ goes to infinity
when $n\to+\infty$. This is sufficient to conclude the proof because 
any non degenerate interval contains a nonempty open interval.
\end{proof}

\begin{lem}\label{lem:repulsive-fixed-point}
Let $f\colon I\to I$ be an interval map, $\lambda>1$ and $a,b \in I$ with 
$a<b$.  Suppose that $f(a)=a$ and 
$$\forall x\in [a,b],\ f(x)-f(a)\ge \lambda (x-a).$$
Then, for all $\eps>0$,
there exists $n\ge 0$ such that $f^n([a,a+\eps])\supset [a,b]$.
\end{lem}

\begin{proof}
Let $\eps\ge 0$. If $\eps> b-a$, then $f^0([a,a+\eps])\supset [a,b]$.
Suppose on the contrary that $a+\eps\le b$. Then  $f(a+\eps)-f(a)\ge
\lambda\eps$ by assumption, and hence  $f([a,a+\eps])\supset
[a,a+\lambda \eps]$ by the intermediate value theorem
(recall that $f(a)=a$). A straightforward induction on $n$
shows  that $f^n([a,a+\eps])\supset [a,a+\lambda^n\eps]$ as
long as $a+\lambda^{n-1}\eps\le b$. Since $\lambda>1$,
there exists an  integer $n\ge 1$ such that
$a+\lambda^{n-1}\eps\le b<a+\lambda^n\eps$. Hence
$f^n([a,a+\eps])\supset [a,b]$.
\end{proof}

\begin{rem}
We shall use Lemmas \ref{lem:N-critical-points} and 
\ref{lem:repulsive-fixed-point} for piecewise linear maps (as in
Example~\ref{ex:tent-map} below) or for maps $f\colon I\to I$ such that
the interval $I$ can be
divided into countably many subintervals on each of which $f$ is
linear. In these situations,  $f$ is $\lambda$-expanding
if and only if the absolute value of the slope of $f$ is greater than or equal to
$\lambda$ on each interval on which $f$ is linear.

In Lemma~\ref{lem:repulsive-fixed-point}, the 
assumption $f(x)-f(a)\ge \lambda (x-a)$ is verified as soon as
$f|_{[a,b]}$ is linear of slope greater than or equal to
$\lambda$.
\end{rem}

\begin{ex}\label{ex:tent-map}
We are going to exhibit a family of topologically mixing interval maps. 
These maps are piecewise linear, and the absolute value of their slope is 
constant. These maps are basic examples; they will be reused later to build 
other examples.

We fix an integer $p\ge 2$. We define the map $T_p\colon [0,1]\to[0,1]$
\index{Tp@$T_2$ (tent map), $T_p$} by:
\begin{eqnarray*}
\forall 0\le k\le \frac{p-1}{2},\ \forall x\in\left[\frac{2k}{p},\frac{2k+1}{p}\right],&& T_p(x) := px-2k,\\
\forall 0\le k \le \frac{p-2}{2},\ \forall x\in\left[\frac{2k+1}{p},
\frac{2k+2}p\right], && T_p(x) := -px+2k+2.
\end{eqnarray*}
\begin{figure}[htb]
\centerline{ \includegraphics{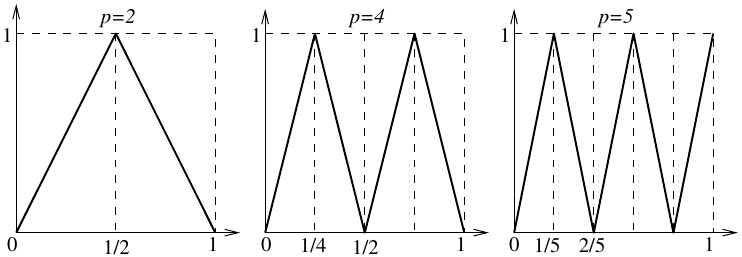}}
\caption{Mixing maps $T_p$ of slope $\pm p$, for $p=2$, $p=4$ and $p=5$.
The map $T_2$ (on the left) is called the tent map.}
\label{fig:ex-transitive}
\end{figure}
The slope of $T_p$ is either $p$ or $-p$ on each interval of
monotonicity. More precisely, starting from the fixed point $0$, the
slope is alternatively $p, -p, p,\ldots, (-1)^{p-1} p$, and the image of each
interval of monotonicity is $[0,1]$. 
See Figure~\ref{fig:ex-transitive} for the graph
of $T_p$. The map $T_2$ is the so called \emph{tent map}\index{tent map}.

Let $J$ be a non degenerate interval in $[0,1]$. The image of a non 
degenerate interval by $T_p$ is obviously non degenerate, so
$T_p^n(J)$ is a non degenerate interval for all $n\ge 0$. By
Lemma~\ref{lem:N-critical-points},  there exists $n$ such that
$T_p^n(J)$ contains $p-1$ distinct critical points. If $p\ge 3$,
$T_p^n(J)$ contains at least one critical point whose image is $0$,
and thus $0\in T^{n+2}(J)$ because $0$ is a fixed point. If
$p=2$, $T_p^n(J)$ contains the unique critical point $1/2$, and
$T_p^2(1/2)=T_p(1)=0$. In both cases, $T_p^{n+2}(J)$ is a non degenerate 
interval containing $0$. Applying
Lemma~\ref{lem:repulsive-fixed-point} with $a=0$ and $b=\frac 1p$
($T_p$ is of slope $\lambda=p$ on $[0,\frac1p]$), we
deduce that there exists an integer $m\ge 0$
such that $T_p^{n+m+2}(J)\supset \left[0,\frac{1}{p}\right]$, which implies 
that $T_p^{n+m+3}(J)\supset [0,1]$, and hence
$T_p^k(J)=[0,1]$ for all $k\ge n+m+3$. 
Conclusion: $T_p$ is topologically mixing.
\end{ex}

%**********************************************************************
\subsection{Transitivity implies denseness of the set of periodic points}
Proposition~\ref{prop:transitivity-periodic-points} below states that
the set of periodic points of a transitive interval map is dense.
This is a consequence of a result of Sharkovsky \cite{Sha3}. 
We are going to follow the proof of \cite{BCop}; see also \cite{BM2} for a 
different proof.  
This result will be needed in the next section;
a stronger theorem will be given in Chapter~\ref{chap:periodic-points}.

\begin{lem}\label{lem:interval-without-periodic-point}
Let $f\colon I\to I$ be an interval map, $x,y\in I$
and $n,m\in\IN$. Let $J$ be a subinterval of $I$
containing no periodic point and suppose that $x$, $y$, $f^m(x)$, $f^n(y)$
belong to $J$. If $x<f^m(x)$ then $y<f^n(y)$.
\end{lem}

\begin{proof}
We set $g:=f^m$. We first prove by induction on $k$ that 
$g^k(x)>x$ for all $k\ge 1$.  By assumption, the statement holds for $k=1$.
Suppose that $g^i(x)>x$ for all $i\in\Lbrack 1,k-1\Rbrack$ and that 
$g^k(x)\le x$. We write 
$$
\{g^i(x)\mid i\in\Lbrack 0,k-1\Rbrack\}=\{x_0\le x_1\le \cdots\le x_{k-1}\}.
$$
We have $x_0=x$ and $x_1\neq x$ because $x$ is not periodic. It follows that
$$
g^k(x)\le x=x_0 < x_1\le\cdots \le x_{k-1}.
$$
Let $j$ be the integer in $\Lbrack 1,k-1\Rbrack$ such that $x_1=g^j(x)$. 
By the intermediate value theorem,
$$g^{k-j}([x_0,x_1])\supset [g^k(x),g^{k-j}(x)]\supset [x_0,x_1].$$ 
Therefore, by
Lemma~\ref{lem:fixed-point},  $g$ has a periodic point
in $[x_0,x_1]$. But $[x_0,x_1]\subset J$ because 
$x_1=\min\{g^i(x)\mid i\in \Lbrack 1,k-1\Rbrack \}\le g(x)$.
This leads to a contradiction since $J$ contains no periodic point for $g=f^m$.
We deduce that $g^k(x)>x$ for all $k\ge 1$ and the induction is over.

Suppose that $f^n(y)<y$. 
The same argument as above (with reverse order) shows that
$f^{kn}(y)<y$ for all $k\ge 1$.  Hence
$$
y>f^{mn}(y)\quad\text{and}\quad x<f^{mn}(x).
$$
The map $t\mapsto f^{mn}(t)-t$ is continuous on the interval 
$\langle x,y\rangle$ (recall that $\langle x,y\rangle$ is either
$[x,y]$ or $[y,x]$ depending on the order of $x,y$). Thus, by the
intermediate value theorem, there exists a point 
$z\in\langle x,y\rangle$ such that $f^{mn}(z)=z$.
This leads to a contradiction because $\langle x,y\rangle\subset J$. 
Thus we conclude that $y<f^n(y)$ (equality is not possible because $y$ 
is not periodic).
\end{proof}

\begin{prop}\label{prop:transitivity-periodic-points}
If $f\colon I\to I$ is a transitive interval map, then the set of
periodic points is dense in $I$.
\end{prop}

\begin{proof}
Suppose that there exist $a,b\in I$, with $a<b$, such that $(a,b)$ contains
no periodic point. Since $f$ is transitive, there exists a point $x\in
(a,b)$ with a dense orbit
(Proposition~\ref{prop:transitive-dense-orbit}). Thus there exist integers
$m>0$ and $0<p<q$ such that $x<f^m(x)<b$ and $a<f^q(x)<f^p(x)<x$. We set
$y:=f^p(x)$. We then have
$$
a<f^{q-p}(y)<y<x<f^m(x)<b.
$$
But this is impossible by Lemma~\ref{lem:interval-without-periodic-point}
applied to $J=(a,b)$. This concludes the proof.
\end{proof}

%**********************************************************************
\subsection{Transitivity, total transitivity and mixing}\label{subsec:transitivity-mixing}

The next proposition states
that, if an interval map $f$ is transitive then,
either $f$ is totally transitive, or the interval can be divided
into two subintervals on each of which $f^2$ is totally transitive.
Then Proposition~\ref{prop:total-transitivity-mixing} states that total
transitivity implies mixing.  These two results were proved by Barge
and Martin \cite{BM2,BM}.  Blokh also showed the same results a little 
earlier, but in an unpublished paper \cite{Blo6}. 
We are going to follow the ideas of the proof of Barge and Martin.
Blokh's proof, which is different, can be found  in \cite{Blo2}.  

\begin{prop}\label{prop:transitivity-total-transitivity}
Let $f\colon [a,b]\to [a,b]$ be a transitive interval map. Then one of
the following cases holds:
\begin{enumerate}
\item The map $f$ is totally transitive. 
\item There exists $c\in (a,b)$ such that $f([a,c])=[c,b]$, $f([c,b])=[a,c]$, 
and both maps $f^2|_{[a,c]}$ and $f^2|_{[c,b]}$ are
totally transitive. Moreover, $c$ is the unique fixed point of $f$.
\end{enumerate}
\end{prop}

\begin{proof}
Since $f$ is transitive, there exists a point $x_0\in [a,b]$ such that
$\omega(x_0,f)=[a,b]$ 
by Proposition~\ref{prop:transitive-dense-orbit}.
We fix an integer $n\ge 1$ and we  set $W_i^n:=\omega(f^i(x_0),f^n)$
for all $i\in\Lbrack 0,n-1\Rbrack$. We have
$[a,b]=W_0^n\cup\dots\cup W_{n-1}^n$ by Lemma~\ref{lem:omega-set}(iv), which
implies that at least one of the sets $W_1^n,\ldots, W_{n-1}^n$ has a
nonempty interior by the Baire category theorem (Corollary~\ref{cor:baire}). 
Moreover, according to Lemma~\ref{lem:omega-set}(ii)-(iii),
\begin{equation}\label{eq:cycleWi}
\forall i\in\Lbrack 0,n-2\Rbrack,\ f(W_i^n)=W_{i+1}^n\quad\text{and}\quad
f(W_{n-1}^n)=W_0^n.
\end{equation}
Thus, all the sets $W_0^n,\ldots, W_{n-1}^n$
have nonempty interiors by Lemma~\ref{lem:transitivity-semi-open}(i).

Suppose that $\Int{W_i^n}\cap \Int{W_j^n}\neq\emptyset$. Since the set
$\Int{W_i^n}\cap \Int{W_j^n}$ is open and included in $W_i^n=
\omega(f^i(x_0),f^n)$, there exists  $k\ge 0$ such that 
$f^{kn+i}(x_0)$ belongs to $\Int{W_i^n}\cap \Int{W_j^n}$. Moreover, $f^n(W_j^n)= W_j^n$
(Lemma~\ref{lem:omega-set}(i)), and thus $f^{k'n+i}(x_0)\in W_j^n$ 
for all $k'\ge k$. This implies that $W_i^n\subset W_j^n$.
The same argument shows that $W_j^n\subset W_i^n$. Therefore
\begin{equation}\label{eq:WiWj}
\text{if }\Int{W_i^n}\cap \Int{W_j^n}\neq\emptyset\text{, then }W_i^n=W_j^n. 
\end{equation}
Let $\CE_n$ be the collection of all connected components of the sets
$(\Int{W_i^n})_{0\le i\le n-1}$. The elements of $\CE_n$  are 
open intervals and, by  \eqref{eq:WiWj},
two different elements of  $\CE_n$ are disjoint. For every
$C\in\CE_n$, the closed interval $f(\overline{C})$ is 
non degenerate by Lemma~\ref{lem:transitivity-semi-open}, and is contained 
in $W_i^n$ for some $i\in\Lbrack 0,n-1\Rbrack$. Thus, by connectedness, 
there exists $C'\in\CE_n$ such that
$f(\overline{C})\subset \overline{C'}$. Moreover, the orbit of $x_0$
enters infinitely many times every element of $\CE_n$, which implies that,
for all $C,C'\in\CE_n$, there exists $k\ge 1$ such
that $f^k(\overline{C})\cap \overline{C'}\neq\emptyset$, and
hence $f^k(\overline{C})\subset
\overline{C'}$. It follows that $\CE_n$ is finite and the closures of
its elements
are cyclically permuted under the action of $f$.  Thus we can write
$\CE_n=\{C_1,\ldots, C_{p_n}\}$ for some integer $p_n\ge 1$, the $C_i$'s 
satisfying 
$$
\forall i\in\Lbrack 1,p_n-1\Rbrack,\ f(\overline{C_i})\subset\overline{C_{i+1}}
\quad\text{and}\quad f(\overline{C_{p_n}})\subset\overline{C_1}.
$$
The fact that the orbit of $x_0$ is dense implies that
$C_1\cup\cdots\cup C_{p_n}$  is dense too. Since
$C_1,\ldots, C_{p_n}$ are disjoint open intervals, we deduce that
$C_1\cup\cdots\cup C_{p_n}$ is equal to $[a,b]$ deprived of finitely many
points, which are the endpoints of $C_1,\ldots, C_{p_n}$.  

If $p_n=1$, then $W_0^n=\cdots=W_{n-1}^n$, and thus
$\omega(x_0,f^n)=[a,b]$.
Therefore, if $\CE_n$ has a single element for every integer $n\ge 1$,
then $f$ is totally transitive and we are in case (i) of the proposition. 
From now on, we suppose that, for a given $n$, the
number $p_n$ of elements of $\CE_n$ is greater than $1$. We are
going to show that $p_n=2$. Let $c\in [a,b]$ be a fixed point of $f$ (such a
point exists by Lemma~\ref{lem:fixed-point}). If there exists
$C\in\CE_n$ with $c\in C$, then $f(\overline{C})=\overline{C}$. Similarly, if
$c$ is an endpoint of $[a,b]$, then there is a unique $C\in\CE_n$
such that $c\in\overline{C}$, and thus  $f(\overline{C})=\overline{C}$.
In both cases, this leads to a contradiction because $\overline{C_1},\ldots,
\overline{C_{p_n}}$ are cyclically permuted and $p_n\ge 2$.
We deduce that $c$ belongs to $(a,b)$ and is a common endpoint
of two distinct elements of $\CE_n$, say $C$ and $C'$. 
The fact that $c$ is a fixed point implies that
the only possibility for permuting cyclically $\overline{C_1},\ldots,
\overline{C_{p_n}}$ is that $p_n=2$, $\CE_n=\{C,C'\}$,
$f(\overline{C})=\overline{C'}$ and  $f(\overline{C'})=\overline{C}$.
We thus have $\CE_n=\{[a,c), (c,b]\}$ and
\begin{equation}\label{eq:J-K}
f([a,c])=[c,b],\quad f([c,b])=[a,c].
\end{equation}
This implies that $c$ is the unique fixed point of $f$.
Let $\CN:=\{i\in\Lbrack 0,n-1\Rbrack \mid C\subset W_i^n\}$ and
$\CN':=\{i\in\Lbrack 0,n-1\Rbrack \mid C'\subset W_i^n\}$. The sets
$\CN,\CN'$ are nonempty and their union is $\Lbrack 0,n-1\Rbrack$ by
definition of $\CE_n$.
We cannot have $C\cup C'\subset W_i^n$ for some $i\in\Lbrack 0,
n-1\Rbrack$; otherwise the connected set $\overline{C\cup C'}$ would be
included in $W_i^n$, which would contradict the fact that $C,C'$ are distinct
elements of $\CE_n$. This implies that $\CN,\CN'$ are disjoint.
Since $W_0^n,\ldots, W_{n-1}^n$ are cyclically 
permuted by $f$ according to \eqref{eq:cycleWi},
a set $W_i^n$ with $i\in\CN$ (resp. $i\in\CN'$)
is sent to a set $W_j^n$ with $j\in\CN'$ (resp. $j\in\CN$). This
implies that $\CN,\CN'$ have the same number of elements, and that the
integer $n$ is necessarily even.
So  $\omega(x_0,f^n)\subset \omega(x_0,f^2)$
by Lemma~\ref{lem:omega-set}(iv).
Combining this with  \eqref{eq:J-K}, we see
that $\{W_0^2,W_1^2\}=\{[a,c],[c,b]\}$.  Therefore both maps $f^2|_{[a,c]}$ and
$f^2|_{[c,b]}$ are transitive.  If $f^2|_{[a,c]}$ is not totally transitive, 
then the same argument as above, applied to the map $f^2|_{[a,c]}$, 
shows that $f^2|_{[a,c]}$ has a unique fixed point, which belongs to $(a,c)$.
But this is impossible because $c$ is already a fixed point of 
$f^2|_{[a,c]}$. We conclude that $f^2|_{[a,c]}$
is totally transitive, and so is $f^2|_{[c,b]}$ for similar reasons, and we are in case (ii) of
the proposition.
\end{proof}

\begin{prop}\label{prop:total-transitivity-mixing}
Let $f\colon I\to I$ be an interval map. If $f$ is totally transitive,
then it is topologically mixing.
\end{prop}

\begin{proof}
We write $I=[a,b]$. Let $J$ be a non degenerate subinterval of $I$ 
and $\eps>0$.
According to Proposition~\ref{prop:transitivity-periodic-points},
the periodic points are dense in $I$. Thus,
there exist periodic points  $x, x_1,x_2$ with $x\in J$, $x_1\in (a,a+\eps)$ 
and $x_2\in (b-\eps,b)$. Moreover, $x_1$ and $x_2$ can be chosen in such a way
that their orbits are included in $(a,b)$ because there is at most one 
periodic orbit containing $a$ (resp. $b$). We set
$$
\forall i\in\{1,2\},\ y_i:=\min\{f^n(x_i)\mid n\ge 0\}\text{ and } 
z_i:=\max\{f^n(x_i)\mid n\ge 0\}.
$$
Then $y_1\in (a,x_1]\subset (a,a+\eps)$, $z_2\in [x_2,b)\subset (b-\eps,b)$
and $y_2,z_1\in (a,b)$. Let $k$ be a common
multiple of the periods of $x$, $y_1$ and $y_2$. We set $g:=f^k$ and
$$
K:=\bigcup_{n=0}^{+\infty}g^n(J).
$$ 
The point $x\in J$ is fixed under the action of $g$ and
thus $g^n(J)$ contains $x$ for all $n\ge 0$. This implies that 
$K$ is an interval.
Moreover $K$ is dense in $[a,b]$ because $g$ is transitive, and hence
$K\supset (a,b)$. It follows that $y_1,y_2,z_1,z_2\in K$. For
$i=1,2$, let $p_i$ and $q_i$ be non negative 
integers such that $y_i\in g^{p_i}(J)$ and $z_i\in
g^{q_i}(J)$. We set $N:=\max\{p_1,p_2,q_1,q_2\}$. Since $y_1,y_2, z_1,z_2$ 
are fixed points of $g$, they belong to $g^N(J)$ and thus, by the 
intermediate value theorem, $[y_i,z_i]\subset g^N(J)=f^{kN}(J)$ for $i=1,2$.
According to the definition of $y_i, z_i$, the interval $[y_i,z_i]$
contains the whole orbit of $x_i$. A trivial induction shows that
$[y_i,z_i]\subset f^n([y_i,z_i])$ for all $n\ge 0$.  Therefore,
$$
\forall n\ge kN,\ [y_1,z_1]\cup[y_2,z_2]\subset f^n(J).
$$
Since $y_1< a+\eps$ and $z_2>b-\eps$, the fact that $f^n(J)$ is connected
implies that $[a-\eps,b+\eps]\subset f^n(J)$ for all $n\ge kN$. We conclude that
$f$ is topologically mixing, using Proposition~\ref{prop:def-mixing}.
\end{proof}

\begin{cor}\label{cor:mixing-odd-periodic-point}
Let $f\colon I\to I$ be a transitive interval map. Then $f$ is topologically
mixing if and only if it has a periodic point of odd period greater than~$1$.
\end{cor}

\begin{proof}
We write $I=[a,b]$. Suppose first that $f$ is topologically mixing. The set 
of fixed points of $f$ is
closed, and it has an empty interior (otherwise, it would contradict the
mixing assumption). Thus we can choose a non degenerate closed
subinterval $J\subset (a,b)$ such that $J$ contains no fixed
point. Since $f$ is topologically mixing, there exists an
integer $N$ such that $f^n(J)\supset J$ for all $n\ge N$
(Proposition~\ref{prop:def-mixing}). We choose an odd integer $n\ge
N$. Applying Lemma~\ref{lem:fixed-point}, we obtain a point $x\in J$
such that $f^n(x)=x$. The period of $x$ is odd because it divides $n$, and 
it is greater than $1$ because $J$ contains no fixed point.

Suppose now that $f$ is transitive but not totally transitive. 
We are in case (ii) of
Proposition~\ref{prop:transitivity-total-transitivity}: there exists
a fixed point $c\in (a,b)$ such that $f([a,c])=[c,b]$ and $f([c,b])=[a,c]$. 
Consequently, every periodic point has an even period, except $c$. 
By refutation, a transitive map with a periodic point of odd period different 
from $1$ is totally transitive, and thus topologically mixing by
Proposition~\ref{prop:total-transitivity-mixing}.
\end{proof}

%**********************************************************************
\subsection{Transitivity vs. mixing -- summary theorems}

The next two theorems sum up the results
\ref{theo:mixing-weak-mixing}, \ref{prop:def-mixing},
\ref{prop:transitivity-total-transitivity},
\ref{prop:total-transitivity-mixing} and
\ref{cor:mixing-odd-periodic-point}.  The first one is about the difference
between transitivity and mixing. The second one states several
properties equivalent to mixing.

\begin{theo}\label{theo:summary-transitivity}
Let $f\colon [a,b]\to [a,b]$ be a transitive interval map. Then one of
the following cases holds:
\begin{itemize}
\item The map $f$ is topologically mixing.
\item There exists $c\in (a,b)$ such that
$f([a,c])=[c,b]$, $f([c,b])=[a,c]$, and both maps $f^2|_{[a,c]}, f^2|_{[c,b]}$
are topologically mixing. In addition,  $c$ is the unique fixed point
of $f$.
\end{itemize}
\end{theo}

\begin{theo}\label{theo:summary-mixing}
Let $f\colon [a,b]\to [a,b]$ be an interval map. The following
properties are equivalent:
\begin{itemize}
\item $f$ is transitive and has a periodic point of odd period
different from $1$.
\item $f^2$ is transitive.
\item $f$ is totally transitive.
\item $f$ is topologically weakly mixing.
\item $f$ is topologically mixing.
\item For all $\eps>0$ and all non degenerate intervals $J$, there
exists an integer $N$ such that $f^n(J)\supset [a+\eps,b-\eps]$
for all $n\ge N$.
\end{itemize}
\end{theo}

\begin{ex}\label{ex:transitive-not-mixing}
We give an example of a transitive, non topologically mixing interval map.
The map $S\colon [-1,1]\to [-1,1]$, represented in
Figure~\ref{fig:transitive-not-mixing}, is defined by:
$$
\left\{\begin{array}{ll}
\forall x\in [-1,-\frac12],&S(x):=2x+2,\\ 
\forall x\in [-\frac12,0],&S(x):=-2x,\\ 
\forall x\in[0,1],&S(x):=-x.
\end{array}\right.
$$
\begin{figure}[htb]
\centerline{\includegraphics{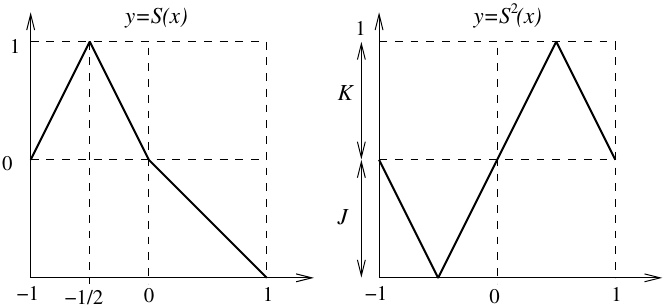}}
\caption{The map $S$ is transitive but not topologically mixing
because $S^2$ is not transitive.}
\label{fig:transitive-not-mixing}
\end{figure}

We set $J:=[-1,0]$ and $K:=[0,1]$.
We have $S(J)=K$ and $S(K)=J$, which implies that $S$ is not topologically
mixing. Since $S^2|_K$ is equal to the
tent map $T_2$ defined in Example~\ref{ex:tent-map}, the map $S^2|_K$ is 
topologically mixing and, for every non degenerate subinterval $U\subset K$, 
there
exists $n\ge 0$ such that $S^{2n}(U)=K$. The map $S^2|_J$ is similar 
to $S^2|_K$ except that its graph is upside down. Therefore, if $U$ is a
nonempty open set, then
\begin{itemize}
\item either $U\cap J$ contains a non degenerate interval, 
and there exists $n\ge 0$ such that $S^{2n}(U)\supset J$, 
\item
or $U\cap K$ contains a non degenerate interval, 
and there exists $n\ge 0$ such that $S^{2n}(U)\supset K$. 
\end{itemize}
In both cases, there exists $n\ge 0$ such that $S^n(U)\cup S^{n+1}(U)=[-1,1]$, 
which implies that $S$ is transitive.
\end{ex}
%*********************************************
\subsection*{Remarks on graph maps}

Rotations are important examples because they
exhibit behaviors that cannot appear for interval maps.

\begin{defi}
Let $\IS=\IR/\IZ$. The \emph{rotation}\index{rotation}
\label{notation:Ralpha}
of angle $\alpha\in\IR$ is the circle map:
$$
R_\alpha\colon \begin{array}[t]{rcl}\IS&\longrightarrow& \IS\\
x&\longmapsto& x+\alpha\bmod\ 1\end{array}
$$
\end{defi}

If $\alpha=\frac pq$ with $p\in\IZ, q\in\IN$ and $\gcd(p,q)=1$, it is 
clear that all points in $\IS$ are periodic of period $q$, and $R_\alpha^q$
is the identity map.
On the contrary, if $\alpha\notin\IQ$ (in this case, $R_\alpha$ is
called an \emph{irrational rotation}\index{irrational rotation}), one can
show that $R_\alpha$ is totally transitive but not topologically
weakly mixing, and there is no periodic point.
Consequently, if one wants to generalize the results concerning 
transitive interval maps, the case of irrational rotations must be excluded.
We shall see that the results of this section extend fairly well to transitive
graph maps, irrational rotations being the only exceptions up to conjugacy.

The next result is due to Blokh \cite{Blo9} (see \cite[Theorem S, p. 506]{Blo11} for a statement in English).

\begin{theo}\label{theo:transitive-noperiodicpoint}
A transitive graph map with no periodic point is conjugate
to an irrational rotation on the circle.
\end{theo}

For graph maps, transitivity is still close to total transitivity, which is
equivalent to topological mixing unless for irrational rotations.
The next two theorems generalize
Propositions \ref{prop:transitivity-total-transitivity} and
\ref{prop:total-transitivity-mixing}; they are due to Blokh \cite{Blo9}
(see \cite{Blo14} for a statement in English).

\begin{theo}\label{theo:totallytransitivegraph-mixing}
Let $f\colon G\to G$ be a totally transitive graph map. If $f$ is not 
conjugate to an irrational rotation, it is topologically mixing.
\end{theo}

\begin{theo}\label{theo:transitivity-total-transitivityG}
Let $f\colon G \to G$ be a transitive graph map. Then there exist
a cycle of graphs $(G_1,\ldots, G_p)$
such that $f^p|_{G_i}$ is totally transitive
for all $i\in\Lbrack 1,p\Rbrack$.
\end{theo}

Alsedà, del Río and Rodríguez proved that a splitting close to the
preceding theorem holds in a broader situation \cite{ADR3}. More precisely,
if $(X,f)$ is a topological dynamical system where $X$ is locally connected,
then, either $f$ is totally transitive, or there exist
$k\ge 2$ and closed subsets $X_1,\ldots, X_k$ with disjoint
interiors, whose union is $X$, such that
$f(X_i)=X_{i+1\bmod k}$ and $f^k|_{X_i}$ is transitive
for all $i\in\Lbrack 1,k\Rbrack$.
Then they showed that a graph map has a splitting of maximal
cardinality (bounded by combinatorial data of the graph), which implies
Theorem~\ref{theo:transitivity-total-transitivityG}.

%**********************************************************************
\section{Accessible endpoints and mixing}

An interval map $f\colon [a,b]\to [a,b]$ is topologically
mixing if the iterates of every
non degenerate interval $J$ eventually cover ``almost all'' $[a,b]$
(to be precise, if $f^n(J)\supset [a+\eps,b-\eps]$ for all $\eps>0$ and all
large  enough $n$). When do the iterates of every non degenerate interval 
eventually cover the whole interval $[a,b]$? Blokh showed that this property 
holds if and only if  the two endpoints of $[a,b]$ are \emph{accessible}
(see Definition~\ref{def:accessible} below).
Proposition~\ref{prop:accessibility} and
Lemma~\ref{lem:accessibility} about (non) accessible points are stated in \cite{Blo6} (see \cite{Blo2} for a published paper).

\begin{defi}[locally eventually onto]\index{locally eventually onto}\index{leo}
A topological dynamical system $(X,f)$ is \emph{locally eventually onto}
(or \emph{leo}) if, for every nonempty open set $U\subset X$,
there exists an integer $N$ such that $f^n(U)=X$ for all $n\ge N$.
\end{defi}

A dynamical system with this property is trivially topologically mixing.

\begin{rem}
In the literature, the name \emph{topologically exact}\index{topologically 
exact} is synonymous to locally eventually onto.
\end{rem}

Alternative definitions of locally eventually onto appear in the 
literature. They are equivalent according to the next lemma.

\begin{lem}
Let $(X,f)$ be a topological dynamical system. The following two
properties are equivalent:
\begin{enumerate}
\item $(X,f)$ is locally eventually onto.
\item For every nonempty open set $U$, there exists 
$n\ge 0$ such that  $f^n(U)=X$.
\end{enumerate}
An interval map $f\colon I\to I$ is locally eventually onto if and only
if
\begin{itemize}
\item[iii)]
$\forall \eps>0,\ \exists M\ge 0, \forall J\text{ subinterval of $I$},\ 
|J|>\eps\Rightarrow \forall n\ge M,\ f^n(J)=I.$
\end{itemize}
\end{lem}

\begin{proof}
The implication (i)$\Rightarrow$(ii) is trivial.
Suppose that (ii) holds and let $U$ be a nonempty open set. There
exists an integer $n$ such that $f^n(U)=X$. This implies that $f$ is onto
and thus, $\forall m\ge n$, $f^m(U)=f^{m-n}(X)=X$. Hence (ii)$\Rightarrow$(i).

\medskip
Let $f\colon I\to I$ be an interval map.
If (iii) holds, then $f$ is locally eventually onto because
every nonempty open set $U$ contains an interval $J$ with $|J|>0$.
Now we assume that $f$ is locally eventually onto.
Let $\eps>0$. We write $I=[a,b]$ and we choose an integer $k\ge 1$ such that
$\frac{b-a}{k}<\frac{\eps}{2}$. For all $i\in\Lbrack 0,k-1\Rbrack$, we set
$$
J_i:=\left(a+\frac ik (b-a),a+\frac{i+1}k (b-a)\right).
$$
For every $i\in\Lbrack 0, k-1\Rbrack$, we choose an integer $N_i$ such that
$f^n(J_i)=I$ for all $n\ge N_i$, and we set 
$M:=\max\{N_i\mid i\in\Lbrack 0, k-1\Rbrack\}$. 
If $J$ is a subinterval of $I$ with $|J|>\eps$, then $J$
contains $J_i$ for some $i\in\Lbrack 0,k-1\Rbrack$. 
Thus $f^n(J)=I$ for all $n\ge M$, that is, (iii) holds.
\end{proof}

\begin{defi}[accessible endpoint]\label{def:accessible}\index{accessible endpoint}
Let $f\colon [a,b]\to [a,b]$ be an interval map. 
The endpoint $a$ (resp. $b$) is \emph{accessible} if there exist $x\in
(a,b)$ and $n\ge 1$ such that $f^n(x)=a$ (resp. $f^n(x)=b$).
\end{defi}

\begin{prop}\label{prop:accessibility}
Let $f\colon [a,b]\to [a,b]$ be a topologically mixing interval map.
Then $f$ is locally eventually onto if and only if
both  $a$ and $b$ are accessible.

More precisely, for every $\eps>0$ and every non degenerate subinterval
$J\subset (a,b)$,
there exists $N$ such that $f^n(J)$ contains $[a, b-\eps]$
(resp. $[a+\eps,b]$) for all $n\ge N$ 
if and only if $a$ (resp. $b$) is accessible.
\end{prop}

\begin{proof}
We show the second part of the proposition; the
first statement follows trivially.

First we suppose that $a$ is accessible. Let $x_0\in (a,b)$
and $n_0\ge 1$ be such that $f^{n_0}(x_0)=a$.
Let $\eps>0$ be such that $x_0\in [a+\eps, b-\eps]$. 
Let $J$ be a non degenerate subinterval in $[a,b]$. Since
$f$ is topologically mixing, there exists an integer $N\ge 0$ such that
$f^n(J)\supset [a+\eps,b-\eps]$ for all $n\ge N$. Since $x_0\in 
[a+\eps,b-\eps]$, the intermediate value theorem implies that 
$f^{n+n_0}(J)\supset [a,b-\eps]$ for all $n\ge N$.
Conversely, if $J$ is a subinterval containing neither $a$ nor
$b$ and such that $a\in f^n(J)$ for some integer $n\ge 1$, then
the point $a$ is accessible by definition. This shows that $a$ is
accessible if and only if, for every $\eps>0$ and every non degenerate 
subinterval $J\subset (a,b)$,
there exists $N$ such that $f^n(J)$ contains $[a, b-\eps]$ for all $n\ge N$.
The case of the endpoint $b$ is similar. 
\end{proof}

\begin{rem}\index{strong transitivity}\label{rem:stronglytransitive}
An interval map $f\colon I\to I$ is called \emph{strongly transitive} if,
for every non degenerate subinterval $J$, there exists $N\ge 0$
such that $\bigcup_{n=0}^N f^n(J)=I$. This definition is due to 
Parry \cite{Par2}. 
This notion is very close to the property of being locally eventually onto. 
Indeed, if a topologically mixing map is strongly transitive, then it is locally
eventually onto. Using  Theorem~\ref{theo:summary-transitivity}, one can
 reduce the transitive case to the mixing one and sees that 
a transitive interval map is strongly transitive if and only if
the two endpoints of $I$ are accessible. In this case, for every
non degenerate subinterval $J$, there exists an integer $n\ge 0$ such that
$f^n(J)\cup f^{n+1}(J)=I$.
\end{rem}

The next lemma specifies the behavior of a mixing map near a non
accessible endpoint. Roughly speaking, a mixing map has infinitely many
oscillations in a neighborhood of a non accessible endpoint. 

\begin{lem}\label{lem:accessibility}
Let $f\colon [a,b]\to [a,b]$ be a topologically mixing interval map.
\begin{enumerate}
\item 
If $a$ (resp. $b$) is the unique non accessible endpoint, then
it is a fixed point. 
If both $a$ and $b$ are non accessible then, either $f(a)=a$
and $f(b)=b$, or $f(a)=b$ and $f(b)=a$.
\item
If $a$ (resp. $b$) is a fixed non accessible point,
then there exists a  decreasing (resp. increasing)
sequence of fixed points $(x_n)_{n\ge 0}$ converging to $a$ (resp. $b$).  
Moreover, for all $n\ge 0$, 
$f|_{[x_{n+1},x_n]}$ is not monotone.
\end{enumerate}
\end{lem}

\begin{proof}
i) If $a$ is not accessible, then $a\notin f((a,b))$.
Since $f$ is topologically mixing, it is onto 
(Lemma~\ref{lem:transitivity-semi-open}(ii)). 
Thus, either $f(a)=a$, or $f(b)=a$. If $b$ is
accessible and $f(b)=a$, then $a$ is accessible too. Therefore, if $a$
is the only non accessible endpoint, then $f(a)=a$. Similarly, if $b$ is the 
only non accessible endpoint, then $f(b)=b$. If both $a$ and
$b$ are non accessible then, either $f(a)=a$ and $f(b)=b$, or
$f(a)=b$ and $f(b)=a$.

\medskip
ii) Assume that $a$ is not accessible and that $f(a)=a$ (the case
of $b$ is symmetric). By
definition, $a\notin f((a,b))$. According to (i), if $b$ is not accessible, 
then $f(b)=b$. If $b$ is accessible, then $f(b)\neq a$. In both cases,
$a\notin f((a,b])$.  Let $\eps\in(0, b-a)$. By transitivity,
$f([a,a+\eps])\not\subset [a,a+\eps]$. Thus there exists $y\in
(a,a+\eps]$ such that $f(y)\ge a+\eps$. In particular, $y$ satisfies
$f(y)\ge y$. Suppose that $f(x)\ge x$ for all $x\in [a,y]$. We set
$$
z:=\min\left\{y,\min (f([y,b]))\right\}.
$$
Then $z>a$, $f([z,b])=f([z,y])\cup f([y,b])$, and both $f([z,y])$ and $f([y,b])$
are included in $[z,b]$ by definition of $z$. Hence
$f([z,b])\subset [z,b]$. But this
contradicts the transitivity of $f$. We deduce that there exists $x\in
[a,y]$ such that $f(x)<x$. Thus $f([x,y])\supset
[f(x),f(y)] \supset [x,y]$, and necessarily $x\neq a$. 
By Lemma~\ref{lem:fixed-point}, there
is a fixed point in $[x,y]\subset (a,a+\eps]$.  Since 
$\eps$ can be chosen arbitrarily small, this implies
that there exists a decreasing sequence of fixed points $(x_n)_{n\ge 0}$ with
$\lim_{n\to+\infty} x_n=a$. Moreover, $f|_{[x_{n+1},x_n]}$ is not
monotone, otherwise we would have $f([x_{n+1},x_n])=[x_{n+1},x_n]$, 
which would contradict the transitivity.
\end{proof}

\begin{rem}
In Lemma~\ref{lem:accessibility}, notice that, if $a$ is a non accessible 
endpoint which is not fixed,  then $f^2(a)=a$ by (i), so statement (ii)
holds for the map $f^2$.
\end{rem}

The next result states that the kind of behavior described in 
Lemma~\ref{lem:accessibility}(ii) is impossible if
$f$ is piecewise monotone or $C^1$. The piecewise monotone case can be
found in \cite{CM} (more precisely, Coven and Mulvey proved in \cite{CM} that
a transitive piecewise monotone map is strongly transitive; see
Remark~\ref{rem:stronglytransitive} for the relation between locally
eventually onto and strong transitivity).
Recall that $f$ is piecewise monotone if the interval can be
divided into finitely many subintervals on each of which $f$ is
monotone (see page~\pageref{def:pm}).

\begin{prop}\label{prop:C1-accessible}
Let $f\colon [a,b]\to [a,b]$ be a topologically mixing interval map. If $f$ is 
piecewise monotone or $C^1$, then the two endpoints $a,b$ are accessible,
and thus $f$ is locally eventually onto.
\end{prop}

\begin{proof}
Suppose that $a$ is not accessible. Then
$f^2(a)=a$ by Lemma~\ref{lem:accessibility}(i).
We set $g:=f^2$. The map $g$ is topologically mixing because $f$ is 
topologically mixing.
If $f$ is $C^1$, then $g$ is $C^1$ too. 
The case $g'(a)<0$ is impossible because $g(a)=a$.
If $g'(a)=0$, then there exists $c\in (a,b)$ such that $g(x)<x$ for all
$x\in (a,c)$, which is impossible because $g$ is transitive.
Thus $g'(a)>0$ and $g$ is increasing in a neighborhood of $a$. Similarly,
if $f$ is piecewise monotone, then $g$ is increasing in a
neighborhood of $a$. In both cases, there exists
$c\in(a,b)$ such that $g|_{[a,c]}$ is increasing. But,
according to Lemma~\ref{lem:accessibility}(ii), there exist two
distinct  points $x<y$ in $(a,c)$ such that $g|_{[x,y]}$ is not
monotone, a contradiction. The case when $b$ is not
accessible is similar. We conclude that both $a, b$ are accessible,
and thus $f$ is locally eventually onto by 
Proposition~\ref{prop:accessibility}.
\end{proof}

\begin{rem}
Proposition~\ref{prop:C1-accessible} remains valid under the assumption
that the mixing map $f$ is monotone (or $C^1$) in a neighborhood of the two 
endpoints.
\end{rem}

\begin{ex}\label{ex:non-accessible-endpoints}
We give an example of an interval map $f\colon [0,1]\to [0,1]$ that is
topologically mixing but not locally eventually onto. This example appears in
\cite{BM} to illustrate another property.

Let $(a_n)_{n\in\IZ}$ be a sequence of points in $(0,1)$ such that
$a_n<a_{n+1}$ for all $n\in\IZ$, and
$$
\lim_{n\to-\infty}a_n=0\quad\text{and}\quad \lim_{n\to+\infty}a_n=1.
$$
For all $n\in\IZ$, we set $I_n:=[a_n,a_{n+1}]$ and we define 
$f_n\colon I_n\to I_{n-1}\cup I_n \cup
I_{n+1}$ by $f_n(a_n):=a_n$, $f_n(a_{n+1})=:a_{n+1}$,
$$
f_n\left(\frac{2a_n+a_{n+1}}{3}\right):=a_{n+2},
\quad 
f_n\left(\frac{a_n+2a_{n+1}}{3}\right):=a_{n-1},
$$
and $f_n$ is linear between the points where it has already been defined.
Then we define the map $f\colon [0,1]\to [0,1]$ (see
Figure~\ref{fig:mixing-non-accessible}) by 
\begin{gather*}
f(0):=0, \quad f(1):=1,\\
\forall n\in\IZ,\ \forall x\in I_n,\ f(x):=f_n(x).
\end{gather*}
\begin{figure}[htb]
\centerline{\includegraphics{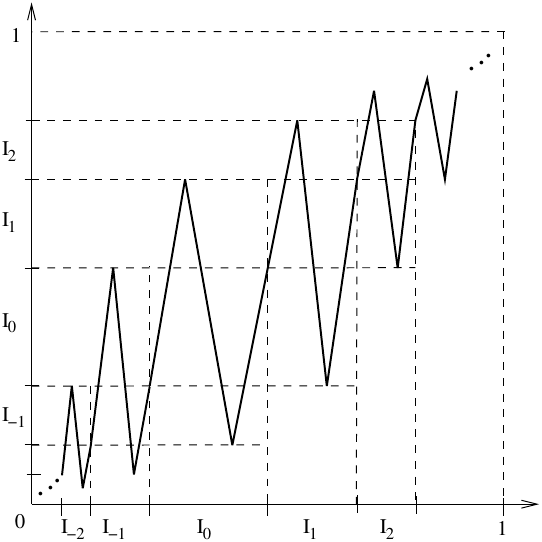}}
\caption{A topologically mixing map on $[0,1]$ whose two endpoints are not 
accessible. For every non degenerate interval $J\subset (0,1)$
and every $n\ge 0$, $f^n(J)\neq [0,1]$.}
\label{fig:mixing-non-accessible}
\end{figure}

It is easy to check that $f$ is continuous and that the points $0$ and
$1$ are not accessible. We are going to show that $f$ is topologically
mixing. Let $J$ be a non degenerate subinterval of $[0,1]$. Since $f$ is
$3$-expanding, we can apply Lemma~\ref{lem:N-critical-points}: 
there exists $n\ge 0$ such that $f^n(J)$
contains two distinct critical points. This implies that $f^{n+1}(J)$ contains
$I_k$ for some $k\in\IZ$. Moreover, it is easy to see that
$$
\forall k\in\IZ,\ \forall m\ge 0,\ f^m(I_k)\supset [a_{k-m},a_{k+m+1}].
$$
Since, for a given $k\in\IZ$, the lengths of $[0,a_{k-m}]$ and
$[a_{k+m+1},1]$ tend to $0$ when $m$ goes to infinity, we deduce that, for all
$\eps>0$, there exists $M$ such that $f^m(J)$ contains
$[\eps,1-\eps]$ for all $m\ge M$. Hence $f$ is topologically mixing.
\end{ex}

%********************************
\subsection*{Remarks on graph maps}

Generalizing the results of this section to graph maps poses no difficulty
provided the notion of non accessible points is extended to points that
may not be endpoints. If $f\colon G\to G$ is a topologically mixing
graph map, a point $x\in G$ is \emph{accessible} if, for every nonempty
open set $U\subset G$, $x\in \bigcup_{n\ge 0} f^n(U)$. 
For graph maps, Lemma~\ref{lem:accessibility}(i) becomes: every
non accessible point is periodic and its orbit is included in the set
of non accessible points. We leave to the reader the ``translation'' of
the other statements. 

The map of 
Example~\ref{ex:non-accessible-endpoints} can be seen as a circle map by
gluing together the endpoints $0$ and $1$, and one gets a
topologically mixing circle map with a
fixed non accessible point (which is obviously not an endpoint).

%**********************************************************************
%**********************************************************************
\section{Sensitivity to initial conditions}

Roughly speaking, sensitivity to initial conditions means that there exist
arbitrarily close points with divergent trajectories. 
We are going to see that, for interval maps,  transitivity implies sensitivity 
and, conversely, sensitivity implies that an iterate of the map is
transitive on a subinterval. This shows that the notions of transitivity and
sensibility are closely related on the interval.
The first implication is quite natural in view
of the fact that transitivity is very close to mixing. The second
implication may seem unexpected.

%**********************************************************************
\subsection{Definitions}
A point $x$ is $\eps$-stable if the trajectories of all points in a
neighborhood of $x$ follow the trajectory of $x$ up to $\eps$, otherwise it is
$\eps$-unstable.  The terminology ``sensitivity to initial
conditions'' was first introduced by Guckenheimer \cite{Guc} to mean
that the $\eps$-unstable points have a positive Lebesgue measure for
some $\eps>0$. We would rather follow  the definition of Devaney \cite{Dev}.

\begin{defi}[unstable point, sensitivity to initial conditions]\label{defi:unstable}
Let $(X,f)$ be a topological dynamical system and $\eps>0$.
A point $x\in X$ is \emph{$\eps$-unstable (in the
sense of Lyapunov)}\index{unstable point}\index{stable point} 
if, for every neighborhood
$U$ of $x$, there exists $y\in U$ and $n\ge 0$ such that
$d(f^n(x),f^n(y))\ge\eps$. The set of $\eps$-unstable points is
denoted by $U_{\eps}(f)$.\index{Ueps@$U_{\eps}(f)$}
\label{notation:Ueps}
A point is \emph{unstable} if it is
$\eps$-unstable for some $\eps>0$.

The map $f$ is \emph{$\eps$-sensitive to initial
conditions}\index{sensitive (to initial conditions)}\index{sensitivity} 
(or more briefly \emph{$\eps$-sensitive}) if $U_{\eps}(f)= X$.
It is \emph{sensitive to initial conditions} if it is $\eps$-sensitive for
some $\eps>0$.
\end{defi}

The next lemma states some basic properties of the sets $U_{\eps}(f)$. The last
assertion gives an equivalent definition for sensibility.

\begin{lem}\label{lem:unstable}
Let $(X,f)$ be a topological dynamical system and $\eps>0$. The
following properties hold:
\begin{enumerate}
\item $\forall n\ge 1$, $U_\eps(f^n)\subset U_\eps(f)$.
\item $\forall n\ge 1$, $\exists \delta>0$, $U_\eps(f)\subset U_\delta(f^n)$.
\item $f(U_{\eps}(f))\subset U_{\eps}(f)$.
\item $\overline{U_{\eps}(f)}\subset U_{\eps/2}(f)$.
\item If $V$ is open and $V\cap U_{\eps}(f)\neq\emptyset$, then
there exists $n\ge 0$ such that $\diam(f^n(V))\ge\eps$.
\item  $f$ is sensitive if and only if there
exists $\delta>0$ such that, for all nonempty open sets $V$, 
there exists $n\ge 0$ such that $\diam(f^n(V))\ge\delta$.
\end{enumerate}
\end{lem}

\begin{proof}
i) Trivial.

ii) The map $f$ is uniformly continuous because $X$ is compact. Thus
\begin{equation}\label{eq:UfUfn}
\exists\delta>0,\ d(x,y)<\delta\Rightarrow \forall i\in\Lbrack 0, n-1\Rbrack,\
d(f^i(x),f^i(y))< \eps.
\end{equation}
Let $x\notin U_\delta(f^n)$. Then there exists a neighborhood $U$ of $x$
such that, for all $y\in U$, $\forall k\ge 0$, $d(f^{kn}(x),f^{kn}(y))<\delta$.
Then \eqref{eq:UfUfn} implies $d(f^{kn+i}(x),f^{kn+i}(y))<\eps$ for
all $k\ge 0$ and all $i\in\Lbrack 0, n-1\Rbrack$. We deduce that $x\notin
U_\eps(f)$. This shows that $U_\eps(f)\subset U_\delta(f^n)$.

iii) Let $x\in U_{\eps}(f)$ and let $V$ be a neighborhood of
$x$. We first show that
\begin{equation}\label{eq:unstable-many-n}
\text{there are infinitely many }n\in\IN\text{ such that }
\exists y\in V,\ d(f^n(x),f^n(y))\ge\eps.
\end{equation}
Suppose on the contrary that there exists $n_0$ such that, 
if $d(f^n(x),f^n(y))\ge\eps$ for some $y\in V$ and $n\ge 0$, then $n\le n_0$.
By the continuity of the maps $f,f^2,\ldots, f^{n_0}$, 
there exists $\delta>0$ such that
\begin{equation}\label{eq:fkxy}
\forall y\in X,\ d(x,y)<\delta\Rightarrow
\forall k\in\Lbrack 0,n_0\Rbrack,\ d(f^k(x),f^k(y))<\eps.
\end{equation} 
The set $W:=V\cap B(x,\delta)$ is a neighborhood of $x$. Let 
$y\in W$ and $n\ge 0$. If $n\le n_0$, then
$d(f^n(x),f^n(y))<\eps$ by \eqref{eq:fkxy}. If $n> n_0$, then
$d(f^n(x),f^n(y))<\eps$ according to the choice of $n_0$. This 
contradicts the fact that $x$ is $\eps$-unstable. Hence
\eqref{eq:unstable-many-n} holds.

Now we consider an open set $V$ containing $f(x)$. Since
$U:=f^{-1}(V)$ is open and contains $x$, what precedes implies that
there exist $y\in U$ and
$n\ge 2$ such that $d(f^n(x),f^n(y))\ge\eps$. We set $z:=f(y)$. Then
$z$ belongs to $V$ and $d(f^{n-1}(f(x)),f^{n-1}(z))\ge\eps$. Thus
$f(x)\in U_{\eps}(f)$.

iv) We fix $x\in \overline{U_{\eps}(f)}$. Let $V$ be an
open set containing $x$. There exists a point $y\in U_{\eps}(f)\cap V$, and thus,
by definition, there exist $z\in V$ and $n\ge 0$ such that
$d(f^n(y),f^n(z))\ge\eps$.  By the triangular inequality, we have
either $d(f^n(x),f^n(y))\ge\eps/2$, or $d(f^n(x),f^n(z))\ge\eps/2$.
We deduce that $x\in U_{\eps/2}(f)$.

v) Let $V$ be an open set such that $V\cap
U_{\eps}(f)\neq\emptyset$.
By definition, there exist $x\in V\cap U_{\eps}(f)$,  
$y\in V$ and $n\ge 0$ such that
$d(f^n(x),f^n(y))\ge\eps$, that is, $\diam(f^n(V))\ge\eps$.

vi)
First we assume that $f$ is $\eps$-sensitive, that is,
$U_\eps(f)=X$. By (v), for every nonempty open set $V$, 
there exists $n\ge 0$ such that $\diam(f^n(V))\ge\eps$.

Now we suppose that there exists $\delta>0$ such that,
for every nonempty open set $V$, there exists
$n\ge 0$ such that $\diam(f^n(V))\ge\delta$. We fix $\eps\in (0,\delta/2)$.
Let $x\in X$. Let $V$ be an open set containing $x$ and  let
$n\ge 0$ be such that $\diam(f^n(V))\ge\delta$. Thus
there exist two points $y,z\in V$ such that $d(f^n(y),f^n(z))\ge \delta\ge
2\eps$. The triangular inequality implies that,
either $(f^n(x),f^n(y))\ge \eps$, or $(f^n(x),f^n(z))\ge \eps$. Hence
$x\in U_\eps(f)$, and the map $f$ is $\eps$-sensitive.
\end{proof}

%**********************************************************************
\subsection{Sensitivity and transitivity}

Barge and Martin proved that, for a transitive interval
map, every point $x$ is $\eps$-unstable for some $\eps$ depending on
$x$ \cite{BM2}.  We give a different proof, which additionally shows that the
constant of instability can be taken uniform for all points $x$.

\begin{prop}\label{prop:transitivity-sensitivity}
Let $f\colon I\to I$ be an interval map. 
\begin{itemize}
\item
If $f$ is topologically mixing, then $f$ is 
$\delta$-sensitive for all $\delta\in(0,\frac{|I|}2)$.
\item
If $f$ is transitive, then $f$ is  $\delta$-sensitive for all
$\delta\in(0,\frac{|I|}4)$.
\end{itemize}
\end{prop}

\begin{proof}
We write $I=[a,b]$.
First we assume that $f$ is topologically mixing.  
Let $\eps\in (0,\frac{|I|}2)$,
$x\in [a,b]$ and $U$ be a neighborhood of $x$. By
Theorem~\ref{theo:summary-mixing}, there exists $n\ge 0$ such
that $f^n(U)\supset [a+\eps,b-\eps]$.  Therefore, there exist $y,z$ in
$U$ such that $f^n(y)=a+\eps$ and $f^n(z)=b-\eps$. This implies that
$$
\max\left\{|f^n(x)-f^n(y)|, |f^n(x)-f^n(z)|\right\}\ge
\frac{b-a-2\eps}2=\frac{|I|}{2}-\eps.
$$
Consequently, $x$ is $\delta$-unstable, where $\delta:=\frac{|I|}2-\eps$. 
Since $\eps$ is arbitrary, the map $f$ is $\delta$-sensitive for every 
$\delta\in(0,\frac{|I|}2)$.

Now we suppose that $f$ is transitive but not topologically mixing. 
According to
Theorem~\ref{theo:summary-transitivity}, there exists $c\in (a,b)$
such that $f([a,c])=[c,b]$, $f([c,b])=[a,c]$ and both maps 
$f^2|_{[a,c]}$ and $f^2|_{[c,b]}$ are topologically mixing. What precedes 
implies that $f^2|_{[a,c]}$ is  $\delta$-sensitive
for all $\delta\in(0,\frac{c-a}2)$.
Therefore, according to Lemma~\ref{lem:unstable}(i)-(iii),
we have $[a,c]\subset U_\delta(f)$ and $f([a,c])=[c,b]\subset U_\delta(f)$.
Thus $f$ is $\delta$-sensitive for all $\delta\in(0,\frac{c-a}2)$.
Similarly, $f$ is $\delta$-sensitive for all $\delta\in(0,\frac{b-c}2)$. 
Finally, we conclude that $f$ is 
$\delta$-sensitive for all $\delta\in(0,\frac{|I|}4)$ because
$\max\{c-a,b-c\}\ge\frac{|I|}2$.
\end{proof}

The converse of Proposition~\ref{prop:transitivity-sensitivity} is obviously 
false. However, somewhat surprisingly, a partial converse holds: 
the instability on a subinterval implies the
existence of a transitive cycle of intervals. This result is due to
Blokh; it is stated without proof in \cite{Blo3}
(we do not know any reference for the proof).

\begin{prop}\label{prop:sensitivity-transitive-component}
Let $f$ be an interval map. Suppose that, for some
$\eps>0$, the set of $\eps$-unstable points $U_{\eps}(f)$ has
a nonempty interior. Then there exists a cycle of intervals
$(J_1,\ldots, J_p)$ such that $f|_{J_1\cup\cdots \cup J_p}$ is transitive.
Moreover, $J_1\cup\cdots \cup J_p\subset
\overline{U_{\eps}(f)}$ and there exists $i\in\Lbrack 1,p\Rbrack$ such that
$|J_i|\ge\eps$.
\end{prop}

\begin{proof}
We consider the family of sets
$$
\CF:=\{Y\subset\overline{U_{\eps}(f)}\mid Y\text{ closed},\ f(Y)\subset
Y,\ \Int{Y}\neq\emptyset\}.
$$
By assumption, there exists a nonempty open interval $K\subset
U_{\eps}(f)$. Moreover,  $f^n(K)\subset
U_{\eps}(f)$ for all $n\ge 0$ by Lemma~\ref{lem:unstable}(iii). The set
$\overline{\bigcup_{n\ge 0} f^n(K)}$ is thus an element of $\CF$, and hence
$\CF\neq\emptyset$. Let $Y$ belong to $\CF$ and let $J$ be a non
degenerate interval included in $Y$. Since  $\Int{J}\cap
U_{\eps}(f)\neq\emptyset$,  there
exists $n\ge 0$ such that $|f^n(J)|\ge\eps$ by Lemma~\ref{lem:unstable}(v). 
This implies that 
\begin{equation}\label{eq:F-eps}
\text{every }Y\in \CF\text{ has
a connected component }C\text{ with }|C|\ge\eps.
\end{equation}
We endow $\CF$ with the partial order given by inclusion. We are first going
to show that every totally ordered family of elements of $\CF$ admits
a lower bound in $\CF$. Let $(Y_{\lambda})_{\lambda\in\Lambda}$ be a family of
elements of $\CF$ which is totally ordered (that is, all the elements
of $\Lambda$ are comparable and $Y_{\lambda}\subset Y_{\lambda'}$ if 
$\lambda\le \lambda'$). We set
$$
Y:=\bigcap_{\lambda\in\Lambda} Y_{\lambda}.
$$
Then $Y$ is a closed set, $f(Y)\subset Y$ and $Y\subset
\overline{U_{\eps}(f)}$.  Moreover, each $Y_{\lambda}$ has a
finite non zero number of connected components of length at least
$\eps$, and thus so has $Y$.
Therefore, $\Int{Y}\neq\emptyset$, so $Y\in\CF$ and $Y$ is a lower
bound for $(Y_{\lambda})_{\lambda\in\Lambda}$. Zorn's Lemma then implies that $\CF$ 
admits at least one minimal element, say $Z$.

We now turn to prove that $f|_Z$ is transitive. The set $Z$ has finitely
or countably many non degenerate connected components, and
at least one of them has a length greater than or equal to $\eps$
by \eqref{eq:F-eps}. Let $(I_i)_{i\ge 1}$ be the (finite or infinite)
family of all 
non degenerate connected components of $Z$, 
where $I_1,\ldots, I_k$ are the connected components
of length at least $\eps$ (for some $k\ge 1$). Let $i\ge 1$. Since
$\Int{I_i}\cap U_{\eps}(f) \neq\emptyset$,
there exists $n_i\ge 1$ such that $|f^{n_i}(I_i)|\ge\eps$ by
Lemma~\ref{lem:unstable}(v).  Therefore, there exists
$\tau_i\in\Lbrack 1,k\Rbrack$  such that $f^{n_i}(I_i) \subset
I_{\tau_i}$. Since $\Lbrack 1,k\Rbrack$ is finite, this implies that 
there exist integers
$j\in\Lbrack 1,k\Rbrack$ and $m\ge 1$ such that $f^m(I_j)\subset
I_j$. The set $Z':=\bigcup_{n=0}^m f^n(I_j)$ obviously
belongs to $\CF$. Thus $Z'=Z$ by minimality, that is, $Z$ has finitely many
connected components which are cyclically mapped under $f$.
We call $J_1,\ldots, J_p$ the connected components of $Z$, labeled 
in such a way that $f(J_i)\subset J_{i+1}$ for all $i\in\Lbrack 1,p-1\Rbrack$
and $f(J_p)\subset J_1$. By minimality of $Z$, 
these inclusions are actually equalities, that is, $(J_1,\ldots, J_p)$ is
a cycle of intervals (note that $J_1,\ldots, J_k$ are closed).
If $f|_Z$ is not transitive, there exist
two open sets $U,V$ such that
$$
U\cap Z\neq\emptyset,\ V\cap Z\neq\emptyset\quad \text{and}\quad \forall
n\ge 0,\ f^n(U\cap Z)\cap (V\cap Z)=\emptyset.
$$
Since $Z$ is the union of finitely many non degenerate intervals,
there exists a nonempty open interval $J\subset U\cap Z$. We set
$$
X:=\overline{\bigcup_{n\ge 0} f^n(J)}.
$$
Then $X$ belongs to $\CF$ and $X\subset Z$, but $X\cap V=\emptyset$, and thus
$X\neq Z$. This contradicts the fact that $Z$ is minimal.  We conclude
that $f|_Z$ is transitive.
\end{proof}

\begin{ex}
For a given $\eps>0$, the number of transitive cycles of intervals
$(J_1,\ldots, J_p)$ given by
Proposition~\ref{prop:sensitivity-transitive-component} is finite
because one of these intervals has a length at least $\eps$ and 
two different cycles have disjoint interiors by transitivity. Nevertheless,
infinitely many transitive cycles of intervals can coexist if their
constants of sensitivity tend to $0$, as illustrated in
Figure~\ref{fig:infinitely-many-transitive-cycles}.
\begin{figure}[htb]
\centerline{\includegraphics{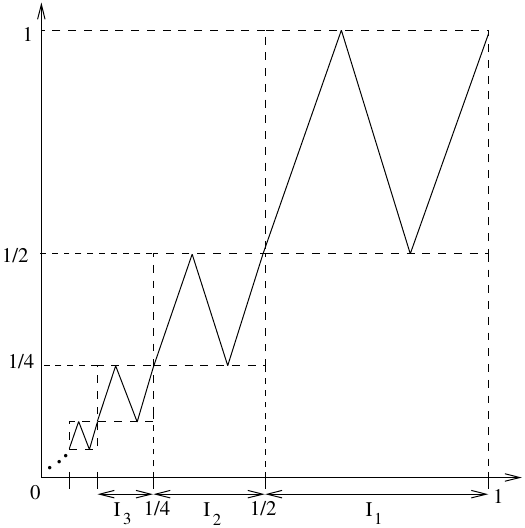}}
\caption{An interval map $f$ with infinitely many transitive subintervals
$(I_n)_{n\ge 1}$, where
$I_n:=\left[\frac{1}{2^n},\frac{1}{2^{n-1}}\right]$. 
The map $f|_{I_n}$ is equal to the map $T_3$ of
Example~\ref{ex:tent-map}, up to a rescaling.
It is easy to show that $f|_{I_n}$ is $\eps_n$-sensitive with 
$\eps_n:=\frac{1}{2^{n+1}}$ for all $n\ge 1$.}
\label{fig:infinitely-many-transitive-cycles}
\end{figure}
\end{ex}

\begin{ex}\label{ex:sensitive-not-transitive}
Even if there is $\eps>0$ such that all points are $\eps$-unstable, 
the union of all  
transitive cycles of intervals is not necessarily dense.  In order to illustrate
this fact, we are going to  build a sensitive interval map which admits a 
transitive cycle of $p+1$ intervals and no other transitive cycle of intervals.

We fix an integer $p\ge 1$ and we set 
$$
\forall i\in\Lbrack 0,2p+1\Rbrack,\ x_i:=\frac{i}{2p+1}\quad\text{and}\quad
\forall i\in\Lbrack 0,2p\Rbrack,\ J_i:=[x_i,x_{i+1}].
$$
We define the continuous map $f\colon [0,1]\to [0,1]$ by
\begin{gather*}
f(x_0):=x_3,\quad f\left(\frac{x_0+x_1}{2}\right):=x_2,\quad f(x_1):=x_3,\\
f(x_{2p-1}):=x_{2p+1},\quad f(x_{2p}):=x_0,\quad f(x_{2p+1}):=x_1,
\end{gather*}
and $f$ is linear between the points where it has already been defined (see
Figure~\ref{fig:sensitive-one-cycle}).  
\begin{figure}[htb]
\centerline{\includegraphics{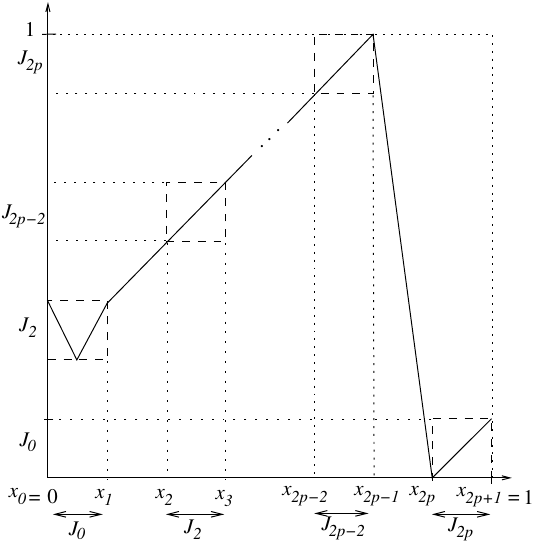}}
\caption{For this interval map, all points are $\eps$-unstable with
$\eps:= \frac{1}{8(2p+1)}$, but the system admits a single transitive
cycle of intervals $(J_0, J_2,\ldots, J_{2p-2}, J_{2p})$, where
$J_i:=\left[\frac{i}{2p+1},\frac{i+1}{2p+1}\right]$.}
\label{fig:sensitive-one-cycle}
\end{figure}
Note that $f$ is of slope $1$ on $[x_1,x_{2p-1}]$ and $f(x_i)= x_{i+2}$ 
for all $i\in\Lbrack 1,2p-1\Rbrack$.
It is trivial to see that $(J_0,J_2,\ldots, J_{2i},\ldots,J_{2p})$ is a cycle of
intervals. Moreover, $f^{p+1}|_{J_0}$ is the map $T_2$ defined
in Example~\ref{ex:tent-map} except that its graph is upside
down (and rescaled). Since $T_2$ is topologically mixing, so is 
$f^{p+1}|_{J_0}$. By Proposition~\ref{prop:transitivity-sensitivity},
$f^{p+1}|_{J_0}$ is $\delta$-sensitive for $\delta:=\frac{1}{2(2p+1)}$. 
Therefore, $C:=J_0\cup J_2\cup\cdots\cup J_{2p}$ is a transitive cycle of $p+1$
intervals and $C\subset U_{\delta}(f)$.

We now turn to show that $f$ is $\frac{\delta}4$-sensitive.  Let $x\in
[0,1]$. Suppose that there exists $n\ge 0$ such that $f^n(x)$ belongs to $C$. 
The form of $C$ implies that there exists $\eps_0>0$ such that either $f^n([x,x+\eps_0])\subset C$ or
$f^n([x-\eps_0,x])\subset C$. Let  $\eps\in(0, \eps_0]$.  The image of
a non degenerate interval is non degenerate and  $C\subset
U_{\delta}(f)$. Therefore, there exist two distinct points $y,z\in
[x-\eps,x+\eps]$ and an integer $k\ge n$ such that
$|f^k(y)-f^k(z)|\ge\delta$. By the triangular inequality, 
either $|f^k(x)-f^k(y)|\ge \frac{\delta}2$ or $|f^k(x)-f^k(z)|\ge 
\frac{\delta}2$. 
In other words:
\begin{equation}\label{eq:sensitive-not-transitive2}
\bigcup_{n\ge 0} f^{-n}(C)\subset U_{\frac{\delta}2}(f).
\end{equation}
It remains to consider the points whose orbit does not meet $C$. These points
are included in the set
\begin{eqnarray*}
X&:=&\{x\in[0,1]\mid \forall n\ge 0, f^n(x)\in J_1\cup
J_3\cup\cdots\cup J_{2p-1}\}\\ &=& \bigcap_{n\ge 0}
f^{-n}\left(\bigcup_{i=0}^{p-1} J_{2i+1}\right).
\end{eqnarray*}
We are going to show that $X$ is a Cantor set (see Section~\ref{sec:cantor}
in Appendix for the definition of a Cantor set).  Note that
\begin{equation}\label{eq:sensitive-not-transitive1}
\forall i\in\Lbrack 0,p-2\Rbrack,\ f|_{J_{2i+1}}\colon J_{2i+1}\to J_{2i+3}\text{ is a linear
homeomorphism}.
\end{equation}
This fact implies that it is enough to focus on the set $X\cap J_{2p-1}$. The map
$f|_{J_{2p-1}}\colon J_{2p-1} \to [0,1]$  is a linear homeomorphism,
and thus the set of points $x\in J_{2p-1}$ such that $f(x)\in J_1\cup
J_3\cup \cdots\cup J_{2p-1}$ is the union of $p$ disjoint closed
subintervals of equal lengths,  
which are $K_i:=J_{2p-1}\cap f^{-1}(J_{2i+1})$, $0\le
i\le p-1$.  Using \eqref{eq:sensitive-not-transitive1}, we can
see that the map $f^{p-i}|_{K_i}\colon K_i\to [0,1]$ is a linear
homeomorphism for every  $i\in\Lbrack 0,p-1\Rbrack$, and thus 
$\{x\in K_i\mid f^{p-i+1}(x)\in J_1\cup J_3\cup \cdots \cup J_{2p-1}\}$ 
is the union of $p$ disjoint closed subintervals of equal lengths. 
Applying this argument inductively, 
we can show that $X$ is a
Cantor set. In particular, the interior of $X$ is empty, and thus
$\overline{[0,1]\setminus X}=[0,1]$. According to
\eqref{eq:sensitive-not-transitive2}, $[0,1]\setminus X$ is
included in $U_{\frac{\delta}2}(f)$. Then $f$ is $\frac{\delta}4$-sensitive by
Lemma~\ref{lem:unstable}(iv).

Finally, we show that $C$ is the only transitive cycle of intervals of $f$.
Suppose on the contrary that there exists another transitive cycle of 
intervals $C'$. 
By transitivity, the interiors of
$C$ and $C'$ are disjoint, and thus there exists a non
degenerate subinterval $J\subset [0,1]$ such that $f^n(J)\cap \Int{C}=
\emptyset$ for all $n\ge 0$.  This exactly means that $J\subset X$, 
but this is impossible because $X$ is a Cantor set.
\end{ex}

%*********************************************
\subsection*{Remarks on graph maps and general dynamical systems}

The first part of the proof of Proposition~\ref{prop:transitivity-sensitivity}
can be easily adapted to show the following result.

\begin{prop}\label{prop:mixing-sensitive}
Let $(X,f)$ be a topologically mixing dynamical system. Then
$f$ is  $\delta$-sensitive for every $\delta\in\left(0,\frac{\diam(X)}2\right)$.
\end{prop}

On the other hand, the next result is proved in \cite{BBCDS} and \cite{GW}.

\begin{theo}\label{theo:transitive+periodicpoints-sensitive}
Let $(X,f)$ be a transitive dynamical system having a dense set of
periodic points. Then $f$ is sensitive to initial conditions provided $X$ is 
infinite.
\end{theo}

For interval maps, Proposition~\ref{prop:transitivity-sensitivity} can be 
derived from Proposition~\ref{prop:mixing-sensitive} using the fact that, on 
the interval,  transitivity is close to mixing; or it can be seen
as a particular case of Theorem~\ref{theo:transitive+periodicpoints-sensitive}
since a transitive interval map has
a dense set of periodic points by 
Proposition~\ref{prop:transitivity-periodic-points}.

The next theorem is a straightforward consequence of a result of Blokh
\cite{Blo9} (see \cite[Theorem 1]{Blo14} for a statement in English).
Alsedà, Kolyada, Llibre and Snoha showed a related
result in a broader context: if $(X,f)$ is a transitive topological 
dynamical system and if $X$ has a disconnecting interval (that is, a
subset $I\subset X$ homeomorphic to $(0,1)$ such that, $\forall x\in I$, 
$X\setminus \{x\}$ is not connected), then the set of periodic points is dense
\cite[Theorem 1.1]{AKLS}.

\begin{theo}\label{theo:transitivegraphmap-rotation}
Let $f\colon G\to G$ be a transitive graph map. If $f$ is not conjugate to
an irrational rotation, the set of periodic points is dense.
\end{theo}

Together with Theorem~\ref{theo:transitive-noperiodicpoint},
this implies that for a transitive graph map $f$, either $f$ is
conjugate to an irrational rotation and has no periodic point, or $f$ has
a dense set of periodic points.

The previous theorem, combined with
Theorem~\ref{theo:transitive+periodicpoints-sensitive},
implies the next result.

\begin{cor}\label{cor:transitive-sensitiveG}
Let $f\colon G\to G$ be a transitive graph map. If $f$ is not conjugate to
an irrational rotation, it is sensitive to initial conditions.
\end{cor}

It is not difficult to extend the proof of 
Proposition~\ref{prop:sensitivity-transitive-component} to graph maps,
which leads to the following result.

\begin{prop}
Let $f\colon G\to G$ be a graph map. Suppose that, for some
$\eps>0$, the set of $\eps$-unstable points $U_{\eps}(f)$ has
a nonempty interior. Then there exists a cycle of 
graphs $(G_1,\ldots, G_p)$ such that $f|_{G_1\cup\cdots \cup G_p}$ is 
transitive. Moreover, $G_1\cup\cdots \cup G_p\subset
\overline{U_{\eps}(f)}$ and there exists $i\in\Lbrack 1,p\Rbrack$ such that
$\diam(G_i)\ge\eps$.
\end{prop}

Since a rotation is not sensitive, Theorems 
\ref{theo:totallytransitivegraph-mixing} and
\ref{theo:transitivity-total-transitivityG} imply that, in
the previous proposition, each subgraph $G_i$ can be 
decomposed in subgraphs $(H_i)_{1\le i\le k_i}$ that are cyclically
mapped under $f^p$ and are topologically mixing for $f^{pk_i}$.
A similar result is stated in \cite[Theorem~4.4]{Kato}.

%**********************************************************************
%Periodic points
\chapter{Periodic points}\label{chap:periodic-points}

%**********************************************************************
\section{Specification}

We saw that, for a transitive interval map, the set of periodic points is dense
(Proposition~\ref{prop:transitivity-periodic-points}). We are going to
see that if in addition
the map is mixing, then it satisfies the \emph{specification property},
which roughly means that one can approximate any finite collection of pieces of
trajectories by a periodic orbit provided enough time is left to pass from
a piece of trajectory to another.
This result was stated by Blokh, without proof in \cite{Blo1}; see \cite{Blo2} for the proof.

Specification is a strong property. In particular, 
a topological dynamical system $(X,f)$ with the specification property is 
topologically mixing \cite[(21.3)]{DGS}.
Therefore, specification and topological mixing are equivalent for interval 
maps.

\begin{defi}[specification\index{specification}]
Let $(X,f)$ be a topological dynamical system. 
The map $f$ has the \emph{specification property}
if the following property holds:
for all $\eps>0$, there exists an integer $N\ge 1$ such that, for all $p\ge 1$,
for all points $x_1,\ldots, x_p\in X$ and all positive integers $m_i,n_i$, 
$i=1,\ldots, p$, satisfying 
$$
m_1\le n_1<m_2\le n_2<\cdots<m_p\le n_p\quad\text{and}\quad
\forall i\in\Lbrack 2, p\Rbrack ,\ m_i-n_{i-1}\ge N,
$$
then, for all integers $q\ge N
+n_p-m_1$, there exists a point $x\in X$ such that
$$
f^q(x)=x\quad\text{and}\quad
\forall i\in\Lbrack 1,p\Rbrack,\ \forall k\in \Lbrack m_i,n_i\Rbrack,\ d(f^k(x),f^k(x_i))\le \eps.
$$
\end{defi}

We first state two lemmas.

\begin{lem}\label{lem:speci1}
Let $f\colon I\to I$ be an interval map and $0<\eps<\frac{|I|}2$. 
For all $x\in I$ and all integers $n\ge 0$, there exist closed 
subintervals $J_0,\ldots, J_n$ in $I$ such that:
\begin{itemize}
\item $\forall i\in\Lbrack 0,n-1\Rbrack$, $f(J_i)=J_{i+1}$,
\item $\forall i\in\Lbrack 0, n\Rbrack$, $f^i(x)\in J_i$ and
$J_i\subset [f^i(x)-\eps,f^i(x)+\eps]$,
\item there exists $i\in\Lbrack 0,n\Rbrack$ such that $J_i$
contains either $f^i(x)-\eps$ or $f^i(y)+\eps$.
\end{itemize}
Moreover, if $x+\eps\in I$ (resp. $x-\eps\in I$), then $J_0$ can be chosen
in such a way that $J_0$ is included in $[x,x+\eps]$ (resp. $[x-\eps,x]$).
\end{lem}

\begin{proof}
We fix $x\in I$ and we set $x_k:=f^k(x)$ for all $k\ge 0$. 
We show the lemma by induction on $n$.

$\bullet$ Case $n=0$: since $\eps<\frac{|I|}2$, the interval $I$
contains either $x-\eps$ or $x+\eps$. We can set 
$J_0:=[x,x+\eps]$ if $x+\eps\in I$, or
$J_0:=[x,x+\eps]$ if $x- \eps\in I$. The interval $J_0$ is
suitable.

$\bullet$ Suppose that the lemma is true at rank $n-1$, and let
$J_0,\ldots, J_{n-1}$ be the subintervals given by the lemma.
If $f(J_{n-1})\subset [x_n-\eps,x_n+\eps]$, we set $J_n:=f(I_{n-1})$ and the
intervals $(J_0,\ldots, J_n)$ are suitable.
From now on, we suppose that $f(J_{n-1})$ is not included
in $[x_n-\eps,x_n+\eps]$. Thus, by connectedness, $f(J_{n-1})$
contains either $x_n-\eps$ or $x_n+\eps$.
We may assume that $J_0\subset [x,x+\eps]$, the case
when $J_0\subset [x-\eps, x]$ being similar. According to the assumption on
$f(J_{n-1})=f^n(J_0)$, we can define
$$
y:=\min\{z\in J_0 \mid f^n(z)\in\{x_n-\eps,x_n+\eps\}\}.
$$
It follows that $f^n([x,y])$ equals either $[x_n-\eps,x_n]$ or
$[x_n,x_n+\eps]$. We set $J_0':=[x,y]$ and $J_i':=f^i(J_0')$
for all $i\in\Lbrack 1,n\Rbrack$. 
The intervals $(J_0',\ldots,J_n')$ are suitable
because $J_i'\subset J_i\subset [x_i-\eps,x_i+\eps]$ for all 
$i\in\Lbrack 0,n-1\Rbrack$
and $J_n'$ contains $x_n$ and one of the points
$x_n-\eps$, $x_n+\eps$. This ends the induction.
\end{proof}

\begin{lem}\label{lem:speci2}
Let $f\colon [a,b]\to [a,b]$ be a topologically mixing interval map 
and $0<\eps<\frac{b-a}2$. Suppose that 
the endpoint $a$ (resp. $b$) is fixed and non accessible. 
Then there exists $\delta\in (0,\eps)$ such that, for all 
$x\in [a,a+\delta]$  (resp. $x\in [b-\delta, b]$)
and all $n\ge 0$, there exist
closed subintervals $J_0,\ldots, J_n$ satisfying:
\begin{itemize}
\item $J_0\subset [a+\delta, b-\delta]$,
\item $\forall i\in\Lbrack 0,n-1\Rbrack$, $f(J_i)=J_{i+1}$,
\item  $\forall i\in\Lbrack 0,n\Rbrack$, 
$J_i\subset [f^i(x)-\eps, f^i(x)+\eps]$,
\item there exists $i\in\Lbrack 0,n\Rbrack$  such that $|J_i|\ge \frac{\eps}4$.
\end{itemize}
\end{lem}

\begin{proof} 
We prove the lemma when $a$ is a non accessible fixed point. The proof for $b$ 
is similar.
If both endpoints are fixed and non accessible, the same $\delta$ can be 
chosen for $a$ and $b$ by taking the minimum of the values 
found for $a$ and $b$ respectively.
By continuity, there exists $\eta>0$ such that
\begin{equation}\label{eq:a-eta}
\forall y\in [a,a+\eta],\ f(y)<a+\eps.
\end{equation} 
By transitivity,
$f([a,a+\eps])$ is not included in $[a,a+\eps]$; that is, there exists 
$z$ in $[a,a+\eps]$ such that $f(z)\ge a+\eps$. In fact, $z\in(a+\eta,a+\eps]$ 
by \eqref{eq:a-eta}. According to Lemma~\ref{lem:accessibility}(ii),
there exists a fixed point $c$ in the interval 
$(a,a+\min \{\eta,\frac{\eps}2\})$.
We set $\delta:=c-a\in (0,\frac{\eps}2]$ and $K:=[c,a+\eps]$. 
The interval $K$ contains both $c$ and 
$z$, and hence $f(K)\supset K$ by the intermediate value theorem.
Notice that $a+\eps<b-\delta$ because $\delta\le\frac{\eps}2<\frac{b-a}4$.

We fix $x\in [a,a+\delta]=[a,c]$ and $n\ge 0$. We set $x_k:=f^k(x)$
for all $k\ge 0$. Let $m\in\Lbrack 0,n\Rbrack$
be the greatest integer such that $x_0,\ldots,x_m\in [a,c]$. Notice
that $K\subset [x_i-\eps,x_i+\eps]$ for all $i\in\Lbrack 0,m\Rbrack$.
Applying Lemma~\ref{lem:chain-of-intervals}(i) to the chain of intervals 
$(K,\ldots,K)$ with $m+1$ times $K$, we see that there exist closed 
subintervals $J_0,\ldots, J_m$ such that
$J_m=K$ and $J_i\subset K$, $f(J_i)=J_{i+1}$
for all $i\in\Lbrack 0, m-1\Rbrack$.
If $m=n$, then the proof is over because the length of $K$ is 
$a+\eps-c>\eps/2$.
If $m<n$, then $x_{m+1}>c$ according to the choice of $m$, and
$x_{m+1}=f(x_m)<a+ \eps$  by \eqref{eq:a-eta} (recall that $c< a+\eta$).
Hence $x_{m+1}\in K$.
The interval $K$ contains either $x_{m+1}-\frac{\eps}4$
or $x_{m+1}+\frac{\eps}4$ because $|K|\ge \frac{\eps}2$. 
Applying Lemma~\ref{lem:speci1} to $x_{m+1}$, $\frac{\eps}4$ and $n-m+1$
(instead of $x$, $\eps$ and $n$ respectively), we see that there exist
closed intervals $J_{m+1}',\ldots, J_n'$ such that
\begin{itemize}
\item $J_{m+1}'\subset K$,
\item $\forall i\in\Lbrack m+1,n-1\Rbrack$, $f(J_i')=J_{i+1}$,
\item $\forall i\in\Lbrack m+1,n\Rbrack$, $x_i\in J_i'$ and
$J_i'\subset [x_i-\frac{\eps}4,x_i+\frac{\eps}4]$
\item there exists $i\in\Lbrack m+1,n\Rbrack$ such that  $J_i'$ contains
either $x_i-\frac{\eps}4$ or $x_i+\frac{\eps}4$, and thus the length of 
$J_i'$ is at least $\frac{\eps}4$.
\end{itemize}
It follows that $(J_0,\ldots, J_m=K,J_{m+1}')$ is a chain of intervals. 
Therefore, according to Lemma~\ref{lem:chain-of-intervals}(i), there exist
$J_0',\ldots, J_m'$, subintervals of $J_0,\ldots, J_m$ respectively,
such that $f(J_m')=J_{m+1}'$  and $f(J_i')=f(J_{i+1}')$
for all $i\in\Lbrack 0,m-1\Rbrack$.
The sequence $(J_0',\ldots, J_n')$ satisfies the required properties.
\end{proof}

\begin{theo}\label{theo:mixing-specification}
A topologically mixing interval map $f\colon I\to I$
has the specification property.
\end{theo}

\begin{proof}
If $f^2$ has the specification property, then so has $f$
by continuity. Moreover, if $f$ is topologically mixing, then
so is $f^2$ by Theorem~\ref{theo:summary-mixing}. Therefore, it is 
equivalent to prove the theorem
for $f$ or for $f^2$. Then, in view of 
Lemma~\ref{lem:accessibility}, we can assume that 
the non accessible endpoints (if any) are fixed, by 
considering $f^2$ instead of $f$ if necessary.

Let $0<\eps< \frac{|I|}4$. We write $I=[a,b]$. If both $a$ and $b$ are
accessible, we define $I_0:=[a,b]$. Otherwise, let $0<\delta<\eps$  be
given by Lemma~\ref{lem:speci2} and 
\begin{eqnarray*}
I_0&:=&[a+\delta,b]\quad\text{if $a$ is the only non accessible endpoint},\\
I_0&:=&[a,b-\delta]\quad\text{if $b$ is the only non accessible endpoint},\\
I_0&:=&[a+\delta,b-\delta]\quad\text{if both $a$ and $b$ are non accessible}. 
\end{eqnarray*}
We fix a positive integer $p$ such that $\frac{b-a}p<\frac{\eps}8$, 
and we define
$$
\forall k\in\Lbrack 0,p-1\Rbrack,\ A_k:=
\left(a+\frac{k(b-a)}{p},a+\frac{(k+1)(b-a)}{p}\right). 
$$
According to Proposition~\ref{prop:accessibility}, 
for every $k\in\Lbrack 0,p-1\Rbrack$, there exists an integer
$N_k$ such that $f^n(A_k) \supset I_0$ for all $n\ge N_k$.
We set $N:=\max\{N_0,\ldots,N_{p-1}\}$. 
Let $J_0,\ldots, J_k$ be intervals such that
$f(J_i)=J_{i+1}$ for all $i\in\Lbrack 0,k-1\Rbrack$. 
Then, according to the definition of $N$,
\begin{equation}\label{eq:speci1}
\exists\, i\in\Lbrack 0,k\Rbrack,\ |J_i|\ge \eps/4,
\Longrightarrow
\forall n\ge N,\ f^n(J_k)\supset I_0
\end{equation}
because the assumption $|J_i|\ge \frac{\eps}4$ implies that $A_j\subset I_i$ 
for some $j\in\Lbrack 0,p-1\Rbrack$. 

\medskip\textsc{Fact 1.}
{\it Let $x\in I$ and $n\ge 0$. There exist closed intervals $J_0,\ldots, J_n$ such that:
\begin{enumerate}
\item $J_0\subset I_0$,
\item $\forall i\in\Lbrack 0,n\Rbrack$, $J_i\subset [f^i(x)-\eps,f^i(x)+\eps]$,
\item $\forall i\in\Lbrack 0,n-1\Rbrack$, $f(J_i)=J_{i+1}$,
\item there exists $i\in \Lbrack 0,n\Rbrack$ such that $|J_i|\ge \eps/4$.
\end{enumerate}}
We split the proof of the fact depending on $x\in I_0$ or not.
If $x\in I_0$, let $J_0,\ldots, J_n$ denote the intervals given by 
Lemma~\ref{lem:speci1}. They satisfy (ii)-(iv). Moreover, $|I_0|\ge 2\eps$
by definition. This implies that
either $[x-\eps,x]\subset I_0$ or $[x,x+\eps]\subset I_0$, and thus $J_0$
can be chosen to be a subinterval of $I_0$ (still by Lemma~\ref{lem:speci1}), 
which is (i).
If $a$ is not accessible and if $x\in [a,a+\delta]$,
then Lemma~\ref{lem:speci2} gives the suitable subintervals. 
The same conclusion holds if $x\in [b-\delta,b]$ 
and if $b$ is non accessible.

\medskip\textsc{Fact 2.}
{\it Let $x_1,\ldots, x_p$ be points in $I$ and let
$m_1\le n_1<m_2\le n_2<\cdots <m_p\le n_p$ be integers satisfying 
$m_{i+1}-n_i\ge N$ for all $i\in\Lbrack 1,p-1\Rbrack$.
Then there exist closed intervals $(J_i)_{m_1\le i\le n_p}$
such that
\begin{itemize}
\item $J_{m_1}\subset I_0$,
\item $\forall i\in\Lbrack m_1,n_p-1\Rbrack$, $f(J_i)=J_{i+1}$,
\item $\forall k\in\Lbrack 1,p\Rbrack$, $\forall i\in\Lbrack m_k,n_k\Rbrack$, 
$J_i\subset [f^i(x_k)-\eps,f^i(x_k)+\eps]$,
\item $\forall n\ge N$, $f^n(J_{n_p})\supset I_0$.
\end{itemize}
}

We prove Fact~2 by induction on $p$.

$\bullet$ Case $p=1$: we apply Fact~1 to
$x:=f^{m_1}(x_1)$ and $n:=n_1-m_1$. The last condition is satisfied because of 
\eqref{eq:speci1}.

$\bullet$ Suppose that Fact~2 holds at rank $p-1$
and let $J_{m_1},\ldots, J_{n_{p-1}}$ be the intervals given by Fact~2.
We apply Fact~1 with $x:=f^{m_p}(x_p)$ and $n:=n_p-m_p$ and we
call the resulting intervals $J_{m_p}',\ldots, J_{n_p}'$. 
Then $f^n(J_{n_p}')\supset I_0$ for all $n\ge N$ by \eqref{eq:speci1}.
We set $J_i:=f^{i-n_{p-1}}(J_{n_{p-1}})$ for all 
$i\in\Lbrack n_{p-1}+1, m_p\Rbrack$. By assumption, $m_p-n_{p-1}\ge N$, and thus
$J_{m_p}=f^{m_p-n_{p-1}}(J_{n_{p-1}})\supset I_0$
by \eqref{eq:speci1}. Therefore $(J_{m_1},\ldots, J_{m_p-1}, J_{m_p}')$
is a chain of intervals because $J_{m_p}'\subset I_0$ by construction.
By Lemma~\ref{lem:chain-of-intervals}(i), 
there exist subintervals $J_i'\subset J_i$
such that $f(J_i')=f(J_{i+1}')$ for all $i\in\Lbrack m_1,m_p-1\Rbrack$.
It follows that the sequence $J_{m_1}',\ldots, J_{n_p}'$ satisfies 
Fact~2. This concludes the induction.

\medskip
It is now easy to prove that $f$ has the specification property.
Let $x_1,\ldots, x_p$ be points in $I$, let 
$m_1\le n_1<m_2\le n_2<\cdots< m_p \le n_p$ be integers satisfying 
$$
\forall i\in\Lbrack 1,p-1\Rbrack,\ m_{i+1}-n_i\ge N\text{ and }
q\ge n_p-m_1+N.
$$
Let $J_{m_1},\ldots, J_{n_p}$ be the intervals given by Fact~2. Then
$f^n(J_{n_p})$ contains $I_0$ for all $n\ge N$, so $f^q(J_{m_1})=
f^{q-n_p+m_1}(J_{n_p})\supset I_0\supset J_{m_1}$. By
Lemma~\ref{lem:fixed-point}, there exists $x\in J_{m_1}$
such that $f^q(x)=x$. We set $y:=f^{q-m_1}(x)$ in order to have 
$f^{m_1}(y)=x\in J_{m_1}$. Then $f^q(y)=y$ and
$$
\forall k\in\Lbrack 1,p\Rbrack,\ \forall i\in\Lbrack m_k,n_k\Rbrack,\ f^i(y)\in
J_i\subset [f^i(x_k)-\eps,f^i(x_k)+\eps].
$$
This is exactly the specification property.
\end{proof}

%********************
\subsection*{Remarks on graph maps}

Theorem~\ref{theo:mixing-specification} was extended to graph maps,
by Blokh  \cite{Blo9}; see \cite{Blo14} for a statement in English.

\begin{theo}\label{theo:mixing-specificationG}
A topologically mixing graph map has the specification property.
\end{theo}

%**********************************************************************
%**********************************************************************
\section{Periodic points and transitivity}\label{sec:periodic-points-transitivity}

We recall Proposition~\ref{prop:transitivity-periodic-points}: 
for a transitive interval map, the set of periodic points is dense.
The converse is obviously false, but one may ask the following question: if the
set of periodic points is dense, does there exist a transitive cycle of
intervals? The identity and the map $f(x)=1-x$ on $[0,1]$ give 
counter-examples. Therefore one has to consider only interval maps such that
$f^2$ is different from the identity. With this restriction, the answer is
positive.
This is a result of Blokh, which is stated without proof in \cite{Blo3}.
The same result was proved independently by Barge and Martin \cite{BM3},
but their proof relies on complicated notions.
We give a more basic proof here.

We start with a lemma. Then
Proposition~\ref{prop:dense-periodic-points-unstable} states
that, if the set of periodic points of $f$ is dense, all the points outside 
$P_2(f):=\{x\mid f^2(x)=x\}$ are unstable. This result makes a link with 
the results on sensitivity from the previous chapter and will allow us
to conclude.

\begin{lem}\label{lem:2-connected-components}
Let $f$ be an interval map such that the set of periodic points is dense. 
Then, for every non degenerated interval $J$, the set
$\bigcup_{n\ge 0}f^n(J)$ has at most two connected components.
\end{lem}

\begin{proof}
Let $J$ be a non degenerate interval. By assumption, there
exists a periodic point in $J$, say $x$. Let $p$ denote the period of $x$.
For all $i\in\Lbrack 0,p-1\Rbrack$ and all $n\ge 0$, 
$f^{np+i}(J)$ contains $f^i(x)$. Thus the set $\bigcup_{n\ge 0} f^n(J)$ is 
invariant and has at most $p$ connected components; 
we call them $J_1,\ldots, J_q$ (with $q\in\Lbrack 1,p\Rbrack$) in such a 
way that $J_1<J_2<\cdots<J_q$. By the intermediate value theorem,
for every $i\in\Lbrack 1,q\Rbrack$ there exists $\sigma(i)\in\Lbrack 1,
q\Rbrack$ such that $f(J_i)\subset J_{\sigma(i)}$. 
Let $i_0$ be the integer such that $J\subset J_{i_0}$. For every 
$i\in\Lbrack 1,q\Rbrack$, there exists $n\ge 0$
such that $f^n(J)\subset J_i$, that is, $\sigma^n(i_0)=i$. This implies
that the orbit of $i_0$ under $\sigma$ is the whole set, and hence
$\sigma$ is necessarily a cyclic permutation of $\{1,\ldots,q\}$.
We want to show that $q=1$ or $2$. From now on, we assume that $q\ge 2$. 

If there exists $i\in\Lbrack 1, q-1\Rbrack$ such that $|\sigma(i)-\sigma(i+1)|\ge 2$,
we choose an integer $k$ strictly between $\sigma(i)$ and $\sigma(i+1)$.
Let $a:=\sup J_i$ and $b:=\inf J_{i+1}$. Then 
$f(a)\in \overline{J_{\sigma(i)}}$ and
$f(b)\in \overline{J_{\sigma(i+1)}}$, which implies  that
$f((a,b))$ contains $J_k$ by  the intermediate value theorem 
and that $(a,b)$ is not empty. 
Let $V\subset (a,b)$ be a nonempty open interval such that $f(V)\subset J_k$.
It follows that
$$
\forall n\ge 1,\ f^n(V)\subset\CO_f(J_k)\subset J_1\cup\cdots\cup J_q.
$$ 
This implies that, $\forall n\ge 1$, 
$f^n(V)\cap V=\emptyset$, but this contradicts the
assumption that $V$ contains periodic points. We deduce that
$$
\forall i\in\Lbrack 1,q-1\Rbrack,\ |\sigma(i)-\sigma(i+1)|=1.
$$
If $\sigma(2)-\sigma(1)=1$, we obtain from place to place:  $\sigma(k)=
\sigma(1)+k-1$. Since $\sigma$ is a cyclic permutation of length $q\ge 2$,
we have $\sigma(1)\ge 2$, and thus $\sigma(q)\ge q+1$, which is impossible.
We deduce that $\sigma(2)-\sigma(1)=-1$ and $\sigma(q)=\sigma(1)-q+1$.
The only possibility is $\sigma(1)=q$ and
$\sigma(q)=1$ because $\sigma(q)\ge 1$. Since $\sigma$ is a cycle of length 
$q$, we must have $q=2$. This concludes the proof.
\end{proof}

\begin{prop}\label{prop:dense-periodic-points-unstable}
Let $f\colon I\to I$ be an interval map such that the set of periodic points 
is dense. For every point
$x$ such that $f^2(x)\neq x$, there exists $\eps>0$
(depending on $x$) such that $x$ is $\eps$-unstable.
\end{prop}

\begin{proof}
If $U$ is a nonempty open interval, the set $\CO_f(U):=
\bigcup_{n\ge 0}f^n(U)$ has
one or two connected components according to 
Lemma~\ref{lem:2-connected-components}. 
We fix a point $x$ such that $f^2(x)\neq x$.
If there exists an open interval $U_0$ containing $x$ such that
$\CO_f(U_0)$ has two connected components, we call them $J_1$ and 
$J_2$ in such a way that $U_0\subset J_1$, and we set
$g:=f^2$. In this situation, we necessarily have $f(J_1)\subset J_2$ and 
$f(J_2)\subset J_1$. Moreover, for every 
nonempty open interval $U\subset U_0$, we see that
$U\subset J_1$, $f(U)\subset J_2$ and
$\CO_f(U)\subset J_1\cup J_2$, and hence $\CO_f(U)$
has two connected components too. On the other hand, if $\CO_f(U)$
is a connected set for every open interval $U$ containing $x$, 
we set $g:=f$ and $U_0:=I$. With this notation, 
for every open subinterval $U\subset U_0$ containing $x$,
the set $\CO_g(U)$ is connected. The two points
$$
a:=\inf_{n\ge 0} \CO_g(x)\quad\text{and}\quad b:=\sup_{n\ge 0}\CO_g(x)
$$
are distinct because $g(x)\neq x$ by assumption. 

First we suppose that
$\CO_g(x)$ is not dense in $[a,b]$, which means that there exist $z\in (a,b)$ 
and $\eps>0$ such that $(z-\eps,z+\eps)\subset (a,b)\setminus \CO_g(x)$.
Let $U\subset U_0$ be an open interval containing $x$.
The set $\CO_g(U)$ is connected and contains $\CO_g(x)$, and 
thus it contains $(a,b)$ too.
In particular, there exist $y\in U$ and $k\ge 0$ such that
$g^k(y)=z$, and hence
$$
|g^k(x)-g^k(y)|\ge \inf_{n\ge 0} |g^n(x)-z|\ge\eps.
$$
We deduce that the point $x$ is $\eps$-unstable.

Now we suppose that $\CO_g(x)$ is dense in $[a,b]$.
This implies that $g([a,b])=[a,b]$ and that $g|_{[a,b]}$ is transitive. 
Then, by Proposition~\ref{prop:transitivity-sensitivity},
the map $g|_{[a,b]}$ is $\eps'$-sensitive for every $\eps'\in
\left(0,\frac{b-a}4\right)$,
In particular, the point $x$ is $\eps'$-unstable.
\end{proof}

\begin{prop}\label{prop:periodic-points-transitive-cycle}
Let $f\colon I\to I$ be an interval map such that the set of periodic points 
is dense. Suppose that $f^2$ is different from the identity map. 
Then at least one of the following holds:
\begin{itemize}
\item there exists a non degenerate closed interval $J$ such that 
$f(J)=J$ and $f|_J$ is transitive,
\item there exist two disjoint non degenerate closed intervals $J_1,J_2$ 
such that $f(J_1)=J_2$, $f(J_2)=J_1$ and $f|_{J_1\cup J_2}$ is transitive.
\end{itemize}
\end{prop}

\begin{proof}
Recall that $P_2(f):=\{x\in I\mid f^2(x)=x\}$. This is a closed set and,
by assumption, the open set $I\setminus P_2(f)$ is not empty.
By Proposition~\ref{prop:dense-periodic-points-unstable},
all the points of $I\setminus P_2(f)$ are unstable, and thus
$$
I\setminus P_2(f)\subset \bigcup_{n=1}^{+\infty}U_{\frac{1}{n}}(f)\subset 
\bigcup_{n=1}^{+\infty}\overline{U_{\frac{1}{n}}(f)}.
$$
By the Baire category theorem, there exists $n\ge 1$ such that
$\overline{U_{\frac{1}{n}}(f)}$ has a nonempty interior.
It follows that $U_{\frac{1}{2n}}(f)$ has a nonempty interior too, because
$\overline{U_{\frac{1}{n}(f)}} \subset U_{\frac{1}{2n}}(f)$ by 
Lemma~\ref{lem:unstable}(iii). Then, by
Proposition~\ref{prop:sensitivity-transitive-component},
there exists a cycle of intervals $(J_1,\ldots, J_p)$
such that $f|_{J_1\cup\cdots J_p}$ is transitive. Finally, $p=1$ or $2$ by
Lemma~\ref{lem:2-connected-components}.
\end{proof}

Proposition~\ref{prop:periodic-points-transitive-cycle} makes possible a
decomposition of the interval into transitive components, as stated in 
the next theorem and illustrated in Figure~\ref{fig:transitive-components}.

\begin{theo}\label{theo:dense-periodic-points-transitivity}
Let $f\colon I\to I$ be an interval map such that the set of periodic points
is dense. Then there exists a finite (possibly empty)
or countable family of sets $\CE$ such that: 
\begin{enumerate}
\item $\forall C\in \CE$, the set $C$ is either a non degenerate closed
interval or the union of two disjoint non degenerate closed intervals,
\item $\forall C\in \CE$, $C$ is invariant and $f|_C$ is transitive,
\item the sets in $\CE$ have pairwise disjoint interiors,
\item $I\setminus \bigcup_{C\in \CE} C \subset P_2(f)$.
\end{enumerate}
\end{theo}

\begin{proof}
We define
$$
\CE:=\{C\subset I\mid C\text{ cycle of intervals},\ f|_C \text{ transitive}\}.
$$
By Lemma~\ref{lem:2-connected-components}, every element in $\CE$ 
has at most two connected components, and thus it satisfies
(i) and (ii). Moreover,
\begin{equation}\label{eq:transitive-component-f2}
\text{if }J \text{ is a connected component of  }C\in\CE,
\text{ then }f^2(J)\subset J.
\end{equation}
Let $C,C'\in\CE$ and $V:=\Int{C}\cap
\Int{C'}$. If $V\neq\emptyset$, then, by transitivity,
$$
\overline{\bigcup_{n\ge 0} f^n(V)}=C=C'.
$$
Therefore, two different elements of $\CE$ have nonempty disjoint interiors,
which is (iii). This implies that $\CE$ is at most countable because
for every $C\in\CE$, $\Int{C}$ contains a rational number $r_C$, and $r_C
\ne r_{C'}$ if $C\ne C'$.
It remains to prove (iv). We set
$$
X_0:=\bigcup_{C\in\CE}C.
$$
We are going to show first that $I\setminus\overline{X_0}\subset P_2(f)$, 
and second that $\overline{X_0}\setminus X_0\subset P_2(f)$; these two
facts clearly imply (iv). We set $Y:=I\setminus\overline{X_0}$; this is an
open set. 
If $f(Y)\cap \Int{\overline{X_0}}\neq \emptyset$, then there exists
a non degenerate subinterval $J\subset Y$ such that $f(J)\subset 
\Int{\overline{X_0}}$. 
Since $f(\overline{X_0})\subset \overline{X_0}$, this implies that 
$f^n(J)\cap J=\emptyset$ for all $n\ge 1$.
But this contradicts the fact that $J$ contains periodic points.
We deduce that
$f(Y)\subset I\setminus \Int{\overline{X_0}}=\overline{Y},$
and thus
\begin{equation}\label{eq:dense-transitive-components}
f(\overline{Y})\subset \overline{Y}.
\end{equation}
Suppose that 
\begin{equation}\label{eq:YsetminusP2}
Y\setminus P_2(f)\ne\emptyset
\end{equation}
and 
let $K'$ be a connected component of 
$Y$ such that $K'\setminus P_2(f)\neq \emptyset$.
Since $Y$ is open, $K'$ is an open interval. Moreover,
there exists $n\ge 1$ such that $f^n(K')\cap K'\neq \emptyset$ because
the set of periodic points is dense by assumption.
Let $K$ be the connected component of $\overline{Y}$ containing $K'$.
Then $f^n(K)\cap K\neq\emptyset$ and
$f^n(K)$ is included in a connected component of 
$\overline{Y}$ by \eqref{eq:dense-transitive-components}. 
So $f^n(K)\subset K$. We consider the interval map
$$
g:=f^n|_K\colon K\to K.
$$
By Proposition~\ref{prop:dense-periodic-points-unstable}, all 
points in $K\setminus P_2(f)$ are unstable for $f$. Therefore, all 
points in $K\setminus P_2(f)$ are
unstable for $g$ by Lemma~\ref{lem:unstable}(i).
Thus
$$
K'\setminus P_2(f)\subset K\setminus P_2(f)\subset
\bigcup_{k=1}^{+\infty}U_{\frac{1}{k}}(g)\subset
\bigcup_{k=1}^{+\infty}\overline{U_{\frac{1}{k}}(g)},
$$
where $K'\setminus P_2(f)$ is a nonempty open set.
Then we use the same argument as in the proof of
Proposition~\ref{prop:periodic-points-transitive-cycle}: 
using the Baire category theorem, we find $k$ such that 
$\Int{\overline{U_{\frac{1}{k}}(g)}}\neq\emptyset$, so
$\Int{U_{\frac{1}{2k}}(g)}\neq \emptyset$  by Lemma~\ref{lem:unstable}(iv).
Then, according to
Proposition~\ref{prop:sensitivity-transitive-component},
there exists a transitive cycle of intervals $C\subset K$ for $g$. It is
straightforward to see that the set 
$C':=C\cup f(C)\cup\cdots\cup f^{n-1}(C)$ is a finite union of
non degenerate closed intervals and that $f|_{C'}$ is transitive.
Hence $C'\subset X_0$ and $C'\cap \Int{X_0}\ne\emptyset$ because
$\Int{C'}\ne\emptyset$. On the other hand,
$C'\subset\overline{Y}$ by 
\eqref{eq:dense-transitive-components}, and thus
$C'\subset I\setminus\Int{X_0}$, which leads to a contradiction.
Therefore, \eqref{eq:YsetminusP2} does not hold, that is,
$
Y=I\setminus \overline{X_0}\subset P_2(f).
$

Now we are going to show that $\overline{X_0}\setminus X_0\subset P_2(f)$.
We define 
$$\CA:=\{J\subset I \mid \exists C\in\CE, J\text{ is a connected component of } C\}.
$$ 
We fix $x\in \overline{X_0}\setminus X_0$. Let $(x_n)_{n\ge 0}$ be
a monotone sequence in $X_0$ converging to
$x$. For all $n\ge 0$, let $J_n\in\CA$ be such that $x_n\in J_n$.
If there exists $n_0$ such that $J_n=J_{n_0}$ for all $n\ge n_0$, 
then $\lim_{n\to +\infty} x_n$ belongs to the closed set $J_{n_0}$, and
thus $x\in X_0$, which is impossible. Thus the sequence $(J_n)_{n\ge 0}$
is not eventually constant, which implies that $\lim_{n\to +\infty}|J_n|=0$
because two distinct elements in $\CA$ have disjoint interiors.
Let $\eps>0$. By continuity, there exists $0<\alpha<\eps$ such that
$$
\forall y\in I,\ |x-y|< \alpha\Rightarrow |f^2(x)-f^2(y)|<\eps.
$$ 
We choose  $n$ such that $|x-x_n|<\alpha$ and $|J_n|<\eps$.
Then 
$f^2(x_n)\in J_n$ by \eqref{eq:transitive-component-f2}, so
$$
|x-f^2(x)|\le |x-x_n|+|x_n-f^2(x_n)|+|f^2(x_n)-f^2(x)|<3\eps.
$$
Since this is true for all $\eps>0$, we deduce that $f^2(x)=x$. In other 
words,  $\overline{X_0}\setminus X_0\subset P_2(f)$. This concludes the proof.
\end{proof}

\begin{figure}[htb]
\centerline{\includegraphics{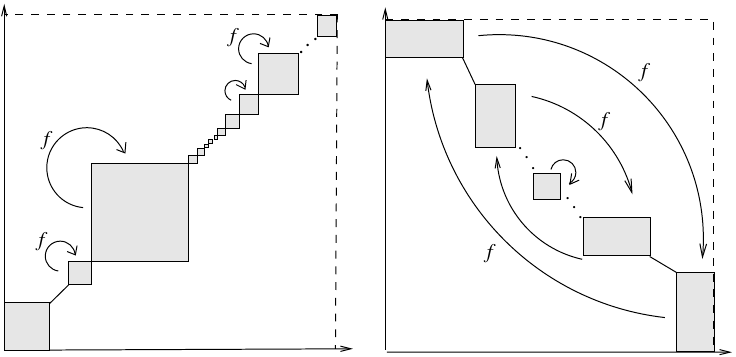}}
\caption{Decomposition into transitive components of a map $f$ when the set of
periodic points is dense: two cases. The gray areas represent
the transitive components, whereas the part of the graph of $f$ made of black 
lines  is the set $P_2(f)$.
The points where transitive
components accumulate on both sides are not unstable, neither are the
points in $\Int{P_2(f)}$.}
\label{fig:transitive-components}
\end{figure}

Theorem~\ref{theo:dense-periodic-points-transitivity} can be more
precise. Let $\CE$ be the family of transitive components given
by Theorem~\ref{theo:dense-periodic-points-transitivity} and $C\in\CE$. 
We write $I=[a,b]$. Let $J$ be a connected component of $C$
and suppose that $J$ has an endpoint $c\notin\{a,b\}$.
Theorem~\ref{theo:dense-periodic-points-transitivity} implies that
either $c$ is the endpoint of a connected component $J'$ of 
another element of $\CE$,  or $c\in P_2(f)$. In the first case,
using the facts that $f^2(J)=J$ and $f^2(J')=J'$, we deduce that
$f^2(c)=c$. Therefore, all the endpoints of the connected components, except
maybe $a$ and $b$, belong to $P_2(f)$.
In particular, if a connected component is made of two disjoint 
intervals $J, K$, then
Theorem~\ref{theo:summary-transitivity} implies that 
$f^2|_J$ and $f^2|_K$ are topologically mixing.
Now suppose that $C$ is an interval and that there exists a fixed 
point $z$ outside $C$
(in particular, such a fixed point exists when there exist other transitive 
intervals). We may assume that $z<C$, the other case being symmetric.
Then it can be shown that the decomposition given by
Theorem~\ref{theo:dense-periodic-points-transitivity} implies that
 $f([a,\min C])=[a,\min C]$. Thus $\min C$ is a fixed point
and $f|_C$ is topologically mixing by Theorem~\ref{theo:summary-transitivity}.

Figure~\ref{fig:transitive-components} illustrates what kind of decomposition
can exist when there are several transitive components.
On the left side, all transitive components are intervals and $f$ is
topologically mixing on each of them.
On the right side, there is only one transitive interval in the middle,
$f$ may or may not be topologically mixing on this interval (this middle
transitive interval may not exist), and
$f^2$ is topologically mixing on every connected component of the 
transitive components made of two intervals.

%*************************
\subsection*{Remarks on graph maps}

Theorem~\ref{theo:dense-periodic-points-transitivity} was first
extended to tree maps by Roe \cite{Roe}, then to graph maps by
Yokoi \cite{Yok2}.

\begin{theo}
Let $f\colon G\to G$ be a graph map such that the set of periodic points
is dense. Then there exist a positive integer $N$ and a finite (possibly 
empty) or countable family of subgraphs $\CE$ with disjoint interiors such that
\begin{itemize}
\item $\forall H\in \CE$, the set $H$ is $f^N$-invariant and $f^N|_H$ is 
topologically mixing,
\item $G\setminus\bigcup_{H\in\CE}X\subset P_N(f)$.
\end{itemize}
If $G$ is a tree with $e$ endpoints, then one can take
$N=\gcd(2,3,\ldots, e)$.
\end{theo}

For topological graphs that are not trees,
the integer $N$ in the preceding theorem can be arbitrarily large.
For example, the rational rotation 
$R_{\frac1n}$ gives a system in which all points are periodic of period $n$. 

%**********************************************************************
%************************************************************
\section{Sharkovsky's Theorem, Sharkovsky's order and type}\label{sec:Sharkovsky}

Sharkovsky's Theorem states that, for an interval map, the presence of a
periodic point with a given period implies the existence of other
periods determined by so-called \emph{Sharkovsky's order} \cite{Sha}.

\begin{defi}
\emph{Sharkovsky's order} is the total ordering on $\IN$ defined by:
\index{Sharkovsky's order}
\label{notation:sharkovskyorder}
$$
3\lhd 5\lhd 7\lhd 9\lhd\cdots\lhd 2\cdot 3 \lhd 2\cdot 5\lhd\cdots\lhd
2^2\cdot 3\lhd 2^2\cdot 5\lhd\cdots\lhd 2^3\lhd 2^2 \lhd 2 \lhd 1
$$
(first, all odd integers $n>1$, then $2$ times the odd integers $n>1$,
then successively $2^2\times$, $2^3\times$, \ldots, $2^k\times$ $\ldots$ the odd integers $n>1$, 
and finally
all the powers of $2$ by decreasing order).

$a \rhd b$ means $b\lhd a$. The notation $\unlhd, \unrhd$ will denote 
the order with possible equality.
\end{defi}

\begin{rem}
In Sharkovsky's order, 3 is the minimum (as above) in some papers whereas
it is the maximum in other ones (i.e. all inequalities are reversed).
The ordering above is the same as in Sharkovsky's original paper \cite{Sha},
but there is no consensus in the literature. Even in Sharkovsky's papers,
both orderings appear.
In addition, the symbol for the inequalities varies much: one can
find $\lhd$, $\prec$, $<_{sh}$, $\vdash$, $\ll$.
The spelling of ``Sharkovsky'' varies much too.
\end{rem}

\begin{theo}[Sharkovsky]\label{theo:Sharkovsky}\index{Sharkovsky's Theorem}
If an interval map $f$ has a periodic point of period
$n$, then, for all integers  $m\rhd n$, $f$ has periodic points of period $m$.
\end{theo}

This striking result is one of
the first about the dynamics on the interval and, more generally,
one of the earliest results pointing out the existence of
``complicated'' behavior in some dynamical systems.

The original paper of Sharkovsky, in 1964, was in Russian \cite{Sha}.
The first proof in English, different from Sharkovsky's, is due
to \v{S}tefan in 1976 \cite{Ste0} (published in \cite{Ste}). 
In the meantime, Li and Yorke, unaware of the work of Sharkovsky, 
re-proved a particular case, namely that the existence of a periodic 
point of period $3$ implies that there are periodic points of all periods
\cite{LY}. This illustrates the lack of communication between Russian and
English literatures.
Later, several proofs of Sharkovsky's theorem were given 
\cite{BGMY, OO, BCop}.
The presentation we are going to give  derives from the one of
Block, Guckenheimer, Misiurewicz and Young \cite{BGMY}.
We shall first introduce the notion of a \emph{graph of a periodic orbit}
and its main properties;
then we shall prove Theorem~\ref{theo:Sharkovsky}
in Subsection~\ref{subsec:proofSharkovsky}.

%*************************************************************************
\subsection{Graph of a periodic orbit}

We are going to associate a directed graph to a periodic orbit,
and show that the existence of other periodic points can be read in
this graph.

Recall that $\langle a,b\rangle$ denotes $[a,b]$ or $[b,a]$ 
depending on $a\le b$ or $b\le a$.

\begin{defi}\label{defi:graph-periodic-orbit}
Let $f$ be an interval map and $x$ a periodic point of period 
$n\ge 2$. Let $x_1<\cdots<x_n$ denote the ordered points in 
$\{x,f(x),\ldots,f^{n-1}(x)\}$ and let $I_j:=[x_j,x_{j+1}]$
for all $j\in\Lbrack 1,n-1\Rbrack$.
The \emph{graph of the
periodic orbit of $x$}\index{graph of a periodic orbit} is the directed graph
whose vertices are $I_1,\ldots, I_{n-1}$ and
$$\forall j,k\in\Lbrack 1,n-1\Rbrack, \text{ there is an arrow }
I_j\to I_k \text{ iff }I_k\subset \langle f(x_j),f(x_{j+1})\rangle.$$

In this graph, a \emph{fundamental}\index{fundamental cycle (in the graph
of a periodic orbit)}
cycle is a cycle of length $n$, say $J_0\to J_1\to\cdots\to J_{n-1}\to J_0$,
such that there exists a point $c\in \CO_f(x)$ with the property that
$f^k(c)$ is an endpoint of $J_k$ for all $k\in\Lbrack 0,n-1\Rbrack$.
\end{defi}

It is important to notice that if $I_i\to I_j$ is an arrow in the
graph of a periodic orbit, then $I_i$ covers $I_j$. Therefore, a
cycle in this graph is a chain of intervals, starting and ending with
the same interval.

\medskip
Recall that a cycle is \emph{primitive} if it is not the repetition of a 
shorter cycle.

\begin{lem}\label{lem:fundamental-cycle}
In the graph of a periodic orbit, there exists a unique fundamental cycle
(up to cyclic permutation).
In this cycle, each vertex of the graph appears at most twice and
one of them appears exactly twice. The fundamental cycle can be
decomposed into two shorter primitive cycles.
\end{lem}

\begin{proof}
We consider a periodic orbit of period $n\ge 2$ composed of the points
$x_1<\cdots<x_n$ and we set $I_j:=[x_j,x_{j+1}]$
for all $j\in\Lbrack 1,n-1\Rbrack$.
We fix $i\in\Lbrack 1,n-1\Rbrack$ and $c\in\{x_i,x_{i+1}\}$. We are going
to show by induction on $k$ that 
\begin{equation}\label{eq:Jk}
\text{\begin{minipage}{10cm}
there is a unique sequence of intervals 
$(J_k)_{k\ge 0}$, which are vertices of the graph of the
periodic orbit, such that \\ $J_0=I_i$ and 
$\forall k\ge 0$, $f^k(c)\in\End{J_k}$ and 
$J_k\to J_{k+1}$. \end{minipage}}
\end{equation}
Suppose that $J_{k-1}=[a,b]$ is
already defined. The interval $J_k$ must satisfy 
\begin{equation}\label{eq:fundamental-cycle}
J_k\subset \langle f(a),f(b)\rangle\text{ and }f^k(c)
\text{ is an endpoint of }J_k.
\end{equation}
According to the induction hypothesis for $J_{k-1}$,  the points $f^{k-1}(c)$
belong to $\{a,b\}$. Thus either $f(a)$ or $f(b)$ is equal to $f^k(c)$,
and  \eqref{eq:fundamental-cycle} determines uniquely $J_k\in 
\{I_1,\ldots, I_{n-1}\}$. This ends the induction.

From now on, let $(J_k)_{k\ge 0}$ denote the sequence defined above starting 
with $J_0:=I_1$ and $c:=x_1$. Since $f^n(x_1)=x_1$ and $x_1<x_i$
for all $i\in\Lbrack 2,n\Rbrack$, the interval $J_n$
is necessarily equal to
$J_0$. Therefore, $J_0\to J_1\to\cdots\to J_{n-1}\to J_0$ is a fundamental 
cycle. Now, we are going to prove the uniqueness of the fundamental cycle.
Let $K_0\to K_1\to\cdots\to K_n=K_0$ be a fundamental cycle and $d$ a
point of the periodic orbit
such that $f^i(d)$ is an endpoint of $K_i$ for all $i\in\Lbrack 0,n-1\Rbrack$.
Since $d\in\{x_1,\ldots, x_n\}$,
there exists $k\in\Lbrack 0,n-1\Rbrack$ such that $d=f^k(x_1)$. Thus
$f^{n-k}(d)=f^n(x_1)=x_1$
is an endpoint of $K_{n-k}$, so $K_{n-k}=J_0$. Then
\eqref{eq:Jk} implies that
$$
(K_0,K_1,\ldots,K_{n-1},K_0)=(J_k,J_{k+1},\ldots,J_{n-1},J_0,\ldots,J_k),
$$
that is, the fundamental cycle is unique up to cyclic permutation.

For every $k\in\Lbrack 1,n-1\Rbrack$, there exist two distinct integers
$i,j\in \Lbrack 0,n-1\Rbrack$ such that $I_k=[f^i(x_1),f^j(x_1)]$.
Consequently, $J_i$ and $J_j$ are the only two intervals of the fundamental
cycle that may be equal to $I_k$. This implies that
every vertex appears at most twice in the fundamental cycle.
Moreover, there are only $n-1$ vertices in the graph
and the fundamental cycle is of length $n$. Thus, one of the vertices
of the graph appears at least twice in the fundamental cycle. 
Finally, if we cut the fundamental cycle at a vertex $I_k$ appearing twice,
we obtain two cycles which are primitive because each of them
contains $I_k$ only once.
\end{proof}

The next lemma 
is originally due to \v{S}tefan~\cite{Ste}. We follow the proof of \cite{BGMY}.
This is a key tool for finding other periods when one periodic orbit
is known.

\begin{lem}\label{lem:cycle-periodic-point}
Let $f$ be an interval map and $x$ a periodic point. If
the graph $G$ of the periodic orbit of $x$ contains a primitive
cycle $J_0\to J_1\to\cdots\to J_{n-1}\to J_0$ of length $n$, then there
exists a periodic point $y$ of period $n$ such that $f^k(y)\in J_k$
for all $k\in\Lbrack 0,n-1\Rbrack$.
\end{lem}

\begin{proof}
By Lemma~\ref{lem:chain-of-intervals}(ii), there exists $y$ such that $f^n(y)=
y$ and $f^k(y)\in J_k$ for all $k\in\Lbrack 0,n-1\Rbrack$. 
Let $p$ be the period of $y$, which is a divisor of $n$.
Suppose that $f^k(y)$ does not belong to $\CO_f(x)$
for all $k\in\Lbrack 0,n-1\Rbrack$.
Then $J_k$ is the unique vertex of $G$
containing $f^k(y)$ for all $k\in\Lbrack 0,n-1\Rbrack$.
This implies that $p=n$, otherwise the cycle would not be primitive.
On the contrary, suppose that there exists $k\in\Lbrack 0,n-1\Rbrack$ such that 
$f^k(y)\in\CO_f(x)$. Then $y=f^{n-k}(f^k(y))$ belongs to the orbit of $x$, and
thus $x$ is of period $p$. Moreover $f^k(y)$ is an 
endpoint of $J_k$ for all $k\in\Lbrack 0,n-1\Rbrack$, which implies that
$J_0\to J_1\to\cdots J_{n-1}\to J_0$ is equal to either the fundamental
cycle or a repetition of it. Finally, $p=n$ because this cycle is primitive.
\end{proof}

The next lemma describes the graph of a periodic orbit
whose period is the smallest odd period greater than $1$. Such a
periodic orbit  is called a \emph{\v{S}tefan cycle}\index{Stefan@\v{S}tefan cycle}, 
and so is any periodic orbit with the same graph
as in Figure~\ref{fig:graph-of-odd-periodic-point}.

\begin{lem}\label{lem:graph-n-minimal}
Let $f$ be an interval map having a periodic point of
odd period different from $1$. Let $p$ be the smallest odd period greater 
than $1$, and  $x$ a periodic point of period $p$.
Let $c$ denote the
median point of the orbit of $x$ (that is, $c\in \CO_f(x)$ and $\CO_f(x)$
contains $(p-1)/2$ points less than $c$ and $(p-1)/2$ greater than $c$). 
If $c<f(c)$, the points of
its orbit are ordered as follows:
$$
f^{p-1}(c)<f^{p-3}(c)<\cdots<f^2(c)<c<f(c)<f^3(c)<\cdots<f^{p-2}(c).
$$
If $c>f(c)$, the reverse order holds.
Moreover, the graph of this periodic orbit is the one represented in
Figure~\ref{fig:graph-of-odd-periodic-point}. 
\begin{figure}[htb]
\centerline{\includegraphics{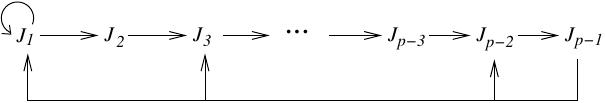}}
\caption{Graph of a periodic orbit of minimal odd period $p>1$.}
\label{fig:graph-of-odd-periodic-point}
\end{figure}
\end{lem}

\begin{proof}
By Lemma~\ref{lem:fundamental-cycle}, the graph of the periodic
orbit of $c$ contains a fundamental cycle that can be split
into two primitive cycles. One of these primitive cycles is of 
odd length because the fundamental cycle is of odd length $p$. 
According to Lemma~\ref{lem:cycle-periodic-point} and because of the
minimality of $p$, this length is necessarily equal to $1$. 
Therefore, the fundamental cycle can
be written as $J_1\to J_1\to J_2\to\cdots\to J_{p-1}\to J_1$. Moreover
$J_i\neq J_1$ for all $i\in\Lbrack 2,p-1\Rbrack$ because each vertex appears at
most twice by Lemma~\ref{lem:fundamental-cycle}.
If $J_i=J_j$ for some $i,j$ with $2\le i<j\le p-1$, then the two cycles
$$
J_1\to J_2\to\cdots\to J_i=J_j\to J_{j+1}\to\cdots\to J_{p-1}\to J_1
$$
and
$$
J_1\to J_1\to J_2\to\cdots\to J_i=J_j\to J_{j+1}\to\cdots\to
J_{p-1}\to J_1
$$
are of respective lengths $p+i-j-1$ and $p+i-j$. 
These lengths are in $\Lbrack 1,p-1\Rbrack$, and one of them  is odd.
But then, using  Lemma~\ref{lem:cycle-periodic-point}, we get a contradiction 
with the choice of $p$. Therefore, we have $J_i\neq J_j$
for all $i,j\in\Lbrack 2,p-1\Rbrack$ with $i<j$.
which implies that $(J_1,J_2,\ldots, J_{p-1})$ is a permutation of the
$p-1$ vertices of the graph of the orbit of $x$. 
Similarly, if $J_i\to J_k$
for some $i,k\in \Lbrack 1,p-1\Rbrack$ with $k>i+1$, or if $J_i\to J_1$ for some
$i\in\Lbrack 2,p-2\Rbrack$, there exists a primitive cycle of odd length (the cycle
$J_1\to J_1$ may be added if necessary to get an odd length) 
with a length in $\Lbrack 2,p-1\Rbrack$, which leads to a contradiction again 
by Lemma~\ref{lem:cycle-periodic-point}.

Let $x_1<\cdots<x_p$ be the ordered points of $\CO_f(x)$. 
We set $I_j:=[x_j,x_{j+1}]$ for all $j\in\Lbrack 1,p-1\Rbrack$. 
Let $k\in\Lbrack 1,p-1\Rbrack$  be the integer
such that $J_1=I_k$. We have shown that the vertex $J_1$ is only directed to 
$J_1$ and $J_2$. This implies that the intervals
$J_1$ and $J_2$ have a common endpoint, and thus $J_2$ is equal to
$I_{k-1}$ or $I_{k+1}$. Since $f(x_j)\neq x_j$
for all $j\in\Lbrack 1,p\Rbrack$, it is easy to check that we are in one
of the following two cases:
\begin{itemize}
\item $J_2=I_{k-1}$, $x_{k+1}=f(x_k)$ and $x_{k-1}=f^2(x_k)$,
\item $J_2=I_{k+1}$, $x_k=f(x_{k+1})$ and $x_{k+2}=f^2(x_{k+1})$.
\end{itemize}
We assume that we are in the first case, the second one being symmetric 
and leading
to the reverse order. We encourage the reader to redraw the points
of Figure \ref{fig:proof-lem} step by step when reading the proof.
\begin{figure}[htb]
\centerline{\includegraphics{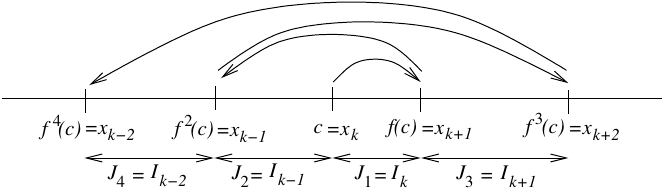}}
\caption{Position of the first iterates of $c$.}
\label{fig:proof-lem}
\end{figure}
We set $c:=x_k$. If $p=3$, then the proof is
complete. If $p>3$, then $f^3(c)>c$, otherwise there would be an arrow
$J_2\to J_1$. Hence $f^3(c)=x_i$ for some $i>k+1$. Since there is
an arrow $J_2\to J_3$ and no arrow $J_2\to J_j$ for all $j>3$, the only
possibility is that $f^3(c)=x_{k+2}$
and $J_3=[f(c),f^3(c)]=I_{k+1}$. If $f^4(c)>f^2(c)$, then necessarily
$f^4(c)>f^3(c)=x_{k+2}$. But this implies that $J_3\to J_1$, which is 
impossible, and hence $f^4(c)<f^2(c)$. Since there is
an arrow $J_3\to J_4$ and no arrow $J_3\to J_j$ for all $j>4$, the only
possibility is that $f^4(c)=x_{k-2}$ and $J_4=[f^4(c),f^2(c)]=I_{k-2}$.
We can go on in this way, and finally we find that the points are ordered 
as follows:
$$
f^{p-1}(c)<f^{p-3}(c)<\cdots<f^2(c)<c<f(c)<f^3(c)<\cdots<f^{p-2}(c).
$$
The point $c$ is the median point of the orbit, and we are in the case
$c<f(c)$. The order of these points enables us to check that the graph of
the periodic orbit is the one represented in
Figure~\ref{fig:graph-of-odd-periodic-point}.
\end{proof}

In the proof of Sharkovsky's Theorem, we shall need the next elementary~lemma.

\begin{lem}\label{lem:period-f-fn}
Let $f$ be an interval map and let $x$ be a point.
\begin{enumerate}
\item If $x$ is a periodic point of period $n$ for $f$, then $x$ is periodic
of period $n/d$ for $f^k$, where $d:=\gcd(n,k)$.
\item If $x$ is a periodic point of period $m$ for $f^k$, then there
exists $d$ a divisor of $k$ satisfying $\gcd(m,d)=1$ and 
such that $x$ is periodic of period $m\frac{k}{d}$ for $f$.
\end{enumerate}
\end{lem}

\begin{proof}
i)
Let $x$ be a periodic point of period $n$ for $f$.  We set $d:=\gcd(n,k)$ and 
$k':=k/d$. We have $f^{k\frac{n}{d}}(x)=f^{k'n}(x)=x$. Let $m\ge 1$ be an
integer such that  $f^{km}(x)=x$. This implies that  $km$ is a multiple of 
$n$, say $km=pn$ for some $p\in\IN$.  Then
$m=\frac{pn}{k}=\frac{pn}{k'd}$. Since $\gcd(n,k')=1$, the quantity
$\frac{p}{k'}$ must be an integer, and thus $\frac{n}{d}$ divides $m$. We
conclude that $x$ is periodic of period $\frac{n}{d}$ for $f^k$.

ii)
Let $x$ be a periodic point of period $m$ for $f^k$, and let
$n$ be the period of $x$ for $f$. Then $n$ 
divides $km$ because $f^{km}(x)=x$. Let $d$ be the integer such that $km=dn$.
We set $p:=\gcd(m,d)$. Then $\frac{m}{p}$ and $\frac{d}{p}$ are integers
and $k\frac{m}{p}=\frac{d}{p}n$. This implies that 
$f^{k\frac{m}{p}}(x)=f^{n\frac{d}{p}}(x)=x$, and thus $p=1$ because $x$ is of period $m$ for $f^k$.
Since $n=\frac{km}{d}$ and $p=\gcd(m,d)=1$, we deduce that $d$ divides $k$.
\end{proof}

%*************************************************************************
\subsection{Proof of Sharkovsky's Theorem}\label{subsec:proofSharkovsky}

\begin{proof}[Proof of Theorem~\ref{theo:Sharkovsky}]
We first deal with the existence of periodic
points of period $1$ or $2$. By Lemma~\ref{lem:fixed-point},
$f$ has a fixed point. We are going to show:
\begin{equation}\label{eq:period2}
f \text{ has a periodic point of period }p>1\!\Rightarrow
\!f \text{ has a periodic point of period }2.
\end{equation}
Let $n$ denote the least period greater than $1$ and
suppose that $n\ge 3$.
According to Lemma~\ref{lem:fundamental-cycle}, the fundamental cycle of a
periodic point of period $n$ can be split into two shorter
primitive cycles, and thus one of them is of length $m$ with $m\in\Lbrack 2,
n-1\Rbrack$.
But then Lemma~\ref{lem:cycle-periodic-point} implies that 
there exists a periodic point of period $m$, 
which contradicts the definition of $n$.
Therefore $n=2$ and \eqref{eq:period2} is proved.

Second, we show that if $f$ has a periodic point of period $p$,
\begin{equation}\label{eq:odd-Sharkovsky}
p>1,\ p \text{ odd}\Rightarrow\forall n\unrhd p,\ f \text{ has a periodic point of period }n.
\end{equation}
According to the definition of Sharkovsky's order, 
it is sufficient to prove \eqref{eq:odd-Sharkovsky} when $p$ is the smallest 
odd period greater than $1$. For such a $p$, the graph of 
a periodic point of period $p$ is given by
Lemma~\ref{lem:graph-n-minimal}. We keep the same notation
as in Figure~\ref{fig:graph-of-odd-periodic-point}.
If $n$ is even and $n\in\Lbrack 2,p-1\Rbrack$, then
$$
J_{p-n}\to J_{p-n+1}\to\cdots\to J_{p-1}\to J_{p-n}
$$
is a primitive cycle of length $n$. If $n$ is greater than $p$, 
we add $n-p$ times the cycle
$J_1\to J_1$ at the end of the fundamental cycle in order to
obtain a primitive cycle of length $n$.
Then, for all even integers $n\ge 2$ and all odd integers $n\ge p$,
there exists a periodic point of period $n$ by 
Lemma~\ref{lem:cycle-periodic-point}. This proves \eqref{eq:odd-Sharkovsky}.

We now turn to the general case. Assume that $f$ has a periodic point $x$
of period $n=2^d q$, where
$q\ge 1$ is an odd integer. We want to show that,
for all $m\neq 1$ with $m\rhd n$, $f$ has a periodic
point of period $m$. We split the proof into three cases.

i) Case $q=1$ and $m=2^e$ for some $0<e<d$. According to
Lemma~\ref{lem:period-f-fn}(i), the point $x$ is of period $2^{d-e+1}>1$ for 
$g:=f^{\frac{m}{2}}$. 
Thus $g$ has a periodic point $y$ of period $2$ 
by \eqref{eq:period2}, and $y$ is periodic of period $m=2^e$ for $f$
by Lemma~\ref{lem:period-f-fn}(ii).

ii) Case $q>1$ and $m=2^d r$ for some $r\ge 2$, $r$ even. 
By Lemma~\ref{lem:period-f-fn}(i), the point
$x$ is of period $q$ for $g:=f^{2^d}$. Since $q$ is odd and greater than $1$, 
$g$ has a periodic point
$y$ of period $r$ according to \eqref{eq:odd-Sharkovsky}. 
Then $y$ is periodic of period $m=2^d r$ for $f$ by
Lemma~\ref{lem:period-f-fn}(ii).
 
iii) Case $q>1$ and $m=2^d r$ for some $r>q$, $r$ odd. 
By Lemma~\ref{lem:period-f-fn}(i), the point $x$ is of period 
$q$ for $g:=f^{2^d}$. Since $q$ is odd and greater than~$1$, 
$g$ has a periodic point
$y$ of period $r$ by \eqref{eq:odd-Sharkovsky}. 
According to Lemma~\ref{lem:period-f-fn}(ii),
there exists an integer $e\in\Lbrack 1,d\Rbrack$ such that $y$ is of 
period $2^e r$ for $f$.
If $e=d$, then $f$ has a periodic point of period $m$. 
Otherwise, we set $r':=2^{d-e} r$. The map 
$f$ has a periodic point of period $2^e r$ with $r$ odd,
and the integer $m$ can be written as $m=2^e r'$ with $r'$ even.
Then the case (ii) above implies that $f$ has a periodic point of
period $m$.

This concludes the proof.
\end{proof}

%**********************************************************************
\subsection{Interval maps of all types}

Because of the structure of Sharkovsky's order, Theorem~\ref{theo:Sharkovsky}
implies that the set of periods of an interval map is of the form either 
$\{m\in\IN\mid m\unrhd n\}$ for some $n\in\IN$ or
$\{2^k\mid k\ge 0\}$. This motivates the next definition.

\begin{defi}\index{type (for Sharkovsky's order)}
Let $n\in\IN\cup\{2^\infty\}$.
An interval map $f$ is \emph{of type $n$ (for
Sharkovsky's order)} if the periods of the periodic points of $f$
form exactly the set $\{m\in\IN\mid m\unrhd n\}$, where the notation
$\{m\in\IN\mid m\unrhd 2^\infty\}$ stands for $\{2^k\mid k\ge 0\}$.
\label{notation:2infty}
\end{defi}

Every interval map has a type. Conversely, there exist maps of all types.
This result was shown by Sharkovsky in \cite{Sha}
for integer types and in \cite{Sha2} for type $2^\infty$.
We are going to exhibit interval maps of all types. Some of these examples
will be referred to in other chapters. We first
state a lemma, which is a partial converse of 
Lemma~\ref{lem:cycle-periodic-point}.

\begin{lem}\label{lem:graph-and-type}
Let $f$ be an interval map. Let $\{x_1<\cdots< x_n\}$ be a
periodic orbit of period $n>1$, and let $G$ denote
the graph of this periodic orbit. Suppose that $f|_{[x_i,x_{i+1}]}$ is monotone
for all $i\in\Lbrack 1, n-1\Rbrack$.
If $f$ has a periodic point of period $m$ in $[x_1,x_n]$, then 
\begin{itemize}
\item either $G$ contains a primitive cycle of length $m$, 
\item or $m$ is even and $G$ contains a
primitive cycle of length $m/2$.
\end{itemize}
\end{lem}

\begin{proof}
Let $y\in [x_1,x_n]$ be a periodic point of period $m$. If 
$y\in\{x_1,\ldots,x_n\}$, then $n=m$ and $G$ contains  a primitive cycle of 
length $m$ by Lemma~\ref{lem:fundamental-cycle}. From now on, we
suppose that $y\notin\{x_1,\ldots,x_n\}$.
We show by induction that for all $k\ge 0$ there is a unique vertex $J_k$ 
in $G$ such that $f^k(y)\in J_k$,  and in addition $J_k\to J_{k+1}.$

\noindent$\bullet$ There is a unique vertex $J_0$ containing $y$ because
$y\in [x_1,x_n]\setminus \{x_1,\ldots,x_n\}$.

\noindent$\bullet$ 
Suppose that $J_k=[x_i,x_{i+1}]$ is already defined. Since $f$ is
monotone on $J_k$, the point $f^{k+1}(y)$ belongs to 
$f(J_k)=\langle f(x_i),f(x_{i+1})\rangle$. Since 
$\langle f(x_i),f(x_{i+1})\rangle$ is a nonempty union of vertices of $G$,
this implies that there exists a vertex $J_{k+1}$ in $G$
such that $f^{k+1}(y)\in J_{k+1}$ and $J_k\to J_{k+1}$. The vertex 
$J_{k+1}$ is unique because $f^{k+1}(y)\notin\{x_1,\ldots,x_n\}$. 
This concludes the induction.

Since $f^m(y)=y$, we have $J_m=J_0$, and thus 
$J_0\to\cdots J_{m-1}\to J_0$
is a cycle in $G$. This cycle is a multiple of a primitive cycle of
length $p$ for some $p$ dividing $m$. Therefore $J_p=J_0$ and
$f^{kp}(y)\in J_0$ for all $k\ge 0$.
Since $f^p|_{J_0}$ is monotone, the set $J:=f^{-p}(J_0)\cap J_0$ is
an interval and $f^{2p}|_J$ is non decreasing. Moreover,
$f^{2kp}(y)$ belongs to $J$ for all $k\ge 0$. If $y\le f^{2p}(y)$, a 
straightforward induction leads to:
$$
y\le f^{2p}(y)\le f^{4p}(y)\le \cdots \le f^{2m}(y).
$$
The reverse inequalities hold if $y\ge f^{2p}(y)$.
In both cases, the fact that $y=f^{2m}(y)$ implies $y=f^{2p}(y)$.
We deduce that $m$ divides $2p$. Since $m$ is a multiple of $p$,
this implies that $m=p$ or $m=2p$.
\end{proof}

\begin{ex}\label{ex:odd-type}
We fix $n\ge 1$.
We are going to build a map $f_p$ of odd type $p=2n+1>1$.
\begin{figure}[htb]
\centerline{\includegraphics{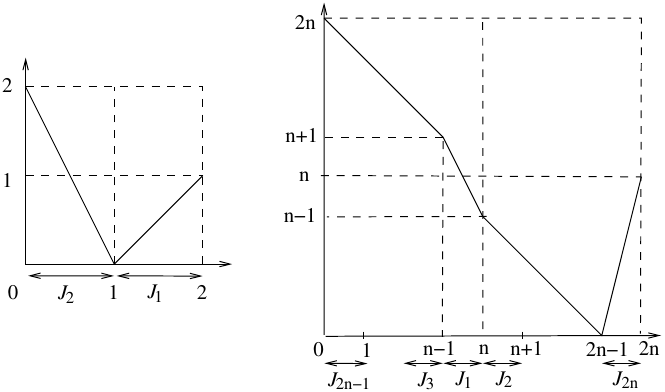}}
\caption{On the left:  an interval map of type 3. On the right: an interval
map of odd type $p=2n+1>3$.}
\label{fig:type-2n+1}
\end{figure}
The map $f_p\colon [0,2n]\to [0,2n]$, represented in Figure~\ref{fig:type-2n+1},
is defined as follows: it
is linear on $[0,n-1]$, $[n-1,n]$, $[n,2n-1]$ and $[2n-1,2n]$, and
$$
f_p(0):=2n,\ f_p(n-1):=n+1,\ f_p(n):=n-1,\ f_p(2n-1):=0,\ f_p(2n):=n.
$$
Notice that $n=1$ is a particular case because $0=n-1$ and $n=2n-1$.
This map satisfies:
\begin{gather*}
\forall k\in\Lbrack 1,n\Rbrack,\ f_p(n-k)=n+k,\\
\forall k\in\Lbrack 0,n-1\Rbrack,\ f_p(n+k)=n-k-1.
\end{gather*}
It follows that $f_p^{2k-1}(n)=n-k$ and  $f_p^{2k}(n)=n+k$
for all $k\in\Lbrack 1,n\Rbrack$.
Thus $f_p^{2n+1}(n)=n$, and the point $n$ is periodic of period $p=2n+1$.

We set $J_{2k-1}:=[n-k,n-k+1]$ and $J_{2k}:=[n+k-1,n+k]$
for all  $k\in\Lbrack 1,n\Rbrack$.
It is easy to check that the graph of the periodic point $n$ is the one given in
Lemma~\ref{lem:graph-n-minimal}, that is:

\centerline{\includegraphics{ruette-fig1}}

This graph does not contain any
primitive cycle of odd length $m\in\Lbrack 2, p-1\Rbrack$. Thus, by
Lemma~\ref{lem:graph-and-type}, $f_p$ has no periodic point of odd period 
$m\in\Lbrack 2, p-1\Rbrack$. This means that $f_p$ is of type $p$.

For further reference, we are going to show that $f_p$ is 
topologically mixing. The transitivity of $f_p$ was shown by
Block and Coven \cite{BCov}.
We are first going to show that, for every subinterval $J$ in $[0,2n]$,
\begin{equation}\label{eq:2J}
\exists i\in\Lbrack 1,2n\Rbrack,\ J\subset J_i\Rightarrow \exists k\ge 1,
\ |f_p^k(J)|\ge 2|J|.
\end{equation}
If $J\subset J_1$, then $|f_p(J)|=2|J|$ because $\slope(f_p|_{J_1})=-2$.
If there exists $i\in\Lbrack 2,2n\Rbrack$ such that
$J\subset J_i$, then $f_p^{2n-i}(J)\subset
J_{2n}$. For all $i\in\Lbrack 2,2n-1\Rbrack$, $\slope(f_p|_{J_i})=-1$ 
and $\slope(f_p|_{J_{2n}})=n$. It follows that
$|f_p^{2n-i+1}(J)|=n|f_p^{2n-i}(J)|=n|J|$.  If $n\ge 2$, then
$|f_p^{2n-i+1}(J)|\ge 2|J|$. If $n=1$, then  $f_p(J_{2n})=J_1$, and thus
$f_p^{2n-i+1}(J)\subset J_1$ and $|f_p^{2n-i+2}(J)|\ge 2|J|$.  In both
cases, there exists an integer $k\ge 1$ such that $|f_p^k(J)|\ge
2|J|$. This proves \eqref{eq:2J}.

Let $J$ be a non degenerate subinterval of $[0,2n]$. 
If, for all integers $k$, the interval $f_p^k(J)$ does not meet
$\{0,1,\ldots,2n\}$, then \eqref{eq:2J} implies that the length of
$f_p^k(J)$ grows to infinity, which is impossible. Thus there exists an
integer $k\ge 0$ such that $f_p^k(J)$ contains one of the points
$0,1,\ldots,2n$. Since $\{0,1,\ldots, 2n\}$ is a periodic orbit,
$k$ can be chosen such that $0\in
f_p^k(J)$. Moreover, $f_p^k(J)$ is a non degenerate interval according to
the definition of $f_p$.
The point $0$ is fixed for $g:=(f_p)^p$ and $g$ is of
slope $4n>1$ on $\left[0,\frac{1}{4n}\right]$. Thus, by
Lemma~\ref{lem:repulsive-fixed-point}, there exists $i\ge 0$ such that
$g^i(f^k(J))\supset  \left[0,\frac{1}{4n}\right]$, and hence 
$g^{i+1}(f^k(J))\supset [0,1]$.  Moreover, for every $i\in\Lbrack 1,p-1\Rbrack$,
$$
J_{p-2}\to J_{p-1}\to\underbrace{J_1 \to \cdots \to J_1}_{p-1-i \rm\ arrows} \to J_2 \cdots \to J_i
$$
is a path of length $p$ from $J_{p-2}=[0,1]$ to $J_i$ in the graph of
the periodic point $n$. This implies that $f_p^p([0,1])\supset \bigcup_{i=1}^{p-1}
J_i= [0,2n]$. Therefore, $f_p^{p(i+2)+k}(J)=[0,2n]$. We conclude that $f_p$
is  topologically mixing.
\end{ex}

\begin{ex}\label{ex:all-types}\index{square root of a map}
We are going to build interval maps of type $n$ for all integers
$n\in\IN$, following the construction in \cite{Ste}.
We start with the definition of the so-called \emph{square root} of a map.
If $f\colon [0,b]\to [0,b]$ is an interval map
and $\delta\in [0,b]$, the \emph{square root}
of $f$ (more precisely, one realization of the square root of $f$)
is the continuous map $g\colon [0,2b+\delta]\to [0,2b+\delta]$ defined by:
\begin{itemize}
\item $\forall x\in [0,b]$, $g(x):=f(x)+(b+\delta)$,
\item $\forall x\in [b+\delta, 2b+\delta]$, $g(x):=x-(b+\delta)$,
\item $g$ is linear on $[b,b+\delta]$.
\end{itemize} 
The map $g$ is not well defined if $\delta=0$ and $g(b)>0$.
The value chosen for $\delta$ is usually $\delta=0$ if $g(b)=0$,
and $\delta=b$ otherwise.
This construction is represented in Figure~\ref{fig:type-2n} with $\delta=b$.
This map satisfies: 
\begin{gather}
\forall x\in[0,b],\ g^2(x)=f(x),\label{eq:squareroot}\\
g([0,b])\subset [b+\delta,2b+\delta]\text{ and } g([b+\delta,2b+\delta])=[0,b].\label{eq:squareroot2}
\end{gather}
\begin{figure}[htb]
\centerline{\includegraphics{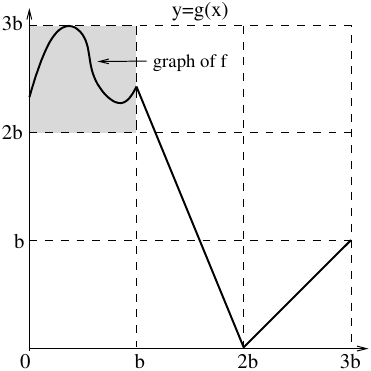}}
\caption{The map $g$ is the square root of $f$. If $f$ is of type $n$, $g$ is of type $2n$.}
\label{fig:type-2n}
\end{figure}
It is clear that $g$ has a unique fixed point $c$, and that $c\in [b,b+\delta]$.
Moreover, $\lambda:=\slope(g|_{[b,b+\delta]})<-1$.
Thus, if $x,g(x),\ldots, g^k(x)$ belong to $[b,b+\delta]$, then 
$|g^k(x)-c|\ge |\lambda|^k|x-c|$. This implies that, for all
$x\in [b,b+\delta]\setminus\{c\}$, there exists $k\ge 0$
such that $g^k(x)\in [0,b]\cup [b+\delta,2b+\delta]$. Thus, by \eqref{eq:squareroot2},
all periodic orbits of $g$, except $c$, have at least one point in $[0,b]$ and 
are of even period.
Moreover, \eqref{eq:squareroot} implies that a point 
$x\in [0,b]$  is a periodic point of period $2m$ for $g$ if and only if it is 
a periodic point of period $m$ for 
$f$. We deduce that, if $f$ is of type $n$, then $g$ is of type $2n$.

With this procedure, it is possible to build an interval map of type $n$
for every positive integer $n$. We write $n=2^dq$ with $d\ge 0$ and $q$ odd.
If $q=1$, we start with a constant map $f\colon [0,1]\to [0,1]$, 
which is of type $1$. If $q>1$, we start with the 
interval map $f_q$ of type $q$ defined in Example~\ref{ex:odd-type}.
Then we build the square root of $f$, then the square root of the
square root, etc. At step $d$, we get an interval map of type $n=2^d q$.
\end{ex}

\begin{ex}\label{ex:type-infty}
We are going to build an interval map of type $2^{\infty}$. 
We follow \cite{Del}; see also \cite{BP}. For all $n\ge 0$, we set
$$
I_n:=\left[1-\frac{1}{3^n},1-\frac{2}{3^{n+1}}\right].
$$
For every $n\ge 0$, let $f_n\colon I_n\to I_n$ be the map of type $2^n$
built in Example~\ref{ex:all-types} and rescaled
(i.e., conjugate by an increasing linear homeomorphism) to fit into $I_n$.
Then the continuous map $f\colon [0,1]\to [0,1]$, illustrated in 
Figure~\ref{fig:type-2-infty}, is defined by:
\begin{itemize}
\item $\forall x\in I_n$, $f(x):=f_n(x)$,
\item $f(1):=1$,
\item $\forall n\ge 0$, $f$ is linear on $\left[1-\frac{2}{3^{n+1}},1-\frac{1}{3^{n+1}}\right]$.
\end{itemize}
\begin{figure}[htb]
\centerline{\includegraphics{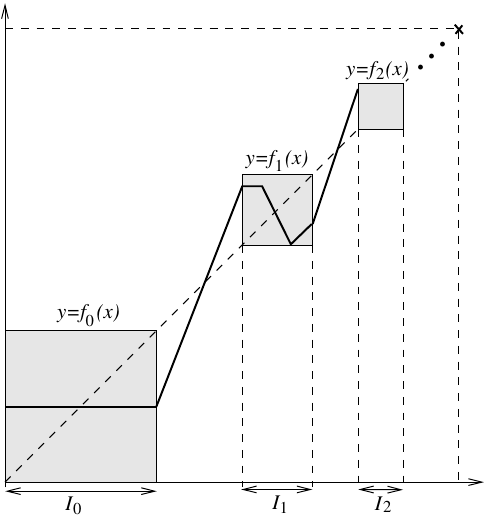}}
\caption{Each map $f_n$ is of type $2^n$, the whole map is of type 
$2^{\infty}$.}
\label{fig:type-2-infty}
\end{figure}

It is obvious that the only periodic points of $f$ in 
$\left[1-\frac{2}{3^{n+1}},1-\frac{1}{3^{n+1}}\right]$ are fixed,
and $x$ is a periodic point of period $p>1$ for $f$ if and only if there is
some $n\ge 0$ such that $x$ is a periodic point of period $p$ for $f_n$.
Therefore the type of $f$ is $2^{\infty}$.

We remark that, for all $x\in [0,1]$, the $\omega$-limit set of $x$ is a 
periodic orbit of period $2^n$ for some integer $n\ge 0$. This is not
always the case for maps of type $2^\infty$. 
In \cite{Del}, there is another example of a map of
type $2^{\infty}$ with an infinite $\omega$-limit set. 
We shall see such a map in depth in Example~\ref{ex:non-chaos-LY-htop0}.
\end{ex}

\begin{rem}\label{rem:alltypesinunimodal}
There is a completely different way of proving that all types are
realized.
It consists of the study of a one-parameter family of interval maps that
exhibits all possible types. The most famous family is the
\emph{logistic}\index{logistic family} family
$f_{\lambda}(x)=\lambda x(1-x)$, where $x\in [0,1]$ and $\lambda\in[0,4]$.
For every 
$n\in\IN\cup\{2^\infty\}$, there exists $\lambda_n\in [0,4]$ such that
$f_{\lambda_n}$ is of type $n$; the map $f_{\lambda_{2^\infty}}$ is
called the \emph{Feigenbaum map}\index{Feigenbaum map}.
More generally, every ``typical'' family of smooth unimodal maps exhibits
all possible types; an interval map $f\colon [0,1]\to [0,1]$ is
\emph{unimodal}\index{unimodal map} if $f(0)=f(1)=0$ and there is 
$c\in (0,1)$ such that $f|_{[0,c]}$ is increasing and $f|_{[c,1]}$ is
decreasing.
The proofs are non-constructive and
rely on quite sophisticate tools like the kneading theory.
The study of such families of interval maps is out of the scope of this book.
See \cite{dMvS, CE, MT2, Guc2, Jon, KH}. 

In \cite[Section 2.2]{ALM}, Alsedà, Llibre and Misiurewicz gave a short proof 
consisting of showing that the family of truncated tent maps
exhibits all types. The truncated tent maps are defined as
$g_{\lambda}(x)=\min (T_2(x),\lambda)$, where $x\in [0,1]$,
$\lambda\in[0,1]$ and $T_2$ is the tent map defined in 
Example~\ref{ex:tent-map}. The proof is non constructive, 
as for families of smooth unimodal maps, but is much simpler.
\end{rem}

%*****************
\subsection*{Remarks on graph maps}

Sharkovsky's Theorem~\ref{theo:Sharkovsky} has motivated a lot of work
aimed at finding characterizations of the set of periods for more
general one-dimensional spaces.

One of the lines of generalization of Sharkovsky's Theorem consists of
characterizing the possible sets of periods of tree maps.
The first remarkable results in this line are
due to Alsedà, Llibre and Misiurewicz \cite{ALM2} and Baldwin
\cite{Bal}. In \cite{ALM2} the characterization of the
set of periods of the maps on the $3$-star with a fixed
branching point, in terms of three linear orderings, was obtained, whereas in
\cite{Bal} the characterization of the set of periods of all
dynamical systems on $n$-stars is given (an $n$-star\index{star} is a 
tree made of $n$ segments glued together
by one of their endpoints at a single point, e.g., 
$S_n:=\{z\in\IC\mid z^n \in [0,1]\}$). Further extensions
were given by Baldwin and Llibre \cite{BL} for tree maps such
that all the branching points are fixed, then by Bernhardt \cite{Bern2} for
tree maps such that all the branching points are periodic.
Finally, Alsedà, Juher and Mumbrú  overcame
the general case of tree maps \cite{AJM2, AJM, AJM3, AJM4}. 
They showed that the set of periods of a tree map is the union of finitely many 
terminal segments of the orders of Baldwin and of a finite set
(for every integer $p\ge 2$, the $p$-order of Baldwin is a partial ordering
on $\IN$,  coinciding with Sharkovsky's order for $p=2$). The precise
statement is quite complicated; we refer to \cite[Theorem 1.1]{AJM3}.

Another direction is to consider topological
graphs which are not trees, the circle being the simplest one. 
Circle maps display a new feature: the set of periods depend on the
degree of the map and, in the case of degree 1, on the rotation interval.

Consider a circle map $f\colon \IS\to \IS$, where $\IS:=\IR/\IZ$, and 
a lifting\index{lifting of a circle map} of $f$, that is,
a continuous map  $F\colon \IR\to\IR$ such that $\pi\circ F=f\circ\pi$,
where $\pi\colon \IR\to\IS$ denotes the canonical projection
($F$ is uniquely defined up to the addition of an integer). 
The \emph{degree}\index{degree of a circle map} of $f$ (or $F$)
is the integer $d\in\IZ$ such that $F(x+1)=F(x)+d$ for all $x\in \IR$.

The characterization of the sets of periods for circle maps of degree
different from $1$ is simpler than the one for the case of degree $1$.
The case of degree different from $1,-1$ and part of the degree $-1$ case
are due to Block, Guckenheimer, Misiurewicz and Young \cite{BGMY}.  See also
\cite[Section 4.7]{ALM}.

\begin{theo}\label{theo:period-degreenot1}
Let $f\colon \IS\to \IS$ be a circle map of degree $d\neq 1$.
\begin{itemize}
\item If $|d|\ge 2$, $d\neq -2$, then the set of periods of $f$ is $\IN$.
\item If $d=-2$, then the set of period of $f$ is either $\IN$ or
$\IN\setminus\{2\}$.
\item If $d\in\{0,-1\}$, then there exists $s\in\IN\cup \{2^\infty\}$ such that
the set of periods of $f$ is $\{m\in\IN\mid m \unrhd s\}$. 
\end{itemize}
Moreover, all cases are realized by some circle maps.
\end{theo}

The  characterization of the sets of periods of circle maps of degree
$1$ is due to Misiurewicz \cite{Mis5} and uses as a key tool the
rotation theory. The reader can refer to \cite{ALM} for an exposition
of rotation theory for (non invertible) circle maps of degree $1$.
The sets of periods of circle maps of degree $1$
contain the set of all denominators of all rational numbers (not necessarily
written in irreducible form) in the interior of an interval of the
real line. As a consequence, these sets of periods cannot be expressed
in terms of a finite collection of orderings.

\begin{theo}\label{theo:periods-degree1}
Let $f\colon \IS\to \IS$ be a circle map of degree $1$, and let $[a,b]$
be the rotation interval of a lifting of $f$.
Then there exist $s_a, s_b\in\IN \cup\{2^\infty\}$ such that the set of
periods of $f$ is equal to
$S(a,s_a)\cup M(a,b)\cup S(b,s_b)$,
where
\begin{itemize}
\item
$M(a,b):=\{q\in\IN\mid\exists p\in\IZ,\ \frac{p}q\in (a,b)\}$.
\item
$S(x,s):=\emptyset$ if $a\in\IR\setminus\IQ$ and $S(x,s):=\{nq\mid n\unrhd s\}$
if $a=\frac{p}q$ with $p\in\IN$, $q\in\IZ$ and $\gcd(p,q)=1$.
\end{itemize}
Moreover, all cases are realized.
\end{theo}

Finding a characterization of the sets of periods of graphs maps
when the graph is neither a tree nor the circle is a big challenge and in
general it is not known what the sets of periods may look like.
Only two cases have been studied: the graph shaped  like $\sigma$
\cite{LL}, and the graph shaped like 
8~\cite{LPR}, with the restriction, in both cases, that the map
fixes the branching point.
This assumption greatly facilitates the study (similarly,
tree maps fixing all the branching points were dealt with first).

A rotation theory has been developed by Alsedà and the author 
for maps of degree $1$ on graphs containing a single loop \cite{AR}. 
It leads to results similar to, but weaker than, the ones obtained from the
rotation theory for circle maps. They
give information about the periods, but this is far from leading to
a characterization of the sets of periods, even in the simplest case of 
the graph $\sigma$ \cite{AR3}.

%***********************************************************************
\section{Relations between types  and horseshoes}

If an interval is mapped across itself twice, the effect on the dynamics
is similar to Smale's horseshoe for two-dimensional homeomorphisms \cite{Sma}.
This leads to the following definition. The name \emph{horseshoe} for
interval maps was given by Misiurewicz \cite{Mis}, but the notion
was introduced much earlier by Sharkovsky under the name of
\emph{L-scheme} \cite{Sha}.

\begin{defi}[horseshoe]\index{horseshoe}\label{defi:horseshoe}
Let $f$ be an interval map. If $J_1,\ldots, J_n$ are
non degenerate closed intervals with pairwise disjoint interiors
such that  $J_1\cup\cdots \cup J_n\subset f(J_i)$
for all $i\in\Lbrack 1,n\Rbrack$, then
$(J_1,\ldots,J_n)$ is called an \emph{$n$-horseshoe}, or simply a 
\emph{horseshoe} if $n=2$. If in addition the intervals are disjoint, 
$(J_1,\ldots,J_n)$ is called a \emph{strict} $n$-horseshoe.
\end{defi}

\begin{rem}
The definition of horseshoe slightly varies in the literature.
For some authors, a horseshoe is made of disjoint closed subintervals, or
is a partition of an interval into subintervals such that
the image of every subinterval contains the whole interval (thus,
the subintervals forming the horseshoe are disjoint but not 
closed). The definition above follows \cite{ALM}. 
For Block and Coppel, an interval map with a horseshoe  is called 
\emph{turbulent}\index{turbulent} \cite{BCop2};  this terminology was
suggested by Lasota and Yorke \cite{LasY}. However, \emph{turbulent} may
refer to a point with an infinite $\omega$-limit set 
(see e.g. \cite{Del}), and so this word might be confusing. 
\end{rem}

Sometimes it will be useful to boil down to a particular form of
a horseshoe, as given by the next lemma.

\begin{lem}\label{lem:particularhorseshoe}
Let $f$ be an interval map and $(J,K)$ a horseshoe. Then there exist
points $u,v,w$ such that $f(u)=u$, $f(v)=w$, $f(w)=u$ and, either
$u<v<w$, or $u>v>w$. Note that $\langle u,v\rangle, \langle v,w\rangle$
form a horseshoe.
\end{lem}

\begin{proof}
We assume $J\le K$. Let $a,b\in J$ be such that $f(a)=\min J$ and
$f(b)=\max K$. We have $\langle a,b\rangle\subset J$ and
$f(\langle a,b\rangle)\supset J\cup K$. 

\noindent$\bullet$ \textbf{Case 1: $a<b$.} Let $u\in [a,b]$ be a fixed point
($u$ exists by Lemma~\ref{lem:fixed-point}); 
$u\ne b$ because $f(b)\notin [a,b]$, hence $u\notin K$. Since
$f(K)\supset J\supset [a,b]$, there is $w\in K$ such that $f(w)=u$
(by the intermediate value theorem). Moreover $f([u,b])\supset [u, \max K]\supset K$, 
thus there
is $v\in[u,b]$ such that $f(v)=w$. We have $u<v$ by definition 
(note that $u=v$ is impossible because $u\notin K$ and $f(v)\in K$) and
$v\le w$ because $v\in J$ and $w\in K$. Moreover $v\ne w$ because 
$f(w)=u<v\le w=f(v)$.

\noindent$\bullet$ \textbf{Case 2: $b<a$.} Let $u\in K$ be a fixed point; $u>a$ because
$f(a)\notin K$. Let $w\in [b,a]$ be such that $f(w)=u$. Since $f([w,a])
\supset [\min J, u]\supset J$, there is $v\in [w,a]$ such that
$f(v)=w$. We have $w\le v\le u$, and equalities $v=u$, $w=v$ are not
possible because $v\le a<u$ and $f(w)=u>v\ge w=f(v)$.
\end{proof}

The next lemma is a straightforward consequence of 
Lemma~\ref{lem:chain-of-intervals}(iii).

\begin{lem}\label{lem:horseshoe-fn}
Let $f$ be an interval map and $(J_1,\ldots, J_p)$
a $p$-horseshoe for $f$. Then, for all $n\ge 1$,
\begin{enumerate}
\item $\forall i,j\in\Lbrack 1,p\Rbrack$, $J_i$ covers $p^{n-1}$ times $J_j$
for $f^n$,
\item $f^n$ has a $p^n$-horseshoe.
\end{enumerate}
\end{lem}

We shall see in Section~\ref{sec:entropy-horseshoes} that  horseshoes
are intimately related to entropy, but what interests us now is the 
relationship between horseshoes and the periods of periodic points.
We first show that a map with a horseshoe is of type $3$.
This result appears as part of a proof due to Sharkovsky 
\cite[Lemma~4]{Sha}. It was also stated by Block and Coppel \cite{BCop2}.

\begin{prop}\label{prop:turbulent-all-periods}
An interval map $f$ with a horseshoe has periodic points of all periods.
\end{prop}

\begin{proof}
Let $(J,K)$ be a horseshoe for $f$. First, we assume that $J$ and $K$ are
disjoint. Let $n\ge 1$. Applying 
Lemma~\ref{lem:chain-of-intervals}(ii) to the chain of intervals
$(I_0,\ldots, I_n)$ with $I_i:=K$ for all $i\in\Lbrack 1,n-1\Rbrack$ and
$I_0=I_n:=J$, we see there exists a periodic point $x\in J$ such that 
$f^n(x)=x$ and $f^k(x)\in K$ for all $k\in\Lbrack 1,n-1\Rbrack$. 
The fact that $J$ and $K$ are disjoint implies that the period of $x$ is 
exactly $n$.

Now we assume that $J$ and $K$ have a common endpoint.
We write $J=[a,b]$ and $K=[b,c]$
(we may suppose with no loss of generality that $J$ is on the left of $K$).
If $b$ is a fixed point, we set
$$
d:=\min\{x\ge b\mid f(x)\in\{a,c\}\}.
$$
It follows that $d>b$ and the image of $[b,d)$ contains neither $a$ nor $c$. 
Thus $f([d,c])$ contains $a$ and $c$ because $[a,c]\subset f([b,c])$,
so $[a,c]\subset f([d,c])$ by connectedness. We deduce that
$(J,[d,c])$ is a strict horseshoe. The first part of the proof implies 
that $f$ has periodic points of all periods.

Suppose now that $b$ is not a fixed point. Applying
Lemma~\ref{lem:chain-of-intervals}(ii) to the chain of intervals
$(J,K,K,J)$, we see that there exists a periodic point $x\in J$ such that 
$f^3(x)=x$, 
$f(x)\in K$ and $f^2(x)\in K$.  The period of $x$
divides $3$, and thus it is equal to $1$ or $3$. If $x$ is a fixed point, 
then $x\in J\cap K=\{b\}$, which is impossible because $b$ is not fixed. 
Thus  $x$ is of period $3$. Then $f$ has  periodic points of all
periods according to Sharkovsky's Theorem~\ref{theo:Sharkovsky}.
\end{proof}

An interval map with a periodic point of period $3$ may have no horseshoe.
Such a map will be built in
Example~\ref{ex:htop-mixing-log-lambda}.  However, if $f$ has a
periodic point of odd period greater than $1$, then $f^2$ has a
horseshoe. This result was underlying
in a paper of Block \cite{Bloc3} and was stated by 
Osikawa and Oono in \cite{OO}; see also \cite{BCop2}.

\begin{prop}\label{prop:odd-period-turbulent}
Let $f\colon I\to I$ be an interval map. If $f$ has a periodic point
of odd period greater than $1$, then there exist two intervals $J,K$ 
containing no endpoint of
$I$ and such that $(J,K)$ is a strict horseshoe for $f^2$.
\end{prop}

\begin{proof}
Let $p$ be the least odd integer different from $1$ such that $f$ has
a  periodic point of period $p$ and let $x$ be a periodic point of
period $p$. According to Lemma~\ref{lem:graph-n-minimal}, there
exists a point $x_0$ in the orbit of $x$ such that
the points $x_i:=f^i(x_0), 0\le i\le p-1$, are ordered as:
$$
x_{p-1}<x_{p-3}<\cdots<x_2<x_0<x_1<\cdots<x_{p-2}
$$
or in the reverse order. Suppose that the order above holds, the other case
being symmetric.

The interval $f([x_0,x_1])$ contains $[x_2,x_0]$, which implies that 
there exists $d$ in $(x_0,x_1)$ such that $f(d)=x_0$, and hence
$d<f^2(d)=x_1$. Since $f^2([x_{p-1},x_{p-3}])$ contains $[x_{p-1},x_1]$,
there exists a point $a\in (x_{p-1},x_{p-3})$ such that $f^2(a)>d$.
Then $f^2([a,x_{p-3}])\supset [x_{p-1},d]$,  and thus there exists
$b\in (a,x_{p-3})$ such that $f^2(b)<a$. Similarly, there exists
$c\in (x_{p-3},d)$ such that $f^2(c)<a$ because 
$f^2([x_{p-3},d])\supset [x_{p-1},f^2(d)]\supset [x_{p-1},d]$. Then $J:=[a,b]$ and $K:=[c,d]$ are
disjoint intervals and form a horseshoe for $f^2$.
Finally, $J$ and $K$ do not contain any endpoint of $I$ because $x_{p-1}<a$ 
and $d<x_1$.
\end{proof}

We end this section with two small results, related to horseshoes and
periodic points of odd period; both will be referred to later. 
The first one states that, if $f$ has no 
horseshoe, every orbit splits into 
two sets $U$ and $D$ with $U\le D$ such that all points in $U$ (resp. $D$) are 
going ``up'' (resp. ``down'') under the action of $f$. 
The second one is a tool to prove the existence of periodic 
points when only partial information on the location of the points is known. 

The next result, already implicit in 
a paper of Sharkovsky \cite[proof of Lemma~4]{Sha}, was proved by Li, 
Misiurewicz, Pianigiani and Yorke under a slightly weaker assumption 
\cite[Corollary 3.2]{LMPY}.

\begin{lem}\label{lem:U<D}
Let $f$ be an interval map with no horseshoe, and let $x_0$ be a point. 
Let $U(x_0):=\{x\in \CO_f(x_0) \mid f(x)\ge x\}$ and
$D(x_0):=\{x\in \CO_f(x_0)\mid f(x)\le x\}$. If these two sets are nonempty, 
then $\sup U(x_0)\le \inf D(x_0)$ and there exists a fixed point $z\in
[\sup U(x_0),\inf D(x_0)]$.
\end{lem}

\begin{proof} 
We set $x_n:=f^n(x_0)$ for all $n\ge 0$. Let $n,m$ be integers such that
$x_n\in U(x_0)$ and $x_m\in D(x_0)$. We are going to show that $x_n\le x_m$.
Suppose on the contrary that $x_n> x_m$. 
We assume that $m>n$, the case $m<n$ being similar.
The point $x_n$ is not fixed because $x_m=f^{m-n}(x_n)> x_n$. Thus, 
according to the definition of $U(x_0)$ and $D(x_0)$, we have
\begin{equation}\label{eq:xmn}
f(x_m)\le x_m < x_n<f(x_n).
\end{equation}
By continuity, there exists a fixed point in $[x_m,x_n]$. Let
$y$ be the maximal fixed point in $[x_m,x_n]$. Then $y< x_n$ and, since
$f(x_n)>x_n$,
\begin{equation}\label{eq:fy}
\forall x\in (y,x_n],\ f(x)>x.
\end{equation}
By \eqref{eq:xmn}, there exists an integer $k\in\Lbrack n+1,m\Rbrack$ such 
that $x_i>y$ for all $i\in\Lbrack n,k\Rbrack$ and $x_{k+1}\le y$.
We show by induction on $i$ that $x_i< x_{k}$
for all $i\in\Lbrack n,k\Rbrack$.

\medskip
\noindent$\bullet$
Case $i=n$: since $f(x_{k})=x_{k+1}\le y<x_{k}$, the point
$x_{k}$ does not belong to $(y,x_n]$ by \eqref{eq:fy}, and thus
$x_n<x_{k}$.

\medskip
\noindent$\bullet$
Suppose that $x_i<x_{k}$ for some $i\in\Lbrack n,k-1\Rbrack$. If 
$x_{i+1}\ge x_{k}$, then $([y,x_i],
[x_i,x_{k}])$ is a horseshoe for $f$, which is a contradiction. Hence
$x_{i+1}< x_{k}$.

\medskip
For $i=k$, the induction statement is that $x_{k}<x_{k}$, which is absurd. 
Hence $x_n\le x_m$. We deduce that $\sup U(x_0)\le \inf D(x_0)$.
Moreover, the definitions of $U(x_0), D(x_0)$ imply that
$f(\sup(U(x_0)))\ge \sup U(x_0)$ and $f(\inf(D(x_0)))\le \inf D(x_0)$. Thus
there exists a fixed point $z\in [\sup U(x_0),\inf D(x_0)]$ 
by continuity.
\end{proof}

The next result was shown by Li, Misiurewicz, Pianigiani and Yorke \cite{LMPY2}.

\begin{prop}\label{prop:xn-x0-x1}
Let $f$ be an interval map and let $x$ be a point. Let $p\ge 3$ be an odd
integer and suppose that either $f^p(x)\le x<f(x)$ or $f^p(x)\ge x> f(x)$. 
Then $f$ has a periodic point of period $p$.
\end{prop}

\begin{proof}
We assume that $f^p(x)\le x< f(x)$, the
case with reverse inequalities being symmetric. 
We also assume that $f$ has no horseshoe, otherwise $f$ has periodic
points of all periods by Proposition~\ref{prop:turbulent-all-periods}.

We set $x_n:=f^n(x)$ for all $n\ge 0$. We define the sets 
$$U:=\{x_n\mid x_{n+1}\ge x_n,\ n\in\Lbrack 0,p\Rbrack\}\quad\text{and}\quad
D:=\{x_n\mid x_{n+1}\le x_n,\ n\in\Lbrack 0,p\Rbrack\}.$$
By assumption, $x_p\le x_0<x_1$, which implies that 
$x_0\in U$ and 
\begin{equation}\label{eq:UDj}
\text{there exists }j\in\Lbrack 1, p-1\Rbrack\text{ such that }x_{j+1}<x_j,
\end{equation}
so $x_j\in D$. Since $U$ and $D$ are not empty, 
Lemma~\ref{lem:U<D} implies that 
$$\max U\le \min D\text{ and
there exists a fixed point }z\in [\max U,\min D].
$$
If $x_i=z$ for some  $i\in\Lbrack 0,p\Rbrack$, then $x_p=z\ge \max U\ge x_0$.
Since $x_p\le x_0$, this implies $x_0=z$,
which is a contradiction because $x_0$ is not a fixed point. 
We deduce that $\max U<z<\min D$, and thus
$$
x_p\le x_0\le \max U<z<\min D\le x_j.
$$
We claim that there exists $k\in\Lbrack 0,p-1\Rbrack$ such that
$$
\text{either }x_k, x_{k+1}\in U\quad\text{or}\quad x_k, x_{k+1}\in D.
$$
Otherwise, all the points $x_i$ with even index
$i\in\Lbrack 0,p\Rbrack$ would be in $U$ (because $x_0\in U$)
and all the points $x_i$ with 
odd index $i\in\Lbrack 0,p\Rbrack$ would be in $D$, and thus
$x_0<z<x_p$ because $p$ is odd. This would
contradict the assumption that $x_p\le x_0$. Therefore the claim holds, 
which implies that
\begin{equation}\label{eq:xkk+1}
\text{either }x_k\le x_{k+1}<z \quad\text{or}\quad z<x_{k+1}<x_k.
\end{equation}
We assume that the case $x_k\le x_{k+1}<z$ holds in \eqref{eq:xkk+1}, the
other case being symmetric. We set $J_k:=[x_k,\max U]$ and
$J_i:=\langle x_i, z\rangle$ for all $i\in\Lbrack 0,p\Rbrack$ with $i\neq  k$.
Then $f(J_k)\supset [x_{k+1},\min D]\supset [x_{k+1},z]$ and 
$f(J_i)\supset \langle x_{i+1},z\rangle$ for all
$i\in\Lbrack 0,p\Rbrack$ with $i\neq  k$.
Then $(J_0,\ldots,J_p)$ is a chain of intervals. Moreover, we have 
$J_0\subset J_p$. Thus there exists $y\in J_0$ such that $f^p(y)=y$
by Lemma~\ref{lem:chain-of-intervals}(ii). Let $q$ be the period of $y$;
this is a divisor of $p$.
If $q=1$, then $y\in J_0\cap J_j=[x_0,z]\cap [z,x_j]$ (recall that $j$ is 
such that $x_j>z$ by \eqref{eq:UDj}), and hence $y=z$. 
But this is not possible because $y\in J_k=
[x_k,\max U]$, with $\max U<z$. We deduce that  $q>1$. 
Since $p$ is odd, then $q$ is odd too, and $1<q\le p$. Then
Sharkovsky's Theorem~\ref{theo:Sharkovsky} gives the conclusion. 
\end{proof}

%*******************************************************************
\section{Types of transitive and mixing maps}

We saw that a mixing interval map has a periodic point of odd
period greater than $1$ (Theorem~\ref{theo:summary-mixing}).
Moreover, Example~\ref{ex:odd-type} shows that, for every odd $q>1$,
there exists a mixing map of type $q$. If an interval map $f$ is transitive 
but not mixing, then, according to Theorem~\ref{theo:summary-transitivity},
there exists a subinterval $J$ such that $f^2|_J$ is mixing, and thus
$f^2$  is of type $q$ for some odd $q>1$. Actually, $q$ is always
equal to $3$ in this case, 
which implies that $f$ is of type $6$. This result was proved by Block and 
Coven \cite{BCov}; it is also a consequence of a result of Blokh \cite{Blo8}. 
We start with a lemma, stated in \cite{BCov}.

\begin{lem}\label{lem:non-turbulent}
Let $f\colon [a,b]\to [a,b]$ be a transitive interval map. If $f$ has no
horseshoe, then it has a unique fixed point. Moreover, this fixed
point is neither $a$ nor $b$.
\end{lem}

\begin{proof}
Suppose that $f$ is transitive and has at least two
fixed points. Then Theorem~\ref{theo:summary-transitivity} implies that
$f$ is topologically mixing. The set of fixed points $P_1(f)$ has 
an empty interior by transitivity, and it is a closed set. This implies
that there exist two points $x_1<x_2$ in $P_1(f)$ such that 
$(x_1,x_2)\cap P_1(f)=\emptyset$ and thus, either
$$
\forall x\in (x_1,x_2),\ f(x)<x,
$$ 
or
\begin{equation}\label{eq:fx>x}
\forall x\in (x_1,x_2),\ f(x)>x.
\end{equation} 
We assume that \eqref{eq:fx>x} holds, the other case being symmetric.
If 
$$
\forall x\in (x_1,b],\ f(x)>x_1,
$$
then the interval $[x_1,b]$ is invariant,
which is impossible by transitivity  except if $x_1=a$. In this
case, $a$ is a non accessible endpoint because $a\notin
f((a,b])$, and thus there exists a sequence of fixed points
that tend to $a$ by Lemma~\ref{lem:accessibility}. But this contradicts
the choice of $x_1$ and $x_2$. We deduce that there exists $t\in (x_1,b]$
such that $f(t)\le x_1$. Actually, $t$ belongs to $[x_2,b]$ because of
\eqref{eq:fx>x}. Since $f(x_2)=x_2>x_1$, there exists $z\in [x_2,t]$
such that $f(z)=x_1$ by the intermediate value theorem. Thus we can define
$$
z:=\min\{x\in [x_1,b]\mid f(x)=x_1\}.
$$
Actually $z \in [x_2,b]$ because of \eqref{eq:fx>x}.
If $f(x)\neq z$ for all $x\in (x_1,z)$, then $f(x)<z$ 
for all $x\in (x_1,z)$ (because $f(x_1)=x_1<z$), and 
the minimality of $z$ implies that $f(x)>x_1$
for all $x\in (x_1,z)$. Thus
the non degenerate interval $[x_1,z]$ is 
invariant and $z\notin f([x_1,z])$, which is impossible because $f$ is
transitive. We deduce
that there exists $y\in (x_1,z)$ such that $f(y)=z$. If we set $J:=[x_1,y]$
and $K:=[y,z]$, then $(J,K)$ is a horseshoe.

If $f$ is transitive and has no horseshoe, what precedes implies that
$f$ has at most one fixed point.  Thus $f$ has a unique fixed
point according to  Lemma~\ref{lem:fixed-point}. If $a$ (resp. $b$)
is the unique fixed point of $f$, then $f(x)<x$ for all $x\in (a,b]$
(resp. $f(x)>x$ for all $x\in [a,b)$), and thus $f$ is not onto.
This is impossible because $f$ is transitive, so we conclude that
the unique fixed point of $f$ is neither $a$ nor $b$.
\end{proof}

\begin{prop}\label{prop:transitivity-type}
Let $f\colon I\to I$ be a transitive interval map. Then $f^2$ has a
horseshoe and $f$ has a periodic point of period $6$. Moreover,
\begin{itemize}
\item if $f$ is topologically mixing, then it is of type $p$ for some odd $p>1$,
\item if $f$ is transitive but not mixing, then it is of type $6$.
\end{itemize}
\end{prop}

\begin{proof}
If $f$ is topologically mixing, it has a periodic point of odd period $q>1$ by
Theorem~\ref{theo:summary-mixing}. 
Sharkovsky's Theorem~\ref{theo:Sharkovsky} implies that the type of $f$
is an odd integer $p$ in $\Lbrack 3,q\Rbrack$ and $f$ has a periodic point of
period  $6$. Moreover, $f^2$ has a
horseshoe by Proposition~\ref{prop:odd-period-turbulent}.

If $f$ is transitive but not topologically mixing, then  it has no periodic 
point of odd period greater than $1$ by
Theorem~\ref{theo:summary-mixing}, and thus the type of $f$ is at least 6
for Sharkovsky's order. Moreover, according to
Theorem~\ref{theo:summary-transitivity}, there exists a fixed
point $c\in I$ which is not an endpoint of $I$ and such that, if we set
$J:=[\min I,c]$ and $K:=[c,\max I]$, the subintervals $J,K$ are invariant under
$f^2$, and both maps $f^2|_J$, $f^2|_K$ are topologically mixing. Then $f^2|_J$
is transitive and has a fixed endpoint, and thus it has a horseshoe
according to Lemma~\ref{lem:non-turbulent}. Therefore  $f^2$ has a periodic
point of
period $3$ by Proposition~\ref{prop:turbulent-all-periods}. The period of
this point for $f$ cannot be an odd integer, and thus it is equal to $6$.
We conclude that the type of $f$ is $6$.
\end{proof}

%**********************************************************************
%Topological entropy
\chapter{Topological entropy}\label{chap4}

\section{Definitions}\index{topological entropy}\index{entropy}

The notion of topological entropy for a dynamical system 
was introduced by Adler, Konheim and McAndrew \cite{AKM}.
Topological entropy is a conjugacy invariant. The aim of this
first section is to recall briefly the definitions and introduce the
notation used in the sequel, without entering into details.
The readers who are not
familiar with topological entropy can refer to \cite{Wal} or \cite{DGS}.

\subsection{Definition with open covers}\label{subsec:htop-covers}

Let $(X,f)$ be a topological dynamical system.
A \emph{finite cover}\index{cover} is a collection of sets $\CC=
\{C_1,\ldots,C_p\}$ such that $C_1\cup\cdots\cup C_p=X$. It is an \emph{open 
cover}\index{open cover} if in addition the sets $C_1,\ldots,C_p$ are open.
A \emph{partition}\index{partition} is a cover made of pairwise disjoint
sets.
The topological entropy is usually defined for open covers only.
Nevertheless we give the definition for any finite cover because we shall 
sometimes deal with the entropy of covers composed of intervals 
which are not open.

Let $\CC=\{C_1,\ldots,C_p\}$ and $\CD=\{D_1,\ldots, D_q\}$ be two covers.
The cover $\CC\vee\CD$
\label{notation:vee}
is defined by
$$
\CC\vee\CD:=\{C_i\cap D_j\mid i\in\Lbrack 1,p\Rbrack,\ j\in\Lbrack 1,q\Rbrack,
\ C_i\cap D_j\ne\emptyset\}.
$$
We say that $\CD$ is \emph{finer}\index{finer cover} than $\CC$, and
we write $\CC\prec \CD$,
\label{notation:finer}
if  every element of $\CD$ is included in an element of $\CC$.
Let $N(\CC)$ denote the minimal cardinality of a subcover of $\CC$, that
is,
$$
\label{notation:NU}
\index{N(U)@$N(\CU)$}
N(\CC):=\min\{n\mid \exists i_1,\ldots,i_n\in\Lbrack1,p\Rbrack\},
\ X=C_{i_1}\cup\cdots\cup C_{i_n}\}.
$$
Then, for all integers $n\ge 1$, we define
$$
\label{notation:coverUn}\label{notation:NnU}
\index{Un@$\CU^n$}\index{Nn(U,f)@$N_n(\CU,f)$}
N_n(\CC,f):=N\left(\CC\vee f^{-1}(\CC)\vee\cdots\vee f^{-(n-1)}(\CC)\right).
$$
If there is no ambiguity on the map, 
$\CC^n$ will denote 
$\CC\vee f^{-1}(\CC)\vee\cdots\vee f^{-(n-1)}(\CC)$.
Note that $N(\CC)\le \#\CC$. Moreover, if $\CP$ is a partition (not
containing the empty set), then $\CP^n$ is a partition too, 
and $N(\CP^n)=\#(\CP^n)$ for all $n\ge 1$.

\begin{lem}\label{lem:subadditive}
Let $(a_n)_{n\ge 1}$ be a sub-additive\index{sub-additive sequence} 
sequence, that is, $a_{n+k}\le a_n+a_k$ for all $n\ge 1$ and all $k\ge 1$.\index{sub-additive sequence} 
Then $\lim_{n\to+\infty}\frac{1}{n}a_n$
exists and is equal to $\inf_{n\ge 1}\frac{1}{n}a_n$.
\end{lem}

\begin{proof}
The inequality 
\begin{equation}\label{eq:subadd1}
\liminf_{n\to+\infty}\frac{1}{n}a_n\ge \inf_{n\ge 1}\frac{1}{n}a_n
\end{equation} 
is obvious. Let $k$ be a positive integer.
For every positive integer $n$, there exist integers $q,r$ such that
$n=qk+r$ and $r\in\Lbrack 0,k-1\Rbrack$. The sub-additivity implies that
$a_n\le q a_k+a_r$, and thus $\limsup_{n\to+\infty}\frac{1}{n}a_n\le
\frac{1}{k}a_k$. Therefore, 
\begin{equation}\label{eq:subadd2}
\limsup_{n\to+\infty}\frac{1}{n}a_n\le\inf_{k\ge 1}\frac{1}{k}a_k,
\end{equation} 
and the lemma follows from \eqref{eq:subadd1} and 
\eqref{eq:subadd2}.
\end{proof}

It is easy to show that, for all finite covers $\CC$, the sequence 
$\left(\frac{1}{n}\log N_n(\CC,f)\right)_{n\ge 1}$ 
is sub-additive. Thus Lemma~\ref{lem:subadditive} can be used to
define the \emph{topological entropy} of the cover $\CC$ by:
$$
\label{notation:htopU}
\index{htop(U,f)@$h_{top}(\CU,f)$}
h_{top}(\CC,f):=\lim_{n\to+\infty}\frac{\log N_n(\CC,f)}{n}
=\inf_{n\ge 1}\frac{\log N_n(\CC,f)}{n}.
$$

The next lemma follows straightforwardly from the definitions.

\begin{lem}\label{lem:C-finer-D}
Let $(X,f)$ be a topological dynamical system. If $\CC$ and $\CD$ are two 
finite covers
such that $\CC\prec \CD$, then $h_{top}(\CC,T)\le h_{top}(\CD,T)$.
\end{lem}

The \emph{topological entropy} of a dynamical system $(X,f)$ is defined by
$$
\label{notation:htopf}
\index{htop(f)@$h_{top}(f)$}
h_{top}(f):=\sup\{h_{top}(\CU,f)\mid \CU\text{ finite open cover of } X\}.
$$
The topological entropy is a non negative number (it may be infinite). It
satisfies the following properties:

\begin{prop}\label{prop:htop-Tn}
Let $(X,f)$ be a topological dynamical system.
\begin{itemize}
\item For all integers $n\ge 1$, $h_{top}(f^n)=n h_{top}(f)$.
\item If $Y$ is an invariant subset of $X$,
then $h_{top}(f|_Y)\le h_{top}(f)$.
\item if $(Y,g)$ is a topological dynamical system that is 
conjugate to $(X,f)$, then $h_{top}(f)=h_{top}(g)$.
\end{itemize}
\end{prop}

When dealing with entropy in the sequel, we shall often use 
that $h_{top}(f^n)=n h_{top}(f)$, without referring systematically
to Proposition~\ref{prop:htop-Tn}.

%******************************************************************
\subsection{Definition with Bowen's formula}

The topological entropy can be computed with
\emph{Bowen's formula}. The following notions were introduced in \cite{Bow}.

Let $X$ be a metric space with a distance $d$, and let $f\colon
X\to X$ be a continuous map. Let $\eps>0$ and $n\ge 1$.
The \emph{Bowen ball}\index{Bowen ball}
\label{notation:Bowenball}
\index{Bn(x,e)@$B_n(x,\eps)$}
of center $x$, radius $\eps$ and order $n$ is defined by
\begin{eqnarray*}
B_n(x,\eps)&:=&\{y\in X\mid d(f^k(x),f^k(y))\le \eps,\ k\in\Lbrack 0,n-1\Rbrack\}\\
&=&\bigcap_{i=0}^{n-1}f^{-i}(\overline{B}(f^i(x),\eps)).
\end{eqnarray*}
Let $E\subset X$.
The set $E$ is \emph{$(n,\eps)$-separated}\index{separated set ($(n,\eps)$- )} if for all distinct points $x,y$ in $E$,
there exists $k\in\Lbrack 0,n-1\Rbrack$ such that $d(f^k(x),f^k(y))>\eps$.
The maximal cardinality of an $(n,\eps)$-separated set
is denoted by $s_n(f,\eps)$.
\label{notation:snfeps}
\index{sn(f,e)@$s_n(f,\eps)$}
The set $E$ is an \emph{$(n,\eps)$-spanning set}\index{spanning set ($(n,\eps)$- )}
if $X\subset \bigcup_{x\in E} B_n(x,\eps)$. The 
minimal cardinality of an $(n,\eps)$-spanning set is  denoted by
$r_n(f,\eps)$.
\label{notation:rnfeps}
\index{rn(f,e)@$r_n(f,\eps)$}

\begin{lem}\label{lem:feps-separated}
Let $(X,f)$ be a topological dynamical system, $\eps>0$ and $n\in\IN$.
\begin{enumerate}
\item If $0<\eps'<\eps$, then
$s_n(f,\eps')\ge s_n(f,\eps)$ and $r_n(f,\eps')\ge r_n(f,\eps)$.
\item $r_n(f,\eps)\le s_n(f,\eps)\le r_n(f,\frac{\eps}2)$.
\end{enumerate}
\end{lem}

\begin{proof}
\noindent(i) Obvious.

\noindent(ii) Let $E$ be an $(n,\eps)$-separated set of maximal cardinality
$s_n(f,\eps)$. By maximality, for every $y\in X\setminus E$, $E\cup\{y\}$ is not
$(n,\eps)$-separated, that is, $y\in \bigcup_{x\in E} B_n(x,\eps)$. 
Moreover, $E$ is clearly included in $\bigcup_{x\in E} B_n(x,\eps)$.
This means that
$E$ is an $(n,\eps)$-spanning set, and so $r_n(f,\eps)\le s_n(f,\eps)$.
Let $F$ be an $(n,\frac{\eps}2)$-spanning set 
of cardinality $r_n(f,\frac{\eps}2)$. 
For every $x\in X$, there exists $y(x)\in F$ such that 
$x\in B_n(y(x),\frac{\eps}2)$. If $x_1,x_2$ are two distinct points in $E$, then
$y(x_1)\neq y(x_2)$ (otherwise this would imply that $d(f^i(x_1),f^i(x_2))<\eps$
for all $i\in\Lbrack 0,n-1\Rbrack$). Thus $\# E\le \# F$, that is,
$s_n(f,\eps)\le r_n(f,\frac{\eps}2)$.
\end{proof}

The next result is due to Bowen \cite{Bow2}; see also \cite{Rud}.

\begin{theo}[Bowen's formula]\label{theo:Bowen-formula}\index{Bowen's formula}
Let $(X,f)$ be a topological dynamical system. Then 
$$
h_{top}(f)=\lim_{\eps\to 0}\limsup_{n\to+\infty}\frac{1}{n}\log
s_n(f,\eps)=\lim_{\eps\to 0}\limsup_{n\to+\infty}\frac{1}{n}\log
r_n(f,\eps).
$$
\end{theo}

\begin{proof}
First, the limits
$$
\lim_{\eps\to 0}\limsup_{n\to+\infty}\frac{1}{n}\log
s_n(f,\eps)\quad\text{and}\quad
\lim_{\eps\to 0}\limsup_{n\to+\infty}\frac{1}{n}\log
r_n(f,\eps)
$$
exist by Lemma~\ref{lem:feps-separated}(i), and they are equal by
Lemma~\ref{lem:feps-separated}(ii). Let $h$ denote the value of these
limits. We are going to show that $h_{top}(f)=h$.

Let $\eps>0$ and $n\in\IN$. Let $E$ be an $(n,\eps)$-separated set of
cardinality $s_n(f,\eps)$. Let $\CU$ be a finite 
open cover such that the diameter of all elements of $\CU$ is less than 
$\eps$ (such a cover exists because $X$ is compact). Two 
distinct points in $E$ are in distinct elements of $\CU^n$, so
$s_n(f,\eps)\le N_n(\CU)$. This implies that
$$
\limsup_{n\to+\infty}\frac 1n \log s_n(f,\eps)\le h_{top}(\CU,f),
$$
and so $h\le h_{top}(f)$.

Let $\CV$ be a finite 
open cover and let $\delta>0$ be a Lebesgue number for $\CV$,
that is, for all $x\in X$, there exists $V\in\CV$ such that $B(x,\delta)\subset
V$. Let $\eps\in (0,\delta)$ and 
let $F$ be an $(n,\eps)$-spanning set of cardinality $r_n(f,\eps)$. For 
all $y\in F$ and all $k\in\Lbrack 0,n-1\Rbrack$, there exists
$V_{y,k}\in\CV$ such that $B(f^k(y),\delta)\subset V_{y,k}$.
Let $x\in X$. By definition of $F$, there exists $y\in F$ such that
$x\in B_n(y,\eps)$, hence $f^k(x)\in \overline{B}(f^k(y),\eps)\subset
B(f^k(y),\delta)$ for all $k\in\Lbrack 0,n-1\Rbrack$. Thus
$$
x\in \bigvee_{k=0}^{n-1}f^{-k}(V_{y,k}).
$$
This implies that $\CV':=\{\bigvee_{k=0}^{n-1}f^{-k}(V_{y,k})\mid y\in F\}$
is a subcover of $\CV^n$, and so $N_n(\CV)\le N(\CV')\le\#F=r_n(f,\eps)$.
This implies that 
$$
h_{top}(\CV,f)\le \limsup_{n\to+\infty}\frac 1n \log r_n(f,\eps),
$$ 
and so $h_{top}(f)\le h$. Finally, we get $h_{top}(f)=h$.
\end{proof}

%**************************************************************************
\section{Entropy and horseshoes}\label{sec:entropy-horseshoes}
\subsection{Horseshoes imply positive entropy}

Recall that $(J_1,\ldots, J_p)$ is a $p$-horseshoe for the interval map 
$f$ if $J_1,\ldots, J_p$ are non degenerate closed intervals with pairwise
disjoint interiors such that $J_1\cup \cdots\cup J_p\subset f(J_i)$
for all $i\in\Lbrack 1,p\Rbrack$.

The next proposition appears under this form belatedly in the
literature (e.g., \cite[Proposition VIII.8]{BCop}). 
However it basically follows from the computations of Adler and
McAndrew in \cite{AM2}. 

\begin{prop}\label{prop:horseshoe-htop}
Let $f\colon I\to I$ be an interval map. If $f$ has a $p$-horseshoe, then
$h_{top}(f)\ge\log p$.
\end{prop}

\begin{proof}
We first suppose that $f$ has a strict $p$-horseshoe, say $(J_1,\ldots,J_p)$.
There exist disjoint open sets $U_1,\ldots, U_p$ in $I$ such that
$J_i\subset U_i$ for all $i\in\Lbrack 1,p\Rbrack$.
Let $U_{p+1}:=I\setminus \bigcup_{i=1}^p J_i$. Then
$\CU:=(U_1,\ldots,U_p,U_{p+1})$ is an open cover of $I$ and
$U_{p+1}\cap J_i=\emptyset$ for all $i\in\Lbrack 1,p\Rbrack$. Let $n\ge 1$.
For all $n$-tuples $(i_0,\ldots, i_{n-1})\in \Lbrack 1,p\Rbrack^n$, we set
$$
J_{i_0,\ldots,i_{n-1}}:=\{x\in I\mid \forall k\in\Lbrack 0,n-1\Rbrack,\ f^k(x)
\in J_{i_k}\}.
$$
Since $(J_1,\ldots, J_p)$ is a $p$-horseshoe, the set 
$J_{i_0,\ldots,i_{n-1}}$ is not empty by 
Lemma~\ref{lem:chain-of-intervals}(i).
Moreover, it is contained in a unique element of $\CU^n$, namely
$$
U_{i_0}\cap f^{-1}(U_{i_1})\cap\cdots\cap f^{-(n-1)}(U_{i_{n-1}}).
$$
Thus $N_n(\CU,f)\ge p^n$ for all integers $n\ge 1$, so
$$
h_{top}(f)\ge h_{top}(\CU,f)= \lim_{n\to+\infty}\frac{N_n(\CU,f)}{n}
\ge \log p.
$$

We now turn to the general case, i.e., $f$ has a $p$-horseshoe. 
Let $n\ge 1$. According to
Lemma~\ref{lem:horseshoe-fn}, $f^n$ has a $p^n$-horseshoe. 
We number the $p^n$ intervals
of this horseshoe from left to right in $I$ and we consider only the
intervals whose number is odd. In this way, we obtain a strict $\left\lceil
\frac{p^n}{2}\right\rceil$-horseshoe.
Then, applying the first part of the proof to $f^n$, we get
$$
h_{top}(f^n)\ge \log\left(\frac{p^n}{2}\right). 
$$
By Proposition~\ref{prop:htop-Tn}, we have
$h_{top}(f)=\frac{1}{n}h_{top}(f^n)$, and thus
$h_{top}(f)\ge \log p-\frac{\log 2}{n}$. Finally,
$h_{top}(f)\ge \log p$ by taking the limit when $n$ goes to infinity.
\end{proof}

%***********************************************************************
\subsection{Misiurewicz's Theorem}

Misiurewicz's Theorem states that the existence of horseshoes is necessary
to have positive entropy. This theorem was first
proved for piecewise monotone maps by Misiurewicz and Szlenk 
\cite{MS,MS2}, then Misiurewicz generalized the result 
for all continuous interval maps
\cite{Mis,Mis2}. There is no significant difference between the piecewise
monotone case and the general case.

\begin{theo}[Misiurewicz]\label{theo:Misiurewicz}\index{Misiurewicz's Theorem}
Let $f\colon I\to I$ be an interval map of positive topological entropy. 
For every $\lambda< h_{top}(f)$ and every $N$, there exist intervals 
$J_1,\ldots, J_p$ and a positive integer $n\ge N$ such that
$(J_1,\ldots, J_p)$ is a strict $p$-horseshoe for $f^n$ and $\frac{\log p}{n}
\ge \lambda$. 
\end{theo}

We are first going to state three technical lemmas about limits of
sequences,  then we shall prove Theorem~\ref{theo:Misiurewicz}.

\begin{lem}\label{lem:sequences-an-bn1}
Let $(a_n)_{n\ge 1}$ and $(b_n)_{n\ge 1}$ be 
two sequences of positive numbers. Then
$$
\limsup_{n\to+\infty}\frac{1}{n}\log(a_n+b_n)=
\max\left\{\limsup_{n\to+\infty}\frac{1}{n}\log a_n,
\limsup_{n\to+\infty}\frac{1}{n}\log b_n\right\}.
$$
The same result holds for finitely many sequences of positive
numbers:
$$
\limsup_{n\to+\infty}\frac{1}{n}\log(a_n^1+\cdots+a_n^k)=
\max\left\{\limsup_{n\to+\infty}\frac{1}{n}\log a_n^i\mid
i\in\Lbrack 1,k\Rbrack\right\}.
$$
\end{lem}

\begin{proof}
We show the lemma for two sequences, the general case 
follows by a straightforward induction. We set 
$$
L:=\max\left\{\limsup_{n\to+\infty}\frac{1}{n}\log a_n,\ 
\limsup_{n\to+\infty}\frac{1}{n}\log b_n\right\}.
$$
Since $a_n+b_n\ge a_n$ and $a_n+b_n\ge b_n$, it is obvious that
$$
\limsup_{n\to+\infty}\frac{1}{n}\log(a_n+b_n)\ge L.
$$
Conversely, for every $\eps>0$, there exists an integer $n_0$ such that,
for all $n\ge n_0$,
$a_n\le e^{(L+\eps)n}$ and $b_n\le e^{(L+\eps)n}$.
This implies that
$$
\forall n\ge n_0,\ \frac{1}{n}\log(a_n+b_n)\le L+\eps+\frac{\log 2}{n}.$$
To conclude, we first take the limsup when $n\to +\infty$, 
then  we let $\eps$ tend to zero.
\end{proof}

\begin{lem}\label{lem:sequences-an-bn2}
Let $(a_n)_{n\ge 1}$ and $(b_n)_{n\ge 0}$ be two sequences of real numbers.
Then
$$
\limsup_{n\to+\infty}\frac{1}{n}\log\sum_{k=1}^n\exp(a_k+b_{n-k})\le
\max\left\{\limsup_{n\to+\infty}\frac{a_n}{n},\ 
\limsup_{n\to+\infty}\frac{b_n}{n}\right\}.
$$
\end{lem}

\begin{proof}
We set 
$$
L:=\max\left\{\limsup_{n\to+\infty}\frac{a_n}{n},\ 
\limsup_{n\to+\infty}\frac{b_n}{n}\right\}.
$$
We assume that $L<+\infty$, otherwise there is nothing to prove. 
For every $\eps>0$, there exists an integer $n_0$ such 
that
\begin{equation}\label{eq:an-bn}
\forall n\ge n_0,\quad \frac{a_n}{n}\le L+\eps
\quad\text{and}\quad\frac{b_n}{n}\le L+\eps.
\end{equation}
We set
$M:=\max\left\{0,\frac{a_n}{n},\frac{b_n}{n}\mid n\in\Lbrack 1,n_0-1\Rbrack
\right\}$. Let $n,k$ be two integers such that $n\ge 2n_0$ and 
$k\in\Lbrack 1,n\Rbrack$. Necessarily, they satisfy either
$k\ge n_0$ or $n-k\ge n_0$. We split into three cases.
\begin{itemize}
\item If $k\ge n_0$ and $n-k\ge n_0$, then by \eqref{eq:an-bn}:
$$
a_k+b_{n-k}\le k(L+\eps)+(n-k)(L+\eps)=n(L+\eps).
$$
\item If $k\ge n_0$ and $n-k<n_0$, then by \eqref{eq:an-bn} and the definition
of $M$:
$$
a_k+b_{n-k}\le k(L+\eps)+(n-k)M\le n(L+\eps)+n_0M.
$$
\item If $k<n_0$ and $n-k\ge n_0$, then by \eqref{eq:an-bn} and the definition
of $M$:
$$
a_k+b_{n-k}\le kM+(n-k)(L+\eps)\le n_0M+n(L+\eps).
$$
\end{itemize}
In the three cases, we have $a_k+b_{n-k}\le n(L+\eps)+n_0M$,
and thus
$$
\forall n\ge 2n_0,\ \frac{1}{n}\log\sum_{k=1}^n\exp(a_k+b_{n-k})\le
\frac{1}{n}\log n+L+\eps+\frac{n_0M}{n}.
$$
We first take the limsup  when $n\to+\infty$; then we let $\eps$ tend to zero, 
and we get
$$
\limsup_{n\to+\infty}\frac{1}{n}\log\sum_{k=1}^n\exp(a_k+b_{n-k})\le L,
$$
which proves the lemma.
\end{proof}

\begin{lem}\label{lem:alpha-beta-an}
Let $(a_n)_{n\ge 1}$ be a sequence of real numbers and $\alpha,\beta\in\IR$.
Suppose that there exists $C>0$ such that $a_{n+1}\le a_n+C$
for all $n\ge 1$, and that
$$
0<\alpha<\beta<\limsup_{n\to+\infty}\frac{a_n}{n}.
$$
Then, for all integers $N$, there exists $n\ge N$ such that
$a_n\ge \beta n$ and $a_{n+1}\ge a_n+\alpha$.
\end{lem}

\begin{proof}
We suppose that the lemma is false, that is, there exists $N$
such that
\begin{equation}\label{eq:alphabeta1}
\forall n\ge N,\ a_n\ge \beta n\Longrightarrow
a_{n+1}<a_n+\alpha.
\end{equation}
If $a_n\ge \beta n$ for all $n\ge N$,
then $a_{n+N}< a_N+\alpha n$ for all $n\ge 1$ by \eqref{eq:alphabeta1}, 
which implies that
$$
\limsup_{n\to+\infty}\frac{a_n}{n}\le \alpha.
$$
But this contradicts the assumption on $\alpha$.
Thus there exists an integer $n\ge N$ such that $a_{n}<\beta n$.
We set $N_0:=\min\{n\ge N\mid a_n<\beta n\}$.
Suppose that $a_p<\beta p$ for some integer $p\ge N$ and that
$r$ is a positive integer such that $a_n\ge \beta n$
for all $n\in\Lbrack p+1,p+r\Rbrack$. We are going to show that $r$ is bounded
by a constant independent from $p$. By 
\eqref{eq:alphabeta1}, we have 
$a_{p+r}\le a_{p+1}+\alpha(r-1)$. Since $a_{p+1}\le a_p+C$ and 
$a_{p}<\beta p$, we have
\begin{equation}\label{eq:ap+r-beta}
\beta(p+r)\le a_{p+r} \le \beta p +C+\alpha(r-1).
\end{equation}
This implies that $\beta r\le C+\alpha(r-1)$, and thus $r\le M$
if we set 
$$
M:=\frac{C-\alpha}{\beta-\alpha}.
$$
Notice that $M>0$ because $a_n\le a_1+(n-1)C$ for all $n\ge 1$, and hence
$\limsup_{n\to +\infty} \frac{a_n}n\le C$, which implies that
$C-\alpha>0$; moreover, $\beta-\alpha>0$ by assumption.
Let $n$ be an integer greater than $N_0+M$. We consider two cases.\\
$\bullet$ If $a_n\ge \beta n$,
what precedes implies that there exists an integer $p\in [n-M,n)$ such that 
$a_p<\beta p$
and $a_i\ge \beta i$ for all $i\in\Lbrack p+1,n\Rbrack$, 
Then, by \eqref{eq:ap+r-beta}, we have
$$
a_n\le \beta p+C+\alpha(M-1)\le \beta n+C+\alpha M.
$$
$\bullet$ If  $a_n< \beta n$,
the inequality $a_n\le \beta n+C+\alpha M$ trivially holds.

Since the inequality $a_n\le \beta n+C+\alpha M$ holds in the two cases,
we have
$$
\limsup_{n\to+\infty}\frac{a_n}{n}\le \beta.
$$
But this contradicts the assumption on $\beta$.
We conclude that the lemma is true.
\end{proof}

\begin{proof}[Proof of Theorem~\ref{theo:Misiurewicz}]
First we are going to prove the theorem under the extra assumptions that
$h_{top}(f)>\log 3$ and $\log 3<\lambda<h_{top}(f)$.  
We choose $\lambda'$ such that $\lambda<\lambda'<h_{top}(f)$.
According to the definition of topological entropy, there exists a finite open
cover $\CU$ such that $h_{top}(\CU,f)>\lambda$'. We choose a
partition $\CP$ consisting of finitely many disjoint non degenerate
intervals such that $\CP$ is finer than $\CU$. 
Then $h_{top}(\CP,f)\ge h_{top}(\CU,f)$ 
by Lemma~\ref{lem:C-finer-D}, so $h_{top}(\CP,f)>\lambda'$.
We have $N(\CP^n)=\#(\CP^n)$ because $\CP^n$ is a partition, and thus
$$
h_{top}(\CP,f)=\lim_{n\to+\infty}\frac{1}{n}\log\#\left(\CP^n\right).
$$
If $\CQ$ is a family of subsets of $I$, we define, for all $n\ge 1$ and all 
$A\in\CQ$,
\begin{gather*}
\CQ^n:=\left\{(A_0,\ldots, A_{n-1})\mid \forall i\in\Lbrack 0,n-1\Rbrack, A_i\in\CQ
\text{ and }\bigcap_{i=0}^{n-1}f^{-i}(A_i)\neq \emptyset\right\}\\
\text{and}\quad\CQ^n|A:=\{(A_0,\ldots, A_{n-1})\in\CQ^n\mid A_0=A\}.
\end{gather*}
We have $\#(\CP^n)=\sum_{A\in \CP}\#(\CP^n|A)$. Thus, by
Lemma~\ref{lem:sequences-an-bn1}, there exists $A\in\CP$
such that
\begin{equation}\label{eq:Pn|A-htop}
h_{top}(\CP,f)=\limsup_{n\to+\infty}\frac{1}{n}\log\#\left(\CP^n|A\right).
\end{equation}
Let $\CF$ be the family of $A\in\CP$ satisfying \eqref{eq:Pn|A-htop}. 
We claim that:
\begin{equation}\label{fact:Fn|A-htop}
\forall A\in\CF,\ h_{top}(\CP,f)=\limsup_{n\to+\infty}\frac{1}{n}
\log\#\left(\CF^n|A\right).
\end{equation}

\begin{proof}[Proof of \eqref{fact:Fn|A-htop}]
The inequality $\ge$ is straightforward. We are going to
prove the reverse inequality.
We fix $A\in\CF$. Let $(A_0,\ldots,A_{n-1})\in\CP^n|A$ and let $k$
be the greatest integer in $\Lbrack 1,n\Rbrack$ 
such that $A_i\in \CF$ for all $i\in\Lbrack 0,k-1\Rbrack$.
Then $(A_0,\ldots, A_{k-1})\in \CF^k|A$ and, if $k<n$,
$(A_k,\ldots, A_{n-1})\in\CP^{n-k}|B$ for some $B\in\CP\setminus\CF$.
Thus
\begin{equation}\label{eq:card-Pn|A}
\#(\CP^n|A)\le\sum_{k=1}^{n-1}\left(\#(\CF^k|A)\sum_{B\in\CP\setminus\CF}
\#(\CP^{n-k}|B)\right)+\#(\CF^n|A).
\end{equation}
We set $b_0:=0$ and
$$
\forall n\ge 1,\ a_n:=\log\#(\CF^n|A)\text{ and }
b_n:=\log \sum_{B\in\CP\setminus\CF}\#(\CP^{n}|B).
$$
Then \eqref{eq:card-Pn|A} can be rewritten as
$$
\#(\CP^n|A)\le \sum_{k=1}^n\exp(a_k+b_{n-k}).
$$
Inserting this inequality in \eqref{eq:Pn|A-htop}, we get
$$
h_{top}(\CP,f)\le \limsup_{n\to+\infty}\frac{1}{n}\log\left(
\sum_{k=1}^n\exp(a_k+b_{n-k})\right).
$$
Thus, by Lemma~\ref{lem:sequences-an-bn2},
\begin{equation}\label{eq:htop(P,f)}
h_{top}(\CP,f)\le \max\left\{\limsup_{n\to+\infty}\frac{a_n}{n},
\limsup_{n\to+\infty}\frac{b_n}{n}\right\}.
\end{equation}
According to the definition of $\CF$, we have
$$
\forall B\in\CP\setminus \CF,\quad \limsup_{n\to+\infty}\frac{1}{n}\log\#(\CP^n|B)<h_{top}(\CP,f),
$$
and thus $\limsup_{n\to+\infty}\frac{b_n}{n}<h_ {top}(\CP,f)$ by
Lemma~\ref{lem:sequences-an-bn1}. Finally, in view of \eqref{eq:htop(P,f)},
we have $h_{top}(\CP,f)\le \limsup_{n\to+\infty}\frac{a_n}{n}$. This
concludes the proof of \eqref{fact:Fn|A-htop}.
\end{proof}

Let $A_0,\ldots, A_{n-1}\in\CF$.
We set $A_0':=A_0$  and $A_i':=A_i\cap f(A'_{i-1})$
for all $i\in\Lbrack 1,n-1\Rbrack$. We claim that:
\begin{eqnarray}\label{eq:tildeA}
A'_{n-1}&=&f^{n-1}\left(\left\{x_0\mid \forall i\in\Lbrack 0,n-1\Rbrack,
f^i(x_0)\in A_i\right\}\right)\\&=&
f^{n-1}\left(\bigcap_{i=0}^{n-1}f^{-i}(A_i)\right).\nonumber
\end{eqnarray}
Indeed, 
\begin{eqnarray*}
\lefteqn{x_{n-1}\in A'_{n-1}}\\
&\Leftrightarrow& \exists x_{n-2}\in A'_{n-2},\ f(x_{n-2})=x_{n-1}\in A_{n-1},\\
&\Leftrightarrow& \exists x_{n-3}\in A'_{n-3},\ f(x_{n-3})=x_{n-2}\in A_{n-2} 
\text{ and }
f^2(x_{n-3})=f(x_{n-2})=x_{n-1},\\
&\vdots&\\
&\Leftrightarrow& \exists x_0\in A_0',\ f(x_0)=x_1\in A_1,
f^2(x_0)=x_2\in A_2,
\ldots f^{n-1}(x_0)=x_{n-1}\in A_{n-1}.
\end{eqnarray*}
Therefore, \eqref{eq:tildeA} holds, which implies, according to the definition
of $\CF^n$:
\begin{equation}\label{eq:CFn}
(A_0,\ldots, A_{n-1})\in\CF^n\Longleftrightarrow  
A'_{n-1}\neq \emptyset.
\end{equation}
If $A'_{n-1}\neq \emptyset$, then $A'_i$ is
nonempty and $A'_i\subset f(A'_{i-1})$ for all $i\in\Lbrack i,n-1\Rbrack$.
Thus, by Lemma~\ref{lem:chain-of-intervals}, for every 
$(A_0,\ldots, A_{n-1})\in\CF^n$, there exists a nonempty interval 
$J_{A_0\ldots A_{n-1}}$ such that
\begin{gather}
f^{n-1}(J_{A_0\ldots A_{n-1}})=A'_{n-1}\quad\text{and}\label{eq:J1}\\
\forall i\in\Lbrack 0,n-1\Rbrack,\  f^i(J_{A_0\ldots A_{n-1}})\subset 
A'_i\subset A_i.\label{eq:J2}
\end{gather}
Moreover, $(J_{A_0\ldots A_{n-1}})_{(A_0,\ldots, A_{n-1})
\in\CF^n}$ is a family of pairwise disjoint intervals.
If $(A_0,\ldots,A_{n-1})\in\CF^n$ and $A_n\in\CF$, then 
$f^n(J_{A_0\ldots A_{n-1}})\cap A_n =f\left(A'_{n-1}\right)
\cap A_n=A'_n$, and thus, according to \eqref{eq:CFn},
$$
f^n(J_{A_0\ldots A_{n-1}})\cap A_n\neq\emptyset\Longleftrightarrow 
(A_0,\ldots, A_n)\in\mathcal{F}^{n+1}.
$$
Therefore
\begin{equation}\label{eq:sum-c(A,B)-2}
\#(\CF^{n+1}|A)=\sum_{\doubleindice{(A_0,\ldots, A_{n-1})\in\CF^n}{A_0=A}}
\#\{B\in\CF\mid f^n(J_{A_0\ldots A_{n-1}})\cap B\neq \emptyset\}.
\end{equation}

For all $A,B\in\CF$, we set
$$
c(A,B,n):=\#\{(A_0,\ldots, A_{n-1})\in\CF^n\mid A_0=A,\ f^n(J_{A_0\ldots
A_{n-1}})\supset B\}.
$$
We shall need the following result:
\begin{equation}\label{fact:c-ABC}
\forall A,B,C\in\CF,\ \forall n,m\ge 1,\ c(A,B,n) c(B,C,m)\le c(A,C,m+n).
\end{equation}

\begin{proof}[Proof of \eqref{fact:c-ABC}]
For all $A,B\in\CF$ and all $n\ge 1$, we set 
$$
\CC(A,B,n):=\{(A_0,\ldots,A_{n-1})\in\CF^n\mid A_0=A,\ 
f^n(J_{A_0\ldots A_{n-1}})\supset B\}.
$$
Let $(A_0,\ldots,A_{n-1})\in \CC(A,B,n)$ and
$(B_0,\ldots,B_{m-1})\in \CC(B,C,m)$. We are going to show that
$(A_0,\ldots, A_{n-1}, B_0,\ldots, B_{m-1})\in \CC(A,C,n+m)$.
The set $f^n(J_{A_0\ldots A_{n-1}})$ contains $B$ by definition, and 
$J_{B_0\ldots B_{m-1}}\subset B_0=B$ by \eqref{eq:J2}.
This implies (by Lemma~\ref{lem:chain-of-intervals}(i)) that there exists a 
nonempty interval $K\subset J_{A_0\ldots A_{n-1}}$
such that $f^n(K)=J_{B_0\ldots B_{m-1}}$. Moreover, by \eqref{eq:J2},
this interval satisfies: 
$$
\forall i\in\Lbrack 0,n-1\Rbrack,\ f^i(K)\subset A_i\text{ and }\forall 
j\in\Lbrack 0,m-1\Rbrack,\ f^{n-1+j}(K)=f^j(J_{B_0\ldots B_{m-1}})
\subset B_j.
$$
Consequently,
$$
K\subset\bigcap_{i=0}^{n-1}f^{-i}(A_i)\cap\bigcap_{i=n}^{n+m-1}f^{-i}(B_{i-n}).
$$
This implies the following facts. First, 
$(A_0,\ldots, A_{n-1}, B_0,\ldots, B_{m-1})\in \CF^{n+m}$ by
\eqref{eq:CFn}+\eqref{eq:tildeA}, using the fact that $K\neq \emptyset$. 
Second, the set $f^{n+m-1}(K)$ is included in 
$f^{m+n-1}(J_{A_0\ldots A_{n-1}B_0\ldots B_{m-1}})$
by combining \eqref{eq:J1} and \eqref{eq:tildeA}.
Then, since 
$f^{n+m}(K)=f^{m}(J_{B_0\ldots B_{m-1}})$, we have $f^{n+m}(K)\supset C$
by definition of $\CC(B,C,m)$, so
$f^{n+m}(J_{A_0\ldots A_{n-1}B_0\ldots B_{m-1}})\supset C$.
We conclude that 
$(A_0,\ldots, A_{n-1}, B_0,\ldots, B_{m-1})\in \CC(A,C,n+m)$.
This clearly implies \eqref{fact:c-ABC}.
\end{proof}

We fix $A\in\CF$. We have
\begin{eqnarray*}
\lefteqn{\sum_{B\in\CF}c(A,B,n)}\\
&=&
\#\left\{((A_0,\ldots, A_{n-1}),B)\in \CF^n\times \CF \mid A_0=A, 
f^n(J_{A_0\ldots A_{n-1}})\supset B\right\}\\
&=& \sum_{\doubleindice{(A_0,\ldots, A_{n-1})\in\CF^n}{A_0= A}}
\#\{B\in\CF\mid f^n(J_{A_0\ldots A_{n-1}})\supset B\}.
\end{eqnarray*}
Consider $(A_0,\ldots, A_{n-1})\in\CF^n$.
If $f^n(J_{A_0\ldots A_{n-1}})$ meets $k$ intervals of $\CF$, then
$f^n(J_{A_0\ldots A_{n-1}})$ contains at least $k-2$ of them because $f^n(J_{A_0\ldots A_{n-1}})$ is an
interval. Therefore
$$
\sum_{B\in\CF}c(A,B,n)\ge \sum_{\doubleindice{(A_0,\ldots, A_{n-1})\in\CF^n}{A_0= A}}
\left(\#\{B\in\CF\mid f^n(J_{A_0\ldots A_{n-1}})\cap B\neq \emptyset\}-2\right).
$$
Combining this inequality with \eqref{eq:sum-c(A,B)-2},
we get:
\begin{equation}\label{eq:sum-c(A,B)-3}
\sum_{B\in\CF}c(A,B,n)\ge \#(\CF^{n+1}|A)-2\#(\CF^n|A).
\end{equation}
We set $a'_n:=\log \#(\CF^n|A)$ for all $n\ge 1$.
According to \eqref{fact:Fn|A-htop}, we have
$$
h_{top}(\CP,f)= \limsup_{n\to+\infty}\frac{a'_n}{n}.
$$
Moreover, $a'_{n+1}\le a'_n+\log \#\CF$ for all $n\ge 1$.
Therefore, we can apply Lemma~\ref{lem:alpha-beta-an} with 
$\alpha=\log 3$, $\beta=\lambda'$ and $C=\log\#\CF$, and we see that, for all integers $N$,
\begin{equation}\label{eq:Fnn+1}
\exists n\ge N,\ \#(\CF^n|A)>e^{\lambda' n}\text{ and }
\#(\CF^{n+1}|A)\ge 3\#(\CF^n|A).
\end{equation}
For an integer $n$ satisfying \eqref{eq:Fnn+1}, we inject these inequalities
in \eqref{eq:sum-c(A,B)-3} and we get
$$
\sum_{B\in\CF}c(A,B,n)\ge 3\#(\CF^n|A)-2\#(\CF^n|A)=\#(\CF^n|A)\ge
e^{\lambda' n}.
$$
Therefore
$$
\limsup_{n\to+\infty}\frac{1}{n}\log \sum_{B\in\CF}c(A,B,n)\ge \lambda'.
$$
According to Lemma~\ref{lem:sequences-an-bn1}, 
for all $A\in \CF$ there exists $B=\vfi(A)\in\CF$ such that
$$
\limsup_{n\to+\infty}\frac{1}{n}\log c(A,\vfi(A),n)\ge \lambda'>\lambda.
$$
Since $\CF$ is finite, the map $\vfi\colon\CF\to\CF$ has a periodic
point, that is, there exist $A_0\in\CF$ and $p\in\IN$ such that 
$\vfi^p(A_0)=A_0$.
Using the preceding inequality with $A=\vfi^i(A_0)$, $i=0,\ldots,
p-1$, we have
\begin{equation}\label{eq:phiA}
\forall i\in\Lbrack 0,p-1\Rbrack,\ \forall N_i\ge 1,\ \exists n_i\ge N_i,
\ c(\vfi^i(A_0),\vfi^{i+1}(A_0),n_i)\ge e^{n_i\lambda}.
\end{equation}
For every $i\in\Lbrack 0,p-1\Rbrack$, let $N_i$ be a positive integer and let $n_i\ge N_i$
be given by \eqref{eq:phiA}. We set $n:=\sum_{i=0}^{p-1} n_i$ and
$k:=c(A_0,A_0,n)$. We have
\begin{eqnarray*}
k=c(A_0,A_0,n)&\ge&
\prod_{i=0}^{p-1}c(\vfi^i(A_0),\vfi^{i+1}(A_0),n_i)
\quad\text{by }\eqref{fact:c-ABC} \\
&\ge& \prod_{i=0}^{p-1}e^{n_i\lambda}=e^{n\lambda}\quad\text{by }\eqref{eq:phiA}.
\end{eqnarray*}
According to the definition of $c(A_0,A_0,n)$, this means that there exist $k$ 
disjoint intervals $J_1,\ldots,J_k\subset A_0$ such that $f^n(J_i)\supset A_0$
for all $i\in\Lbrack 1,k\Rbrack$.
Thus $(\overline{J_1},\ldots,\overline{J_k})$ is a 
$k$-horseshoe for $f^n$, and 
$\frac{1}{n}\log k\ge \lambda$. 

%\medskip
This result implies the theorem in the general case in the following way.
Suppose that $h_{top}(f)>0$ and $0<\lambda<h_{top}(f)$. 
We choose $\lambda''$ such that $\lambda<\lambda''<h_{top}(f)$;
then we choose an integer $q$
such that $q\lambda''>\log 3$ and $q(\lambda''-\lambda)\ge \log 2$. 
By Proposition~\ref{prop:htop-Tn}, $h_{top}(f^q)=qh_{top}(f)>q\lambda''$.
Therefore, applying what precedes to $f^q$, we obtain that, 
for every integer $N$, there exist positive integers $n, k$ with $n\ge N$
and a $k$-horseshoe $(J_1,\ldots, J_k)$ for $f^{nq}$ such that
$\frac{1}{n}\log k\ge q\lambda''$. 
We order the intervals of this horseshoe such that $J_1\le J_2\cdots\le J_k$
and we set $k':=\lceil \frac k2 \rceil$.
Then the intervals $J_i$ with odd indices are pairwise disjoint
and form a $k'$-horseshoe for $f^{nq}$, and we have
$$
\frac{\log k'}{qn}\ge \frac{\log k-\log 2}{qn}
\ge \lambda''- \frac{\log 2}{qn}\ge \lambda.
$$
This ends the proof of the theorem.
\end{proof}

%****************
\subsection*{Remarks on graph maps}

The notion of a horseshoe can be extended to graph maps in the following way.

\begin{defi}\index{horseshoe for graph maps}
Let $f\colon G\to G$ be a graph map. Let $I$ be a closed interval of $G$
containing no branching point except maybe its endpoints,
and let $J_1,\ldots, J_n$ be
non degenerate closed subintervals of $I$ with pairwise disjoint interiors
such that $f(J_i)=I$ for all $i\in\Lbrack 1,n\Rbrack$. Then
$(J_1,\ldots,J_n)$ is called an \emph{$n$-horseshoe} for $f$. 
If in addition the intervals are disjoint, 
$(J_1,\ldots,J_n)$ is called a \emph{strict} $n$-horseshoe.
\end{defi}

Llibre and Misiurewicz proved that, with this definition, 
Proposition~\ref{prop:horseshoe-htop}
and Theorem~\ref{theo:Misiurewicz} remain valid for graph maps \cite{LM3}.

\begin{theo}\label{theo:htop-horseshoeG}
Let $f$ be a graph map.
If $f$ has an $n$-horseshoe, then $h_{top}(f)\ge \log n$. Conversely,
if $h_{top}(f)>0$, then for all $\lambda < h_{top}(f)$ and all $N\ge 1$, there exist
integers $n\ge N$, $p\ge 1$ and a strict $p$-horseshoe for $f^n$ such that
$\frac{\log p}{n}\ge \lambda$.
\end{theo}

%*************************************************************************
%*************************************************************************
\section{Homoclinic points}

The notion of a homoclinic point for a diffeomorphism of a smooth manifold was
introduced by Poincaré; this is a point belonging to both the stable and
the unstable manifolds of a hyperbolic point (see, e.g., \cite{Sma}).
In \cite{Bloc3}, Block defined unstable manifolds and homoclinic points
for an interval map.

\begin{defi}
Let $f$ be an interval map. Let $z$ be a periodic point
and let $p$ denote its period. 
Let $\CV(z)$\index{V(z)@$\CV(z)$} 
\label{notation:Vz}
denote the family of neighborhoods of $z$.
The \emph{unstable manifold}\index{unstable manifold of a periodic point}
\label{notation:Wuz}
\index{Wu(z,fp)@$W^u(z,f^p)$} of $z$ is the set
$$
W^u(z,f^p):=\bigcap_{V\in\CV(z)}\bigcup_{n\ge 1}f^{np}(V).
$$
Equivalently, $x\in W^u(z,f^p)$ if and only if there exist sequences of
points $(x_k)_{k\ge 0}$ and of positive integers $(n_k)_{k\ge 0}$ such
that $\displaystyle \lim_{k\to+\infty} x_k=z$ and $\forall k\ge 0$, $f^{pn_k}(x_k)=x$.
\end{defi}

\begin{defi}
Let $f$ be an interval map. A point $x$ is
a \emph{homoclinic point}\index{homoclinic point} if there exists
a periodic point $z$ such that:
\begin{enumerate}
\item $x\neq z$,
\item $x\in W^u(z,f^p)$, where $p$ is the period of $z$,
\item $z\in \omega(x,f^p)$.
\end{enumerate}
The point $x$ is an \emph{eventually periodic homoclinic point}\index{eventually periodic homoclinic point} if it satisfies the above conditions (i) and (ii)
and if there exists a positive integer $k$ such that $f^{kp}(x)=z$ (this 
condition is trivially stronger than (iii)).
\end{defi}

\begin{rem}
The notion of a homoclinic point introduced by Block in \cite{Bloc3}
corresponds to what is called an eventually periodic homoclinic point
in the above definition. In \cite{BCop}, homoclinic points are called
\emph{homoclinic in the sense of Poincaré} to make a distinction
with the previous kind of homoclinic point, which is more restrictive.
We rather follow the terminology of \cite{KKM, Mak}.
\end{rem}

\subsection{Preliminary results about unstable manifolds}

We shall need a few results about unstable manifolds. We state them for
fixed points. In view of the definition, there is no loss of generality
since a periodic point of period $p$ for $f$ is a fixed point for $f^p$.

The next easy lemma states that an unstable manifold is
connected and invariant. The other three lemmas of the section
are more technical.

\begin{lem}\label{lem:homoclinic-basicprop}
Let $f\colon I\to I$ be an interval map and let $z$ be a fixed point.
Then
\begin{enumerate}
\item $W^u(z,f)$ is an interval containing $z$ (it may be reduced to $\{z\}$),
\item $f(W^u(z,f))\subset W^u(z,f)$.
\end{enumerate}
\end{lem}

\begin{proof}
For every $\eps>0$, the set $\bigcup_{n\ge 1} f^n((z-\eps,z+\eps)\cap I)$ is an 
interval containing $z$ because it is a union of intervals containing the
fixed point $z$. It follows straightforwardly from the definition that
$$
W^u(z,f)=\bigcap_{\eps>0}\bigcup_{n\ge 1}f^n((z-\eps,z+\eps)\cap I),
$$
which is an intersection of intervals containing $z$. Thus
$W^u(z,f)$ is also an interval containing $z$, which gives (i).

Let $x\in W^u(z,f)$. This means that for all $V\in\CV(z)$, there
exists $n\ge 1$ such that $x\in f^n(V)$. Thus $f(x)\in f^{n+1}(V)$.
This implies that $f(x)\in W^u(z,f)$, which proves (ii).
\end{proof}

\begin{lem}\label{lem:2fixedpoints-unstableM}
Let $f$ be an interval map. Let $z_1,z_2$ be two fixed points with $z_1<z_2$ 
and such that there is no fixed point in $(z_1,z_2)$. Then
there exists $i\in\{1,2\}$ such that $(z_1,z_2)\subset W^u(z_i,f)$.
\end{lem}

\begin{proof}
The assumption that $(z_1,z_2)$ contains no fixed point implies that
\begin{eqnarray}
\text{either}&&\forall x\in (z_1,z_2),\ f(x)>x\label{eq:homocl-fx>x},\\
\text{or}&&\forall x\in (z_1,z_2),\ f(x)<x\label{eq:homocl-fx<x}.
\end{eqnarray}
We assume that \eqref{eq:homocl-fx>x} holds and we are going to show that
$(z_1,z_2)\subset W^u(z_1,f)$. If \eqref{eq:homocl-fx<x} holds, then by
symmetry we have $(z_1,z_2)\subset W^u(z_2,f)$.

Let $y\in (z_1,z_2)$. Let $V$ be a neighborhood of $z_1$ and
$x\in (z_1,y)\cap V$. We set
$\delta:=\min\{f(t)-t\mid t\in [x,y]\}$.
By compactness, \eqref{eq:homocl-fx>x} implies $\delta>0$. If $t\in [x,y]$, then
\begin{equation}\label{eq:zt+m}
f([z_1,t])\supset [z_1,t+\delta].
\end{equation}
We define a sequence $(b_n)_{n\ge 0}$ by
\begin{itemize}
\item $b_0:=x$,
\item if $b_n\le z_2$, $b_{n+1}:= \max f([z_1,b_n])$;
if $b_n>z_2$, the sequence is not defined for greater indices.
\end{itemize}
By \eqref{eq:homocl-fx>x}, the sequence $(b_n)_{n\ge 0}$ is increasing and
$b_n=\max f^n([z_1,x])$. According to \eqref{eq:zt+m}, if $b_n\in [x,y]$,
then $b_{n+1}\ge b_n+\delta$, so $b_{n+1}\ge b_0+(n+1)\delta$ by induction. 
Since $[x,y]$ is bounded, this implies that there
exists $n_0$ such that $b_{n_0}\ge y$. Thus 
$y\in f^{n_0}([z_1,x])\subset f^{n_0}(V)$.
This implies that $y\in W^u(z_1,f)$, and we conclude that
$(z_1,z_2)\subset W^u(z_1,f)$.
\end{proof}

\begin{lem}\label{lem:limhomoclinic=z}
Let $f\colon I\to I$ be an interval map. Let $z$ be a fixed point and
let $y$ be a point such that
$y\neq z$ and $y\in W^u(z,f)$. Then for every neighborhood $V$ of
$z$, there exist $y'\in V\cap W^u(z,f)$ and an integer $n\ge 1$ such that
$f^n(y')=y$.
\end{lem}

\begin{proof}
Suppose that the result is false, that is, there exists a neighborhood $V$
of $z$ such that
\begin{equation}\label{eq:ynotinVW}
\forall n\ge 1,\ y\notin f^n(V\cap W^u(z,f)).
\end{equation}
We can assume that $V$ is an interval. We also assume that $y>z$, the case
$y<z$ being symmetric.

Since $y\in W^u(z,f)$ and according to 
the definition of an unstable manifold, \eqref{eq:ynotinVW} implies:
\begin{equation}\label{eq:VcapW}
V\cap W^u(z,f)\text{ is not a neighborhood of }z.
\end{equation}
By Lemma~\ref{lem:homoclinic-basicprop}(i), $W^u(f,z)$ is an interval
containing $[z,y]$. Thus $V\cap W^u(z,f)$ is also an interval containing
$z$, and \eqref{eq:VcapW} implies:
\begin{gather}
z=\min \left(V\cap W^u(z,f)\right)=\min W^u(z,f)\label{eq:zminW}\\
\text{and}\quad z\neq \min I.\nonumber
\end{gather}
Let $b\in V\cap (z,y)$. Since $f(z)=z$ and $z\neq \min I$, there
exists a point $c$ such that
\begin{equation}\label{eq:fx<b}
c\in V,\ c<z\quad\text{and}\quad\forall x\in[c,z],\ f(x)<b.
\end{equation}
By \eqref{eq:zminW}, $c\notin W^u(z,f)$. Thus, by definition of $W^u(z,f)$,
there exists $d\in (c,z)$ such that
\begin{equation}\label{eq:cnotinW}
\forall n\ge 1,\ c\notin f^n([d,z]).
\end{equation}
We set $d_n:=\min f^n([d,z])$ for all $n\ge 0$. Then 
$d_n\le z$ and \eqref{eq:cnotinW}
implies that $d_n>c$ for all $n\ge 0$. We show by induction on $n$ that
\begin{equation}\label{eq:fndz}
\forall n\ge 0,\ f^n([d,z])\subset [d_n,z]\cup\bigcup_{i=0}^{n-1}f^i([z,b]).
\end{equation}

\noindent $\bullet$ \eqref{eq:fndz} is satisfied for $n=0$.\\
$\bullet$ Suppose that \eqref{eq:fndz} holds for $n$. Then
$f([d_n,z])\subset [d_{n+1},b)=[d_{n+1},z]\cup [z,b)$ using \eqref{eq:fx<b}
and the fact that $d_n\in (c,z]$. Then
\begin{eqnarray*}
f^{n+1}([d,z])&\!\!\subset\!\!& f([d_n,z])\cup\bigcup_{i=1}^n f^i([z,b])
\quad\text{by the induction hypothesis},\\
&\!\!\subset\!\!& [d_{n+1},z]\cup\bigcup_{i=0}^n f^i([z,b])\quad\text{by what precedes}.
\end{eqnarray*}
This is \eqref{eq:fndz} for $n+1$.
This proves that \eqref{eq:fndz} holds for all $n\ge 0$.
Moreover, according to \eqref{eq:ynotinVW},  $y\notin f^n([z,b])$ 
for any $n\ge 1$ because $[z,b]\subset V\cap W^u(z,f)$. 
By \eqref{eq:fndz}, $y\notin f^n([d,z])$ for all $n\ge 1$ (recall
that $b<y$). Thus  $y\notin f^n((d,b))$ for any $n\ge 1$.
But this contradicts the fact that $y\in W^u(z,f)$ because $(d,b)$ is a
neighborhood of $z$. This concludes the proof.
\end{proof}

\begin{lem}\label{lem:x<yinW}
Let $f$ be an interval map.  Let $z$ be a fixed point and
let $y$ be a point such that $y\in W^u(z,f)$ and $y>z$. Then there
exists $x\in W^u(z,f)$ such that $f(x)=y$ and $x<y$.
\end{lem}

\begin{proof}
We prove the lemma by refutation. Suppose that
\begin{equation}\label{eq:x<yfx<y}
\forall x\in W^u(z,f),\ x<y\Rightarrow f(x)<y.
\end{equation}
Then a straightforward induction, using the fact that 
$f(W^u(z,f))\subset W^u(z,f)$
(by Lemma~\ref{lem:homoclinic-basicprop}), gives:
\begin{equation}\label{eq:x<yfnx<y}
\forall x\in W^u(z,f),\ x<y\Rightarrow \forall n\ge 0, f^n(x)\in W^u(z,f)
\text{ and }f^n(x)<y.
\end{equation}
Let $V$ be a neighborhood of $z$ such that $\sup V<y$.
According to Lemma~\ref{lem:limhomoclinic=z}, there exist
$x\in W^u(z,f)\cap V$ and $n\ge 1$ such that $f^n(x)=y$. The fact that
$x\in V$ implies $x<y$. But this contradicts \eqref{eq:x<yfnx<y}. We deduce that
\eqref{eq:x<yfx<y} does not hold, that is, there exists $x_0\in W^u(z,f)$
such that $x_0<y$ and $f(x_0)\ge y$. Since $f(z)=z$, the continuity of
$f$ implies that there exists $x\in \langle z, x_0\rangle$ such that $f(x)=y$.
Then $x<y$. Moreover, $W^u(z,f)$ is an interval 
(by Lemma~\ref{lem:homoclinic-basicprop}) and it contains $z$ and
$x_0$, and so $W^u(z,f)$ contains $x$ too.
\end{proof}

\subsection{Homoclinic points and horseshoes}

In \cite{Bloc3}, Block showed that an interval map 
$f$ has an eventually periodic homoclinic point if and
only if $f$ has a periodic point whose period is not a power of $2$.
As we shall show in Theorem~\ref{theo:htop-power-of-2}, $f$ has
a periodic point  whose period is not a power of $2$ if and only if $f$ has
positive entropy, which is also equivalent to the fact that $f^n$ has a 
horseshoe for some $n$ (note that this theorem is 
posterior to \cite{Bloc3}).
We are going to show a result very close to Block's:
$f$ has  an eventually periodic homoclinic point if and only if some 
iterate of $f$ has a horseshoe. Moreover, the integer $n$ such that $f^n$ has
a horseshoe and the period of the eventually periodic homoclinic point
are related. 

The next result is a variant of \cite[Theorem~5]{Bloc3}.

\begin{prop}\label{prop:horseshoe-evperhomoclinic}
Let $f$ be an interval map having a horseshoe. Then there
exist points $x,z$ such that $x\neq z$, $f(z)=z$, $f(x)=z$ and
$x\in W^u(z,f)$. In particular, $x$ is an eventually periodic homoclinic point.
\end{prop}

\begin{proof}
According to Lemma~\ref{lem:particularhorseshoe}, there exist
points $a,b,c$ such that $f(a)=f(c)=a$, $f(b)=c$ and, either
$a<b<c$, or $a>b>c$. We assume that $a<b<c$, the other case being symmetric.
Then
$([a,b],[b,c])$ is a horseshoe; in particular there exist fixed points
in $[a,b]$ and in $[b,c]$. We set
$$
z_1:=\max\{x\in [a,b]\mid f(x)=x\}\quad\text{and}\quad
z_2:=\min\{x\in [b,c]\mid f(x)=x\}.
$$
There exist $x_1\in [b,c]$ and $x_2\in [a,b]$ such that $f(x_i)=z_i$ for
$i\in\{1,2\}$. 
Since $b$ is not a fixed point, we have $z_1<b<z_2$ and $x_i\neq z_i$
for $i\in\{1,2\}$. Moreover, there is no fixed point in $(z_1,z_2)$.
Therefore, according to Lemma~\ref{lem:2fixedpoints-unstableM}, there
exists $i\in\{1,2\}$ such that $(z_1,z_2)\subset W^u(z_i,f)$, and hence
$b\in W^u(z_i,f)$. By Lemma~\ref{lem:homoclinic-basicprop},
the points $c=f(b)$ and $a=f(c)$ belong to $W^u(z_i,f)$ too, and thus
$[a,c]\subset W^u(z_i,f)$ because $W^u(z_i,f)$ is an interval. 
Since $x_1,x_2\in [a,c] \subset W^u(z_i,f)$,
we conclude that the proposition holds for $z:=z_i$ and $x:=x_i$.
\end{proof}

The next proposition is \cite[Theorem~III.16]{BCop}, which is more precise
than the original result of Block \cite[Theorem~A2]{Bloc3}.

\begin{prop}\label{prop:evperhomoclinic-horseshoe}
Let $f$ be an interval map. Let $y$ be an eventually periodic
homoclinic point with respect to a fixed point (that is, there
exist a point $z$ and a positive integer $k$ such that 
$y\ne z$, $f(z)=z$, $y\in W^u(z,f)$ and $f^k(y)=z$).
Then $f^2$ has a horseshoe.
\end{prop}

\begin{proof}
Let $k\ge 1$ be the minimal integer such that $f^k(y)=z$. We set 
$y':=f^{k-1}(y)$. Then $y'\in W^u(z,f)$ by 
Lemma~\ref{lem:homoclinic-basicprop}, $y'\neq z$ because of the choice of 
$k$, and $f(y')=z$. We assume $y'>z$, the case $y'<z$ being symmetric.

If there exists $x\in (z,y')$ such that $f(x)=y'$, then $([z,x], [x,y'])$
is a horseshoe for $f$, and for $f^2$ too. From now on, we assume
that:
\begin{equation}\label{eq:fxneqy}
\forall x\in (z,y'),\ f(x)\neq y'.
\end{equation}
According to Lemma~\ref{lem:x<yinW}, there exists $x\in W^u(z,f)$
such that $x<y'$ and $f(x)=y'$. We set $w:=\max\{x\le y'\mid f(x)=y'\}$.
Then $w\in W^u(z,f)$ because $W^u(z,f)$ is an interval 
by Lemma~\ref{lem:homoclinic-basicprop}. Moreover, $w<z$ by \eqref{eq:fxneqy}
(notice that $w\notin\{y',z\}$ because $f(z)=f(y')=z\neq y'$). Since
$f(y')=z<y'$, the definition of $w$ and the continuity of $f$ imply that
\begin{equation}\label{eq:xinwy}
\forall x\in (w,y'),\ f(x)<y'.
\end{equation}
Suppose that $f(x)>w$ for all $x\in (w,y')$. Combined with
\eqref{eq:xinwy}, this implies $f((w,y'))\subset (w,y')$. Thus
$$
\forall x\in (w,y'),\forall n\ge 0,\ f^n(x)\neq w.
$$
But this contradicts Lemma~\ref{lem:limhomoclinic=z} because $(w,y')$
is a neighborhood of $z$ and $w$ is in $W^u(z,f)$. We deduce that there exists 
$x\in (w,y')$ such that $f(x)\le w$. Since $f(z)=z>w$, the continuity of $f$
implies that there exists $v\in (w,y')$ such that $f(v)=w$ and $v\neq z$.
If $v\in (w,z)$, then $([w,v],[v,z])$ is a horseshoe for $f$. If 
$v\in (z,y')$, it is easy to check that $[z,v], [v,y']$ form a horseshoe
for $f^2$. This concludes the proof.
\end{proof}

\begin{rem}
Propositions \ref{prop:horseshoe-evperhomoclinic} and 
\ref{prop:evperhomoclinic-horseshoe} can be restated for some iterate of~$f$:
\begin{itemize}
\item If $f^n$ has a horseshoe, then there exists an eventually periodic
homoclinic point with respect to a periodic point whose period divides $n$.
\item If $f$ has an eventually periodic homoclinic point  with respect to 
a periodic point of period $p$, then $f^{2p}$ has a horseshoe.
\end{itemize}
\end{rem}

It seems that the next result was first stated by Block and Coppel
\cite[Proposition~VI.35]{BCop}. We give a different proof.

\begin{prop}\label{prop:homoclinic-horseshoe}
Let $f$ be an interval map having a homoclinic point. Then
there exists a positive integer $n$ such that $f^n$ has a horseshoe.
\end{prop}

\begin{proof}
Let $y$ be a homoclinic point with respect to the periodic point $z$ and let
$p$ be the period of $z$. Then $y\in W^u(z,f^p)$ and $z\in \omega(y,f^p)$.
An induction using Lemma~\ref{lem:limhomoclinic=z} shows that there
exist a sequence of points $(y_n)_{n\ge 0}$ and a sequence of positive
integers $(k_n)_{n\ge 1}$ such that
\begin{itemize}
\item $y_0:=y$,
\item $\forall n\ge 1$, $y_n\in W^u(z,f^p)$ and $f^{pk_n}(y_n)=y_{n-1}$,
\item $\displaystyle\lim_{n\to+\infty}y_n=z$.
\end{itemize}
This implies that $y_n\neq z$ for all $n\ge 0$ (because $y\neq z$).
We assume that there are infinitely many integers $n$ such that $y_n>z$
(otherwise, there are infinitely many integers $n$ such that $y_n<z$
and the arguments are symmetric). 
Thus there exist $n'>n\ge 0$ such that $z<y_{n'}<y_n$.
We set $x_0:=y_n$, $x_1:=y_{n'}$, $m:=\sum_{i=n+1}^{n'} k_i$ and
$g:=f^{mp}$. In this way, $z<x_1<x_0$ and $g(x_1)=x_0$.

The point $z$ belongs to $\omega(y,f^p)=\omega(x_1,f^p)$. 
Since $z$ is a fixed point for
$f^p$, Lemma~\ref{lem:omega-set} implies that $z$ belongs to
$\omega(x_1,f^{pn})$ for all
$n\ge 1$, in particular $z\in \omega(x_1,g)$. Thus there exists $j\ge 1$ such
that $g^j(x_1)<x_1$ (by choosing $g^j(x_1)$ close enough to $z$).
This implies that there exists  $k\in\Lbrack 2,j\Rbrack$ such that
$$
g^k(x_1)<x_1\quad\text{and}\quad\forall i\in\Lbrack 1,k-1\Rbrack,\ g^i(x_1)\ge x_1
$$
(notice that $k=1$ is not possible because $g(x_1)=x_0$).
We deduce that
\begin{equation}\label{eq:x1ininterval}
x_1\in [g^k(x_1), g^{k-1}(x_1)].
\end{equation}
The interval $g^{k-1}([x_1,x_0])$ contains the points $g^{k-1}(x_0)=g^k(x_1)$
and $g^{k-1}(x_1)$. Thus $g^{k-1}([x_1,x_0])$ also contains $x_1$ by
\eqref{eq:x1ininterval}. This implies that
$g^k([x_1,x_0])$ contains $g(x_1)$ and $g^k(x_1)$ with $g(x_1)=x_0$ and 
$g^k(x_1)<x_1$, so
\begin{equation}\label{eq:gkx1x0}
g^k([x_1,x_0])\supset [x_1,x_0].
\end{equation}

On the other hand, $g([z,x_1])\supset [z,x_0]\supset [z,x_1]$. Thus there
exists $x_2\in (z,x_1)$ such that $g(x_2)=x_1$, and we have
\begin{equation}\label{eq:coveringx2x1}
g([x_2,x_1])\supset [x_1,x_0].
\end{equation}
As above, since $z\in \omega(x_1,g^k)$ and $z<x_2$,
there exists $j\ge 1$ such that 
$g^{kj}(x_1)<x_2$. Then $g^{kj}([x_1,x_0])$ contains $g^{kj}(x_1)<x_2$ and
it also contains $x_0$ by \eqref{eq:gkx1x0}. Thus
\begin{equation}\label{eq:coveringx1x0}
g^{kj}([x_1,x_0])\supset [x_2,x_0]=[x_2,x_1]\cup [x_1,x_0].
\end{equation}
Let $J:=[x_2,x_1]$ and $K:=[x_1,x_0]$.
The coverings given by \eqref{eq:coveringx2x1} and \eqref{eq:coveringx1x0}
are represented in Figure~\ref{fig:horsehoe-homoclinic}.
\begin{figure}[htb]
\centerline{\includegraphics{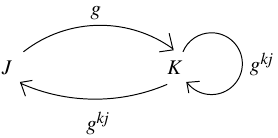}}
\caption{The coverings between the intervals $J:=[x_2,x_1]$ and
$K:=[x_1,x_0]$.}
\label{fig:horsehoe-homoclinic}
\end{figure}

If we consider the following chains of coverings:
\begin{gather*}
J\labelarrow{g}K\labelarrow{g^{kj}}K\labelarrow{g^{kj}}J,\qquad
J\labelarrow{g}K\labelarrow{g^{kj}}K\labelarrow{g^{kj}}K,\\
K\labelarrow{g^{kj}}J\labelarrow{g}K\labelarrow{g^{kj}}J,\qquad
K\labelarrow{g^{kj}}J\labelarrow{g}K\labelarrow{g^{kj}}K,
\end{gather*}
we see that $J, K$ form a horseshoe for
$g^{1+2kj}=f^{p(1+2kj)}$.
\end{proof}

According to Propositions \ref{prop:horseshoe-evperhomoclinic}
and \ref{prop:homoclinic-horseshoe}, the existence of a
homoclinic point implies that $f^n$ has a horseshoe for some $n$;
and if $f^n$ has a horseshoe, then $f$ has an eventually periodic 
homoclinic point. This leads to the following theorem.

\begin{theo}
Let $f$ be an interval map. The following are equivalent:
\begin{enumerate}
\item $h_{top}(f)>0$,
\item $f$ has an eventually periodic homoclinic point,
\item $f$ has a homoclinic point.
\end{enumerate}
\end{theo}

\begin{proof}
The implication (ii)$\Rightarrow$(iii) is trivial. According to
Misiurewicz's Theorem~\ref{theo:Misiurewicz}, the topological entropy of $f$
is positive if and only if there exists $n\ge 1$ such that $f^n$ has
a horseshoe. Then the implications
(i)$\Rightarrow$(ii) and (iii)$\Rightarrow$(i)
follow straightforwardly from Propositions
\ref{prop:horseshoe-evperhomoclinic} and \ref{prop:homoclinic-horseshoe}
respectively.
\end{proof}

%*************
\subsection*{Remarks on graph maps}

The notions of an unstable manifold and a homoclinic point can be extended
with no change to graph maps. In view of the definition of horseshoe for
graph maps, it is natural to think that 
Proposition~\ref{prop:horseshoe-evperhomoclinic} can be generalized to
graph maps. Indeed, Makhrova proved that a tree map of positive entropy
has a homoclinic point \cite[Corollary~1.2]{Mak}; and Kočan, Kornecká-Kurková 
and Málek showed the same result for graph maps \cite[Theorem~1]{KKM}.
Recall that a graph map $f$ has positive topological entropy  if and only
if $f^n$ has a horseshoe for some $n$ by Theorem~\ref{theo:htop-horseshoeG}.

\begin{theo}\label{theo:htop-homoclinic-G}
Let $f\colon G\to G$ be a graph map of positive topological entropy.
Then $f$ has an eventually periodic homoclinic point.
\end{theo}

The converse of Theorem~\ref{theo:htop-homoclinic-G} holds for tree maps
\cite[Corollary~1.2]{Mak} but not for graph maps \cite[Example~3]{KKM}.

\begin{theo}
Let $f\colon T\to T$ be a tree map. If $f$ has a homoclinic point, then
$h_{top}(f)>0$.
\end{theo}

\begin{prop}
There exists a circle map $f\colon \IS\to \IS$ of zero topological entropy
having an eventually periodic homoclinic point.
\end{prop}

%*************************************************************************
%*************************************************************************
\section[Upper bounds for Lipschitz and piecewise monotone maps]{Upper bounds for entropy of Lipschitz and~piecewise~monotone~maps}

An interval map can have an infinite topological entropy, as 
illustrated in Example~\ref{ex:htop=infty}. 
However Lipschitz (in particular $C^1$) interval maps and piecewise monotone 
interval maps have finite topological entropy.

\begin{ex}\label{ex:htop=infty}
We choose an increasing sequence $(a_n)_{n\ge 0}$ with $a_0:=0$ and
$\lim_{n\to+\infty}a_n=1$. Let $a_{-1}:=0$. 
We set $I_n:=[a_{n-1}, a_n]$ for all $n\ge 1$.
We consider a continuous map $f\colon [0,1]\to [0,1]$ which is rather similar
to the map of Example~\ref{ex:non-accessible-endpoints} but with
$2n+1$ linear pieces in $I_n$. This map is represented in 
Figure~\ref{fig:htop=infty}.
More precisely, $f$ is defined by
\begin{gather*}
\forall n\ge 0,\ f_n(a_n):=a_n,\quad f(1):=1\\
\forall n\ge 1,\ \forall i\in\Lbrack 1,2n\Rbrack,\ f\left(a_{n-1}+i\cdot \frac{a_n-a_{n-1}}{2n+1}\right):=
\left\{\begin{array}{ll}a_{n-2}&
\text { if }i\text{ odd,}\\ a_{n+1}&\text{ if }i\text{ even,}
\end{array}\right.
\end{gather*}
and $f$ is linear between these points.
\begin{figure}[htb]
\centerline{\includegraphics{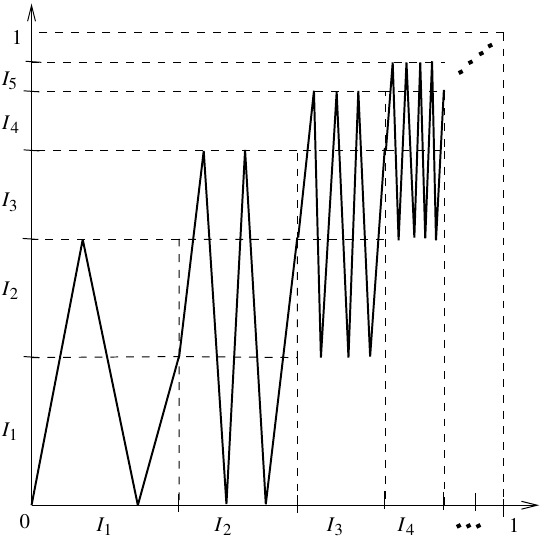}}
\caption{This map is topologically mixing and its topological 
entropy is infinite.} 
\label{fig:htop=infty}
\end{figure}

The map $f$ clearly has a $(2n+1)$-horseshoe in $I_n$ for every $n\ge 1$, 
Therefore, $h_{top}(f)=+\infty$ by Proposition~\ref{prop:horseshoe-htop}.
Moreover, the same arguments as in Example~\ref{ex:non-accessible-endpoints}
show that $f$ is topologically mixing.
\end{ex}

The next result is a particular case of \cite[Proposition~(14.20)]{DGS}, which
states that, if $(X,f)$ is a topological 
dynamical system with $X\subset \IR^d$ and  if $f$ is $\lambda$-Lipschitz 
for some $\lambda\ge 1$, then $h_{top}(f)\le d\log\lambda$.

\begin{prop}\label{prop:htop-lipschitz}
Let $f\colon I\to I$ be an interval map and $\lambda\ge 1$. If $f$ is 
$\lambda$-Lipschitz, then $h_{top}(f)\le \log\lambda$.
\end{prop}

\begin{proof}
Let $\eps>0$ and $n\ge 1$. Let $E=\{x_1<x_2<\cdots<x_s\}$ be an
$(n,\eps)$-separated set of cardinality
$s:=s_n(f,\eps)$. For every $i\in\Lbrack 1,s-1\Rbrack$, there exists
$k\in\Lbrack 0,n-1\Rbrack$ such that $|f^k(x_{i+1})-f^k(x_i)|>\eps$. 
Since $f$ is $\lambda$-Lipschitz with $\lambda\ge 1$, 
$$
|f^k(x_{i+1})-f^k(x_i)|\le\lambda^k|x_{i+1}-x_i|\le \lambda^n|x_{i+1}-x_i|.
$$
Thus $x_{i+1}-x_i\ge \lambda^{-n}\eps$ and $x_s-x_1\ge (s-1)\lambda^{-n}\eps$. 
Since $x_s-x_1\le |I|$, this implies that 
$$
s \le \frac{|I|}{\eps}\lambda^n+1.
$$
Finally, $h_{top}(f)\le\log\lambda$ by Bowen's formula 
(Theorem~\ref{theo:Bowen-formula}).
\end{proof}

%***************************************************************
%\subsection{Upper bound for piecewise monotone maps}

In \cite{MS2}, Misiurewicz and Szlenk showed that the topological entropy of a 
piecewise monotone interval map $f$ is equal to the exponential growth rate
of the minimal number $c_n$ of monotone subintervals for $f^n$.
Furthermore, $h_{top}(f)$ is less than or equal to $\frac{1}{n}\log c_n$ 
for all $n\ge 1$, which may be useful to estimate the entropy
of a given map since we may not know $c_n$ for all $n$.
We first state two lemmas before proving this result.

\begin{defi}
Let $f$ be a piecewise monotone map. A \emph{monotone cover
(resp. partition) for $f^n$} is a cover (resp. partition) $\CC$ such that, 
for all $C\in\CC$, $C$ is an interval and
$f^n|_C$ is monotone. 
\end{defi}

\begin{lem}\label{lem:interval+monotone}
Let $f$ be an interval map. 
If $\CA$ and $\CB$ are  monotone covers for $f^n$ and 
$f^k$ respectively, then $\CA\vee f^{-n}(\CB)$ is a monotone cover for 
$f^{n+k}$. In particular,
if $\CA$ is a monotone cover for $f$, then  $\CA^n$ is a monotone cover for 
$f^n$ for all $n\ge 1$.
\end{lem}

\begin{proof} 
Let $J\in\CA$ and $K\in\CB$.
We set $g:=f^n|_J$.
Since $g$ is monotone, the set $g^{-1}(K)=J\cap f^{-n}(K)$ is
an interval. Moreover, $f^{n+k}|_{J\cap f^{-n}(K)}=f^k|_K\circ 
g|_{g^{-1}(K)}$ is monotone as a composition of two monotone maps.
This implies that $\CA\vee f^{-n}(\CB)$ is a monotone cover for 
$f^{n+k}$.
The second assertion of the lemma trivially follows from the first one.
\end{proof}

The next result is stated in \cite[Remark 1]{MS2}.

\begin{prop}\label{prop:monotone-cover}
Let $f\colon I\to I$ be a piecewise monotone interval map and $\CA$ a 
monotone cover. Then $h_{top}(f)= h_{top}(\CA,f)$.
\end{prop}

\begin{proof}
Let $\CU$ be an open cover, and let $\CV$
be the open cover composed of the connected
components of the elements of $\CU$. 
We fix an integer $n\ge 1$ and an element $A\in \CA^n$.
We set $\CV^n\cap A:=\{V\cap A\mid V\in\CV^n\}$.
For all $i\in\Lbrack 0,n-1\Rbrack$, $A$ is a subinterval of an element
of $\CA^i$, and thus $f^i|_A$ is monotone because $\CA^i$ is a monotone cover 
for $f^i$ by Lemma~\ref{lem:interval+monotone}.
Thus, for every $U\in\CV$, $A\cap f^{-i}(U)$ is an interval
(note that this interval may not be open because $A$ is not assumed to be
open).  This implies that
all elements of $\CV^n\cap A$ are 
subintervals of $A$. Moreover, their endpoints are in
the set
$$
\bigcup_{i=0}^{n-1}\bigcup_{U\in\CV}\End{A\cap f^{-i}(U)}.
$$ 
Therefore the number of endpoints of the elements of $\CV^n\cap A$
is at most $2n\#\CV$. Since a nonempty interval is determined by its two
endpoints and by its type (open, close, half-open, half-close), 
we deduce that $\#(\CV^n\cap A)\le 4(2n\#\CV)^2=(4n\#\CV)^2$.
Let $\widetilde{\CV}_n$ (resp. $\widetilde{\CA}_n$) be a subcover of minimal
cardinality of $\CV^n$ (resp. $\CA^n$). We then have
$$
\#\widetilde{\CV}_n\le \sum_{A\in \widetilde{\CA}_n}\#(\widetilde{\CV}_n\cap A)
\le \#(\widetilde{\CA}_n) (4n\#\CV)^2.
$$
It follows that
$$
\log N_n(\CV,f)\le \log N_n(\CA,f)+2
\log(4n\#\CV).
$$
Dividing by $n$ and taking the limit when $n$ goes to infinity,
we get
$$h_{top}(\CV,f)\le h_{top}(\CA,f).$$
Moreover, $h_{top}(\CU,f)\le h_{top}(\CV,f)$ because $\CV \prec \CU$. 
Since what precedes is valid for all
open covers $\CU$, we deduce that $h_{top}(f)\le h_{top}(\CA,f)$.
It remains to show the reverse inequality.

We fix $n\ge 1$. Let $\CB:=\CA\vee f^{-1}\CA\vee\cdots\vee 
f^{-(n-1)}\CA$.
From now on, we work with the map $g:=f^n$ and the
iterated covers (like $\CB^k$) will be relative to $g$. 
By Lemma~\ref{lem:interval+monotone}, $\CB$ is a monotone cover for $g$.
Let $\eps>0$ be such that
$$
\eps<\min\{ |B|\mid B\in \CB, B\text{ non degenerate}\}.
$$
Let $E:=\bigcup_{B\in\CB}\End{B}$ be the set of endpoints of $\CB$.
We define the open cover
$$
\CU:=\{\Int{B}\mid B\in\CB\}\cup\{(x-\eps,x+\eps)\cap I\mid x\in E\}.
$$
For every $x\in E$, the interval
$(x-\eps,x+\eps)$ meets at most three elements of $\CB$ (one of them
may be reduced to $\{x\}$)
because of the choice of $\eps$. Thus, for all $U\in\CU$,
$\#\{B\in \CB\mid U\cap B\neq \emptyset\}\le 3$.
This implies that
$$\forall\, k\ge 1, \forall\, V\in \CU^k,\ 
\#\{B\in \CB^k\mid V\cap B\neq\emptyset\}\le 3^k.
$$
Consequently, if $\widetilde \CU_k$ is a subcover of minimal
cardinality of $\CU^k$, we have
$$
N_k(\CB,g)\le \sum_{V\in\widetilde \CU_k}\#\{B\in\CB^k\mid V\cap B
\neq \emptyset\}\le 3^k \#\, \widetilde \CU_k= 3^k N_k(\CU,g).
$$
Dividing by $nk$ and taking the limit when $k$ goes to infinity, we get
\begin{equation}\label{eq:htopB}
\frac 1n h_{top}(\CB,g)\le \frac 1n h_{top}(\CU,g)+
\frac{\log 3}{n}\le \frac 1n h_{top}(g)+
\frac{\log 3}{n}.
\end{equation}
Since $\frac 1n h_{top}(\CB,f^n)=h_{top}(\CA,f)$ and
$\frac 1n h_{top}(g)=h_{top}(f)$, we deduce from \eqref{eq:htopB} that
$h_{top}(\CA,f)\le h_{top}(f)+\frac{\log 3}{n}$.
Finally, taking the limit when $n$ goes to infinity, we conclude that
$h_{top}(\CA,f)\le h_{top}(f).
$
\end{proof}

\begin{prop}\label{prop:htop-piecewise-monotone}
Let $f$ be a piecewise monotone interval map and, for all $n\ge 1$, 
let $c_n$ be the minimal cardinality of a monotone partition for $f^n$.
Then
$$h_{top}(f)=\lim_{n\to+\infty}\frac{1}{n}\log c_n=\inf_{n\ge 1}\frac{1}{n}\log c_n.
$$
\end{prop}

\begin{proof}
For every $n\ge 1$, let $\CA_n$ be a monotone partition for $f^n$ with 
minimal cardinality, that is, $\#\CA_n=c_n$.
By Lemma~\ref{lem:interval+monotone}, $\CA_n\vee f^{-n}(\CA_k)$ is
a monotone partition for $f^{n+k}$, and thus, by definition of $c_{n+k}$,
$$
c_{n+k}\le \#(\CA_n\vee f^{-n}(\CA_k))\le \#\CA_n\cdot \#\CA_k=c_n\cdot c_k.
$$
This means that the sequence $(\log c_n)_{n\ge 1}$ is sub-additive.
Thus, by Lemma~\ref{lem:subadditive},
$\lim_{n\to+\infty}\frac{1}{n}
\log c_n$ exists and is equal to $\inf_{n\ge 1}\frac{1}{n}\log c_n$. 
Applying Proposition~\ref{prop:monotone-cover}  to $f^n$ and $\CA_n$,
we get $h_{top}(f^n)= h_{top}(\CA_n,f^n)$. Since $h_{top}(\CA_n,f^n)
\le \log N(\CA^n)=\log \# \CA^n$ (see Section~\ref{subsec:htop-covers}),
we have $h_{top}(f^n)\le \log\# \CA_n=\log c_n$. Consequently,
$$
h_{top}(f)\le \lim_{n\to+\infty}\frac{1}{n}\log c_n.$$
It remains to show the reverse inequality.

We fix $n\ge 1$. From now on, we work with the map $g:=f^n$, and $(\CA_n)^k$
will denote the iterated partition relative to $g$. 
By Lemma~\ref{lem:interval+monotone}, 
$(\CA_n)^k$ is a monotone partition for $g^k$, so $c_{nk}\le N_k(\CA_n,g)$.
Dividing by $nk$ and taking the limit when $k$ goes to infinity, we deduce that
$$
\lim_{k\to+\infty}\frac{1}{nk}\log c_{nk}\le \frac 1n h_{top}(\CA_n,g).
$$
According to Proposition~\ref{prop:monotone-cover},
$h_{top}(\CA_n,g)=h_{top}(g)$. Thus
$$
\lim_{m\to +\infty} \frac 1 m \log c_m=\lim_{k\to+\infty}\frac{1}{nk}\log c_{nk}\le \frac 1n h_{top}(f^n)=h_{top}(f).
$$
This concludes the proof.
\end{proof}

\begin{rem}
The bounds of Propositions \ref{prop:htop-lipschitz} and 
\ref{prop:htop-piecewise-monotone} are optimal since they can be reached:
the map $T_p$ in Example~\ref{ex:tent-map} is $p$-Lipschitz and has a 
$p$-horseshoe, and thus $h_{top}(T_p)=\log p$ by  Propositions \ref{prop:horseshoe-htop} and \ref{prop:htop-lipschitz}. Moreover, it can be easily
computed that, for all $n\ge 1$, the minimal cardinality of a monotone
partition for $T_p^n$ is $c_n=p^n$. Therefore, the inequality
$h_{top}(T_p)\le \frac 1n \log c_n$ given by
Proposition~\ref{prop:htop-piecewise-monotone} 
is an equality for all $n$ in this example.
\end{rem}

%***********************************************************************
\section{Graph associated to a family of intervals}

\subsection{A generalization of horseshoes}
The existence of a horseshoe implies positive entropy because an
exponential number of chains of intervals of a given length can be
made by using the intervals forming the horseshoe.  This idea can be
generalized by counting the number of chains within a family of
closed  intervals. A convenient way to determine the possible chains
of  intervals is to build a directed graph. This idea is originally
due to Bowen and Franks \cite{BF} and was improved by Block,
Guckenheimer, Misiurewicz and Young \cite{BGMY}.

\begin{defi}\index{graph associated to a family of intervals}
Let $f$ be an interval map and let $I_1,\ldots, I_p$ be
non degenerate closed intervals with disjoint interiors. The
\emph{graph associated to the intervals $I_1,\ldots, I_p$} is
the directed graph $G$ whose set of vertices is $\{I_1,\ldots, I_p\}$ and,
for all $i,j\in\Lbrack 1,p\Rbrack$, there are exactly $k$ arrows from $I_i$ 
to $I_j$ if $k$ is the maximal integer such that $I_i$ covers $k$ times $I_j$.

If $P=\{p_0<p_1<\cdots<p_n\}$ is a finite set containing at least 
two points, the \text{\emph{$P$-intervals}}\index{P-interval@$P$-interval}\index{interval ($P$- )} are 
$[p_0,p_1]$, $[p_1,p_2],\ldots, [p_{n-1}, p_n]$. The graph associated to 
the $P$-intervals is denoted by $G(f|P)$ and its adjacency matrix
by $M(f|P)$.
\label{notation:Gfp}\label{notation:Mfp}
\index{G(f,P)@$G(f\vert P)$}\index{M(f,P)@$M(f\vert P)$}
\end{defi}

\begin{rem}
If $P=\{x_1<\ldots<x_n\}$ is a periodic orbit,
the graph $G(f|P)$ contains the graph of the periodic orbit introduced
in Definition~\ref{defi:graph-periodic-orbit}. These two graphs 
coincide if $f$ is monotone on every $P$-interval.
\end{rem}

The next result follows easily from the definitions. 

\begin{prop}\label{prop:path-matrix}
Let $f$ be an interval map and let $I_1,\ldots, I_p$ be
non degenerate closed intervals with disjoint interiors. Let
$G$ be the graph associated to $I_1,\ldots, I_p$.
Then, for every $n$-tuple 
$\{i_1,\ldots, i_n\}\in\Lbrack 1,p\Rbrack^n$, $(I_{i_1},I_{i_2},\ldots, I_{i_n})$ is
a chain of intervals if and only if there is a path 
$I_{i_1}\to I_{i_2}\to\cdots\to I_{i_n}$ in $G$. 
\end{prop}

We are going to show that the topological entropy of an interval map is
greater than or equal to the logarithm of the spectral
radius of the adjacency matrix of the graph associated to a family of
intervals. We need some more 
definitions and results about matrices. 
One can refer to \cite[Chapter~1]{Sen} or \cite[\S 1.3]{Kit} for the proofs.

\begin{prop}\label{prop:eigenvalue-rootchi}
Let $M$ be a square matrix of size $n\times n$. 
A complex number $\lambda$ is an eigenvalue of
$M$ if and only if $\lambda$ is a root of the characteristic
polynomial of $M$, which is $\chi_M(X):=\det(M-X \Id)$, where $\Id$ is
the identity matrix of size $n\times n$.
\end{prop}

\begin{defi}[spectral radius]\index{spectral radius}
Let $M$ be a square matrix. The \emph{spectral radius of $M$} is
$$
\lambda(M):=\max\{|\lambda|\mid \lambda \text{ eigenvalue of } M \}.
$$
\end{defi}

The next results can be easily proved by using
the Jordan normal form of a square matrix.

\begin{prop}[Gelfand's formula]\label{prop:spectral-radius}
Let $M$ be a square matrix. Then $\displaystyle\lambda(M)=
\lim_{n\to+\infty}\|M^n\|^{\frac{1}{n}}$.
\end{prop}

\begin{lem}\label{lem:lambdaMn}
Let $A,B$ be two square matrices of the same size. Then 
$\lambda(A^k)=\lambda(A)^k$ for all positive integers $k$.
\end{lem}

\begin{proof}
According to Proposition~\ref{prop:spectral-radius},
$$
\lambda(A^k)=\lim_{n\to+\infty}\|A^{kn}\|^{\frac{1}{n}}
=\lim_{n\to+\infty}\left(\|A^{kn}\|^{\frac{1}{kn}}\right)^k=\lambda(A)^k;
$$
the last equality comes from the facts that 
$(kn)_{n\ge 1}$ is a subsequence of $\IN$ and the map $t\mapsto t^k$ 
is continuous.
\end{proof}

\begin{defi}
Let $A=(a_{ij})_{1\le i,j\le p}$ be a square matrix.
The matrix $A$ is \emph{non negative}\index{non negative matrix},
\label{notation:nonnegativematrix}
or equivalently $A\ge 0$,
if $a_{ij}\ge 0$ for all $i,j\in\Lbrack 1,p\Rbrack$, and 
\emph{positive}\index{positive matrix},
\label{notation:positivematrix}
or equivalently $A>0$, if
$a_{ij}> 0$ for all $i,j\in\Lbrack 1,p\Rbrack$. If $B$ is another matrix of the
same size, then $A\le B$ (resp. $A<B$) means that $B-A\ge 0$ (resp. $B-A>0$).
\label{notation:AlessBmatrix}

For all integers $n\ge 1$, let $(a_{ij}^n)_{1\le i,j\le p}$ be the 
coefficients of $A^n$. The matrix $A$ is called:
\begin{itemize}
\item \emph{irreducible}\index{irreducible matrix}
if for all $i,j\in\Lbrack 1,p\Rbrack$, there exists $n\ge 1$ such that $a_{ij}^n>0$,
\item \emph{primitive}\index{primitive matrix} if there exists 
$n\ge 1$ such that $A^n>0$.
\end{itemize}
\end{defi}

\begin{lem}\label{lem:irreducible-submatrices}
Let $A$ be a non negative square matrix. Then there exists a permutation
matrix $P$ such that $M:=P^{-1}AP$ is equal to
\begin{equation}\label{eq:irreducible-submatrices}
M=\left(\begin{array}{ccccc}
M_1 & 0   & 0   & \cdots & 0 \\
{}* & M_2 & 0   & \cdots & 0 \\
{*} & *   & M_3 & \cdots & 0 \\
\vdots & \vdots & \vdots & \ddots & \vdots \\
{*} & *   & *   & \cdots & M_k
\end{array}\right)
\end{equation}
where, for every $i\in\Lbrack 1,k\Rbrack$, $M_i$ is either an irreducible 
square matrix or is equal to the $1\times 1$ matrix $(0)$, and the $*$'s 
represent possibly non zero  submatrices.

In particular, if $G$ is a directed graph, the vertices of $G$ can be labeled
in such a way that the adjacency matrix of $G$ is of the form given by
\eqref{eq:irreducible-submatrices}.
\end{lem}

\begin{theo}[Perron-Frobenius]\label{theo:Perron-Frobenius}
\index{Perron-Frobenius Theorem}
Let $A$ be an irreducible non negative square matrix. 
\begin{enumerate}
\item $\lambda(A)$ is a positive eigenvalue, which is a simple root of 
the characteristic polynomial of $A$. 
If in addition $A$ is primitive, then $|\mu|<\lambda(A)$ for every
 eigenvalue $\mu \neq\lambda(A)$.
\item If $B$ is another matrix of the same size such that $0\le B\le A$
and $B\neq A$, then $\lambda(B)<\lambda(A)$.
\end{enumerate}
\end{theo}

\begin{cor}\label{cor:Perron-Frobenius}
Let $A$ be a non negative square matrix. Then $\lambda(A)$ is a
non negative eigenvalue and $|\mu|\le \lambda(A)$ for every
eigenvalue $\mu$; $\lambda(A)$ is called the 
\emph{maximal eigenvalue} of $A$.\index{maximal eigenvalue}
\end{cor}

\begin{proof}
According to Lemma~\ref{lem:irreducible-submatrices}, there exists a permutation
matrix $P$ such that $M:=P^{-1}AP$ is of the form given by
Equation~(\ref{eq:irreducible-submatrices}). Let $i\in\Lbrack 1,k\Rbrack$.
If $M_i$ is the $1\times 1$ null matrix, then $0$ is its only
eigenvalue, that is, $\lambda(M_i)=0$. If $M_i$ is irreducible, then, by 
Theorem~\ref{theo:Perron-Frobenius}, $\lambda(M_i)$ is a positive
eigenvalue of $M_i$ and every eigenvalue $\mu$ of
$M_i$ satisfies $|\mu|\le \lambda(M_i)$. Let 
$\lambda:=\max\{\lambda(M_1),\lambda(M_2),\ldots,\lambda(M_k)\}$. 
The set of eigenvalues of $M$ is equal to the union of the sets of 
eigenvalues of the matrices $(M_i)_{1\le i\le k}$.  Moreover,
the set of eigenvalues of $A$ is equal to the set of eigenvalues of $M$.
Consequently, $\lambda(A)=\lambda$, $\lambda$ is an eigenvalue of $A$
and, if $\mu$ is an eigenvalue of $A$, then $|\mu|\le \lambda$.
\end{proof}

We are now ready to prove the result stated at the beginning of the section.
Its proof is given by Block, Guckenheimer, Misiurewicz and Young in 
\cite{BGMY} with few details.

\begin{prop}\label{prop:matrix-htop}
Let $f$ be an interval map and let $I_1,\ldots, I_p$ be
non degenerate closed intervals with disjoint interiors. Let $G$ be a
subgraph of the graph associated to $I_1,\ldots, I_p$ and
$M$ the adjacency matrix of $G$. 
Then $h_{top}(f)\ge \log \lambda(M)$.
\end{prop}

\begin{proof}
According to Lemma~\ref{lem:irreducible-submatrices}, 
we may re-label the intervals $I_1,\ldots, I_p$
in such a way that the adjacency matrix of $G$ is of the form given in
\eqref{eq:irreducible-submatrices}.  We set  $\lambda:=\lambda(M)$.
By Corollary~\ref{cor:Perron-Frobenius}, $\lambda$ is an eigenvalue of $M$.
We assume that $\lambda>0$, otherwise there is nothing to prove.
We keep the notation of \eqref{eq:irreducible-submatrices}.
Since $\lambda$ is the maximal eigenvalue of $M$,
it is also the maximal eigenvalue of 
$M_p$ for some $p\in\Lbrack 1,k\Rbrack$.
Let $\CI$ be the finite set of indices supporting $M_p$.
For all integers $n\ge 1$, we write $(M_p)^n=(m_{ij}^n)_{i,j\in \CI}$.
Then, by Proposition~\ref{prop:spectral-radius},
$$
\lim_{n\to +\infty} \frac 1n \log \left(\sum_{i,j\in \CI} m_{ij}^n\right)=\log\lambda.
$$
Thus, by Lemma~\ref{lem:sequences-an-bn1}, 
there exist two indices $i_0,j_0\in \CI$ such that
\begin{equation}\label{eq:logaij}
\limsup_{n\to+\infty}\frac 1n \log m_{i_0j_0}^n=\log \lambda.
\end{equation}

According to Proposition~\ref{prop:Mn}, for all $i,j\in\CI$ and all $n\ge 1$,
$m_{ij}^n$ is equal to the number of paths of length $n$ from
$I_i$ to $I_j$ in the directed graph $G$.
Since $M_p$ is irreducible (notice that $M_p\neq (0)$ because $\lambda(M_p)>0$),
there exists $n_0\ge 1$
such that $m_{j_0i_0}^{n_0}>0$. Thus there exists a path $\CA$ of length $n_0$
from $I_{j_0}$ to $I_{i_0}$ in $G$. For every path $\CB$ of length $n$ from $I_{i_0}$
to $I_{j_0}$, the concatenated path $\CB\CA$ is a path of length $n+n_0$ from
$I_{i_0}$ to itself. Therefore,
the number of paths of length $n+n_0$ from $I_{i_0}$ to itself
is greater than
or equal to the number of paths of length $n$ from $I_{i_0}$ to $I_{j_0}$. 
In other words, $m_{i_0i_0}^{n+n_0}\ge m_{i_0j_0}^n$ for all $n\ge 1$. 
Combined with \eqref{eq:logaij}, this implies:
$$
\limsup_{n\to+\infty}\frac 1n \log m_{i_0i_0}^n\ge \log\lambda.
$$
We fix $\eps\in(0,\lambda)$ and a positive integer $n$ such that
\begin{equation}\label{eq:a-ii}
\frac 1n \log m_{i_0i_0}^n\ge \log (\lambda-\eps).
\end{equation}
Then $N:=m_{i_0i_0}^n$ is the number of paths of length
$n$ from $I_{i_0}$ to itself in $G$ and, by 
Proposition~\ref{prop:path-matrix}, there exist $N$ distinct chains of intervals
of the form $(I_{i_0},I_{i_1},\ldots,I_{i_n})$ with $i_0$ fixed, $i_n=i_0$ and
$i_j\in\CI$ for all $j\in\Lbrack 1,n-1\Rbrack$.
Then, according to Lemma~\ref{lem:chain-of-intervals}(iii), 
there exist $N$ closed
subintervals $K_1,\ldots, K_N$ with pairwise disjoint interiors such
that, $\forall j\in\Lbrack 1,N\Rbrack$, $K_j\subset I_{i_0}$ and $f^n(K_j)=I_{i_0}$. 
Thus $(K_1,\ldots, K_N)$ is an $N$-horseshoe for $f^n$, which implies
that $h_{top}(f^n)\ge \log N$ by Proposition~\ref{prop:horseshoe-htop}.
Then, since $h_{top}(f)=\frac{1}{n}h_{top}(f^n)$ by 
Proposition~\ref{prop:htop-Tn}
and $\frac{1}{n}\log N\ge \log(\lambda-\eps)$ by \eqref{eq:a-ii},
we deduce that $h_{top}(f)\ge \log(\lambda-\eps)$. Finally, letting $\eps$ tend
to zero, we get $h_{top}(f)\ge \log\lambda$.
\end{proof}

\begin{rem}
If $(I_1,\ldots, I_p)$ is a $p$-horseshoe, then the 
graph associated to $(I_1,\ldots, I_p)$ contains a subgraph whose
transition matrix
is $M=(m_{ij})_{1\le i,j\le p}$ with $m_{ij}=1$ for all 
$i,j\in\Lbrack 1,p\Rbrack$. Moreover, $ ^t(1,1,\ldots,1)$ is clearly an eigenvector
of $M$ for the eigenvalue $p$, so $\lambda(M)\ge p$. This shows that Proposition~\ref{prop:horseshoe-htop}
is a particular case of Proposition~\ref{prop:matrix-htop}.
\end{rem}

%**********************
\subsection{The connect-the-dots map associated to a finite invariant set}

When considering the graph associated to $P$-intervals, a particularly
convenient case is when $P$ is a periodic orbit, or more generally a
finite invariant set. Knowing only the values of $f$ on $P$ is sufficient
to determine a subgraph of $G(f|P)$. This subgraph is intimately related to
the ``connect-the-dots'' map $f_P$ associated to $f$ on $P$ (for all
$x\in P$, plot the points $(x,f(x))$; then connect linearly the dots to
get the graph $y=f_P(x)$).

\begin{defi}
Let $f\colon I\to I$ be an interval map and let $P=\{p_0<p_1<\cdots< p_k\}$ be
a finite subset of $I$ with $k\ge 1$. 
The map $f$ is \emph{$P$-monotone}\index{P-monotone map@$P$-monotone map}\index{monotone ($P$- )} if $f(P)\subset P$, $I=[p_0,p_k]$ and 
$f$ is monotone on every $P$-interval $[p_0,p_1],\ldots, [p_{k-1},p_k]$.
The map $f$ is \emph{$P$-linear}\index{P-linear map@$P$-linear map}\index{linear ($P$- )} if in addition $f$ is linear on every
$P$-interval.

If $P$ is an invariant set, let $f_P\colon [p_0,p_k]\to [p_0,p_k]$ denote the 
unique $P$-linear map agreeing with $f$ on $P$. The map $f_P$ is called
the \emph{connect-the-dots}\index{connect-the-dots map}
\label{notation:ctdfP}
map associated to 
$f|_P\colon P\to P$. For short, we write $G(f_P)$ and $M(f_P)$ instead of 
$G(f_P|P)$ and $M(f_P|P)$.
\label{notation:ctdGfP}\label{notation:cdtMfP}
\end{defi}

Let $P=\{p_0<p_1<\cdots <p_k\}$ be a finite invariant set with $k\ge 1$
and, for every $i\in\Lbrack 0,k\Rbrack$, let $\sigma(i)\in\Lbrack 0,k\Rbrack$ 
be such that $f(p_i)=p_{\sigma(i)}$. 
The graph $G(f_P)$ is determined by the values of $f$ on $P$. Indeed,
for every $i\in\Lbrack 0,k-1\Rbrack$, there is an arrow $[p_i,p_{i+1}]\to [p_j,p_{j+1}]$
in $G(f_P)$ if and only if $[p_j,p_{j+1}]\subset \langle p_{\sigma(i)},
p_{\sigma(i+1)}\rangle$, i.e., either $\sigma(i)\le j<j+1\le \sigma(i+1)$
or $\sigma(i+1)\le j<j+1\le \sigma(i)$.
Notice that, if $f$ is $P$-monotone, then $G(f|P)=G(f_P)$.

\medskip
The next lemma follows trivially from the definition.

\begin{lem}\label{lem:image-Pint}
Let $f$ be an interval map and $P$ a finite invariant set. 
Suppose that
$f$ is $P$-monotone. Then, for every $P$-interval $J$, either
$f$ is constant on $J$ or $f(J)$ is a nonempty union of $P$-intervals.
\end{lem}

\begin{rem}
If $P=\{p_0<p_1<\cdots <p_k\}$, a map $f\colon I\to I$ is 
\emph{Markov}\index{Markov map} with respect to the pseudo-partition $[p_0,p_1],
\ldots, [p_{k-1},p_k]$ if $I=[p_0,p_k]$ and, for every $P$-interval $J$,
$f|_J$ is monotone and $f(J)$ is a union of
$P$-intervals. This notion is very close to $P$-monotonicity. Indeed, $f$
is Markov with respect to the $P$-intervals if and only if $f$ is 
$P$-monotone and $f(p_i)\neq f(p_{i+1})$ for all $i\in\Lbrack 0,k-1\Rbrack$
(i.e., $f$ is non constant on every $P$-interval). The main additional
property of Markov maps is that, for every point $x\in I$, there exists an
infinite path $A_0\to A_1\to A_2\to\cdots A_n\to\cdots$ in $G(f|P)$
such that $f^n(x)\in A_n$ for all $n\ge 0$. 
\end{rem}

\begin{lem}\label{lem:MfPn}
Let $f$ be an interval map, $P$ a finite invariant set and $n$ a positive 
integer.  Then $M(f^n|P)\ge M(f_P)^n$.
\end{lem}

\begin{proof}
Let $I_1,\ldots, I_k$ be the $P$-intervals. We write
$M(f_P)^n:=(a_{ij})_{1\le i,j\le k}$. By Proposition~\ref{prop:Mn},
$a_{ij}$ is the number of paths of length $n$ from $I_i$ to $I_j$ in
the graph $G(f_P)$. Each path is a chain of intervals for $f_P$, and also
for $f$, which implies that $I_i$ covers $I_j$ $a_{ij}$ times for $f$.
This exactly means that the $(i,j)$-coefficient of $M(f^n|P)$ is
greater than or equal to $a_{ij}$. Hence $M(f^n|P)\ge M(f_P)^n$.
\end{proof}

The next result was first stated by Coppel in \cite{Cop2}. 
It shows that the graph associated to a $P$-monotone map represents well
the dynamics from the point of view of entropy.

\begin{prop}\label{prop:htop-Markov-map}
Let $f\colon I\to I$ be an interval map and let $P$ be 
a finite invariant set. If $f$ is $P$-monotone, then 
$$
h_{top}(f)=h_{top}(f_P)=\max(0,\log \lambda(M(f_P))).
$$
\end{prop}

\begin{proof}
Let $\CA$ be the family of all $P$-intervals. This is a cover of $I$.
Let $\CC\subset\CA$ be the family of $P$-intervals on which $f$ is constant.
For every $n\ge 1$, we set
\begin{gather*}
\CB_n^+:=\left\{\bigcap_{i=0}^{n-1} f^{-i}(I_i)\mid I_0\to\ I_1\to\cdots
\to I_{n-1}\text{ is a path in }G(f_P),\ I_{n-1}\notin\CC\right\},\\
\CB_n^-:=\left\{\bigcap_{i=0}^k f^{-i}(I_i)\mid k\in\Lbrack 0,n-1\Rbrack,\ I_0\to\cdots
\to I_k\text{ is a path in }G(f_P),\ I_k\in\CC\right\},
\end{gather*}
and $\CB_n:=\CB_n^+\cup\CB_n^-$. We are going to show that $\CB_n$ is a subcover
of $\CA^n$.
It is clear that $\CB_n^+\subset \CA^n$, and the elements of $\CB_n^-$ 
of the form $\bigcap_{i=0}^{n-1}f^{-i}(I_i)$ are in $\CA^n$ too.
Let $J\in \CB_n^-$ with $J=\bigcap_{i=0}^k f^{-i}(I_i)$ and $k<n-1$.
By definition, $f(I_k)$ is reduced to one point $\{x\}$. Since 
$\CA$ is a cover, 
there exist $I_{k+1},\ldots, I_{n-1}\in\CA$ such that
$f^i(x)\in I_{k+i}$ for all $i\in\Lbrack 1,n-1-k\Rbrack$. We have
$J=\bigcap_{i=0}^{n-1}f^{-i}(I_i)$, so $J\in\CA^n$. This proves
that $\CB_n\subset \CA^n$.
We now show that $\CB_n$ is a cover of $I$ by induction on $n$.

\medskip
\noindent$\bullet$ $\CB_1=\CA$ is a cover.

\medskip
\noindent$\bullet$ Let $n\ge 2$.
We have $\CB_{n-1}^-\subset \CB_n$. Let $J\in \CB_{n-1}^+$ with
$J=\bigcap_{i=0}^{n-2} f^{-i}(I_i)$. By definition, $I_{n-2}\notin\CC$,
so $f(I_{n-2})$ is a nonempty union of $P$-intervals
by Lemma~\ref{lem:image-Pint}, say
$f(I_{n-2})=A_1\cup\cdots\cup A_j$ with $A_1,\ldots, A_j\in\CA$.
Then, for every $i\in\Lbrack 1,j\Rbrack$, $I_0\to\cdots I_{n-2}\to A_i$ is a path in
$G(f_P)$, $J_i:=J\cap f^{-(n-1)}(A_i)$ is an element of $\CB_n$ and 
$J=J_1\cup\cdots\cup J_j$. 
Therefore, if $\CB_{n-1}$ is a cover of $I$, then $\CB_n$ is a cover too.
This ends the induction.

\medskip
Since $\CB_n$ is a subcover of $\CA^n$, we have $N_n(\CA,f)\le \#\CB_n$.
By Proposition~\ref{prop:Mn}, $\|M(f_P)^k\|$ is the number 
of paths of length $k$ in $G(f_P)$. Thus
$$
\# \CB_n\le \sum_{k=0}^{n-1}\|M(f_P)^k\|,
$$ 
and hence
$$
\frac 1n \log N_n(\CA,f)\le \frac 1n \log \# \CB_n\le \frac 1n \log 
\left(n\max_{k\in\Lbrack 0,n-1\Rbrack} \|M(f_P)^k\|\right).
$$
If the sequence $\left(\|M(f_P)^k\|\right)_{k\ge 0}$ is bounded, then
$$
\lim_{n\to+\infty}\frac 1n \log 
\left(n\max_{k\in\Lbrack 0,n-1\Rbrack} \|M(f_P)^k\|\right)=0,$$
and thus $h_{top}(\CA,f)=0$.
Otherwise,  there exists an increasing sequence of integers $(n_i)_{i\ge 0}$
such that $\|M(f_P)^{n_i}\|=\max_{k\in\Lbrack 0,n_i\Rbrack} \|M(f_P)^k\|$
for all $i\ge 0$. This implies that
\begin{eqnarray*}
h_{top}(\CA,f)&=&\lim_{n\to+\infty}\frac 1n \log N_n(\CA,f)\\
&\le&\limsup \frac 1n \log 
\left(n\|M(f_P)^n\|\right)
=\limsup_{n\to+\infty} \frac 1n \log (\|M(f_P)^n\|)\\
&\le& \log \lambda(M(f_P))\quad
\text{by Proposition~\ref{prop:spectral-radius}.}
\end{eqnarray*}
Since $\CA$ is a monotone cover, we have $h_{top}(f)=h_{top}(\CA,f)$
by Proposition~\ref{prop:monotone-cover}, so
$h_{top}(f)\le \max(0,\log \lambda(M(f_P)))$.
Proposition~\ref{prop:matrix-htop} and the fact that $h_{top}(f)\ge 0$
imply the converse inequality 
$h_{top}(f)\ge \max(0,\log \lambda(M(f_P)))$. We conclude that 
$h_{top}(f)= \max(0,\log \lambda(M(f_P)))$.
\end{proof}

The converse of Proposition~\ref{prop:htop-Markov-map} does not hold in 
general: there exist
interval maps $f$ with a finite invariant set $P$ such that $h_{top}(f)=
h_{top}(f_P)$ although $f$ is not $P$-monotone and is not 
constant on any subinterval.
See Figure~\ref{fig:Markov-nonPmonotone} for a counter-example.
However, we shall see later that it does hold for transitive maps, 
i.e., a transitive interval map $f$ such that $h_{top}(f)=h_{top}(f_P)$ is
necessarily $P$-monotone (Proposition~\ref{prop:h=hfP}).

\begin{figure}[htb]
\centerline{\includegraphics{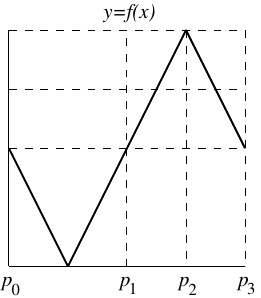}}
\caption{The set $P=\{p_0,p_1,p_2,p_3\}$ is invariant
and it is easy to show that $h_{top}(f)=h_{top}(f_P)=\log 2$ ,
but $f$ is not $P$-monotone.}
\label{fig:Markov-nonPmonotone}
\end{figure}

\medskip
For every finite invariant set $P$, $G(f_P)$ is a subgraph of $G(f|P)$,
and thus Proposition~\ref{prop:matrix-htop} implies that
$h_{top}(f)\ge h_{top}(f_P)$. The next proposition shows that the entropy of $f$
can be approached arbitrarily close in this way.
This result was first stated by Takahashi \cite{Tak2}, but 
it appears that this proof is valid only for piecewise monotone maps, as
noticed by Block and Coven, who gave a complete proof in \cite{BCov2}.
See also the extensive paper of Misiurewicz and Nitecki \cite{MN}.

\begin{prop}\label{prop:htop=suphfP}
Let $f$ be an interval map. Then
\begin{eqnarray*}
h_{top}(f)&=&\sup\{h_{top}(f_P) \mid P \text{ finite invariant set}\}\\
&=&\sup\{h_{top}(f_P) \mid P \text{ periodic orbit}\}.
\end{eqnarray*}
\end{prop}

\begin{proof}
The inequality $h_{top}(f)\ge \sup\{h_{top}(f_P)\mid P \text{ finite invariant set}\}$ follows from  Propositions \ref{prop:matrix-htop} and 
\ref{prop:htop-Markov-map}, 
and the inequality
$$\sup\{h_{top}(f_P) \mid P \text{ finite invariant set}\}
\ge\sup\{h_{top}(f_P)\mid P \text{ periodic orbit}\}$$
is trivial.
We are going to show that $h_{top}(f)\le \sup\{h_{top}(f_P) \mid P \text{ periodic orbit}\}$ when $h_{top}(f)>0$ 
(if $h_{top}(f)=0$, there is nothing to prove). Let 
$\lambda,\lambda'$ be such that $0<\lambda<\lambda'<h_{top}(f)$.
By Misiurewicz's Theorem \ref{theo:Misiurewicz},
there exist an arbitrarily large integer $N$ and a strict $p$-horseshoe for $f^N$ such
that $\frac{\log p}{N} \ge \lambda'$. We denote by $I_1< I_2<\cdots< I_p$  the 
intervals composing this horseshoe and we set $g:=f^N$. 
By applying 
Lemma~\ref{lem:chain-of-intervals}(ii) to the chain of intervals
\begin{eqnarray*}
&&I_2\to I_1\to I_3\to I_1\to\cdots \to I_{p-1}\to I_1\to\\
&&I_2\to I_p\to I_3\to I_p\to\cdots \to I_{p-1}\to I_p\to I_2,
\end{eqnarray*}
we can build 
a periodic point $x$ of period $4p-8$ for $g$ such that 
\begin{itemize}
\item for all $i=0,\ldots, p-3$,
$g^{2i}(x)$ belongs successively to $I_2, I_3,
\ldots, I_{p-1}$ and $g^{2i+1}(x)$ belongs to $I_1$; we set $y_i:=g^{2i}(x)$;
\item for all $i=p-2,\ldots, 2p-5$, $g^{2i}(x)$ belongs successively 
to $I_2, I_3, \ldots, I_{p-1}$ and $g^{2i+1}(x)$ belongs to $I_p$;
we set $z_i:=g^{2i}(x)$.
\end{itemize}

Let $Q:=\CO_g(x)$. For every $i\in\Lbrack 2,p-1\Rbrack$, $I_i$ contains only two
points of $Q$, namely $y_i,z_i$, and thus $\langle y_i, z_i\rangle$ is
a $Q$-interval. Moreover, $g(y_i)\in I_1$ and $g(z_i)\in I_p$, which implies
that $\langle g(y_i),g(z_i)\rangle$ contains $I_2\cup\cdots\cup I_{p-1}$
by connectedness. 
Therefore, the intervals $(\langle y_i,z_i\rangle)_{2\le i\le p-2}$ 
form a $(p-2)$-horseshoe for the map $g_Q$, which implies that
$h_{top}(g_Q)\ge \log (p-2)$ by Proposition~\ref{prop:horseshoe-htop}.
We set $M_Q:=M(g_Q)$.

Now we come back to the map $f$. Let $P:=\CO_f(x)$. Since $x$ is periodic 
for $g=f^N$, it is also periodic for $f$. Moreover,
$Q$ is a periodic orbit for the map $(f_P)^N$ because $f_P$ and $f$ coincide
on the set $P\supset Q$. Therefore, 
$$
h_{top}((f_P)^N)\ge \max(0,\log \lambda( M_Q))=h_{top} (g_Q)$$
by Propositions \ref{prop:matrix-htop} and
\ref{prop:htop-Markov-map}, and thus
$$
h_{top}(f_P)=\frac 1N h_{top}((f_P)^N)\ge \frac1N \log (p-2).
$$
Since $\frac1N \log p\ge \lambda'$, we have $h_{top}(f_P)\ge \lambda'-\frac 1N
\log \frac{p}{p-2}$. If $N$ is large enough, then 
$\frac1N\log 2<\lambda'-\lambda$ and $p$ can be arbitrarily large. In
particular, $\frac{p}{p-2}\le 2$ if $p\ge 4$, so
$\frac1N\log\frac{p}{p-2}\le\frac1N\log 2<\lambda'-\lambda$.
We thus have $h_{top}(f_P)\ge \lambda$. We deduce the required
result by taking $\lambda$ tending to $h_{top}(f)$.
\end{proof}

%*********
\subsection*{Remarks on graph maps}

The notion of graph associated to a family of intervals is
meaningful for graph maps provided Definition~\ref{def:coveringG} is used for 
covering, and Proposition~\ref{prop:matrix-htop} holds for graph maps with no change.

For graph maps, one can define $P$-monotone maps when the finite invariant set 
$P$ contains all the branching points of the graph 
(which requires that the orbit
of every branching point is finite), and there is
no difficulty to extend Proposition~\ref{prop:htop-Markov-map} to
$P$-monotone graph maps in this case (see, e.g., a remark in \cite{AMM}).
However, the 
\emph{connect-the-dots map} associated to a finite invariant set $P$ 
is not well defined in general (it is well defined when the space is a tree 
and $P$ contains the branching points).

\begin{defi}
Let $f\colon G\to G$ be a graph map and let $P$ be a finite
invariant set containing all the branching points and all the
endpoints of $G$. 
A \emph{$P$-basic interval} is any connected component
of $G\setminus P$.
The map $f$ is called \emph{$P$-monotone}\index{P-monotone graph map@$P$-monotone graph map}\index{monotone graph map ($P$- )}
if $f$ is monotone in restriction to every $P$-basic interval.
If $f$ is $P$-monotone, let
$M(f|P)$ denote the adjacency matrix of the graph associated to the family
of all $P$-basic intervals.
\end{defi}

\begin{prop}
Let $f\colon G\to G$ be a graph map and let $P$ be a finite
invariant set containing all the branching points and all the
endpoints of $G$. If $f$ is $P$-monotone, then
$h_{top}(f)=\max (0,\log (\lambda(M(f|P))))$.
\end{prop}

In \cite{AJM5}, Alsedà, Juher and Mumbrú showed that 
an equality similar to Proposition~\ref{prop:htop=suphfP} holds for graph maps;
an inequality was previously proved by  Alsedà, Mañosas and Mumbrú
\cite{AMM}. Since connect-the-dots
maps cannot be defined, it is necessary to introduce an equivalence between
actions on a pointed graph, in order to be able to tell when a $P$-monotone map
is a ``good'' candidate to replace the connect-the-dots map. 
The next definition follows \cite{AMM}.

\begin{defi}
Let $G$ be a topological graph and let $B(G)$ denote the set of all
branching points of $G$. Let $A$ be a finite set of $G$. 
Let $G_A$ denote the graph $G$ deprived of the connected components of
$G\setminus (A\cup B(G))$ containing an endpoint of $G$. Let $r_A\colon
G\to G_A$ denote the retraction from $G$ to $G_A$ (that is,
$r_A$ is the identity on $G_A$ and, if $C$ is a connected component
of $G\setminus G_A$ and $x\in C$, then $r_A(x)$ is the unique point
in $\overline{C}\cap (A\cup B(G))$).
Let $f\colon G\to G$ and $g\colon G\to G$ be two graph maps
and assume that $A$ is both $f$-invariant and
$g$-invariant. Set $\tilde f:=r_A\circ f|_{G_A}$ and
$\tilde g:=r_A\circ g|_{G_A}$.
Then one writes $(G,A,f)\sim(G,A,g)$ if there
exists a homeomorphism $\vfi\colon G_A\to G_A$ with $\vfi(A)=A$
such that $\tilde f$ and $\vfi^{-1}\circ \tilde g\circ \vfi$ are homotopic
relative to $A$.
\end{defi}

\begin{theo}
Let $f\colon G\to G$ be a graph map, and let $B(G)$ and $E(G)$ denote 
respectively the set of
branching points and endpoints of $G$.
For every finite $f$-invariant set $A$, there exists a map 
$g_A\colon G\to G$ such that
\begin{itemize}
\item
$P:=A\cup B(G)\cup E(G)$ is $g_A$-invariant,
\item $g_A$ is $P$-monotone,
\item 
$(G,A,g_A)\sim (G,A,f)$,
\item 
$h_{top}(g_A)\le h_{top}(f)$.
\end{itemize}
Furthermore, 
$h_{top}(f)=\sup\{h_{top}(g_A)\mid A\text{ is a periodic orbit of }f\}.$
\end{theo}

%************************************************************************
\section{Entropy and periodic points}
\subsection{Equivalent condition for positive entropy}

We are going to show that 
an interval map has positive topological entropy if and only if 
it has a periodic point whose period is not a power of $2$.
This relation between entropy and periods is one of the most striking 
results in interval dynamics.
This result can be expressed in term of types for Sharkovsky's order: 
an interval map has positive entropy if and only if it is of type $n$ with
$n\lhd 2^\infty$. 
This explains why Coppel calls \emph{chaotic} an interval map having a 
periodic point whose period is not a power of $2$ \cite{Cop2, BCop},
the type $2^\infty$ being the ``frontier'' between chaos and non chaos.

This result was proved in several steps. First, Sharkovsky showed in 1965
that an interval map $f$ is of type $n\lhd 2^\infty$ if and only if
$f^k$ has a horseshoe for some $k$ \cite{Sha2};
see \cite{Sha6} for a statement in English.
The same result was re-proved by Block \cite{Bloc3}. Then
Bowen and Franks stated that the presence of a periodic point whose
period is not a period of $2$ implies positive entropy \cite{BF}. This result
relies on the observation that horseshoes imply positive entropy
(Proposition~\ref{prop:horseshoe-htop}).
Finally, the last step is due to Misiurewicz and Szlenk for piecewise 
monotone maps  \cite{MS2} and  Misiurewicz for all interval maps
\cite{Mis,Mis2}. This is a corollary of Misiurewicz's Theorem 
(Theorem~\ref{theo:Misiurewicz}), proved in the same papers, 
stating that an interval map with
positive entropy has a horseshoe for some iterate of the map.

\begin{theo}\label{theo:htop-power-of-2}
For an interval map $f$, the following assertions
are equivalent:
\begin{enumerate}
\item the topological entropy of $f$ is positive,
\item $f$ has a periodic point whose period is not a power of $2$,
\item there exists an integer $n\ge 1$ such that $f^n$ has a strict
horseshoe.
\end{enumerate}
\end{theo}

\begin{proof}
If $h_{top}(f)>0$, then, according to Misiurewicz's 
Theorem~\ref{theo:Misiurewicz}, there exists a positive 
integer $n$ such that $f^n$ has a horseshoe. Therefore $f^n$
has periodic points of all periods  by 
Proposition~\ref{prop:turbulent-all-periods}, and thus $f$ has a periodic
point whose period is not a power of $2$. This shows (i)$\Rightarrow$(ii).

If $f$ has a periodic point of period $2^dq$, where $q$ is an odd
integer greater than $1$ and $d\ge 0$, then $f^{2^d}$ has a periodic point 
of period $q$ and thus, by Proposition~\ref{prop:odd-period-turbulent},
$f^{2^{d+1}}$ has a strict horseshoe.
That is, (ii)$\Rightarrow$(iii). 

If $f^n$ has a horseshoe, then, according to 
Proposition~\ref{prop:horseshoe-htop},
$h_{top}(f)=\frac{1}{n}h_{top}(f^n)\ge \frac{\log 2}{n}>0$. Hence
(iii)$\Rightarrow$(i).
\end{proof}

%*****************************
\subsection{Lower bound for the entropy depending on Sharkovsky's type}

The relation between the entropy of an interval map and its type
is much more accurate than the one stated in
Theorem~\ref{theo:htop-power-of-2}, and one can give a lower bound 
for the entropy depending on the periods of the periodic points. 
First, Bowen and Franks proved that,
if $f$ has a periodic point of period $n=2^dq$, where $q>1$ is odd, then 
$h_{top}(f)>\frac{1}{n}\log 2$ \cite{BF}. Then \v{S}tefan improved this
result and showed that,
under the same assumption, $h_{top}(f)>\frac{\log\sqrt{2}}{2^d}$ \cite{Ste}.
Finally, Block, Guckenheimer, Misiurewicz and Young 
gave an optimal bound by proving that, under the same assumption,
$h_{top}(f)\ge\frac{\log\lambda_q}{2^d}$, where $\lambda_q$ is the
maximal real root of $X^q-2X^{q-2}-1$~\cite{BGMY}.
Actually, the value $\frac{\log\lambda_q}{2^d}$ already appears in 
\v{S}tefan's proof, where $\lambda_q$ is proved to be greater than $\sqrt{2}$.
Moreover,
there exist maps of type $2^d q$ whose entropy is equal to
$\frac{\log\lambda_q}{2^d}$. Examples of such maps were given without details 
in \cite{BGMY}.

We start with a lemma, which is the key point of the proof.

\begin{lem}\label{lem:Mq}
Let $f$ be an interval map having a periodic point
of odd period greater than $1$. Let $p$ be the minimal odd period 
greater than $1$, let
$G_p$ be the graph of a periodic orbit of period $p$ and $M_p$ its
adjacency matrix.
Then $M_p$ is a primitive matrix and $\lambda(M_p)$ is equal to 
the unique positive root $\lambda_p$ of $X^p-2X^{p-2}-1$.
Moreover, for all odd $p>1$,
$$\sqrt{2}<\lambda_{p+2}<\lambda_p<\sqrt{2}+\frac{1}{(\sqrt{2})^{p+1}}.$$
\end{lem}

\begin{proof}
According to Lemma~\ref{lem:graph-n-minimal}, the graph $G_p$ is of the form:

\medskip
\centerline{\includegraphics{ruette-fig1}}

\medskip
This gives the following adjacency matrix:
\begin{equation}\label{eq:Mq}
M_p=\left(\begin{array}{cccccc}
1       &   &       &                  && 1\\
1       & 0 &       &\text{\huge 0}&& 0\\
        & 1 &\ddots &                  && 1\\
        &   &\ddots &\ddots            && \vdots \\ 
        & \text{\huge 0}  &       &\ddots &\ddots     & 1\\
        &   &       &               & 1 & 0\\
\end{array}\right).
\end{equation}
More precisely, if $M_p=\left(m_{ij}\right)_{1\le i,j\le p-1}$, then 
\begin{itemize}
\item 
on the diagonal: $m_{11}=1$ and $\forall i\in\Lbrack 2,p-1\Rbrack, 
m_{ii}=0$,
\item
below the diagonal: $\forall i\in\Lbrack 1,p-2\Rbrack,\ m_{i+1i}=1$,
\item 
last column: $\forall i\in\Lbrack 1,p-1\Rbrack,\ m_{i p-1}=\left\{\begin{array}{ll}
1&\text{ if }i\text{ is odd,}\\ 
0&\text{ if }i\text{ is even,}\end{array}\right.$
\item
for all other indices, $m_{ij}=0$.
\end{itemize}
We write $(M_p)^n=(m_{ij}^n)_{1\le i,j\le p-1}$ for all $n\ge 1$.
Then $m_{ij}^n$ is
the number of paths of length $n$ from $J_i$ to $J_j$ in $G_p$
(Proposition~\ref{prop:Mn}). For all $i,j\in\Lbrack 1,p-1\Rbrack$, the path
$$
J_i\to J_{i+1}\to\cdots\to \underbrace{J_1\to J_1\to \cdots 
\to J_1}_{i+j-2\rm\ arrows}\to J_2\cdots \to J_j
$$
is a path from $I_i$ to $I_j$ of length $2p-2$.
Thus $(M_p)^{2p-2}>0$ and $M_p$ is primitive.

In order to find the maximal eigenvalue of $M_p$,
we compute the characteristic polynomial $\chi_p(X):=\det (M_p-X\Id)$
(see Proposition~\ref{prop:eigenvalue-rootchi}).
We develop it with respect to the first row (the coefficients left
blank are equal to zero):
$$
\chi_p(X)=(1-X)\underbrace{\left|\begin{array}{ccccc}
-X & &  && 0\\
 1 &\ddots& &   & 1 \\
 &\ddots&\ddots &   & \vdots \\
   &&\ddots & \ddots & 1\\
&       && 1      & -X\\
\end{array}\right|}_{\textstyle := Q_{p-2}(X)}
-\left|\begin{array}{cccc}
1  & -X & &\\
0  & 1 &\ddots      & \\
   && \ddots  & -X \\
& &   & 1 \\
\end{array}\right|.
$$
We get $\chi_p(X)=(1-X)Q_{p-2}(X)-1$. It remains to compute $Q_k(X)$ for all
odd $k$. If we develop twice the determinant $Q_k(X)$ with respect to the first row, we
get $Q_k(X)=X^2 Q_{k-2}(X)+X$. We have $Q_1=-X$ and an easy induction gives
$$
\text{for all odd }k\ge 3,\ Q_k(X)=-X^k+\sum_{i=0}^{\frac{k-3}2} X^{2i+1}.
$$
Therefore
$$
\text{for all odd }p\ge 3,\ \chi_p(X)=X^{p-1}-X^{p-2}-\sum_{i=0}^{p-3}(-X)^i.
$$
We set $P_p(X):=(X+1)\chi_p(X)$. A straightforward computation gives 
$$
P_p(X)=X^p-2X^{p-2}-1.$$
We do a short study of the polynomial function $x\mapsto P_p(x)$ on 
$\IR^+$.  Its differential is
$$
P_p'(x)=px^{p-1}-2(p-2)x^{p-3}=x^{p-3}(px^2-2(p-2)).
$$ 
Thus, for $x\in (0,+\infty)$, 
$P_p'(x)>0\Leftrightarrow x>x_p:=\sqrt{\frac{2(p-2)}{p}}$.
This implies that $P_p$ is decreasing on $[0,x_p]$
and increasing on $[x_p,+\infty)$, and also that $x_p<\sqrt{2}$. Moreover,
$P_p(0)=-1$, and $\lim_{x\to +\infty} P_p(x)=+\infty$. We deduce that
there exists a unique $\lambda_p>0$ such that $P_p(\lambda_p)=0$.
Since $\lambda_p$ is the maximal real 
root of $P_p$, it is also the maximal real root of $\chi_p=\frac{P_p}{X+1}$.
By Corollary~\ref{cor:Perron-Frobenius}, this implies that
$\lambda(M_p)=\lambda_p$.

Now we are going to bound $\lambda_p$.
Since $P_p(x)=x^{p-2}(x^2-2)-1$, we have $P_p(\sqrt{2})=-1$, and thus
$\lambda_p>\sqrt{2}$ (recall that $P_p$ is increasing on
$[x_p,+\infty)\supset [\sqrt{2},+\infty)$). Moreover, since
$$
\lambda_p^{p-2}(\lambda_p^2-2)=1,
$$ 
we have
$\lambda_p^p(\lambda_p^2-2)=\lambda_p^2>2$,
which implies $P_{p+2}(\lambda_p)>0$. Therefore,
$\lambda_{p+2}<\lambda_p$ for all odd $p\ge 3$.
We set $y_p:=\sqrt{2}+\frac{1}{(\sqrt{2})^{p+1}}$. We have
$$
y_p^2=2+\frac{1}{2^{p+1}}+\frac{1}{(\sqrt{2})^{p-2}}$$
and $y_p>\sqrt{2}$. Therefore
\begin{eqnarray*}
P_p(y_p)=y_p^{p-2}(y_p^2-2)-1&=&y_p^{p-2}\left(\frac{1}{2^{p+1}}+\frac{1}{(\sqrt{2})^{p-2}}\right)-1\\
&> & (\sqrt{2})^{p-2}\frac{1}{(\sqrt{2})^{p-2}}-1\\
&> 0.
\end{eqnarray*}
We deduce that $\lambda_p<y_p=\sqrt{2}+\frac{1}{(\sqrt{2})^{p+1}}$. This concludes the proof.
\end{proof}

\begin{theo}\label{theo:htop-type}
If an interval map $f$ has a periodic point of
period $2^d q$ with $d\ge 0$, $q>1$, $q$ odd, then $h_{top}(f)\ge\frac{\log 
\lambda_q}{2^d}$, where $\lambda_q$ is the unique positive root of
$X^q-2X^{q-2}-1$. 

Moreover, for all integers $d\ge 0$ and all $q>1$ with $q$ odd, there exists an 
interval map with a periodic point of period $2^dq$ and whose 
topological entropy is equal to $\frac{\log \lambda_q}{2^d}$.
\end{theo}

\begin{proof}
First we suppose that $d=0$ and that $q$ is the minimal odd period greater than
1. Let $G_q$ be the graph of a periodic orbit of period $q$. 
According to Lemma~\ref{lem:Mq}, the spectral radius of the adjacency
matrix of $G_q$ is equal to $\lambda_q$. Thus
$h_{top}(f)\ge \log \lambda_q$ by  Proposition~\ref{prop:matrix-htop}.

Now we suppose that $x$ is a periodic point of period $2^d q$ with
$q>1$, $q$ odd. The point $x$ is periodic of period $q$ for $f^{2^d}$.
Let $p$ be the minimal odd period greater than 1 for the map 
$f^{2^d}$. What precedes
shows that $h_{top}(f^{2^d})\ge \log\lambda_p$. Since $p\le q$, 
Lemma~\ref{lem:Mq} implies that 
$\lambda_p\ge \lambda_q$, and thus
$$
h_{top}(f)=\frac{1}{2^d}h_{top}(f^{2^d})\ge \frac{\log\lambda_q}{2^d}.
$$
For the sharpness of the bound, see Examples \ref{ex:htop-mixing-log-lambda}
and \ref{ex:htop-2-m-log-lambda} below.
\end{proof}

\begin{ex}\label{ex:htop-mixing-log-lambda}
Let $n$ be a positive integer and $p:=2n+1$ (i.e., $p$ is an odd integer
greater than $1$). We consider the map 
$f_p\colon [0,2n]\to [0,2n]$ built in Example~\ref{ex:odd-type}. We already
proved that it
is topologically mixing and that its type for Sharkovsky's order is $p$.
We recall that the map $f_p$ (represented in Figure~\ref{fig:htop-lambdaq}) 
is linear between the points $0,n-1,n,2n-1, 2n$, and 
\begin{itemize}
\item $\forall k\in\Lbrack 1,n\Rbrack$, $f_p(n-k)=n+k$,
\item $\forall k\in\Lbrack 0,n-1\Rbrack$, $f_p(n+k)=n-k-1$.
\end{itemize}

\begin{figure}[htb]
\centerline{\includegraphics{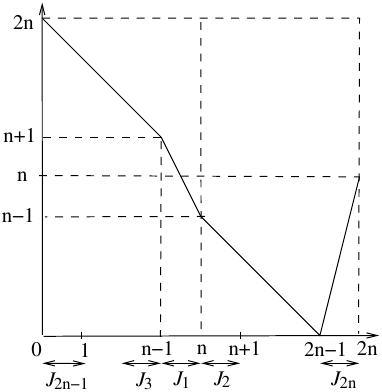}}
\caption{This interval map is of type $p=2n+1$, it is topologically
mixing and its topological entropy 
is equal to $\log\lambda_p$.} 
\label{fig:htop-lambdaq}
\end{figure}

We set $P:=\{0,1,2,\ldots, ,2n\}$ and, for all $k\in\Lbrack 1,n\Rbrack$,
$$
J_{2k-1}:=[n-k,n-k+1]\quad\text{and}\quad J_{2k}:=[n+k-1,n+k].
$$
Then $P$ is a periodic orbit of period 
$p$ and the graph associated to $P$ is: 
\begin{equation}\label{fig:graph-type-q}
\lower1em\hbox{\includegraphics{ruette-fig1}}
\end{equation}
Moreover, $f_p$ is $P$-linear and the matrix $M_p=M(f_p|P)$ is exactly the 
one given by \eqref{eq:Mq} in the proof of Lemma~\ref{lem:Mq}.
Hence $\lambda(M_p)=\lambda_p,$ where $\lambda_p$ is defined in
Theorem~\ref{theo:htop-type}. By Proposition~\ref{prop:htop-Markov-map},
$h_{top}(f_p)=\log\lambda_p$. 
This proves that the bound of Theorem~\ref{theo:htop-type} is sharp for $d=0$. 

Finally, we remark that, for $p=3$, we get a map of type $3$ with no horseshoe
because $h_{top}(f_3)<\log 2$. This shows that the converse of 
Proposition~\ref{prop:turbulent-all-periods} does not hold, as said
in Chapter~\ref{chap:periodic-points}.
\end{ex}

\begin{ex}\label{ex:htop-2-m-log-lambda}\index{square root of a map}
Our goal is to show that, for all integers $d\ge 0$ and all $p>1$ with $p$ odd,
there exists an interval map of type $2^d p$ such that its topological
entropy is equal to $\frac{\log\lambda_p}{2^d}$, where $\lambda_p$ is
defined in Theorem~\ref{theo:htop-type}. This will prove that the bound of
Theorem~\ref{theo:htop-type} is optimal.

In Example~\ref{ex:all-types}, we defined the \emph{square root} of an 
interval map and showed that the square root of a map of type $n$
is of type $2n$. We recall this construction.
If $f\colon [0,b]\to [0,b]$ is an interval map, the square root
of $f$ is the continuous map
$g\colon [0,3b]\to [0,3b]$ defined by
\begin{itemize}
\item $\forall x\in [0,b]$, $g(x):=f(x)+2b$,
\item $\forall x\in [2b,3b]$, $g(x):=x-2b$,
\item $g$ is linear on $[b,2b]$.
\end{itemize}
The graphs of $g$ and $g^2$ are represented in Figure~\ref{fig:type-2n-g-g2}.

\begin{figure}[htb]
\centerline{\includegraphics{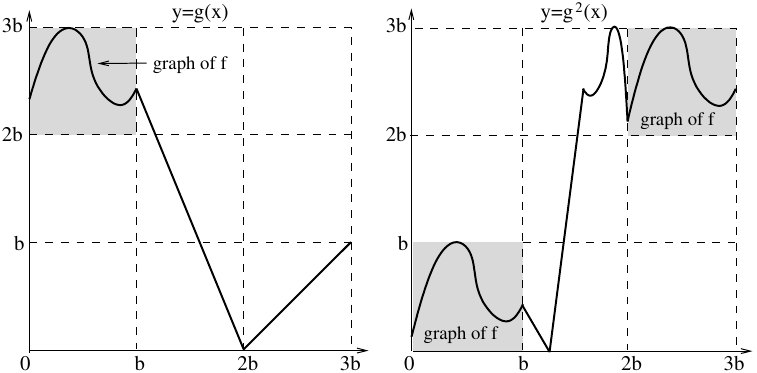}}
\caption{The left side represents the map $g$, which is the square root of $f$; the topological entropy of 
$g$ is $\frac{h_{top}(f)}2$. The right side represents the map $g^2$.} 
\label{fig:type-2n-g-g2}
\end{figure}

Suppose that $f$ is $P$-monotone with $P=\{x_0<x_1<\cdots<x_p\}$, $x_0=0$ 
and $x_p=b$. We set
$$
Q:=\{x_0,\ldots,x_p,x_0+2p,x_1+2p,\ldots,x_p+2p\}.
$$
By definition of $g$, it is obvious that $g$ is $Q$-monotone.
The matrix $A:=M(f|P)$ is of size $p\times p$. Let 
$B:=M(g|Q)$, with the convention that the $Q$-interval $[b,2b]$ corresponds 
to the last column and row. The matrix $B$ is of size $(2p+1)\times(2p+1)$
and, looking at Figure~\ref{fig:type-2n-g-g2}, it is clear that 
$B$ is of the form
$$
B=\left(\begin{array}{ccc}
0 _{p\times p}& A & 0_p\\
{\rm Id}_p & 0_{p\times p} & 0_p\\
* & ^t1_p & 1
\end{array}\right)
$$
(where $0_{p\times p}$ denotes the $p\times p$ null matrix and
$x_p$ denotes the $1\times p$ matrix with all coefficients equal to $x$) and thus
$$
B^2=\left(\begin{array}{ccc}
A & 0_{p\times p} & 0_p\\
0_{p\times p}& A& 0_p\\
* & * & 1
\end{array}\right).
$$
We deduce that $\lambda(B^2)=\lambda(A)$ provided $\lambda(A)\ge 1$. 
According to  Proposition~\ref{prop:htop-Markov-map},
$h_{top}(f)=\max(0,\log\lambda(A))$ and $h_{top}(g)=\max(0,\log\lambda(B))$.
Since $\lambda(B^2)=\lambda(B)^2$, we get
$h_{top}(g)=\frac{1}{2}h_{top}(f)$.

\medskip
We fix an odd integer $p>1$. Starting with the map $f_p$ of type $p$ and
topological entropy $\log\lambda_p$ defined in Example~\ref{ex:htop-mixing-log-lambda} and
applying inductively the square root construction, we can build an
interval map of type $2^d p$ and topological entropy $\frac{\log\lambda_p}{2^d}$
for any integer $d\ge 0$. This completes the construction.
\end{ex}

%**************************
\subsection{Number of periodic points}

We have seen that the knowledge of the periods of the periodic points
gives a lower bound on the entropy. Conversely, the entropy gives
some information on the number of periodic points. The next result, 
due to Misiurewicz \cite{Mis2}, is a straightforward consequence of
Misiurewicz's Theorem~\ref{theo:Misiurewicz}. 
Recall that $P_n(f)$ is the set of points $x$ such that $f^n(x)=x$.

\begin{prop}\label{prop:htop-cardPn}
If $f$ is an interval map of positive topological entropy, then 
$$
\limsup_{n\to+\infty}\frac{1}{n}\log\# P_n(f)\ge h_{top}(f).
$$
\end{prop}

\begin{proof}
Let $0<\lambda<h_{top}(f)$. According to Theorem~\ref{theo:Misiurewicz},
for all integers $N$, there exist integers $n\ge N$ and $p\ge 2$
such that $f^n$ has a strict $p$-horseshoe $(J_1,\ldots, J_p)$
and  $\frac{1}{n}\log p\ge \lambda$.
In particular, $f^n(J_i)\supset J_i$ for all $i\in\Lbrack 1,p\Rbrack$, and thus 
there exists $x\in J_i$ such that $f^n(x)=x$
by Lemma~\ref{lem:fixed-point}. Since the intervals $J_1,\ldots, J_p$
are pairwise disjoint, this implies that $\# P_n(f)\ge p$, so
$$
\limsup_{n\to+\infty}\frac{1}{n}\log\# P_n(f)\ge \lambda.
$$
Since $\lambda$ is arbitrarily close to $h_{top}(f)$, this gives the required
result.
\end{proof}

The next proposition follows
a theorem of \v{S}tefan \cite{Ste}, which strengthens a previous result
of Bowen and Franks \cite{BF}.

\begin{prop}\label{prop:type-nb-periodic-points}
Let $f\colon I\to I$ be an interval map. If $f$ has a periodic point of
period $2^d q$ with $d\ge 0$ and $q$ an odd integer greater than $1$, then
$$
\liminf_{n\to+\infty}\frac{1}{n}\log
\#\{x\in I\mid x\text{ periodic point of period } 2^d n\} 
\ge \log\lambda_q,
$$
where $\lambda_q$ is the unique positive root of $X^q-2X^{q-2}-1$.
\end{prop}

\begin{proof}
We first assume that $d=0$. Let $p$ be the minimal odd period greater than
$1$. We fix $P$ as a periodic orbit of period $p$ and we denote $G$ by the graph
associated to $P$. For all $n\ge 1$,
let $N_n$ be the number of primitive cycles of length $n$ in $G$.
According to Lemma~\ref{lem:cycle-periodic-point},
for every primitive cycle $I_0\to I_1\to\cdots I_n=I_0$ in $G$,
there exists a periodic point $y$ of period $n$ such that
$f^i(y)\in I_i$ for all $i\in\Lbrack 0,n-1\Rbrack$. The periodic points 
$y,y'$ corresponding
to two different primitive cycles are different, except maybe if
$y,y'$ are endpoints of one of the $P$-intervals, which implies that they
are of period $p$. Therefore,
\begin{equation}\label{eq:number-periodic-points-n}
\forall n\neq p,\ \#\{x\in I\mid x\text{ periodic point of period } n\}\ge N_n.
\end{equation}

Let $M$ be the adjacency matrix of $G$. We write 
$M^n=(m_{ij}^n)_{1\le i,j\le p-1}$ for every $n\ge 1$. 
By Proposition~\ref{prop:Mn}, the number of cycles of length $n$ in 
$G$ is equal to $\sum_{i=1}^{p-1}m_{ii}^n={\rm Tr} (M^n)$.
By Lemma~\ref{lem:Mq}, $M$ is primitive and its maximal
eigenvalue is $\lambda_p$. Let $(\lambda_p,\mu_2,\ldots,\mu_{p-1})$ be the
set of eigenvalues (with possible repetitions corresponding to the size of the
generalized eigenspaces) of $M$.
According to the Perron-Frobenius Theorem~\ref{theo:Perron-Frobenius}, 
$|\mu_i|<\lambda_p$ for all $i\in\Lbrack 2,p-1\Rbrack$.
By triangularization of $M$, the matrix $M^n$ is equivalent to
$$
\left(\begin{array}{ccccc}
\lambda_p^n & *       & * &\cdots & *   \\
0         & \mu_2^n & * &\cdots & *  \\
0         & 0 & \mu_3^n &\cdots & * \\
\vdots & \vdots & \vdots  & \ddots & \vdots  \\
0         & 0 & 0     & \cdots & \mu_{p-1}^n\\
\end{array}\right).
$$
We deduce that
\begin{equation}\label{eq:Tr-Mn}
{\rm Tr} (M^n)=\lambda_p^n+\mu_2^n+\cdots+\mu_{p-1}^n
=\lambda_p^n\left(1+\sum_{i=2}^{p-1}\left(\frac{\mu_i}{\lambda_p}\right)^n
\right).
\end{equation}
We fix $n\ne p$. If a cycle of length $n$ is not primitive, then it is a 
multiple of a primitive cycle of length $k$ for some $k$ dividing $n$ 
with $k<n$. Therefore
$$
{\rm Tr}(M^n)=N_n+\sum_{\doubleindice{k|n}{k<n}} N_k
$$
(where $k|n$ means that $k$ divides $n$).
\label{notation:divides}
Moreover, $N_k\le {\rm Tr}(M^k)$ because ${\rm Tr}(M^k)$ is the number 
of cycles of length $k$. Thus 
\begin{equation}\label{eq:Nn}
N_n\ge {\rm Tr}(M^n)-\sum_{\doubleindice{k|n}{k<n}}{\rm Tr}(M^k).
\end{equation}
Let $k$ be an integer dividing $n$ such that $k<n$. 
Necessarily, $k\le n/2$ and thus, by \eqref{eq:Tr-Mn},
${\rm Tr}(M^k)\le (p-1)\lambda_p^k\le (p-1)\lambda_p^{n/2}$.
Combining this with \eqref{eq:Nn} and \eqref{eq:Tr-Mn},
we get
\begin{eqnarray*}
N_n&\ge& \lambda_p^n\left(1+\sum_{i=2}^{p-1}
\left(\frac{\mu_i}{\lambda_p}\right)^n\right)
-(p-1)\frac{n}{2}(\lambda_p)^{n/2}\\
&\ge&\lambda_p^n\left(1+\sum_{i=2}^{p-1}
\left(\frac{\mu_i}{\lambda_p}\right)^n-(p-1)\frac{n}{2}(\lambda_p)^{-n/2}
\right).
\end{eqnarray*}
Since $\lambda_p>1$ and $|\mu_i|<\lambda_p$ for all $i\in\Lbrack 2,p-1\Rbrack$,
we have
$$
\lim_{n\to+\infty}\left(\frac{\mu_i}{\lambda_p}\right)^n=0\quad\text{and}
\quad\lim_{n\to+\infty}\frac{n}{2}(\lambda_p)^{-n/2}=0,
$$
so
$$
\liminf_{n\to+\infty}\frac{1}{n}\log N_n\ge \log\lambda_p.$$
Thus, by \eqref{eq:number-periodic-points-n}
\begin{equation}\label{eq:Nn-lambdap}
\liminf_{n\to+\infty}\frac{1}{n}\log 
\#\{x\in I\mid x\text{ periodic point of period } n\}\ge \log\lambda_p.
\end{equation}

We now suppose that $f$ has a periodic point of period $2^d q$ with $d\ge 0$
and $q>1$, $q$ odd. Then the map $g:=f^{2^d}$ has a periodic point of
period $q$ (Lemma~\ref{lem:period-f-fn}(i)).
Let $p$ be the minimal odd period greater than $1$ for $g$. Since a periodic
point of period $2^d n$ for $f$ is a periodic point of period $n$ for $g$, 
\eqref{eq:Nn-lambdap} implies that
$$
\liminf_{n\to+\infty}\frac{1}{n}\log 
\#\{x\in I\mid x\text{ periodic point of period } 2^d n \text{ for } f\}
\ge \log\lambda_p.
$$
Moreover $\lambda_p\ge \lambda_q$ by Lemma~\ref{lem:Mq}.
This concludes the proof.
\end{proof}

%*******************
\subsection*{Remarks on graph maps}

Llibre and Misiurewicz showed that Proposition~\ref{prop:htop-cardPn}
is also valid for graph maps \cite{LM3}. The technique is similar.

For graph maps, there exist conditions equivalent to positive entropy
in terms of sets of periods,  but they cannot be expressed in such a 
simple dichotomy as the equivalence (i)$\Leftrightarrow$(ii) in 
Theorem~\ref{theo:htop-power-of-2}.

Optimal lower bounds on entropy are known for circle maps, in the same
vein as Theorem~\ref{theo:htop-type}.
The results for circle maps of degree different from $1$ are
mainly due to Block, Guckenheimer, Misiurewicz and Young \cite{BGMY}.
When the degree is $0$ or $-1$, one essentially has the same results
as for interval maps.
Several papers deal with entropy of circle maps of degree 1. In particular,
Ito gave an optimal lower bound on entropy when 
there exist two periods $p,q>1$ such that $\gcd(p,q)=1$ \cite{Ito2}.
The lower bound stated below in Theorem~\ref{theo:h-degree1}, which is
the most precise one, depends on the rotation interval; it is due to
Alsedà, Llibre, Mañosas and Misiurewicz \cite{ALMM}. 
The reader is advised to refer
to \cite[Section 4.7]{ALM} for an extensive exposition on circle maps.
Recall that the possible sets of periods of circle maps were given in
Theorems \ref{theo:period-degreenot1} and \ref{theo:periods-degree1}.

\begin{prop}\label{prop:h-degreed}
Let $f\colon \IS\to \IS$ be a circle map of degree $d$ with $|d|\ge 2$.
Then $f$ admits a $|d|$-horseshoe and $h_{top}(f)\ge \log |d|$.
\end{prop}
 
\begin{prop}
Let $f\colon \IS\to \IS$ be a circle map of degree $0$ or $-1$. The following
assertions are equivalent:
\begin{itemize}
\item $h_{top}(f)>0$.
\item $f$ has a periodic point whose period is not a power of $2$. 
\end{itemize}
Moreover, if a lifting of $f$ has a periodic point of period
$2^dq$ with $d\ge 0$, $q>1$, $q$ odd, then $h_{top}(f)\ge \frac{\lambda_q}{2^d}$,
where $\lambda_q$ is the unique positive root of $X^q-2X^{q-2}-1$.
\end{prop}

\begin{theo}\label{theo:h-degree1}
Let $f\colon \IS\to \IS$ be a circle map of degree $1$, and let
$[a,b]$ be its rotation interval. The following
conditions are equivalent:
\begin{itemize}
\item $h_{top}(f)>0$.
\item There exists two integers $m,n$ with $1<n<m$ such that
$f$ has two periodic points of periods $n,m$ respectively
and $m/n$ is not an integer.
\item Either $a<b$, or there exist $p\in\IZ$ and $q\in\IN$ with $\gcd(p,q)=1$
such that $a=b=p/q$ and $f$ has a periodic point whose period
is not of the form $2^dq$, $d\ge 0$.
\end{itemize}
Moreover, if $a<b$, then $h_{top}(f)\ge \log \beta_{a,b}$, where $\beta_{a,b}$ 
is the largest root of
$$
\sum_{(p,q)\in\IZ\times \IN, \frac pq\in (a,b)} t^{-q}-\frac 12,
$$
and for all real numbers $a<b$, there exists a circle map of
degree $1$ and topological entropy $\log \beta_{a,b}$.
\end{theo}

The set of periods implied by the existence of a given periodic orbit 
depends on its \emph{pattern}, that is, the relative position of the points 
within the orbit, and not only on its period. 
For an interval map $f$, a periodic orbit $P$ \emph{has a division}\index{division}\index{no division} if there exists a point $y\notin P$ such that, 
$$
\forall x\in P,\ x<y\Rightarrow f(x)>y\text{ and }x>y\Rightarrow f(x)<y.
$$
A periodic orbit of
odd period $p>1$ has clearly no division; this fact is important 
(although hidden) when proving that an odd period greater than $1$ implies
a cofinite set of periods (a subset of $\IN$ is cofinite\index{cofinite set} 
if it contains all but at most finitely many integers), which is a
part of Sharkovsky's Theorem~\ref{theo:Sharkovsky}.
On the other hand, an interval map can have a periodic orbit with a 
division but a set of periods which is not cofinite (e.g., the set of
all even integers and $1$). The notion of division was extended to
tree maps by Alsedà and Ye, and led to the following
results \cite{AY2, Ye3}. Since the definition of division for tree maps is
more technical than for interval maps, we do not give it here and we refer
the interested readers to the cited papers.
See also Blokh's paper \cite{Blo10} for results about periods
and entropy of tree maps.

\begin{theo}
Let $f\colon T\to T$ be a tree map. The following assertions are equivalent:
\begin{itemize}
\item $h_{top}(f)>0$.
\item There exists $n\in\IN$ such that $f^n$ has a periodic orbit
of period greater than $1$ with no division.
\end{itemize}
Moreover, if $f$ has a periodic orbit with no division, 
then $h_{top}(f)\ge \frac{1}{e(T)}\log 2$, 
where $e(T)$ denotes the number of endpoints of $T$.
\end{theo}

Finally, the next theorem holds for any graph map. It was first shown by 
Blokh by the means of spectral decomposition \cite{Blo15},
then Llibre and Misiurewicz gave a more direct proof \cite{LM3}.

\begin{theo}
Let $f$ be a graph map. The following assertions are equivalent:
\begin{itemize}
\item $h_{top}(f)>0$.
\item There exists $p\in\IN$ such that the set of periods of $f$
contains $p\IN$.
\end{itemize}
\end{theo}

%************************************************************************
\section{Entropy of transitive and topologically mixing maps}

A transitive interval map always has a positive entropy. Moreover, this 
entropy can be uniformly bounded from below. The lower bound of
entropy of transitive interval maps (resp. transitive interval 
maps with two fixed points) is classical, as well as the examples
realizing the minimum.  
Entropy of topologically mixing interval maps can also be bounded from below,
but the infimum is not reached.

In the next proposition, statement (i) was first stated by Blokh \cite{Blo3}
(see \cite{Blo2} for the proof) and was also proved by Block and Coven
\cite{BCov}; statement (ii) was shown by Block and Coven \cite{BCov}; 
statement (iii) follows from a result of Bobok and Kuchta  \cite{BK2}
(it can also be seen as a consequence of 
Theorem~\ref{theo:summary-mixing} and Lemma~\ref{lem:Mq}).
Notice that, in (ii), a transitive interval map with two fixed points is 
necessarily topologically mixing by Theorem~\ref{theo:summary-transitivity}.

\begin{prop}\label{prop:transitivity-lower-bound-htop}
Let $f\colon I\to I$ be an interval map.
\begin{enumerate}
\item If $f$ is transitive, then $h_{top}(f)\ge\frac{\log 2}{2}$.
\item If $f$ is transitive and has at least two fixed points, then $h_{top}(f)\ge\log 2$.
\item If $f$ is topologically mixing, then $h_{top}(f)>\frac{\log 2}{2}$.
\end{enumerate}
\end{prop}

\begin{proof}
\par\noindent
First we prove (iii). Suppose that $f$ is topologically mixing. Then  $f$ 
has a periodic point
of odd period greater than $1$ by Theorem~\ref{theo:summary-mixing}.
According to Proposition~\ref{prop:odd-period-turbulent}, there exist
two intervals $J=[a,b]$ and $K=[c,d]$ with $b<c$, $a\ne \min I$,
$d\ne \max I$ and such that $(J,K)$ is a strict horseshoe for $f^2$. 
We set $A:=[a,d]$ and $L:=[b,c]$.
By the intermediate value theorem,
$f^2(J)\supset A$ because $f^2(J)\supset J\cup K$. Similarly, 
$f^2(K)\supset A$. The map $f$ is topologically mixing and the non
degenerate closed
interval $L$ does not contains the endpoints of $I$, thus
there exists a positive integer $n$ such that $f^{2n}(L)\supset A$ 
by Theorem~\ref{theo:summary-mixing}. Applying 
Lemma~\ref{lem:chain-of-intervals}(iii) to the family of chains of intervals
$$
\{(I_0,\ldots, I_{n-1}, A)\mid \forall i\in\Lbrack 0,n-1\Rbrack, I_i\in\{J,K\}\},
$$
we see that there exist $2^{n}$ closed intervals  
$(L_i)_{1\le i\le 2^{n}}$ with pairwise disjoint
interiors such that $L_i\subset J\cup K$ and 
$f^{2n}(L_i)\supset A= J\cup L\cup K$ for all 
$i\in\Lbrack 1,2^n\Rbrack$. We deduce that 
$(L_1,L_2,\ldots,L_{2^{n}},L)$ is
a $(2^{n}+1)$-horseshoe for $f^{2n}$. Thus, by 
Proposition~\ref{prop:horseshoe-htop},
$$
h_{top}(f)=\frac{1}{2n}h_{top}(f^{2n})\ge \frac{\log (2^n+1)}{2n}
>\frac{\log 2}{2}.
$$
This is (iii).

Now we suppose that $f$ is transitive. If $f$ is topologically mixing, then it
follows from (iii) that $h_{top}(f)\ge\frac{\log 2}{2}$. If $f$ is
transitive but not topologically mixing, then, according to 
Theorem~\ref{theo:summary-transitivity}, there exists a fixed point $c$ in
the interior of $I$ such that, if we set $J:=[\min I,c]$ and $K:=[c,\max I]$,
then both maps $f^2|_J$ and $f^2|_K$ are topologically mixing. 
The point $c$ is also fixed for the map $f^2|_J$, and $c$ is not in
the interior of $J$. Therefore, $f^2|_J$ has a horseshoe
by Lemma~\ref{lem:non-turbulent}, and hence
$h_{top}(f^2)\ge\log 2$ by Proposition~\ref{prop:horseshoe-htop}. Thus  $h_{top}(f)=\frac{1}{2}h_{top}(f^2) \ge \frac{\log 2}{2}$, which gives~(i).

Finally, (ii) follows straightforwardly from  
Lemma~\ref{lem:non-turbulent} and Proposition~\ref{prop:horseshoe-htop}.
\end{proof}

The bounds given in the preceding proposition 
are sharp. In Example~\ref{ex:htop-transitive} below, 
two maps realizing respectively the equalities in
Proposition~\ref{prop:transitivity-lower-bound-htop}(i)-(ii) are exhibited.
In Example~\ref{ex:htop-mixing-log-lambda},
we saw that, for every odd integer $p>1$, there exists a
topologically mixing map $f_p$ whose entropy is equal to $\log \lambda_p$,
where $\lambda_p$ is the unique positive root of $X^p-2X^{p-2}-1$.
According to Lemma~\ref{lem:Mq}, 
$\lim_{p\to+\infty}\lambda_p=\sqrt{2}$. 
Combining this with Proposition~\ref{prop:transitivity-lower-bound-htop}(iii), 
this shows that
$$
\inf\{h_{top}(f)\mid f \text{ topologically mixing interval map}\}=
\frac{\log 2}{2}.
$$

\begin{ex}\label{ex:htop-transitive}
We are going to exhibit a transitive map $S$ of topological entropy
$\frac{\log 2}{2}$ and a transitive map $T_2$ with two fixed points
of topological entropy $\log 2$.  
We define $T_2\colon [0,1]\to [0,1]$ and $S\colon [-1,1]\to [-1,1]$ by
$$
\left\{\begin{array}{ll}
\forall x\in [0,\frac12],&T_2(x):=2x,\\   
\forall x\in [\frac12,1],&T_2(x):=2(1-x),
\end{array}\right.\qquad
\left\{\begin{array}{ll}
\forall x\in[-1,-\frac12],&S(x):=2x+2,\\ 
\forall x\in[-\frac12,0],&S(x):=-2x,\\ 
\forall x\in[0,1],&S(x):=-x.
\end{array}\right.
$$
These two maps are represented in Figure~\ref{fig:htop-transitive}.
See also Figure~\ref{fig:transitive-not-mixing} 
page~\pageref{fig:transitive-not-mixing} for the graph of $S^2$.

\begin{figure}[htb]
\centerline{\includegraphics{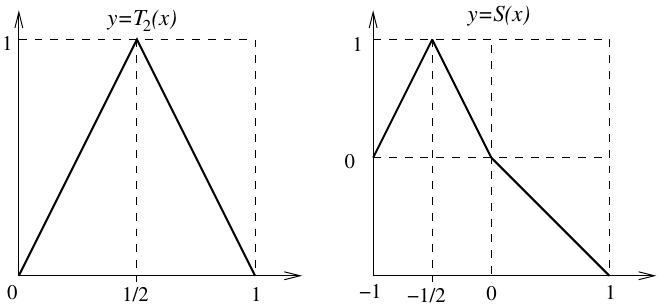}}
\caption{The map on the left (the \emph{tent map}) is transitive with two
fixed points and its topological entropy is $\log 2$. The map on the right
is transitive with a unique fixed point and its topological entropy is 
$\frac{\log 2}{2}$.} 
\label{fig:htop-transitive}
\end{figure}

It was proved in Example~\ref{ex:tent-map} that $T_2$ is topologically mixing.
Since $T_2$ is $2$-Lipschitz, its topological entropy is less than or equal to $\log 2$ by
Proposition~\ref{prop:htop-lipschitz}. Moreover, $T_2$ has two fixed
points ($0$ and $\frac{2}{3}$), and thus $h_{top}(T_2)\ge\log 2$
by Proposition~\ref{prop:transitivity-lower-bound-htop}(ii).
Consequently, $h_{top}(T_2)=\log 2$.

The map $S$ was proved to be transitive
in Example~\ref{ex:transitive-not-mixing}. Thus
$h_{top}(S)\ge \frac{\log 2}{2}$
by Proposition~\ref{prop:transitivity-lower-bound-htop}(i). 
Moreover, $S^2$ is $2$-Lipschitz, and
thus $h_{top}(S^2)\le \log 2$ by Proposition~\ref{prop:htop-lipschitz}.
We deduce that $h_{top}(S)=\frac{\log 2}{2}$.
\end{ex}

The two maps in the preceding example have a common property: 
they are $P$-linear for some finite invariant set $P$. 
Coven and Hidalgo proved that a transitive
interval map $f$ satisfying $h_{top}(f)=h_{top}(f_P)$ for some finite
$f$-invariant set $P$ is necessarily $P$-monotone \cite{CH}. 
This implies that there is little freedom for maps realizing the
bounds in Proposition~\ref{prop:transitivity-lower-bound-htop}(i)-(ii).
In particular, Bobok and Kuchta showed that there is a unique
transitive interval map of entropy $\frac{\log 2}2$, up to conjugacy 
\cite[Theorem~4.1]{BK2}.

Before proving these results, we are going to show that a transitive map $f$
satisfying $h_{top}(f)=h_{top}(f_P)$ cannot have non accessible endpoints.
We shall use the next lemma, which is
an easy corollary of the Perron-Frobenius Theorem~\ref{theo:Perron-Frobenius},
several times.

\begin{lem}\label{lem:rho-submatrix}
Let $B$ be a positive $n\times n$ matrix. Let $E$ be a nonempty subset of 
$\Lbrack 1,n\Rbrack$ and let $B'$ denote the matrix obtained from $B$ by removing
the rows and the columns with indices $i\in E$. Then $\lambda(B)>\lambda(B')$.
\end{lem}

\begin{proof}
It is sufficient to prove the lemma for $E=\{1\}$. We set
$$
A:=\left(\begin{array}{cccc}
0&0&\cdots&0\\
0&&&\\
\vdots&& B'&\\
0&&&
\end{array}\right),
$$
that is, $A$ is the
matrix obtained from $B$ by filling the first line and the first column
of $B$ with $0$'s. Then $0\le A\le B$ and $A\neq B$ because $B>0$.
Thus $\lambda(A)<\lambda(B)$ by the Perron-Frobenius 
Theorem~\ref{theo:Perron-Frobenius}(ii).
Moreover, the set of eigenvalues of $A$ is equal to the set
of eigenvalues of $B'$ union $\{0\}$, and hence $\lambda(B')=\lambda(A)<\lambda(B)$.
\end{proof}

\begin{lem}\label{lem:h>hfP}
Let $f\colon [a,b]\to [a,b]$ be a topologically mixing interval map with
a non accessible endpoint.
\begin{enumerate}
\item For every finite invariant set $P$, $h_{top}(f)> h_{top}(f_P)$.
\item If $f$ has a horseshoe, then $h_{top}(f)>\log 2$.
\end{enumerate}
\end{lem}

\begin{proof}
Let $E$ denote the set of non accessible endpoints of
$f$. By Proposition~\ref{prop:accessibility}, there are four cases:
either $E=\{a\}$ and $f(a)=a$, or $E=\{b\}$ and $f(b)=b$, or 
$E=\{a,b\}$ and both $a,b$ are fixed, or $E=\{a,b\}$ and $f(a)=b, f(b)=a$.

Let $P=\{p_0<p_1<\cdots <p_k\}$ be a finite invariant set. 
By Proposition~\ref{prop:htop-Markov-map}, $h_{top}(f_P)=
\max(0,\log \lambda(M(f_P)))$. It is sufficient
to consider the case $\lambda(M(f_P))>1$ because $h_{top}(f)>0$ by 
Proposition~\ref{prop:transitivity-lower-bound-htop}.
The set $P':=P\setminus E$ is invariant too. 
First we are going to show
\begin{equation}\label{eq:PP'}
\lambda(M(f_P))=\lambda(M(f_{P'})).
\end{equation}
We split the proof into four cases.

\medskip
\noindent\textbf{Case \boldmath $P\cap E=\emptyset$.} Then $P=P'$. 

\medskip
\noindent\textbf{Case \boldmath $P\cap E=\{a\}$.} Then $f(a)=a= p_0$, and the 
matrix $M(f_{P'})$  is obtained from $M(f_P)$ by deleting the 
first row and the first column. Moreover, 
$a$ does not belong to $f([p_i,p_{i+1}])$ for any $i\in\Lbrack 1,k-1\Rbrack$,
which implies that the first column of $M(f_P)$ is $ ^t(10\cdots 0)$.
Then $\det (M(f_P)-X\Id)=(1-X)\det(M(f_{P'})-X\Id)$. Since
$\lambda(M(f_P))>1$, we have $\lambda(M(f_P))=\lambda(M(f_{P'}))$ 
by Proposition~\ref{prop:eigenvalue-rootchi}.

\medskip
\noindent\textbf{Case \boldmath $P\cap E=\{b\}$.}
This case is similar to the previous one.

\medskip
\noindent\textbf{Case \boldmath $P\cap E=\{a,b\}$.} Either
both $a$ and $b$ are fixed or $\{a,b\}$ is a periodic orbit of period $2$.
In both cases, $f^2(a)=a=p_0$ and $f^2(b)=b=p_k$.
We have $\{a,b\}\cap f([p_i,p_{i+1}])=\emptyset$ for all 
$i\in\Lbrack 1,k-2\Rbrack$. This implies
that, for all $i\in\Lbrack 1,k-2\Rbrack$, there is no path of length $2$ from
$[p_i,p_{i+1}]$ to $[a,p_1]$ or to $[p_{k-1},b]$ in $G(f_P)$, and thus
the first and last columns of $M(f_P)^2$ are respectively
$ ^t(10\cdots 0)$ and $ ^t(0\cdots 01)$ (recall
that, by Proposition~\ref{prop:Mn}, the $(i,j)$-coefficient of $M(f_P)^n$ is 
the number of paths of length
$n$ from $[p_i,p_{i+1}]$ to $[p_j,p_{j+1}]$ in $G(f_P)$).
Moreover, the matrix $M(f_{P'})^2$  is obtained from $M(f_P)^2$ by 
deleting the first and last rows and the first and last columns.
As in the second case, this implies that
$\lambda(M(f_P)^2)=\lambda(M(f_{P'})^2)$,
so $\lambda(M(f_P))=\lambda(M(f_{P'}))$ by Lemma~\ref{lem:lambdaMn}.

\medskip
Now we are going to show that $h_{top}(f)>\log \lambda(M(f_P))$. According
to \eqref{eq:PP'}, we can assume that $P\cap E=
\emptyset$ (otherwise, we replace $P$ by $P'$). Moreover, since $E\neq
\emptyset$, we can assume that $a\in E$ (the case $b\in E$ is similar).
If $b\in E$, then $f^2(b)=b$ by
Lemma~\ref{lem:accessibility}. If $b\notin E$, then there exists
$x\in (a,b)$ such that $f^n(x)=b$, which implies that $f^2(b)\ne a$,
otherwise $a$ would be accessible. In both cases, $f^2(b)\ne a$.
We set $c:=\min f^2([p_0,b])$. Then $c>a$ because $a$ is non accessible
and $f^2(b)\ne a$, and $c<p_0$ because $[p_0,b]$ is
not $f^2$-invariant (otherwise, it would contradict the fact that $f$ is
topologically mixing). If $b\in E$, we set $d:=\max f^2([a,p_k])$, and
a similar argument implies that $p_k<d<b$. If $b\notin E$, we set $d:=b$.
Let $Q:=P\cup\{c,d\}$ (this may not be an invariant set). 
Then $c=\min Q$ (because $f^2([p_0,b])\supset f^2([p_0,p_k])
\supset P$), and similarly $d=\max Q$. According to
Proposition~\ref{prop:accessibility}, there exists a positive integer $n$
such that, for all $Q$-intervals $J$, $f^n(J)\supset [c,d]$. Thus the matrix
$B:=M(f^n|Q)$ is positive. We remove from $B$ 
its first line and column (corresponding to the $Q$-interval $[c,p_0]$)
and, if $b\in E$, its last line and column (corresponding to $[p_k,d]$);
we call $B'$ the resulting matrix.
Then $B'=M(f^n|P)$. By Proposition~\ref{prop:matrix-htop},
$h_{top}(f^n)\ge \log \lambda (B)$, and by Lemma~\ref{lem:rho-submatrix}, $\lambda(B)>\lambda(B')$. Moreover, $B'\ge M(f_P)^n$ by Lemma~\ref{lem:MfPn}, 
which implies that $\lambda(B')\ge \lambda(M(f_P)^n)$ 
by the Perron-Frobenius Theorem~\ref{theo:Perron-Frobenius}.
Combining these inequalities with the fact that
$\lambda(M(f_P)^n)=(\lambda(M(f_P)))^n$ (Lemma~\ref{lem:lambdaMn}), we get
$$
h_{top}(f)=\frac 1n h_{top}(f^n)\ge \frac 1n  \log \lambda(B)>
\frac 1n \log \lambda(B')\ge \log \lambda(M(f_P)).
$$
This proves (i).

\medskip
Now we assume that $f$ has a horseshoe.
According to Lemma~\ref{lem:particularhorseshoe}, there exist
three points $u,v,w$ in $[a,b]$ such that
$f(u)=f(w)=u$, $f(v)=w$ and, either $u<v<w$, or $u>v>w$.
The set $P:=\{u,v,w\}$ is invariant and $M(f_P)$ is the
matrix of a $2$-horseshoe:
$$
M(f_P)=\left(\begin{array}{cc}1&1\\1&1\end{array}\right).
$$ 
Thus $\lambda(M(f_P))=2$ and $h_{top}(f_P)=\log \lambda(M(f_P))=\log 2$
(Proposition~\ref{prop:htop-Markov-map}). 
Since $h_{top}(f)>h_{top}(f_P)$ by (i),
we get $h_{top}(f)>\log 2$, which gives (ii).
\end{proof}

\begin{prop}\label{prop:h=hfP}
Let $f\colon [a,b]\to [a,b]$ be a transitive interval map and $P$ a finite
invariant set. If $h_{top}(f)=h_{top}(f_P)$, 
then $f$ is $P$-monotone.
\end{prop}

\begin{proof}
Recall that $h_{top}(f)>0$ because $f$ is transitive 
(Proposition~\ref{prop:transitivity-lower-bound-htop}), and thus
$h_{top}(f)=\log\lambda(M(f_P))$ by 
Proposition~\ref{prop:htop-Markov-map}.

The first step of the proof consists of showing that the endpoints 
$a,b$ belong to $P$. 
Suppose that $f$ is topologically mixing. If $a\notin P$, we set 
$Q:=P\cup\{a\}$. According to Lemma~\ref{lem:h>hfP}(i), $f$ has no non 
accessible endpoint and thus, by Proposition~\ref{prop:accessibility},
there exists an integer $n\ge 0$ such
that, for all $Q$-intervals $J$, $f^n(J)=[a,b]$. Thus the matrix
$B:=M(f^n|Q)$ is positive. Let $B'$ be the matrix obtained 
from $B$ by removing the first line and the first column, 
that is, $B'=M(f^n|P)$. 
By Lemma~\ref{lem:rho-submatrix},
$\lambda(B)>\lambda(B')$. Moreover, $B'\ge M(f_P)^n$ by Lemma~\ref{lem:MfPn}, 
and hence $\log\lambda(B')\ge \log\lambda(M(f_P)^n)=n\log\lambda(M(f_P))$ 
by the Perron-Frobenius Theorem~\ref{theo:Perron-Frobenius}(ii) and 
Lemma~\ref{lem:lambdaMn}. Furthermore, 
$h_{top}(f^n)\ge \log \lambda(B)$
according to Proposition~\ref{prop:matrix-htop}. 
These inequalities imply
$h_{top}(f)>\log \lambda(M(f_P))$, a contradiction. We deduce that $a\in P$.
Similarly, $b$ belongs to $P$ too. We have proved:
\begin{equation}\label{eq:endpoints-mix}
h_{top}(f)=h_{top}(f_P)\text{ and }f\text{ topologically mixing}\Rightarrow a,b\in P.
\end{equation}

Suppose now that $f$ is transitive but not topologically mixing.
According to Theorem~\ref{theo:summary-transitivity}, there exists
$c\in (a,b)$ such that 
\begin{equation}\label{eq:nonmix}
f(c)=c,\ f([a,c])=[c,b],\ f([c,b])=[a,c],\text{ and }f^2|_{[a,c]},\ f^2|_{[c,b]}
\text{ are mixing}.
\end{equation}
If $c\notin P$, we set $P':=P\cup\{c\}$ and we consider 
$p:=\max (P\cap [a,c))$ and $q:=\min (P\cap (c,b])$ 
(these points exist by \eqref{eq:nonmix} and the $f$-invariance of $P$). 
According to
\eqref{eq:nonmix}, $f(p)\ge c>p$ and $f(q)\le c<q$. This implies that
$f_{P'}$ is decreasing on $[p,c]\cup [c,q]=[p,q]$. Thus $f_{P'}$
is $P$-monotone and, according to Proposition~\ref{prop:htop-Markov-map},
$h_{top}(f_P)=h_{top}(f_{P'})$. If we prove the proposition for 
$P'$, this will imply that the proposition holds for $P$ too. Therefore, we may assume that 
$c\in P$ (otherwise we replace $P$ by $P'$).
We set $P_1:=P\cap [a,c]$, $P_2:=P\cap [c,b]$ and $g:=f^2$. Then the family of 
$P$-intervals splits into $P_1$-intervals and $P_2$-intervals.

One can show that $h_{top}(g_P)=h_{top}(g_{P_1})=h_{top}(g_{P_2})$ by
using Bowen's formula (Theorem~\ref{theo:Bowen-formula}),
the uniform continuity of $f$ and
the fact that $f$ swaps the intervals $[a,c]$ and $[c,b]$ by \eqref{eq:nonmix}.
Moreover, \eqref{eq:nonmix} implies that $f_P^2=g_P$.
Since $h_{top}(f)=h_{top}(f_P)$ by assumption,
and $h_{top}(g)=2h_{top}(f)$, we get $h_{top}(g)=h_{top}(g_P)$ and
$$
h_{top}(g)=h_{top}(g|_{[a,c]})=h_{top}(g|_{[c,b]})=h_{top}(g_{P_1})=h_{top}(g_{P_2}).
$$
Moreover, according to Proposition~\ref{prop:htop-Markov-map},
\begin{equation}\label{eq:B1B2}
\lambda(M(f_P)^2)=\lambda(M(g_{P_1}))=\lambda( M(g_{P_2})).
\end{equation} 
Applying  \eqref{eq:endpoints-mix} to $g|_{[a,c]}$ and $g|_{[c,b]}$,
we see that $a,b\in P$. % and hence $a,b\in P$. 
Moreover, $f$ has no non accessible endpoint by 
Lemma~\ref{lem:h>hfP} applied to $g|_{[a,c]}$ and $g|_{[c,b]}$. 
This concludes the proof of the first step, that is:
$$
h_{top}(f)=h_{top}(f_P)\text{ and }f\text{ transitive}\Rightarrow a,b\in P.
$$

In the second step, we are going to show that $f$ is $P$-monotone.
Suppose on the contrary that there exists a $P$-interval $I$ 
such that $f|_I$ is not one-to-one, that is, there exist two points
$u<v$ in $I$ such that $f(u)=f(v)$.
Since $f$ is transitive, $f([u,v])$ is not degenerate, and thus
\begin{gather}
\text{either }\max f([u,v]) >f(u),\label{eq:max>f(u)}\\ 
\text{or }\min f([u,v]) <f(u).\nonumber
\end{gather}
We assume we are in case \eqref{eq:max>f(u)}, the other case being similar.
Let $w\in (u,v)$ be such that $f(w)=\max f([u,v])$, and let $U$ be an open
interval containing $w$ such that $f(x)>f(v)$ for all $x\in U$.

By Proposition~\ref{prop:transitivity-periodic-points}, the set of 
periodic points is dense, and thus we can choose a
periodic point $p_0\in U$ such that $w\notin\CO_f(p_0)$. Let
$p_1,p_2$ be the two points in $\CO_f(p_0)$ such that $p_1<w<p_2$
and there is no point of $\CO_f(p_0)$ between $w$ and $p_i$ for $i\in\{1,2\}$
and let $p\in \{p_1,p_2\}$ be such that $f(p)=\max\{f(p_1), f(p_2)\}$.
It is possible that either $p_1$ or $p_2$ does not exist: if
$\CO_f(p_0)>w$ (resp. $\CO_f(p_0)<w$), then there is no $p_1$ (resp. $p_2$); 
in this case
$p$ is just equal to the unique existing $p_i$. Since $p_0\in U$, at least
one of the points $p_1,p_2$ belongs to $U$, and hence $f(p)>f(v)$.
In the sequel, we assume $p=p_1$, the case
$p=p_2$ being symmetric. We define 
$$
z:=\min\{x\in [w,v]\mid f(x)=f(p)\}.
$$
This point is well defined because $f(v)<f(p)\le f(w)$, and $z>p$. Moreover, 
$\CO_f(p)\cap (w,z)=\emptyset$ according to the choice of $p$ and $z$.
The set $Q:=P\cup\CO_f(p)\cup\{z\}$ is a finite invariant set,
and $[p,z]$ is a $Q$-interval. Moreover,  $f_Q$ is constant on $[p,z]$
because $f(p)=f(z)$, and thus the row of $M(f_Q)$ corresponding to $[p,z]$ is 
$(00\cdots  0)$. Then, according to Proposition~\ref{prop:htop=suphfP},
$h_{top}(f_Q)\ge h_{top}(f_P)$ (because $P$ is $f_Q$-invariant)
and $h_{top}(f)\ge h_{top}(f_Q)$. Since $h_{top}(f)=h_{top}(f_P)$ by
assumption, we have 
\begin{equation}\label{eq:htopf-fQ}
h_{top}(f)=h_{top}(f_Q).
\end{equation}
We split into two
cases, depending on $f$ being mixing or not.

\medskip
\noindent$\bullet$ If $f$ is topologically mixing, then, by Proposition~\ref{prop:accessibility},
there exists an integer $n\ge 0$ such 
that, for all $Q$-intervals $J$, $f^n(J)=[a,b]$ (recall that we have shown that
$f$ has no non accessible endpoint). Therefore $B:=M(f^n|Q)$ is 
a positive matrix.
Moreover, $B\ge M(f_Q)^n$ (by Lemma~\ref{lem:MfPn})
and $M(f_Q)^n\neq B$ because the row of $M(f_Q)^n$
corresponding to $[p,z]$ is $(0\cdots 0)$. 
Thus $\lambda(B)> \lambda(M(f_Q)^n)$ by the
Perron-Frobenius Theorem~\ref{theo:Perron-Frobenius}(ii).

\medskip
\noindent$\bullet$ If $f$ is transitive but not topologically mixing, then we
are in the situation described in \eqref{eq:nonmix}. As in the first step, 
we can assume that $c\in P$ and we set $g:=f^2$.
By Proposition~\ref{prop:accessibility},
there exists an integer $k\ge 0$ such 
that, for all $Q$-intervals $J_1\subset [a,c]$, $g^k(J_1)=[a,c]$,
and for all $Q$-intervals $J_2\subset [c,b]$, $g^k(J_2)=[c,b]$. Therefore
the matrix $B:=M(g^k|Q)$ is of the form
$$
B=\left(\begin{array}{cc} B_1&0\\0&B_2 \end{array}\right)
\quad\text{with }B_1>0, B_2>0.
$$
Moreover, $B\ge M(g_Q)^k$ (by Lemma~\ref{lem:MfPn}) and
$$
M(g_Q)^k=\left(\begin{array}{cc} M(g_{P_1})^k& 0 \\ 0 & M(g_{P_2})^k
\end{array}\right).
$$
We deduce that $B_1 \ge M(g_{P_1})^k$ and 
$B_2\ge M(g_{P_2})^k$. Moreover, the mixing case above implies
that, if $[p,z]\subset [a,c]$
(resp. $[p,z]\subset [c,b]$), then $B_1\neq M(g_{P_1})^k$
(resp. $B_2\neq M(g_{P_2})^k$), so
$$
\lambda(B_1)>\lambda(M(g_{P_1})^k)\quad (\text{resp. }
\lambda(B_2)> \lambda(M(g_{P_2})^k))
$$
by the Perron-Frobenius Theorem~\ref{theo:Perron-Frobenius}(ii).
Combining this with \eqref{eq:B1B2}, we get
$\lambda(B)> \lambda(M(f_Q)^{2k})$.

\medskip
In both cases ($f$ topologically mixing or not), there exists $n\ge 0$ such that 
$\lambda(M(f^n|Q))> \lambda(M(f_Q)^n)$. Recall that 
$\lambda(M(f_Q)^n)=\lambda(M(f_Q))^n$. This leads to
$$
h_{top}(f)=\frac 1n h_{top}(f^n)\ge \frac 1n \log \lambda(M(f^n|Q))>
\log\lambda(M(f_Q))=h_{top}(f_Q)=h_{top}(f)
$$
(the last equality comes from \eqref{eq:htopf-fQ}). But this is a contradiction.
We conclude that $f$ is one-to-one on every $P$-interval. Since $a,b$ belong
to $P$, we deduce that $f$ is $P$-monotone.
\end{proof}

In the next proposition, the first assertion is due to Bobok and Kuchta 
\cite[Theorem~4.1]{BK2}.

\begin{prop}
Let $f\colon [a,b]\to [a,b]$ be a transitive interval map. 
\begin{itemize}
\item If $h_{top}(f)=
\frac{\log 2}{2}$, then $f$ is topologically conjugate to the map
$S$ defined in Example~\ref{ex:htop-transitive}.
\item If $f$ has at least two fixed points and $h_{top}(f)=\log 2$, then $f$ is
topologically conjugate to the tent map $T_2$ defined in
Example~\ref{ex:htop-transitive}.
\end{itemize}
\end{prop}

\begin{proof}
Let $f\colon [a,b]\to [a,b]$ be a transitive interval map such that 
$h_{top}(f)=\frac{\log 2}{2}$.
By Proposition~\ref{prop:transitivity-lower-bound-htop}(iii), $f$ is not
topologically mixing. Thus, by Theorem~\ref{theo:summary-transitivity}.
there exists $c\in (a,b)$ such that 
\begin{equation}\label{eq:nonmix2}
f(c)=c,\ f([a,c])=[c,b],\ f([c,b])=[a,c],\text{ and }f^2|_{[a,c]},\ f^2|_{[c,b]}
\text{ are mixing}.
\end{equation}
Moreover, $f^2|_{[a,c]}$ has a 
horseshoe by Lemma~\ref{lem:non-turbulent}. Thus $f^2|_{[a,c]}$ 
has no non accessible endpoint by Lemma~\ref{lem:h>hfP}(ii).
This implies that 
there exists $d\neq c$ such that $f(d)=c$, that is,
\begin{gather}
\text{either }d\in [a,c)\text{ and }f(d)=c,\label{eq:d<c}\\
\text{or }d\in(c,b]\text{ and }f(d)=c.\label{eq:d>c}
\end{gather}
Assume that \eqref{eq:d<c} holds, the case \eqref{eq:d>c} being symmetric.
Let $m:=\max f([d,c])$. Then $m\in [c,b]$ and $f([d,c])=[c,m]$
by \eqref{eq:nonmix2}. Suppose that 
\begin{equation}\label{eq:fm<d}
\min f([c,m])<d.
\end{equation}
Then $f([c,m])\supset [d,c]$, so $f^2([d,c])= f([c,m])\supset
[d,c]$. Thus there exists $e$ in $[d,c]$ such that $f^2(e)=d$. See 
the positions of these points in Figure~\ref{fig:false-conj-fS}.
\begin{figure}[htb]
\centerline{\includegraphics{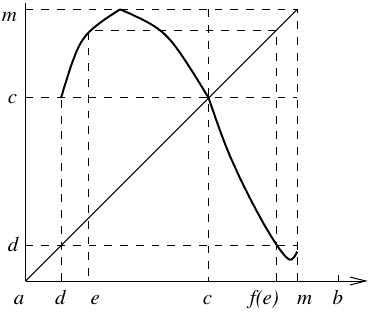}}
\caption{The positions of the various points in the case 
\eqref{eq:fm<d}.}
\label{fig:false-conj-fS}
\end{figure}

We set $P:=\{d,e,c,f(e)\}$. Then $P$ is invariant
and $h_{top}(f_P)\ge\frac{\log 2}2$ because
$([d,e],[e,c])$ is a horseshoe for $f_P^2$. Since $h_{top}(f_P)\le h_{top}(f)$
by Proposition~\ref{prop:htop=suphfP}, we have 
$h_{top}(f_P)=h_{top}(f)=\frac{\log 2}2$. Thus $f$ is $P$-monotone 
by Proposition~\ref{prop:h=hfP}, which implies that $d=a$. But this
contradicts the fact that $\min f([c,m])<d$.
We deduce that \eqref{eq:fm<d} does not hold, that is,
$\min f([c,m])\ge d$. Then 
$f([c,m])\subset [d,c]$ and $f([d,m])=f([d,c])\cup f([c,m])\subset
[c,m]\cup[d,c]=[d,m]$. Therefore
the interval $[d,m]$ is invariant,
which is possible only if $d=a$ and $m=b$ because $f$ is transitive.
Since $f$ is onto, there exists $e'\in[c,b]$ such that $f(e')=a$.
Let $e\in [a,c]$ be such that $f(e)=e'$. 
We set $Q:=\{a,e,c,e'\}$.
Then $Q$ is invariant (recall that $a=d$ and $f(d)=f(c)=c$)
and $h_{top}(f_Q)\ge\frac{\log 2}2$ because
$([a,e],[e,c])$ is a horseshoe for $f_Q^2$. 
As above, Proposition~\ref{prop:htop=suphfP} implies that
$h_{top}(f_Q)=h_{top}(f)=\frac{\log 2}2$. Thus $f$ is $Q$-monotone 
by Proposition~\ref{prop:h=hfP}. This implies that $e'=b$.
The map $f$ looks like the one represented
on the left of Figure~\ref{fig:conj-fS}. Notice that the case 
\eqref{eq:d>c} leads to the reverse figure (central symmetry) 
by exchanging the roles of $a$ and $b$.

\begin{figure}[htb]
\centerline{\includegraphics{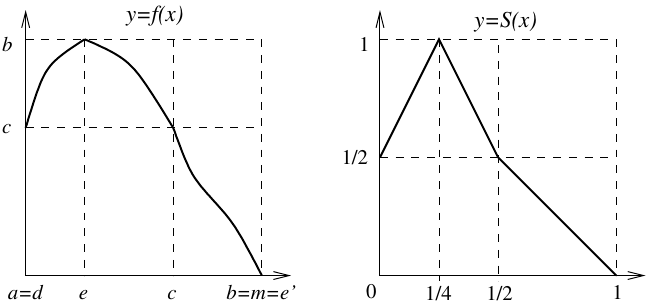}}
\caption{On the left, the map $f$ is $Q$-monotone with $Q:=\{a,e,c,b\}$. It
is conjugate to the map $S$, on the right.}
\label{fig:conj-fS}
\end{figure}

Now we consider $g\colon [a,b]\to [a,b]$ 
a transitive interval map with two fixed 
points such that $h_{top}(g)=\log 2$. Then $g$ is topologically mixing
by Theorem~\ref{theo:summary-transitivity}, and $g$ has a horseshoe
by Lemma~\ref{lem:non-turbulent}. According to 
Lemma~\ref{lem:particularhorseshoe}, there exist
points $a'<c<b'$ such that 
\begin{gather}
\text{either }g(a')=g(b')=a'\text{ and }g(c)=b',\label{eq:gT2-case1}\\
\text{or }g(a')=g(b')=b'\text{ and }g(c)=a'.\label{eq:gT2-case2}
\end{gather}
Therefore, the set $P:=\{a',c,b'\}$ is invariant, and $h_{top}(g_P)=\log
2=h_{top}(g)$. Thus $g$ is $P$-monotone by Proposition~\ref{prop:h=hfP},
which implies that $a'=a$ and $b'=b$.
The map $g$ looks like the one represented on the left of
Figure~\ref{fig:conj-gT2}
in case \eqref{eq:gT2-case1}, and the reverse figure (central symmetry) 
in case \eqref{eq:gT2-case2}.

\begin{figure}[htb]
\centerline{\includegraphics{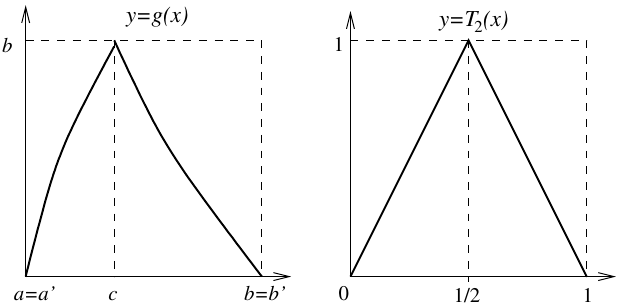}}
\caption{On the left, the map $g$
is $P$-monotone with $P:=\{a,c,b\}$. It
is conjugate to the tent map $T_2$, on the right.}
\label{fig:conj-gT2}
\end{figure}

It remains to show that $f$ and $g$ are topologically conjugate to $S$ and
$T_2$ respectively. This can be seen as a consequence of
the following general result of Parry: 
every transitive piecewise monotone interval
map is conjugate to a piecewise linear map such that the absolute value of
its slope is constant \cite{Par2}. 
This result is easier to prove for $P$-monotone maps. 
We are going to give a proof in the case of the map $g$, which is the 
simplest one. The conjugacy between $f$ and $S$ can be defined in a 
similar way.

We may assume that $g$ satisfies \eqref{eq:gT2-case1} (otherwise, 
we conjugate $g$ by $\psi\colon [a,b]\to[a,b], \psi(x):=b+a-x$
in order to get a map satisfying \eqref{eq:gT2-case1}). 
We set $J_0:=[a,c]$ and $J_1:=[c,b]$. Then, since
$g$ is $P$-monotone, $g$ is increasing
on $J_0$, decreasing on $J_1$, and $g(J_0)=g(J_1)=[a,b]$.
For all $n\ge 1$ and all $(\alpha_0,\ldots,\alpha_{n-1})\in\{0,1\}^n$, 
we define
\begin{equation}\label{eq:Jalpha}
J_{\alpha_0\ldots \alpha_{n-1}}:=\{x\in [a,b]\mid\forall i\in\Lbrack 0,n-1\Rbrack,
\ g^i(x)\in J_{\alpha_i}\}=\bigcap_{i=0}^{n-1} g^{-i}(J_{\alpha_i}).
\end{equation}
This definition implies that, for all $n\ge 2$ and
all $(\alpha_0,\ldots,\alpha_{n-1})\in\{0,1\}^n$,
\begin{equation}\label{eq:g(J-alpha-n)}
g(J_{\alpha_0\ldots \alpha_{n-1}})=J_{\alpha_1\ldots \alpha_{n-1}}.
\end{equation}

\medskip
\textsc{Fact 1.} {\it Let $n\ge 1$.
\begin{enumerate}
\item
$(J_{\alpha_0\ldots\alpha_{n-1}})_{(\alpha_0,\ldots,\alpha_{n-1})\in\{0,1\}^n}$
is a cover of $[a,b]$.
\item 
$(J_{\alpha_0\ldots\alpha_{n-1}})_{(\alpha_0,\ldots,\alpha_{n-1})\in\{0,1\}^n}$
is a family of nonempty compact intervals with pairwise disjoint interiors.
\item $g^n|_{J_{\alpha_0\ldots\alpha_{n-1}}}$ is a homeomorphism
from $J_{\alpha_0\ldots\alpha_{n-1}}$ to $[a,b]$.
\end{enumerate}}

Assertion (i) follows straightforwardly from the fact that, for all
$x\in [a,b]$ and all $i\in\Lbrack 0,n-1\Rbrack$, there is $\alpha_i\in\{0,1\}$
such that $g^i(x)\in J_{\alpha_i}$.

We prove (ii)-(iii) by induction on $n$.

$\bullet$ For $n=1$, the sets are $J_0, J_1$, and (ii)-(iii) are satisfied.

$\bullet$ Suppose that (ii)-(iii) are satisfied for $n\ge 1$.
For $i\in\{0,1\}$, the map $\widetilde{g}_i:=g|_{J_i}$ is a homeomorphism
from $J_i$ to $[a,b]$. We can write $J_{\alpha_0\ldots\alpha_n}$
as
$$
J_{\alpha_0\ldots\alpha_n}=\{x\in J_{\alpha_0}\mid 
g(x)\in J_{\alpha_1\ldots\alpha_n}\}=
\widetilde{g}_{\alpha_0}^{-1}(J_{\alpha_1\ldots\alpha_n}).
$$
This is a nonempty compact interval because 
$\widetilde{g}_{\alpha_0}^{-1}$ is continuous and 
$J_{\alpha_1\ldots\alpha_n}$ is a nonempty compact interval by the
induction hypothesis. 

Let $(\beta_0,\ldots,\beta_n)\ne (\alpha_0,\ldots,\alpha_n)$.
If $\beta_0\ne \alpha_0$, then 
$$
J_{\alpha_0\ldots\alpha_n}\subset J_{\alpha_0},\ J_{\beta_0\ldots\beta_n}\subset J_{\beta_0}
$$ 
and $\Int{J_{\alpha_0}}\cap\Int{J_{\beta_0}}=\emptyset$.
If $\beta_0=\alpha_0$, then 
$(\beta_1,\ldots,\beta_n)\ne (\alpha_1,\ldots,\alpha_n)$,
$$
J_{\alpha_0\ldots\alpha_n}=
\widetilde{g}_{\alpha_0}^{-1}(J_{\alpha_1\ldots\alpha_n}),\ 
J_{\beta_0\ldots\beta_n}=
\widetilde{g}_{\alpha_0}^{-1}(J_{\beta_1\ldots\beta_n})
$$
and these sets have disjoint interiors because $\widetilde{g}_{\alpha_0}$
is a homeomorphism and by the induction hypothesis for
$J_{\alpha_1\ldots\alpha_n}, J_{\beta_1\ldots\beta_n}$.

Finally, $g^n|_{J_{\alpha_0\ldots\alpha_n}}$ is a 
homeomorphism from $J_{\alpha_0\ldots\alpha_n}\subset
J_{\alpha_0\ldots\alpha_{n-1}}$ to its image by the induction hypothesis,
and $g^n(J_{\alpha_0\ldots\alpha_n})=J_{\alpha_n}$ by 
\eqref{eq:g(J-alpha-n)}. Thus
$$
g^{n+1}|_{J_{\alpha_0\ldots\alpha_n}}=\left(g|_{J_{\alpha_n}}\right)
\circ \left(g^n|_{J_{\alpha_0\ldots\alpha_n}}\right)
$$
is a homeomorphism from $J_{\alpha_0\ldots\alpha_n}$ to $[a,b]$.
This ends the induction and the proof of Fact 1.

\medskip
In a similar way, we define intervals $J'_{\alpha_0\ldots\alpha_n}$
for $T_2$, starting with 
$J_0':=[0,\frac12]$ and $J_1':=[\frac12,1]$; at level $n-1$,
we get the cover $\left(\left[\frac{i}{2^n},
\frac{i+1}{2^n}\right]\right)_{0\le i<2^n}$.
The idea is to build a map $\vfi\colon [a,b]\to [0,1]$
such that the image of every interval of the form 
$J_{\alpha_0\ldots\alpha_{n-1}}$ is $J'_{\alpha_0\ldots\alpha_{n-1}}$.

We set $R_{\alpha_0\ldots\alpha_{n-1}}:=J_{\alpha_0\ldots\alpha_{n-1}}\setminus
\min J_{\alpha_0\ldots\alpha_{n-1}}$ (half-open interval).
For all $n\ge 1$, we define the staircase function 
$\vfi_n\colon [a,b]\to [0,1]$ by:
\begin{itemize}
\item $\vfi_n(a)=0$, $\vfi_n(b)=1$,
\item $\forall (\alpha_0,\ldots,\alpha_{n-1})\in\{0,1\}^n$, $\vfi_n$
is constant on $R_{\alpha_0\ldots\alpha_{n-1}}$,
\item $\vfi_n$ is non decreasing and every step is of high $\frac{1}{2^n}$
\end{itemize}
(see Figure~\ref{fig:phi3}).
\begin{figure}[htb]
\centerline{\includegraphics{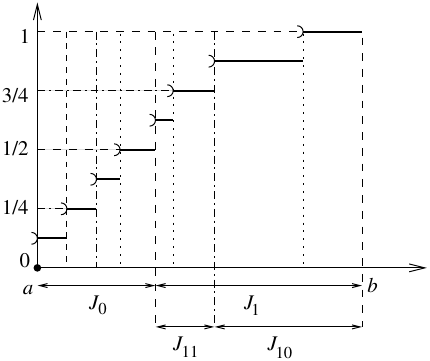}}
\caption{The map $\vfi_3$: the interval $[a,b]$ is divided into 
$2^3=8$ subintervals
$(J_{\alpha_0\alpha_1\alpha_2})_{(\alpha_0,\alpha_1,\alpha_2)\in\{0,1\}^3}$,
$\vfi_3$ is a non decreasing staircase  function and takes its values in 
$\{\frac{i}{8}\mid i\in\Lbrack 0,8\Rbrack\}$.}
\label{fig:phi3}
\end{figure}

%\medskip
\textsc{Fact 2.} {\it  The sequence $(\vfi_n)_{n\ge 1}$ uniformly converges 
to a map $\vfi\colon [a,b]\to [0,1]$. Moreover, $\vfi$ is 
an increasing homeomorphism.}

\medskip
First we show that $(\vfi_n)_{n\ge 0}$ is a Cauchy sequence for the
uniform distance. Let $\eps>0$ and $N\in\IN$ such that $\frac{1}{2^N}<\eps$.
Let $n>m\ge N$. 
Note that $\vfi_n(a)=\vfi_m(a)=0$.
Let $x\in (a,b]$. There exists $(\alpha_0,\ldots\alpha_{n-1})\in\{0,1\}^n$
such that $x$ belongs to $R_{\alpha_0\ldots\alpha_{n-1}}\subset 
R_{\alpha_0\ldots\alpha_{m-1}} $, and there is
$i\in\Lbrack 1, 2^m\Rbrack$ such that 
$\vfi_m$ is equal to $\frac{i}{2^m}$ on ${R_{\alpha_0\ldots\alpha_{m-1}}}$.
Moreover, $\vfi_n$ is equal to a constant 
$c\in (\frac{i-1}{2^m},\frac{i}{2^m}]$ on 
${R_{\alpha_0\ldots\alpha_{n-1}}}$ by construction. Therefore
\begin{equation}\label{eq:fin-fim}
\forall x\in[a,b],\ |\vfi_n(x)-\vfi_m(x)|\le \frac{1}{2^m}<\eps.
\end{equation}
This shows that $(\vfi_n)_{n\ge 0}$ is a Cauchy sequence, and thus
it uniformly converges to a map $\vfi\colon [a,b]\to [0,1]$.
We are going to show that $\vfi$ is continuous, onto and increasing.

In \eqref{eq:fin-fim}, we take $n\to +\infty$.
We get
\begin{equation}\label{eq:fi-fim}
\forall x\in [a,b],\ |\vfi(x)-\vfi_m(x)|<\eps.
\end{equation}
If $x\in \Int{R_{\alpha_0\ldots \alpha_{m-1}}}$, there exists a neighborhood
$U$ of $x$ such that $U\subset R_{\alpha_0\ldots \alpha_{m-1}}$, and thus
\eqref{eq:fi-fim} implies that, for all $y\in U$,
\begin{eqnarray*}
|\vfi(x)-\vfi(y)|&\le& |\vfi(x)-\vfi_m(x)|+|\vfi_m(x)-\vfi_m(y)|
+|\vfi_m(y)-\vfi(y)|\\
&<&\eps+0+\eps=2\eps.
\end{eqnarray*}
If $x=\max R_{\alpha_0\ldots \alpha_{m-1}}$ (resp $x=a$), there
exists $(\beta_0,\ldots,\beta_{m-1})\in\{0,1\}^m$ and a neighborhood $U$
of $x$ such that $U\subset R_{\alpha_0\ldots \alpha_{m-1}}\cup
R_{\beta_0\ldots\beta_{m-1}}$ 
(resp. $U\subset \{a\}\cup R_{\beta_0\ldots\beta_{m-1}}$). There exists
$i\in \Lbrack 1, 2^m\Rbrack$ such that $\vfi_m$ takes only the values
$\frac{i-1}{2^m}$ and $\frac{i}{2^m}$ on $U$. Thus
\eqref{eq:fi-fim} implies that, for all $y\in U$,
\begin{eqnarray*}
|\vfi(x)-\vfi(y)|&\le& |\vfi(x)-\vfi_m(x)|+|\vfi_m(x)-\vfi_m(y)|
+|\vfi_m(y)-\vfi(y)|\\
&<&\eps+\frac{1}{2^m}+\eps<3\eps.
\end{eqnarray*}
This proves that $\vfi$ is continuous on $[a,b]$. By definition of 
$(\vfi_n)_{n\ge 1}$, $\vfi(a)=0$ and $\vfi(b)=1$, which implies that
the map $\vfi$ is onto. Moreover, 
$\vfi$ is non decreasing because, if $x\le y$, then
$\vfi_n(x)\le \vfi_n(x)$ for all $n$, which implies 
$\vfi(x)\le \vfi(x)$. 

Now we are going to show that $\vfi$ is increasing.
For every $\bar\alpha=(\alpha_n)_{n\ge 0}\in \{0,1\}^{\IZ^+}$, we set
$$
J_{\bar\alpha}:=\bigcap_{n=1}^{+\infty} J_{\alpha_0\ldots\alpha_{n-1}}.
$$
This is a decreasing intersection of nonempty compact intervals, so
$J_{\bar\alpha}$ is a nonempty compact interval. According to the
definition of $\vfi$, the map $\vfi$ is constant on an interval $J$
if and only if $J\subset J_{\bar\alpha}$ for some
$\bar\alpha\in \{0,1\}^{\IZ^+}$. Moreover,
\eqref{eq:g(J-alpha-n)} implies that
$g^n(J_{\bar\alpha})\subset J_{\alpha_n}$ for all $n\ge 1$. On the other hand,
for every non degenerate interval $J$, there exists $n\ge 1$ such that
$g^n(J)=[a,b]$ because $g$ is topologically mixing and has no 
non accessible endpoint (Proposition~\ref{prop:accessibility}).
This implies that 
$J_{\bar\alpha}$ is degenerate, hence reduced to a single point for every
$\bar\alpha\in \{0,1\}^{\IZ^+}$. The map $\vfi$ is non decreasing and 
it is non constant on any non degenerate interval. Thus $\vfi$ is 
increasing, which implies that $\vfi$ is a homeomorphism.
This concludes the proof of Fact 2.

\medskip
\textsc{Fact 3.} {\it For all $(\alpha_0,\ldots,\alpha_{n-1})\in\{0,1\}^n$
$(n\ge 1)$,
$\vfi(J_{\alpha_0\ldots\alpha_n})=J'_{\alpha_0\ldots\alpha_n}$.}

\medskip
We fix $n\ge 1$ and $(\alpha_0,\ldots,\alpha_{n-1})\in\{0,1\}^n$.
By construction of the sequence of maps $(\vfi_m)_{m\ge 1}$,
$$
\forall m\ge n,\ \vfi_m(R_{\alpha_0\ldots\alpha_{n-1}})=
J'_{\alpha_0\ldots\alpha_n}\setminus \min J'_{\alpha_0\ldots\alpha_n}.
$$
Taking the limit when $m$ tends to infinity, we get
$$
\vfi(R_{\alpha_0\ldots\alpha_{n-1}})=
J'_{\alpha_0\ldots\alpha_n}\setminus \min J'_{\alpha_0\ldots\alpha_n}.
$$
Since $\vfi$ is continuous and increasing by Fact 2, $\vfi$
sends $\inf R_{\alpha_0\ldots\alpha_{n-1}}$ to 
$\min J'_{\alpha_0\ldots\alpha_n}$, and hence
$\vfi(J_{\alpha_0\ldots\alpha_n})=J'_{\alpha_0\ldots\alpha_n}$.

\medskip
It remains to show that $\vfi$ is a conjugacy
between $g$ and $T_2$, that is, $T_2=\vfi\circ g\circ\vfi^{-1}$.
First note that $\vfi\circ g\circ\vfi^{-1}(0)=\vfi\circ g\circ\vfi^{-1}(1)=0$,
$\vfi\circ g\circ\vfi^{-1}(\frac12)=1$ 
and $\vfi\circ g\circ\vfi^{-1}$ is increasing on $[0,\frac12]$ 
(resp. decreasing on $[\frac12,1]$) because $g(a)=g(b)=a$, $g(c)=b$ and
$g$ is increasing on $J_0=[a,c]$ (resp. decreasing on $J_1=[c,b]$).
Let $(\alpha_0,\ldots,\alpha_{n-1})\in\{0,1\}^n$.
Then $J'_{\alpha_0\ldots \alpha_{n-1}}$ is an interval of length
$\frac{1}{2^n}$. Moreover, Fact~3 and \eqref{eq:g(J-alpha-n)}
imply that 
$J'_{\alpha_0\ldots \alpha_{n-1}}=\vfi(J_{\alpha_0\ldots \alpha_{n-1}})$ and
\begin{equation}\label{eq:TA0}
\vfi\circ g\circ \vfi^{-1}(J'_{\alpha_0\ldots \alpha_{n-1}})=\vfi\circ g(J_{\alpha_0\ldots \alpha_{n-1}})
=\vfi(J_{\alpha_1\ldots \alpha_{n-1}})=J'_{\alpha_1\ldots \alpha_{n-1}}.
\end{equation}
Thus
\begin{equation}\label{eq:TA}
|\vfi\circ g\circ \vfi^{-1}(J'_{\alpha_0\ldots \alpha_{n-1}})|=
\frac{1}{2^{n-1}}=2|J'_{\alpha_0\ldots \alpha_{n-1}}|.
\end{equation}
Now we consider a point $x_0\in\End{J'_{\alpha_0\ldots \alpha_{n-1}}}$. If $\alpha_0=0$ (which 
implies that $x_0\in [0,\frac12]$), then $[0,x_0]$ is the union of
intervals of the form $J'_{0\beta_1\ldots\beta_{n-1}}$,
%=\vfi(J_{0\beta_1\ldots\beta_{n-1}})$, 
and the length of $[0,x_0]$ is equal to the sum of the lengths of these 
intervals. We have
$\vfi\circ g\circ\vfi^{-1}(J'_{0\beta_1\ldots\beta_{n-1}})=
J'_{\beta_1\ldots\beta_{n-1}}$ by \eqref{eq:TA0}
and $|\vfi\circ g\circ \vfi^{-1}(J_{0\beta_1\ldots\beta_{n-1}})|=
2|J_{0\beta_1\ldots\beta_{n-1}}|$ by \eqref{eq:TA}.
Moreover, the intervals
$(J'_{\beta_1\ldots\beta_{n-1}})_{(\beta_1,\ldots,\beta_{n-1})
\in\{0,1\}^{n-1}}$ have pairwise disjoint interiors. Thus
\begin{equation}\label{eq:length0x_0}
|\vfi\circ g\circ \vfi^{-1}([0,x_0])|=2|[0,x_0]|=2x_0.
\end{equation}
Since $\vfi\circ g\circ \vfi^{-1}$ fixes the point $0$ and is increasing on 
$[0,\frac12]\supset [0,x_0]$, we have $\vfi\circ g\circ \vfi^{-1}([0,x_0])
=[0,\vfi\circ g\circ \vfi^{-1}(x_0)]$. Combined with 
\eqref{eq:length0x_0}, 
this implies that 
\begin{equation}\label{eq:xinJ0}
\forall x_0\in\End{J'_{0\alpha_1\ldots \alpha_{n-1}}},\quad
\vfi\circ g\circ \vfi^{-1}(x_0)=2x_0=T_2(x_0).
\end{equation}
If $\alpha_0=1$ (which implies that $x_0\in [\frac12,1]$), one can show with
similar arguments that 
$|\vfi\circ g\circ \vfi^{-1}([x_0,1])|=2|[x_0,1]|=2(1-x_0)$, and thus
\begin{equation}\label{eq:xinJ1}
\forall x_0\in\End{J'_{1\alpha_1\ldots \alpha_{n-1}}},\quad
\vfi\circ g\circ \vfi^{-1}(x_0)=T_2(x_0).
\end{equation}
The set 
$$
\left\{\End{J'_{\alpha_0\ldots,\alpha_{n-1}}}\mid n\ge 1,(\alpha_0,\ldots,\alpha_{n-1})\in
\{0,1\}^n\right\}
=\left\{\frac{i}{2^n}\mid
n\ge 0, i\in\Lbrack 0,2^n\Rbrack\right\}
$$
is dense is $[0,1]$. Since $\vfi$ is continuous according to Fact~2,
\eqref{eq:xinJ0} and \eqref{eq:xinJ1} imply that
$$
\forall x\in [0,1],\ \vfi\circ g\circ \vfi^{-1}(x)=T_2(x),
$$
that is, $g$ and $T_2$ are conjugate by $\vfi$.
\end{proof}

%*************************
\subsection*{Remarks on graph maps}

There exist results similar to 
Proposition~\ref{prop:transitivity-lower-bound-htop}
for circle and tree maps. For transitive 
circle maps, the lower bound on the entropy depends on the degree,
the interesting cases being the degrees $-1,0,1$. Indeed,
we saw that, if $f$ is a circle map of degree $d$ with $|d|\ge 2$, then
$h_{top}(f)\ge \log|d|$, regardless of whether $f$ is transitive or not 
(Proposition~\ref{prop:h-degreed}). Moreover,
for every integer $d\in\IZ\setminus\{-1,0,1\}$, 
there exist transitive circle maps of degree $d$ realizing the equality,
for example, the map $\IS\to \IS$, $x\mapsto dx\bmod 1$.
The cases of transitive circle maps of degree $0$ or $-1$ were dealt with by
Alsedà, Kolyada, Llibre and Snoha \cite{AKLS2}. Notice that
transitive circle maps of degree $0$ are very similar to 
transitive interval maps with two fixed points. In particular, 
the map $T_2$ in Example~\ref{ex:htop-transitive} can be seen as a
circle map by identifying the two endpoints of the interval.

\begin{prop}
Let $f\colon \IS\to \IS$ be a transitive circle map of degree $d$. 
\begin{itemize}
\item If $d=0$, then $h_{top}(f)\ge \log 2$.
\item If $d=-1$, then $h_{top}(f)\ge \frac{\log 3}{2}$.
\end{itemize}
Moreover, there exist transitive circle maps with the prescribed
degree realizing the equalities.
\end{prop}

Irrational rotations provide examples of degree $1$ circle maps that
are transitive and have a null entropy. This is actually the only
possibility, up to conjugacy. More generally, Blokh proved that
a transitive graph map has positive entropy except if it
is conjugate to an irrational rotation \cite{Blo14}. Recall
that a transitive graph map with no periodic point is conjugate
to an irrational rotation on the circle
(Theorem~\ref{theo:transitivegraphmap-rotation}).

\begin{theo}
Let $f\colon \IS\to \IS$ be a transitive graph map. If $f$ has periodic 
points, then $h_{top}(f)>0$.
\end{theo}

For circle maps, this is the best possible lower bound: there exist 
transitive degree $1$ circle maps with  arbitrarily small positive
topological entropy. This is a folklore result; see \cite{AKLS} for a proof.

For tree maps, a lower bound depending on the number of endpoints was
found by Alsedà, Baldwin, Llibre and Misiurewicz \cite{ABLM}.

\begin{prop}
Let $f\colon T\to T$ be a transitive tree map. Let $e(T)$ denote the number
of endpoints of $T$. Then $h_{top}(f)\ge\frac{\log 2}{e(T)}$.
\end{prop}

This is the best lower bound for star maps. However, this is not the case
in general. 
The next proposition, due to Alsedà, Kolyada, Llibre and Snoha \cite{AKLS},
states more specific bounds in the case of star maps.

\begin{prop}
Let $f\colon S_n\to S_n$ be a transitive map, where $S_n$ is an $n$-star, 
$n\ge 2$. Let $b$ denote the unique
branching point of $S_n$. 
\begin{itemize}
\item If $f(b)=b$, then $h_{top}(f)\ge\frac{\log 2}n$. Moreover, equality is
possible.
\item If $f(b)\neq b$, then  $h_{top}(f)\ge \frac{\log 2}2$ (it is not
known whether this is the best lower bound).
\end{itemize}
\end{prop}

Proposition~\ref{prop:h=hfP} holds for tree maps, with only obvious changes
in its proof: if $f\colon T\to T$ is a transitive tree map such that
$h_{top}(f)=h_{top}(f_P)$, where $P$ is a finite invariant set containing 
all the branching points of $T$, then $f$ is $P$-monotone.
However its interest is limited since the assumption implies that every
branching point has a finite orbit under $f$.

%************************************************************************
\section{Uniformly positive entropy}

The following notion was introduced by Blanchard \cite{Bla3}, by analogy
with $K$-systems in ergodic theory.

\begin{defi}[uniformly positive entropy]\index{uniformly positive entropy}\index{upe}
A topological dynamical system $(X,f)$ has \emph{uniformly positive entropy
(upe)} if every open cover of $X$ by two non dense open sets has a positive 
topological entropy.
\end{defi}

A topologically mixing interval map has positive entropy by 
Proposition~\ref{prop:transitivity-lower-bound-htop}. The next theorem states that
it has the stronger property of uniformly positive entropy.

\begin{theo}\label{theo:upe}
Let $f\colon I\to I$ be an interval map. The following assertions are
equivalent:
\begin{enumerate}
\item $f$ is topologically mixing,
\item $f$ has uniformly positive entropy.
\end{enumerate}
\end{theo}

\begin{proof}
We first assume that $f$ is topologically mixing.
Let $\CU=(U_0,U_1)$ be an open cover of $I$ such that 
$U_0,U_1$ are not dense. Since $U_0\setminus\overline{U_1}$ is a nonempty
open set, there is a non degenerate closed interval
$I_0\subset U_0\setminus\overline{U_1}$ such that $I_0$ does not
contains any endpoint of $I$. Similarly, there is a non degenerate
closed interval $I_1\subset U_1\setminus\overline{U_0}$ containing
no endpoint of $I$.
According to Theorem~\ref{theo:summary-mixing},
there exists an integer $k>0$ such that $f^k(I_0)\cap f^k(I_1)
\supset I_0\cup I_1$. We set $g:=f^k$. Let $n\ge 1$
and $(\eps_0,\ldots,\eps_{n-1})\in \{0,1\}^n$. By 
Lemma~\ref{lem:chain-of-intervals}(i),
there exists a non degenerate closed interval $J$ such that
$g^i(J)\subset J_{\eps_i}$ for all $i\in\Lbrack 0,n-1\Rbrack$. Consequently, 
for every $n$-tuple $(\eps_0,\ldots, \eps_{n-1})\in \{0,1\}^n$,
the set
$$I_{\eps_0}\cap g^{-1}(J_{\eps_1})\cap\cdots\cap 
g^{-(n-1)}(J_{\eps_{n-1}})$$
is nonempty and
$$
U_{\eps_0}\cap g^{-1}(U_{\eps_1})\cap\cdots\cap g^{-(n-1)}(U_{\eps_{n-1}})$$
is the unique element of 
$\CU\vee g^{-1}(\CU)\vee\cdots\vee g^{-(n-1)}(\CU)$ that meets (actually
contains) this set.
This implies that $N_n(\CU,g)\ge 2^n$ for all $n\ge 1$, so
$h_{top}(\CU,g)\ge \log 2$. Finally, we have
$$
h_{top}(\CU,f)=\frac{1}{k}h_{top}(\CU\vee f^{-1}(\CU)\vee\cdots\vee
f^{-(k-1)}(\CU),g)\ge \frac{1}{k}h_{top}(\CU,g)\ge \frac{\log 2}{k}>0.
$$
This proves (i)$\Rightarrow$(ii).

Now we are going to show that, if $f$ is not topologically mixing, then it 
does not have
uniformly positive entropy. This will prove (ii)$\Rightarrow$(i) by
refutation.

Suppose that $f$ is not transitive. This means that there exist two non 
degenerate closed intervals $I_0,I_1$ such that
\begin{equation}\label{eq:nontransitive}
\forall n\ge 0,\ f^{-n}(I_1)\cap I_0=\emptyset.
\end{equation}
We  set $U_i :=I\setminus I_i$ for $i\in\{0,1\}$. 
Then $\CU:=(U_0,U_1)$ is an open cover
of $I$ by two non dense sets. We see that \eqref{eq:nontransitive}
implies  
\begin{equation}\label{eq:I_0}
\forall n\ge 0,\ I_0\subset f^{-n}(U_1).
\end{equation}
Let $x\in I$ and $n\ge 0$. If $f^i(x)\notin I_0$
for any $i\in\Lbrack 0,n-1\Rbrack$, then $x\in \bigcap_{i=0}^{n-1}f^{-i}(U_0)$.
Otherwise, let $k$ be the minimal non negative integer such that
$f^k(x)\in I_0$. By \eqref{eq:I_0}, we have
$f^k(x)\in\bigcap_{i\ge 0}f^{-i}(U_1)$, and thus
$$
x\in \bigcap_{i=0}^{k-1}f^{-i}(U_0)\cap
\bigcap_{i=k}^{n-1}f^{-i}(U_1).
$$
This implies that $N_n(\CU,f)\le n+1$ for all $n\ge 0$. 
We deduce that $h_{top}(\CU,f)=0$, so $f$ does not have uniformly 
positive entropy.

Suppose now that $f$ is transitive but not topologically mixing. Then, by
Theorem~\ref{theo:summary-transitivity}, there exist
two non degenerate closed intervals $J,K$ with disjoint interiors such that 
$I=J\cup K$, $f(J)=K$ and $f(K)=J$. We  choose two non dense open
sets $U_0, U_1$ such that $J\subset U_0$ and $K\subset U_1$ and we set
$\CU:=(U_0,U_1)$. For all $n\ge 0$, $f^{2n}(J)\subset U_0$
and $f^{2n+1}(J)\subset U_1$. Similarly, for all $n\ge 0$, 
$f^{2n}(K)\subset U_1$
and $f^{2n+1}(K)\subset U_0$.
This implies that $I$ is covered by the two sets 
$$
\bigcap_{i= 0}^{+\infty} f^{-2i}(U_0)\cap\bigcap_{i\ge 0}f^{-(2i+1)}(U_1)\quad\text{and}\quad
\bigcap_{i= 0}^{+\infty} f^{-2i}(U_1)\cap\bigcap_{i\ge 0}f^{-(2i+1)}(U_0).
$$
This means that $N_n(\CU,f)\le 2$ for all $n\ge 1$. Thus,
$h_{top}(\CU,f)=0$, and $f$ does not have uniform positive entropy.
This concludes the proof.
\end{proof}

\begin{rem}\label{rem:specif}
Theorem~\ref{theo:upe} above can be seen as a consequence of results about
general dynamical systems. Indeed, it is proved in \cite{Bla3} that
a topological dynamical system with the
specification property has uniformly positive entropy, and
uniformly positive entropy implies topological weak mixing.
For interval maps, topological mixing implies the specification
property (Theorem~\ref{theo:mixing-specification}), and 
topological weak mixing is equivalent to topological mixing
(Theorem~\ref{theo:summary-mixing}). This implies that an interval map
is topologically mixing if and only if it has uniformly positive entropy.
\end{rem}

\subsection*{Remarks on graph maps}
Theorem~\ref{theo:upe} is valid for graph maps in view of 
Remark~\ref{rem:specif} and Theorem~\ref{theo:mixing-specificationG}.

%**********************************************************************
%Chaos in the sense of Li-Yorke, scrambled sets
\chapter{Chaos in the sense of Li-Yorke, scrambled sets}\label{chap5}

%**********************************************************************
\section{Definitions}

In \cite{LY}, Li and Yorke showed that, if an interval map $f$ has a periodic
point of period 3, there exists an uncountable set $S$ such that, for
all distinct points $x,y$ in~$S$,
\begin{gather*}
\limsup_{n\to +\infty}|f^n (x)- f^n (y)|>0,\quad
\liminf_{n\to +\infty}|f^n (x)-f^n (y)|=0,\\
\text{and }\forall z\text{ periodic point},\ \limsup_{n\to +\infty}|f^n (x)-f^n (z)|>0.
\end{gather*}
They called this behavior \emph{chaotic}, without formally defining what 
chaos is. This leads to the following definitions.

\begin{defi}[Li-Yorke pair, scrambled set, Li-Yorke chaos]\index{scrambled set}
\index{Li-Yorke pair}\index{chaos!chaos in the sense of Li-Yorke}\index{chaotic|see{chaos}}
Let $(X,f)$ be a topological dynamical system, $x,y\in X$ and $\delta>0$.
The pair $(x,y)$ is a \emph{Li-Yorke pair of modulus
$\delta$} if
\begin{equation}\label{eq:LY1}
\limsup_{n\to +\infty}d(f^n (x), f^n (y))\geq \delta\quad\text{and}\quad
\liminf_{n\to +\infty}d(f^n (x), f^n (y))=0,
\end{equation}
and $(x,y)$ is a \emph{Li-Yorke pair} if 
it is a Li-Yorke pair of modulus $\delta$ for some $\delta>0$.
A set $S\subset X$ is a \emph{scrambled} (resp. \emph{$\delta$-scrambled}) set 
if, for all distinct points $x,y$ in $S$, $(x,y)$ is
a Li-Yorke pair (resp. a Li-Yorke pair of modulus $\delta$).

The topological dynamical system  $(X,f)$ is \emph{chaotic
in the sense of Li-Yorke} if there exists an uncountable scrambled set.
\end{defi}

The next proposition is straightforward (the second assertion uses the fact that
$f$ is uniformly continuous because $X$ is compact).

\begin{prop}\label{prop:scrambled-fn}
Let $(X,f)$ be a topological dynamical system, $S\subset X$ and $\delta>0$. 
\begin{itemize}
\item
If $S$ is a scrambled (resp. $\delta$-scrambled) set for $f^n$, then it is also
a scrambled (resp. $\delta$-scrambled) set for $f$.  
\item
If $S$ is a scrambled (resp. $\delta$-scrambled) set for $f$, then it is also
a scrambled (resp. $\delta'$-scrambled for some $\delta'>0$) set for $f^n$.
\end{itemize}
\end{prop}

\begin{rem}
The definition of a scrambled set is not unified in the
literature.  In particular, in the spirit of the properties exhibited 
by Li and Yorke, some people say that $S$ is a scrambled set if, for all
distinct points $x,y$ in $S$,
\begin{gather}
\limsup_{n\to +\infty}d(f^n (x), f^n (y))>0,\quad
\liminf_{n\to +\infty}d(f^n (x), f^n (y))=0,\label{eq:LY1bis}\\
\forall z\text{ periodic point},\ \limsup_{n\to +\infty}d(f^n (x),f^n (z))>0,\label{eq:LY2}
\end{gather}
and the same properties with ``$\ge\delta$'' instead of ``$>0$''
in \eqref{eq:LY1bis} and \eqref{eq:LY2}
for $\delta$-scrambled sets.
Actually, it makes no difference for chaos in the sense of Li-Yorke, nor for 
existence of an uncountable $\delta$-scrambled set for some $\delta>0$. 
More precisely,
if $S$ is a scrambled set, then all but at most one point of $S$ satisfy
\eqref{eq:LY2}, and if $S$ is an uncountable $\delta$-scrambled set, 
then there exists
an uncountable set $S'$ included in $S$ such that, for all $x\in S'$ and all
periodic points $z$, $\limsup_{n\to+\infty}d(f^n (x),f^n (z))\ge\delta/2$.
These results are consequences of Lemmas \ref{lem:approx-periodic}
and \ref{lem:JL2} below; they were first
noticed by Jiménez López \cite[p 117--118]{Jim4}, 
\cite[Proposition~1.2.2]{Jim3}.
\end{rem}

\begin{defi}
Let $(X,f)$ be a topological dynamical system. A point $x$ is 
\emph{approximately periodic}\index{approximately periodic point} 
if, for all $\eps>0$, there exists a periodic point $z$ such that 
$\displaystyle\limsup_{n\to+\infty}d(f^n(x),f^n(z))\le \eps$.
\end{defi}

\begin{lem}\label{lem:approx-periodic}
Let $(X,f)$ be a topological dynamical system and $x,x'\in X$. Suppose that
$x$ and $x'$ are approximately periodic. Then
$$
\text{either}\quad\lim_{n\to+\infty}d(f^n(x),f^n(x'))=0\quad
\text{or}\quad\liminf_{n\to+\infty}d(f^n(x),f^n(x'))>0.
$$
In particular, if $S$ is a scrambled set, then 
$S$ contains at most one approximately periodic point.
\end{lem}

\begin{proof}
Suppose that 
\begin{equation}\label{eq:liminf}
\liminf_{n\to+\infty}d(f^n(x),f^n(x'))=0.
\end{equation}
Let $\eps>0$.
By definition, there exist two periodic points $z,z'$ and an integer $N$
such that
$$
\forall n\ge N,\ d(f^n(x),f^n(z))\le\eps\quad\text{and}\quad
d(f^n(x'),f^n(z'))\le\eps.
$$
Let $p$ be a multiple of the periods of $z$ and $z'$.
By continuity and \eqref{eq:liminf}, there exists $M\ge N$ such that
$d(f^{M+i}(x),f^{M+i}(x'))\le\eps$ for all $i\in\Lbrack 0,p-1\Rbrack$.
Let $n\ge N$ and let $i\in\Lbrack 0,p-1\Rbrack$ be such that 
$n-M\equiv i\bmod p$. Since $f^n(z)=f^{M+i}(z)$ and
$f^n(z')=f^{m+i}(z')$, we have
\begin{eqnarray*}
d(f^n(x),f^n(x'))&\le& d(f^n(x),f^n(z))+d(f^{M+i}(z),f^{M+i}(x))\\
&&+
d(f^{M+i}(x),f^{M+i}(x'))\\&&+d(f^{M+i}(x'),f^{M+i}(z'))
+d(f^n(z'),f^n(x'))\\
&\le&5\eps.
\end{eqnarray*}
This implies that $\lim_{n\to+\infty}d(f^n(x),f^n(x'))=0$. This proves the first statement of the lemma, which straightforwardly
implies the second one.
\end{proof}

\begin{lem}\label{lem:JL2}
Let $(X,f)$ be a topological dynamical system, $S\subset X$ and $\delta>0$.
Suppose that
$$
\forall x,y\in S,\ x\neq y,\ \limsup_{n\to+\infty}d(f^n(x),f^n(y))\ge \delta.
$$
Then there exists a countable set $C\subset X$ such that, for all
$x\in S\setminus C$ and all periodic points $z\in X$, 
$
\limsup_{n\to+\infty}d(f^n(x),f^n(z))\ge\frac{\delta}2.
$
\end{lem}

\begin{proof}
Let $C$ be the set of points in $S$ such that, 
for all $x\in C$, there exists a periodic point $z_x\in X$
such that
$
\limsup_{n\to+\infty}d(f^n(x),f^n(z_x))<\frac{\delta}2.
$
Suppose that $C$ is uncountable.
Since $C$ is the countable union of the sets
$$
\left\{x\in S\mid \limsup_{n\to+\infty}d(f^n(x),f^n(z_x))
\le\frac{\delta}2-\frac1n\right\},\ n\in\IN,
$$
one of these sets is uncountable. Moreover, 
the set of periods of the points $z_x$ is countable. Therefore, there
exist an uncountable subset $R\subset C$, a number $\eps>0$ and an integer 
$p\ge 1$ such that
$$
\forall x\in R,\ f^p(z_x)=z_x\text{ and }\limsup_{n\to+\infty}d(f^n(x),f^n(z_x))\le\frac{\delta}2-\eps.
$$
Since $X$ is compact, $f$ is uniformly
continuous and there exists $\eta>0$ such that
$$\forall x,y\in X,\ d(x,y)<\eta\Longrightarrow
\forall i\in\Lbrack 0, p-1\Rbrack,\ d(f^i(x),f^i(y))<\eps.
$$
Since $X$ is compact and $R$ is infinite, the family $(z_x)_{x\in R}$ has
a limit point. Thus
there exist two distinct points $x,x'$ in $R$ such that $d(z_x,z_{x'})<\eta$ 
(the case $z_x=z_{x'}$ is possible). Then
$d(f^i(z_x),f^i(z_{x'}))<\eps$ for all $i\in\Lbrack 0,p-1\Rbrack$.
We have
\begin{eqnarray*}
\forall n\ge 0,\ d(f^n(x), f^n(x'))&\le& d(f^n(x), f^n(z_x))+ d(f^n(x'), f^n(z_{x'}))\\
&&+\max_{i\in\Lbrack 0, p-1\Rbrack} d(f^i(z_x),f^i(z_{x'})),
\end{eqnarray*}
so
$$
\limsup_{n\to+\infty}d(f^n(x), f^n(x'))< (\delta/2-\eps) +(\delta/2-\eps)
+\eps<\delta.$$
This contradicts the fact that $x,x'$ are two distinct points in the
set $S$. We conclude that $C$ is countable.
\end{proof}

%*************************************************************************
\section{Weakly mixing maps are Li-Yorke chaotic}

It is easy to see that every topologically weakly mixing dynamical
system $(X,f)$ has a dense $G_\delta$-set of Li-Yorke pairs. Indeed,
every $(x,y)\in X^2$ with a dense orbit is a Li-Yorke pair of modulus
$\delta:=\diam(X)$. Using results of topology (e.g., 
\cite[Theorem~22.V.1]{Kur1}), this implies that the system has
an uncountable $\diam(X)$-scrambled set (called an \emph{extremally 
scrambled set}\index{scrambled set (extremally)}\index{extremally scrambled set} 
\cite{Kan}).
This result was first stated for interval maps by Bruckner and Hu \cite{BH}.
More precisely, they showed that a topologically mixing interval
map $f\colon [0,1]\to [0,1]$ admits 
a dense uncountable scrambled set $S$ such that,
for all distinct points $x,y$ in $S$,  the sequence
$(f^n(x)-f^n(y))_{n\ge 0}$ is dense in $[-1,1]$. 
Then Iwanik proved a stronger result, valid for any topologically weakly mixing 
dynamical system, which implies the existence of
an extremally scrambled set \cite{Iwa, Iwa2}. Iwanik's results rely on
Mycielski's Theorem \cite{Myc}, that we restate under weaker hypotheses
in order not to introduce irrelevant notions. We recall that a 
\emph{perfect set}\index{perfect set} is a nonempty closed set with no 
isolated point; a perfect set is uncountable.  

\begin{theo}[Mycielski]\label{theo:Mycielski}\index{Mycielski's Theorem}
Let $X$ be a complete metric space with no isolated point. For all
integers $n\ge 1$, let $r_n$ be a positive integer and let $G_n$ be a 
dense $G_\delta$-set of $X^{r_n}$ such that 
$$
G_n\cap \{(x_1,\ldots,x_{r_n})\in X^{r_n}\mid 
\exists j,k\in\Lbrack 1, r_n\Rbrack,\  j\neq k,\ x_j=x_k\}=\emptyset.
$$
Let $(U_n)_{n\geq 1}$ be a sequence of nonempty open sets of
$X$. Then there exists a sequence of compact perfect subsets $(K_n)_{n\ge 0}$
with $K_n\subset U_n$ such that,
for all $k\ge 1$ and all distinct points $x_1,\ldots, x_{r_k}$ in 
$\bigcup_{n=1}^{+\infty} K_n$, $(x_1,\ldots,
x_{r_k})\in G_k$.
\end{theo}

\begin{theo}\label{theo:independent-sets}
Let $(X,f)$ be a topological dynamical system. 
If $(X,f)$ is topologically weakly mixing, then there 
exists a dense set $K\subset X$ which is a countable union of
perfect sets and such that, for all $n\geq 1$, for all $k\geq 1$
and all distinct points $x_1,\ldots, x_n$ in $K$, the orbit
$(f^{ik}(x_1),\ldots, f^{ik}(x_n))_{i\geq 0}$ is dense in $X^n$.
\end{theo}

\begin{proof}
Let $n,k$ be positive integers. By Proposition \ref{prop:weakly-mixing-product}
and Theorem~\ref{theo:mixing-weak-mixing}, the 
system $(X^n,f^k\times\cdots\times f^k)$ is transitive.  Let
$G_n^k$ be the set of points of dense orbit in this system.
According to Proposition~\ref{prop:transitive-dense-orbit}(i),
$G_n^k$ is a dense $G_{\delta}$-set and $X$ has no isolated point
(note that $(X,f)$ cannot be weakly mixing if $X$ is finite).
Moreover, if the $n$-tuple 
$(x_1,\ldots,x_n)\in X^n$ has two equal coordinates, then
its orbit is not dense. Finally, the conclusion is given by applying
Mycielski's Theorem~\ref{theo:Mycielski} with the countable family
$(G_n^k)_{n,k\ge 1}$ and 
$(U_i)_{i\geq 0}$ a countable basis of nonempty open sets of $X$.
\end{proof}

\begin{cor}\label{cor:weak-mixing-scrambled}
Let $(X,f)$ be a topological dynamical system. 
If $(X,f)$ is topologically weakly mixing, then there 
exists a dense set $K\subset X$ which is a countable union of
perfect sets and such that, for all distinct points $x,y\in K$ and all
periodic points $z\in X$, 
\begin{gather*}
\limsup_{n\to+\infty}d(f^n(x),f^n(y))=\diam(X),\quad
\liminf_{n\to+\infty}d(f^n(x),f^n(y))=0\\
\text{and }
\limsup_{n\to+\infty}d(f^n(x),f^n(z))\geq \frac{\diam(X)}{2}.
\end{gather*}
In particular, $K$ is a $\delta$-scrambled set for $\delta:=\diam(X)$.
\end{cor}

\begin{proof}
Let $K$ be the set given by Theorem~\ref{theo:independent-sets}. 
By compactness of $X$, there exist $x_0,y_0\in X$ such that 
$d(x_0,y_0)=\diam(X)$. Let $x,y$ be two distinct points in $K$. 
Since the orbit of $(x,y)$ under $f\times f$ is dense in $X^2$,
there exist two increasing sequences of positive integers
$(i_n)_{n\geq 0}$ and $(j_n)_{n\geq 0}$ such that 
$$
\lim_{n\to+\infty}(f^{i_n}(x),f^{i_n}(y))=(x_0,y_0)\quad
\text{and}\lim_{n\to+\infty}(f^{j_n}(x),f^{j_n}(y))=(x_0,x_0).
$$
Thus
$$
\limsup_{n\to+\infty}d(f^n(x),f^n(y))=\diam(X)\quad\text{and}\quad
\liminf_{n\to+\infty}d(f^n(x),f^n(y))=0.
$$
Let $z\in X$ be a periodic point (if any) and let $p$ be its period.
By the triangular inequality,
there exists $z'\in\{x_0,y_0\}$ such that $d(z,z')\geq \frac{\diam(X)}{2}$.
Since $x$ has a dense orbit under $f^p$ by definition of $K$,
there exists an increasing sequence of positive integers
$(k_n)_{n\geq 0}$ such that $f^{pk_n}(x)$ tends to $z'$. Thus
$$
\limsup_{n\to+\infty}d(f^n(x),f^n(z))\geq 
\limsup_{n\to+\infty}d(f^{pk_n}(x),f^{pk_n}(z))=d(z,z')\geq \frac{\diam(X)}{2}.
$$
\end{proof}

\begin{rem}
A set $K$ satisfying the conclusion of Theorem~\ref{theo:independent-sets}
is called \emph{totally independent}\index{totally independent set} 
\cite{Iwa}. If a dynamical
system $(X,f)$ has such a set, then $(X\times X,f\times f)$
has a point of dense orbit, and thus $(X,f)$ is topologically weakly mixing.
Therefore, the existence of a totally independent set is equivalent to 
topological weak mixing. 
\end{rem}

The next proposition, due to Bruckner and Hu \cite{BH}, is in some sense
the converse of Corollary~\ref{cor:weak-mixing-scrambled} for
interval maps.

\begin{prop}
Let $f\colon [0,1]\to [0,1]$ be an interval map. Assume that  there
exists a dense set $S\subset [0,1]$ such that
$$
\forall x,y\in S,\ x\neq y,\
\limsup_{n\to+\infty}|f^n(x)-f^n(y)|=1.
$$
Then $f$ is topologically mixing.
\end{prop}

\begin{proof}
Let $\eps>0$. The assumption implies that $f$ is onto. Thus there
exists $\delta\in(0,\eps)$ such that $f([\delta,1-\delta])\supset 
[\eps,1-\eps]$. Let $J$ be a non degenerate subinterval of $[0,1]$.
Since $S$ is dense, there exist two distinct points $x,y$ in $J\cap S$.
Let $n$ be an integer such that $|f^n(x)-f^n(y)|> 1-\delta$.
Then $f^n(J)\supset [\delta,1-\delta]$, which implies that
both $f^n(J)$ and $f^{n+1}(J)$ contain $[\eps,1-\eps]$
(recall that $\delta<\eps$). Either $n$ or
$n+1$ is even, and thus there exists $m\geq 1$ such that $f^{2m}(J)\supset
[\eps,1-\eps]$. This implies that $f^2$ is transitive, so $f$ is
topologically mixing by Theorem~\ref{theo:summary-mixing}.
\end{proof}

%*******************************************************************
\section{Positive entropy maps are Li-Yorke chaotic}\label{sec:5-htop>0}

The original result of Li and Yorke (period 3 implies chaos in the sense of 
Li-Yorke \cite{LY}) was generalized in several steps.
Nathanson stated the same result for periods which are multiple of 
$3$, $5$ or $7$ \cite{Nat}. Then, simultaneously, Butler and Pianigiani 
\cite{BP} and Oono \cite{Oon} proved that an interval map
$f$ with a periodic point whose period is not a power of $2$ (i.e., the period is
$2^m q$  for some odd $q>1$) is chaotic in the sense of Li-Yorke.  
Actually, this result can be derived from Li-Yorke's result using
Sharkovsky's Theorem, but these authors were not aware of Sharkovsky's 
article. Later, Jankov{á} and Sm{í}tal proved a stronger
result: an interval map with positive entropy (or
equivalently with a periodic point whose period is not a power of
$2$, see Theorem~\ref{theo:htop-power-of-2}) 
admits a perfect $\delta$-scrambled set for some
$\delta>0$ \cite{JS}.  The proof we shall give 
follows the spirit of \cite{JS}, although it is slightly different.

\medskip
Block \cite{Bloc3} showed that, if an interval map $f$ has a strict horseshoe,
then there exists a subsystem which is semi-conjugate to a full 
shift on two letters, and the semi-conjugacy is ``almost'' a 
conjugacy.
This semi-conjugacy with a full shift,
stated in Proposition~\ref{prop:strictly-turbulent-shift} below,
is a key tool in several results.

\begin{rem}
In \cite[Theorem 9]{Moo3}, Moothathu stated that, 
if the entropy of $f$ is positive, 
there exist $n\ge 0$ and an invariant set 
on which the action of $f^{2^n}$ is conjugate to a full shift. 
Having a conjugacy rather than a semi-conjugacy would  make some
arguments easier. Unfortunately, there is something wrong in the proof;
Li, Moothathu and Oprocha built a counter-example \cite{LMO}.
\end{rem}

\begin{defi}
Let $\Sigma:=\{0,1\}^{\IZ^+}$, endowed with the product topology; this is a 
compact metric set. The \emph{shift} map $\sigma\colon
\Sigma\to \Sigma$ is defined by
$\sigma((\alpha_n)_{n\ge 0}):=(\alpha_{n+1})_{n\ge 0}$.
\label{notation:setSigma}
\index{$\alpha$ sa@$\Sigma$}
\label{notation:mapshift}
\index{$\alpha$ sb@$\sigma$}
Then $(\Sigma,\sigma)$ is a topological  dynamical system,
called the \emph{full shift}\index{full shift}\index{shift} on two letters.  
\end{defi}

\begin{lem}\label{lem:fi-omegaset}
Let $(X,f)$ and $(Y,g)$ be two topological dynamical systems, and
let $\vfi\colon X\to Y$ be a semi-conjugacy. For every $x\in X$,
$\vfi(\omega(x,f))=\omega(\vfi(x),g)$.
\end{lem}

\begin{proof}
Let $x\in X$ and $y:=\vfi(x)\in Y$. We are going to show that 
$\vfi(\omega(x,f))\subset \omega(y,g)$ and 
$\omega(y,g)\subset \vfi(\omega(x,f))$, which gives the equality
of the two sets.

Let $x'\in\omega(x,f)$. There exists
an increasing sequence of integers $(n_k)_{k\ge 0}$ such that
$\lim_{k\to+\infty}f^{n_k}(x)=x'$. By continuity of $\vfi$,
$\lim_{k\to+\infty}\vfi(f^{n_k}(x))=\vfi(x')$. 
Since $\vfi$ is a semi-conjugacy, $\vfi(f^{n_k}(x))=g^{n_k}(\vfi(x))=
g^{n_k}(y)$. Thus $\vfi(x')\in \omega(y,g)$. This implies that
$\vfi(\omega(x,f))\subset \omega(y,g)$.

Let $y'\in \omega(y,g)$  and let $(n_k)_{k\ge 0}$ be an increasing sequence
of integers such that $\lim_{k\to+\infty}g^{n_k}(y)=y'$. Since $X$
is compact, there exist a subsequence $(n_{k_i})_{i\ge 0}$ and a point
$x'\in X$ such that $\lim_{i\to+\infty}f^{n_{k_i}}(x)=x'$, and hence
$x'\in \omega(x,f)$. Then $\vfi(x')=y'$ because $\vfi$ is continuous.
This implies that $\omega(y,g)\subset \vfi(\omega(x,f))$.
\end{proof}

\begin{prop}\label{prop:strictly-turbulent-shift}
Let $f\colon I\to I$ be an interval map and let $(J_0, J_1)$ be a strict 
horseshoe for $f$.
There exist an invariant Cantor set $X\subset I$ and a continuous map
$\vfi\colon X\to  \Sigma:=\{0,1\}^{\IZ^+}$ such that $\vfi$
is a semi-conjugacy between $(X,f|_X)$ and $(\Sigma,\sigma)$;
the system $(X,f|_X)$ is transitive and
there exists a countable set $E\subset X$ such that $\vfi$ is one-to-one
on $X\setminus E$ and two-to-one on $E$.

Moreover, there exists a family of nonempty closed intervals 
$$
(J_{\alpha_0\ldots\alpha_{n-1}})_{n\ge 1, (\alpha_0,\ldots,\alpha_{n-1})\in\{0,1\}^n}$$
such that, for all $n\ge 1$ and all 
$(\alpha_0,\ldots,\alpha_{n-1})\in\{0,1\}^n$,
\begin{gather}
J_{\alpha_0\ldots \alpha_{n-1}}\cap J_{\beta_0\ldots \beta_{n-1}}=\emptyset\;
\text{ if }\;(\alpha_0,\ldots,\alpha_{n-1})\ne(\beta_0,\ldots,\beta_{n-1}),\label{eq:v2-semicong1}\\
J_{\alpha_0\ldots \alpha_{n-1}}\subset 
J_{\alpha_0\ldots \alpha_{n-2}}\text{ if }n\ge 2,\label{eq:v2-semicong2}\\
f(J_{\alpha_0\ldots\alpha_{n-1}})=J_{\alpha_1\ldots\alpha_{n-1}}
\text{ and }f(\End{J_{\alpha_0\ldots\alpha_{n-1}}})=
\End{J_{\alpha_1\ldots\alpha_{n-1}}}
\text{ if }n\ge 2,\label{eq:v2-semicong3}\\
\{x\in X\mid  \vfi(x)\text{ begins with
}\alpha_0\ldots\alpha_{n-1}\}=X\cap J_{\alpha_0\ldots\alpha_{n-1}},\label{eq:v2-semicong4}
\end{gather}
and, for all $(\alpha_n)_{n\ge 0}\in \Sigma\setminus\vfi(E)$,
\begin{equation}\label{eq:v2-semicong5}
\lim_{n\to+\infty}|J_{\alpha_0\ldots\alpha_{n-1}}|=0.
\end{equation}
\end{prop}

\begin{proof}
First, we show by induction on $n$ that
there exists a family of nonempty closed intervals satisfying
\eqref{eq:v2-semicong1}, \eqref{eq:v2-semicong2} and \eqref{eq:v2-semicong3}.

$\bullet$ 
For $n=1$, the intervals $J_0, J_1$ satisfy \eqref{eq:v2-semicong1} and there is nothing more to prove.

$\bullet$ Suppose that \eqref{eq:v2-semicong1}, \eqref{eq:v2-semicong2}, \eqref{eq:v2-semicong3} are satisfied for some
$n\ge 1$. Fix $(\alpha_0,\ldots,\alpha_n)$ in $\{0,1\}^{n+1}$.
If $n=1$, we apply Lemma~1.13(i) to the
chain of intervals $(J_{\alpha_0},J_{\alpha_1})$ and we obtain
a closed interval $J_{\alpha_0\alpha_1}$ with
$J_{\alpha_0\alpha_1}\subset J_{\alpha_0}$, $f(J_{\alpha_0\alpha_1})=
J_{\alpha_1}$ and $f(\End{J_{\alpha_0\alpha_1}})=
\End{J_{\alpha_1}}$. If $n\ge 2$, $f(J_{\alpha_0\ldots\alpha_{n-1}})=
J_{\alpha_1\ldots\alpha_{n-1}}$ and
$J_{\alpha_1\ldots\alpha_n}\subset J_{\alpha_1\ldots\alpha_{n-1}}$
by the induction hypothesis. Thus we can apply Lemma~1.13(i) to the
chain of intervals $(J_{\alpha_0\ldots\alpha_{n-1}},J_{\alpha_1\ldots \alpha_n})$ and we obtain
a closed interval $J_{\alpha_0\ldots\alpha_n}$ with
$J_{\alpha_0\ldots\alpha_n}\subset J_{\alpha_0\ldots \alpha_{n-1}}$, 
$f(J_{\alpha_0\ldots\alpha_n})=J_{\alpha_1\ldots\alpha_n}$ and
$f(\End{J_{\alpha_0\ldots\alpha_n}})=\End{J_{\alpha_1\ldots\alpha_n}}$.
In both cases, we get \eqref{eq:v2-semicong2} and \eqref{eq:v2-semicong3} for $n+1$. 

Let $(\alpha_0,\ldots,\alpha_n)$ and $(\beta_0,\ldots,\beta_n)$ be two distinct
elements of $\{0,1\}^{n+1}$. If 
$(\alpha_0,\ldots,\alpha_{n-1})\ne(\beta_0,\ldots,\beta_{n-1})$, then
$J_{\alpha_0\ldots \alpha_{n-1}}\cap J_{\beta_0\ldots \beta_{n-1}}=\emptyset$
by the induction hypothesis, which implies that 
$J_{\alpha_0\ldots \alpha_n}\cap J_{\beta_0\ldots \beta_n}=\emptyset$ because
of \eqref{eq:v2-semicong2}. 
If $(\alpha_0,\ldots,\alpha_{n-1})=(\beta_0,\ldots,\beta_{n-1})$,
then $\alpha_n\ne\beta_n$, and hence
$J_{\alpha_1\ldots \alpha_n}\cap J_{\beta_1\ldots \beta_n}=\emptyset$ by
the induction hypothesis. According to \eqref{eq:v2-semicong3}, 
$$
f(J_{\alpha_0\ldots \alpha_n})\cap f(J_{\beta_0\ldots \beta_n})=
J_{\alpha_1\ldots \alpha_n}\cap J_{\beta_1\ldots \beta_n}=\emptyset,
$$
which implies that  
$J_{\alpha_0\ldots \alpha_n}\cap J_{\beta_0\ldots \beta_n}=\emptyset$.
This proves \eqref{eq:v2-semicong1} for $n+1$ and this ends the induction.

\medskip
For every $\bar\alpha=(\alpha_n)_{n\ge 0}\in \Sigma$, we set
$$
J_{\bar\alpha}:=\bigcap_{n=1}^{+\infty} J_{\alpha_0\ldots\alpha_{n-1}}.
$$
This is a decreasing intersection of nonempty compact intervals, and thus
$J_{\bar\alpha}$ is a nonempty compact interval. 
According to \eqref{eq:v2-semicong1}, we have
\begin{equation}\label{eq:v2-alpha-beta-infty}
\forall \bar\alpha,\bar\beta\in\Sigma,\ \bar\alpha\neq\bar\beta\Longrightarrow
J_{\bar\alpha}\cap J_{\bar\beta}=\emptyset.
\end{equation}
We set
$$
Y_0:=\bigcap_{n=1}^{+\infty}\bigcup_{\doubleindice{\alpha_i\in\{0,1\}}{i\in
\Lbrack 0, n-1\Rbrack}}J_{\alpha_0\ldots\alpha_{n-1}}\quad\text{and}\quad
Y:=Y_0\setminus
\bigcup_{\bar\alpha\in\Sigma}{\rm Int}(J_{\bar\alpha}).
$$
The sets $Y_0$ and $Y$ are compact.
One can see that $Y_0=\bigcup_{\bar\alpha\in\Sigma}J_{\bar\alpha}$, and
$(J_{\bar\alpha})_{\bar\alpha\in\Sigma}$ are the connected components
of $Y_0$. This implies that $Y$ is totally disconnected and 
$Y=\bigcup_{\bar\alpha\in\Sigma}\End{J_{\bar\alpha}}$, 
which is a disjoint union 
by \eqref{eq:v2-alpha-beta-infty}. 

Recall that $\sigma((\alpha_n)_{n\ge 0})=(\alpha_{n+1})_{n\ge 0}$.
We are going to show that
\begin{equation}\label{eq:f(partialJalpha)}
\forall \bar\alpha\in\Sigma,\ 
f(\End{J_{\bar\alpha}})=\End{J_{\sigma(\bar\alpha)}},
\end{equation}
which implies that $f(Y)= Y$.

Notice that if $x\in \End{J_{\bar\alpha}}$, then 
\begin{itemize}
\item either $x=\min J_{\bar\alpha}$ and $x=\lim_{n\to +\infty} \min 
J_{\alpha_0\ldots\alpha_{n-1}}$, 
\item or $x=\max J_{\bar\alpha}$ and $x=\lim_{n\to +\infty} \max
J_{\alpha_0\ldots\alpha_{n-1}}$, 
\end{itemize}
because $J_{\bar\alpha}=\bigcap_{n=1}^{+\infty}J_{\alpha_0\ldots J_{n-1}}$
is a decreasing intersection of nonempty compact intervals.

Let $\bar\alpha\in\Sigma$ and $x\in \End{J_{\bar\alpha}}$. What precedes
shows that there exists a sequence of points $(x_n)_{n\ge 1}$ such that 
$x=\lim_{n\to+\infty} x_n$ and
$x_n\in\End{J_{\alpha_0\ldots\alpha_{n-1}}}$ for all $n\ge 1$.
Then $f(x_n)\in \End{J_{\alpha_1\ldots\alpha_{n-1}}}$ by
\eqref{eq:v2-semicong3}. Moreover, there is an increasing sequence
$(n_k)_{k\ge 0}$ such that,
either $f(x_{n_k})=\min J_{\alpha_1\ldots\alpha_{n_k-1}}$ for all $k\ge 0$,
or $f(x_{n_k})=\max J_{\alpha_1\ldots\alpha_{n_k-1}}$ for all $k\ge 0$.
Thus $\lim_{k\to+\infty} x_{n_k}$ is equal to $\min J_{\sigma(\bar\alpha)}$
or $\max J_{\sigma(\bar\alpha)}$, because 
$J_{\sigma(\bar\alpha)}=\bigcap_{n=1}^{+\infty}J_{\alpha_1\ldots J_n}$
is a decreasing intersection of nonempty compact intervals.
Since $f$ is continuous, $f(x)=\lim_{k\to+\infty}f(x_{n_k})$, and thus
$f(x)\in \End{J_{\sigma(\bar\alpha)}}$. This shows that
$f(\End{J_{\bar\alpha}})\subset \End{J_{\sigma(\bar\alpha)}}$.

Let $y\in \End{J_{\sigma(\bar\alpha)}}$. As above, there exists a
sequence of points $(y_n)_{n\ge 1}$ such that $y=\lim_{n\to+\infty} y_n$ and
$y_n\in\End{J_{\alpha_1\ldots\alpha_n}}$ for all $n\ge 1$. By
\eqref{eq:v2-semicong3}, there exists $x_n\in
\End{J_{\alpha_0\ldots\alpha_n}}$ such that $f(x_n)=y_n$ for all $n\ge 1$.
Then the same argument as above shows that there is an 
increasing sequence $(n_k)_{k\ge 0}$ such that
$\lim_{k\to+\infty} x_{n_k}=x$ with $x\in \End{J_{\bar\alpha}}$. Thus
$f(x)=\lim_{k\to+\infty} f(x_{n_k})=\lim_{k\to+\infty} y_{n_k}=y$.
This shows that
$\End{J_{\sigma(\bar\alpha)}}\subset f(\End{J_{\bar\alpha}})$, and this
ends the proof of \eqref{eq:f(partialJalpha)}.

\medskip
We define the map
$$
\vfi\colon\begin{array}[t]{ccl} Y&\longrightarrow& \Sigma\\
x&\longmapsto& \bar\alpha \text{ if }x\in J_{\bar\alpha}
\end{array}
$$
The map $\vfi$ is trivially onto according to the definition of $Y$.
Let $\CA$ be the collection of $\bar\alpha\in\Sigma$ such
that $J_{\bar\alpha}$ is a non degenerate interval. Then $\CA$
is countable, and 
$$
F:=\bigcup_{\bar\alpha \in \CA}\End{J_{\bar\alpha}}=\vfi^{-1}(\CA)
$$
is a countable subset of $Y$. 
It is clear that, if $\bar\alpha\not\in \CA$, then there is
a single point $x\in Y$ (with $J_{\bar\alpha}=\{x\}$) such that $\vfi(x)=\bar\alpha$;
and if $\bar\alpha\in \CA$, there are exactly two distinct points
$x_1,x_2\in Y$ (with $\End{J_{\bar\alpha}}=\{x_1,x_2\}$) such that
$\vfi(x_i)=\bar\alpha$. This shows that $\vfi$ is one-to-one
on $Y\setminus F$ and two-to-one on $F$. 

Let $\delta_n$ be
the minimal distance between two distinct intervals among
$$
(J_{\alpha_0\ldots\alpha_{n-1}})_{(\alpha_0,\ldots, \alpha_{n-1})
\in \{0,1\}^n}.
$$ 
Then $\delta_n>0$ because these
intervals are compact and pairwise disjoint by \eqref{eq:v2-semicong1}.
Let $x,y\in Y$. If $|x-y|<\delta_n$, then $x$ and $y$ are in the same
interval $J_{\alpha_0\ldots\alpha_{n-1}}$ for some
$(\alpha_0,\ldots,\alpha_{n-1})\in\{0,1\}^n$. This means that
both $\vfi(x)$ and $\vfi(y)$ begin with $\alpha_0\ldots\alpha_{n-1}$,
which implies that $\vfi$ is continuous. Moreover, 
\eqref{eq:f(partialJalpha)} implies that $\vfi\circ f(x)=\sigma\circ\vfi(x)$
for all $x\in Y$, that is, $\vfi$ is a semi-conjugacy.

\medskip
It is easy to see that $(\Sigma,\sigma)$ is transitive and that $\Sigma$
is uncountable. By
Proposition~\ref{prop:transitive-dense-orbit}(i),
there exists a dense $G_\delta$-set of elements $\bar\alpha\in\Sigma$ such that
$\omega(\bar\alpha,\sigma)=\Sigma$. Thus there exists 
$\bar\alpha\in\Sigma\setminus\CA$ such that $\omega(\bar\alpha,\sigma)=\Sigma$
because $\CA$ is countable. Let $x_0\in Y$ be the unique point such that
$\vfi(x_0)=\bar\alpha$ and set $X:=\omega(x_0,f)$. The set $X$ is closed and
invariant by Lemma~\ref{lem:omega-set}(i),
and $X\subset Y$. By Lemma~\ref{lem:fi-omegaset}, 
$\vfi(\omega(x_0,f))=\omega(\bar\alpha,\sigma)
=\Sigma$. Thus $\vfi|_X\colon X\to \Sigma$ is onto. This implies that
$\vfi|_X$ is a semi-conjugacy between $(X,f|_X)$ and $(\Sigma,\sigma)$.
Moreover, there exists a countable set $E\subset F$ such that
$\vfi$ is two-to-one on $E$ and one-to-one on $X\setminus E$.
Since $\vfi(X)=\Sigma$ and $\vfi^{-1}(\bar\alpha)=\{x_0\}$, the point
$x_0$ belongs to $X$, and the set $X$ is infinite. 
According to Proposition~\ref{prop:transitive-dense-orbit}, 
the fact that $X=\omega(x_0,f)$ implies that
$(X,f|_X)$ is transitive and $X$ has no isolated point.
Moreover, the set $X$ is totally disconnected because $Y$ is  
totally disconnected, and thus $X$ is a Cantor set.

By definition of $\vfi$, \eqref{eq:v2-semicong4} is satisfied.
Finally, if $\bar\alpha=
(\alpha_n)_{n\ge 0}$ does not belong to $\CA$, the fact
that $J_{\bar\alpha}$ is reduced to a single point implies that
$$
\lim_{n\to+\infty}|J_{\alpha_0\ldots\alpha_{n-1}}|=0,
$$
which is \eqref{eq:v2-semicong5}. This concludes the proof.
\end{proof}

For the following result of topology, one can refer, e.g., to
\cite[Theorem 37.I.3]{Kur1}.

\begin{theo}[Alexandrov-Hausdorff]\label{theo:perfect-set}
Let $X$ be a topological space. Every uncountable Borel set contains a
Cantor set.
\end{theo}

\begin{theo}\label{theo:htop-positive-chaos-LY}
Let $f$ be an interval map. If $h_{top}(f)>0$,
there exists a $\delta$-scrambled Cantor set for some $\delta>0$.
In particular, $f$ is chaotic in the sense of Li-Yorke.
\end{theo}

\begin{proof}
By Theorem~\ref{theo:htop-power-of-2}, 
there exists an integer $p$ such that
$f^p$ has a strict horseshoe  $(J_0,J_1)$. Let $\delta>0$ be the distance
between  $J_0$ and $J_1$ and $g:=f^p$. Let $X, E$  and $\vfi\colon X\to \Sigma$ 
be given by Proposition~\ref{prop:strictly-turbulent-shift}
for the map $g$. We fix an
element $\bar\omega=(\omega_n)_{n\ge 0}$ in $\Sigma\setminus\vfi(E)$.
We define $\psi\colon \Sigma\to \Sigma$ by
$$
\psi((\alpha_n)_{n\ge 0}):=
(\omega_0\;\alpha_0\; \omega_0\omega_1\;\alpha_0\alpha_1\ldots
\omega_0\ldots\omega_{n-1}\; \alpha_0\alpha_1\ldots\alpha_{n-1}\ldots).
$$
This map is clearly continuous and one-to-one.
For every $\bar\alpha\in \Sigma$, we choose a point 
$x_{\bar\alpha}$ in $\vfi^{-1}\circ
\psi(\bar\alpha)$ and we set $S:=\{x_{\bar\alpha}\in
X \mid  \bar\alpha\in\Sigma\}$. 
According to Proposition~\ref{prop:strictly-turbulent-shift},
the set $\vfi^{-1}\circ \psi(\bar\alpha)$ contains two points
if $\psi(\bar\alpha)\in \vfi(E)$ and is reduced to a single point
if $\psi(\bar\alpha)\notin \vfi(E)$. Thus there exists a countable set
$F\subset X$ such that  $S=(\vfi^{-1}\circ \psi(\Sigma))\setminus F$.

Let $\bar\alpha,\bar\beta$ be two
distinct elements of $\Sigma$, and let $k\ge 0$ be such that
$\alpha_k\neq\beta_k$. By definition of $\psi$, there exists an increasing
sequence of integers $(n_i)_{i\ge 0}$ such that the $n_i$-th coordinates of
$\psi(\bar\alpha)$ and $\psi(\bar\beta)$ are equal respectively to
$\alpha_k$ and $\beta_k$, and hence are distinct. Then, by
Proposition~\ref{prop:strictly-turbulent-shift},  
either
$g^{n_i}(x_{\bar\alpha})$ belongs to $J_0$ and
$g^{n_i}(x_{\bar\beta})$ belongs to $J_1$, or the converse. In
particular, $|g^{n_i}(x_{\bar\alpha})-g^{n_i}(x_{\bar\beta})|\geq
\delta$. This proves that, for all distinct points $x,x'$ in $S$,
$$
\limsup_{n\to+\infty}|g^{n}(x)-g^{n}(x')|\geq \delta.
$$
According to Proposition~\ref{prop:strictly-turbulent-shift}
and the choice of $\bar\omega$, 
$$
\lim_{n\to+\infty}\diam \{x\in X \mid \vfi(x)\text{ begins with }
\omega_0\ldots\omega_{n-1}\}=0.
$$
By definition of $\psi$, there exists an increasing
sequence of integers $(m_i)_{i\ge 0}$ such that, for every 
$\bar\alpha\in\Sigma$,
$\sigma^{m_i}(\psi(\bar\alpha))$ begins with $(\omega_0\ldots\omega_{i-1})$.
Since $\sigma^{m_i}(\psi(\bar\alpha))=\vfi(g^{m_i}(x_{\bar\alpha}))$, we get
$$
\forall\bar\alpha,\bar\beta\in\Sigma,\ \liminf_{n\to+\infty}|g^{n}(x_{\bar\alpha})-g^{n}(x_{\bar\beta})|=0.
$$
Therefore, $S$ is a $\delta$-scrambled set for $g$. Moreover, by
Theorem~\ref{theo:perfect-set}, 
there exists a Cantor set $K\subset S$
because $S=(\vfi^{-1}\circ \psi(\Sigma))\setminus F$ 
is an uncountable Borel set.
We have proved that $g=f^p$ admits a $\delta$-scrambled Cantor set,
so  $K$ is also a $\delta$-scrambled set for $f$ by
Proposition~\ref{prop:scrambled-fn}.
\end{proof}

\subsection*{Remarks on graph maps and general dynamical systems}

Proposition~\ref{prop:strictly-turbulent-shift} can be generalized to
graph maps (actually to any dynamical system having a horseshoe made of
two intervals).
Theorem~\ref{theo:htop-positive-chaos-LY} remains valid for graph maps,
and the same proof works because, according to 
Theorem~\ref{theo:htop-horseshoeG}, if a graph map $f$ has positive topological entropy,
$f^n$ has a strict horseshoe for some $n$. The proof that a graph map of
positive entropy is chaotic in the sense of Li-Yorke does not appear in
the literature, this result being a consequence of the next
theorem, which is due to Blanchard, Glasner, Kolyada and Maass \cite{BGKM}.

\begin{theo}\label{theo:general-system-htop-positive-chaos-LY}
Let $(X,f)$ be a topological dynamical system with
positive topological entropy. Then it admits a $\delta$-scrambled Cantor set
for some $\delta>0$.
\end{theo}

This theorem is a remarkable 
result and answers a longstanding question. Its proof is much more difficult
than the proof in the interval case. 

\begin{rem}
In \cite{BGKM}, the main theorem
states the existence of a scrambled Cantor set, but the proof 
actually gives a $\delta$-scrambled Cantor set.
\end{rem}

%*********************************************************************
\section{Zero entropy maps}\label{sec:5-htop0}

The converse of Theorem~\ref{theo:htop-positive-chaos-LY} is not true:
there exist zero entropy interval maps that are chaotic in the sense of
Li-Yorke; we shall give an example in Section~\ref{sec:chaos-LY-htop0}.  
A zero entropy map that is chaotic in the sense of
Li-Yorke is sometimes called \emph{weakly chaotic}\index{chaos!weak chaos}
(e.g. in \cite{FSS}).
In \cite{Smi}, Smítal proved that a zero entropy interval map that is
chaotic in the sense of Li-Yorke has a $\delta$-scrambled Cantor set for
some $\delta>0$, as it is the case for positive entropy interval maps.
He also gave a necessary and sufficient condition for a zero entropy
interval map to be chaotic in the sense of Li-Yorke in terms of non separable
points (condition (iv) in Theorem~\ref{theo:htop0-chaos-LY} below).
This condition looks technical, but it can be useful to show that a
map is chaotic in the sense of Li-Yorke or not; in particular, it will be
needed for the examples in Section~\ref{sec:chaos-LY-htop0}.

\begin{defi}[$f$-non separable points]
Let $f$ be an interval map and let $a_0,a_1$ be two distinct
points. The points $a_0, a_1$ are \emph{$f$-separable} if
there exist two disjoint intervals $J_0,J_1$ and two positive integers
$n_0,n_1$ such that $a_i\in J_i$, $f^{n_i}(J_i)=J_i$ and
$(f^k(J_i))_{0\leq k<n_i}$ are disjoint for $i\in\{0,1\}$. Otherwise they are
\emph{$f$-non separable}.\index{separable ($f$- )}\index{non separable ($f$-
)}
\end{defi}

\begin{theo}\label{theo:htop0-chaos-LY}
Let $f$ be an interval map of zero topological entropy. The following
properties are equivalent:
\begin{enumerate}
\item $f$ is chaotic in the sense of Li-Yorke,
\item there exists a $\delta$-scrambled Cantor set
for some $\delta>0$,
\item there exists a point $x$ that is not approximately periodic,
\item there exists an infinite $\omega$-limit set containing 
two $f$-non separable points.
\end{enumerate}
\end{theo}

Before proving this theorem, we
need an important number of intermediate results. Some of them have
an interest on their own.

\medskip
The next lemma is stated in the case of 
zero entropy interval maps in \cite{Smi}. We give a different proof.

\begin{lem}\label{lem:alternating}
Let $f$ be an interval map such that $f^2$ has no
horseshoe.  Let $x_0$ be a point with an infinite orbit,
and $x_n:=f^n(x_0)$ for all $n\ge 1$. Suppose that  there exists
$k_0\geq 2$ such that either $x_{k_0}<x_0<x_1$ or
$x_{k_0}>x_0>x_1$. Then there exist a fixed point $z$ and an integer
$N$ such that
$$
\forall n\geq N,\ x_n>z\Longleftrightarrow x_{n+1}<z.
$$
\end{lem}

\begin{proof}
All the points $(x_n)_{n\ge 0}$ are distinct by assumption.
Let 
$$U:=\{x_n\mid n\ge 0, x_{n+1}>x_n\}\quad\text{and}\quad
D:=\{x_n\mid n\ge 0, x_{n+1}<x_n\}.
$$
The map $f$ has no horseshoe (because $f^2$ has no horseshoe either), 
and thus Lemma~\ref{lem:U<D} 
applies: there exists a fixed point $z$ such that
\begin{equation}\label{eq:AzB}
U< z< D.
\end{equation}
By assumption, there exist couples of integers $(p,k)$ with
$p\ge 0$, $k\ge 2$, such that,
either $x_{p+k}<x_p<x_{p+1}$, or $x_{p+k}>x_p>x_{p+1}$. We choose $(p,k)$
satisfying this property with $k$ minimal. 
We are going to show that $k=2$. Suppose on the contrary that $k\geq 3$. 
We may assume that $x_{p+k}<x_p<x_{p+1}$, the case
with reverse inequalities being symmetric.  The minimality of $k$ 
and the fact that $k\ge 3$ imply that
$x_{p+2}>x_p$ and that we do not have $x_{p+k}<x_{p+1}<x_{p+2}$. Thus
\begin{equation}\label{eq:xp}
x_{p+k}<x_p<x_{p+2}<x_{p+1}.
\end{equation}
Then $x_p<z<x_{p+1}$ by \eqref{eq:AzB}.  Let $q\in
\Lbrack p+1,p+k-1\Rbrack$ be the integer such that
$x_n>z$  for all $n\in\Lbrack p+1, q\Rbrack$ 
and $x_{q+1}<z$.  By \eqref{eq:AzB}, $x_{q+2}>x_{q+1}$ and
$x_{n+1}<x_n$ for all $n\in\Lbrack p+1, q\Rbrack$. If
$x_{q+1}>x_{p+k}$, then  $x_{p+k}<x_{q+1}<x_{q+2}$, which
contradicts the minimality of $k$.  Therefore the points are ordered as follows:
$$
x_{q+1}\leq x_{p+k}<x_p<z<x_q<x_{q-1}<\cdots<x_{p+2}<x_{p+1}.
$$
If $q=p+1$, then $x_{p+2}\le x_{p+k}<x_p<x_{p+1}$, which contradicts the
fact that $x_p<x_{p+2}$ by \eqref{eq:xp}. Thus $q\ge p+2$. Let 
$I_0:=[x_p,x_q]$ and $I_1:=[x_q,x_{p+1}]$. Then
$f(I_0)\supset [x_{q+1},x_{p+1}]\supset I_0\cup I_1$ and
$f(I_1)\supset [x_{q+1},x_{p+2}]\supset I_0$. This implies that 
$(I_0,I_1)$ is a horseshoe for $f^2$, which is a
contradiction.  We deduce that $k=2$, that is, there exists an integer
$p$ such that
\begin{equation}\label{eq:xn-k=2}
\text{either } x_{p+2}<x_p<x_{p+1}\quad\text{or}\quad x_{p+2}>x_p>x_{p+1}.
\end{equation}

\medskip
Now we are going to show that there exists an integer $N$ such that, for all
$i\geq N$, $x_i>z\Leftrightarrow x_{i+1}<z$. Assume that the contrary holds, 
which implies the following by \eqref{eq:AzB}:
\begin{equation}\label{eq:xi<xi+1<z}
\forall n\geq 0,\ \exists\, i\geq n, \text{ either } x_{i}<x_{i+1}<z
\quad\text{or}\quad z>x_{i+1}>x_{i}.
\end{equation}
For every $n\ge 0$, let $i(n)$ be the minimal integer $i$ satisfying this 
property.
Among all integers $p$ satisfying \eqref{eq:xn-k=2}, we choose
$p$ such  that $i(p)-p$ is minimal and we set $i:=i(p)$. 
In \eqref{eq:xn-k=2}, we assume that  $x_{p+2}<x_p<x_{p+1}$, the case with 
reverse inequalities being symmetric. Then $x_p<z<x_{p+1}$ by \eqref{eq:AzB},
which implies that $i\ge p+2$.
First we suppose that $i=p+2$, that is,
$x_{p+2}<x_{p+3}<z$. If we set $J:=[x_{p+2}, x_p]$ and $K:=[x_p,z]$, then
$f(J)\supset [x_{p+3},x_{p+1}]\ni z$, so $f^2(J)\supset
[x_{p+2},z]\supset J\cup K$ and $f^2(K)\supset
[x_{p+2},z]\supset J\cup K$. Therefore $J,K$ form a horseshoe for $f^2$,
which is a contradiction. We deduce that $i\geq p+3$.

If $n$ is in $\Lbrack p+1,i\Rbrack$, then $n$ does not satisfy
\eqref{eq:xn-k=2}, otherwise $i(n)-n=i-n<i-p$, which
contradicts the choice of $p$.  Thus
\begin{equation}\label{eq:xn+1+2}
\forall n\in\Lbrack p+1,i\Rbrack,\  x_n 
\text{ is not between }x_{n+1}\text{ and }x_{n+2}.
\end{equation}
We show by
induction on $n$ that, for all $n\in\Lbrack p+2,i\Rbrack$ 
with $(n-p)$ even, we have:
\begin{equation}\label{eq:x2x4-x1}
\begin{array}{l}
x_{p+2}<x_{p+4}<x_{p+6}<\cdots<x_{n-2}<x_n\qquad\qquad\\
\qquad <z<x_{n-1}<x_{n-3}<\cdots<x_{p+3}<x_{p+1}.
\end{array}
\end{equation}

\par\noindent  $\bullet$ We have seen that $x_{p+2}<z<x_{p+1}$,
which is \eqref{eq:x2x4-x1} for $n=p+2$. 

\par\noindent $\bullet$  Suppose that \eqref{eq:x2x4-x1}
holds for some $n\in\Lbrack p+2,i-2\Rbrack$ with $n-p$ even. Then $x_{n+1}>x_n$ by
\eqref{eq:AzB}, and $x_n>x_{n+2}$ by \eqref{eq:xn+1+2}.
Moreover, $x_{n+1}>z$, otherwise we would have $x_n<x_{n+1}<z$,
which would contradict the minimality of $i-p$.
Furthermore, $x_{n-1}$ cannot be between $x_n$ and $x_{n+1}$ 
by \eqref{eq:xn+1+2}, and $x_{n-1}>z$ according to the induction hypothesis.
In summary, these inequalities give
$x_n<z<x_{n+1}<x_{n-1}$. Similarly, $x_{n+2}<x_{n+1}$ by
\eqref{eq:AzB} and $x_{n+2}<z$ by \eqref{eq:xn+1+2}+\eqref{eq:AzB};
thus $x_n<x_{n+2}<z$. This shows that \eqref{eq:x2x4-x1} holds for $n+2$,
which ends the proof of \eqref{eq:x2x4-x1}.

\medskip Now we show by induction on $n$ that $x_{n}<x_p$ for all 
$n\in\Lbrack p+2,i\Rbrack$ with  $(n-p)$ even.

\noindent $\bullet$  We know that $x_{p+2}<x_p$.

\par\noindent $\bullet$  Suppose that the statement holds for $n$
but not for $n+2$, for some $n\in\Lbrack p+2,i-2\Rbrack$ such that $(n-p)$
is even. This means that $x_n<x_p<x_{n+2}$.
Combining this with \eqref{eq:x2x4-x1}, we get
$x_{p+2}\le x_n<x_p<x_{n+2}<z$. We set  $J:=[x_n,x_p]$ and
$K:=[x_p,x_{n+2}]$.  Then $f^2(J)\supset [x_{p+2},x_{n+2}]\supset J\cup
K$ and $f^2(K)\supset [x_{p+2},x_{n+4}]$. 
\begin{itemize}
\item[--]
If $n+4\leq i$, then $x_{n+4}>x_{n+2}$ by \eqref{eq:x2x4-x1}. 
\item[--]
If $n+4=i+1$, then
$x_{i-1}<z$ by \eqref{eq:x2x4-x1} and $x_i>z$ by minimality
of $i-p$, so $z<x_{i+1}<x_i$ according to the definition of $i$. 
This implies that $x_{n+4}=x_{i+1}>z>x_{n+2}$. 
\item[--]
If $n+4=i+2$, then $x_i<x_{i+1}<z$ by
definition of $i$ and \eqref{eq:x2x4-x1}, and thus
$x_{n+4}=x_{i+2}>x_{i+1}$ by \eqref{eq:AzB} and we get $x_{n+4}>x_i=x_{n+2}$.  
\end{itemize}
In the three cases, we have $x_{n+4}>x_{n+2}$.
This implies $f^2(K)\supset J\cup K$, so
$(J,K)$ is a horseshoe for $f^2$,
which is a contradiction.  We deduce that, if $x_n<x_p$, then
$x_{n+2}<x_p$ too. This completes the induction.

\medskip We end the proof by showing that \eqref{eq:xi<xi+1<z}
is absurd.  We set  $j:=i$ if $i-p$ is even, and $j:=i-1$ if $i-p$ is odd.
Let $J':=[x_j,x_p]$ and $K':=[x_p,z]$ (recall that $x_p<z$). 
If $j=i$, then $x_i<x_{i+1}<z$ (by \eqref{eq:xi<xi+1<z} and \eqref{eq:x2x4-x1}) 
and $f(J')\supset [x_{i+1},x_{p+1}]\ni z$, so $f^2(J')\supset [x_{p+2},z]
\supset J'\cup K'$ by \eqref{eq:x2x4-x1}. If $j=i-1$, then $z<x_{i+1}<x_i$ 
(by \eqref{eq:xi<xi+1<z} and \eqref{eq:x2x4-x1} again) and
$f^2(J')\supset [x_{p+2},x_{i+1}] \supset J'\cup K'$
by \eqref{eq:x2x4-x1}. Moreover,
$f^2(K')\supset [x_{p+2},z]\supset J'\cup K'$. Therefore, $(J',K')$ is a
horseshoe for $f^2$, which is not possible.  The lemma is proved.
\end{proof}

The following result is due to Sharkovsky \cite[Corollary 3]{Sha5}.
The proof we give relies on Lemma~\ref{lem:alternating} above.

\begin{prop}\label{prop:omega-no-periodic-point}
Let $f$ be an interval map of zero topological
entropy and let $x$ be a point. If
$\omega(x,f)$ is infinite, then $\omega(x,f)$ contains no
periodic point.
\end{prop}

\begin{proof}
The point $x$ is not eventually periodic because $\omega(x,f)$
is infinite. Thus all the points $(f^n(x))_{n\ge 0}$ are distinct.
Moreover, 
$$
\omega(x,f)\cap (\min\omega(x,f),\max\omega(x,f))\neq\emptyset
$$ 
because $\omega(x,f)$ is infinite, and thus
there exists an integer $p$ such that
$\min\omega(x,f)<f^p(x)<\max\omega(x,f)$.
If $f^{p+1}(x)>f^p(x)$, then
there exists $q>p$ such that $f^q(x)$ is arbitrarily close to 
$\min\omega(x,f)$, in such a way that
$f^q(x)<f^p(x)<f^{p+1}(x)$. Similarly,  if
$f^{p+1}(x)<f^p(x)$, then there exists $q>p$ such that
$f^q(x)>f^p(x)>f^{p+1}(x)$. Since $f^2$ has no horseshoe  by
Proposition~\ref{prop:horseshoe-htop}, Lemma~\ref{lem:alternating} applies:
there exist a fixed point $z$ and an
integer $N$ such that 
$$
\forall n\ge N,\ f^n(x)<z\Longleftrightarrow f^{n+1}(x)>z.
$$
We define either $y:=f^N(x)$ or $y:=f^{N+1}(x)$ in order to have 
$y<z$, and we set 
$y_n:=f^n(y)$ for all $n\geq 0$. In this way, $\omega(y,f)=\omega(x,f)$
and
\begin{equation}\label{eq:y2i<z<y2i+1}
\forall i\ge0,\ y_{2i}<z<y_{2i+1}.
\end{equation}

First we prove that $\omega(x,f)$ contains no fixed point. Suppose on the 
contrary that there exists an increasing sequence of positive integers
$(n_i)_{i\ge 0}$ such that
$\lim_{i\to+\infty} y_{n_i}=a$ with $f(a)=a$. 
By continuity, $(y_{n_i+1})_{i\ge 0}$ tends to $a$ too.
The set $\{n_i\mid i\geq 0\}$ contains either
infinitely many odd integers or infinitely many even integers, and thus
there exists an increasing sequence $(k_i)_{i\ge 0}$ such that
$a=\lim_{i\to+\infty}y_{2k_i}=\lim_{i\to+\infty}y_{2k_i+1}$.
Then \eqref{eq:y2i<z<y2i+1} implies that $a=z$. Hence $z\in \omega(y,f)$.
Let $g:=f^2$. The map $g^2$ has no horseshoe 
(by Proposition~\ref{prop:horseshoe-htop} again), and $\omega(y,g)$ is
infinite (by Lemma~\ref{lem:omega-set}(vi)). Thus we can apply 
Lemma~\ref{lem:alternating} to $g$ and $y$: there exist a
point $z'$ and an integer $N'$ such that $g(z')=z'$ and 
\begin{equation}\label{eq:gz'}
\forall i\ge 0,\ y_{2N'+4i}<z'<y_{2N'+4i+2}<z
\end{equation}
(the last inequality follows from \eqref{eq:y2i<z<y2i+1}).
Since $z$ is in $\omega(y,g)$, there exists a sequence $(m_i)_{i\ge 0}$ 
such that $g^{N'+m_i}(y)=y_{2N'+2m_i}$ tends to
$z$. By \eqref{eq:gz'}, 
$m_i$ must be odd for all large enough $i$. By continuity,
$g^{N'+m_i+1}(y)$ converges to $z$ too, and at the same time
$g^{N'+m_i+1}(y)=y_{2N'+2m_i+2}<z'<z$, which is absurd. We deduce that
$\omega(x,f)$ contains no fixed point.
 
Let $n\geq 1$. According to Lemma~\ref{lem:omega-set}, 
$\omega(x,f)=\bigcup_{i=0}^{2^{n-1}}\omega(f^i(x),f^{2^n})$ and 
the set $\omega(f^i(x),f^n)$ is infinite for every 
$i\in\Lbrack 0, n-1\Rbrack$.  Applying the previous
result to $f^n$, we deduce that $\omega(x,f)$ contains no
periodic point of period $n$.
\end{proof}

The next proposition states that an infinite $\omega$-limit set of a zero
entropy interval map
is a \emph{solenoidal set}\index{solenoidal $\omega$-limit set}, that is, it
is included in a nested sequence of cycles of intervals of periods
tending to infinity. This is a key tool when studying zero entropy interval 
maps. This result is implicitly contained in several papers of Sharkovsky,
and stated without proof in a paper of Blokh \cite{Blo7}. A very similar
result, dealing with infinite transitive sets of zero entropy 
interval maps, was proved by 
Misiurewicz \cite{Mis7}.
The formulation we give follows Smítal's \cite{Smi}, except the property
that the intervals can be chosen to be closed, which is due
to Fedorenko, Sharkovsky and Smítal \cite{FSS}. Although the result is
mostly interesting for infinite $\omega$-limit sets, 
the proposition below also deals with
finite $\omega$-limit sets because this case will be used in the sequel.

\begin{prop}\label{prop:htop0-Lki}
Let $f$ be an interval map of zero topological 
entropy and let $x_0$ be a point. If $\omega(x_0,f)$ is a periodic orbit of 
period $2^p$ for some $p\ge 0$, set $\CI:=\Lbrack 0,p\Rbrack$.
If $\omega(x_0,f)$ is infinite, set $\CI:=\IZ^+$. 
There exists a (finite or infinite) sequence of
closed intervals $(L_k)_{k\in\CI}$ such that, for all $k\in\CI$,
\begin{enumerate}
\item $(L_k,f(L_k),\ldots, f^{2^k-1}(L_k))$ is a
cycle of intervals, that is, these intervals are pairwise disjoint and
$f^{2^k}(L_k)=L_k$,
\item $\forall i,j\in\Lbrack 0,2^k-1\Rbrack$, 
$i\neq j$, there is a 
point $z$ between $f^i(L_k)$ and $f^j(L_k)$ such that $f^{2^{k-1}}(z)=z$,
\item$ L_{k+1}\cup f^{2^k}(L_{k+1})\subset L_k$ provided $k+1\in \CI$, 
\item $\disp\omega(x_0,f)\subset \bigcup_{i=0}^{2^k-1} f^i(L_k)$,
\item
$f^i(L_k)$ is the smallest $f^{2^k}$-invariant 
interval containing $\omega(f^i(x_0),f^{2^k})$,
\item if $k+2\in \CI$, $\exists N\ge 0$, $\forall n\ge N$, 
$f^n(x_0)\in f^n(L_k)$.
\end{enumerate}
Moreover, if $\omega(x_0,f)$ is infinite, then $f$ is of type $2^\infty$ for
Sharkovsky's order.
\end{prop}

\begin{proof} 
According to Theorem~\ref{theo:htop-power-of-2}, the period of any periodic
point is a power of $2$, so the map $f$ is of type $\unrhd 2^\infty$.
If $\omega(x_0,f)$ is infinite, the fact that $f$ is of type $2^\infty$ 
follows from the existence of the infinite sequence of intervals 
$(L_k)_{k\ge 0}$. Indeed, if $L_k$ satisfies (i),
then there exists $x\in L_k$ such that $f^{2^k}(x)=x$
(by Lemma~\ref{lem:fixed-point}), and $x$ is of period $2^k$ because
$L_k,f(L_k),\ldots, f^{2^k-1}(L_k)$ are pairwise disjoint; thus the set
of periods of $f$ contains $\{2^k\mid k\ge 0\}$.
The rest of the proof is devoted to the definition and the properties of 
$(L_k)_{k\in\CI}$.

Let $k\in\CI$. We set $g_k:=f^{2^k}$, 
\begin{gather*}
I_k:=[\min\omega(x_0,g_k),\max\omega(x_0,g)]\\
\text{and}\quad L_k:=\overline{\bigcup_{n\geq 0} (g_k)^n(I_k)}.
\end{gather*}
Trivially, $g_k(L_k)\subset L_k$. For all $n\geq 0$, the set $(g_k)^n(I_k)$ 
is an interval containing
$\omega(x_0,g_k)$ because $g_k(\omega(x_0,g_k))=\omega(x_0,g_k)$ 
by Lemma~\ref{lem:omega-set}(iii), so $(g_k)^n(I_k)\supset I_k$.
Thus $L_k$ is an interval and $g_k(L_k)=L_k$. Therefore
\begin{equation}\label{eq:smallest-invariantint}
L_k\text{ is the smallest $g_k$-invariant 
interval containing }\omega(x_0,g_k).
\end{equation}
Let $i\in\Lbrack 0,2^k-1\Rbrack$. It is clear that $f^i(L_k)$ is a 
$g_k$-invariant interval containing $f^i(\omega(x_0,g_k))=
\omega(f^i(x_0),g_k)$ (by Lemma~\ref{lem:omega-set}(iii)). 
Let $J$ be a $g_k$-invariant interval 
containing $\omega(f^i(x_0),g_k)$. Then
$f^{2^k-i}(J)\supset \omega(f^{2^k}(x_0),g_k)=\omega(x_0,g_k)$
(by Lemma~\ref{lem:omega-set}(iv)+(ii)). Thus $f^{2^k-i}(J)\supset L_k$
by \eqref{eq:smallest-invariantint}, so $g_k(J)\supset f^i(L_k)$.
This implies that, for all $i\in\Lbrack 0,2^k-1\Rbrack$,
\begin{equation}\label{eq:smallest-invariantint-i}
f^i(L_k)\text{ is the smallest $g_k$-invariant 
interval containing }\omega(f^i(x_0),g_k)
\end{equation}
which is (v).
Moreover, by Lemma~\ref{lem:omega-set}(v),
$$
\omega(x_0,f)=\bigcup_{i=0}^{2^k-1}(\omega(f^i(x_0),g_k))\subset
\bigcup_{i=0}^{2^k-1}f^i(L_k),
$$
which gives (iv). If $k+1$ belongs to $\CI$, the interval $I_{k+1}$ is included in $I_k$
because $\omega(x_0,(g_k)^2)\subset \omega(x_0,g_k)$. Thus
$$
L_{k+1}\cup g_k(L_{k+1})=\overline{\bigcup_{n\geq 0}(g_k)^{2n}(I_{k+1})}
\cup \overline{\bigcup_{n\geq 0}(g_k)^{2n+1}(I_{k+1})}
\subset \overline{\bigcup_{n\geq 0}(g_k)^n(I_{k})}=L_k,
$$
which is (iii).

We are going to prove (ii) by induction on $k$. This will show
at the same time that the intervals $(f^i(L_k))_{0\le i<2^k}$ are
pairwise disjoint, which in turn implies (i) in view of the fact that
$L_k$ is strongly invariant under $g_k$.

\medskip
$\bullet$ There is nothing to prove for $k=0$.

$\bullet$ Suppose that $k:=1$ belongs to $\CI$.
If $\#\omega(x_0,f)=2$, then $\omega(x_0,f)$ is a periodic orbit of period
$2$ and the interval $I_1$ is reduced to a single point $\{y\}$ satisfying
$f^2(y)=y$ and $f(y)\ne y$; moreover, $L_1=\{y\}$ and $f(L_1)=\{f(y)\}$. 
Then $f(\langle y,f(y)\rangle)\supset
\langle y,f(y)\rangle$, which implies that there exists a point $z\in
\langle y,f(y)\rangle$ such that $f(z)=z$, and $z$ is different from
$y,f(y)$. This proves (ii) for $k=1$ in this case.
From now on, we suppose that $\#\omega(x_0,f)\ge 3$ and we set $g:=g_1=f^2$.
We write $I_1=[a,b]$ and
$L_1=[c,d]$. We have $I_1\subset L_1$, that is, $c\le a\le b\le d$. 
Since $\omega(x_0,f)$ contains at least 3 points, 
$$
\omega(x_0,f)\cap (\min\omega(x_0,f),\max\omega(x_0,f))\neq\emptyset,
$$ 
and thus there exists $n\ge 0$ such that
$\min\omega(x_0,f)<f^n(x_0)<\max\omega(x_0,f)$.
If $f^{n+1}(x_0)>f^n(x_0)$,
there exists $j\ge 1$ such that $f^{n+j}(x_0)$ is arbitrarily close to 
$\min\omega(x_0,f)$, in such a way that
$f^{n+j}(x_0)<f^n(x_0)<f^{n+1}(x_0)$. Similarly,  if
$f^{n+1}(x_0)<f^n(x_0)$, there exists $j\ge 1$ such that
$f^{n+j}(x_0)>f^n(x_0)>f^{n+1}(x_0)$. Since $f$ has zero topological
entropy, $f$ has no horseshoe by Proposition~\ref{prop:horseshoe-htop}.
According to
Lemma~\ref{lem:alternating}, there exist a point $z$ and an integer 
$N$ such that $f(z)=z$ and
\begin{eqnarray}
\text{either }& \forall n\ge N,\ f^{2n}(x_0)<z<f^{2n+1}(x_0),  &\label{eq:altern}\\ 
\text{or }& \forall n\ge N,\ f^{2n}(x_0)>z>f^{2n+1}(x_0). &\nonumber
\end{eqnarray}
We assume that we are in case \eqref{eq:altern}, the other case being 
symmetric. This implies that $b\le z$.
We are going to show that $z>d$. Suppose on the contrary that 
\begin{equation}\label{eq:zinbd}
z\in [b,d].
\end{equation}
Since $z$ is a fixed point, we can define
$z':=\min\{x\in [b,d]\mid g(x)=x\}$.
Since $b\in\omega(x_0,g)$, we have $g(b)\in \omega(x_0,g)$; 
moreover $b$ is not a fixed point for $g$
(by Lemma~\ref{lem:omega-finite} when $\omega(x_0,g)$ is
finite,
and by Proposition~\ref{prop:omega-no-periodic-point} 
when $\omega(x_0,g)$ is infinite). Hence 
$$
g(b)<b<z'
$$ (recall that $\omega(x_0,f)\subset [a,b]$). 
Since $z'$ is the minimal fixed point for $g$  greater than $b$,
this implies that
\begin{equation}\label{eq:xinbz}
\forall x\in [b,z'),\ g(x)<x.
\end{equation}
See Figure~\ref{fig:cabzd}.
\begin{figure}[htb]
\centerline{\includegraphics{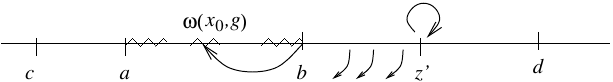}}
\caption{The points $a,b,c,d,z'$ and the set 
$\omega(x_0,g)$ (represented by a zigzag).}
\label{fig:cabzd}
\end{figure}

Let $v:=\max g([c,b])$. Then $v\ge b$ because $g(\omega(x_0,g))=\omega(x_0,g)$,
and $v\le d$ because $g(L_1)\subset L_1$. 
Suppose that $v<z'$. Then $g([b,v])\subset [c,v]$ by 
\eqref{eq:xinbz}, and $g([c,v])=g([c,b])\cup g([b,v]) \subset [c,v]\cup [c,v]=
[c,v]$. Thus $[c,v]$ is a $g$-invariant interval containing $I_1$.
But this is a contraction to \eqref{eq:smallest-invariantint} 
because $[c,v]\neq L_1$. We deduce that $\max g([c,b])\ge z'$, and thus there
exists $y\in [c,b]$ such that $g(y)=z'$. We choose $y$ maximal with this
property; $y<b$ by \eqref{eq:xinbz}.
Let $x\in (y,z')$. The maximality of $y$ and \eqref{eq:xinbz}
imply that $g(x)<z'$. If $g(x)\le y$, then
$g(x)\le y<x<z'=g(y)$, and thus $[y,x],[x,z']$ form a horseshoe for $g$,
which contradicts the fact that $h_{top}(f)=0$ by
Proposition~\ref{prop:horseshoe-htop}. Therefore,
\begin{equation}\label{eq:xinxz'}
\forall x\in (y,z'),\ y<g(x)<z'.
\end{equation}
Using \eqref{eq:xinxz'} and the fact that $g(y)=g(z')=z'>y$, we get
$$
w:=\inf g((y,z'))=\min g([y,z'])>y.
$$
Thus $[y,w]$ is mapped into $[w,z']$, and 
$[w,z']$ is $g$-invariant (see Figure~\ref{fig:ywz'}).
\begin{figure}[htb]
\centerline{\includegraphics{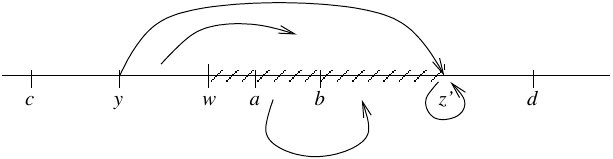}}
\caption{The interval $[y,w]$ is mapped into $[w,z']$, and 
$[w,z']$ (hatched) is $g$-invariant.}
\label{fig:ywz'}
\end{figure}
We are going to show that $w\le a$.
Recall that $y<b<z'$, so $[y,z']$ is a neighborhood of $b$. 
Since $b\in\omega(x_0,g)$, there exists $i$ such that $g^i(x_0)\in 
[y,z']=[y,w]\cup [w,z']$. Thus $g^j(x_0)\in [w,z']$ for all $j\ge i+1$,
which implies that $\omega(x_0,g)\subset [w,z']$. Since $a\in\omega(x_0,g)$,
this implies 
$$w\le a.$$
Therefore, $[w,z']$ is
a $g$-invariant interval containing $I_1=[a,b]$, but this is a contradiction 
to \eqref{eq:smallest-invariantint}  because $[w,z']$ does not contain
$L_1$. We deduce that \eqref{eq:zinbd} is false, that is,
$z>d=\max L_1$. According to \eqref{eq:altern},
we have $z\le \omega(f(x_0),f^2)$. Moreover, $z$ does not belong to 
$f(L_1)$ because $f^2(L_1)=L_1$ and $z\notin L_1$. Since $f(L_1)$ is an interval
containing $\omega(f(x_0),f^2)$, we conclude that $L_1<z< f(L_1)$.
This ends the step $k=1$.

\medskip
$\bullet$ Let $k\geq 1$ such that $k+1\in\CI$ 
and suppose that (ii) is satisfied
for $k$. Let $i,j$ be such that $0\leq i<j<2^{k+1}$. 
If $j-i\neq 2^k$, then, according
to the induction hypothesis, there exists a point $z$
between $f^i(L_k)$ and $f^j(L_k)$, satisfying $f^{2^{k-1}}(z)=z$. Thus $z$ is also between
$f^i(L_{k+1})$ and $f^j(L_{k+1})$ because $L_{k+1}\subset L_k$.
If $j=i+2^k$, then, using \eqref{eq:smallest-invariantint-i},
we can apply the case $k=1$ to the map $g_k$ and the point $f^i(x_0)$:
there exists a point  $z$ strictly between $f^i(L_{k+1})$ and 
$f^j(L_{k+1})=g_k(f^i(L_{k+1}))$ such that
$f^{2^k}(z)=z$.
This completes the induction.

\medskip
It remains to show (vi). Let $k\ge 0$ be an integer such that $k+2\in\CI$. 
Then $L_k$ contains the four
disjoint intervals $L_{k+2}, f^{2^k}(L_{k+2}), f^{2^{k+1}}(L_{k+2}),
f^{2^{k+2}}(L_{k+2})$ by (iii). One of these intervals is included in
$\Int{L_k}$. Moreover, this interval contains $\omega(f^{2^ki}(x_0),f^{2k+2})$
for some $i\in\Lbrack 0,3\Rbrack$; thus there exists an integer $N\ge 0$ 
such that $f^N(x)\in L_k$. 
Since $L_k$ is a $2^k$-periodic interval containing $\omega(x_0,f^{2^k})$,
this implies that $N$ is a multiple of $2^k$, so
$f^N(L_k)=L_k$. Therefore  $f^n(x)\in f^n(L_k)$ for all $n\ge N$. 
This concludes the proof of the proposition.
\end{proof}

\begin{lem}\label{lem:periodic-interval-power-of-2}
Let $f$ be an interval map of zero topological entropy. 
If $J$ is a
nonempty (non necessarily closed) interval such that $f^p(J)=J$ and 
$(f^i(J))_{0\leq i<p}$ are pairwise disjoint, then $p$ is a power of $2$.
\end{lem}

\begin{proof}
If $J$ is reduced to one point, then it is a periodic point and thus
$p$ is a power of $2$ by
Theorem~\ref{theo:htop-power-of-2}. From now on, we assume that
$J$ is non degenerate, which implies that
$f^n(J)$ is a non degenerate interval for all $n\geq 0$.
Since $f^p(\overline{J})=\overline{J}$, there exists $x\in \overline{J}$ 
such that $f^p(x)=x$ by Lemma~\ref{lem:fixed-point}.  
By Theorem~\ref{theo:htop-power-of-2}, the period of $x$
is equal to $2^k$ for some $k\ge 0$, and thus
there exists  $m\ge 1$ such that $p=m2^k$. 
If $x\in J$, then $(f^i(x))_{0\leq i<p}$ are pairwise
distinct, so $p=2^k$. 

Suppose that $x\in \End{J}$ and $m\ge 3$.
The point $x=f^{2^k}(x)$ belongs to $f^{2^k}(\overline{J})$. Since
$f^{2^k}(J)\cap J=\emptyset$, this implies that $x$ is an endpoint of 
$f^{2^k}(J)$. We also have $x=f^{2^{k+1}}(x)\in f^{2^{k+1}}(\overline{J})$,
which implies that $x\in \End{f^{2^{k+1}}(J)}$. But this
contradicts the fact that
$J,f^{2^k}(J), f^{2^{k+1}}(J)$ are pairwise disjoint non degenerate intervals. 
Therefore, if $x\in\End{J}$, then $m=1$ or $2$ and $p$ is a power of $2$.
\end{proof}

The next lemma states that two points in the same infinite $\omega$-limit set are
$f$-separable if and only if they are separable by intervals in the family
$(f^i(L_k))$ given by Proposition~\ref{prop:htop0-Lki}.

\begin{lem}\label{lem:f-separable-Lni}
Let $f$ be an interval map of zero topological entropy and
let $a_0,a_1$ be two distinct points in the same infinite $\omega$-limit
set $\omega(x_0,f)$. Let $(L_n)_{n\ge 0}$ be the intervals given by
Proposition~\ref{prop:htop0-Lki}. Then the following assertions are
equivalent:
\begin{enumerate}
\item $a_0,a_1$ are $f$-separable,
\item there exist $n\geq 1$ and  $i,j\in\Lbrack 0,2^n-1\Rbrack$
such that $i\neq j$, $a_0\in f^i(L_n)$ and $a_1\in f^j(L_n)$.
\end{enumerate}
\end{lem}

\begin{proof}
The implication (ii) $\Rightarrow$ (i) is trivial. Suppose that
the points $a_0$ and $a_1$ are $f$-separable. By definition, there exist an interval
$J$ and an integer $p\geq 1$ such that $a_0\in J$, $a_1\notin J$,
$f^p(J)=J$ and $(f^i(J))_{0\leq i<p}$ are pairwise disjoint. 
By Lemma~\ref{lem:periodic-interval-power-of-2}, $p$ is a power of $2$,
that is, $p=2^n$ for some $n\ge 0$. Since $a_0$ is in $\omega(x_0,f)\subset
\bigcup_{i=0}^{2^n-1}f^i(L_n)$, there exists $i\in\Lbrack 0,2^n-1\Rbrack$
such that $a_0\in f^i(L_n)$. We set $K:=f^i(L_n)\cap J$. Then 
$f^{2^n}(K)\subset f^i(L_n)\cap J=K$ (recall that $f^{2^n}(L_n)=L_n$). 
Therefore, the interval $K$
contains the three points $a_0, f^{2^n}(a_0), f^{2^{n+1}}(a_0)$, 
which belong to
$\omega(x_0,f)$ (by Lemma~\ref{lem:omega-set}(iii)) and are distinct by
Proposition~\ref{prop:omega-no-periodic-point}. This implies that
$\Int{K}\cap\omega(x_0,f)\neq\emptyset$, and thus there exists an integer 
$m\ge 0$ such that $f^m(x_0)\in K$. Hence
\begin{equation}\label{eq:bigcupK}
\omega(x_0,f)\subset \bigcup_{k=0}^{2^n-1} f^k(\overline{K})\quad
\text{and}\quad \omega(f^m(x_0),f^{2^n})\subset \overline K.
\end{equation}
According to Proposition~\ref{prop:htop0-Lki}, $f^i(L_n)$ is the smallest
$f^{2^n}$-invariant interval containing $\omega(f^i(x_0),f^{2^n})$.
Therefore, \eqref{eq:bigcupK} and the fact that $\overline{K}\subset f^i(L_n)$
imply that $\overline{K}=f^i(L_n)$.
Since $a_1\in\omega(x_0,f)$, there exists $j\in\Lbrack 0,2^n-1\Rbrack$ such that 
$a_1\in f^j(L_n)$. We are going to show that $j\neq i$. Suppose on
the contrary that $j=i$, that is, $a_1\in f^i(L_n)=\overline K$. This
implies that $a_1$ is an endpoint of $f^i(L_n)$ because $a_1\notin J$
and $J\supset K$. Let $a_2$ denote the other endpoint of $f^i(L_n)$. 
There are two cases:
\begin{itemize}
\item Case 1:
$K=f^i(L_n)\setminus\{a_1\}$.
\item Case 2:
$K=f^i(L_n)\setminus\{a_1,a_2\}$.
\end{itemize}
Recall that $f^{2^n}(K)\subset K$ and $f^{2^n}(f^i(L_n))=f^i(L_n)$.
In Case 1, this implies that $f^{2^n}(a_1)=a_1$. In Case 2,
this implies that $f^{2^n}(a_1)\in\{a_1,a_2\}$ and
$f^{2^n}(a_2)\in\{a_1,a_2\}$. In both cases, $f^{2^n}(a_1)$ is a periodic
point, which is impossible by Proposition~\ref{prop:omega-no-periodic-point}
because $f^{2^n}(a_1)\in\omega(x_0,f)$. We conclude that $i\neq j$,
which proves the implication (i) $\Rightarrow$ (ii).
\end{proof}

\begin{rem}
In the previous proof, we saw that $\overline{J\cap f^i(L_n)}=f^i(L_n)$. 
Therefore, if $J$ is any periodic closed interval containing $a_0$ with 
$a_0\in\omega(x_0,f)$, then $J$ contains $f^i(L_n)$ for some 
integers $n,i$ such that $a_0\in f^i(L_n)$. 
\end{rem}

\begin{lem}\label{lem:omega-periodic-points}
Let $f$ be an interval map of zero topological entropy and 
let $x_0$ be a point such that $\omega(x_0,f)$ is infinite.
\begin{enumerate}
\item
If $J$ is an interval containing three distinct points of
$\omega(x_0,f)$, then $J$ contains a periodic point.
\item
If $U$ is an open interval such that $U\cap
\omega(x_0,f)\neq \emptyset$, then there exists an integer $n\ge 0$ such
that $f^n(U)$ contains a periodic point.
\end{enumerate}
\end{lem}

\begin{proof}
i) Let $J$ be an interval and let $z_1<z_2<z_3$ be three points in 
$J\cap \omega(x_0,f)$. Let
$(L_n)_{n\ge 0}$ be the intervals given by Proposition~\ref{prop:htop0-Lki}
for $\omega(x_0,f)$.
Suppose that, for every integer $n\geq 0$, there is $i_n\in\Lbrack 0,2^n-1\Rbrack$
such that $z_1,z_3\in f^{i_n}(L_n)$. Since $z_2\in (z_1,z_3)\cap 
\omega(x_0,f)$,
there exist two positive integers $m>k$ such that the two points
$f^{m-k}(x_0)$, $f^m(x_0)$ belong to $(z_1,z_3)$. Since $[z_1,z_3]\subset f^{i_n}(L_n)$, 
this implies that $f^m(x_0)\in f^{k+i_n}(L_n) \cap f^{i_n}(L_n)$ for all $n\geq
0$. On the other hand, $f^{k+i_n}(L_n) \cap f^{i_n}(L_n)=\emptyset$ 
if $2^n>k$ because the intervals $(f^k(L_n))_{0\le k<2^n}$ are
pairwise disjoint; we get a contradiction. We deduce that there
exist $n\ge 0$  and $i,j \in\Lbrack 0,2^n-1\Rbrack$ 
such that $i\neq j$, $z_1\in f^i(L_n)$ and
$z_3\in f^j(L_n)$.  We know by Proposition~\ref{prop:htop0-Lki}
that there exists a periodic point $z$ between $f^i(L_n)$ and
$f^j(L_n)$, so $z\in [z_1,z_3]\subset J$.

ii) Now we consider an open interval $U$ such that
$U\cap\omega(x_0,f)\neq  \emptyset$ and let $y\in U\cap
\omega(x_0,f)$. If $U$ contains a periodic point, the proof is over.
From now on, we suppose that $U$ contains no periodic point. Let
$L\supset U$ be the maximal interval containing no periodic point.
Since $U$ is open and contains $y\in\omega(x_0,f)$,
there exist integers $k\geq 0$ and  $n_2>n_1>0$ such that the points $f^k(x_0),
f^{k+n_1}(x_0)$ and  $f^{k+n_2}(x_0)$ belong to $U$. 
The points $y, f^{n_1}(y)$ and $f^{n_2}(y)$ belong to $\omega(x_0,f)$ 
(by Lemma~\ref{lem:omega-set}(iii)), and they are distinct by
Proposition~\ref{prop:omega-no-periodic-point}. If
$y,f^{n_1}(y),f^{n_2}(y)$ belong to $L$, then $L$ contains a periodic point
by (i), which contradicts
the definition of $L$. Thus there exists $i\in\{1,2\}$ such that
$f^{n_i}(y) \notin L$. The interval $f^{n_i}(U)$ contains both
$f^{k+n_i}(x_0)$ and $f^{n_i}(y)$, with
$f^{k+n_i}(x_0)\in L$ and $f^{n_i}(y)\notin L$, and thus the maximality of
the interval $L$ implies that $f^{n_i}(U)$ contains a periodic point.
\end{proof}

Proposition~\ref{prop:cantor-delta-scrambled} below was shown by Smítal
\cite{Smi} in the case $x_0=x_1$. It states that, if $a_0,a_1$
belong to the same infinite $\omega$-limit set and are
$f$-non separable, where $f$ is a zero entropy interval map, then $f$
admits a $\delta$-scrambled Cantor set with $\delta=|a_1-a_0|$. The next
lemma is the first step of the proof of this result.

\begin{lem}\label{lem:non-separable-Jm}
Let $f$ be an interval map of zero topological entropy and
let $x_0,x_1,a_0,a_1$ be four points 
such that $\omega(x_0,f)$ and $\omega(x_1,f)$ are infinite, $a_0\in
\omega(x_0,f)$ and $a_1\in\omega(x_1,f)$. Let $(L_n)_{n\ge 0}$ denote the
intervals given by Proposition~\ref{prop:htop0-Lki} for
$\omega(x_0,f)$. Suppose that, for all $n\geq 0$, 
$\omega(x_1,f)\subset \bigcup_{i=0}^{2^n-1}f^i(L_n)$ and 
there exists $i_n\in\Lbrack 0, 2^n-1\Rbrack$ 
such that both points $a_0,a_1$ belong to $J_n:=f^{i_n}(L_n)$. Let
$A_0,A_1$ be two intervals such that $a_0\in\Int{A_0}$ and
$a_1\in\Int{A_1}$. Then there exists $m\ge 0$ such that
$f^{2^m}(A_0)\cap f^{2^m}(A_1)\supset J_m$.
\end{lem}

\begin{proof}
By Lemma~\ref{lem:omega-periodic-points}(ii), there exist $n_0$ and $n_1$
such that $f^{n_0}(A_0)$ contains a periodic point $y_0$ and
$f^{n_1}(A_1)$ contains a periodic point $y_1$. According to
Theorem~\ref{theo:htop-power-of-2}, the periods of $y_0,y_1$ are some
powers of $2$. Let $2^p$ be a common multiple of their periods and let
$q$ be such that $q>p$ and $2^q\geq\max\{n_0,n_1\}$. We fix $j\in\{0,1\}$ and 
we set $y_j':=f^{2^q-n_j}(y_j)$.  Then $f^{2^p}(y_j')=y_j'$ and 
$y_j'\in  f^{2^q}(A_j)$. Moreover, $y_j'\notin
J_q$ because $J_q=f^{i_q}(L_q)$ is a periodic interval of period
$2^q>2^p$. Suppose that 
\begin{equation}\label{eq:yj'<Jq}
y_j'< J_q,
\end{equation}
the case with the reverse inequality being symmetric.

Let $g:=f^{2^q}$. Then $g(y'_j)=y'_j$. The intervals  $J_{q+1}$
and $g(J_{q+1})$ are disjoint, $g^2(J_{q+1})=J_{q+1}$
and $J_{q+1}\cup g(J_{q+1})\subset J_q$.  Moreover,
\begin{equation}\label{eq:gAj}
\{y_j',g(a_j)\}\subset g(A_j)\quad\text{and}\quad
g(a_j)\in g(J_{q+1}).
\end{equation} 
We consider two cases.

\textbf{Case 1.} If $g(J_{q+1})> J_{q+1}$ (Figure~\ref{fig:gAj-1}),
then, by connectedness,
$g(A_j)\supset J_{q+1}$ by \eqref{eq:gAj} and \eqref{eq:yj'<Jq}. 
Thus $g^2(A_j)\supset g(J_{q+1}) \cup\{y_j'\}$ and, by
connectedness, $g^2(A_j)$ contains $J_{q+1}$.
This implies that $g^4(A_j)\supset J_{q+1}$.
\begin{figure}[hbt]
\centerline{\includegraphics{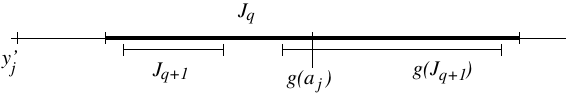}}
\caption{Relative positions in Case 1 of 
the proof of Lemma~\ref{lem:non-separable-Jm}. The interval
$g(A_j)$ contains $y'_j$ and $g(a_j)$, so $g(A_j)\supset 
J_{q+1}$.}
\label{fig:gAj-1}
\end{figure}

\textbf{Case 2.} If $g(J_{q+1})< J_{q+1}$ (Figure~\ref{fig:gAj-2}), 
then $g^2(A_j)$ contains the points $g^2(a_j)$ and $y_j'$ by \eqref{eq:gAj}.  
Since $a_j\in J_{q+1}$, the point $g^2(a_j)$ belongs to $J_{q+1}$ too.
Thus, by
connectedness, $g^2(A_j)$ contains $g(J_{q+1})$ by \eqref{eq:yj'<Jq}.
\begin{figure}[hbt]
\centerline{\includegraphics{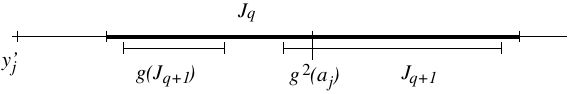}}
\caption{Relative positions in Case 2 of the proof of
Lemma~\ref{lem:non-separable-Jm}. The interval
$g^2(A_j)$ contains $y'_j$ and $g^2(a_j)$, so $g^2(A_j)\supset 
g(J_{q+1})$.}
\label{fig:gAj-2}
\end{figure}
Then $g^3(A_j)\supset J_{q+1}\cup \{y_j'\}$, so
$g^3(A_j)\supset g(J_{q+1})$ by connectedness, and finally
$g^4(A_i)\supset J_{q+1}$.

We conclude that $g^4(A_j)\supset J_{q+1}\supset
J_{q+2}$ for $j\in\{0,1\}$. This is the required result with $m:=q+2$.
\end{proof}

\begin{prop}\label{prop:cantor-delta-scrambled}
Let $f\colon I\to I$ be an interval map of zero topological entropy 
and let $x_0, x_1, a_0,a_1$ be four points
such that $\omega(x_0,f)$ and $\omega(x_1,f)$ are infinite,
$a_0\in \omega(x_0,f)$, $a_1\in \omega(x_1,f)$ and
$a_0\neq a_1$. Let $(L_n)_{n\ge 0}$ denote the intervals given by
Proposition~\ref{prop:htop0-Lki} for
$\omega(x_0,f)$. Suppose that, for all $n\geq 0$,
$\omega(x_1,f)\subset \bigcup_{i=0}^{2^n-1}f^i(L_n)$ and 
there exists $i_n\in\Lbrack 0,2^n-1\Rbrack$ 
such that both points $a_0,a_1$ belong to
$f^{i_n}(L_n)$.  Then  $f$ has a $\delta$-scrambled Cantor set
with $\delta:=|a_1-a_0|$. Moreover, if $K_0,K_1$ are disjoint closed
intervals such that $a_i\in \Int{K_i}$ for $i\in\{0,1\}$, then there
exists an increasing  sequence of integers $(n_k)_{k\ge 0}$
such that
\begin{equation}\label{eq:quasishift}
\forall (\alpha_k)_{k\ge 0}\in\{0,1\}^{\IZ^+},\ \exists\, x\in I,
\ \forall k\ge 0,\ f^{n_k}(x)\in K_{\alpha_k}.
\end{equation}
\end{prop}

\begin{proof}
For every $n\ge 0$, we set $J_n:=f^{i_n}(L_n)$. According to
Proposition~\ref{prop:htop0-Lki}, the intervals $(f^i(J_n))_{0\leq
i<2^n}$ are disjoint, $f^{2^n}(J_n)=J_n$ and
$J_{n+1}\cup f^{2^n}(J_{n+1})\subset J_n$.

First we build by induction two sequences of integers $(n(k))_{k\geq
0}$ and $(m(k))_{k\geq 0}$ and a family of closed subintervals
$\{I_{\alpha_0\ldots\alpha_k}\mid k\geq 0, \alpha_i \in\{0,1\}\}$
satisfying the following properties for all  $k\ge 0$ and
all $(\alpha_0,\ldots,\alpha_{k+1})\in\{0,1\}^{k+2}$:
\begin{enumerate}
\item 
$I_{\alpha_0\ldots \alpha_k\alpha_{k+1}}\subset I_{\alpha_0\ldots
\alpha_k}$,
\item
$I_{\alpha_0\ldots \alpha_k}\cap I_{\beta_0\ldots \beta_k}=\emptyset$
if $(\alpha_0,\ldots, \alpha_k)\neq(\beta_0,\ldots, \beta_k)$, where
$(\beta_0,\ldots, \beta_k)\in\{0,1\}^{k+1}$,
\item
$f^{n(k)}(I_{\alpha_0\ldots \alpha_k})=J_{m(k)}$,
\item
$m(k)\geq k$ and $n(k+1)-n(k)=2^{m(k+1)}$,
\item
for $i \in\{0,1\}$,
$f^{n(k)}(I_{\alpha_0\ldots \alpha_k i })\subset [a_i -\frac1k,a_i +\frac1k]$ .
\end{enumerate}

\textbf{Step \boldmath  $k=0$.} Let $\eps\in(0,\frac{\delta}2)$. 
We set $A_i:=[a_i-\eps,a_i+\eps]\cap
I$ for $i\in\{0,1\}$. According to the choice of $\eps$, the 
intervals $A_0, A_1$ are disjoint. By
Lemma~\ref{lem:non-separable-Jm}, there exists an integer $m$ such that
$f^{2^m}(A_0)\cap f^{2^m}(A_1)\supset J_m$.  Thus there exist closed
subintervals  $I_0\subset A_0$ and $I_1\subset A_1$ such that
$f^{2^m}(I_i)=J_m$ for $i\in\{0,1\}$
(Lemma~\ref{lem:chain-of-intervals}(i)). Letting $m(0)=m$ and $n(0)=2^m$, this
ends the construction at step $k=0$.

\textbf{Step \boldmath $k+1$.} Suppose that
$n(k), m(k)$ and $(I_{\alpha_0\ldots\alpha_k})_{(\alpha_0,\ldots,\alpha_k)\in \{0,1\}^{k+1}}$
are already defined. Let $\eps\in(0,\min\{\frac{1}{k+1},\frac{\delta}{2}\})$.  
We set $B_i:=[a_i-\eps,a_i+\eps]\cap I$ for $i\in\{0,1\}$.  
According to the choice of $\eps$,
the intervals $B_0, B_1$ are disjoint.

We set $g:=f^{2^{m(k)}}$.  The interval $J_{m(k)}$ contains the four
disjoint intervals $(g^i(J_{m(k)+2}))_{0\le i\le 3}$. 
We order these intervals from left to right, and we call
$J_{m(k)+2}'$ the second one. Let $j\in\Lbrack 0,3\Rbrack$ be such that
$g^j(J_{m(k)+2}')=J_{m(k)+2}$. For $i\in\{0,1\}$, let  $a_i'$ be a point
in $J_{m(k)+2}'$ such that $g^j(a_i')=a_i$. It is clear that, for all
$n\geq 1$, the points $a_0',a_1'$ are in the same interval among
$(f^k(L_n))_{0\le k\le 2^n-1}$, otherwise it would be false for $a_0,a_1$.
Moreover, the points $a_0',a_1'$ belong to $\Int{J_{m(k)}}$  
according to the choice of $J_{m(k)+2}'$.

Since $B_i$ is a neighborhood of $a_i$, the set
$g^{-j}(B_i)$ is a neighborhood of $a_i'$ for $i\in\{0,1\}$. Let $U_i$ be the 
connected component of $g^{-j}(B_i)$ containing $a_i'$. Then $U_0\cap
J_{m(k)}$ and  $U_1\cap J_{m(k)}$ are a connected neighborhood of
$a_0'$ and $a_1'$ respectively. Thus, according to 
Lemma~\ref{lem:non-separable-Jm}, there exists an integer $q\ge 0$ such that,
for $i\in\{0,1\}$,
$f^{2^q}(U_i\cap J_{m(k)})\supset f^p(L_q)$, where $p\in\Lbrack 0,2^q-1\Rbrack$ 
is the integer such that $a_0',a_1'\in f^p(L_q)$.  
Since $g^j(U_i\cap J_{m(k)})\subset B_i\cap J_{m(k)}$, we have
\begin{equation}\label{eq:BiJm}
f^{2^q}(B_i\cap J_{m(k)})\supset 
f^{2^q}(g^j(U_i\cap J_{m(k)}))\supset g^j(f^p(L_q)).
\end{equation}
Moreover, $g^j(f^p(L_q))$ contains $a_0=g^j(a_0')$ and $a_1=g^j(a_1')$, 
which implies that
\begin{equation}\label{eq:gjfpLq}
g^j(f^p(L_q))=J_q
\end{equation}
because $(f^k(J_q))_{0\le k\le 2^q-1}$ is a cycle of intervals and $J_q$
is the unique interval of this cycle containing $a_0,a_1$.
We choose $m(k+1)\geq\max\{q,k+1\}$. Then
$f^{2^{m(k+1)}-2^q}(J_q)=J_q$, and thus, by \eqref{eq:BiJm} and 
\eqref{eq:gjfpLq},
$$
f^{2^{m(k+1)}}(B_i\cap J_{m(k)})\supset 
f^{2^{m(k+1)}-2^q}(g^j(f^p(L_q)))=f^{2^{m(k+1)}-2^q}(J_q)=
J_q\supset J_{m(k+1)}.
$$
Then, for $i\in\{0,1\}$, there exists a closed subinterval
$F_i\subset B_i\cap J_{m(k)}$ such that $f^{2^{m(k+1)}}(F_i)=
J_{m(k+1)}$ (by Lemma~\ref{lem:chain-of-intervals}(i)).
Let $(\alpha_0,\ldots,\alpha_k)\in\{0,1\}^{k+1}$ and $i \in\{0,1\}$.
Since $f^{n(k)}(I_{\alpha_0\ldots\alpha_k})=J_{m(k)}$ 
by the induction hypothesis,
there exists a closed subinterval $I_{\alpha_0\ldots\alpha_k i } \subset
I_{\alpha_0\ldots\alpha_k}$ such that
$f^{n(k)}(I_{\alpha_0\ldots\alpha_k i })=F_{i }$. 
By choice of $F_i $,
this implies that $f^{n(k)+2^{m(k+1)}}(I_{\alpha_0\ldots\alpha_k i })=
J_{m(k+1)}$.  We define $n(k+1):=n(k)+2^{m(k+1)}$. It is clear that
properties (i), (iii), (iv), (v) are satisfied. The intervals
$I_{\alpha_0\ldots \alpha_k 0}$ and $I_{\alpha_0\ldots \alpha_k 1}$
are disjoint because their images under $f^{n(k+1)}$  are included
respectively in $B_0$ and $B_1$. Moreover, $I_{\alpha_0\ldots
\alpha_k\alpha_{k+1}}\cap
I_{\beta_0\ldots\beta_k\beta_{k+1}}=\emptyset$ if $(\alpha_0,\ldots,
\alpha_k)\neq  (\beta_0,\ldots,\beta_k)$ because $I_{\alpha_0\ldots
\alpha_k\alpha_{k+1}} \subset I_{\alpha_0\ldots \alpha_k}$,
$I_{\beta_0\ldots\beta_k\beta_{k+1}}\subset I_{\beta_0\ldots\beta_k}$
and these sets are disjoint by the induction hypothesis.  
This gives (ii) and the induction is over.

\medskip Now we prove the proposition.  Let $\Sigma:=\{0,1\}^{\IZ^+}$,
endowed with the product topology.
For every $\bar\alpha=(\alpha_n)_{n\ge 0} \in\Sigma$, we set
$$
I_{\bar\alpha}:=\bigcap_{n=0}^{+\infty}I_{\alpha_0\ldots\alpha_n}.
$$
By (i), this is a decreasing intersection of nonempty compact
intervals, and thus $I_{\bar\alpha}$ is a nonempty compact interval.
Moreover, $I_{\bar\alpha}\cap
I_{\bar\beta}= \emptyset$ if $\bar\alpha\neq \bar\beta$,
$\bar\alpha,\bar\beta\in\Sigma$.
We define
$$
E:=\{\bar\alpha\in\Sigma\mid I_{\bar\alpha}\text{ is not reduced to a
single  point}\}.
$$
The set $E$ is at most countable because the sets
$(I_{\bar\alpha})_{\bar\alpha\in E}$ are disjoint intervals and they
are non degenerate by definition.  We set
$$
X:=\left(\bigcap_{n=0}^{+\infty}\bigcup_{\doubleindice{\alpha_i\in\{0,1\}}{i
\in\Lbrack 0, n\Rbrack}} I_{\alpha_0\ldots \alpha_n}\right)\setminus
\bigcup_{\bar\alpha\in E} \Int{I_{\bar\alpha}}.
$$
This is a totally disconnected compact set. We define
$$
\vfi\colon\begin{array}[t]{rcl} X& \longrightarrow & \Sigma\\ x&
\longmapsto & \bar\alpha\quad \text{if } x\in I_{\bar\alpha}.
\end{array}
$$
The map $\vfi$ is well defined and is clearly onto.  Let $\delta_n$ be
the minimal distance between two distinct intervals of the form
$I_{\alpha_0\ldots\alpha_n}$. Clearly $\delta_n>0$ because these
intervals are closed and pairwise disjoint.  Let $x,y\in X$ and
$(\alpha_n)_{n\ge 0}:=\vfi(x)$,  $(\beta_n)_{n\ge 0}:=\vfi(y)$. If
$|x-y|<\delta_n$, then  necessarily $\alpha_0\ldots
\alpha_n=\beta_0\ldots\beta_n$, and thus $\vfi$ is continuous.

Let $K_0, K_1$ be two disjoint closed intervals such that
$a_i\in \Int{K_i}$ for $i\in\{0,1\}$. Then there exists a positive integer
$N$ such that $[a_i-\frac1N,a_i+\frac1N]\subset K_i$ for $i\in\{0,1\}$.
Let $\bar\alpha=(\alpha_n)_{n\ge 0}\in\Sigma$ and $x\in \vfi^{-1}(\alpha)$.
Then, according to (v), 
for every $k\ge N$, $f^{n(k)}(x)\in f^{n(k)}(I_{\alpha_0\ldots\alpha_{k+1}})
\subset K_{\alpha_{k+1}}$. This proves statement 
\eqref{eq:quasishift} in the proposition (with the the sequence 
$(n_{N+k})_{k\ge 0}$).

We define $\psi\colon \Sigma\to \Sigma$ by
$$
\psi((\alpha_n)_{n\ge 0}):=
(0\alpha_0 00\alpha_0\alpha_1\ldots
0^n\alpha_0\alpha_1\ldots\alpha_{n-1}\ldots)\quad\text{where }
0^n=\underbrace{0\cdots 0}_{n\ \rm times}.
$$
The map $\psi$ is clearly continuous and one-to-one, and thus $\psi(\Sigma)$ is 
compact and uncountable.
For every $\bar\alpha\in \Sigma$, we choose $x_{\bar\alpha}\in X$ such that
$\vfi(x_{\bar\alpha})=\psi(\bar\alpha)$ and we set $S:=\{x_{\bar\alpha}\in
X \mid  \bar\alpha\in\Sigma\}$. If $\psi(\bar\alpha)\notin E$,
there is a unique choice for $x_{\alpha}$, and if $\psi(\bar\alpha)\in E$,
there are two possible choices. Therefore, $S$ is equal to
$\vfi^{-1}(\psi(\Sigma))$  deprived of a countable set, and thus it is an
uncountable Borel set.  
By Theorem~\ref{theo:perfect-set}, there exists a Cantor set $C\subset S$.  

Let $\bar\alpha,\bar\beta$ be two distinct elements of $\Sigma$, and let
$i\ge 0$ be such that $\alpha_i\neq\beta_i$.  According to the 
definition of $\psi$, for every
$N\geq 0$, there exists $k\geq N$ such that  the $(k+1)$-th coordinate
of $\psi(\bar\alpha)$ and $\psi(\bar\beta)$ are equal respectively to
$\alpha_i$ and $\beta_i$, and thus they are distinct. Using (v), 
this implies that
either $f^{n(k)}(x_{\bar\alpha})$ belongs to $[a_0-\frac1k,a_0+\frac1k]$ and
$f^{n(k)}(x_{\bar\beta})$ belongs to $[a_1-\frac1k,a_1+\frac1k]$, or the
converse. In particular,
$|f^{n(k)}(x_{\bar\alpha})-f^{n(k)}(x_{\bar\beta})| \geq \delta-\frac2k$,
which implies that
$$
\text{for all }x,x'\in S,\ x\neq x',\ 
\limsup_{n\to+\infty}|f^{n}(x)-f^{n}(x')|\geq \delta.
$$
By definition of $\psi$, for every $N\geq 0$, there
exists $k\geq N$ such that, for all $\bar\alpha\in\Sigma$, the
$(k+1)$-th coordinate of $\psi(\bar\alpha)$ is equal to $0$. Using
(v), we have that, for all $\bar\alpha\in\Sigma$,
$f^{n(k)}(x_{\bar\alpha}) \in [a_0-\frac1k,a_0+\frac1k]$, and hence
$$
\text{for all }x,x'\in S,\ \liminf_{n\to+\infty}|f^{n}(x)-f^{n}(x')|=0.
$$
We deduce that $S$ is a $\delta$-scrambled set, and thus
$C$ is a $\delta$-scrambled Cantor set.
\end{proof}

In the next proposition, assertion (ii) is stated in \cite[Theorem~2.4]{Smi}. 
In view of Lemma~\ref{lem:approx-periodic}, it implies
that, if $f$ is a zero entropy map admitting no pair
of $f$-non separable points in the same infinite $\omega$-limit set,
then $f$ is not chaotic in the sense of Li-Yorke.

\begin{rem}
In \cite{Smi}, it is claimed without explanation that, for every $\eps>0$,
there exists $n\ge 0$ such that $\max_{i\in\Lbrack 0, 2^n-1\Rbrack} |f^i(L_n)|<\eps$,
where $(L_n)_{n\ge 0}$ are the intervals given by 
Proposition~\ref{prop:htop0-Lki}. It is not clear to us whether this property
does hold. The weaker assertion (i) below
is sufficient to prove Theorem~\ref{theo:htop0-chaos-LY}. See also
Lemma~\ref{lem:diamLcapomega} for a refinement.
\end{rem}

\begin{prop}\label{prop:non-LY-chaotic}
Let $f$ be an interval map of zero topological entropy.
\begin{enumerate}
\item If $\omega(x,f)$ is infinite and contains no $f$-non separable
points, then
$$
\lim_{n\to+\infty}\max_{i\in\Lbrack 0, 2^n-1\Rbrack} 
\diam (\omega(f^i(x),f^{2^n}))=0.
$$
\item Suppose that all pairs of distinct points in an
infinite $\omega$-limit set are $f$-separable. Then
every point $x$ is approximately periodic.
\end{enumerate}
\end{prop}

\begin{proof}

i) Suppose that $\omega(x,f)$ is infinite and contains no $f$-non separable
points. We set $a_n^i:=\min\omega(f^i(x),f^{2^n})$,
$b_n^i:=\max\omega(f^i(x),f^{2^n})$ and $I_n^i:=[a_n^i,b_n^i]$
for all $n\geq 0$ and all $i\in\Lbrack 0,2^n-1\Rbrack$. Then
by Lemma~\ref{lem:omega-set},
$a_n^i, b_n^i\in \omega(x,f)$ and
\begin{equation}\label{eq:nested-In}
\forall i\in\Lbrack 0, 2^n-1\Rbrack,\ 
I_{n+1}^i\cup I_{n+1}^{i+2^n}\subset I_n^i.
\end{equation}
Suppose that there exists $\eps>0$ such that, for all $n\geq 0$, there
is $i\in\Lbrack 0,2^n-1\Rbrack$ with $|I_n^i|\geq\eps$. Using
\eqref{eq:nested-In}, we can build a sequence $(i_n)_{n\geq
0}$ such that
$$
\forall n\ge 0,\ i_n\in\Lbrack 0,2^n-1\Rbrack,\
I_{n+1}^{i_{n+1}}\subset I_n^{i_n}\text{ and }
|I_n^{i_n}|\geq \eps.
$$
We set $J:=\bigcap_{n\geq 1}
I_n^{i_n}$.  Since it is a decreasing intersection of compact intervals, $J$ is
a compact interval and $|J|\geq \eps$. We write $J=[a,b]$. Then
$$
a=\lim_{n\to+\infty} a_n^{i_n}\quad\text{and}\quad
b=\lim_{n\to+\infty}b_n^{i_n}.
$$
Since $\omega(x,f)$ is a closed set (by Lemma~\ref{lem:omega-set}(i)), 
the points $a,b$ belong to $\omega(x,f)$. The intervals $(L_n)_{n\ge 0}$ 
given by
Proposition~\ref{prop:htop0-Lki} are defined in such a way that
$I_n^i\subset f^i(L_n)$ for all $n\ge 0$ and all $i\in\Lbrack 0,2^n-1\Rbrack$.
Therefore $a,b\in f^{i_n}(L_n)$ for all $n\ge 0$.
Then Lemma~\ref{lem:f-separable-Lni} implies that $a,b$  are $f$-non
separable, which is a contradiction (notice that $a\neq b$ because 
$b-a\ge \eps$). We deduce that, for all
$\eps>0$, there exists $m\geq 0$ such that $|I_m^i|<\eps$
for all  $i\in\Lbrack 0,2^m-1\Rbrack$.
Combined with the fact that these intervals are nested, this gives (i).

\medskip
ii) Let $x$ be a point and $\eps>0$. We split the proof into two cases 
depending on $\omega(x,f)$ being finite or not.

First we suppose that $\omega(x,f)$ is infinite. The intervals
$I_n^i$ are defined as above.
It was shown in the proof of (i) that there exists  an integer $m$
such that $|I_m^i|<\eps$ for all $i\in\Lbrack 0,2^m-1\Rbrack$. Moreover,
there exists a point $z\in I_m^0$
such that $f^{2^m}(z)=z$ and $f^i(z)\in I_m^i$ for all 
$i\in\Lbrack 0,2^m-1\Rbrack$ (Lemma~\ref{lem:chain-of-intervals}(ii)).  Since
$f$ is uniformly continuous, there exists $\delta>0$ such that
$$
\forall y,y',\ |y-y'|\leq\delta\Rightarrow
\forall i\in\Lbrack 0,2^m-1\Rbrack,\ |f^i(y)-f^i(y')|\leq\eps.
$$
Let $N$ be a positive integer such that,
for all $k\geq N$, there exists a point $a_k$ in $\omega(x,f^{2^m})$
with $|f^{k2^m}(x)-a_k|\leq\delta$. For all $i\in\Lbrack 0,2^m-1\Rbrack$, 
the two points $f^i(z), f^i(a_k)$ belong to $I_m^i$, so
$$
|f^{k2^m+i}(x)-f^{k2^m+i}(z)| \leq
|f^{k2^m+i}(x)-f^i(a_k)|+|f^i(a_k)-f^i(z)|\leq 2\eps.
$$
We get: $\forall n\ge N2^m$, $|f^n(x)-f^n(z)|\le 2\eps$.

Now we suppose that $\omega(x,f)$ is finite. By
Lemma~\ref{lem:omega-finite}, the set $\omega(x,f)$ is a periodic
orbit.  Let $p$ be the period of this orbit and
$z:=\lim_{n\to+\infty}f^{np}(x)$; the point $z$ is periodic and
$f^p(z)=z$.  Since $f$ is continuous, 
there exists $\delta>0$ such that
$$
\forall y,\ |y-z|\leq\delta\Rightarrow
\forall i\in\Lbrack 0,p-1\Rbrack,\ |f^i(y)-f^i(z)|\leq\eps.
$$
Let $N$ be an integer such that $|f^{np}(x)-z|\leq\delta$ for all $n\geq N$.
Then
$$
\forall m\geq Np,\ |f^m(x)-f^m(z)|\leq\eps.
$$
This completes the proof of (ii).
\end{proof}

Now we are ready to give the proof of Theorem~\ref{theo:htop0-chaos-LY},
which follows from 
Propositions \ref{prop:cantor-delta-scrambled} and
\ref{prop:non-LY-chaotic}. For clarity, we recall the statement of the
theorem.

\medskip
{\sc Theorem~\ref{theo:htop0-chaos-LY}}.
Let $f$ be an interval map of zero topological entropy. The following
properties are equivalent:
\begin{enumerate}
\item $f$ is chaotic in the sense of Li-Yorke,
\item there exists a $\delta$-scrambled Cantor set
for some $\delta>0$,
\item there exists a point $x$ that is not approximately periodic,
\item there exists an infinite $\omega$-limit set containing 
two $f$-non separable points.
\end{enumerate}

\begin{proof}
If (iv) does not hold, then, according to 
Proposition~\ref{prop:non-LY-chaotic}(ii), all points 
$x$ are approximately periodic.
By refutation, we get (iii)$\Rightarrow$(iv).

Suppose that (iv) holds, that is, there exists an infinite $\omega$-limit
set $\omega(x_0,f)$ containing two $f$-non separable points $a_0,a_1$. 
Then, according to
Lemma~\ref{lem:f-separable-Lni} and
Proposition~\ref{prop:cantor-delta-scrambled} applied with
$x_1=x_0$, there exists a $\delta$-scrambled Cantor set with
$\delta:=|a_1-a_0|$,
which is (ii). Obviously, (ii)$\Rightarrow$(i).

Suppose that (iii) does not hold, that is, every point is
approximately periodic. Then Lemma~\ref{lem:approx-periodic}
implies that there is no Li-Yorke pair. By refutation,
we get (i)$\Rightarrow$(iii).
\end{proof}

%*************************************************************************
\section{One Li-Yorke pair implies chaos in the sense of Li-Yorke}\label{sec:equivLY}

Kuchta and Smítal showed that, for interval maps,
the existence of one Li-Yorke pair of points is
enough to imply the existence of a  $\delta$-scrambled Cantor set \cite{KuS}. 
We give a different proof, suggested by Jiménez López, which follows
easily from Theorem~\ref{theo:htop0-chaos-LY}.

\begin{prop}\label{prop:one-LY-pair}
Let $f$ be an interval map. If
there exists one Li-Yorke pair, then 
there exists a $\delta$-scrambled Cantor for some $\delta>0$.
\end{prop}

\begin{proof}
If $h_{top}(f)>0$, the result follows from
Theorem~\ref{theo:htop-positive-chaos-LY}. 
We assume that $h_{top}(f)=0$. Let $(x,y)$ be a Li-Yorke pair.
By Lemma~\ref{lem:approx-periodic}, either $x$ or $y$ is not
approximately periodic. Therefore, the result is given by
the implication (iv)$\Rightarrow$(ii) in Theorem~\ref{theo:htop0-chaos-LY}.
\end{proof}

As a corollary, we get the following summary theorem. We shall see
another condition equivalent to chaos in the sense of Li-Yorke in the
next section.

\begin{theo}\label{theo:summary-chaos-LY}
Let $f$ be an interval map. The following properties are
equivalent:
\begin{enumerate}
\item there exists one Li-Yorke pair,
\item $f$ is chaotic in the sense of Li-Yorke,
\item $f$ admits a $\delta$-scrambled Cantor set for some $\delta>0$,
\item there exists a point $x$ that is not approximately periodic.
\end{enumerate}
\end{theo}

\begin{proof}
The first three assertions are equivalent by Theorem~\ref{prop:one-LY-pair}.
According to Lemma~\ref{lem:approx-periodic}, we have (ii)$\Rightarrow$(iv).
If $h_{top}(f)=0$, then (iv)$\Rightarrow$(ii) by
Theorem~\ref{theo:htop0-chaos-LY}. If $h_{top}(f)>0$, then the equivalence
follows from 
Theorem~\ref{theo:general-system-htop-positive-chaos-LY}.
\end{proof}

%*************************
\subsection*{Remarks on graph maps}

A key tool to generalize the results of the last two sections
to graphs is the
topological characterization of $\omega$-limit sets of graph maps, which
was given by Blokh \cite{Blo5,Blo13}; see also the more recent paper of
Hric and Málek \cite{HM} (the classification of $\omega$-limit sets in 
\cite{HM} is equivalent to the one in \cite{Blo5,Blo13}, although the 
equivalence is not straightforward and does not
seem to be explicitly proved in the literature). We rather follow
Blokh's works.

\begin{theo}\label{theo:omega-graph-htop0}
Let $f\colon G\to G$ be a graph map of zero topological entropy and $x\in G$.
If $\omega(x,f)$ is infinite, it is of one of the following kinds:

\noindent$\bullet$ 
\emph{Solenoidal}\index{solenoidal $\omega$-limit set}:
there exist a sequence of subgraphs $(G_n)_{n\ge 1}$ and an increasing
sequence of positive integers $(k_n)_{n\ge 1}$  such that 
$(f^i(G_n))_{0\le i<k_n}$ is a cycle of graphs
and, for all $n\ge 1$, $G_{n+1}\subset G_n$, $k_{n+1}$ is a multiple of
$k_n$ and $\omega(f^i(x), f^{k_n})\subset f^i(G_n)$ 
for all $i\in\Lbrack 0,k_n-1\Rbrack$ (which implies that
$\omega(x,f)\subset \bigcup_{i=0}^{k_n-1}f^i(G_n)$).

\noindent$\bullet$ 
\emph{Circumferential}\index{circumferential
$\omega$-limit set}: $\omega(x,f)$ contains no periodic point
and there exists a minimal cycle of graphs $(f^i(G'))_{0\le i<k}$ 
such that $\omega(x,f^k)\subset G'$ (which implies that
$\omega(x,f)\subset \bigcup_{i=0}^{k-1}f^i(G')$).
\end{theo}

Notice that a solenoidal set cannot contain periodic points, and thus,
for a zero entropy graph map, any infinite $\omega$-limit set contains
no periodic point. That is, Proposition~\ref{prop:omega-no-periodic-point}
is valid for graph maps too.

Blokh \cite{Blo9,Blo11}\nocite{Blo12,Blo13} 
showed that, in the case of a circumferential
$\omega$-limit set, $f^k|_{G'}$ is almost conjugate to an irrational
rotation, that is, semi-conjugate by a map
that collapses any connected component
of $G'\setminus \omega(x,f^k)$ to a single point.
In particular, this implies that a tree map has no circumferential set.

\begin{theo}\label{theo:circumferential}
Let $f\colon G\to G$ be a graph map and $x\in G$. Suppose that
$\omega(x,f)$ is circumferential, and let $(f^i(G'))_{0\le i<k}$ denote
the minimal cycle of graphs containing $\omega(x,f)$, with $G'\supset
\omega(x,f^k)$.
Then there exists an irrational rotation $R\colon \IS\to \IS$, 
and a semi-conjugacy $\varphi\colon G'\to \IS$
between $f^k|_{G'}$ and $R$ such that
\begin{itemize}
\item  $\varphi(\omega(x,f^k))=\IS$,
\item  $\forall y\in \IS$, $\varphi^{-1}(y)$ is connected,
\item  $\forall y\in \IS$, $\varphi^{-1}(y)\cap \omega(x,f^k)=
\End{\varphi^{-1}(y)}$.
\end{itemize}
\end{theo}

In \cite{RS2}, the author and Snoha studied chaos in the sense of
Li-Yorke for graph maps. We present the main ideas.
Suppose that $(x,y)$ is a Li-Yorke pair
for the graph map $f$. We showed that neither $\omega(x,f)$ nor
$\omega(y,f)$ can be circumferential \cite[proof of Theorem 3]{RS2}.
Moreover, it is easy to see that either $\omega(x,f)$ or
$\omega(y,f)$ is infinite. Therefore, if $h_{top}(f)=0$, one of these
$\omega$-limit sets is solenoidal. If $\omega(x,f)$ is solenoidal, with 
the notation of Theorem~\ref{theo:omega-graph-htop0}, then for all
large enough $n$, there exists $i\ge 0$ such that $J:=f^i(G_n)$ is
an interval (because the graph has finitely many
branching points, and thus one of the graphs $(f^i(G_n))_{0\le i<k_n}$ 
contains no branching point if $k_n$ is large enough). We have
$\omega(f^i(x),f^{k_n})\subset J$; in addition, it is possible to
show that one can choose $n,i$
such that $f^i(x),f^i(y)\in J$. Thus $(f^i(x),f^i(y))$ is a Li-Yorke pair
for $f^{k_n}|_J$, and we come down to the interval case.
On the other hand, 
Theorem~\ref{theo:general-system-htop-positive-chaos-LY} applies when
$h_{top}(f)>0$.
These ideas make it possible to show the following
result  \cite[Theorem~3]{RS2}. 

\begin{prop}\label{prop:graph-h0-LY}
Let $f\colon G\to G$ be a graph map. 
The following properties are equivalent:
\begin{enumerate}
\item there exists one Li-Yorke pair,
\item $f$ is chaotic in the sense of Li-Yorke,
\item there exists a $\delta$-scrambled Cantor set for some $\delta>0$.
\end{enumerate}
\end{prop}

\begin{rem}
Contrary to what happens for graph maps,
there exist topological dynamical systems
admitting a finite (resp. countable) scrambled set but no infinite
(resp. uncountable) scrambled set \cite{BDM}.
\end{rem}

%**********************************************************
\section{Topological sequence entropy}
Any positive entropy interval map is chaotic in the sense of Li-Yorke
(Theorem~\ref{theo:htop-positive-chaos-LY}), but the converse is not true 
(see Example~\ref{ex:chaos-LY-htop0} below). We are going to see that an
interval map is chaotic in the sense of Li-Yorke if and only if it has positive
topological sequence entropy.

\subsection{Definition of sequence entropy}
The notion of topological sequence entropy was introduced by
Goodman \cite{Goo}. Its definition is analogous to the one of
topological entropy,
the difference is that one considers a subsequence of the family of all
iterates of the map.
The definition we give is analogous to Bowen's formula 
(Theorem~\ref{theo:Bowen-formula}), but topological sequence entropy can
also be defined using open covers in a similar way as topological entropy
in Section~\ref{subsec:htop-covers}.

\begin{defi}\label{defi:toposeqentropy}
Let $(X,f)$ be a topological dynamical system and let $A=(a_n)_{n\ge 0}$
be an increasing sequence of non negative integers. 
Let $\eps>0$ and $n\in\IN$. 
A set $E\subset X$ is 
\emph{$(A,n,\eps)$-separated}\index{separated set ($(A,n,\eps)$- )} if
for all distinct points $x,y$ in $E$,
there exists $k\in\Lbrack 0,n-1\Rbrack$ such that $d(f^{a_k}(x),f^{a_k}(y))>\eps$.
Let $s_n(A,f,\eps)$
\label{notation:snAfeps}
\index{snAfe@$s_n(A,f,\eps)$}
denote the maximal cardinality of an $(A,n,\eps)$-separated set.
The set $E$ is an \emph{$(A,n,\eps)$-spanning set}\index{spanning set ($(A,n,\eps)$- )}
if for all $x\in X$, there exists $y\in E$ such that
$d(f^{a_k}(x),f^{a_k}(y))\le \eps$ for all $k\in\Lbrack 0,n-1\Rbrack$.
Let $r_n(A,f,\eps)$
\label{notation:rnAfeps}
\index{rnAfe@$r_n(A,f,\eps)$}
denote the minimal cardinality of an $(A,n,\eps)$-spanning set. 

The \emph{topological sequence entropy}\index{topological sequence entropy}\index{sequence entropy}
\label{notation:hAf}
\index{hAf@$h_A(f)$}
of $f$ with respect to the sequence $A$ is
$$
h_A(f):=\lim_{\eps\to0}\limsup_{n\to+\infty}\frac 1n\log s_n(A,f,\eps)=
\lim_{\eps\to0}\limsup_{n\to+\infty}\frac 1n\log r_n(A,f,\eps).
$$
\end{defi}

\begin{rem}
As in Lemma~\ref{lem:feps-separated}, we have
\begin{itemize}
\item if $0<\eps'<\eps$, then
$s_n(A,f,\eps')\ge s_n(A,f,\eps)$ and $r_n(A,f,\eps')\ge r_n(A,f,\eps)$,
\item $r_n(A,f,\eps)\le s_n(A,f,\eps)\le r_n(A,f,\frac{\eps}2)$.
\end{itemize}
This implies that the two limits in the definition above
exist and are equal. Therefore, $h_A(f)$ is well defined.
\end{rem}

According to the definition, $h_{top}(f)=h_A(f)$ with
$A:=(n)_{n\ge 0}$.

\subsection{Li-Yorke chaos and sequence entropy}

The rest of this section will be devoted to proving the following
theorem, due to Franzová and Smítal \cite{FS}.

\begin{theo}\label{theo:LY-equiv-hA>0}
Let $f$ be an interval map. Then $f$ is chaotic in the sense of Li-Yorke
if and only if there exists an increasing sequence $A$ such that
$h_A(f)>0$.
\end{theo}

The ``only if'' part can be easily shown by using previous results in this
chapter; this is done in Proposition~\ref{prop:LY-hA>0}. Before
proving the reverse implication, we shall need to show several
preliminary results; this will be done in Subsections \ref{subsec:wf}
and \ref{subsec:hA>0-LY}.

\subsection{Li-Yorke chaos implies positive sequence entropy}

\begin{prop}\label{prop:LY-hA>0}
If an interval map $f\colon I\to I$ is chaotic in the sense
of Li-Yorke, there exists an increasing sequence $A$ such that $h_A(f)>0$.
\end{prop}

\begin{proof}
If $f$ has positive topological entropy, then $h_A(f)=h_{top}(f)>0$
with $A:=(n)_{n\ge 0}$. From now on, we assume that $h_{top}(f)=0$.
According to Theorem~\ref{theo:htop0-chaos-LY}, there exist two $f$-non
separable points $a_0,a_1$ belonging to the same infinite $\omega$-limit set.
Let $J_0,J_1$ be two disjoint closed intervals such that $a_i\in\Int{J_i}$
for $i\in\{0,1\}$. By Lemma~\ref{lem:f-separable-Lni} and
Proposition~\ref{prop:cantor-delta-scrambled}, there exists
an increasing sequence of positive integers $A=(n_k)_{k\ge 0}$ such that
$$
\forall \bar\alpha=(\alpha_k)_{k\ge 0}\in\{0,1\}^{\IZ^+},\
\exists\, x_{\bar\alpha}\in I,\ \forall k\ge 0,\ 
f^{n_k}(x_{\bar\alpha})\in J_{\alpha_k}.
$$
For all $n\ge 1$, we set
$$
E_n:=\{x_{(\alpha_k)_{k\ge 0}}\mid
\forall k\ge n, \alpha_k=0\text{ and } \alpha_0,\ldots,
\alpha_{n-1}\in\{0,1\}\}.
$$
Let $\delta>0$ be the distance between $J_0$ and $J_1$. Then $E_n$
is an $(A,n,\eps)$-separated set for all $\eps\in(0,\delta)$, and
$\# E=2^n$. Thus $s_n(A,f,\eps)\ge 2^n$ for all $\eps\in(0,\delta)$,
and so $h_A(f)\ge \log 2>0$.
\end{proof}

%**********
\subsection{Preliminary results on $\omega$-limit set}\label{subsec:wf}

In this subsection, we are going to show several results concerning the
$\omega$-limit set of an interval map. These results are due to Sharkovsky 
\cite{Sha9}; see also \cite[Chapter IV]{BCop} (in
English). In \cite{MSu2}, Mai and Sun generalized these results to graph maps.
For Lemma~\ref{lem:wf-U+}, we follow the ideas of Mai and Sun
\cite[Proposition~2]{MSu2}, whose  proof is simpler.
Recall that the $\omega$-limit set of a map $f\colon I\to I$ is
\[
\omega(f):=\bigcup_{x\in I}\omega(x,f).
\]

\begin{rem}
In the previous version of this book (v4 on arxiv) as well as in
\cite{Sha9} and \cite[Proposition~IV.6]{BCop}, Lemma~\ref{lem:wf-U+} was
stated without the part (i)-(ii), and its proof 
was split into two cases, the second one being
more difficult to deal with. In the previous version of this book, the
main part of the second case was isolated in Lemma~5.43, whose 
assumptions were the following ones.

\medskip
{\it
Let $f\colon I\to I$ be an interval map and $c\in I\setminus \{\max I\}$.
Let $J_0:=[c,c']$ for some $c'>c$.
Suppose that:
\begin{gather}
\forall n\ge 1,\ \forall d>c,\ [c,d]\not\subset f^n(J_0),
\label{eq:assumption3}\\
\forall n\ge 1,\ f^n(c)\notin J_0.\label{eq:assumption4}
\end{gather}
Let $\CV_c:=
\{U\text{ nonempty open subinterval}\mid U\subset J_0, \inf U=c\}$.
Suppose that
\begin{equation}\label{eq:assumption5}
\forall\, U\in \CV_c,\ \exists k\ge 1,\  U\cap f^k(U)\ne\emptyset.
\end{equation}
}

\medskip
Actually, one can show, by using the proof of Lemma~\ref{lem:wf-U+},
that these assumptions are never satisfied, and thus the second case
of the previous proofs was void.
\end{rem}

\begin{lem}\label{lem:wf-U+}
Let $f\colon I\to I$ be an interval map and $c\in I\setminus\{\max I\}$
(resp. $c\in I\setminus\{\min I\}$).
Suppose that for every nonempty open interval $U$ such that $\inf U=c$
(resp. $\sup U=c$), 
\begin{equation}\label{eq:hyp-lem-:wf-U+}
\exists k\ge 1, U\cap f^k(U)\ne\emptyset.
\end{equation}
Then $c\in \omega(f)$. Moreover, one of the following statements holds:
\begin{enumerate}
\item There exists $x>c$ (resp. $x<c$) and an increasing sequence of
positive integers $(m_k)_{k\ge 0}$ such that for all $k\ge 0$, $f^{m_k}(x)>c$
(resp. $f^{m_k}(x)<c$) and $\lim_{k\to+\infty} f^{m_k}(x)=c$.
\item $c$ is a periodic point.
\end{enumerate}
\end{lem}

\begin{proof}
We deal only with the case $c\in I\setminus\max\{I\}$ and
$\inf U=c$, the other case being symmetric. 
First we suppose
\begin{equation}\label{eq:preliminary-hA-1}
\forall\eps>0, \exists a,b\in(c,c+\eps]\text{ with }a<b,\
\exists d>c, \exists n\ge 1,f^n([a,b])=[c,d].
\end{equation}%\nolabel
A straightforward induction shows that there exist decreasing sequences of
points $(a_k)_{k\ge 0}$, $(b_k)_{k\ge 0}$ and a sequence of positive
integers $(n_k)_{k\ge 0}$ such that, for all $k\ge 0$,
$c<a_k<b_k$, $\lim_{k\to+\infty}b_k=c$ and 
$f^{n_k}([a_k,b_k])\supset [a_{k+1},b_{k+1}]$. Thus
we have the following coverings:
\begin{equation}%\nolabel
[a_0,b_0]\cover{f^{n_0}}[a_1,b_1]\cover{f^{n_1}}[a_2,b_2]
\cover{f^{n_2}}\cdots [a_k,b_k]\cover{f^{n_k}}[a_{k+1},b_{k+1}]\cdots.
\end{equation}
Using inductively Lemma~\ref{lem:chain-of-intervals}, we can build a sequence
of closed intervals $(I_k)_{k\ge 0}$ such that 
\begin{equation}%\nolabel
I_0:=[a_0,b_0],\ I_{k+1}\subset I_k\text{ and }\forall k\ge 0,\ 
f^{n_0+\cdots+n_k}(I_k)=[a_{k+1},b_{k+1}].
\end{equation}
Let $x$ be in $\bigcap_{k\ge 0} I_k$ (this set is nonempty because
it is a decreasing intersection of nonempty compact sets). For all $k\ge 0$,
we set $m_k:=n_0+\cdots+n_k$. Then
$f^{m_k}(x)\in[a_{k+1},b_{k+1}]$ for all $k\ge 0$. Thus
$\lim_{k\to+\infty}f^{m_k}(x)=c$, which implies 
that $c\in \omega(x,f)$. Moreover, $f^{m_k}(x)>c$ for all $k\ge 0$
and statement (i) holds.

From now on, we assume that \eqref{eq:preliminary-hA-1} does not hold.
We are going to show that $c$ is a periodic point, that is, statement (ii)
holds, which implies that $c\in \omega(c,f)$. Assume on the contrary that
$c$ is not a periodic point.
The negation of \eqref{eq:preliminary-hA-1} means that there exists
$\eps>0$ such that:
\begin{equation}\label{eq:assumption3b}
\forall a,b\in(c,c+\eps]\text{ with }a<b, \forall d>c, \forall k\ge 1,
f^k([a,b])\ne [c,d].
\end{equation}
We set $J:=[c,c+\eps]$. We first prove the following fact:
\begin{equation}\label{eq:fact44}
\forall m\ge 1,\exists \delta_m>0,\ (c,c+\delta_m)\cap f^m(J)=\emptyset.
\end{equation}
Suppose that the contrary holds, that is, there exists $m\ge 1$ such that,
for all $\delta>0$, $(c,c+\delta)\cap f^m(J)\ne\emptyset$. This
implies that $c\in\overline{f^m(J)}$, and thus $c\in f^m(J)$
because $f^m(J)$ is compact. Let $c'\in (c,c+\eps)\cap f^m(J)$.
Then $[c,c']\subset f^m(J)$ because $f^m(J)$ is connected. We choose
a sequence of points $y_n\in (c,c')$ such that $\lim_{n\to+\infty} y_n=c$.
For all $n\ge 0$, there exists $x_n\in J$ such that $f^m(x_n)=y_n$.
By taking a subsequence, we may assume that the sequence $(x_n)_{n\ge 0}$
converges to a point $x$, and $x\in J$ by compactness. Then $f^m(x)=c$
by continuity, $x\ne c$ because $c$ is not periodic by
assumption, and for all $n\ge 0$, $x\ne x_n$ because $f^m(x_n)\ne c$.
Since $\lim_{n\to+\infty}x_n=x>c$, there exists $n\ge 0$ such that $x_n>c$.
We set
\begin{eqnarray*}
x'&:=&\max\{t\in [x,x_n]\mid f^m(t)=c\}\quad\text{if } x<x_n,\\
x'&:=&\min\{t\in [x_n,x]\mid f^m(t)=c\}\quad\text{if } x>x_n.
\end{eqnarray*}
Then $f^m(\langle x,x'\rangle)\ge c$ by definition of $x'$ and continuity 
of $f$. Moreover, $f^m(\langle x,x'\rangle)$ contains
both $f^m(x')=c$ and $f^m(x_n)=y_n>c$. Thus there exists $d>c$ such that
$f^m(\langle x,x'\rangle)=[c,d]$. We set $\{a,b\}:=\{x,x'\}$ with $a<b$.
Note that $a>c$ because $x>c$ and $x_n>c$.
Then $f^m([a,b])=[c,d]$ with $a,b\in J\setminus\{c\}=(c,c+\eps]$, which contradicts \eqref{eq:assumption3b}.
This proves that the fact \eqref{eq:fact44} holds.

We set 
\begin{equation}%\nolabel
Y:=\bigcup_{n=1}^{\infty}f^n(J).
\end{equation}
According to the
assumption \eqref{eq:hyp-lem-:wf-U+}, for all $\eps'\in(0,\eps]$, 
there exists an
integer $k\ge 1$ such that $(c,c+\eps')\cap f^k((c,c+\eps'))\ne\emptyset$. Thus
for all $\eps'\in(0,\eps]$, $(c,c+\eps')\cap Y\ne\emptyset$, 
which implies that $c\in
\overline{Y}$. On the other hand, the fact \eqref{eq:fact44} implies that
$c\notin Y$. Thus $c\in \overline{Y}\setminus Y$. Moreover, according to the
assumption \eqref{eq:hyp-lem-:wf-U+}, 
there exists $k\ge 1$ such that $(c,c+\eps)\cap
f^k((c,c+\eps))\ne\emptyset$, and thus $J\cap f^k(J)\ne\emptyset$. This
implies that $Y$ has at most $k$ connected components and that
$\overline{Y}\setminus Y$ is a finite set. By definition,
$Y=f(Y)\cup f(J)$. Thus, since $J$ is compact, $\overline{Y}=
f(\overline{Y})\cup f(J)$. Moreover,
\begin{eqnarray}
\overline{Y}\setminus Y&=&\left(f(\overline{Y})\cup f(J)\right)\setminus Y,\\%\nolabel
&=&f(\overline{Y})\setminus Y\quad\text{because } f(J)\subset Y,\\%\nolabel
&\subset&f(\overline{Y})\setminus f(Y)\quad\text{because } f(Y)\subset Y,\\%\nolabel
&\subset&f(\overline{Y}\setminus Y).\label{eq:Ybar-Y}
\end{eqnarray}
Since $\overline{Y}\setminus Y$ is finite, \eqref{eq:Ybar-Y} implies that
$\overline{Y}\setminus Y= f(\overline{Y}\setminus Y)$ and all points
in $\overline{Y}\setminus Y$ are periodic. Since $c\in\overline{Y}\setminus Y$,
this contradicts the fact that $c$ is not periodic. Conclusion:
if \eqref{eq:preliminary-hA-1} does not hold, then $c$ is periodic. This
concludes the proof.
\end{proof}

The next result gives a characterization of the points in the
$\omega$-limit set. Note that its statement is not optimal since one can
replace the bound $4$ by $3$ in \eqref{eq:assumption-prop-wf}. 
Since the value of this bound has no consequence on the other results of
the book, we have chosen to give a simple proof with a non optimal bound.
To prove this result with the bound $3$, one can either use additional 
lemmas about interval maps (which
gives a longer proof) as in
\cite[Proposition V.11]{BCop}, or use Sierpiński's 
Theorem\footnote{Sierpiński's Theorem: If $(F_n)_{n\ge 0}$ is a pairwise
disjoint closed cover of the compact connected Hausdorff set $S$, then 
there exists $n\ge 0$ such that $F_n=S$. 
See e.g. \cite[Theorem 6.1.27]{Eng}.}
(which gives a short but non elementary proof) as in
\cite[Theorem~2]{MSu2}.

\begin{prop}\label{prop-wf}
Let $f\colon I\to I$ be an interval map and $c\in I$. Then $c\in \omega(f)$
if and only if
\begin{equation}\label{eq:assumption-prop-wf}
\text{for every neighborhood }U\text{ of }c,\ \exists x\in I,\
\#\{n\ge 0,\mid f^n(x)\in U\}\ge 4.
\end{equation}
\end{prop}

\begin{proof}
If $c\in \omega(f)$, there exists $x$ such that
$c\in \omega(x,f)$, and we trivially have 
$\#\{n\ge 0,\mid f^n(x)\in U\}\ge 4$ for every neighborhood $U$ of $c$.

Assume that \eqref{eq:assumption-prop-wf} holds. For every set $U\subset I$, 
we define
$$
U^-:=\{x\in U\mid x<c\}\quad\text{and}\quad U^+:=\{x\in U\mid x>c\}.
$$
We assume
\begin{equation}\label{eq:prop-wf-21}
\exists\, U\text{ neighborhood of }c,\ \forall k\ge 1,\
U^-\cap f^k(U^-)=\emptyset\text{ and }U^+\cap f^k(U^+)=\emptyset.
\end{equation}
Let $U$ be such a neighborhood.
We also assume that $c\notin\omega(c,f)$ (otherwise there is nothing to prove).
In this way, we may replace $U$ by a smaller neighborhood in order to have
\begin{equation}\label{eq:prop-wf-22}
\forall k\ge 1,\ f^k(c)\notin U.
\end{equation}

By assumption~\eqref{eq:assumption-prop-wf}, there exist a point $x$
and positive integers $p<q<r$ such that $x,f^p(x),f^q(x),f^r(x)\in U$. 
By \eqref{eq:prop-wf-22}, the point $x$ is not equal to $c$ because 
$f^p(x)\in U$. Thus $x\in U^-\cup U^+$. 
We suppose $x\in U^+$, the case $x\in U^-$ being symmetric. 
Similarly, \eqref{eq:prop-wf-22} implies that $f^p(x)\ne c$ because $f^{q-p}(f^p(x))\in U$;
and $f^q(x)\ne c$ because $f^{r-q}(f^q(c))\in U$.
Since $f^p(x)\in f^p(U^+)$ and $f^p(x)\in U\setminus\{c\}$,
we have $f^p(x)\in U^-$ by \eqref{eq:prop-wf-21}.
The same argument with the points $x,f^q(x)$ (resp. 
$x':=f^p(x), f^{q-p}(x')=f^q(x)$) leads to $f^q(x)\in U^-$ (resp. 
$f^q(x)\in U^+$), which is impossible. Thus \eqref{eq:prop-wf-21} does
not hold. It is easy to see, by
considering the neighborhoods $(c-\eps_k,c+\eps_k)\cap I$, where
$(\eps_k)_{k\ge 0}$ is a decreasing sequence of positive numbers tending to 
$0$, that 
there exists $s\in\{+,-\}$ such that, for every neighborhood $U$ of 
$c$, there exists $k\ge 1$ such that $U^s\cap f^k(U^s)\ne\emptyset$. Then
$c\in\omega(f)$ according to Lemma~\ref{lem:wf-U+}.
\end{proof}

\begin{cor}\label{cor:wf-closed}
Let $f\colon I\to I$ be an interval map. The set $\omega(f)$ is compact.
\end{cor}

\begin{proof}
Let $(c_n)_{n\ge 0}$ be a sequence of points in $\omega(f)$ that converges
to some point $c$. Let $U$ be a neighborhood of $c$. There
exists $n\ge 0$ such that $U$ is a neighborhood of $c_n$. Thus, according
to Proposition~\ref{prop-wf}, there exists a point $x$ such that
$\#\{n\ge 0,\mid f^n(x)\in U\}\ge 4$. Then, by Proposition~\ref{prop-wf},
$c\in \omega(f)$. This shows that $\omega(f)$ is closed, and hence compact
because $I$ is compact.
\end{proof}

By definition, for every open set $U$ containing $\omega(f)$ and
every point $x$, all but finitely many points of the trajectory of $x$
lie in $U$. The next result states that the number of points of the
trajectory of $x$ falling outside $U$ is bounded independently of $x$.

\begin{cor}\label{cor:Usupsetwf}
Let $f\colon I\to I$ be an interval map. For every open set $U$ containing
$\omega(f)$, there exists a positive integer $N$ such that, for all
points  $x\in I$, $\#\{n\ge 0\mid f^n(x)\notin U\}\le N$.
\end{cor}

\begin{proof}
Let $U$ be an open set containing $\omega(f)$. Let $y\in I\setminus U$.
According to Proposition~\ref{prop-wf}, there exists an open set $V_y$ 
containing $y$ such that $V_y$ contains at most three points of any 
trajectory. 
%Moreover, $\omega(f)$ is closed by Corollary~\ref{cor:wf-closed},
%and thus $V_y$ can be chosen small enough so that 
%$V_y\cap \omega(f)=\emptyset$. 
Since $I\setminus U$ is compact,
there exist finitely many points $y_1,\ldots, y_p\in I\setminus U$ such that
$I\setminus U\subset V_{y_1}\cup\cdots \cup V_{y_p}$. Therefore,
the open set $V:=V_{y_1}\cup\cdots \cup V_{y_p}$ contains at most $3p$ points
of any trajectory. This gives the conclusion with $N=3p$.
\end{proof}

%*********
\subsection{Positive sequence entropy implies Li-Yorke chaos}\label{subsec:hA>0-LY}

We are going to show several preliminary results about interval maps that 
are not chaotic in the sense of Li-Yorke. Then we shall be able to show that 
such a map has zero sequence entropy for any sequence.

The next lemma is a refinement of Proposition~\ref{prop:non-LY-chaotic}(i).

\begin{lem}\label{lem:diamLcapomega}
Let $f$ be an interval map that is not chaotic in the sense of
Li-Yorke and let $x_0$ be a point. 
Suppose that $\omega(x_0,f)$ is infinite and let 
$(L_n)_{n\ge 0}$ be the intervals given by 
Proposition~\ref{prop:htop0-Lki}.
Then
$$
\lim_{n\to+\infty}
\max_{i\in\Lbrack 0,2^n-1\Rbrack}\diam(f^i(L_n)\cap \omega(f))=0.
$$
\end{lem}

\begin{proof}
Recall that the intervals $(L_n)_{n\ge 0}$ satisfy: for all $n, i\ge 0$,
$f^i(L_{n+1})$ and $f^{i+2^n}(L_{n+1})$ are included in $f^i(L_n)$,
and  $(f^i(L_n))_{0\le i<2^n}$ is the smallest 
cycle of intervals of period $2^n$ containing $\omega(x_0,f)$.
Suppose that the lemma does not hold; this implies
\begin{equation}\label{eq:no-diamLcapomega}
\exists\delta>0,\ \forall n\ge 0,\ \exists i\in\Lbrack 0,2^n-1\Rbrack,\
\diam(f^i(L_n)\cap \omega(f))\ge\delta.
\end{equation}
Using \eqref{eq:no-diamLcapomega},
one can build a sequence $(i_n)_{n\ge 0}$ such that
$$
\forall n\ge 0,\ f^{i_{n+1}}(L_{n+1})\subset f^{i_n}(L_n)\text{ and }
\diam(f^{i_n}(L_n)\cap\omega(f))\ge\delta.
$$
We set $J_n:=f^{i_n}(L_n)$. For every $n\ge 0$, let $b_n,c_n$ be two points
in $J_n\cap \omega(f)$ such that $|b_n-c_n|\ge\delta$. By compactness, there
exist two points 
$b,c\in I$ and an increasing sequence of integers $(n_k)_{k\ge 0}$ such that 
$\lim_{k\to+\infty}b_{n_k}=b$ and $\lim_{k\to+\infty}c_{n_k}=c$.
Since $\omega(f)$ is closed by Corollary~\ref{cor:wf-closed}, the points
$b,c$ belong to $\omega(f)$. Moreover, $[b-c|\ge \delta$ and $b,c$ belong
to $J_n$ for all $n\ge 0$ (because $(J_n)_{n\ge 0}$ is a decreasing sequence of
closed intervals).

According to Proposition~\ref{prop:non-LY-chaotic}(i), $\diam(J_n\cap
\omega(x_0,f))$ tends to $0$ when $n$ goes to infinity. Thus
there exists a unique point $a\in\omega(x_0,f)$ such that
$$
\bigcap_{n\ge 0}J_n\cap \omega(x_0,f)=\{a\}
$$
because this is a decreasing intersection of nonempty compact sets.
By the triangular inequality, either $|a-b|\ge \frac{\delta}2$ or
$|a-c|\ge \frac{\delta}2$. With no loss of generality, we suppose
$|a-b|\ge \frac{\delta}2$.
For every $n\ge 0$, the point $b$ is in the interval 
$J_n$, which belongs to a periodic cycle of intervals of period $2^n$.
This implies that the points $(f^k(b))_{k\ge 0}$ are all distinct.
Therefore, 
since $b\in \omega(f)$, there exists a point $x_1$ such that $b\in\omega(x_1,f)$
and $\omega(x_1,f)$ is infinite. Moreover, since $b, f^{2^n}(b)$,
$f^{2^{n+1}}(b)$ are three distinct points in the interval $J_n$, one of
them is in $\Int{J_n}$, and thus there exists $k\ge 0$ such that
$f^k(x_1)\in J_n$. The periodicity of $J_n$ implies that
$\omega(f^k(x_1),f^{2^n})\subset J_n$. Then, by Lemma~\ref{lem:omega-set}, 
we get
$$
\forall n\ge 0,\ \omega(x_1,f)\subset \bigcup_{i=0}^{2^n-1}f^i(L_n).
$$
Since $f$ is not chaotic in the sense of Li-Yorke by assumption,
we have $h_{top}(f)=0$ by Theorem~\ref{theo:htop-positive-chaos-LY}.
Then the assumptions of Proposition~\ref{prop:cantor-delta-scrambled} 
are fulfilled (with $a_0:=a, a_1:=b$), and this proposition implies 
that $f$ is chaotic in the sense of Li-Yorke,
a contradiction. This ends the proof of the lemma.
\end{proof}

The next result is due to Fedorenko, Sharkovsky and 
Smítal \cite[Theorem~2.1]{FSS}.
Recall that the notion of an unstable point is
defined in Definition~\ref{defi:unstable}.

\begin{prop}\label{prop:notLY-wfstable}
Let $f$ be an interval map that is not chaotic in the sense of
Li-Yorke. Then $f|_{\omega(f)}$ has no unstable point, that is,
$$
\forall a\in\omega(f),\; \forall \eps>0,\; \exists\delta>0,\; \forall b\in
\omega(f),\; |a-b|\le\delta\Rightarrow \forall n\ge 0,\; |f^n(a)-f^n(b)|\le\eps.
$$
\end{prop}

\begin{proof}
According to Theorem~\ref{theo:htop-positive-chaos-LY}, $h_{top}(f)=0$.
We fix $\eps>0$ and $a\in \omega(f)$. Let $x_0$ be a point such that $a\in
\omega(x_0,f)$. We split the proof depending on $\omega(x_0,f)$ being
finite or infinite.

First we suppose that $\omega(x_0,f)$ is infinite. Let 
$(L_n)_{n\ge 0}$ be the closed intervals given by 
Proposition~\ref{prop:htop0-Lki}: for all $n, i\ge 0$,
$f^i(L_{n+1})$ and $f^{i+2^n}(L_{n+1})$ are included in $f^i(L_n)$,
and  $(f^i(L_n))_{0\le i<2^n}$ is the smallest 
cycle of intervals of period $2^n$ containing $\omega(x_0,f)$.
By Lemma~\ref{lem:diamLcapomega}, 
there exists $n\ge 0$ such that
\begin{equation}\label{eq:diamLcapomega}
\forall i\in\Lbrack 0,2^n-1\Rbrack,\ \diam(f^i(L_n)\cap \omega(f))<\eps.
\end{equation}
We set $J:=f^j(L_n)$, where $j\in\Lbrack 0,2^n-1\Rbrack$ is such that 
$a\in f^j(L_n)$. Since the four intervals $(f^{j+i2^n}(L_{n+2}))_{0\le i\le 3}$
are pairwise disjoint and included in $J$, one of them is included
in $\Int{J}$. Since $a$ is in one of these four intervals,
there exists $i\in\Lbrack 0,3\Rbrack$ such that
$a':=f^{i2^n}(a)$ belongs to $\Int{J}$. Let $\delta>0$ be such that
$(a'-\delta,a'+\delta)\subset J$. Let $b\in \omega(f)$ be such that 
$|a'-b|<\delta$.
Then $b$ belongs to $J\cap \omega(f)$, so $f^k(b)\in f^k(J)\cap \omega(f)$
for all $k\ge 0$ (the set $\omega(f)$ is invariant by
Lemma~\ref{lem:omega-set}(vi)). Using \eqref{eq:diamLcapomega}
and the fact that $f^{2^n}(J)=J$, we get
$$
\forall k\ge 0,\ |f^k(a')-f^k(b)|<\eps.
$$
We deduce that $a'$ is not $\eps$-unstable for the map $f|_{\omega(f)}$.
Consequently, $a$ is not $\eps$-unstable for the map $f|_{\omega(f)}$ 
by Lemma~\ref{lem:unstable}(iii), and this holds for any $\eps>0$.

Now, we suppose that $\omega(x_0,f)$ is finite, that is, $a$ is
a periodic point (by Lemma~\ref{lem:omega-finite}). Let $p$ denote the
period of $a$ and $g:=f^p$. In this way, $a$ is a fixed point for $g$.
We are going to prove that $a$ is not
unstable for $g|_{\omega(g)}$, which is enough to ensure that $a$ is not
unstable for $f|_{\omega(f)}$ according to 
Lemma~\ref{lem:unstable}(ii) and the fact that $\omega(f)=\omega(g)$ 
(Lemma~\ref{lem:omega-set}(vii)).
Since $g$ is uniformly continuous, there exist $\eps_1,\eps_2$ such that
$0<\eps_2<\eps_1<\eps$ and
\begin{eqnarray}
\forall x,y,\ |x-y|<\eps_1&\Rightarrow& \forall i\in\Lbrack 0,3\Rbrack,\ 
|g^i(x)-g^i(y)|<\eps,\label{eq:eps1}\\
\forall x,y,\ |x-y|<\eps_2&\Rightarrow& \forall i\in\Lbrack 0,3\Rbrack,\ 
|g^i(x)-g^i(y)|<\eps_1.\label{eq:eps2}
\end{eqnarray}
Let $b$ be in $\omega(g)\cap (a-\eps_2,a+\eps_2)$. If $g^2(b)=b$, then
\eqref{eq:eps2} implies
$$
\forall n\ge 0,\ \forall i\in\{0,1\},\ 
|g^{2n+i}(a)-g^{2n+i}(b)|=|g^i(a)-g^i(b)|<\eps_1<\eps.
$$ 
From now on, we suppose that $g^2(b)\ne b$. Let $x_1$ be a point such that
$b\in \omega(x_1,g)$; by assumption, $\omega(x_1,g)$ contains more than
$2$ points. Then,
according to Proposition~\ref{prop:htop0-Lki}, there exists a closed
interval $L_2$ such that $(g^i(L_2))_{0\le i\le 3}$ is a cycle of
intervals of period $4$ for $g$ 
and $\omega(x_1,g)\subset \bigcup_{i=0}^3 g^i(L_2)$. We can choose
$L_2$ such that $b\in L_2$. By \eqref{eq:eps2},  $|a-g^i(b)|<\eps_1$
for every $i\in\Lbrack 0,3\Rbrack$, so
$g^i(L_2)\cap (a-\eps_1,a+\eps_1)\ne\emptyset$. 
Since the four intervals $(g^i(L_2))_{0\le i\le 3}$ are
pairwise disjoint, there exists $i_0\in\Lbrack 0,3\Rbrack$ such that
$g^{i_0}(L_2)\subset (a-\eps_1,a+\eps_1)$. Then \eqref{eq:eps1} 
implies that $g^{i_0+j}(L_2)\subset (a-\eps,a+\eps)$ for all 
$j\in\Lbrack 0,3\Rbrack$. Since $L_2$
is a periodic interval of period $4$ for $g$, we get
$$
\forall n\ge 0,\ g^n(L_2)\subset (a-\eps,a+\eps).
$$
Moreover, $g^n(b)$ is in $g^n(L_2)$, so $|a-g^n(b)|<\eps$ for
all $n\ge 0$. We deduce that $a$ is not $\eps$-unstable for
$g|_{\omega(g)}$, for any $\eps>0$. This concludes the proof.
\end{proof}

The next lemma is due to Fedorenko, Sharkovsky and Smítal \cite{FSS2};
see also \cite[Theorem~3.13]{SMR2} for a statement in English.

\begin{lem}\label{lem:noLYchaos-wfalmostperiodic}
Let  $f$ be an interval map that is not chaotic in the sense of
Li-Yorke. Then every point in $\omega(f)$ is almost periodic\index{almost periodic point}, that is,
$$
\forall y\in\omega(f),\ \forall\, U\text{ neighborhood of }y,\ \exists p\ge 1,\
\forall n\ge 0,\ f^{np}(y)\in U.
$$
\end{lem}

\begin{proof}
Let $y$ belong to $\omega(x_0,f)$ for some point $x_0$. 
If $\omega(x_0,f)$ is finite, then
$y$ is periodic (Lemma~\ref{lem:omega-finite}) and the conclusion is
trivial with $p$ the period of $y$.
Suppose that $\omega(x_0,f)$ is infinite and let $U$ be a neighborhood of
$y$. For every $k\ge 0$, we set
$$
I_k:=\left[\min\omega(f^{i_k}(x_0),f^{2^k}),\max\omega(f^{i_k}(x_0),f^{2^k})
\right],
$$
where $i_k$ is an integer such that $y\in \omega(f^{i_k}(x_0),f^{2^k})$
(such an integer exists by  Lemma~\ref{lem:omega-set}(iv)).
According to Theorem~\ref {theo:htop0-chaos-LY} and 
Proposition~\ref{prop:non-LY-chaotic}(i),
$$
\lim_{k\to +\infty} |I_k|=0.
$$
This implies that there exists $k\ge 0$ such that $I_k\subset U$
because $y$ belongs to $I_k$ for all $k$. Moreover,
$f^{2^k}(\omega(f^{i_k}(x_0),f^{2^k}))=\omega(f^{i_k}(x_0),f^{2^k})$
(Lemma~\ref{lem:omega-set}(i)), so
$f^{n2^k}(y)\in I_k\subset U$ for all $n\ge 0$. This is the expected
result with $p:=2^k$.
\end{proof}

The next lemma will be a key tool in the proof of Theorem~\ref{theo:hA>0-LY};
it is due to Franzová and Smítal \cite{FS}.

\begin{rem}
In the proof in \cite{FS}, the fact that the open sets must satisfy 
\eqref{eq:xnotinwf-uniquei}  is omitted, although the proof does not work
without this condition.
\end{rem}

\begin{lem}\label{lem:approximate-xinU-yi}
Let  $f$ be an interval map that is not chaotic in the sense of
Li-Yorke. Let $\eps>0$. Then there exist finitely many points $y_1,\ldots, y_r$
in $\omega(f)$ and an open set $U$ containing $\omega(f)$ such that, for
every point $x$ satisfying
$$
\exists N_0,N_1\in\IZ^+, N_0\le N_1,\ \forall n\in\Lbrack N_0, N_1\Rbrack, \ f^n(x)\in U,
$$
then there exists $i\in\Lbrack 1,r\Rbrack$ such that
$$
\forall n\in\Lbrack N_0, N_1\Rbrack,\ |f^n(x)-f^n(y_i)|\le\eps.
$$
\end{lem}

\begin{proof}
According to Proposition~\ref{prop:notLY-wfstable}, for every 
$x\in\omega(f)$ there exists a connected neighborhood $W(x)$ of $x$ such that
\begin{equation}\label{eq:stableWy}
\forall z\in W(x)\cap\omega(f),\ \forall n\ge 0,\ |f^n(x)-f^n(z)|\le
\frac{\eps}2.
\end{equation}
Since $\omega(f)$ is compact by Corollary~\ref{cor:wf-closed}, there
exist finitely many distinct points $x_1,\ldots, x_s$ in $\omega(f)$ such that
$\omega(f)\subset W(x_1)\cup\cdots\cup W(x_s)$.
We would like these sets not to overlap too much, so we replace them
by smaller but more numerous sets. We define inductively
on $k\in\Lbrack 1,s\Rbrack$ a family of connected open sets 
$(W_k^j)_{1\le j\le \alpha_k}$ that are subsets of $W(x_k)$,
and points $(x_k^j)_{1\le j\le \alpha_k}$ such that 
$x_k^j\in W_k^j\cap\omega(f)$.

\textbf{Construction at step \boldmath $k\in\Lbrack 1,s\Rbrack$.}
Suppose that $(W_i^j)_{1\le j\le \alpha_i}$ and $(x_i^j)_{1\le j\le \alpha_i}$
have been defined for all $i\le k-1$ (for $k=1$, these two families are
empty).
We consider all the connected components $C$ of
\[
W(x_k)\setminus\{x_i^j\mid i\in\Lbrack 1,k-1\Rbrack,
j\in\Lbrack 1,\alpha_i\Rbrack\}
\]
such that
\[
\left(C\setminus \left(\bigcup_{i\in\Lbrack 1,k-1\Rbrack,
j\in\Lbrack 1,\alpha_i\Rbrack}W_i^j\right)\right)\cap \omega(f)\ne\emptyset.
\]
We call them $W_k^1,\ldots, W_k^{\alpha_k}$ (notice that $W(x_k)\setminus
\{x_i^j\mid i\in\Lbrack 1,k-1\Rbrack,
j\in\Lbrack 1,\alpha_i\Rbrack\}$ has finitely many connected components
because $W(x_k)$ is connected and the set $\{x_i^j\mid i\in\Lbrack 1,k-1\Rbrack,
j\in\Lbrack 1,\alpha_i\Rbrack\}$ is finite).
For every 
$j\in\Lbrack 1,\alpha_k\Rbrack$, we choose a point $x_k^j$ in
\[
\left(W_k^j\setminus \left(\bigcup_{i\in\Lbrack 1,k-1\Rbrack,
j\in\Lbrack 1,\alpha_i\Rbrack}W_i^j\right)\right)\cap \omega(f).
\]
This ends the construction at step $k$. Note that
\[
\bigcup_{i\in\Lbrack 1,k\Rbrack,j\in\Lbrack 1,\alpha_i\Rbrack}W_i^j
\cap\omega(f)=\bigcup_{i=1}^kW(x_i)\cap \omega(f).
\]

To simplify the notation, we call $V_1,\ldots, V_r$ and $y_1,\ldots, y_r$
the family of sets $(W_i^j)_{i\in\Lbrack 1,s\Rbrack,
j\in\Lbrack 1,\alpha_i\Rbrack}$ and the associated points
$(x_i^j)_{i\in\Lbrack 1,s\Rbrack,
j\in\Lbrack 1,\alpha_i\Rbrack}$, and we order them in order to have
$y_1<y_2<\cdots<y_r$.
Then $V_i$ is a connected open set containing $y_i\in\omega(f)$,
$V_i$ is included in $W(x_j)$ for some $j$, and
$\omega(f)\subset V_1\cup\cdots \cup V_r$.
Moreover, the construction above ensures that:
\begin{equation}\label{eq:WicapWj}
\forall i,j\in\Lbrack 1,r\Rbrack, i\ne j, V_i\cap V_j\subset
\langle y_i,y_j\rangle
\end{equation}
because $V_i$ (resp. $V_j$) is an interval and does not contain 
$y_j$ (resp. $y_i$). This implies that $V_i\cap V_j=\emptyset$
if $|i-j|\ge 2$ (that is, only intervals corresponding to
consecutive points may intersect).

We modify once more these sets
by an inductive construction for $i=1,\ldots, r-1$:\\
$\bullet$ if $V_i\cap V_{i+1}$ is not included in $\omega(f)$,
we choose a point $x\in (V_i\cap V_{i+1})\setminus\omega(f)$ and we
replace $V_i$ and $V_{i+1}$ by $V_i\cap (-\infty,x)$ and 
$V_{i+1}\cap (x,+\infty)$ respectively; we still call these sets $V_i$
and $V_{i+1}$;\\
$\bullet$ if $V_i\cap V_{i+1}\subset\omega(f)$, we do not change the sets
at step $i$.

At the end of this construction, we get intervals $V_1,\ldots, V_r$ that
are open set and satisfy:
\begin{gather*}
\omega(f)\subset V_1\cup\cdots V_r,\\
\forall i\in\Lbrack 1,r\Rbrack,\ y_i\in V_i\cap \omega(f),\\
\forall i,j\in\Lbrack 1,r\Rbrack,\ i\ne j,\ V_i\cap V_j\subset \omega(f).
\end{gather*}
This last condition implies:
\begin{equation}
\forall x\in \bigcup_{i=1}^r V_i,\ x\notin \omega(f)\Longrightarrow
\exists ! i\in\Lbrack 1,r\Rbrack,\ x\in V_i.\label{eq:xnotinwf-uniquei}
\end{equation}
Moreover, since $V_i\subset W(x_j)$ for some $j\in\Lbrack 1,s\Rbrack$,
the triangular inequality and \eqref{eq:stableWy} imply:
\begin{equation}\label{eq:stableVi}
\forall i\in \Lbrack 1,r\Rbrack,\ \forall y,z\in V_i\cap\omega(f),\ 
\forall n\ge 0,\ |f^n(y)-f^n(z)|\le\eps.
\end{equation}
Let $i\in\Lbrack 1,r\Rbrack$ and $z\in V_i\cap\omega(f)$. According to
Lemma~\ref{lem:noLYchaos-wfalmostperiodic}, there exists a positive
integer $p_i(z)$ such that 
\begin{equation}\label{eq:piz}
\forall n\ge 0,\ f^{np_i(z)}(z)\in V_i.
\end{equation}
We can assume that
\begin{equation}\label{eq:pimultiple}
p_i(z)\text{ is a multiple of }p_i(y_i).
\end{equation}
Since $f$ is continuous, there exists an open neighborhood $U_i(z)$ of $z$
such that:
\begin{gather}
U_i(z)\subset V_i\nonumber\\
f^{p_i(z)}(U_i(z))\subset V_i\label{eq:fpUisubsetVi}\\
\forall n\in\Lbrack 0, p_i(z)\Rbrack,\ \diam \left(f^n(U_i(z))\right)\le
\eps.\label{eq:diamfnUi}
\end{gather}
We set 
$$
U_i:=\bigcup_{z\in \omega(f)\cap V_i}U_i(z)\quad\text{and}\quad
U:=\bigcup_{i=1}^r U_i.
$$
The sets $U_i$ are open and satisfy:
\begin{equation}\label{eq:Uiwf-Viwf}
\forall i\in\Lbrack 1,r\Rbrack,\ U_i\cap\omega(f)=V_i\cap\omega(f).
\end{equation}
Indeed, the inclusion $U_i\cap\omega(f)\subset V_i\cap\omega(f)$ is trivial
because $U_i\subset V_i$. Conversely, if $z\in V_i\cap\omega(f)$, then
$z\in U_i(z)\subset U_i$, so $V_i\cap\omega(f)\subset U_i\cap\omega(f)$.

By definition, the set $U$ is open and contains $\omega(f)$. Let
$x_0\in U$ and $N\ge 0$ be such that
\begin{equation}\label{eq:assumtionx-lemBirkhoff}
\forall n\in\Lbrack 0,N\Rbrack,\ f^n(x_0)\in U.
\end{equation}
We are going to show by induction the following:

\medskip \textsc{Fact 1.}
{\it There exist integers $k\ge 0$ and $i_0\in\Lbrack 1,r\Rbrack$ and finite 
sequences of points $(z_n)_{0\le n\le k}$
and  $(x_n)_{0\le n\le k}$ such that, for all $n\in\Lbrack 0,k\Rbrack$, 
$$
z_n\in\omega(f)\cap V_{i_0},\quad x_n\in U_{i_0}(z_n),\quad x_{n+1}=f^{p_{i_0}(z_n)}(x_n).
$$
If we set $q_0:=0$ and $q_n:=p_{i_0}(z_0)+\cdots+p_{i_0}(z_{n-1})$ for all
$n\in\Lbrack 1,k+1\Rbrack$, the integer $k$ is such that $q_k\le N < q_{k+1}$.}

\medskip
$\bullet$ According to the definition of $U$, there exists ${i_0}\in
\Lbrack 1,r\Rbrack$
and $z_0\in \omega(f)\cap V_{i_0}$ such that $x_0\in U_{i_0}(z_0)$. If 
$q_1:=p_{i_0}(z_0)>N$, then the construction is over with $k:=0$.

$\bullet$ Suppose that the points $(z_n)_{0\le n\le j}$
and  $(x_n)_{0\le n\le j}$ are already defined up to some integer $j$ with
$q_{j+1}\le N$.
We set $x_{j+1}:=f^{p_{i_0}(z_j)}(x_j)$. Thus $x_{j+1}=f^{q_{j+1}}(x_0)$.
Then $x_{j+1}\in U$ by \eqref{eq:assumtionx-lemBirkhoff} because $q_{j+1}\le N$.
Since $x_j\in U_{i_0}(z_j)$, we have $x_{j+1}\in V_{i_0}$ by 
\eqref{eq:fpUisubsetVi}. If $x_{j+1}\in \omega(f)$, then $x_{j+1}\in U_{i_0}$
by \eqref{eq:Uiwf-Viwf}, and we set $z_{j+1}:=x_{j+1}$; trivially
$x_{j+1}\in U_{i_0}(z_{j+1})$. If $x_{j+1}\notin\omega(f)$, the fact that $x_{j+1}
\in U$ implies that there exists $i\in\Lbrack 1,r\Rbrack$ such that
$x_{j+1}\in U_i\subset V_i$. Necessarily, $i={i_0}$ because of
\eqref{eq:xnotinwf-uniquei}. Thus there exists $z_{j+1}\in\omega(f)\cap U_{i_0}$
such that $x_{j+1}\in U_{i_0}(z_{j+1})$. If 
$q_{j+2}:=q_{j+1}+p_{i_0}(z_{j+1})>N$,
then the construction is over with $k:=j+1$.

Since all the integers $p_{i_0}(z)$ are positive, the sequence $(q_n)$ is 
increasing, and thus the construction finishes. This ends the proof of
Fact~1.

\medskip
Let $x_0$ satisfy \eqref{eq:assumtionx-lemBirkhoff} and 
$n\in\Lbrack 0,N\Rbrack$.
We keep the notation of Fact~1. Let
$j\in\Lbrack 0,k+1\Rbrack$ be such that $q_j\le n< q_{j+1}$. We have
\begin{eqnarray}\label{eq:fnx0-fnyi}
\lefteqn{|f^n(x_0)-f^n(y_{i_0})|}\\
&&\le |f^{n-q_j}(f^{q_j}(x_0))-f^{n-q_j}(z_j)|
+|f^{n-q_j}(z_j)-f^{n-q_j}(f^{q_j}(y_{i_0}))|.\nonumber
\end{eqnarray}
Since $n-q_j<q_{j+1}-q_j=p_{i_0}(z_j)$, the fact that the points
$x_j=f^{q_j}(x_0)$ and $z_j$ belong to $U_{i_0}(z_j)$, combined with 
\eqref{eq:diamfnUi}, implies that
$$
|f^{n-q_j}(f^{q_j}(x_0))-f^{n-q_j}(z_j)|\le\eps.
$$
By \eqref{eq:piz}+\eqref{eq:pimultiple} and the $f$-invariance of $\omega(f)$,
the point
$f^{q_j}(y_{i_0})$ is in $V_{i_0}\cap\omega(f)$. Moreover, $z_j$ belongs to
$V_{i_0}\cap\omega(f)$. Therefore, \eqref{eq:stableVi} implies that
$$
|f^{n-q_j}(z_j)-f^{n-q_j}(f^{q_j}(y_{i_0}))|\le\eps.
$$
Inserting these inequalities in \eqref{eq:fnx0-fnyi}, we get
\begin{equation}\label{eq:ON}
\text{if }x_0\text{ satisfies \eqref{eq:assumtionx-lemBirkhoff}},
\exists i_0\in\Lbrack 1,r\Rbrack,
\forall n\in\Lbrack 0,N\Rbrack,\ |f^n(x_0)-f^n(y_{i_0})|\le2\eps.
\end{equation}

Now, let $x$ be a point and let $N_0\le N_1$ be integers such that
$$
\forall n\in\Lbrack N_0,N_1\Rbrack,\ f^n(x)\in U.
$$
We apply \eqref{eq:ON} to $x_0:=f^{N_0}(x)$ and $N:=N_1-N_0$:
\begin{equation}\label{eq:N0N1}
\exists i_0\in \Lbrack 1,r\Rbrack,\ \forall n\in\Lbrack 0,N\Rbrack,\ 
|f^n(x_0)-f^n(y_{i_0})|\le 2\eps.
\end{equation}
Since $f(\omega(f))=\omega(f)$, there exists $y\in \omega(f)$ such that
$f^{N_0}(y)=y_{i_0}$, and this point satisfies: 
$\forall n\in \Lbrack 0,N_1\Rbrack, f^n(y)\in\omega(f)\subset U$.
Thus we can apply \eqref{eq:ON} to $x_0:=y$ and $N:=N_1$:
\begin{equation}\label{eq:0N1}
\exists i\in \Lbrack 1,r\Rbrack,\ \forall n\in\Lbrack 0,N_1\Rbrack,\ 
|f^n(y)-f^n(y_i)|\le 2\eps.
\end{equation}
Combining \eqref{eq:N0N1} and \eqref{eq:0N1}, we get, for all
$n\in \Lbrack N_0,N_1\Rbrack$:
\begin{eqnarray*}
|f^n(x)-f^n(y_i)|&\le &
|f^n(x)-f^{n-N_0}(y_{i_0})|+|f^{n-N_0}(y_{i_0})-f^n(y_i)|\\
&=&|f^{n-N_0}(x_0)-f^{n-N_0}(y_{i_0})|+|f^n(y)-f^n(y_i)|\\
&\le& 2\eps+2\eps=4\eps,
\end{eqnarray*}
which gives the expected result.
\end{proof}

At last, we are able to show the following result, which is the ``if'' part
of Theorem~\ref{theo:LY-equiv-hA>0}. Together with 
Proposition~\ref{prop:LY-hA>0}, this finally proves 
Theorem~\ref{theo:LY-equiv-hA>0} and concludes this section.

\begin{theo}\label{theo:hA>0-LY}
Let $f\colon I\to I$ be an interval map. If 
there exists a sequence $A$ such that $h_A(f)>0$, then
$f$ is chaotic in the sense of Li-Yorke.
\end{theo}

\begin{proof}
 Suppose that $f$ is not chaotic in the sense of Li-Yorke. We are going to
show that $h_A(f)=0$ for every increasing sequence of non negative
integers $A$, which proves the theorem by refutation.

Let $\eps>0$. Let $U$ and $y_1,\ldots, y_r$ be given by 
Lemma~\ref{lem:approximate-xinU-yi}. Let $(X_i)_{r+1\le i\le s}$ be pairwise
disjoint nonempty sets 
such that $\diam X_i\le\eps$ for all $i\in\Lbrack r+1,s\Rbrack$
and $X_{r+1}\cup\cdots\cup X_s\supset I\setminus U$. We choose a point 
$y_i\in X_i$ for every $i\in \Lbrack r+1,s\Rbrack$.
According to Corollary~\ref{cor:Usupsetwf}, there exists an integer 
$N\ge 0$ such that
$$
\forall x\in I,\ \#\{n\ge 0\mid f^n(x)\notin U\}\le N.
$$
Let $x$ be a point in $I$. Let $n_1<n_2<\cdots<n_M$ be the integers
such that $f^{n_i}(x)\notin U$; we have $M\le N$. We set
\begin{gather*}
\CI_{2i-1}:=\{n_i\}\quad\text{for all } 
i\in\Lbrack 1,M\Rbrack,\\
\CI_{2i}:=\Lbrack n_i+1,n_{i-1}-1\Rbrack\quad\text{for all } 
i\in\Lbrack 1,M-1\Rbrack,\\
\CI_0:=\Lbrack 0,n_1-1\Rbrack,\quad \CI_{2M}:=\{n\mid n\ge M+1\}.
\end{gather*}
The sets $(\CI_i)_{0\le i\le 2M}$ form a partition of $\IZ^+$ into intervals
of integers; the point $f^n(x)$ belongs to $U$ if and only if $n\in\CI_i$
for some $i$ even.
According to Lemma~\ref{lem:approximate-xinU-yi}, for every 
$i\in\Lbrack 0,M\Rbrack$, there exists $j_{2i}\in\Lbrack 1,r\Rbrack$ such that
$$
\forall n\in\CI_{2i},\ |f^n(x)-f^n(y_{j_{2i}})|\le\eps.
$$
For every $i\in\Lbrack 1,M\Rbrack$, let $j_{2i-1}$ be the integer in
$\Lbrack r+1,s\Rbrack$ such that $f^{n_i}(x)\in X_{j_{2i-1}}$.
We then have
$$
\forall i\in\Lbrack 0,2M\Rbrack,\  
\forall n\in\CI_{i},\ |f^n(x)-f^n(y_{j_i})|\le\eps.
$$
We associate to $x$ a sequence 
$\overline C(x)=(C_n(x))_{n\ge 0}\in\Lbrack 1,s\Rbrack^{\IZ^+}$ 
coding the trajectory of $x$ and defined by
$$
\forall i\in\Lbrack 0,2M\Rbrack,\ \forall n\in \CI_i,\ C_n(x)=j_i.
$$
Let $A=(a_n)_{n\ge 0}$ be an increasing sequence of non negative integers.
If $x,y$ satisfy $C_k(x)=C_k(y)$, then $|f^k(x)-f^k(y)|\le 2\eps$.
Thus it is easy to see that $r_n(A,f,2\eps)$ is
bounded by the number of different sequences $(C_{a_i}(x))_{0\le i<n}$ when $x$
varies in $I$. We are going to bound this number.

For a given point $x$, the sets
$\left(\{k\in\Lbrack 0,n-1\Rbrack\mid a_k\in \CI_i\}\right)_{0\le i\le 2M}$
form a partition of $\Lbrack 0,n-1\Rbrack$ into $2M+1$ intervals of integers;
some may be empty. We call $\CJ_1,\ldots, \CJ_m$ the nonempty sets among
them, with $m\le 2M+1\le 2N+1$. To determine the sets $(\CJ_i)_{1\le i\le m}$,
it is sufficient to give the positions of the first integer of each
$\CJ_i$; the number of such choices is bounded by ${n}\choose{m}$. Then
for every $i\in\Lbrack 1,m\Rbrack$, there exists $j_i\in\Lbrack 1,s\Rbrack$
such that $C_{a_k}(x)=j_i$ for all $k\in \CJ_i$;
the number of choices of $(j_i)_{1\le i\le m}$ is bounded by $s^m$.
Therefore we have
\begin{eqnarray*}
r_n(A,f,2\eps)&\le&\sum_{m=0}^{2N+1}{{n}\choose{m}} s^m
=\sum_{m=0}^{2N+1}\frac{n(n-1)\cdots (n-m+1)}{m!} s^m\\
&\le& (2N+2)n^{2N+1}s^{2N+1}.
\end{eqnarray*}
This implies that $\lim_{n\to+\infty}\frac 1n\log r_n(A,f,2\eps)=0$. We
deduce that $h_A(f)=0$ for any sequence $A$.
\end{proof}

\subsection*{Remarks on graph maps}

Theorem~\ref{theo:LY-equiv-hA>0} was generalized to circle maps
by Hric \cite{Hri2} and to some star maps by Cánovas \cite{Cano4}.

\begin{theo}
Let $f\colon\IS\to\IS$ be a circle map. Then $f$ is chaotic in the sense of 
Li-Yorke if and only if there exists an increasing sequence $A$ such that
$h_A(f)>0$.
\end{theo}

\begin{theo}
Let $f\colon S_n\to S_n$ be a continuous map, where $n\ge 3$ and
$S_n:=\{z\in\IC\mid z^n\in [0,1]\}$. Suppose that $f(0)=0$. Then $f$ is chaotic in the sense of Li-Yorke if and only if there 
exists an increasing sequence $A$ such that $h_A(f)>0$.
\end{theo}

%**********************************************************
\section{Examples of maps of type $2^\infty$, Li-Yorke chaotic or not} \label{sec:chaos-LY-htop0}

According to Theorem~\ref{theo:htop-positive-chaos-LY}, all interval
maps of positive entropy are chaotic in the sense of Li-Yorke; 
positive entropy interval maps are exactly the maps of type $2^n q$ for some 
odd $q>1$ by Theorem~\ref{theo:htop-power-of-2}.
On the other hand, an interval map of type $2^n$ for some finite $n\ge 0$ 
has no infinite $\omega$-limit set
by Proposition~\ref{prop:htop0-Lki}, and thus it is not chaotic in the
sense of  Li-Yorke according to Theorem~\ref{theo:htop0-chaos-LY}. 
What about maps of type $2^{\infty}$?
There exist maps of type
$2^{\infty}$ that are not chaotic in the sense of Li-Yorke; some have
no infinite $\omega$-limit set, such as Example~\ref{ex:type-infty},
whereas some have an infinite
$\omega$-limit set as in the example built by Smítal \cite{Smi}. 
On the other hand, there exist zero entropy interval maps that are chaotic 
in the sense of
Li-Yorke, as it was shown simultaneously by  Smítal \cite{Smi} and
Xiong \cite{Xio3}. 
We are going to give two
examples of maps of type $2^{\infty}$ with an infinite
$\omega$-limit set, the first one (Example~\ref{ex:non-chaos-LY-htop0})
is not chaotic in the sense of Li-Yorke
and the second one (Example~\ref{ex:chaos-LY-htop0})
is chaotic in the sense of Li-Yorke. 

\begin{ex}\label{ex:non-chaos-LY-htop0}
We define the map $f\colon [0,1]\to [0,1]$ by
\begin{gather*}
f(0):=2/3,\\ 
\forall n\ge 1,\ f\left(1-\frac{2}{3^n}\right):=\frac{1}{3^{n-1}}\quad
\text{and}\quad
f\left(1-\frac{1}{3^n}\right):=\frac{2}{3^{n+1}},
\end{gather*}
and $f$ is linear between two consecutive points among the values we have
just defined; see Figure~\ref{fig:non-chaos-LY-htop0}. Finally, $f(1):=0$. 
It is clear that $f$ is
continuous at $1$, and thus $f$ is continuous on $[0,1]$.

\begin{figure}[htb]
\centerline{\includegraphics{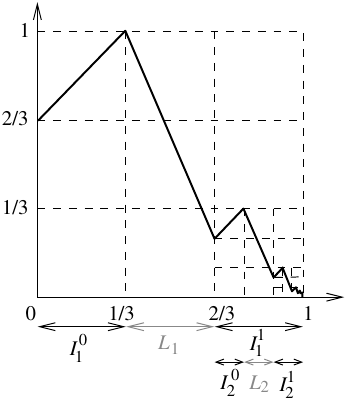}}
\caption{This map $f$ is of type $2^{\infty}$, the set $\omega(0,f)$ is
infinite but $f$ is not chaotic in the sense of Li-Yorke.}
\label{fig:non-chaos-LY-htop0}
\end{figure}

Let us give the idea of the construction. The map $f$ swaps
the intervals $[0,1/3]$ and $[2/3,1]$, and  we ``fill the
gap'' linearly on $[1/3,2/3]$ to get a continuous map 
(we shall see that the core of the
dynamics is in  $[0,1/3]\cup[2/3,1]$).
More precisely, $f$ sends $[0,1/3]$ linearly
onto $[2/3,1]$ and it sends $[2/3,1]$ to $[0,1/3]$ in such a way that
$f^2|_{[2/3,1]}$ is the same map as $f$ up to a rescaling. On the graph
of $f$, it means that, if one magnifies the square $[2/3,1]\times [0,1/3]$
(bottom right square among the 9 big squares in
Figure~\ref{fig:non-chaos-LY-htop0}), then one sees the same picture
as the graph of the whole map.

This map appears in \cite{Del}, where Delahaye proved that it is of
type $2^{\infty}$ and has an infinite $\omega$-limit set.  
We are going to show in addition that $f$ is
not chaotic in the sense of Li-Yorke. We follow Smítal's ideas \cite{Smi},
although the construction is slightly different from the one in \cite{Smi}.

We set $I_0^1:=[0,1]$ and, for all $n\geq 1$,
$$
I_n^0:=\left[1-\frac{1}{3^{n-1}},1-\frac{2}{3^n}\right],\ 
L_n:=\left[1-\frac{2}{3^n}, 1-\frac{1}{3^n}\right],\ 
I_n^1:=\left[1-\frac{1}{3^n},1\right].
$$
It is clear that
$$
\forall n\ge 1,\ I_n^0\cup L_n\cup I_n^1=I_{n-1}^1\quad\text{and}\quad
|I_n^0|=|L_n|=|I_n^1|=\frac{1}{3^n}.
$$
Moreover, one can check from the definition of $f$ that, for all $n\geq 1$,
\begin{gather}
\label{eq:non-chaos-LY-1}
f^{2^{n-1}}|_{I_n^0} \text{ is linear increasing and } 
f^{2^{n-1}}(I_n^0)=I_n^1,\\
\label{eq:non-chaos-LY-2} f^{2^{n-1}}|_{L_n} \text{ is linear decreasing and }
f^{2^{n-1}}(L_n)\supset L_n\cup I_n^1,\\
\label{eq:non-chaos-LY-3}f^{2^{n-1}}(I_n^1)=I_n^0.
\end{gather}
Then $f^{2^n}(I_n^1)=I_n^1$ by \eqref{eq:non-chaos-LY-1} and
\eqref{eq:non-chaos-LY-3}. Moreover, the intervals 
$(f^i(I_n^1))_{0\leq i<
2^n}$ are pairwise disjoint, they have the same length
$\frac{1}{3^n}$ and, if we set
$$
\forall n\ge 0,\  K_n:=\bigcup_{i=0}^{2^n-1} f^i(I_n^1)\quad\text{and}\quad
K:=\bigcap_{n\geq 0}K_n,
$$
then $K$ is the triadic Cantor set (see Example~\ref{ex:triadicCantor}
in the Appendix).

\begin{figure}[htb]
\centerline{\includegraphics{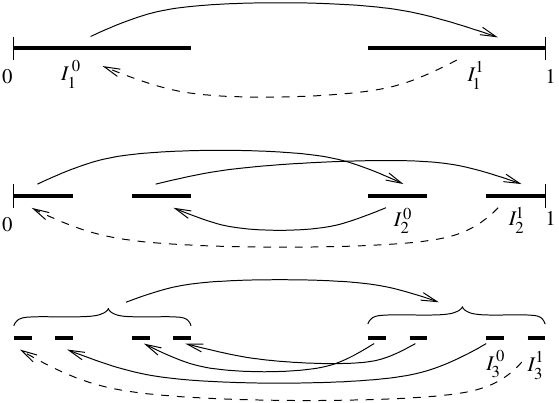}}
\caption{Orbits of $I_n^1$ for $n=1,2,3$. An arrow means that an interval is
sent onto another. The arrows are solid when
the action of the map is linear increasing, dotted otherwise.}
\label{fig:orbIn}
\end{figure}

First we are going to show that $f$ is of type $2^\infty$ and that all
but finitely many trajectories eventually fall in $\CO_f(I_n^1)$. We fix $n\ge 1$.
Since $f^{2^{n-1}}(L_n)\supset L_n$ by \eqref{eq:non-chaos-LY-2},
there exists a point $z_n\in
L_n$ such that $f^{2^{n-1}}(z_n)=z_n$ (Lemma~\ref{lem:fixed-point}).
The period of $z_n$ is exactly $2^{n-1}$ because $L_n\subset I_{n-1}^1$
and the intervals $(f^i(I_{n-1}^1))_{0\leq i<2^{n-1}}$ are pairwise disjoint.
Moreover, using \eqref{eq:non-chaos-LY-2} and the fact that
$|L_n|=|I_n^1|$, we see that $\slope(f^{2^{n-1}}|_{L_n})\le -2$. 
This implies that, if $x,
f^{2^{n-1}}(x),\ldots, f^{k2^{n-1}}(x)$ are in $L_n$, then
$|f^{(k+1)2^{n-1}}(x)-z_n|\geq 2^k|x-z_n|$. Therefore, for all $x\in L_n
\setminus\{z_n\}$, there exists $k\geq 1$ such that
$f^{k2^{n-1}}(x)\notin L_n$. 
Thus the point $f^{k2^{n-1}}(x)$ belongs to $I_n^0\cup I_n^1$ because
$L_n$ is included in $I_{n-1}^1$, which is $f^{2^{n-1}}$-invariant.
Using the fact that $f^{2^{n-1}}(I_n^0)=I_n^1$, we  see that
for all $n\ge 1$ and all $x\in I_{n-1}^1\setminus\{z_n\}$, there exists 
$k\geq 0$ such that $f^k(x)\in I_n^1$.
Starting with $I_0^1=[0,1]$, a straightforward induction shows that
$$
\forall x\in [0,1],\ \CO_f(x)\cap \{z_n\mid n\geq 1\}
=\emptyset\Longrightarrow \forall n\geq 1,\ \exists
k\geq 0,\ f^k(x)\in I_n^1.
$$
This implies that, 
for all $x\in [0,1]$, either $\omega(x,f)=\CO_f(z_n)$ for some $n\ge 1$, or
$\omega(x,f)\subset K$. Thus 
\begin{equation}\label{eq:omegasetincludedK}
\text{all infinite $\omega$-limit sets are included in }K.
\end{equation}
Moreover, $K$ contains no periodic point because,
for all $n\ge 1$, $K$ is included in $\CO_f(I_n^1)$, which is a cycle of
intervals of period $2^n$ (and thus $\CO_f(I_n^1)$ contains no periodic point of
period less than $2^n$). Therefore, the only periodic orbits
are $(\CO_f(z_n))_{n\ge 1}$, and  $f$ is of type $2^\infty$.

Now we are going to show that $\omega(0,f)=K$. 
For all $n\ge 1$, $f^{2^{n-1}}(\min I_n^0)=\min I_n^1$
by \eqref{eq:non-chaos-LY-1}. Since $\min I_n^0=\min I_{n-1}^1$ and 
$\min I_0^1=0$, we get
\begin{equation}\label{eq:orbit0}
\forall n\ge 1,\ f^{2^n-1}(0)=\min I_n^1=1-\frac{1}{3^n}.
\end{equation}
Let $x\in K$ and $\eps>0$. We fix
$n\geq 1$ such that $\frac{1}{3^n}<\eps$. By definition of $K$, there exists 
$i\in\Lbrack 0,2^n-1\Rbrack$ such that
$x\in f^i(I_n^1)$. The point $f^{2^n-1+i}(0)$ belongs to $f^i(I_n^1)$
by \eqref{eq:orbit0}, so
$
|f^{2^n-1+i}(0)-x|\leq|f^i(I_n^1)|=\frac{1}{3^n}<\eps.
$
Since $\eps$ is arbitrarily small and $n$ is arbitrarily large, 
this implies that $x\in \omega(0,f)$, that is,
$K\subset \omega(0,f)$.  We deduce that $\omega(0,f)$ is infinite,
and $\omega(0,f)=K$ by \eqref{eq:omegasetincludedK}.

Finally, we are going to show that $f$ is not chaotic in the sense of
Li-Yorke. Let $x,y$ be two distinct points in $K$ and let $n$ be a positive 
integer such that $\frac{1}{3^n}<|x-y|$. There exist 
$i,j\in\Lbrack 0,2^n-1\Rbrack$ such that $x\in f^i(I_n^1)$ and $y\in f^j(I_n^1)$. 
Because of the choice of $n$, the integers $i$ and $j$ are not equal. 
Thus $x,y$ are $f$-separable. Since $K$ contains
all infinite $\omega$-limit sets,
Theorem~\ref{theo:htop0-chaos-LY} implies that $f$ is not chaotic in
the sense of Li-Yorke.
\end{ex}

\begin{rem}
Let $\Sigma:=\{0,1\}^{\IZ^+}$ and let $A\colon \Sigma\to\Sigma$ be the map
consisting in adding $(1,0,0,\ldots)$ $\bmod\ 2$ with carrying.
More formally, 
$$
A((\alpha_n)_{n\ge 0})=(\beta_n)_{n\ge 0}\quad\text{with}\quad
\beta_n=\left\{\begin{array}{ll}
1-\alpha_n&\text{ if }\forall i\in\Lbrack 0,n-1\Rbrack, \alpha_i=1,\\
\alpha_n&\text{ otherwise}.\end{array}\right.
$$
For instance $A(0,0,1,0,0,...)=(1,0,1,0,0,...)$ and $A(1,1,0,0,...)=
(0,0,1,0,...)$.
The map $A$ is called the \emph{dyadic adding machine}\index{adding machine}.

Let $f$, $K_n$ and $K$ be as defined in 
Example~\ref{ex:non-chaos-LY-htop0}.
Let $h\colon K\to \Sigma$ be the map defined by $h(x)=(\alpha_n)_{n\ge 0}$
such that:
if $C_n$ is the connected component of $K_n$ containing $x$, then
$\alpha_n=0$ if $x$ is in the left connected component of $K_{n+1}$ 
contained in $C_n$, and $\alpha_n=1$ otherwise (recall that each connected 
component of $K_n$ contains two connected components of $K_{n+1}$).
One can show that $h$ is a homeomorphism and that it is a topological
conjugacy between $(K,f|_K)$ and $(\Sigma, A)$.

The readers interested in the dynamics of adding machines and their 
relations with interval maps can look at \cite{BK3,BK4,BKM}.
The adding machine belongs to the wider family of dynamical systems
called \emph{odometers}\index{odometer}, see e.g., the survey \cite{Dow3}.
\end{rem}

\begin{rem}
The map in Example~\ref{ex:non-chaos-LY-htop0} is
made of countably many linear pieces. The Feigenbaum map\index{Feigenbaum map} 
is another example of completely different nature,
since it is $C^{\infty}$ and unimodal.
Recall that the Feigenbaum map $f_{\lambda_{2^\infty}}$ is
a map of the logistic family $f_{\lambda}\colon [0,1]\to [0,1],
x\mapsto \lambda x(1-x)$ (see 
Remark~\ref{rem:alltypesinunimodal}). The map $f_{\lambda_{2^\infty}}$
is of type $2^{\infty}$, it has an infinite $\omega$-limit set $S$ (which
is a Cantor set) and it is not Li-Yorke chaotic. Moreover, the 
restriction of $f_{\lambda_{2^\infty}}$ to $S$ is
topologically conjugate to the dyadic adding machine.
See e.g. \cite[Theorem 11.3.11]{HK4}.
\end{rem}

%********************
\begin{ex}\label{ex:chaos-LY-htop0}
We are going to build a zero topological entropy interval map $g$
that is chaotic
in  the sense of Li-Yorke. This map will look like the map $f$ of
the previous example, except that the 
intervals $I_n^0, I_n^1$ are rescaled in such a way
that the set $K=\bigcap_{n\geq 0}\CO_g(I_n^1)$ is not a Cantor set any more,
its interior being nonempty.
This example is inspired by Smítal's \cite{Smi}.
We shall first define $g$, then prove several lemmas in order to show the
expected properties.

Let $(a_n)_{n\geq 0}$ be an increasing sequence of numbers less
than $1$ such that $a_0=0$. We define
$I_0^1:=[a_0,1]$ and, for all $n\geq 1$,
$$
I_n^0:=[a_{2n-2},a_{2n-1}],\ L_n:=[a_{2n-1},a_{2n}],\ I_n^1:=[a_{2n},1].
$$
It is clear that $I_n^0\cup L_n\cup I_n^1=I_{n-1}^1$.
We choose $(a_n)_{n\geq 0}$ such that the lengths of the intervals
$I_n^0, I_n^1$ satisfy:
\begin{gather*}
|I_n^0|=\frac{1}{3^n}|I_{n-1}^1|,\quad 
|I_n^1|=\left(1-\frac{2}{3^n}\right)|I_{n-1}^1| \quad
\text{if $n$ is odd},\\
|I_n^0|=\left(1-\frac{2}{3^n}\right)|I_{n-1}^1|,\quad
|I_n^1|=\frac{1}{3^n}|I_{n-1}^1|\quad\text{if $n$ is even}.
\end{gather*}
This implies that $|L_n|=\frac{1}{3^n}|I_{n-1}^1|$ for all $n\ge 1$.
Note that $|I_n^1|\to 0$ when $n$ goes to infinity, 
that is, $\lim_{n\to +\infty} a_n=1$; hence
$\bigcup_{n\geq 1}(I_n^0\cup L_n)=[0,1)$.

For all $n\geq 1$, let $\vfi_n\colon I_n^0\to I_n^1$ be the increasing linear
homeomorphism mapping $I_n^0$ onto $I_n^1$; its slope is
$\slope(\vfi_n)=\frac{|I_n^1|}{|I_n^0|}$. 
We define the map $g\colon
[0,1]\to [0,1]$ such that $g$ is continuous on $[0,1)$
and
\begin{eqnarray}
\label{eq:LY-htop0-f1}\forall n\ge 1,\ \forall x\in I_n^0,&& g(x):=
\vfi_1^{-1}\circ
\vfi_2^{-1}\circ\cdots\circ\vfi_{n-1}^{-1}\circ \vfi_n(x),\\ 
\label{eq:LY-htop0-f2}
\forall n\ge 1,&& g|_{L_n} \text{ is linear},\\
&&g(1):=0.\nonumber
\end{eqnarray}
Note that $g|_{I_n^0}$ is linear increasing. 
We shall show below that $g$ is continuous at~$1$.

Let us explain the underlying construction.  At step $n=1$, the
interval $I_1^0$ is sent linearly onto $I_1^1$  (hence
$g|_{I_1^0}=\vfi_1$) and we require that $g(I_1^1)\subset I_1^0$ (i.e.,
the graph of $g|_{I_1^1}$ is included in the gray
area in  Figure~\ref{fig:LY-htop0-partial-construction}). Then we do the same
kind of construction in the gray area with respect to
$I_2^0,I_2^1\subset I_1^1$: we rescale $I_2^0,I_2^1$ as
$\vfi_1^{-1}(I_2^0),\vfi_1^{-1}(I_2^1) \subset I_1^0$ (on the vertical
axis), then we send linearly $I_2^0$ onto $\vfi_1^{-1}(I_2^1)$
and we decide that $g(I_2^1)\subset \vfi^{-1}(I_2^0)$; in this
way, $g|_{I_2^0}=\vfi_1^{-1}\circ\vfi_2$ (which is
\eqref{eq:LY-htop0-f1} for $n=2$) and the graph of
$g|_{I_2^1}$ is included in the black
area in  Figure~\ref{fig:LY-htop0-partial-construction}. 
We repeat this construction on $I_2^1$ (black area), and so on. 
Finally, we fill the gaps in a linear
way (which is \eqref{eq:LY-htop0-f2}) to get the whole map, 
which is represented on the right side of
Figure~\ref{fig:LY-htop0-partial-construction}. 

\begin{figure}[htb]
\centerline{\includegraphics{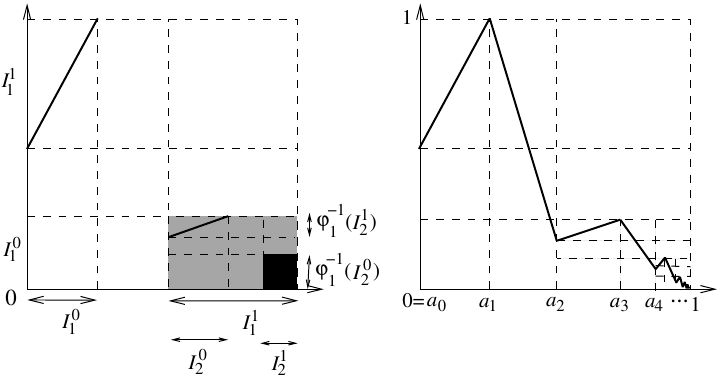}}
\caption{On the left: the first steps of the construction of $g$.  On
the right: the graph of $g$.}
\label{fig:LY-htop0-partial-construction}
\end{figure}

We introduce an auxiliary family of intervals.
We set $J_0^0:=[0,1]$ and, for all $n\geq 1$, the subintervals 
$J_n^0,J_n^1\subset J_{n-1}^0$ are defined by
$$
\min J_n^0=0,\ \max J_n^1=\max J_{n-1}^0\quad\text{and}\quad
\frac{|J_n^i|}{|J_{n-1}^0|}=\frac{|I_n^i|}{|I_{n-1}^1|}\quad\text{for }i\in
\{0,1\}.
$$
We have
$$
|J_n^0|=\prod_{i=1}^n\frac{|I_i^0|}{|I_{i-1}^1|}=\prod_{\doubleindice{k\in
\Lbrack 1, n\Rbrack}{k\ \rm even}}\left(1-\frac{2}{3^k}\right)\prod_{\doubleindice{k\in\Lbrack 1, n\Rbrack}{k\ \rm odd}}\frac{1}{3^k}.
$$
In this product, we bound by $1$ all the
factors except the last one, and we get 
\begin{equation}\label{eq:lengthJn}
|J_n^0|\le \frac{1}{3^{n-1}}.
\end{equation}

To show that $g$ is continuous at $1$, it is enough to prove that
$\max (g|_{I_n^1})\to 0$ when $n$ goes to infinity. 
For all $n\geq 1$, we have 
\begin{eqnarray*}
\vfi_n(\max I_n^0)&\!=\!&\max I_n^1=1=\min I_{n-1}^1+|I_{n-1}^1|\\
\vfi_{n-1}^{-1}\circ \vfi_n(\max I_n^0)&\!=\!&
\min I_{n-1}^0+|I_{n-1}^1|\slope (\vfi_{n-1}^{-1})\\
&\!=\!&
\min I_{n-2}^1+|I_{n-1}^1|\slope (\vfi_{n-1}^{-1})\\
\vfi_{n-2}^{-1}\circ \vfi_{n-1}^{-1}\circ\vfi_n(\max I_n^0)&\!=\!& 
\min I_{n-2}^0+|I_{n-1}^1|\slope (\vfi_{n-2}^{-1})\slope (\vfi_{n-1}^{-1})\\
\vdots\qquad &&\\
\vfi_1^{-1}\circ \vfi_2^{-1}\circ\cdots\circ 
\vfi_{n-1}^{-1}\circ\vfi_n(\max I_n^0)&\!=\!&
\disp \min I_1^0+|I_{n-1}^1|\prod_{i=1}^{n-1}\slope(\vfi_i^{-1}).
\end{eqnarray*}
We have $\max I_1^0=1$, $\slope(\vfi_i^{-1})=\frac{|I_i^0|}{|I_i^1|}$
and $|I_0^1|=1$. Thus
$$
\vfi_1^{-1}\circ \vfi_2^{-1}\circ\cdots\circ 
\vfi_{n-1}^{-1}\circ\vfi_n(\max I_n^0)
=\prod_{i=1}^{n-1}\frac{|I_i^0|}{|I_{i-1}^1|}=|J_{n-1}^0|.
$$
Consequently
\begin{equation}\label{eq:LY-htop0-max-f(In0)}
g(\max I_n^0)=|J_{n-1}^0|=\max J_{n-1}^0.
\end{equation}
According to the definition of $g$, $\max(g|_{I_{n-1}^1})=g(\max I_n^0)$,
so $\max(g|_{I_{n-1}^1})=|J_{n-1}^0|$. By \eqref{eq:lengthJn},
$\lim_{n\to+\infty}\max (g|_{I_{n-1}^1})= 0$; 
therefore $g$ is continuous at $1$. 

\medskip
The next lemma describes the action of $g$ on the intervals $(J_n^i)$,
$(I_n^i)$, and collects the properties that we shall use later.
As in Example~\ref{ex:non-chaos-LY-htop0}, the interval $I_n^1$ is periodic
of period $2^n$ and the map $g^{2^{n-1}}$ swaps $I_n^0$ and $I_n^1$;
Figure~\ref{fig:orbIn} is still valid, except that 
the intervals have different lengths.
However, we prefer to deal with $J_n^0=g(I_n^1)$ instead of $I_n^1$;
this will simplify
the proofs because $g|_{I_n^1}$
is not monotone, whereas $g^i|_{J_n^0}$ is linear for all 
$i\in\Lbrack 1,2^n-1\Rbrack$ (assertion (iii) in the lemma below).

\begin{lem}\label{lem:LY-htop0-summary-Jn0}
Let $g$ be the map defined above. Then for all
$n\geq 1$:
\begin{enumerate}
\item $g(I_n^1)=J_n^0$,
\item $g(I_n^0)=J_n^1$,
\item $g^i|_{J_n^0}$ is linear increasing for all 
$i\in\Lbrack 1, 2^n-1\Rbrack$, 
\item $g^{2^{n-1}-1}(J_n^0)=I_n^0$ and $g^{2^n-1}(J_n^0)=I_n^1$,
\item $g^i(J_n^0)\subset \bigcup_{k=1}^n I_k^0$
for all $i\in\Lbrack 0,2^n-2\Rbrack$, 
\item $\left(g^i(J_n^0)\right)_{0\leq i<2^n}$ are pairwise disjoint.
\end{enumerate}
Moreover, the previous assertions imply:
\begin{enumerate}
\addtocounter{enumi}{6}
\item $g^{2^{n-1}}(J_n^0)=J_n^1$,
\item $g^{2^n}(J_n^0)=J_n^0$,
\item $g^{2^{n-1}}|_{I_n^0}$ is linear increasing and
$g^{2^{n-1}}(I_n^0)=I_n^1$,
\item $g^{2^{n-1}}(I_n^1)=I_n^0$,
\item $(g^i(I_n^0))_{0\leq i<2^n}$ are pairwise disjoint and 
$g^{2^n}(I_n^1)=I_n^1$.
\end{enumerate}
\end{lem}

\begin{proof}
According to \eqref{eq:LY-htop0-max-f(In0)},
\begin{equation}\label{eq:fIn0}
\max (g|_{I_n^1})=g(\max I_{n+1}^0)=|J_n^0|=\max J_n^0.
\end{equation}
Moreover, $\min (g|_{I_n^1})= g(1)=0=\min J_n^0$. 
Thus $g(I_n^1)=J_n^0$ by the intermediate value theorem; this is~(i). 

\medskip
According to the definition of $g$,
$$
|g(I_n^0)|=
|I_n^0| \slope(\vfi_n) \prod_{i=1}^{n-1}\slope(\vfi_i^{-1})
= |I_n^1|\prod_{i=1}^{n-1}\frac{|I_i^0|}{|I_i^1|}
=\frac{|I_n^1|}{|I_{n-1}^1|}\prod_{i=1}^{n-1}\frac{|I_i^0|}{|I_{i-1}^1|}
=|J_n^1|.
$$
Moreover, $g(\max I_n^0)=\max J_{n-1}^0=\max J_n^1$ by \eqref{eq:fIn0}, 
so $g(I_n^0)=J_n^1$. This is~(ii).

\medskip
We show by induction on $n$ that assertions (iii) and (iv) are satisfied.

\noindent$\bullet$
This is true for $n=1$ because $J_1^0=I_1^0$, $g(I_1^0)=I_1^1$ and
$g|_{I_1^0}=\vfi_1$ is linear increasing.

\noindent$\bullet$
Suppose that (iii) and (iv) are true for $n$. Since
$J_{n+1}^0\subset J_n^0$, the map $g^i|_{J_{n+1}^0}$ is linear increasing
for all $i\in\Lbrack 1, 2^n-1\Rbrack$, and 
$g^{2^n-1}(J_{n+1}^0)\subset I_n^1$.
Moreover, the facts that $g^{2^n-1}|_{J_n^0}$ is linear increasing 
and $g^{2^n-1}(J_n^0)=I_n^1$ imply
\begin{equation}
\min g^{2^n-1}(J_{n+1}^0)=\min g^{2^n-1}(J_n^0)=\min I_n^1=\min I_{n+1}^0,
\label{eq:iiiiv-1}
\end{equation}
and
\begin{equation}
\frac{|g^{2^n-1}(J_{n+1}^0)|}{|I_n^1|}=
\frac{|g^{2^n-1}(J_{n+1}^0)|}{|g^{2^n-1}(J_n^0)|}=
\frac{|J_{n+1}^0|}{|J_n^0|}.\label{eq:iiiiv-2}
\end{equation}
Then \eqref{eq:iiiiv-2} implies that $|g^{2^n-1}(J_{n+1}^0)|=|I_{n+1}^0|$
because $\frac{|J_{n+1}^0|}{|J_n^0|}=\frac{|I_{n+1}^0|}{|I_n^1|}$.
Combined with \eqref{eq:iiiiv-1}, we get
$g^{2^n-1}(J_{n+1}^0)=I_{n+1}^0$.
Then $g^{2^n}(J_{n+1}^0)=J_{n+1}^1$ by (ii). Since $J_{n+1}^1
\subset J_n^0$, the induction hypothesis implies that $g^i|_{J_{n+1}^1}$
is linear increasing for all $i\in\Lbrack 1,2^n-1\Rbrack$ 
and $g^{2^n-1}(J_{n+1}^1) \subset I_n^1$. Moreover,
$$
\max g^{2^n-1}(J_{n+1}^1)=\max g^{2^n-1}(J_n^0)=1=\max I_{n+1}^1,
$$
and by linearity
$$
\frac{|g^{2^n-1}(J_{n+1}^1)|}{|I_n^1|}=
\frac{|g^{2^n-1}(J_{n+1}^1)|}{|g^{2^n-1}(J_n^0)|}=
\frac{|J_{n+1}^1|}{|J_n^0|}=\frac{|I_{n+1}^1|}{|I_n^1|}.
$$
By the same argument as above, we get
$g^{2^n-1}(J_{n+1}^1)=I_{n+1}^1$. 
Since $g^{2^{n+1}-1}(J_{n+1}^0)=g^{2^n-1}(J_{n+1}^1)$, 
this shows that (iii) and (iv) hold for $n+1$.
This ends the induction and proves (iii) and (iv).

\medskip
Now we prove (v) by induction on $n$.

\noindent$\bullet$
This is true for $n=1$ because $J_1^0=I_1^0$.

\noindent$\bullet$
Suppose that (v) is true for $n$. Since $J_{n+1}^0 \subset J_n^0$, it follows
that $g^i(J_{n+1}^0)\subset \bigcup_{k=1}^n
I_k^0$ for all $i\in\Lbrack 0, 2^n-2\Rbrack$.
Moreover, $g^{2^n-1}(J_{n+1}^0)=
I_{n+1}^0$ by (iv) and $g^{2^n}(J_{n+1}^0)=g(I_{n+1}^0)=
J_{n+1}^1$ by (ii). Since $J_{n+1}^1\subset J_n^0,$ we can use
the induction hypothesis again, so
$g^{2^n+i}(J_{n+1}^0)\subset \bigcup_{k=1}^n I_k^0$
for all $i\in\Lbrack 0,2^n-2\Rbrack$.
This gives (v) for $n+1$.
This ends the induction and proves (v).

\medskip
Next we prove (vi). 
Suppose that $g^i(J_n^0)\cap g^j(J_n^0)\neq\emptyset$ for some
$i,j\in\Lbrack 0,2^n-1\Rbrack$ with $i<j$. 
Then $g^{2^n-1-j}(g^i(J_n^0))\cap g^{2^n-1-j}(g^j(J_n^0))
\neq\emptyset$. But $g^{2^n-1}(J_n^0)=I_n^1$ by (iv)
and $g^{2^n-1-(j-i)}(J_n^0)
\subset [0,\max I_n^0]$ by (v),
so these two sets are disjoint, which is a contradiction. 
We deduce that $(g^i(J_n^0))_{0\leq i<2^n}$ are pairwise disjoint.

Finally we indicate how to obtain the other assertions.
Assertions (vii) and (viii) follow respectively from (iv)+(ii) and
(iv)+(i). Assertion (ix) follows from (iii)+(iv). Assertion (x)
follows from (i)+(iv). Assertion (xi) follows from the combination of
(i), (iv) and (vi).
\end{proof}

We define
$$
\forall n\ge 0,\ K_n:=\CO_g(I_n^1)=\bigcup_{i=0}^{2^n-1}g^i(J_n^0),\quad
K:=\bigcap_{n=0}^{+\infty} K_n\quad\text{and}\quad B_K:={\rm Bd}_{\IR}(K),
$$
that is, $B_K$ is the boundary of $K$ for the topology of $\IR$ (i.e. the
points $0,1$ are not excluded).
According to Lemma~\ref{lem:LY-htop0-summary-Jn0}, $K_n$ is the disjoint union
of the intervals $(g^i(J_n^0))_{0\le i\le 2^n-1}$. The set $K$ has
a Cantor-like construction: at each step, a middle part of every
connected component of $K_n$ is removed to get $K_{n+1}$. However we shall
see that $K$ is not a Cantor set because its interior is not empty
(see Lemma~\ref{lem:LY-htop0-non-separable-K}), 
contrary to the situation in Example~\ref{ex:non-chaos-LY-htop0}.
The next lemma states that $g$ is of type $2^{\infty}$. Next
we shall show that the set $\omega(0,f)$ contains $B_K$.
Then we shall prove that 
$g$ is chaotic in the sense of Li-Yorke.

\begin{lem}\label{lem:LY-htop0-2-infty}
Let $g$ be the map defined above. Then $g$ is of type $2^{\infty}$.
\end{lem}

\begin{proof}
The same arguments as in Example~\ref{ex:non-chaos-LY-htop0}
can be used to show that $g$ is of type $2^{\infty}$. We do not repeat the
proof, we just
check that the required conditions are satisfied:

\noindent$\bullet$
By definition of $g$, the map $g|_{L_n}$ is linear decreasing, so
$g(L_n)\subset [0,g(\max I_n^0)]$. Moreover, $g(\max I_n^0)=\max J_{n-1}^0$
by \eqref{eq:fIn0}, so $g(L_n)\subset J_{n-1}^0$.
Thus $g^{2^{n-1}}|_{L_n}$ is linear decreasing by
Lemma~\ref{lem:LY-htop0-summary-Jn0}(iii). 

\noindent$\bullet$
The map $g^{2^{n-1}}|_{I_n^0}$
is linear increasing and $g^{2^{n-1}}(I_n^0)=I_n^1$ by 
Lemma~\ref{lem:LY-htop0-summary-Jn0}(ix), 
so $g^{2^{n-1}}(\min L_n)=\max I_n^1=1$.
Moreover, $g^{2^{n-1}}(I_n^1)=I_n^0$ 
by Lemma~\ref{lem:LY-htop0-summary-Jn0}(x), so
$g^{2^{n-1}}(\max L_n)\in I_n^0$. This implies that
$g^{2^{n-1}}(L_n)\supset L_n\cup I_n^1$. Since $|L_n|\leq |I_n^1|$,
we have $\slope(g^{2^{n-1}}|_{L_n})\leq -2$. 

\noindent$\bullet$
$g^{2^{n-1}}(I_n^0)=I_n^1$ by Lemma~\ref{lem:LY-htop0-summary-Jn0}(ix).

\noindent$\bullet$
$g^{2^n}(I_n^1)=I_n^1$ and $(g^i(I_n^1))_{0\leq i<2^n}$ are
pairwise disjoint by Lemma~\ref{lem:LY-htop0-summary-Jn0}(xi).
\end{proof}

Since $\min J_n^0=0$, the orbit of $0$ obviously enters $g^i(J_n^0)$ for 
all $n\geq 0$ and all $i\in\Lbrack 0,2^n-1\Rbrack$, which implies that
$\omega(0,g)$ meets all connected components of $K$.
The next lemma states that $\omega(0,g)$ contains $B_K$; the proof relies
on the idea that the smaller interval among $J_{n+1}^0$ and $J_{n+1}^1$
contains alternatively either $\min J_n^0$ or $\max J_n^0$, when $n$
varies, so that both endpoints of a connected component of $K$ can be
approximated by small intervals of the form $g^i(J_n^0)$.

\begin{lem}\label{lem:LY-htop0-omega(0,f)}
Let $g$ and $K$ be as defined above. Then
$B_K\subset \omega(0,g)$. In particular, $\omega(0,g)$ is infinite.
\end{lem}

\begin{proof}
According to the definition of $K$, the connected components of $K$ are
of the form $C:=\bigcap_{n\ge 0} C_n$, where $C_n$ is a 
connected component of $K_n$ and $C_{n+1}\subset C_n$ for all $n\ge 0$. 
That is, the connected components of $K$ are
exactly the nonempty sets of the form $C:=\bigcap_{n=0}^{+\infty} g^{j_n}(J_n^0)$
with $j_n\in \Lbrack 0,2^n-1\Rbrack$, and
\begin{equation}\label{eq:CC-cantor-like}
\min C=\lim_{n\to+\infty} \min g^{j_n}(J_n^0)\quad\text{and}\quad
\max C=\lim_{n\to+\infty} \max g^{j_n}(J_n^0)
\end{equation}
because $C$ is a decreasing intersection of compact intervals.
Moreover, $B_K$ is equal to the union of the endpoints of all connected 
components of $K$. 
Let $y\in B_K$. By \eqref{eq:CC-cantor-like},
there exists a sequence of points $(y_n)_{n\geq 0}$
such that $y=\lim_{n\to+\infty}y_n$ and $y_n\in\End{g^{j_n}(J_n^0)}=
\{\min g^{j_n}(J_n^0),\max g^{j_n}(J_n^0)\}$ for all $n\ge 0$, where
$j_n\in\Lbrack 0,2^n-1\Rbrack$ is such that $y\in g^{j_n}(J_n^0)$. 
Let $\eps>0$ and $N\geq 0$.
Let $n$ be an even integer such that $\frac{1}{3^{n+1}}<\eps$ and
$|y_n-y|<\eps$, and let $k\geq 0$ be such that $k2^{n+1}\geq N$.

First we suppose that $y_n=\min g^{j_n}(J_n^0)$. The point $0$ belongs to
$J_{n+1}^0$ and,
by  Lemma~\ref{lem:LY-htop0-summary-Jn0}(viii), 
$g^{2^{n+1}}(J_{n+1}^0)=J_{n+1}^0$.
Thus $g^{k2^{n+1}+j_n}(0)$ belongs to $g^{j_n}(J_{n+1}^0)$. According to
Lemma~\ref{lem:LY-htop0-summary-Jn0}(iii), 
$y_n=\min g^{j_n}(J_{n+1}^0)$ and
$$
\frac{|g^{j_n}(J_{n+1}^0)|}{|g^{j_n}(J_n^0)|}=\frac{|J_{n+1}^0|}{|J_n^0|}=
\frac{|I_{n+1}^0|}{|I_n^0|}=\frac{1}{3^{n+1}}<\eps.
$$
Therefore $|g^{k2^{n+1}+j_n}(0)-y_n|<\eps|g^{j_n}(J_n^0)|\leq \eps$.

Secondly we suppose that $y_n=\max g^{j_n}(J_n^0)$. 
The point $g^{k2^{n+2}}(0)$ belongs to $J_{n+2}^0$ and
$g^{2^{n+1}}(J_{n+2}^0)=J_{n+2}^1$ 
by Lemma~\ref{lem:LY-htop0-summary-Jn0}(viii) and (vii) respectively, so
$$
g^{k2^{n+2}+2^{n+1}+2^n+j_n}(0)\in g^{2^n+j_n}(J_{n+2}^1).
$$
According to Lemma~\ref{lem:LY-htop0-summary-Jn0}(iii)+(vii), 
$$
\max g^{2^n+j_n}(J_{n+2}^1)=\max g^{2^n+j_n}(J_{n+1}^0)
=\max g^{j_n}(J_{n+1}^1)
=\max g^{j_n}(J_n^0)=y_n.
$$
Moreover, by Lemma~\ref{lem:LY-htop0-summary-Jn0}(vii), 
$$
g^{2^n}(J_{n+2}^1)\subset g^{2^n}(J_{n+1}^0)=J_{n+1}^1\subset J_n^0.
$$
Thus, using the linearity given by 
Lemma~\ref{lem:LY-htop0-summary-Jn0}(iii) and the definitions of
$(J_n^i)$ and $(I_n^i)$,
$$
\frac{|g^{j_n+2^n}(J_{n+2}^1)|}{|g^{j_n}(J_n^0)|}=
\frac{|g^{2^n}(J_{n+2}^1)|}{|J_n^0|}
=\frac{|g^{2^n}(J_{n+2}^1)|}{|g^{2^n}(J_{n+1}^0)|}\cdot
\frac{|J_{n+1}^1|}{|J_n^0|}
=\frac{1}{3^{n+2}}\cdot\left(1-\frac{2}{3^{n+1}}\right).
$$
Consequently, $y_n\in g^{2^n+j_n}(J_{n+2}^1)$ and
$$
|g^{k2^{n+2}+2^{n+1}+2^n+j_n}(0)-y_n|\leq |g^{j_n+2^n}(J_{n+2}^1)|<\eps.
$$

In both cases, there exists $p\geq N$ such that $|g^p(0)-y_n|<\eps$,
so $|g^p(0)-y|<2\eps$. This means that $y\in \omega(0,g)$, that is,
$B_K\subset \omega(0,g)$. 
Finally, for every $n\ge 0$, $K_n$ has $2^n$ connected components, each of 
which containing two connected components of $K_{n+1}$; thus
$K$ has an infinite number of connected components, which implies that
$B_K$ is infinite.
\end{proof}

In the proof of the next lemma, we shall first show that
$K$ contains a non degenerate connected component $C$; then we shall see
that the two endpoints of $C$ are $g$-non separable. 

\begin{lem}\label{lem:LY-htop0-non-separable-K}
Let $g$ and $K$ be as defined above. Then
$B_K$ contains two $g$-non separable points and $g$ is chaotic
in the sense of Li-Yorke.
\end{lem}

\begin{proof}
First we define by induction a sequence of intervals $C_n:=g^{i_n}(J_n^0)$
for some $i_n\in\Lbrack 0,2^n-1\Rbrack$ such that 
$$
\forall n\ge 1,\  C_n\subset
C_{n-1}\quad\text{and}\quad |C_n|=\left(1-\frac{2}{3^n}\right)|C_{n-1}|.
$$

\noindent$\bullet$
We set $i_0:=0$ and $C_0:=J_0=[0,1]$.

\noindent$\bullet$
Suppose that $C_{n-1}=g^{i_{n-1}}(J_{n-1}^0)$ is already built.
If $n$ is even, we set $i_n:=i_{n-1}$ and $C_n:=g^{i_n}(J_n^0)$. The map 
$g^{i_{n-1}}|_{J_{n-1}^0}$ is linear increasing by 
Lemma~\ref{lem:LY-htop0-summary-Jn0}(iii) and $J_n^0\subset J_{n-1}^0$, so
$$
\frac{|C_n|}{|C_{n-1}|}=\frac{|J_n^0|}{|J_{n-1}^0|}=1-\frac{2}{3^n}.
$$
If $n$ is odd, we set $i_n:=i_{n-1}+2^{n-1}$ and $C_n:=g^{i_n}(J_n^0)$.
By Lemma~\ref{lem:LY-htop0-summary-Jn0}(iii)+(vii),
the map $g^{i_{n-1}}|_{J_{n-1}^0}$ is linear increasing and
$C_n=g^{i_{n-1}}(J_n^1)$, so
$$
\frac{|C_n|}{|C_{n-1}|}=\frac{|J_n^1|}{|J_{n-1}^0|}=1-\frac{2}{3^n}.
$$

We set $C:=\bigcap_{n\geq 0}C_n$. It is a compact interval, and
it is non degenerate because
$$
\log |C|=\log |C_0|+\sum_{n\geq 1}\log \left(1-\frac{2}{3^n}\right)
>-\infty;$$
the last inequality follows from the facts that
$\log(1+x)\sim x$ when $x\to 0$ and $\sum \frac{1}{3^n}<+\infty$.
Moreover, $C$ is a connected component of $K$, so $\End{C}\subset
B_K$. Let $c_0:=\min C$ and $c_1:=\max C$. 
By Lemma~\ref{lem:LY-htop0-omega(0,f)}, the points $c_0,c_1$ belong to
$\omega(0,g)$, which is an infinite $\omega$-limit set.
Suppose that $c_0, c_1$ are $g$-separable and let $A_0,A_1$ be two
disjoint periodic intervals containing respectively $c_0, c_1$.  Let
$k$ be a common multiple of their periods. We choose an integer $n$ such
that $2^n>k$. Then there exists $i\in\Lbrack 0,2^n-1\Rbrack$ such that 
$C\subset g^i(J_n^0)$. Since $g^i(J_n^0)\cap g^{i+k}(J_n^0)=\emptyset$ by
Lemma~\ref{lem:LY-htop0-summary-Jn0}(vi)+(viii), we have
$g^k(C)\cap C=\emptyset$.
Suppose for instance that $g^k(C)<C$. Then
$A_1=g^k(A_1)$ contains both $c_1$ and $g^k(c_1)$, and
we have $g^k(c_1)\in g^k(C)<c_0<c_1$. Thus
$c_0$ belongs to $A_1$ by connectedness, which is a contradiction. The
same conclusion holds if $g^k(C)>C$. We deduce that $c_0,c_1$ are two $g$-non
separable points in $\omega(0,g)$. 
By Theorem~\ref{theo:htop0-chaos-LY}, we conclude that
$g$ is chaotic in the sense of Li-Yorke.
\end{proof}

According to Lemmas \ref{lem:LY-htop0-2-infty}
and \ref{lem:LY-htop0-non-separable-K}, the map $g$ is
of type $2^\infty$ (and thus has zero topological entropy by
Theorem~\ref{theo:htop-power-of-2}) and 
is chaotic in the sense of Li-Yorke. These are the required properties
for the map $g$.
We show one more property for further reference.

\begin{lem}\label{lem:chaos-LY-htop0}
Let $g$ be the map defined above. Then
$g|_{\omega(0,g)}$ is transitive and sensitive to initial conditions.
\end{lem}

\begin{proof}
The point $\{0\}=\bigcap_{n=0}^{+\infty}J_n^0$ is in $B_K$, so
$0\in\omega(0,g)$. This implies that $0$ has a dense orbit
in $\omega(0,g)$, that is, $g|_{\omega(0,g)}$ is transitive.

We consider $(i_n)_{n\ge 0}$ and $c_0,c_1$ as in the
proof of Lemma~\ref{lem:LY-htop0-non-separable-K}. Let $\eps>0$.
Let $n\geq 0$ be such that $|J_n^0|<\eps$. Since $B_K\subset\omega(0,g)$
by Lemma~\ref{lem:LY-htop0-omega(0,f)}, the points $c_0,c_1$ belong
to $g^{i_n}(J_n^0)\cap \omega(0,g)$. 
Since $J_n^0$ is a periodic compact interval 
(Lemma~\ref{lem:LY-htop0-summary-Jn0}(vi)+(viii)) and $\omega(0,g)$ is
strongly invariant (Lemma~\ref{lem:omega-set}(i)), there exist $x_0,x_1
\in J_n^0\cap \omega(0,g)$ such that $g^{i_n}(x_0)=c_0$ and 
$g^{i_n}(x_1)=c_1$. Moreover, $g^{i_n}(0)\in J_n^0$.
Then the triangular inequality implies that, either
$|g^{i_n}(0)-g^{i_n}(x_0)|\geq \frac{|c_1-c_0|}2$, or
$|g^{i_n}(0)-g^{i_n}(x_1)|\geq \frac{|c_1-c_0|}2$. This implies that the point
$0$ is $\delta$-unstable with $\delta:=\frac{|c_1-c_0|}2$ for the map
$g|_{\omega(0,g)}$. Since $g|_{\omega(0,g)}$ is transitive,  the
map $g|_{\omega(0,g)}$ is  $\frac{\delta}{2}$-sensitive
by Lemma~\ref{lem:unstable}(iv).
\end{proof}
\end{ex}

\begin{rem}
As in Example~\ref{ex:non-chaos-LY-htop0}, the map in
Example~\ref{ex:chaos-LY-htop0} is made of countably many linear pieces.
There exist completely different examples. Indeed, there
exist $C^{\infty}$ weakly unimodal maps of zero topological entropy
and chaotic in the sense of Li-Yorke; an interval map
$f\colon [0,1]\to [0,1]$ is \emph{weakly unimodal}\index{weakly unimodal}
if $f(0)=f(1)=0$ and there is 
$c\in (0,1)$ such that $f|_{[0,c]}$ is non decreasing and $f|_{[c,1]}$ is
non increasing.
See the articles of Jim{é}nez L{ó}pez \cite{Jim6} or Misiurewicz-Smítal
\cite{MSmi}+\cite[p674]{Jim} (\cite{Jim} contains 
a correction concerning \cite{MSmi}).
\end{rem}

%************************************************************************
%Generic and dense chaos
\chapter[Other notions related to Li-Yorke pairs]{Other notions related to 
Li-Yorke pairs:\\ generic and dense chaos, distributional chaos}\label{chap6}

%***********************************************************************
\section{Generic and dense chaos}
\subsection{Definitions and general results}

Let $(X,f)$ be a topological dynamical system.
We recall that $(x,y)\in X^2$ is a \emph{Li-Yorke
pair of modulus $\delta>0$} if
$$
\limsup_{n\to+\infty}d(f^n(x),f^n(y))\ge \delta\quad\text{and}\quad
\liminf_{n\to+\infty}d(f^n(x),f^n(y))=0,
$$
and $(x,y)$ is a \emph{Li-Yorke pair} if it is a Li-Yorke pair of modulus
$\delta$ for some $\delta>0$.
Let $\LY(f,\delta)$ and $\LY(f)$ denote respectively the set of 
Li-Yorke pairs of modulus $\delta$ and the set of Li-Yorke pairs for $f$.
\index{Li-Yorke pair}
\label{notation:LY}
\index{LYf@$\LY(f,\delta)$, $\LY(f)$}

The definition of \emph{generic chaos} is due to Lasota (see
\cite{Pio}). It is somehow a two-dimensional notion since the 
Li-Yorke pairs live in $X^2$. Being inspired by this
definition, Snoha defined \emph{generic $\delta$-chaos}, \emph{dense
chaos} and \emph{dense $\delta$-chaos} \cite{Sno}.  

\begin{defi}[generic chaos, dense chaos]\index{chaos!generic
chaos}\index{chaos!generic $\delta$-chaos}\index{chaos!dense chaos}
\index{chaos!dense $\delta$-chaos}\index{generic chaos}
\index{generic $\delta$-chaos}\index{dense chaos}\index{dense $\delta$-chaos}
Let $(X,f)$ be a topological dynamical system and $\delta>0$.
\begin{itemize}
\item $f$ is \emph{generically chaotic} if $\LY(f)$ contains a dense
$G_{\delta}$-set of $X^2$.
\item $f$ is \emph{generically $\delta$-chaotic} if $\LY(f,\delta)$  
contains a dense $G_{\delta}$-set of $X^2$.
\item $f$ is \emph{densely chaotic} if $\LY(f)$ is dense in $X^2$.
\item $f$ is \emph{densely $\delta$-chaotic} if $\LY(f,\delta)$  is
dense in $X^2$.
\end{itemize}
\end{defi}

Some results hold for any dynamical system.
Trivially,  generic $\delta$-chaos implies both generic chaos and
dense $\delta$-chaos; and generic chaos (resp. dense $\delta$-chaos) 
implies dense chaos.  In \cite{Mur} Murinová showed that
generic $\delta$-chaos and dense $\delta$-chaos are equivalent 
(Proposition~\ref{prop:generic-dense-delta-chaos}). Moreover, topological 
weak mixing implies generic $\delta$-chaos for some $\delta>0$ 
(Proposition~\ref{prop:mixing-generic-chaos}); and dense $\delta$-chaos implies
sensitivity to initial conditions (Proposition~\ref{prop:dense-delta-chaos}). 
We start with a lemma; then we prove these three results.

\begin{lem}\label{lem:LY-Gdelta}
Let $(X,f)$ be a topological dynamical system
and $\delta\ge 0$. Let
\begin{gather*}
A(\delta):=\{(x,y)\in X^2\mid\limsup_{n\to+\infty} d(f^n(x),f^n(y))\ge\delta\},\\
B(\delta):=\{(x,y)\in X^2\mid\liminf_{n\to+\infty} d(f^n(x),f^n(y))\le\delta\}.
\end{gather*}
Then $A(\delta)$ and $B(\delta)$ are $G_{\delta}$-sets.
\end{lem}

\begin{proof}
Let
$$
\Delta_{\eps}:=\{(x,y)\in X^2\mid d(x,y)<\eps\}\quad\text{and}\quad
\overline{\Delta}_{\eps}=\{(x,y)\in X^2\mid d(x,y)\le\eps\}.
$$
Then $\Delta_{\eps}$ is open and $\overline{\Delta}_{\eps}$ is closed.
For every integer $n\ge 0$ and every $\eps\ge 0$, we set
$$
A_n(\eps):=\{(x,y)\in X^2\mid \exists i\ge n,\ d(f^i(x),f^i(y))>\eps\}
=\bigcup_{i\ge n}(f\times f)^{-i}(X^2\setminus \overline{\Delta}_{\eps}).
$$
Since $f\times f$ is continuous, the set $A_n(\eps)$ is open. Moreover,
$$
A(\delta)=\bigcap_{k\ge 1}\bigcap_{n\ge 0} A_n(\delta-1/k).$$
Thus $A(\delta)$ is a $G_{\delta}$-set.
Similarly, we set
$$
B_n(\eps):=\{(x,y)\in X^2\mid \exists i\ge n,\ d(f^i(x),f^i(y))<\eps\}
=\bigcup_{i\ge n}(f\times f)^{-i}(\Delta_{\eps}).
$$
The set $B_n(\eps)$ is open, and thus $B(\delta)$ is a $G_{\delta}$-set
because
$$
B(\delta)=
\bigcap_{k\ge 1}\bigcap_{n\ge 0} B_n(\delta+1/k).
$$
\end{proof}

\begin{prop}\label{prop:generic-dense-delta-chaos}
Let $(X,f)$ be a topological dynamical system and $\delta> 0$. 
Then $f$ is generically $\delta$-chaotic if and only if
it is densely $\delta$-chaotic.
\end{prop}

\begin{proof}
It is sufficient to notice that $\LY(f,\delta)=A(\delta)\cap B(0)$,
and thus $\LY(f,\delta)$ is a 
$G_{\delta}$-set by Lemma~\ref{lem:LY-Gdelta}.
\end{proof}

\begin{prop}\label{prop:mixing-generic-chaos}
Let $(X,f)$ be a topological dynamical system. If $f$ is topologically
weakly mixing, then it is generically $\delta$-chaotic with
$\delta:=\diam(X)$.
\end{prop}

\begin{proof}
By assumption, the system $(X\times X, f\times f)$ is transitive.  
Let $G$ be the set of pairs of dense orbits in $X^2$; it is a dense
$G_\delta$-set by Proposition~\ref{prop:transitive-dense-orbit}.
Since $X$ is compact, there exist $x_1,x_2\in X$ such that
$d(x_1,x_2)=\delta$, where $\delta=\diam(X)$. 
Then, for every $(x,y)\in G$, there exist two increasing
sequences of integers $(n_i)_{i\ge 0}$ and $(m_i)_{i\ge 0}$ such that 
$$
\lim_{i\to+\infty} (f^{n_i}(x),f^{n_i}(y))=(x_1,x_2)\quad
\text{and}\quad
\lim_{i\to+\infty}(f^{m_i}(x),f^{m_i}(y))=(x_1,x_1).
$$ 
Therefore
$$
\limsup_{n\to+\infty}d(f^n(x),f^n(y))\ge\delta\quad\text{and}\quad
\liminf_{n\to+\infty}d(f^n(x),f^n(y))=0,
$$
that is, $G\subset \LY(f,\delta)$.
\end{proof}

\begin{prop}\label{prop:dense-delta-chaos}
Let $(X,f)$ be a topological dynamical system and $\delta> 0$. 
If $f$ is densely $\delta$-chaotic, then $f$ is $\eps$-sensitive
for all $\eps\in (0,\frac{\delta}2)$.
\end{prop}

\begin{proof}
We fix $\eps\in (0,\frac{\delta}2)$.
Let $x\in X$ and let $U$ be a neighborhood of $x$. By density
of $\LY(f,\delta)$ in $X^2$, 
there exists $(y_1,y_2)$ in $U\times U$ such that
$$
\limsup_{n\to+\infty}d(f^n(y_1),f^n(y_2))\ge\delta>2\eps.
$$  
Using the triangular inequality, we see that
there exist $n\ge 0$ and $i\in\{1,2\}$ such that 
$d(f^n(x),f^n(y_i))\ge \eps$. Thus $x$ is $\eps$-unstable.
\end{proof}

\begin{rem}
The same argument as in the proof of Proposition~\ref{prop:dense-delta-chaos}
leads to the following result: if $f$ is densely chaotic, then every point
$x$ is $\eps$-unstable, for some $\eps>0$ depending on $x$.
\end{rem}

%*******************************************************************
\subsection{Preliminary results}

In this section, we state several lemmas for 
densely chaotic interval maps. They will be used to study both generic
chaos and dense chaos.

\begin{lem}\label{lem:dense-chaos}
Let $f$ be a densely chaotic interval map.
\begin{enumerate}
\item
If $J$ is a non degenerate interval, then $f^n(J)$ 
is non degenerate for all $n\ge 0$.
\item Let $J_1,\ldots,J_p$ be disjoint non degenerate intervals such
that $f(J_i)\subset J_{i+1}$ for all $i\in\Lbrack 1,p-1\Rbrack$
and $f(J_p)\subset J_1$.
Then either $p=2$ and $J_1, J_2$ have a common endpoint, or $p=1$. 
If  the intervals $(J_i)_{1\le i\le p}$ are closed, then $p=1$.
\item If $J, J'$ are non degenerate invariant intervals,
then $J\cap J'\neq\emptyset$.
\end{enumerate}
\end{lem}

\begin{proof}
i) Let $J$ be a non degenerate interval. By density of $\LY(f)$,
there exists $(x,y)\in J\times J$ such that 
$$
\limsup_{n\to+\infty}|f^n(x)-f^n(y)|>0.
$$ 
Thus $f^n(J)$ is non degenerate for infinitely many $n$.
Since the image of a degenerate interval is degenerate, 
$f^n(J)$ is non degenerate for all $n\ge 0$.

ii) Let $J_1,\ldots,J_p$ be disjoint non degenerate intervals such
that $f(J_i)\subset J_{i+1}$ for all $i\in\Lbrack 1,p-1\Rbrack$
and $f(J_p)\subset J_1$.
Suppose that there exist  two integers $i,j\in\Lbrack 1,p\Rbrack$ such that 
the distance $\delta$ between $J_i$ and $J_j$ is
positive. Since $f$ is uniformly continuous, there exists $\eta>0$ such 
that
$$
\forall x,y,\ |x-y|<\eta\Rightarrow \forall k\in\Lbrack 0,p\Rbrack,\  
|f^k(x)-f^k(y)|<\delta.$$  
Let $(x,y)\in J_i\times J_j$. Then $(f^{kp}(x),f^{kp}(y)) \in J_i\times J_j$
for all $k\ge 0$, and thus $|f^{kp}(x)-f^{kp}(y)|\ge \delta$. Thus 
$|f^n(x)-f^n(y)|\ge  \eta>0$ for all $n\ge 0$, 
which contradicts the fact that $I_i\times J_j$
contains Li-Yorke pairs.  Therefore, the distance between any two intervals
$J_i,J_j$ is null. If the intervals $J_1,\ldots, J_p$ are closed, this 
implies that $p=1$. Otherwise, this implies that $p=1$ or $p=2$; and, if 
$p=2$, then $J_1$ and $J_2$ have a common endpoint.

iii) Let $J, J'$ be two non degenerate invariant intervals. Since $\LY(f)$
is dense, there exist
$(x,x')$ in $J\times J'$ such that $\liminf_{n\to+\infty}|f^n(x)-f^n(x')|=0$. 
By compactness, there exist an increasing sequence of integers 
$(n_i)_{i\ge 0}$ and a point $z$ such that
$$
\lim_{i\to+\infty}f^{n_i}(x)=\lim_{i\to+\infty}f^{n_i}(x')=z.
$$
Since $J$ and $J'$ are invariant (and hence closed), the point $z$ belongs to $J\cap J'$.
\end{proof}

\begin{lem}\label{lem:nested-Jn-to-0}
Let $f$ be a densely chaotic interval map. Suppose that there
exists a sequence of non degenerate invariant intervals 
$(J_n)_{n\ge 0}$ such that $\lim_{n\to+\infty}|J_n|
=0$. Then there exists a fixed point $z$ in $\bigcap_{n\ge 0} J_n$. 
Moreover, there exists a sequence of non degenerate invariant intervals
$(J'_n)_{n\ge 0}$ such that
$\lim_{n\to+\infty}|J'_n|=0$ and, for all $n\ge 0$, $J'_{n+1}$ is included 
in the interior of $J'_n$ with respect to the induced topology on $J'_0$.
\end{lem}

\begin{proof}
First we are going to show that
\begin{equation}\label{eq:nonempty}
\bigcap_{n=0}^{+\infty} J_n\neq \emptyset.
\end{equation}
If the interval $\bigcap_{n=0}^N J_n$ is nonempty for all $N\ge 0$, 
then $\bigcap_{n=0}^{+\infty} J_n\neq \emptyset$
by compactness.  
Otherwise, let $N$ be the greatest integer such
that $\bigcap_{n=0}^N J_n$ is non degenerate.  Then the interval
$K:=\bigcap_{n=0}^N J_n$ is closed, non degenerate and invariant. 
By Lemma~\ref{lem:dense-chaos}(iii), $J_{N+1}\cap K\neq
\emptyset$, and thus the set $J_{N+1}\cap K$ is reduced to one point $z$
according to the definition of $N$. Then, again by 
Lemma~\ref{lem:dense-chaos}(iii), $J_n\cap K\neq\emptyset$ and $J_n\cap
J_{N+1}\neq \emptyset$ for all $n\ge 0$.
Thus $z$ belongs to $J_n$ by connectedness. 
We deduce that $z\in \bigcap_{n=0}^{+\infty}J_n$. This proves \eqref{eq:nonempty}.
The set $\bigcap_{n=0}^{+\infty}J_n$  is reduced to $\{z\}$ because 
$\lim_{n\to +\infty} |J_n|= 0$. Moreover, $f(z)=z$ because,
$f(J_n)\subset J_n$ for all $n\ge 0$.

Either there exist
infinitely many integers $n$ such that $J_n\cap (z,+\infty)\neq\emptyset$, or
there exist infinitely many $n$ such that $J_n\cap (-\infty,
z)\neq\emptyset$.  Without loss of generality, we may suppose that 
the first case holds, that is,
there exists an increasing sequence $(n_i)_{i\ge 0}$
such that, $\forall i\ge 0$,  $J_{n_i}\supset [z,z+\eps_i]$ for some $\eps_i>0$.
We set $K_m:= \bigcap_{i=0}^m J_{n_i}$ for all $m\ge 0$.
Then $K_m$ is a non degenerate invariant
interval and $K_{m+1}\subset K_m$. We split the end of the proof into two
cases.

\textbf{Case 1.} There exists an increasing sequence $(m_i)_{i\ge
0}$ such that $K_{m_{i+1}}\subset \Int{K_{m_i}}$ for all $i\ge 0$.
Then we define $J'_i:=K_{m_i}$ and we get a suitable sequence of intervals.

\textbf{Case 2.} Suppose that the assumption of Case 1 is not satisfied.
Then there exists $M\ge 0$ such that $K_m\not\subset \Int{K_M}$ for all 
$m\ge M$. Since $(K_m)_{m\ge 0}$ is a sequence of nested intervals,
this implies that
\begin{equation}\label{eq:minmax}
\text{either }\forall m\ge M,\ \min K_m=\min K_M,
\quad \text{or }\forall m\ge M,\ \max K_m=\max K_M.
\end{equation} 
Since $\lim_{m\to+\infty}|K_m|= 0$, there exists 
an increasing sequence of integers $(m_i)_{i\ge 0}$ with $m_0=M$ such that
$|K_{m_{i+1}}|<|K_{m_i}|$ for all $i\ge 0$. Together with \eqref{eq:minmax},
this implies that $K_{m_{i+1}}$ is included in the interior of $K_{m_i}$
for the induced topology on $K_M$.
Then $J'_i:=K_{m_i}$ gives a suitable sequence of intervals.
\end{proof}

The next lemma will be an important tool. It 
gives a sufficient condition for  a densely
chaotic map to be generically $\delta$-chaotic for some $\delta>0$.

\begin{lem}\label{lem:dense-chaos-dense-delta-chaos}
Let $f\colon I\to I$ be a densely chaotic interval map.
Suppose that there exists $\eps>0$ such that every non degenerate invariant
interval has a length greater than or equal to $\eps$. Then
there exists $\delta>0$ such that $f$ is generically $\delta$-chaotic.
\end{lem}

\begin{proof}
Suppose that
\begin{equation}\label{eq:fnJ<delta}
\forall \delta>0,\: \exists\, J \text{ non degenerate closed interval
 such that } \forall n\ge 0,\: |f^n(J)|\le\delta.
\end{equation}
We are going to show that this is impossible. 
We fix $\delta\in (0,\frac{\eps}4)$. Let $J$ be a
non degenerate closed interval such that $|f^n(J)|\le\delta$ for all
$n\ge 0$. There exists a Li-Yorke pair in $J\times J$, which implies that
\begin{equation}\label{eq:limsup-fnJ}
\limsup_{n\to+\infty}|f^n(J)|>0.
\end{equation}
Thus there exist positive integers $N,p$ such that $f^N(J)\cap
f^{N+p}(J)\neq \emptyset$ (otherwise, all $(f^n(J))_{n\ge 0}$ would be
disjoint, and \eqref{eq:limsup-fnJ} could not hold because $I$ has a finite
length). Let $X:=\bigcup_{n\ge N} f^n(J)$. The set $X$ has
at most $p$ connected components, which are cyclically mapped under $f$. 
Moreover, the connected components
of $X$ are non degenerate by Lemma~\ref{lem:dense-chaos}(i). Thus, according to
Lemma~\ref{lem:dense-chaos}(ii), $X$ has either one connected component
or two connected components with a common endpoint. In both
cases, $\overline{X}$ is an interval.

Suppose that there exist a point $z$ and an integer $n_0\ge N$ such that
$f^2(z)=z$ and $z\in f^{n_0}(J)$. Then $z\in f^{n_0+2k}(J)$ for all $k\ge 0$. 
By assumption on $J$, $|f^n(J)|\le\delta$ for all integers $n\ge 0$,
which implies that
$\left|\bigcup_{k\ge 0}f^{n_0+2k}(J)\right|\le 2\delta$ and
$\left|\bigcup_{k\ge 0}f^{n_0+2k+1}(J)\right|\le 2\delta$
(these two sets are intervals containing respectively $z$ and $f(z)$).
Let $Y:=\overline{\bigcup_{n\ge n_0}f^n(J)}=f^{n_0-N}(\overline{X})$.
Then $Y$ is a non degenerate closed interval, $f(Y)\subset Y$ and
$|Y|\le 4\delta$. Moreover, $|Y|\ge\eps$ according to the hypothesis 
of the lemma,
which is a contradiction because $\delta<\eps/4$. We deduce that 
\begin{equation}\label{eq:Xf2z}
X\text{ contains no point }z\text{ such that }f^2(z)=z.
\end{equation}

Let $X_0$ be the connected component of $X$ containing $f^N(J)$. We
set $g:=f^2$. Then $g(X_0)\subset X_0$ (because $X$ has at most two
connected components) and $g|_{X_0}$ has no fixed point by \eqref{eq:Xf2z}. 
By continuity, either 
\begin{equation}\label{eq:fx<x}
\forall x\in X_0,\ g(x)<x,
\end{equation}
or 
$$\forall x\in X_0,\ g(x)>x.$$ We assume that \eqref{eq:fx<x} holds, the other 
case being symmetric. Let $a:=\inf X_0$. The fact that $g(X_0)\subset X_0$
combined with \eqref{eq:fx<x} implies that $g(a)=a$ (and $a\notin X_0$). 
Let $b:=\max f^N(J)$.
We set $b_n:=\max g^n([a,b])$ for every $n\ge 0$.
Then, for every $n\ge 0$, there exists $x_n\in [a,b]$ such that
$g^{n+1}(x_n)=b_{n+1}$. Thus, by \eqref{eq:fx<x},
$b_{n+1}=g(g^n(x_n))\le g^n(x_n)\le b_n$. Therefore,
the sequence $(b_n)_{n\ge 0}$ is non increasing, and thus
has a limit in $\overline{X_0}$. We set $b_\infty:=\lim_{n\to+\infty} b_n$.
Note that $b_{\infty}\in X_0\cup\{a\}$ because
$a\le b_{\infty}\le b$ and $b\in X_0$. We have
$$
\bigcap_{n\ge 0}g^n([a,b])=\bigcap_{n\ge 0}[a,b_n]=[a,b_\infty]
$$
and
\begin{eqnarray*}
\bigcap_{n\ge 0}g^n([a,b])&=&g(\bigcap_{n\ge 0}g^n([a,b]))\text{ because the intersection
is decreasing}\\
&=&g([a,b_\infty])\supset [a,g(b_\infty)].
\end{eqnarray*}
Hence $g([a,b_{\infty}])=[a,b_{\infty}]$. Thus there exists $x\in 
[a,b_{\infty}]$ such that $g(x)=b_{\infty}$, which implies that $g(x)\ge x$.
According to \eqref{eq:fx<x}, this is possible only if $x=a$, and hence
$b_{\infty}=g(a)=a$.
Since $|g^{n+N}(J)|\le |b_n-a|$, we have 
$\lim_{n\to+\infty} |g^{n+N}(J)|=0$.
By continuity of $f$, this implies that 
$\lim_{n\to+\infty}|f^{n}(J)|= 0$, which contradicts \eqref{eq:limsup-fnJ}. We
conclude that \eqref{eq:fnJ<delta} does not hold, that is, there
exists $\delta>0$ such that
\begin{equation}\label{eq:fnJ}
\text{for every non degenerate closed interval }J,\ \exists n\ge 0,\ 
|f^n(J)|>\delta.
\end{equation}
Let $J$ be a non degenerate closed interval. Then the closed interval
$f^n(J)$  is also non degenerate by Lemma~\ref{lem:dense-chaos}(i). Thus,
according to \eqref{eq:fnJ},
\begin{equation}\label{eq:limsup-fnJ>delta}
\limsup_{n\to+\infty}|f^n(J)|\ge\delta.
\end{equation}
We define
$$
A_k(\eta):=\{(x,y)\in I\times I\mid\ \exists i\ge k,\, |f^i(x)-f^i(y)|>\eta\}
$$
and
$$
A(\delta):=\bigcap_{n\ge 1}\bigcap_{k\ge 0} A_k(\delta-1/n)=
\{(x,y)\in I\times I\mid \limsup_{k\to+\infty} |f^k(x)-f^k(y)|\ge\delta\}.
$$
We are going to show that $A_k(\eta)$ is dense for all $\eta<\delta$ and all
$k\ge 0$.  Let $J_1, J_2$ be two non degenerate closed intervals. We
consider two cases.

\noindent$\bullet$
For some $m\ge 0$, $f^m(J_1)\subset f^m(J_2)$. By
\eqref{eq:limsup-fnJ>delta}, there exists $n\ge \max\{k,m\}$
such that $|f^n(J_1)|\ge\delta>\eta$, thus there exist $x,x'\in J_1$
such that $|f^n(x)-f^n(x')|>\eta$ and there exists $y\in J_2$ such
that $f^n(y)=f^n(x')$. Consequently $A_k(\eta)\cap (J_1\times
J_2)\neq \emptyset$.

\noindent$\bullet$
For all $m\ge 0$, $f^m(J_1)\setminus f^m(J_2)\neq\emptyset$.  By
\eqref{eq:limsup-fnJ>delta}, there exists $n\ge k$ such that
$|f^n(J_2)|\ge \delta>\eta$, thus there exist $x,x'\in J_2$ such that
$|f^n(x)-f^n(x')|>\eta$. By assumption there exists $y\in J_1$ such
that $f^n(y)\notin f^n(J_2)$. 
Since $f^n(J_2)$ is an interval containing $x,x'$ but not $y$, we have
either $|f^n(y)-f^n(x)|>\eta$ or
$|f^n(y)-f^n(x')|>\eta$. Consequently $A_k(\eta)\cap (J_1\times J_2)
\neq\emptyset$.

The sets $A_k(\eta)$ are open and dense in $I\times I$. Thus, 
according to the Baire category theorem, $A(\delta)$ is a dense 
$G_{\delta}$-set. Moreover, the set 
$$
B(0):=\{(x,y)\in I\times I\mid \liminf_{n\to+\infty}|f^n(x)-f^n(y)|=0\}
$$
is a $G_{\delta}$-set by Lemma~\ref{lem:LY-Gdelta} and it is dense
because $f$ is
densely chaotic. Therefore the set $\LY(f,\delta)=B(0)\cap A(\delta)$
is a dense $G_{\delta}$-set and $f$ is generically $\delta$-chaotic.
\end{proof}

%******************************
\subsection{Generic chaos and transitivity}
Using the structure of transitive non mixing interval maps
(Theorem~\ref{theo:summary-transitivity}), one can show that
Proposition~\ref{prop:mixing-generic-chaos} implies that any
transitive interval map is generically $\delta$-chaotic for some
$\delta>0$. The converse is not true (see 
Example~\ref{ex:generic-chaos-not-transitive} below), 
yet it is partially true since a
generically $\delta$-chaotic interval map has exactly one or two
transitive  intervals, as it was shown by Snoha; he also
proved that for an interval map, the notions of generic
$\delta$-chaos, generic  chaos and dense $\delta$-chaos are
equivalent (Theorem~\ref{theo:generic-chaos-etc} below).

We shall need the Kuratowski-Ulam Theorem \cite{Oxt, Kur}.

\begin{theo}[Kuratowski-Ulam]\label{theo:Kuratowski}
Let $X,Y$ be complete metric spaces. If $G$ is a dense
$G_{\delta}$-set in $X\times Y$, then there exists a dense
$G_{\delta}$-set $A\subset X$ such that, for all $x\in A$, the set
$\{y\in Y\mid (x,y)\in G\}$ is a dense $G_{\delta}$-set.
\end{theo}

The next theorem is due to Snoha \cite{Sno}.

\begin{theo}\label{theo:generic-chaos-etc}
Let $f$ be an interval map.  The following properties are
equivalent:
\begin{enumerate}
\item $f$ is generically chaotic,
\item $f$ is generically $\delta$-chaotic for some $\delta>0$,
\item $f$ is densely $\delta$-chaotic for some $\delta>0$,
\item either there exists a unique non degenerate transitive 
interval, or there exist exactly two non degenerate transitive
intervals having a common endpoint; moreover, if $J$ is a non
degenerate interval, then $f(J)$ is non degenerate, and  there exist a
transitive interval $I_0$ and an integer $n\ge 0$ such that $f^n(J)\cap
\Int{I_0}\neq\emptyset$.
\end{enumerate}
Moreover, (ii) and (iii) hold with the same $\delta$.
\end{theo}

\begin{proof}
The implication (ii)$\Rightarrow$(i) is
trivial and the equivalence (ii)$\Leftrightarrow$(iii) with the same $\delta$
is given by Proposition~\ref{prop:generic-dense-delta-chaos}.
We are going to show the implications (i)$\Rightarrow$(ii) and
(iii)$\Rightarrow$(iv)$\Rightarrow$(ii).

\medskip (i)$\Rightarrow$(ii).\par
Assume that $f$ is generically chaotic.
Suppose that there exists a sequence of non degenerate invariant
intervals $(J_n)_{n\ge 0}$ such that
$|J_n|\to 0$ when $n$ goes to infinity. We are going to show that
this situation is impossible. According to Lemma~\ref{lem:nested-Jn-to-0},
we may assume that $J_{n+1}\subset \Int{J_n}$ for the induced
topology of $J_0$. From now on, we work in $J_0$;
notice that $f|_{J_0}\colon J_0\to J_0$ 
is a generically chaotic interval map and
$\LY(f|_{J_0})=\LY(f)\cap (J_0\times J_0)$.
By compactness, $\bigcap_{n\ge 0}J_n$ is nonempty, and
hence is reduced to a single point $\{z\}$. 

The set $\LY(f)\cap (J_0\times J_n)$ is a dense $G_{\delta}$-set in
$J_0\times J_n$. Thus, by Theorem~\ref{theo:Kuratowski}, there exists a
dense $G_{\delta}$-set $A_n$ in $J_0$ such that for all $x\in A_n$, there
exists $y\in J_n$ with $(x,y)\in \LY(f)$. According to the Baire category
theorem (see Corollary~\ref{cor:baire2}), $A:=\bigcap_{n\ge 0} A_n$ is a dense $G_{\delta}$-set in
$J_0$. Let $x\in A$ and $n\ge 0$. There exists $y\in
J_{n+1}$ such that $(x,y)\in \LY(f)$; in particular
$\liminf_{k\to+\infty} |f^k(x)-f^k(y)|=0$. Since $J_{n+1}$ is included in
$\Int{J_n}$ and the intervals $(J_n)_{n\ge 0}$ are invariant, 
this implies that there exists $p\ge 0$ such that
$f^p(x)\in \Int{J_n}$, and hence $f^k(x) \in J_n$ for all $k\ge p$. 
Since this is true for all $n\ge 0$, we have
$\lim_{k\to+\infty}f^k(x)=z$ (recall that $\bigcap_{i\ge 0} J_n=\{z\}$).
On the other hand, $A\times A$ is a dense
$G_{\delta}$-set in $J_0\times J_0$, and thus $(A\times A)\cap
\LY(f)\neq\emptyset$.  This leads to a contradiction because
$$
\forall (x,x')\in A\times A,
\lim_{k\to+\infty}f^k(x)=\lim_{n\to+\infty}f^k(x')=z,
$$
and thus $(x,x')$ is not a Li-Yorke pair.
This shows that there exists $\eps>0$ such that
$$
\text{if }J\text{ is a non degenerate invariant interval, then }|J|\ge\eps.
$$
Then Lemma~\ref{lem:dense-chaos-dense-delta-chaos} applies:
the map $f$ is generically $\delta$-chaotic for some $\delta>0$.

\medskip (iii)$\Rightarrow$(iv).\par Suppose that $f$ is densely 
$\delta$-chaotic.
According to Proposition~\ref{prop:dense-delta-chaos}, the map $f$ is
sensitive to initial conditions. By
Proposition~\ref{prop:sensitivity-transitive-component}, there exist
some non degenerate closed intervals $I_1,\ldots, I_p$ such that
$f(I_i)= I_{i+1\bmod p}$ for all $i\in\Lbrack 1,p\Rbrack$
and $f|_{I_1\cup\cdots\cup I_p}$ is
transitive; by Lemma~\ref{lem:dense-chaos}(ii), we have $p=1$, that is,
the interval $I_1$ is transitive.

Suppose that $I_2$ is another non degenerate transitive 
interval.  Then $I_1\cap I_2\neq\emptyset$ by Lemma~\ref{lem:dense-chaos}.
If $\Int{I_1\cap I_2}\neq\emptyset$, then
$I_1=I_2= \overline{\CO_f(I_1\cap I_2)}$; otherwise $I_1\cap I_2$ is
reduced to a single point. Since the ambient space is an interval,
we conclude that either there is a unique non degenerate
transitive interval or there are exactly two non degenerate
transitive intervals which have a common endpoint.

Finally consider a non degenerate interval $J$. According to
Lemma~\ref{lem:dense-chaos}(i), $f(J)$ is non degenerate.  Since
$(J\times J)\cap \LY(f,\delta)\neq\emptyset$, we have
$\limsup_{n\to+\infty} |f^n(J)|\ge\delta$, which implies that there
exist some integers $i,p>0$ such that $f^i(J)\cap
f^{i+p}(J)\neq\emptyset$. Let $X:=\overline{\bigcup_{n\ge 0}
f^{n+i}(J)}$. The set $X$ has at most $p$ connected components, which are
non degenerate closed intervals and are mapped cyclically under $f$.
Thus, by Lemma~\ref{lem:dense-chaos}(ii), $X$ is an interval. Moreover,
$f(X)\subset X$ and $f|_X$ is sensitive.  Thus, by
Proposition~\ref{prop:sensitivity-transitive-component} and
Lemma~\ref{lem:dense-chaos}(ii), there exists a non degenerate invariant
interval $K\subset X$ such that $f|_K$ is transitive. According to the
definition of $X$, there is some integer $n\ge 0$ such that
$f^n(J)\cap \Int{K}\neq\emptyset$.

\medskip (iv)$\Rightarrow$(ii).\par 
First we show the following fact.

\medskip\textsc{Fact 1.}
Suppose that the image of
a non degenerate interval is non degenerate, that there exist two
non degenerate invariant intervals $I_1,I_2$ such that $f|_{I_i}$ is
topologically mixing for $i\in\{1,2\}$ ($I_1=I_2$ is allowed) and that, 
for every non
degenerate interval $J$, there exist $n\ge 0$ and $i\in\{1,2\}$ such
that $f^n(J)\cap\Int{I_i}\neq\emptyset$. Then $f$ is
generically $\delta$-chaotic with
$\delta:=\min\{|I_1|,|I_2|\}$.

\medskip
Let $i,j\in\{1,2\}$. Both
$f|_{I_i}$ and $f|_{I_j}$ are topologically mixing, thus 
$(f\times f)|_{I_i\times I_j}$ is transitive by
Proposition~\ref{prop:weakly-mixing-product}. Let $G_{ij}$
be the set of points $(x,y)\in I_i\times I_j$ whose orbit is dense in
$I_i\times I_j$. According to
Proposition~\ref{prop:transitive-dense-orbit}, $G_{ij}$ is a
dense $G_{\delta}$-set in $I_i\times I_j$. By
Lemma~\ref{lem:dense-chaos}(iii), there exists a point $z$ in $I_1\cap I_2$.
We choose $(x_1,x_2)\in I_1\times I_2$ such that $|x_1-x_2|=\delta$.  For
every $(x,y)\in G_{ij}$, there exists a subsequence of 
$(f^n(x),f^n(y))_{n\ge 0}$ that converges to $(z,z)$ and another one that 
converges to 
$(x_1,x_2)$; thus $(x,y)$ is a Li-Yorke pair of modulus $\delta$. It
is clear that $(f\times f)^{-1}\LY(f,\delta)\subset \LY(f,\delta)$,
and thus we get
$$
G:= \bigcup_{i,j\in\{1,2\}}\bigcup_{n\ge 0}(f\times
f)^{-n}(G_{ij}) \subset \LY(f,\delta).
$$
Then $G$ is a $G_{\delta}$-set (see Propositions \ref{prop:unionGdelta}
and \ref{prop:f-1-Gdelta}). We are going to show that
$G$ is dense.

Let $U,V$ be two nonempty open intervals. By assumption, there
exist integers $N,M\ge 0$ and $i,j\in\{1,2\}$ such that $f^N(U)\cap
\Int{I_i}\neq \emptyset$ and $f^M(V)\cap
\Int{I_j}\neq\emptyset$. Let $U_0\subset U$ and $V_0\subset V$ be
nonempty open subintervals such that $f^N(U_0)\subset I_i$ and
$f^M(V_0)\subset I_j$.  If $n:=\max\{N,M\}$, then $f^n(U_0)\subset I_i$
and $f^n(V_0)\subset I_j$ because $I_1, I_2$ are invariant.  
Since the intervals
$f^n(U_0),f^n(V_0)$ are non degenerate by assumption, there exists $(x,y)\in
(f^n(U_0)\times f^n(V_0))\cap G_{ij}$; in other words, 
$$
(U_0\times V_0) \cap (f\times f)^{-n}(G_{ij})\neq\emptyset.
$$
Therefore the set
$G$ is dense and $f$ is generically $\delta$-chaotic. This proves
Fact 1.

\medskip Now we assume that (iv) holds with no additional hypothesis. If
there is a unique non degenerate transitive  interval $I_1$, then,
by Theorem~\ref{theo:summary-transitivity}, either $f|_{I_1}$ is
topologically mixing or there exist two non degenerate closed intervals 
$J,K$ such that $I_1=J\cup K$ and $f^2|_J,f^2|_K$ are topologically mixing. 
In the first case, Fact~1
gives the conclusion (taking $I_2:=I_1$). In the second
case, Fact~1 applied to $f^2$ (with $I_1:=J$ and $I_2:=K$) shows that
$f^2$ is generically
$\delta$-chaotic, and thus $f$ is generically $\delta$-chaotic too. If
there are two different non degenerate transitive intervals
$I_1, I_2$ then, by Lemma~\ref{lem:dense-chaos}(iii), $I_1\cap
I_2\neq\emptyset$, and this intersection must be reduced to a single
point because $I_1\neq I_2$; we call $z$ this common endpoint. Since
$f(z)\in I_1\cap I_2=\{z\}$, we have $f(z)=z$. Thus,
by Theorem~\ref{theo:summary-transitivity}, both maps $f|_{I_1}$ and $f|_{I_2}$
are topologically mixing. Consequently, Fact~1 applies and $f$ is generically
$\delta$-chaotic.
\end{proof}

\begin{rem}
As indicated by Snoha (see \cite{Mur}), there is a misprint in the 
statement of Theorem~1.2 in \cite{Sno}, where 
the condition ``if $J$ is a non degenerate interval, then
$f(J)$ is non degenerate'' in point~(h) (equivalent to our
Theorem~\ref{theo:generic-chaos-etc}(iv)) is omitted.
\end{rem}

\begin{rem}
In \cite{Sno}, Snoha gave several other properties equivalent to generic chaos
for an interval map $f$, in particular $f$ is generically $\delta$-chaotic
if and only if, for every two non degenerate intervals $J,J'$, one has
$$
\liminf_{n\to+\infty}d(f^n(J),f^n(J'))=0\text{ and }
\limsup_{n\to+\infty}|f^n(J)|\ge\delta.
$$
%where $d(A,B)=\inf\{|x-y|; x\in A, y\in B\}$ is the distance between two sets.
Murinová proved that this result is still true in more 
general spaces except that the equivalence does not hold with the same 
$\delta$ \cite{Mur}.
She also built a continuous map on a compact subset of $\IR^2$
which is generically chaotic but not generically $\delta$-chaotic for any
$\delta>0$.
\end{rem}

\begin{ex}\label{ex:generic-chaos-not-transitive}
Figure~\ref{fig:generic-chaos-not-transitive} represents an 
interval map $f\colon [0,a]\to [0,a]$ (for some fixed $a>1$) 
which is generically chaotic 
but not transitive.
The restriction of $f$ to $[0,1]$  is the
tent map $T_2$ (Example~\ref{ex:tent-map}).
Thus $f|_{[0,1]}$ is transitive and $f$ is not transitive. 
The interval $[1,a]$ is mapped linearly
onto $[0,1]$ (thus, for every non degenerate interval $J\subset [1,a]$,
$f(J)$ is non degenerate and $f(J)\subset [0,1]$). It is then clear
that condition (iv) of Theorem~\ref{theo:generic-chaos-etc} is
satisfied, and thus $f$ is generically chaotic.

\begin{figure}[htb]
\centerline{\includegraphics{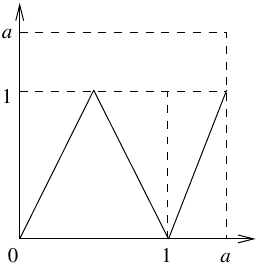}}
\caption{This map is generically chaotic but not transitive.}
\label{fig:generic-chaos-not-transitive}
\end{figure}
\end{ex}

%*******************************************************************
\subsection{Dense chaos}

For interval maps, generic $\delta$-chaos, generic chaos and dense
$\delta$-chaos are equivalent and imply dense chaos; but the reverse 
implication does not hold, as shown in 
Example~\ref{ex:dense-chaos-not-generic}. Next we shall
give a result on the structure of densely, non generically chaotic interval 
maps: in this situation, there exists a decreasing sequence of invariant 
intervals, and each of them contains a horseshoe
for the second iterate of the map (Theorem~\ref{theo:structure-dense-chaos}).

\begin{ex}\label{ex:dense-chaos-not-generic}
We are going to exhibit an interval map that is densely
chaotic but has no non degenerate transitive interval, and hence is 
not generically
chaotic according to Theorem~\ref{theo:generic-chaos-etc}. By
Proposition~\ref{prop:sensitivity-transitive-component},
this map is not sensitive  either.
This example is originally due to Mizera (see \cite{Sno}).
For all $n\ge 0$, we set
$$
a_n:=1-\frac{1}{3^n},\quad b_n:=1-\frac{1}{4\cdot 3^{n-1}},\quad
c_n:=1-\frac{1}{2\cdot 3^n}\quad\text{and}\quad J_n:=[a_n,1].
$$ 
These points are ordered as follows:
$$
a_0:=0<b_0<c_0<a_1<b_1<c_1<a_2<\cdots<a_n<b_n<c_n<a_{n+1}<\cdots<1.
$$
Then we define the continuous map $f\colon [0,1]\to [0,1]$ by
\begin{gather*}
\forall n\ge 0,\ f(a_n):=a_n,\ f(b_n):=1,\ f(c_n):=a_n,\\
f(1):=1,
\end{gather*}
and $f$ is linear on the intervals
$[a_n,b_n]$, $[b_n,c_n]$ and $[c_n,a_{n+1}]$ for all $n\ge 0$; see
Figure~\ref{fig:dense-chaos-not-generic}.

\begin{figure}[htb]
\centerline{\includegraphics{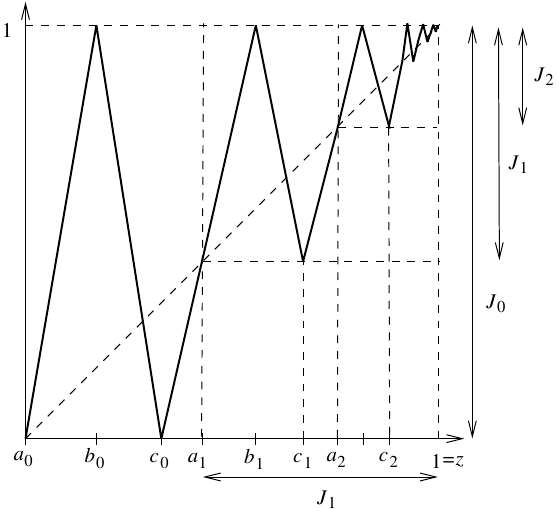}}
\caption{This map is densely chaotic but not generically chaotic.}
\label{fig:dense-chaos-not-generic}
\end{figure}

It is clear from the definition that $J_n$ is invariant
and $([a_n,b_n],[b_n,c_n])$ is a horseshoe for all $n\ge 0$.
Thus, $f|_{J_n}$ is chaotic in the sense of Li-Yorke by
Theorems \ref{theo:Misiurewicz} and \ref{theo:htop-positive-chaos-LY}. In particular, 
\begin{equation}\label{eq:LY_In}
\forall n\ge 0,\ \LY(f)\cap (J_n\times J_n)\neq \emptyset.
\end{equation}
A straightforward computation shows that the absolute value of the slope of $f$
is equal to $4$ on each linear piece. Let $J,J'$ be two non
degenerate intervals. By Lemma~\ref{lem:N-critical-points}, there
exists $n\ge 0$ such that $f^n(J)$ contains three distinct points 
in $\{a_k,b_k,c_k\mid k\ge 0\}$, which implies that
$f^{n+1}(J)\supset J_p$ for some $p\ge
0$. Similarly, there exist $m,q\ge 0$ such that $f^{m+1}(J')\supset J_q$.
Let $N:=\max\{n+1,m+1\}$ and $k:=\max\{p,q\}$. Then 
$f^N(J)\cap f^N(J')\supset J_k$ because $J_k\subset J_p\cap
J_q$ and $f(J_k)= J_k$.  It is obvious that
$(f\times f)^{-N}(\LY(f))\subset \LY(f)$. Thus \eqref{eq:LY_In} 
implies that $(J\times J')\cap \LY(f)\neq \emptyset$. 
In other words, $f$ is densely chaotic.

Let $J$ be a non degenerate invariant interval. Then, as shown 
above, there exist  $n,k\ge 0$ such that $f^n(J)\supset J_k$.
Moreover, $J_k$ strictly contains the invariant interval $J_{k+1}$. 
Thus $f|_J$ is not transitive. We conclude that $f$ has no 
non degenerate transitive interval.
\end{ex}

The next result is due to the author \cite{R8}.

\begin{theo}\label{theo:structure-dense-chaos}
Let $f$ be a densely chaotic interval map that is not generically chaotic.
Then there exists a sequence of non degenerate invariant intervals
$(J_n)_{n\ge 0}$ such that $\lim_{n\to+\infty}|J_n|=0$,
$J_{n+1}\subset J_n$ and $f^2|_{J_n}$ has a horseshoe for all $n\ge 0$.
\end{theo}

\begin{proof}
According to
Lemma~\ref{lem:dense-chaos-dense-delta-chaos}, 
for every $\eps>0$, there exists a non degenerate invariant interval $J$
such that  $|J| <\eps$. Thus there exists a sequence of invariant
non degenerate closed subintervals $(I_n)_{n\ge 0}$ such that 
$\lim_{n\to+\infty}|I_n|= 0$.
Then, by Lemma~\ref{lem:nested-Jn-to-0}, there exists 
a sequence of invariant non
degenerate intervals $(J_n)_{n\ge 0}$ such that 
$\lim_{n\to+\infty}|J_n|=0$ and $J_{n+1}\subset \Int{J_n}$ with respect
to the induced topology on $J_0$ for all $n\ge 0$; 
moreover, there is a fixed point $z$
such that $\bigcap_{n\ge 0}J_n=\{z\}$.
From now on, we consider $J_0$ as the ambient space; in particular, 
when speaking about the interior of a set, it is with respect to $J_0$. 
We fix an integer $n_0\ge 0$.
We are going to show that $f^2|_{J_{n_0}}$ has a horseshoe.
Assume on the contrary that $f^2|_{J_{n_0}}$ has no horseshoe. 
We set
$$\CP:=\{x\in J_{n_0}\mid \exists p\ge 1,\lim_{n\to+\infty}f^{np}(x)
\text{ exists}\}.
$$
If $x,y\in \CP$, then $(x,y)$ is not a Li-Yorke pair. Thus the set 
$J_{n_0}\setminus \CP$ is not empty because the map $f|_{J_{n_0}}$
is densely chaotic.
Let $x_0\in J_{n_0}\setminus \CP$; we set $x_n:=f^n(x_0)$ for all
$n\ge 1$. Since $x_0\notin\CP$, the sequence $(x_n)_{n\ge 0}$ is not 
eventually monotone. Thus, according to Lemma~\ref{lem:alternating},
there exist a fixed point $c$ and an integer $N$ such that
$x_{N+2n}<c<x_{N+2n+1}$ for all $n\ge 0$. We assume $c\le z$, the
case $c\ge z$ being symmetric. Since $x_0\notin\CP$,
the sequence $(x_{N+2n})_{n\ge 0}$ is not eventually monotone, so
there exists $i\ge 0$ such that 
$$
x_{N+2i+2}<x_{N+2i}<c\le z.
$$
By continuity, there exists a non degenerate closed interval $K$ containing
$x_{N+2i}$ such that $z\notin K$ and
\begin{equation}\label{eq:Kygy}
\forall y\in K,\ f^2(y)<y.
\end{equation}
Since $z\in \bigcap_{k\ge 0}J_k$ and $\lim_{k\to +\infty}|J_k|=0$,
there exists $k_0\ge n_0$ such that 
\begin{equation}\label{eq:KJk}
K<J_{k_0}.
\end{equation}

The set $K\times K$ contains a Li-Yorke pair because $f$ is densely
chaotic. Thus $\limsup_{n\to+\infty} |f^n(K)|>0$ and there exist positive 
integers $p,q$ such that $f^{q+p}(K)\cap f^q(K)
\neq\emptyset$. Let $L:=\overline{\bigcup_{n\ge q}f^n(K)}$. 
The set $L$ is invariant, and the same argument as for $\overline{X}$ in
the proof of Lemma~\ref{lem:dense-chaos-dense-delta-chaos} shows that
$L$ is a non degenerate interval.
Moreover, $L\cap J_k\neq\emptyset$ for all $k\ge n_0$ by
Lemma~\ref{lem:dense-chaos}(iii). Since $J_{k_0+1}\subset
\Int{J_{k_0}}$, this implies that there exists $n\ge 0$ such that
$f^n(K)\cap \Int{J_{k_0}}\neq \emptyset$. Thus
there exists a non degenerate closed subinterval $K'\subset K$ such that
$f^n(K')\subset J_{k_0}$.
We set  $g:=f^2|_{J_{n_0}}$ and we fix $m_0\ge n/2$. For all $y\in K'$ and
all $m\ge m_0$, we have $g^m(y)\in J_{k_0}$ because $J_{k_0}$ is invariant.
Hence, by \eqref{eq:Kygy} and \eqref{eq:KJk},
\begin{equation}\label{eq:gyy}
\forall m\ge m_0,\ g(y)<y<g^m(y).
\end{equation}
This implies that there exists $j\in\Lbrack 1, m_0-1\Rbrack$ 
such that $g^j(y)<g^{j+1}(y)$. Let
\begin{gather*}
U(y):=\{y'\in \CO_g(y)\mid g(y')>y'\},\\
D(y):=\{y'\in \CO_g(y)\mid g(y')<y'\}.
\end{gather*}
We have $y\in D(y)$ by  \eqref{eq:gyy} and $g^j(y)\in U(y)$ according to the 
choice of $j$. By assumption, the map $g$ has no horseshoe. Thus, 
according to Lemma~\ref{lem:U<D},
\begin{equation}
U(y)\le D(y).\label{eq:yDU}
\end{equation}
Moreover, for all $m\ge m_0$, $y\le g^m(y)$, so
$g^m(y)\in D(y)$ by \eqref{eq:yDU} and because $y\in D(y)$. This
implies that $g^{m+1}(y)\le g^m(y)$. 
Therefore, the sequence $(g^m(y))_{m\ge m_0}$
is non increasing, and hence convergent. But this implies that
$K'\times K'$ contains no Li-Yorke pair, which contradicts the fact that
$f$ is densely chaotic. We conclude that $f^2|_{J_{n_0}}$ has a horseshoe
for every integer $n_0\ge 0$.
\end{proof}

Consider a densely, non generically chaotic interval map $f$, and 
let $(J_n)_{n\ge 0}$ be the decreasing sequence of invariant intervals
given by Theorem~\ref{theo:structure-dense-chaos}. By 
Lemma~\ref{lem:nested-Jn-to-0}, the intersection $\bigcap_{n\ge 0}J_n$ is
reduced to a fixed point $z$.
Figure~\ref{fig:dense-chaos-not-generic} is an example of a such a map
when $z$ is an endpoint of all intervals $J_n$.
Figure~\ref{fig:dense-chaos-not-generic2} illustrates what the graph
of $f$ may look like when $z$ is in $\Int{J_n}$ for all $n$: in one case, 
the left and right parts of $J_n$ are exchanged 
under $f$; in the other case, the left and right parts of $J_n$ are 
invariant (left and right parts are with respect to~$z$).
\begin{figure}[htb]
\includegraphics{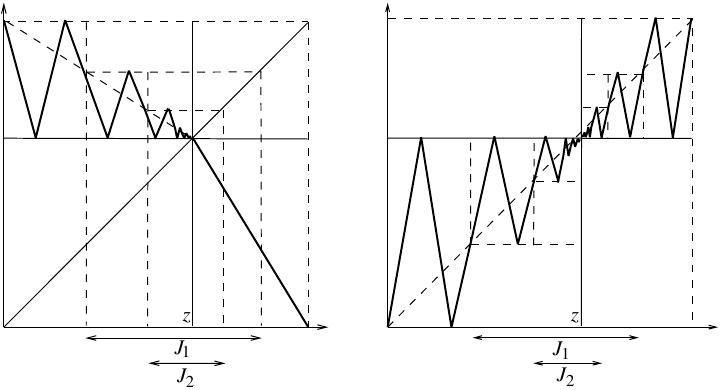}
\caption{The map $f$ on the left is the square root of the map represented in 
Figure~\ref{fig:dense-chaos-not-generic}. The map on the right is $f^2$.
Both are densely, non generically chaotic.}
\label{fig:dense-chaos-not-generic2}
\end{figure}

Using the structure of generically chaotic interval maps and
densely non generically chaotic interval maps, it is possible
to have information on the entropy and the type
of a densely chaotic interval map.
Notice also that a densely chaotic interval map is chaotic in the 
sense of Li-Yorke by Theorem~\ref{theo:summary-chaos-LY}.

\begin{cor}\label{cor:densechaos}
If $f$ is a densely chaotic interval map, then $f^2$ has a
horseshoe. Moreover, $h_{top}(f)\ge \frac{\log 2}{2}$ and $f$ is of type
$n$ for some $n\unlhd 6$ for Sharkovsky's order (i.e., $f$ has a periodic
point of period $6$).
\end{cor}

\begin{proof}
If $f$ is generically chaotic, then $f^2$ has a horseshoe by
Theorem~\ref{theo:generic-chaos-etc}
and Proposition~\ref{prop:transitivity-type}. Otherwise,
$f^2$ has a horseshoe by Theorem~\ref{theo:structure-dense-chaos}.
In both cases, according to Propositions \ref{prop:horseshoe-htop} and 
\ref{prop:turbulent-all-periods}, 
$h_{top}(f)\ge \frac{\log 2}{2}$ and $f^2$ has a periodic point of period
$3$. Thus $f$ has a periodic point of period $3$ or $6$, and hence the type
of $f$ is $\unlhd 6$ by Sharkovsky's Theorem~\ref{theo:Sharkovsky}.
\end{proof}

Equalities are possible in Corollary~\ref{cor:densechaos}:
the map $S$ in Example~\ref{ex:htop-transitive} is transitive, and hence
densely chaotic by Theorem~\ref{theo:generic-chaos-etc}; its entropy
is equal to $\frac{\log 2}{2}$ and, since $S$ is not topologically 
mixing, it is of
type $6$ by Proposition~\ref{prop:transitivity-type}.

Example~\ref{ex:dense-chaos-not-generic} shows that there exists
densely chaotic maps that are not generically chaotic. 
The next result
states that such a map cannot be piecewise monotone nor
$C^1$. The fact that a densely chaotic piecewise monotone map is
generically chaotic is due to Snoha \cite{Sno2}.

\begin{prop}\label{prop:dense-chaos-piecewise-monotone}
Let $f$ be a densely chaotic interval map. If $f$ is piecewise
monotone or $C^1$, then $f$ is generically chaotic.
\end{prop}

\begin{proof}
Suppose that $f$ is not generically chaotic.
According to Theorem~\ref{theo:structure-dense-chaos},
there exists a sequence of non degenerate invariant intervals
$(J_n)_{n\ge 0}$ such that $\lim_{n\to+\infty}|J_n|= 0$ and
$J_{n+1}\subset J_n$ for all $n\ge 0$. 
Moreover, $\bigcap_{n\ge 0}J_n=\{z\}$, where $z$ is
a fixed point, by Lemma~\ref{lem:nested-Jn-to-0}. We
write $J_n=[a_n,b_n]$ for all $n\ge 0$. Thus 
\begin{gather*}
(a_n)_{n\ge 0}\text{ is non decreasing},
\ (b_n)_{n\ge 0}\text{ is non increasing},\\
\lim_{n\to+\infty}a_n=\lim_{n\to+\infty}b_n=z.
\end{gather*}
First we assume that $f$ is piecewise monotone in a neighborhood of $z$.
Then there exists $k\ge 0$ such that both $f|_{[a_k,z]}$ and $f|_{[z,b_k]}$
are monotone. If $z=a_k$ or $z=b_k$, we set $J:=J_k$, and thus the map
$f^2|_J$ is non decreasing because the endpoint $z$ is fixed. 
If $a_k<z<b_k$, the fact that $J_k$ is invariant implies that
either $f|_{[a_k,z]}$
is non decreasing and $f([a_k,z])\subset [a_k,z]$, or
$f|_{[a_k,z]}$ is non increasing and $f([a_k,z])\subset [z,b_k]$;
the symmetric statement holds for $[z,b_k]$. Therefore, there exists
an interval $J$ among $[a_k,z]$ and $[z,b_k]$ such that $f^2(J)\subset J$
and $f^2|_J$ is non decreasing. In all cases, we get a non degenerate 
$f^2$-invariant interval $J$ such that $f^2|_J$ is non decreasing.
We are going to show that
\begin{equation}\label{eq:f2nx}
\forall x\in J,\text{ the sequence }
(f^{2n}(x))_{n\ge 0}\text{ converges}.
\end{equation}
Let $P_2$ be the set of fixed points of $f^2$; this is a closed set.
For all $x\in P_2$, the sequence $(f^{2n}(x))_{n\ge 0}$ 
is stationary, and hence convergent.  Suppose that $J\setminus P_2\neq
\emptyset$ and let $U$ be a connected component of $J\setminus P_2$. 
Either $\inf U\in P_2$ and hence $f^2(\inf U)=\inf(U)$, or
$\inf U=\min J$ and hence $f^2(\inf U)\ge \inf U$ because $J$ is 
$f^2$-invariant. Similarly, $\sup U$ belongs to $P_2\cup\{\max J\}$,
and $f^2(\sup U)\le \sup(U)$. This implies that
$f(\overline{U})\subset \overline{U}$ because $f^2|_J$ is non decreasing.
Moreover, the fact that $U\cap P_2=\emptyset$ implies, by continuity:
\begin{gather*}
\text{either }\forall x\in U,\ f^2(x)>x,\\
\text{or }\forall x\in U,\ f^2(x)<x.
\end{gather*}
Therefore, for every $x\in\overline{U}$, the orbit of
$x$ is included in $\overline{U}$ and $(f^n(x))_{n\ge 0}$ is either 
non decreasing or non increasing, and thus it converges. This proves
\eqref{eq:f2nx}. But this implies  that
$J\times J$ contains no Li-Yorke pair, which is a contradiction.
We conclude that $f$ is not piecewise monotone in a neighborhood of $z$.

Secondly we assume that $f$ is $C^1$. If
$f'(z)\neq 0$, then $f$ is monotone in a neighborhood of $z$
and the previous case leads to a contradiction. Thus $f'(z)=0$.
Then there exists $k\ge 0$ such that
$$
\max_{x\in J_k} |f'(x)|\le\frac 12.
$$

We recall the mean value inequality:\index{mean value inequality}
Let $\vfi\colon I\to \IR$ be a differentiable map (where $I$ is an interval)
and let $M\in\IR$ be such that $|f'(x)|\le M$ for all $x\in I$. Then for
all $x,y\in I$, $|f(y)-f(x)|\le M |y-x|$.

Since $f(J_k)\subset J_k$ and $f(z)=z$, the mean value inequality implies that
$$
\forall x\in J_k,\ \forall n\ge 0,\  
|f^n(x)-z|\le\frac{1}{2^n}|x-z|.
$$
Therefore, for all $x\in J_k$ 
the sequence $(f^n(x))_{n\ge 0}$ converges, and thus 
$J_k\times J_k$ contains no Li-Yorke pair, which is a contradiction.

Conclusion: if $f$ is piecewise monotone or $C^1$, then 
it is generically chaotic.
\end{proof}

\begin{rem}
In Proposition~\ref{prop:dense-chaos-piecewise-monotone}, we get the same
result if we only assume that $f$ is piecewise
monotone or $C^1$ in the neighborhood of every fixed point.
\end{rem}

%***********************************************************************
%***********************************************************************
\section{Distributional chaos}

In \cite{SS}, Schweizer and Smítal defined lower and upper distribution
functions of two points in a dynamical system, and studied them for
interval maps.

\begin{defi}
Let $(X,f)$ be a topological dynamical system and $x,y\in X$.
For all $t\in\IR$ and all $n\in\IN$, set
$$
\xi(f,x,y,n,t):=\#\{i\in\Lbrack 0,n-1\Rbrack \mid d(f^i(x),f^i(y))<t\}.
$$
The lower and upper distribution functions $F_{xy}, F^*_{xy}\colon\IR\to [0,1]$
are defined respectively by: 
\begin{eqnarray*}
&\forall t\in\IR,&F_{xy}(t)=\liminf_{n\to +\infty}\frac 1n\xi(f,x,y,n,t),\\
&&F^*_{xy}(t)=\limsup_{n\to +\infty}\frac 1n\xi(f,x,y,n,t).
\end{eqnarray*}
\end{defi}

The next properties are straightforward from the definition.

\begin{prop}\label{prop:DCobvious}
Let $(X,f)$ be a topological dynamical system and $x,y\in X$.
\begin{itemize}
\item 
$\forall t\le 0$, $F_{xy}(t)=F^*_{xy}(t)=0$ and $\forall t>\diam(X)$,
$F_{xy}(t)=F^*_{xy}(t)=1$. 
\item The maps $F_{xy}$ and $F^*_{xy}$ are
non decreasing and $F_{xy}\le F^*_{xy}$.
\end{itemize}
\end{prop}

The notion of distributional chaos was introduced in \cite{SS}
(although the name ``distributional chaos'' was given later). Three variants of
distributional chaos are now known in the literature (see, e.g., 
\cite{BSS2}). Distributional chaos of type 1
is considered as the original definition of distributional chaos.

\begin{defi}
Let $(X,f)$ be a topological dynamical system. Then $(X,f)$ is called
\emph{distributionally chaotic of type 1, 2, 3}\index{distributional chaos} 
respectively (for short, DC1, DC2, DC3) if the condition
(DC1), (DC2), (DC3) respectively is satisfied:
\begin{eqnarray*}
{\rm (DC1)}&&\exists x,y\in X,\ \exists \delta>0,\ \forall t\in(0,\delta),\ 
F_{xy}(t)=0\text{ and }
\forall t>0,\ F^*_{xy}(t)=1,\\
{\rm (DC2)}&&\exists x,y\in X,\ \exists \delta>0,\ \forall t\in 
(0,\delta),\ F_{xy}(t)< 1\text{ and } \forall t>0,\ F^*_{xy}(t)=1,\\
{\rm (DC3)}&&\exists x,y\in X,\ \exists\, 0<a<b,\ \forall t\in (a,b),\ F_{xy}(t)<
F^*_{xy}(t).\\
\end{eqnarray*}
\end{defi}

Notice that, if condition (DC2) holds for some $x,y$,
then $(x,y)$ is a Li-Yorke pair of modulus $\delta$. Therefore,
DC2 is a refinement of the definition of Li-Yorke pair.

It is clear that (DC1)$\Rightarrow$(DC2)$\Rightarrow$(DC3). 
In \cite{SS}, Schweizer and Smítal showed that, for interval maps,
DC1, DC2 and DC3 coincide and are equivalent to positive entropy
(Corollary~\ref{cor:DChtop} below).

We start with the case of zero entropy interval maps. 
The proofs of Lemma~\ref{lem:DC-xper} and Theorem~\ref{theo:DC-htop0} follow 
the ideas from \cite{SS}.

\begin{lem}\label{lem:DC-xper}
Let $f\colon I\to I$ be an interval map such that $h_{top}(f)=0$. For all
$x\in I$ and all $\eps>0$, there exist a periodic point $z$ 
and a positive integer $K$ such that
\begin{equation}\label{eq:DC-xz}
\forall k\ge K,\ \forall t\ge \eps,\ \frac 1k \xi(f,x,z,k,t)\ge 1-\eps.
\end{equation}
\end{lem}

\begin{proof}
Let $x\in I$ and $\eps>0$.
We split the proof depending on $\omega(x,f)$ being finite or infinite.

First we suppose that $\omega(x,f)$ is finite, that is, there exists a periodic
point $z$ of period $p$ such that $\omega(x,f)=\CO_f(z)$ 
(Lemma~\ref{lem:omega-finite}); we choose $z$ such that 
$\lim_{n\to+\infty}f^{np}(x)=z$.
Thus, by continuity, there exists an integer $N$ such that
$|f^n(x)-f^n(z)|<\eps$ for all $n\ge N$.  Then
$$
\forall n\ge N,\ \forall t\ge\eps,\
\xi(f,x,z,n,t)=\xi(f,x,z,N,t)+(n-N).
$$
This implies that $\lim_{n\to+\infty}\frac 1n\xi(f,x,z,n,t)=1$.
Therefore \eqref{eq:DC-xz} holds for some integer $K$.

Now we suppose that $\omega(x,f)$ is infinite. Let $(L_n)_{n\ge 0}$ be the
sequence of intervals given by Proposition~\ref{prop:htop0-Lki}.
We fix a positive integer $n$ that will be chosen later.
Since $f^{2^n}(L_n)=L_n$, there exists a point $z\in L_n$ such that
$f^{2^n}(z)=z$ (Lemma~\ref{lem:fixed-point}). Let $N$ be an integer such that
$f^i(x)\in f^i(L_n)$ for all $i\ge N$ (such an integer $N$ exists by
Proposition~\ref{prop:htop0-Lki}(vi)). Thus, if $i\ge N$, both points
$f^i(x), f^i(z)$ belong to $f^i(L_n)$, which is an interval of the family
$(f^j(L_n))_{0\le j<2^n}$. Let $k\ge N$; we write
$k=N+k'2^n+r$ with $k'\ge 0$ and $r\in\Lbrack 0,2^n-1\Rbrack$. 
Since $L_n$ is a periodic interval of period $2^n$, we have
\begin{eqnarray*}
\lefteqn{\#\{i\in\Lbrack N,k-1\Rbrack\mid |f^i(x)-f^i(z)|\ge \eps\}
\le\#\{i\in\Lbrack N,k-1\Rbrack\mid |f^i(L_n)|\ge\eps\}}\\
&\le&\#\{i\in\Lbrack 0,r-1\Rbrack \mid |f^{N+i}(L_n)|\ge \eps\}
+k'\#\{j\in\Lbrack 0,2^n-1\Rbrack\mid |f^j(L_n)|\ge \eps\}\\
&\le& (k'+1)\#\{j\in\Lbrack 0,2^n-1\Rbrack\mid |f^j(L_n)|\ge \eps\}.
\end{eqnarray*}
Among the intervals $(f^i(L_n))_{0\le i<2^n}$, at most $\frac{|I|}{\eps}$
have a length greater than or equal to $\eps$ because these intervals are
pairwise disjoint. Thus
$$
\#\{i\in\Lbrack N,k-1\Rbrack\mid |f^i(x)-f^i(z)|\ge \eps\}\le 
\frac{(k'+1)|I|}{\eps},
$$
and hence
\begin{eqnarray*}
\xi(f,x,z,k,\eps)&\ge& \#\{i\in\Lbrack N,k-1\Rbrack \mid |f^i(x)-f^i(z)|< \eps\}\\
&\ge& (k-N)-\frac{(k'+1)|I|}{\eps}=k-N-\frac{|I|}{\eps}-k'\frac{|I|}{\eps}.
\end{eqnarray*}
Thus we have
\begin{eqnarray*}
\frac 1k\xi(f,x,z,k,\eps)
&\ge& 1-\frac{N+\frac{|I|}{\eps}}{k}-\frac{k'}{N+k'2^n+r}\cdot\frac{|I|}{\eps}\\
&\ge& 1-\frac{N+\frac{|I|}{\eps}}{k}-\frac{|I|}{2^n\eps}.
\end{eqnarray*}
We choose $n$ such that $\frac{|I|}{2^n\eps}<\frac{\eps}{2}\Leftrightarrow
2^n>\frac{2|I|}{\eps^2}$, and we choose $K\ge N$ such that
$\frac{N+\frac{|I|}{\eps}}{K}<\frac{\eps}{2}$. Then
$$
\forall k\ge K,\ \forall t\ge \eps,\ 
\frac 1k\xi(f,x,z,k,t)\ge\frac 1k\xi(f,x,z,k,\eps)\ge 1-\eps,
$$
which concludes the proof.
\end{proof}

\begin{lem}\label{lem:DC2perpts}
Let $(X,f)$ be a topological dynamical system and let $z,z'$ be periodic points.
Then $F_{zz'}=F^*_{zz'}$. 
\end{lem}

\begin{proof}
Let $p\in\IN$ be a common multiple of the periods of $z$ and $z'$.
For all integers $k,i\ge 0$, we have $f^{kp+i}(z)=f^i(z)$ and
$f^{kp+i}(z')=f^i(z')$.
Let $n$ be a positive integer and $t\in\IR$. We write $n=kp+r$ with $k\ge 0$
and $r\in\Lbrack 0,p-1\Rbrack$. Then
$$
\xi(f,z,z',n,t)=k\xi(f,z,z',p,t)+\xi(f,z,z',r,t).
$$
This implies that $\lim_{n\to+\infty}\frac1n\xi(f,z,z',n,t)$ exists and is equal
to $\frac 1p\xi(f,z,z',p,t)$. Hence $F_{zz'}(t)=F^*_{zz'}(t)
=\frac 1p\xi(f,z,z',p,t)$ for all $t\in\IR$.
\end{proof}

\begin{theo}\label{theo:DC-htop0}
Let $f\colon I\to I$ be an interval map of zero topological entropy.
Then, for all points $x,y$ in $I$, $\|F_{xy}-F^*_{xy}\|_1=0$, where
$\|\cdot\|_1$ is the $L^1$ norm, that is,
$\disp\|\vfi\|_1:=\int_{-\infty}^{+\infty}|\vfi(t)|\,dt$.
\end{theo}

\begin{proof}
We fix two points $x,y$ in $I$ and a positive number $\eps$.
According to Lemma~\ref{lem:DC-xper},
there exist periodic points $z,z'$ and an integer $K$ such that
\begin{equation}\label{eq:DC0-xi}
\forall k\ge K,\; \forall t\ge\eps,\; \xi(f,x,z,k,t)\ge k(1-\eps)
\text{ and }\xi(f,y,z',k,t)\ge k(1-\eps).
\end{equation}
We set 
$$
\CI_k:=\{i\in\Lbrack 0,k-1\Rbrack\mid |f^i(x)-f^i(z)|<\eps\text{ and }
|f^i(y)-f^i(z')|<\eps\}.
$$
Then \eqref{eq:DC0-xi} implies that
$\#\CI_k\ge k(1-2\eps)$ if $k\ge K$.
For all integers $i$, we have
\begin{equation}\label{eq:DCh0-xyzz'}
|f^i(x)-f^i(y)|\le |f^i(x)-f^i(z)|+|f^i(z)-f^i(z')|+|f^i(y)-f^i(z')|.
\end{equation}
If $i\in\CI_k$ and $|f^i(z)-f^i(z')|<t-2\eps$, then 
$|f^i(x)-f^i(y)|<t$ by \eqref{eq:DCh0-xyzz'}. Thus, for all $k\ge 
K$ and all $t\ge\eps$,
\begin{eqnarray*}
\xi(f,x,y,k,t)&\ge&\#\{i\in\CI_k\mid |f^i(z)-f^i(z')|<t-2\eps\}\\
&\ge& k(1-2\eps)\xi(f,z,z',k,t-2\eps).
\end{eqnarray*}
Dividing by $k$ and taking the limit inf, we get
$$
\forall t\ge\eps,\ F_{xy}(t)\ge (1-2\eps)F_{zz'}(t-2\eps)\ge F_{zz'}(t-2\eps)-2\eps.
$$
As in \eqref{eq:DCh0-xyzz'}, we have:
$$
|f^i(z)-f^i(z)|\le |f^i(x)-f^i(z)|+|f^i(x)-f^i(y)|+|f^i(y)-f^i(z')|.
$$
Similar arguments as above (with $t+2\eps$ and  $\limsup$ instead of $t$
and $\liminf$) give
$$
\forall t\ge\eps,\ F^*_{zz'}(t+2\eps)\ge F^*_{xy}(t)-2\eps.
$$
According to Lemma~\ref{lem:DC2perpts}, $F_{zz'}=F^*_{zz'}$. Thus
$$
\forall t\ge\eps,\ F_{zz'}(t-2\eps)-2\eps\le F_{xy}(t)
\le F^*_{xy}(t)\le F_{zz'}(t+2\eps)+2\eps.
$$
By Proposition~\ref{prop:DCobvious},
$\|F^*_{xy}-F_{xy}\|_1=\int^{|I|}_0 \left(F^*_{xy}(t)-F_{xy}(t)\right)dt$
and $F^*_{xy}(t)-F_{xy}(t)\le 1$ for all $t\in [0,\eps]$.
This implies that
\begin{equation}\label{eq:norm1-Fxy}
\|F_{xy}^*-F_{xy}\|_1\le \eps+\int_{\eps}^{|I|} \left(F_{zz'}(t+2\eps)
-F_{zz'}(t-2\eps)\right)dt +4\eps|I|.
\end{equation}
We set 
$$
A:=\int_{\eps}^{|I|} \left(F_{zz'}(t+2\eps)-F_{zz'}(t-2\eps)\right)\, dt.
$$ 
By Proposition~\ref{prop:DCobvious},
$F_{zz'}(t)=0$ if $t\le 0$
and $0\le F_{zz'}(t)\le 1$ for all $t$.
Thus 
\begin{eqnarray*}
A&=&
\int_{3\eps}^{|I|+2\eps}F_{zz'}(u)\,du-
\int_{-\eps}^{|I|-2\eps}F_{zz'}(u)\,du\\
&\le &
\int_{0}^{|I|}F_{zz'}(u)\,du+2\eps-\int_{0}^{|I|}F_{zz'}(u)\,du+2\eps=4\eps.
\end{eqnarray*}
Including this result in \eqref{eq:norm1-Fxy}, we have
$\|F_{xy}^*-F_{xy}\|_1\le \eps(5+4|I|)$.
Taking the limit when $\eps\to 0$, we get
$\|F_{xy}^*-F_{xy}\|_1=0$.
\end{proof}

The next theorem deals with positive entropy interval maps. The proof is
different from the one in \cite{SS};  it uses the semi-conjugacy
of a subsystem with a full shift, and the arguments are similar to the ones
in the proof of Theorem~\ref{theo:htop-positive-chaos-LY}.

\begin{theo}\label{theo:DC+htop}
Let $f\colon I\to I$ be an interval map of positive topological entropy.
Then there exist a Cantor set $K\subset I$ and a positive number
$\delta$ such that, for all distinct points $x,y$ in $K$,
$$
\forall t\in [0,\delta),\ F_{xy}(t)=0\quad\text{and}\quad
\forall t>0,\ F^*_{xy}(t)=1.
$$
\end{theo}

\begin{proof}
By Theorem~\ref{theo:htop-power-of-2}, 
there exists an integer $r$ such that
$f^r$ has a strict horseshoe  $(J_0,J_1)$. 
Let $X, E$, $\vfi\colon X\to \Sigma$  and 
$(J_{\alpha_0\ldots\alpha_{n-1}})_{n\ge 1, (\alpha_0,\ldots,\alpha_{n-1})\in\{0,1\}^n}$
be given by Proposition~\ref{prop:strictly-turbulent-shift}
for the map $g:=f^r$. 

We set $\gamma^n:=\underbrace{\gamma\cdots \gamma}_{n\ \rm times}$
if $\gamma\in\{0,1\}$. We first prove the
following fact:

\begin{equation}\label{eq:diamJOn1}
\lim_{n\to +\infty}|J_{0^n1}|=0.
\end{equation}

By Proposition~\ref{prop:strictly-turbulent-shift},
$\bigcap_{n=1}^{+\infty} J_{0^n}$ is a decreasing intersection of nonempty
compact intervals (this intersection may be a non degenerate interval because
$(0000\ldots)$ may be in $\vfi(E)$). 
Thus $(\min J_{0^n})_{n\ge 0}$ is a non decreasing
sequence that converges to some point $x_1$, and 
$(\max J_{0^n})_{n\ge 0}$ is a non increasing
sequence that converges to some point $x_2$. One has
$x_1,x_2\in \bigcap_{n=1}^{+\infty} J_{0^n}$ and $x_1\le x_2$.
Let $\eps>0$ and let $N\ge 1$ be such that
\begin{equation}\label{eq:|JOn|}
\forall n\ge N,\ |x_1-\min J_{0^n}|<\eps\quad\text{and}\quad 
|x_2-\max J_{0^n}|<\eps.
\end{equation}
Let $n\ge N$. The intervals $J_{0^n1}$ and $J_{0^{n+1}}$ are disjoint
and included in $J_{0^n}$, and $x_1,x_2$ belong to $J_{0^{n+1}}$. This
implies that
\begin{itemize}
\item either $J_{0^n1}<J_{0^{n+1}}$ and
$J_{0^n1}\subset [\min J_{0^n},x_1]$ (see Figure~\ref{fig:J0n1} on the left),
\item or $J_{0^n1}>J_{0^{n+1}}$ and
$J_{0^n1}\subset [x_2,\max J_{0^n}]$ (see Figure~\ref{fig:J0n1} on the right).
\end{itemize}
In both cases, $|J_{0^n1}|<\eps$ according to \eqref{eq:|JOn|}. This
proves \eqref{eq:diamJOn1}.
\begin{figure}[htb]
\centerline{\includegraphics{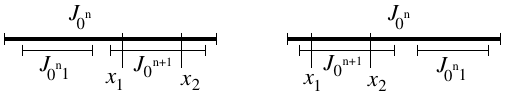}}
\caption{The two cases $J_{0^n1}<J_{0^{n+1}}$ and
$J_{0^n1}>J_{0^{n+1}}$.}
\label{fig:J0n1}
\end{figure}

Let $(n_k)_{k\ge 0}$ be a sequence of positive integers increasing fast 
enough to have
\begin{equation}\label{eq:nkveryfast}
\lim_{k\to+\infty}\frac{1}{n_k}\sum_{i=0}^{k-1}n_i=0.
\end{equation}
%(e.g., $n_k=2^{k^2}$ can be shown to be suitable). 
In particular, \eqref{eq:nkveryfast} implies that
\begin{gather}
\lim_{k\to+\infty}\frac{k}{n_k}=0\label{eq:k/nk}\\
\text{and}\quad\lim_{k\to+\infty}\frac{n_k}{\sum_{i=0}^{k}n_i}=1.\label{eq:nkfast}
\end{gather}

For all $i\ge 1$ and all $\bar\alpha=(\alpha_n)_{n\ge 0}\in\Sigma$, 
we set $W_i:=0^{i-1}1$ 
and 
$$
B_i(\bar\alpha):=W_{n_{k_i}}(\alpha_0)^{n_{k_i+1}}(\alpha_1)^{n_{k_i+2}}\ldots
(\alpha_{i-1})^{n_{k_i+i}},
$$
where $(k_i)_{i\ge 1}$ is the sequence defined by $k_1=0$ and 
$k_{i+1}=k_i+i+1$ (in this way, $B_i(\bar\alpha)$ ends with
$(\alpha_{i-1})^{n_{k_i+i}}$ and $B_{i+1}(\bar\alpha)$ begins with
$W_{n_{k_i+i+1}}$).

We define $\psi\colon \Sigma\to \Sigma$ by
\begin{eqnarray*}
\psi(\bar\alpha)&:=&
(B_1(\bar\alpha)B_2(\bar\alpha)B_3(\bar\alpha)\ldots)\\
&=&(W_{n_0}(\alpha_0)^{n_1} W_{n_2}(\alpha_0)^{n_3}(\alpha_1)^{n_4}W_{n_5}
(\alpha_0)^{n_6}(\alpha_1)^{n_7}(\alpha_2)^{n_8}\ldots)
\end{eqnarray*}
The map $\psi$ is clearly continuous.
For every $\bar\alpha\in \Sigma$, we chose a point $x_{\bar\alpha}$ in 
$\vfi^{-1}\circ \psi(\bar\alpha)$ and we set
$S:=\{x_{\bar\alpha}\in X \mid \bar \alpha \in \Sigma \}$.
According to 
Proposition~\ref{prop:strictly-turbulent-shift}, the set 
$\vfi^{-1}\circ \psi(\bar\alpha)$ contains two points if $\psi(\bar\alpha)
\in \vfi(E)$ and one point if $\psi(\bar\alpha)\notin \vfi(E)$. Thus
there exists a countable set $F\subset X$ such that
$S=\vfi^{-1}\circ \psi(\Sigma)\setminus F$.

\medskip
We fix $\bar\alpha=(\alpha_n)_{n\ge 0}$ and 
$\bar\beta=(\beta_n)_{n\ge 0}$ two distinct elements of $\Sigma$.

Let $t>0$. Since $f$ is uniformly continuous, there exists $\eps>0$
such that
$$
\forall x,y\in I,\ |x-y|<\eps\Rightarrow \forall i\in\Lbrack 0,r-1\Rbrack,\
|f^i(x)-f^i(y)|<t.
$$
According to \eqref{eq:diamJOn1}, there exists a positive integer $N$ such
that $|J_{0^{n-1}1}|<\eps$ for all $n\ge N$.
If $j\ge 0$ is such that both $\sigma^j(\psi(\bar\alpha))$
and $\sigma^j(\psi(\bar\beta))$ begin with $W_n$ with $n\ge N$, then
$|g^j(x_{\bar\alpha})-g^j(x_{\bar\beta})|<\eps$ because 
both points $g^j(x_{\bar\alpha}),g^j(x_{\bar\beta})$ belong to
$J_{0^{n-1}1}$. We set
$$
m_i:=\sum_{k=0}^{k_i}n_k.
$$
The integer $m_i$ is the length of the
sequence $B_1(\bar\alpha)\ldots B_{i-1}(\bar\alpha)W_{n_{k_i}}$.
Then, by the definition of $\psi$, for all $i$ such that $n_{k_i}>N$,
$$
\forall j\in\Lbrack N, n_{k_i}\Rbrack,\ 
|g^{m_i-j}(x_{\bar\alpha})-g^{m_i-j}(x_{\bar\beta})|<\eps.
$$
This implies that 
$\xi(g,x_{\bar\alpha},x_{\bar\beta},m_i,\eps)\ge n_{k_i}-N$
and $\xi(f,x_{\bar\alpha},x_{\bar\beta},r.m_i,t)\ge r(n_{k_i}-N)$.
According to  \eqref{eq:nkfast}, 
$\lim_{i\to+\infty}\frac{n_{k_i}-N}{m_i}=1$, and hence
$F^*_{x_{\bar\alpha}x_{\bar\beta}}(t)=1$.

\medskip
Since $\bar\alpha\ne\bar\beta$, there is an integer $q$ such that
$\alpha_q\neq\beta_q$. Let $D>0$ be the distance between  $J_0$ and $J_1$,
and let $\delta>0$  be such that
$$
\forall x,y\in I,\ |x-y|<\delta\Rightarrow \forall i\in\Lbrack 0,r-1\Rbrack,\
|f^i(x)-f^i(y)|<D.
$$
We set $p_i:=m_i+(k_i+1)+\ldots+(k_i+q)$. If 
$k_{i+1}-k_i>q+1$, then $p_i$ is the length of the sequence 
$$
B_1(\bar\alpha)\ldots B_{i-1}(\bar\alpha)W_{k_i}(\alpha_0)^{n_{k_i+1}}\ldots
(\alpha_{q-1})^{n_{k_i+q}},
$$
and 
$\sigma^{p_i}(\psi(\bar\alpha))$, $\sigma^{p_i}(\psi(\bar\beta))$
begin respectively with $(\alpha_q)^{n_{k_i+q+1}}$ and 
$(\beta_q)^{n_{k_i+q+1}}$. Then
$$
\forall j\in\Lbrack 0, n_{k_i+q+1}-1\Rbrack,\ |g^{p_i+j}(x_{\bar\alpha})-g^{p_i+j}(x_{\bar\beta})|\ge D
$$
because either $g^{p_i+j}(x_{\bar\alpha})\in J_0$ and $g^{p_i+j}(x_{\bar\beta})\in J_1$, or the converse.
This implies that
$\xi(g,x_{\bar\alpha},x_{\bar\beta},p_i+n_{k_i+q+1},D)\le p_i$
and
$\xi(f,x_{\bar\alpha},x_{\bar\beta},r(p_i+n_{k_i+q+1}),\delta)\le 
r.p_i$. One can compute that $p_i=m_i+q k_i+\frac{q(q+1)}2$. Thus
$$
\frac{p_i}{p_i+n_{k_i+q+1}}\le\frac{p_i}{n_{k_i+1}}\le
\frac{m_i+q k_i+q(q+1)/2}{n_{k_i+1}}.
$$
This last quantity tends to $0$ according to \eqref{eq:nkveryfast}
and \eqref{eq:k/nk}, and hence 
$$
\lim_{i\to+\infty}\frac{p_i}{p_i+n_{k_i+q+1}}=0.
$$
We deduce that $F_{x_{\bar\alpha}x_{\bar\beta}}(t)=0$
for all $t\in [0,\delta)$. Finally, by
Theorem~\ref{theo:perfect-set}, there exists a Cantor set $K\subset S$
because $S=\vfi^{-1}\circ \psi(\Sigma)\setminus F$ is a Borel set.
\end{proof}

\begin{cor}\label{cor:DChtop}
Let $f$ be an interval map. The following properties are
equivalent:
\begin{itemize}
\item $f$ is DC1,
\item $f$ is DC2,
\item $f$ is DC3,
\item $h_{top}(f)>0$.
\end{itemize}
\end{cor}

\begin{proof}
It is clear than DC1$\Rightarrow$DC2$\Rightarrow$DC3.
Theorem~\ref{theo:DC-htop0} implies that, if $h_{top}(f)=0$, then
$f$ is not DC3. By refutation, we get DC3 $\Rightarrow$ $h_{top}(f)>0$.
Finally, if $h_{top}(f)>0$, then $f$ is DC1 by Theorem~\ref{theo:DC+htop}.
\end{proof}

%***********************
\subsection*{Remarks on graph maps and general dynamical systems}

The results of Schweizer and Smítal on distributional chaos was generalized 
to graph maps by steps, first to circle maps \cite{Mal2,Mal3}, then 
to tree maps \cite{Cano3, CH2} and finally to general graph maps.
The next result is due to Hric and Málek \cite{HM}.

\begin{theo}
Let $f\colon G\to G$ be a graph map. The following properties are
equivalent:
\begin{itemize}
\item $f$ is DC2,
\item $h_{top}(f)>0$.
\end{itemize}
\end{theo}

For general dynamical systems, Downarowicz showed the following 
implication~\cite{Dow2}.

\begin{theo}
Let $(X,f)$ be a topological dynamical system. If $f$ has positive topological
entropy, then $f$ is DC2.
\end{theo}

The equivalence of the three types of distributional chaos is
not true for general dynamical systems. On the one hand, 
Piku{\l}a showed that positive topological entropy does not imply DC1
\cite{Pik}. On the other hand, Balibrea, Smítal and
Štefánková exhibited a dynamical system which is DC3 and
distal (i.e., for all $x\ne y$, $\liminf_{n\to+\infty}d(f^n(x),f^n(y))>0$)
\cite{BSS}, and thus DC3 does not even imply the existence of Li-Yorke pairs
(recall that, on the contrary, positive entropy implies Li-Yorke chaos
according to
Theorem~\ref{theo:general-system-htop-positive-chaos-LY}). Therefore,
DC1, DC2 and DC3 are distinct notions in general.
Moreover, DC3 is not invariant by conjugacy \cite{BSS}, whereas DC1 and DC2 are.

%************************************************************************
%Chaotic subsystems
\chapter{Chaotic subsystems}\label{chap7}

\section{Subsystems chaotic in the sense of Devaney}

In \cite{Dev}, Devaney mainly studied maps on the interval or on the
real line. Observing some chaotic behavior, he introduced 
a definition of \emph{chaos}. 
For Devaney, chaos is seen as a combination of unpredictability
(sensitivity) and regular behaviors (periodic points), transitivity
ensuring that the system is undecomposable.

\begin{defi}[chaos in the sense of Devaney]
\index{chaos!chaos in the sense of Devaney} 
A topological  dynamical system $(X,f)$ is
\emph{chaotic in the sense of Devaney} if
\begin{itemize}
\item $f$ is transitive,
\item the set of periodic points is dense in $X$,
\item $f$ is sensitive to initial conditions.
\end{itemize}
\end{defi}

For interval maps, 
transitivity is enough to imply the other two conditions, as it was
pointed out by Silverman \cite{Sil} and 
Vellekoop and Berglund \cite{VB}.  It is a
straightforward corollary of  Propositions
\ref{prop:transitivity-periodic-points} and
\ref{prop:transitivity-sensitivity}.

\begin{prop}\label{prop:Devaney=transitive}
An interval map is chaotic in the sense of Devaney if and only if it 
is transitive.
\end{prop}

Devaney was actually interested in systems having a chaotic
subsystem. Shihai Li showed that, for interval maps, this is equivalent to
positive entropy \cite{Li}.

\begin{theo}\label{theo:htop-Devaney}
Let $f$ be an interval map. The following are equivalent:
\begin{enumerate}
\item $h_{top}(f)>0$,
\item there exists an invariant set $X$ such that
$(X,f|_X)$ is chaotic in the sense of Devaney,
\item there exists an infinite invariant set $X$ such that
$(X,f|_X)$ is transitive and $X$ contains a periodic point.
\end{enumerate}
\end{theo}

\begin{proof}
First we suppose that $h_{top}(f)>0$.  By Theorem~\ref{theo:htop-power-of-2},
there exist two closed intervals $J_0,J_1$ and an integer
$n\ge 1$ such that $(J_0,J_1)$ is a strict  horseshoe for $f^n$. 
Let $X, E$  and $\vfi\colon X\to \Sigma$ 
be given by Proposition~\ref{prop:strictly-turbulent-shift} 
for the map $g:=f^n$.
Then $(X,g|_X)$ is transitive and $X$ is a $g$-invariant Cantor set.
We are going to show that $(X,g|_X)$ is sensitive to initial conditions
and has a dense set of periodic points.

We define the following distance on $\Sigma$: for all
$\bar\alpha=(\alpha_n)_{n\ge 0},\bar\beta=(\beta_n)_{n\ge 0}$ in $\Sigma$,
$$
d(\bar\alpha,\bar\beta):=\sum_{n=0}^{+\infty}
\frac{|\beta_n-\alpha_n|}{2^n}
$$
(see also Definition~\ref{defi:producttopology}
in the Appendix).
Since $X$ is compact, the map $\vfi$ is uniformly continuous and there
exists $\delta>0$ such that
\begin{equation}\label{eq:unifcont-sensitive}
\forall x,y\in X,\ |x-y|<\delta\Rightarrow d(\vfi(x),\vfi(y))<1.
\end{equation}

Let $x_0\in X$ and $\eps>0$. The Cantor set $X$ has no isolated point and
$\vfi$ is at most two-to-one, thus there exists $y\in X$ such that
$|x_0-y|<\eps$ and $\vfi(y)\ne\vfi(x_0)$. Let
$\bar\alpha=(\alpha_n)_{n\ge 0}:=\vfi(x_0)$ and 
$\bar\beta=(\beta_n)_{n\ge 0}:=\vfi(y)$, and
let $k\ge 0$ be an integer such that $\alpha_k\ne\beta_k$. Then
$d(\sigma^k(\bar\alpha),\sigma^k(\bar\beta))\ge 1$. Since $\vfi$ is
a semi-conjugacy, $\sigma^k(\bar\alpha)=\vfi(g^k(x_0))$ and
$\sigma^k(\bar\beta)=\vfi(g^k(y))$. According to
\eqref{eq:unifcont-sensitive}, this implies that
$|g^k(x_0)-g^k(y)|\ge\delta$. This proves that 
 $(X,g|_X)$ is $\delta$-sensitive.

\medskip
Let $x_0\in X\setminus E$ and $\eps>0$. Let $(\alpha_n)_{n\ge 0}
:= \vfi(x_0)$. Since $x_0\notin E$, there exists an integer $k$ such that
\begin{equation}\label{eq:diamJ0k-Devaney}
\diam\{x\in X\mid \vfi(x)\text{ begins with }\alpha_0\ldots\alpha_{k-1}\}
<\eps.
\end{equation}
Let $\bar\beta=(\beta_n)_{n\ge 0}\in\Sigma$ be the periodic point
such that $\beta_0\ldots\beta_{k-1}=\alpha_0\ldots\alpha_{k-1}$
and $\sigma^k(\bar\beta)=\bar\beta$ (i.e., $\bar\beta$ is the infinite
repetition of $\alpha_0\ldots\alpha_{k-1}$). 
Since $\vfi$ is onto and at most two-to-one, 
there exist two (possibly equal) points $y_1,y_2$ in $X$ such that
$\vfi^{-1}(\bar\beta)=\{y_1,y_2\}$ (one has $y_1=y_2$ if
$\bar\beta\notin\vfi(E)$). Then, for $i\in\{1,2\}$,
$\vfi(g^k(y_i))=\sigma^k(\vfi(y_i))=\sigma^k(\bar\beta)=\bar\beta$, and
$g^k(y_i)\in\vfi^{-1}(\bar\beta)=\{y_1,y_2\}$. This implies that
either $g^{2k}(y_1)=y_1$ or $g^{2k}(y_2)=y_2$.
Thus there is a periodic point among $y_1,y_2$; we call it $y$.
By \eqref{eq:diamJ0k-Devaney}, $|x_0-y|<\eps$ because $\vfi(y)=\bar\beta$. 
Thus the set of periodic points is dense in $X\setminus E$.
This implies that the set of periodic points is dense in $X$ because
$X$ is an uncountable set with no isolated point and $E$ is
countable.

We set $X':=X\cup f(X)\cup\cdots f^{n-1}(X)$. Then 
$X'$ is closed, $f$-invariant, and $(X',f|_{X'})$ 
is chaotic in the sense of Devaney. Thus (i)$\Rightarrow$(ii).

\medskip
The implication (ii)$\Rightarrow$(iii) is trivial (notice that
a sensitive system is necessarily infinite).

\medskip
Now we suppose that there exists an infinite $f$-invariant set 
$X$ such that $f|_X$ is transitive and $X$ contains
a periodic point. By Proposition~\ref{prop:transitive-dense-orbit},
$X$ has no isolated point and 
there exists $x\in X$ such that $\omega(x,f)=X$. If $h_{top}(f)=0$, then, by
Proposition~\ref{prop:omega-no-periodic-point}, 
the set $\omega(x,f)$
contains no periodic point, which contradicts the fact that
$X$ contains a periodic point. We conclude that $h_{top}(f)>0$,
that is, (iii)$\Rightarrow$(i).
\end{proof}

\subsection*{Remarks on graph maps and general dynamical systems}
The results of this section are still valid for graph maps.
The generalization of Proposition~\ref{prop:Devaney=transitive}
is given by Theorem~\ref{theo:transitivegraphmap-rotation},
Corollary~\ref{cor:transitive-sensitiveG}
and the fact that a rotation is not sensitive to initial conditions.
The proof of Theorem~\ref{theo:htop-Devaney} for graph maps is the same
since Propositions \ref{prop:strictly-turbulent-shift} and
\ref{prop:omega-no-periodic-point} remain valid for graph maps
(see ``Remarks on graph maps'' at the end of Sections
\ref {sec:5-htop>0} and \ref{sec:equivLY}).

\medskip
It was shown simultaneously in several papers that there is
a redundancy in the definition of chaos in the sense of Devaney,
sensitivity being implied by the other two conditions
\cite{BBCDS, Sil, GW}. 

\begin{theo}
Let $(X,f)$ be a topological dynamical system where $X$ is an infinite 
compact space.
Suppose that $f$ is transitive and that the set of periodic points is dense.
Then $f$ is sensitive to initial conditions, and thus $f$ is chaotic
in the sense of Devaney.
\end{theo}

%*******************************
\section{Topologically mixing subsystems}

Xiong showed that an interval map $f$ has an infinite mixing subsystem in 
which the set of periodic points is dense if and only if $f$ has a periodic 
point of odd period greater than $1$ \cite{Xio2}.
The ``if'' part, which is a variant of 
Proposition~\ref{prop:strictly-turbulent-shift},
 relies on the fact that $f$ has a subsystem
``almost'' conjugate to the subshift associated to the graph of a periodic
orbit of odd period $p>1$, and this graph is known when $p$ is minimal.
The ``only if'' part can be strengthened: the existence of an infinite 
subsystem on which $f^2$ is transitive is sufficient to imply that $f$ 
has a periodic point of odd period greater than $1$.

\begin{rem}
According to Xiong's terminology \cite{Xio2}, an interval map 
$f$ is called
\emph{strongly chaotic}\index{chaos!strong chaos} if there 
exists an invariant subset $X$ such that $(X,f|_X)$ is
topologically mixing, the set of periodic points is dense in $X$ and
the periods of periodic points in $X$ form an infinite set.
\end{rem}

Much can be said about subshifts associated to a directed graph, which belong
to the class of subshifts of finite type (see, e.g., \cite{Kit}).
We just give the definition; we shall not explicitly use the properties
of such systems.

\begin{defi}
Let $G$ be a directed graph and $V$ its set of vertices (recall that 
directed graphs are defined in Section~\ref{sec:1-directedgraphs}). 
Let $\Gamma(G)$ denote the set of infinite paths in $G$, that is,
$$
\Gamma(G):=\{(\alpha_n)_{n\ge 0}\in V^{\IZ^+}\mid \forall n\ge 0, \alpha_n\to\alpha_{n+1}
\text{ is an arrow in }G\}.
$$
The set $V^{\IZ^+}$ is endowed with the product topology (where $V$ has the
discrete topology) and $\Gamma(G)\subset V^{\IZ^+}$ is endowed with the 
induced topology.
The shift map $\sigma\colon \Gamma(G)\to \Gamma(G)$ is defined by
$\sigma((\alpha_n)_{n\ge 0}):=(\alpha_{n+1})_{n\ge 0}$.
Then $(\Gamma(G),\sigma)$ is a topological dynamical system, called
the \emph{subshift (or topological Markov shift) associated to the graph $G$}. 
\index{Markov shift}\index{topological Markov shift}\index{subshift associated to a graph}\index{subshift of finite type}
\end{defi}

\begin{theo}\label{theo:oddperiod-mixing-subset}
Let $f\colon I\to I$ be an interval map. Assume that $f$ has 
a periodic point of odd period greater than $1$. Then there exists
an uncountable invariant set $X$ such that 
$f|_X\colon X\to X$ is topologically mixing and the set of periodic points 
is dense in $X$. Moreover, the set of periods of periodic points in $X$
is infinite.
\end{theo}

\begin{proof}
Let $p$ be the smallest odd period greater than $1$, and let $G_p$ be the 
graph of a periodic orbit of period $p$ given by 
Lemma~\ref{lem:graph-n-minimal}.
According to Proposition~\ref{prop:path-matrix},  for every $n$-tuple of
vertices $(\alpha_0,\ldots, \alpha_{n-1})$, if $\alpha_0\to \alpha_1\to\cdots
\to \alpha_{n-1}$  is a path in $G_p$, then $(\alpha_0,\alpha_1,
\ldots,\alpha_{n-1})$  is a chain of intervals for $f$. 
For every $n\ge 0$, let $\Gamma_n$ denote the set of paths of lengths $n$ in $G_p$.
We apply Lemma~\ref{lem:chain-of-intervals}(iii) to the family of chains of 
intervals $(\alpha_0,\alpha_1)_{(\alpha_0,\alpha_1)\in \Gamma_1}$, and we obtain 
non degenerate closed subintervals with disjoint interiors 
$(J_{\alpha_0\alpha_1})_{(\alpha_0,\alpha_1)\in \Gamma_1}$ such that
$J_{\alpha_0\alpha_1}\subset\alpha_0$ and $f(J_{\alpha_0\alpha_1})=\alpha_1$.
Using Lemma~\ref{lem:chain-of-intervals} inductively, we define 
non degenerate closed subintervals
$(J_{\alpha_0\ldots \alpha_n})_{(\alpha_0,\ldots,\alpha_n)\in \Gamma_n}$ 
such that, for all $(\alpha_0,\ldots,\alpha_n), (\beta_0,\ldots,\beta_n)$
in $\Gamma_n$:
\begin{gather}
J_{\alpha_0\ldots \alpha_n}\subset J_{\alpha_0\ldots \alpha_{n-1}},\label{eq:Jnn-1}\\
f(J_{\alpha_0\ldots \alpha_n})=J_{\alpha_1\ldots \alpha_n},\label{eq:fJalphan}\\
(\alpha_0,\ldots,\alpha_n)\neq (\beta_0,\ldots,\beta_n)\Longrightarrow
\Int{J_{\alpha_0\ldots\alpha_n}}\cap \Int{J_{\beta_0\ldots,\beta_n}}=\emptyset.
\label{eq:intJalphabeta}
\end{gather}
For every $\bar\alpha=(\alpha_n)_{n\ge 0}\in \Gamma(G_p)$, we set
$$
J_{\bar\alpha}:=\bigcap_{n=0}^{+\infty} J_{\alpha_0\ldots\alpha_n}.
$$
This is a decreasing intersection of nonempty compact intervals, so
$J_{\bar\alpha}$ is a nonempty compact interval too. 
Moreover, \eqref{eq:intJalphabeta} implies that
$$
\forall \bar\alpha,\bar\beta\in\Gamma(G_p),\ \bar\alpha\neq\bar\beta
\Longrightarrow \Int{J_{\bar\alpha}}\cap\Int{J_{\bar\beta}}=\emptyset.
$$
Now we are going to show that
\begin{equation}\label{eq:fJalpha}
\forall \bar\alpha\in\Gamma(G_p),\quad f(J_{\bar\alpha})=\bigcap_{n=0}^{+\infty}f(J_{\alpha_0\ldots \alpha_n}).
\end{equation}
The inclusion $\subset$ is obvious according to the definition of
$J_{\bar\alpha}$. Let $y$ be a point in $\bigcap_{n\ge 0}f(J_{\alpha_0\ldots
\alpha_n})$ and, for every $n\ge 0$, let $x_n\in J_{\alpha_0\ldots \alpha_n}$ 
be such that $f(x_n)=y$. By compactness, there exist an increasing sequence 
of positive integers $(n_i)_{i\ge 0}$ and a point $x$ such that 
$\lim_{i\to+\infty} x_{n_i}=x$. 
Moreover, $x\in \bigcap_{n=0}^{+\infty}
J_{\alpha_0\ldots \alpha_n}$ because this is a decreasing intersection of
compact sets; and $f(x)=y$ by continuity of $f$. This proves that
$\bigcap_{n=0}^{\infty}f(J_{\alpha_0\ldots \alpha_n})\subset
f(J_{\bar\alpha})$, and thus 
\eqref{eq:fJalpha} holds. Then \eqref{eq:fJalphan} and 
\eqref{eq:fJalpha} imply that
\begin{equation}\label{eq:fsigma}
\forall \bar\alpha\in\Gamma(G_p),\quad 
f(J_{\bar\alpha})=J_{\sigma(\bar\alpha)}.
\end{equation}
Let
$$
E:=\{\bar\alpha\in\Gamma(G_p)\mid J_{\bar\alpha}\text{ is not reduced to 
one point}\}.
$$
By definition, we have 
\begin{equation}\label{eq:diam0}
\forall (\alpha_n)_{n\ge 0}\in\Gamma(G_p)\setminus E,
\ \lim_{n\to+\infty}|J_{\alpha_0\ldots\alpha_n}|=0.
\end{equation}
The set $E$ is countable because the intervals 
$(J_{\bar\alpha})_{\bar\alpha\in E}$ are non degenerate and
have disjoint interiors (see Lemma~\ref{lem:open-finiteCC}). 
Moreover, \eqref{eq:fsigma} implies that
$\sigma(\Gamma(G_p)\setminus E)\subset \Gamma(G_p)\setminus E$.
We define the map 
$$
\vfi\colon \begin{array}[t]{ccl}\Gamma(G_p)\setminus E&\longrightarrow& I\\
\bar\alpha&\longmapsto& x \text{ such that } J_{\bar\alpha}=\{x\}.
\end{array}
$$
It is easy to show that this map is continuous using \eqref{eq:Jnn-1}
and \eqref{eq:diam0},
and $\vfi\circ\sigma=f\circ\vfi$
by \eqref{eq:fsigma}. Moreover, $\vfi$ is at most two-to-one. Indeed,
if $(\alpha_n)_{n\ge 0},(\beta_n)_{n\ge 0},(\gamma_n)_{n\ge 0}$ are 
three distinct elements of $\Gamma(G_p)\setminus E$, there exists 
$n\ge 0$ such that $(\alpha_0,\ldots,\alpha_n)$, $(\beta_0,\ldots,\beta_n)$,
$(\gamma_0,\ldots,\gamma_n)$ are not all three equal to the same $(n+1)$-tuple, 
and thus \eqref{eq:intJalphabeta} implies that
$J_{\alpha_0\ldots\alpha_n}\cap J_{\beta_0\ldots\beta_n}\cap
J_{\gamma_0\ldots\gamma_n}$ is empty.

\medskip
We set 
$$
X_0:=\vfi(\Gamma(G_p)\setminus E)\quad\text{and}\quad X:=\overline{X_0}.
$$
These sets satisfy $f(X_0)\subset X_0$ and $f(X)\subset X$ because
$\sigma(\Gamma(G_p)\setminus E)\subset\Gamma(G_p)\setminus E$.
Moreover, $X_0$ and 
$X$ are uncountable because $\vfi$ is at most two-to-one and 
$\Gamma(G_p)\setminus E$ is uncountable.
Looking at the graph $G_p$ described in Lemma~\ref{lem:graph-n-minimal},
we see that there exists $k\ge 0$ ($k:=2p-3$ is suitable) such that,
for all vertices $\alpha,\beta$ of $G_p$,
\begin{equation}\label{eq:pathalphatobeta}
\text{there exists a path }
(\omega^{\alpha\beta}_0,\ldots,\omega^{\alpha\beta}_k)\in \Gamma_k
\text{ such that }\omega^{\alpha\beta}_0=\alpha \text{ and }
\omega^{\alpha\beta}_k=\beta.
\end{equation}
We are going to show that
\begin{equation}\label{eq:Jalpha-Y}
\forall i\ge 0,\ \forall (\alpha_0,\ldots,\alpha_i)\in \Gamma_i,
\ f^{i+k}(J_{\alpha_0\ldots\alpha_i}\cap X)=X.
\end{equation}
We fix $(\alpha_0,\ldots,\alpha_i)$ in $\Gamma_i$. Let $\eps>0$.
Let $y\in X_0$ and 
$(\beta_n)_{n\ge 0}\in \Gamma(G_p)\setminus E$ be such that
$\vfi((\beta_n)_{n\ge 0})=y$. By \eqref{eq:diam0}, there exists $q\ge 0$
such that $|J_{\beta_0\ldots\beta_q}|<\eps$. 
We define the map $\psi\colon\Gamma(G_p)\to\Gamma(G_p)$ by
$$
\psi((\gamma_n)_{n\ge 0}):=
(\alpha_0\ldots\alpha_i\omega^{\alpha_i\beta_0}_1\ldots \omega^{\alpha_i\beta_0}_{k-1}\beta_0\ldots\beta_q\omega^{\beta_q\gamma_0}_1\ldots\omega^{\beta_q\gamma_0}_{k-1}
\gamma_0\gamma_1\ldots\gamma_n\ldots),
$$
where $\omega_0^{\alpha_i\beta_0}\ldots\omega_k^{\alpha_i\beta_0}$
(resp. $\omega_0^{\beta_q\gamma_0}\ldots\omega_k^{\beta_q\gamma_0}$)
is the path from $\alpha_i$ to $\beta_0$ (resp. from $\beta_q$ to $\gamma_0$)
defined in \eqref{eq:pathalphatobeta}.
The map $\psi$ is one-to-one. Since $\Gamma(G_p)$ is uncountable and $E$ is
countable, there exists
$\bar\gamma\in\Gamma(G_p)$ such that $\psi(\bar\gamma)\notin E$. Let 
$x:=\vfi\circ\psi(\bar\gamma)\in X_0$. Then $x\in
J_{\alpha_0\ldots\alpha_i}$ and $ f^{i+k}(x)\in J_{\beta_0\ldots
\beta_q}$, so $|f^{i+k}(x)-y|<\eps$. Since this is true for
any $\eps>0$, this implies that the set
$f^{i+k}(J_{\alpha_0\ldots\alpha_i}\cap X_0)$ is dense in $X$. By
compactness, we get $f^{i+k}(J_{\alpha_0\ldots\alpha_i}\cap X)=X$; this
is \eqref{eq:Jalpha-Y}.

\medskip
Finally we are going to 
show that $f|_X\colon X\to X$ is topologically mixing and that
the set of periodic points is dense in $X$. 
Let $U$ be an open set of $I$ such
that $U\cap X\neq \emptyset$. By denseness of $X_0$ in $X$, 
there exists $y$ in $U\cap X_0$. Let
$(\alpha_n)_{n\ge 0}\in \Gamma(G_p)\setminus E$ be such that 
$\vfi((\alpha_n)_{n\ge 0})=y$. Since 
$\lim_{n\to+\infty}|J_{\alpha_0\ldots
\alpha_n}|=0$ by \eqref{eq:diam0}, there exists an integer $q$ such that
$J_{\alpha_0\ldots \alpha_{q-1}}\subset U$. Then $f^{q+k}(U\cap X)=X$ by
\eqref{eq:Jalpha-Y}. Therefore, $f|_X$ is topologically mixing. 
Let $\bar\gamma=(\gamma_n)_{n\ge 0}$ be the periodic sequence of period 
$q$ beginning with $(\alpha_0\ldots\alpha_{q-1})$, that is,
$\gamma_n=\alpha_r$ if $n=pq+r$ with $r\in \Lbrack 0,q-1\Rbrack$.
The difficulty to find a periodic point in $U\cap X$ is that 
$\bar\gamma$ may belong to $E$ (if $\bar\gamma\notin E$,
then we have $z:=\vfi(\bar\gamma)\in X_0\cap U$ and $f^q(z)=z$). 
For every $n\ge 0$, there exists $z_n\in (J_{\gamma_0\ldots\gamma_n}\cap X_0)
\setminus J_{\bar\gamma}$. Let $(n_i)_{i\ge 0}$ be an increasing
sequence of integers such that $(z_{n_i})_{i\ge 0}$ converges, and let
$z$ denote the limit.
The point $z$ necessarily belongs to $\End{J_{\bar\gamma}}$ because
$\bigcap_{n\ge 0}J_{\gamma_0\ldots\gamma_n}=J_{\bar\gamma}$ is a 
decreasing intersection of intervals and $z_n\notin J_{\bar\gamma}$. 
Moreover, $z\in X$ because $X$ is closed.
For every $n\ge q$, $f^q(z_n)\in J_{\gamma_0\ldots\gamma_{n-q}}\setminus
J_{\bar\gamma}$ because $\sigma^q(\bar\gamma)=\bar\gamma$. Therefore
the sequence $(f^q(z_{n_i}))_{i\ge 0}$ converges to the point 
$f^q(z)$ by continuity, $f^q(z)\in X$ because $X$ is invariant and
$f^q(z)\in\End{J_{\bar\gamma}}$ for the same reason as above.
Similarly, $f^{2q}(z)\in \End{J_{\bar\gamma}}\cap X$. The three points
$\{z,f^q(z),f^{2q}(z)\}$ belong to $\End{J_{\bar\gamma}}$, and
thus two of these points are equal. 
Therefore, either
$z$ or $f^q(z)$ is a periodic point and belongs to $U\cap X$. This shows that
the set of periodic points is dense in $X$. Finally,
the facts that $f|_X$ is topologically mixing and has a dense set
of periodic points
ensure that the set of periods of periodic points in $X$ is infinite
(if the set of periods is finite and if $N$ is a common multiple of all
the periods, then $f^N|_X$ is the identity map by denseness of the set of
periodic points, and thus $f|_X$ is not mixing).
\end{proof}

\begin{theo}\label{theo:strongchaos}
Let $f$ be an interval map. The following are
equivalent:
\begin{enumerate}
\item $f$ has a periodic point of odd period greater than $1$,
\item there exists an infinite $f$-invariant 
set $X$ such that $(X,f^2|_X)$ is transitive.
\end{enumerate}
\end{theo}

\begin{proof}
The implication (i)$\Rightarrow$(ii) is given by 
Theorem~\ref{theo:oddperiod-mixing-subset}.

We suppose that there exists an infinite $f$-invariant 
set $X$ such that $f^2|_X$ is transitive. 
Let $y\in X$ be a point whose orbit under $f^2$ is dense in $X$. 
Since $X$ is infinite, the points $(f^n(y))_{n\ge 0}$ are pairwise distinct.
We may assume that $\min X<f(y)<\max X$ (otherwise, we can replace $y$ by 
some iterate). We also assume that $f(y)<f^2(y)$, the case with reverse 
inequality being symmetric. Since $\CO_{f^2}(y)$ is dense in $X$, there
exists $n\ge 2$ such that $f^{2n}(y)\in [\min X, f(y))$. Thus we have
$f^{2n}(y)<f(y)<f^2(y)$. According to
Proposition~\ref{prop:xn-x0-x1} applied to the point $x:=f(y)$, 
there exists a periodic point of odd
period greater than $1$. That is, (ii)$\Rightarrow$(i).
\end{proof}

According to Theorem~\ref{theo:htop-power-of-2},
an interval map has positive entropy if and only if it has
a periodic point of period $2^np$ for some $n\ge 0$ and some odd $p>1$.
Therefore the next corollary follows
straightforwardly from Theorems \ref{theo:oddperiod-mixing-subset} and 
\ref{theo:strongchaos}.

\begin{cor}\label{cor:htop-mixingsubsystem}
Let $f$ be an interval map. The following are equivalent:
\begin{enumerate}
\item $h_{top}(f)>0$,
\item there exist a positive integer $n$ and an uncountable $f^n$-invariant
set $X$ such that $(X,f^n|_X)$ is topologically mixing and the set
of periodic points is dense in $X$,
\item there exist a positive integer $n$ and an $f^n$-invariant
set $X$ such that\linebreak $(X,f^n|_X)$ is topologically mixing,
\item there exist a positive integer $n$ and an infinite $f^n$-invariant
set $X$ such that $(X,f^{2n}|_X)$ is transitive.
\end{enumerate}
\end{cor}

\begin{rem}
A result similar to, but weaker than, the equivalence
(i)$\Leftrightarrow$(iii) in Corollary~\ref{cor:htop-mixingsubsystem} 
was stated by Osikawa and Oono in \cite{OO}:
they proved that an interval map $f$ has a periodic point
whose period is not a power of $2$ if and only if there exists a mixing 
$f^n$-invariant measure for some positive integer $n$. The proof relies on the 
construction of a set $X$ such that $f^n(X)\subset X$ and 
$(X,f^n|_X)$ is Borel conjugate to the full shift on two symbols (i.e.,
the conjugacy map is only Borel and may not be continuous); in particular,
$X$ may not be closed in this construction.
\end{rem}

\subsection*{Remarks on graph maps}
For graph maps, there is no simple relation between positive entropy and
the periods of periodic points. Moreover, an irrational rotation
is totally transitive but not topologically mixing and has zero entropy; 
thus there is no way to get a result similar
to Theorem~\ref{theo:strongchaos}. However, 
Corollary~\ref{cor:htop-mixingsubsystem} can be partially
generalized to graph maps. 
Indeed, a graph map $f$ has positive entropy if and only if $f^n$
has a horseshoe for some $n$ (Theorem~\ref{theo:htop-horseshoeG}), 
which implies that there exists an
$f^n$-invariant set $X$ such that $(X,f^n)$ is ``almost'' conjugate to the
shift $(\Sigma,\sigma)$ (Proposition~\ref{prop:strictly-turbulent-shift}). 
Moreover, the properties of $\omega$-limit
sets for zero entropy graph maps (Theorems \ref{theo:omega-graph-htop0}
and \ref{theo:circumferential})
imply that a zero entropy graph map admits no 
topologically mixing subsystem. This leads to the following result.

\begin{theo}
Let $f$ be a graph map. The following are equivalent:
\begin{enumerate}
\item $h_{top}(f)>0$,
\item there exist a positive integer $n$ and an uncountable $f^n$-invariant
set $X$ such that $(X,f^n|_X)$ is topologically mixing and the set
of periodic points is dense in $X$,
\item there exist a positive integer $n$ and an $f^n$-invariant 
set $X$ such that\linebreak $(X,f^n|_X)$ is topologically mixing.
\end{enumerate}
\end{theo}

%***********************************************************************
\section{Transitive sensitive subsystems}\label{sec:Wiggins}

One may consider a variant of Devaney's definition 
of chaos by omitting the assumption on periodic points.
What can be said about interval maps having transitive sensitive
subsystems? 
By Theorem~\ref{theo:htop-Devaney}, a positive entropy interval map
has a transitive sensitive subsystem. The converse is not true: the map
built in Example~\ref{ex:chaos-LY-htop0} has zero entropy but has a
transitive sensitive subsystem by Lemma~\ref{lem:chaos-LY-htop0}.
We are going to show that the existence of a transitive sensitive subsystem 
implies chaos in the sense of 
Li-Yorke. The converse is not true either: a 
(rather complicated) counter-example is given 
in Subsection~\ref{subsec:LYnotWiggins}. 
It follows that, for interval maps, 
the existence of a transitive sensitive subsystem is a strictly
intermediate notion between positive entropy and chaos in the sense of
Li-Yorke. These results were shown by the author in \cite{R5}.

\begin{rem}
A topological dynamical system $(X,f)$ is sometimes called
\emph{chaotic in the sense of Auslander-Yorke}
\index{chaos!chaos in the sense of Auslander-Yorke} 
or \emph{chaotic in the sense of Ruelle and 
Takens}\index{chaos!chaos in the sense of Ruelle and Takens}
if it is transitive and sensitive to initial 
conditions \cite{AY, RT2}, and
\emph{chaotic in the sense of Wiggins}\index{chaos!chaos in the sense of 
Wiggins}  if there exists an invariant
set $Y\subset X$ such that $(Y,f|_Y)$ is transitive and sensitive to
initial conditions \cite{FD}.
\end{rem}

%***********
\subsection{Transitive sensitive subsystem implies Li-Yorke chaos}

The next result is \cite[Theorem~1.7]{R5}.

\begin{theo}\label{theo:Wiggins-implies-LY}
Let $f$ be an interval map. If $Y$ is an 
invariant set such that $f|_Y$ is transitive and sensitive 
to initial conditions, then $f$ is chaotic in the sense of Li-Yorke.
\end{theo}

\begin{proof}
We show the result by refutation.
Suppose that $f$ is not chaotic in the sense of Li-Yorke. By
Theorem~\ref{theo:htop-positive-chaos-LY}, the topological entropy of $f$
is zero.
Consider an invariant set $Y$ such that $f|_Y$ is
transitive. If $Y$ is finite, then 
$f|_Y$ is not sensitive. If $Y$ is infinite, 
there exists $y_0\in Y$ such that $\omega(y_0,f)=Y$ 
(Proposition~\ref{prop:transitive-dense-orbit}(i)). 
By Theorem~\ref{theo:htop0-chaos-LY}, $Y$ does not
contain two $f$-non separable points.
Let $\eps>0$. According to Proposition~\ref{prop:non-LY-chaotic}(i),
there exists an integer $n\ge 1$ such that
$$
\max_{i\in\Lbrack 0, 2^n-1\Rbrack} \diam (\omega(f^i(y_0),f^{2^n}))<\eps.
$$
We set $I_i:=[\min \omega(f^i(y_0),f^{2^n}), \max \omega(f^i(y_0),f^{2^n})]$
for all $i\in\Lbrack 0, 2^n-1\Rbrack$. Then $f(Y\cap I_i)=Y\cap
I_{i+1\bmod 2^n}$ because $f(\omega(f^i(y_0),f^{2^n}))=
\omega(f^{i+1\bmod 2^n}(y_0),f^{2^n})$ and 
$Y=\bigcup_{i\in\Lbrack 0, 2^n-1\Rbrack}\omega(f^i(y_0),f^{2^n})$ by
Lemma~\ref{lem:omega-set}. Moreover the intervals $(I_i)_{0\le i\le 2^n-1}$ 
are pairwise disjoint by Proposition~\ref{prop:htop0-Lki}.
Let $\delta>0$ be such that the distance between two different intervals
among $(I_i)_{0\le i\le 2^n-1}$ is greater than $\delta$. Let $x,y\in Y$
be such that $|x-y|<\delta$. Then there exists $i\in\Lbrack 0, 2^n-1\Rbrack$
such that $x,y\in I_i$ and, for all $k\ge 0$, 
$f^k(x),f^k(y)\in I_{i+k\bmod 2^n}$, so  $|f^k(x)-f^k(y)|<\eps$.
We conclude that $f|_Y$ is  not sensitive to initial conditions.
\end{proof}

%********************************************
\subsection{Li-Yorke chaos does not imply a transitive sensitive subsystem}
\label{subsec:LYnotWiggins}

The aim of this subsection is to exhibit an interval map 
$h\colon [0,3/2]\to [0,3/2]$ that is
chaotic in the sense of Li-Yorke but has no transitive sensitive subsystem.
This example is taken from \cite{R5}.
Let us first explain the main underlying ideas of the
construction of $h$. This map is  obtained by modifying the construction 
of the map $g$ of Example~\ref{ex:chaos-LY-htop0}. The maps $h$ and $g$
have the same construction on the set $\bigcup_{n\ge 1} I_n^0$ 
-- which is the core of the dynamics of $g$ -- but the the 
lengths of the intervals $(I_n^0)_{n\ge 1}$ are not the same and the definition
of $h$ on the intervals $(L_n)_{n\ge 1}$ is different.
For $g$, we showed that $K:=\bigcap_{n\ge 0}\bigcup_{i=0}^{2^n-1}
g^i(J_n^0)$ has a non degenerate connected component $C$ and that
the endpoints of $C$ are 
$g$-non separable. The same remains true for $h$ with
$C:=\bigcap_{n\ge 0}I_n^1=[a,1]$ (the fact that $a,1$ are $h$-non
separable will be proved in Proposition~\ref{prop:wiggins-LY-g-LY-chaotic}).
For $g$, we proved that ${\rm Bd}_{\IR}{K}\subset \omega(0,g)$, hence
$\End{C}\subset \omega(0,g)$.
For $h$,  it is not true that $\End{C}\subset \omega(0,h)$ because
the orbit of $0$ stays in $[0,a]$.
The construction of $h$ on the intervals $L_n$ allows one to approach $1$ from
outside $[0,1]$: we shall see in 
Proposition~\ref{prop:wiggins-LY-g-LY-chaotic} that $\omega(3/2,h)$
contains both $a$ and $1$, which implies chaos in the sense of Li-Yorke
because $a$ and $1$ are $h$-non separable. On the other hand,
the proof showing that $g|_{\omega(0,g)}$ 
is transitive and sensitive 
fails for $h$ because $\omega(0,h)$ does not contain $\{a,1\}$, 
and $\omega(3/2,h)$ is not transitive.
We shall see in Proposition~\ref{prop:wiggins-LY-g-not-wiggins-chaotic}
that $h$ has no transitive sensitive subsystem at all.

\medskip
Let $(a_n)_{n\ge 0}$ be an increasing sequence of numbers
less than $1$ such that $a_0=0$. We set $I_0^1:=[a_0,1]$ and
$$
\forall n\ge 1,\ I_n^0:=[a_{2n-2},a_{2n-1}],\ L_n:=[a_{2n-1},a_{2n}],\ I_n^1:=[a_{2n},1].
$$
It is clear that $I_n^0\cup L_n\cup I_n^1=I_{n-1}^1$.
We fix $(a_n)_{n\ge 0}$ such that the lengths of the intervals
satisfy
$$
\forall n\ge 1,\ |I_n^0|=|L_n|=\frac{1}{3^n}|I_{n-1}^1|\quad\text{and}\quad
|I_n^1|=\left(1-\frac{2}{3^n}\right)|I_{n-1}^1|.
$$
Let $a:=\lim_{n\to +\infty}a_n$. 
Then $\bigcup_{n\ge 1}(I_n^0\cup L_n)=[0,a)$. Moreover, $a<1$ because
$$
\log (1-a)=\sum_{n=1}^{+\infty}\log \left(1-\frac{2}{3^n}\right)>-\infty,
$$
the last inequality follows from the facts that
$\log(1+x)\sim x$ when $x\to 0$ and $\sum \frac{1}{3^n}<+\infty$.

\medskip
\textsc{Notation.}
If $I$ is an interval, let $\middle(I)$ denote its middle point
(that is, $\middle([b,c])=\frac{b+c}2$). For short, we write $\uparrow$ (resp.
$\downarrow$) for ``increasing'' (resp. ``decreasing'').
\label{notation:increasing}\label{notation:decreasing}

\medskip
For all $n\ge 1$,
let $\vfi_n\colon I_n^0\to I_n^1$ be the increasing linear homeomorphism
mapping $I_n^0$ onto $I_n^1$.
We define the map $h\colon [0,3/2]\to  [0,3/2]$ such that $h$ is
continuous on $[0,3/2]\setminus\{a\}$ and
\begin{eqnarray*}
&& h(x)=\vfi_1^{-1}\circ \vfi_2^{-2}\circ\cdots\circ
\vfi_{n-1}^{-1}\circ\vfi_n(x)\text{ for all } x\in I_n^0,\ n\ge 1,\\
&& h\text{ is linear $\uparrow$ of slope }\lambda_n\text{ on } [\min L_n,\middle(L_n)]
\quad\text{for all }n\ge 1,\\
&&h \text{ is linear $\downarrow$ on } 
[\middle(L_n),\max L_n]\quad\text{for all }n\ge 1,\\
&&h(x)=0 \quad\text{for all }x\in [a,1],\\
&&h(x)=x-1 \quad\text{for all } x\in [1,3/2],
\end{eqnarray*}
where the slopes $(\lambda_n)$ will be defined below. We shall also show below
that $h$ is continuous at $a$. The map $h$ is represented on 
Figure~\ref{fig:LY-not-wiggins}.

\begin{figure}[htb]
\centerline{\includegraphics{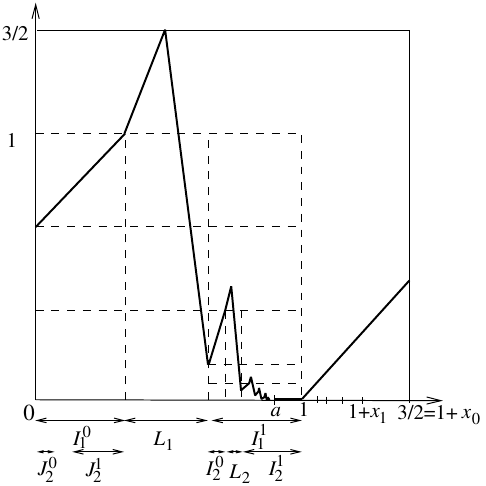}}
\caption{The graph of $h$; this map is chaotic in the sense of Li-Yorke
but has no transitive sensitive subsystem.}
\label{fig:LY-not-wiggins}
\end{figure}

We set $J_0^0:=[0,1]$ and, for all $n\ge 1$, we define $J_n^0,
J_n^1$ as subinterval of $J_{n-1}^0$ such that
$\min J_n^0=0$, $\max J_n^1=\max J_{n-1}^0$ and
$\frac{|J_n^i|}{|J_{n-1}^0|}=\frac{|I_n^i|}{|I_{n-1}^1|}\text{ for }
i\in\{0,1\}$. We also set $M_n:=[\max J_n^0,\min J_n^1]$.

Notice that, on the set $\bigcup_{n\ge 1}I_n^0$, the map $h$ is defined
similarly to the map $g$  of Example~\ref{ex:chaos-LY-htop0}
(the reader can refer to Figure~\ref{fig:LY-htop0-partial-construction} 
page~\pageref{fig:LY-htop0-partial-construction} and the explanations of the
underlying construction of $g$ on this set).
Therefore, the assertions of Lemma~\ref{lem:LY-htop0-summary-Jn0} remain valid for 
$h$, except the point (i) and its derived results (viii), (x), (xi).

\begin{lem}\label{lem:wiggins-LY-summary-Jn0}
Let $h$ be the map defined above.
Then, for all $n\ge 1$,
\begin{enumerate}
\item $h(I_n^0)=J_n^1$,
\item $h^i|_{J_n^0}$ is linear $\uparrow$ for all $i\in\Lbrack 0,2^n-1\Rbrack$,
\item $h^{2^{n-1}-1}(J_n^0)=I_n^0$ and $h^{2^n-1}(J_n^0)=I_n^1$,
\item $h^i(J_n^0)\subset \bigcup_{k=1}^nI_k^0$ for all
$i\in\Lbrack 0,2^n-2\Rbrack$,
\item $(h^i(J_n^0))_{0\le i<2^n}$ are pairwise disjoint,
\item $h^{2^{n-1}}|_{I_n^0}$ is linear $\uparrow$ and 
$h^{2^{n-1}}(I_n^0)=I_n^1$,
\item $h^{2^{n-1}-1}|_{M_n}$ is linear $\uparrow$ and
  $h^{2^{n-1}-1}(M_n)=L_n$,
\item $h(\min L_n)=\min M_{n-1}$,
\item $h^{2^{n-2}}(\min L_n)=\min L_{n-1}$.
\end{enumerate}
\end{lem}

\begin{proof}
For the assertions (i) to (vi), see the proof of
Lemma~\ref{lem:LY-htop0-summary-Jn0}(ii)-(vi)+(ix).

According to (ii), the map $h^{2^{n-1}-1}|_{M_n}$ is
linear $\uparrow$ because $M_n$ is included in $J_{n-1}^0$. 
Since $M_n=[\max J_n^0,\min J_n^1]$ and
$L_n=[\max I_n^0,\min I_n^1]$, the combination of (i), (ii) and (iii) implies 
that $h^{2^{n-1}-1}(M_n)=L_n$; this is (vii).

The map $h|_{I_n^0}$ is increasing and $\min L_n=\max I_n^0$. Hence,
according to (i), we have $h(\min L_n)=\max J_n^1=\max J_{n-1}^0=
\min M_{n-1}$; this is (viii).

Finally, (ix) follows from (vii) and (viii).
\end{proof}

For all $n\ge 0$, we set $x_n:=\middle(M_{n+1})$, that is,
$x_n=\frac{3}{2}\prod_{i=1}^{n+1}\frac{1}{3^i}$.
It is a decreasing
sequence and $x_0=1/2$. Therefore $h(1+x_n)$ is well
defined and equal to $x_n$ for all $n\ge 0$.

For all $n\ge 0$, let $t_n:=\slope\left(h^{2^n-1}|_{J_n^0}\right)$;
by convention, $h^0$ is the identity map, so $t_0=1$. We fix $\lambda_1:=
\frac{2x_1}{|L_1|}$ and we define
inductively $(\lambda_n)_{n\ge 2}$ such that
\begin{equation}
\label{eq:defi-lambda-n}
\frac{|L_n|}{2} \prod_{i=1}^n \lambda_i\prod_{i=0}^{n-2}t_i=x_n.
\end{equation}
By convention, an empty product is equal to $1$,
so \eqref{eq:defi-lambda-n} is satisfied for $n=1$ too.

The slopes $(\lambda_n)_{n\ge 1}$ are such that
$h^{2^{n-1}}([\min L_n,\middle(L_n)])=[1,1+x_n]$, as 
proved  in the next lemma. This means that, under the action of $h^{2^{n-1}}$,
the image of $L_n$ falls outside of $[0,1)$ but remains close to $1$.
We also list some properties of $h$ on
the intervals $L_n$, $I_n^1$ and $[1,1+x_n]$.

\begin{lem}\label{lem:wiggins-LY-summary2}
Let $h$ be the map defined above. Then
\begin{enumerate}
\item 
$h^{2^n}|_{[1,1+x_n]}$ is linear $\uparrow$ and
$h^{2^n}([1,1+x_n])=[\min I_{n+1}^0,\middle(L_{n+1})]$ for all $n\ge 0$,
\item 
$h^{2^{n-1}}|_{[\min L_n,\middle(L_n)]}$ is linear $\uparrow$
and $h^{2^{n-1}}([\min L_n,\middle(L_n)])\!=\![1,1+x_n]$ for all $n\ge 1$,
\item
$h^{2^{n+1}}|_{[1,1+x_n]}$ is $\uparrow$ and 
$h^{2^{n+1}}([1,1+x_n])=I_{n+1}^1\cup [1,1+x_{n+1}]$ for
all $n\ge 1$,
\item
$h(I_n^1)\subset [0,\middle(M_n)]$ for all $n\ge 1$,
\item
$h^{2^n}([\min I_n^1,1+x_n])\subset [\min I_n^1,1+x_n]$ and
$h^i([\min I_n^1,1+x_n])\subset [0,1]$ for all $n\ge 1$ and all
$i\in\Lbrack 1,2^n-1\Rbrack$.
\end{enumerate}
\end{lem}

\begin{proof}
The map $h|_{[1,1+x_n]}$ is linear $\uparrow$ and $h([1,1+x_n])=
[0,\middle(M_{n+1})]\subset J_n^0$, thus $h^{2^n}|_{[1,1+x_n]}$ is linear
$\uparrow$ by Lemma~\ref{lem:wiggins-LY-summary-Jn0}(ii). Moreover
$h^{2^n-1}(0)=\min I_{n+1}^0$ and $h^{2^n-1}(\middle(M_{n+1}))=
\middle(L_{n+1})$ by Lemma~\ref{lem:wiggins-LY-summary-Jn0}(ii)+(iii)+(vii);
this implies (i).

\medskip
Before proving (ii), we show some intermediate results.
Let $n,k$ be integers with $n\ge 2$ and $k\in\Lbrack 2,n\Rbrack$. Then
\begin{eqnarray*}
\lambda_n\ldots \lambda_k\cdot t_{n-2}\ldots t_{k-2}&=&
\frac{\disp\prod_{i=1}^n\lambda_i\prod_{i=0}^{n-2}t_i}{\disp
\prod_{i=1}^{k-1}\lambda_i\prod_{i=0}^{k-3}t_i}\\
&=&\frac{x_n}{x_{k-1}}\cdot\frac{|L_{k-1}|}{|L_n|}\quad\text{by 
\eqref{eq:defi-lambda-n}}\\
&=&\prod_{i=k+1}^{n+1}\frac{1}{3^i}
\prod_{i=k-1}^{n-1}\frac{1}{1-\frac{2}{3^i}}\cdot\frac{3^n}{3^{k-1}}\\
&=&\frac{1}{3^{n-k+1}}\prod_{i=k-1}^{n-1}\frac{1}{3^i-2}
\end{eqnarray*}
and so
\begin{equation}\label{eq:wiggins-LY-summary2:1}
\lambda_n\ldots \lambda_k\cdot t_{n-2}\ldots t_{k-2} <1.
\end{equation}
By definition, $h|_{[\min L_n, \middle(L_n)]}$ is linear $\uparrow$ and
$h(\middle(L_n))=h(\min L_n)+\lambda_n \frac{|L_n|}{2}$.
By \eqref{eq:defi-lambda-n}, 
\begin{eqnarray*}
\lambda_n \frac{|L_n|}{2}
&=&\frac{x_n}{\disp t_{n-2}\prod_{i=1}^{n-1}\lambda_i\prod_{i=0}^{n-3}t_i}\\
&=& \frac{x_n|L_{n-1}|}{2x_{n-1}t_{n-2}}\\
&=&\frac{1}{3^{n+1}}\frac{|M_{n-1}|}{2};
\end{eqnarray*}
the last equality is because
$J_{n-2}^0\supset M_{n-1}$, so $t_{n-2}=\frac{|L_{n-1}|}{|M_{n-1}|}$ by
Lemma~\ref{lem:wiggins-LY-summary-Jn0}(vii). Therefore
$$
\lambda_n \frac{|L_n|}{2}<\frac{|M_{n-1}|}{2}.
$$
Moreover, $h(\min L_n)=\min M_{n-1}$ by 
Lemma~\ref{lem:wiggins-LY-summary-Jn0}(viii), hence
\begin{equation}\label{eq:wiggins-LY-summary2:2}
h([\min L_n,\middle(L_n)])\subset [\min M_{n-1},\middle(M_{n-1})]
\quad\text{for all }n\ge 2.
\end{equation}

\medskip
\textsc{Fact 1.}
{\it For all $k\in\Lbrack 2,n\Rbrack$,
\begin{itemize}
\item the map
$h^{2^{n-2}+2^{n-3}+\cdots+2^{k-2}}$
is linear $\uparrow$ of slope $\lambda_n\ldots\lambda_k t_{n-2}\ldots t_{k-2}$
on $[\min L_n,\middle(L_n)]$ and
sends $\min L_n$ to $\min L_{k-1}$,
\item $h^i([\min L_n,\middle(L_n)])\subset [0,1]$ for all
$i\in\Lbrack 0,2^{n-2}+2^{n-3}+\cdots+2^{k-2}\Rbrack$.
\end{itemize}}

\medskip
We show this fact by induction on $k$, where $k$ decreases from $n$
to $2$.

\noindent$\bullet$ By \eqref{eq:wiggins-LY-summary2:2}
we have $h([\min L_n,\middle(L_n)])\subset M_{n-1}\subset J_{n-2}^0$, so
$h^{2^{n-2}}|_{[\min L_n,\middle(L_n)]}$ is linear $\uparrow$ of slope
$\lambda_n t_{n-2}$.
By Lemma~\ref{lem:wiggins-LY-summary-Jn0}(ix),
$h^{2^{n-2}}(\min L_n)=\min L_{n-1}$. Then
\eqref{eq:wiggins-LY-summary2:2} and 
Lemma~\ref{lem:wiggins-LY-summary-Jn0}(iii)+(iv) imply that
$h^i([\min L_n,\middle(L_n)])\subset [0,1]$
for all $i\in\Lbrack 1,2^{n-2}\Rbrack$. This is Fact 1 for $k=n$.

\noindent$\bullet$ Suppose that Fact 1 holds for some
$k\in\Lbrack 3,n\Rbrack$.
By \eqref{eq:wiggins-LY-summary2:1}, we have 
$$
\lambda_n\ldots\lambda_k\cdot t_{n-2}\ldots t_{k-2}\frac{|L_n|}{2}
\le \frac{|L_{k-1}|}{2}
$$
so that
$$
h^{2^{n-2}+2^{n-3}+\cdots+2^{k-2}}([\min L_n,\middle(L_n)])\subset
[\min L_{k-1},\middle(L_{k-1})].
$$
The map $h$ is of slope $\lambda_{k-1}$
on this interval, $h(\min L_{k-1})=\min M_{k-2}$ according to
Lemma~\ref{lem:wiggins-LY-summary-Jn0}(viii) and 
$h([\min L_{k-1},\middle(L_{k-1})])\subset
M_{k-2}$ by \eqref{eq:wiggins-LY-summary2:2}.
Since $M_{k-2}\subset J_{n-1}^0$, the map
$h^{2^{n-2}+2^{n-3}+\cdots+2^{k-2}+2^{k-3}}$ is linear $\uparrow$ of slope
$\lambda_n\ldots\lambda_{k-1}\cdot t_{n-2}\ldots t_{k-3}$ on
$[\min L_n,\middle(L_n)]$, and it sends $\min L_n$ to $\min L_{k-2}$
by Lemma~\ref{lem:wiggins-LY-summary-Jn0}(ix).
Moreover, $h^i([\min L_n,\middle(L_n)])\subset [0,1]$ for all 
$i\in\Lbrack 0,2^{n-2}+2^{n-3}+\cdots+2^{k-2}+2^{k-3}\Rbrack$ by
Lemma~\ref{lem:wiggins-LY-summary-Jn0}(iv) and the induction hypothesis.
This is Fact 1 for $k-1$. This ends the induction and proves Fact 1.

\medskip
For $k=2$, Fact 1 implies that
$h^{2^{n-2}+\cdots +2^0}=h^{2^{n-1}-1}$ is linear increasing of slope
$\prod_{i=2}^n\lambda_i\prod_{i=0}^{n-2}t_i$ on $[\min L_n,\middle(L_n)]$,
with 
\begin{gather*}
h^{2^{n-1}-1}(\min L_n)=\min L_1\\
\text{and }h^{2^{n-1}-1}
([\min L_n,\middle(L_n)])\subset [\min L_1,\middle(L_1)].
\end{gather*}
The map $h$ is of slope $\lambda_1$ on this interval, hence,
according to the definition of $(\lambda_n)$, 
(ii) holds for all $n\ge 2$; it also trivially holds for $n=1$.
Fact 1 for $k=2$ also shows that
\begin{equation}\label{eq:wiggins-LY-summary2:5}
h^i([\min L_n,\middle(L_n)])\subset [0,1]\quad\text{for all }
i\in\Lbrack 0,2^{n-1}-1\Rbrack\text{ and all } n\ge 1.
\end{equation}

Then (iii) follows from (i), (ii) and 
Lemma~\ref{lem:wiggins-LY-summary-Jn0}(vi).

\medskip
We have $I_n^1=\bigcup_{k\ge n+1}(I_k^0\cup L_k)\cup [a,1]$.
From the definition of $h$, we can see that 
$$
\max\{h(x)\mid x\in I_k^0\cup
L_k\}=h(\middle(L_k)),
$$ 
so $h(I_k^0\cup L_k)\subset [0,\middle(M_{k-1})]$ 
by \eqref{eq:wiggins-LY-summary2:2}. Hence
\begin{equation}\label{eq:hIn1}
h(I_n^1)\subset [0,\middle(M_n)]=J_n^0\cup [\min M_n,\middle(M_n)];
\end{equation}
this is (iv).

\medskip
According to Lemma~\ref{lem:wiggins-LY-summary-Jn0}(iii)+(vii), 
$$
h^{2^n-1}(J_n^0)=I_n^1\text{ and }h^{2^{n-1}-1}([\min M_n,\middle(M_n)])=[\min L_n,
\middle(L_n)],
$$
and by (ii),
$h^{2^{n-1}}([\min L_n,\middle(L_n)])=[1,1+x_n]$.
Combined with \eqref{eq:hIn1}, we get
\begin{equation}\label{eq:g2n}
h^{2^n}(I_n^1)\subset I_n^1\cup [1,1+x_n].
\end{equation}
Moreover, $h^i(J_n^0)\subset [0,1]$ for all $i\in\Lbrack 0,2^n-2\Rbrack$ 
and $h^i([\min M_n,\middle(M_n)])\subset [0,1]$ for all $i\in
\Lbrack 0,2^{n-1}-2\Rbrack$
according to Lemma~\ref{lem:wiggins-LY-summary-Jn0}(iv). In addition,
$$
h^{2^{n-1}+i-1}([\min M_n,\middle(M_n)])=h^i([\min L_n, \middle(L_n)])\subset
[0,1]
$$
for all $i\in\Lbrack 0,2^{n-1}-1\Rbrack$ by 
Lemma~\ref{lem:wiggins-LY-summary-Jn0}(vii) 
and \eqref{eq:wiggins-LY-summary2:5}.
Therefore,
\begin{equation}\label{eq:wiggins-LY-summary2:8}
h^i(I_n^1)\subset [0,1]\text{ for all }i\in\Lbrack 0,2^n-1\Rbrack. 
\end{equation}

Finally, since $h([1,1+x_n])=[0,\middle(M_{n+1})]\subset J_n^0$,
statement (i) implies that $h^{2^n}([1,1+x_n])\subset I_n^1$. Combined
with \eqref{eq:wiggins-LY-summary2:8}, \eqref{eq:g2n} and
Lemma~\ref{lem:wiggins-LY-summary-Jn0}(iv), this implies (v).
\end{proof}

Now we show that
$h$ is continuous at the point $a$ as claimed at the beginning of the section.

\begin{lem}
The map $h$ defined above is continuous.
\end{lem}

\begin{proof}
We just have to show the continuity at $a$.
It is clear from the definition that $h$ is continuous at $a^+$.
According to Lemma~\ref{lem:wiggins-LY-summary2}(iv), we have
$h(I_n^1)\subset J_{n-1}^0$. By definition, $h(a)=0$
and $a=\max I_n^1$ for all $n$.
Moreover, by definition of the intervals $(J_n^0)_{n\ge 0}$,
$$
\lim_{n\to+\infty}\max J_n^0=\lim_{n\to+\infty} |J_n^0|=0,
$$
and thus $\disp\lim_{n\to+\infty}\max h(I_n^1)=0$.
This implies that 
$h$ is continuous at $a^-$.
\end{proof}

\begin{prop}\label{prop:wiggins-LY-g-LY-chaotic}
Let $h$ be the map defined above. Then the set $\omega(1+x_0,h)$ is
infinite and contains the points $a$ and $1$, which are $h$-non separable.
Consequently, the map $h$ is chaotic in the sense of Li-Yorke.
\end{prop}

\begin{proof}
Lemma~\ref{lem:wiggins-LY-summary2}(iii) implies that
$h^{2^{n+1}}(1+x_n)=1+x_{n+1}$ for all $n\ge 0$.
Since $x_n\to 0$ when $n$ goes to infinity, this implies that
$1$ belongs to $\omega(1+x_0,h)$ (recall that $\omega(1+x_0,h)$
is closed by Lemma~\ref{lem:omega-set}(i)). 
Moreover, Lemma~\ref{lem:wiggins-LY-summary2}(i) implies that 
$h^{2^n}(1)=\min I_{n+1}^0=a_{2n}$ for all $n\ge 1$, 
so $a$ belongs to $\omega(1,h)\subset \omega(1+x_0,h)$. 

Suppose that $A_1,A_2$ are two
periodic intervals such that $a\in A_1$ and $1\in A_2$,
and let $p$ be a common multiple of their periods. Since $h(a)=h(1)=0$,
it follows that $h^p(a)=h^p(1)\in A_1\cap A_2$, so $A_1,A_2$ are not disjoint.
This means that $a$ and  $1$ are $h$-non separable. 

A finite $\omega$-limit set is a periodic orbit (Lemma~\ref{lem:omega-finite}).
Therefore, if $y_0,y_1$ are two distinct points in 
a finite $\omega$-set, the degenerate intervals $\{y_0\}$, $\{y_1\}$
are periodic and $y_0,y_1$ are $h$-separable. This implies that
$\omega(1+x_0,h)$ is infinite. 
We deduce that the map $h$ is chaotic in the sense
of Li-Yorke by Theorem~\ref{theo:htop0-chaos-LY}.
\end{proof}

The next lemma is about the location of transitive subsystems.

\begin{lem}\label{lem:wiggins-LY-transitive-set}
Let $h$ be the map defined above and let $Y$ be an invariant 
set with no isolated point such that $h|_Y$ is transitive.
Then
\begin{enumerate}
\item $Y\subset [0,a]$,
\item $\disp Y\subset \bigcup_{i=0}^{2^n-1}h^i(J_n^0)$ for all $n\ge 1$,
\item $h^i(J_n^0\cap Y)=h^i(J_n^0)\cap Y=h^{i\bmod 2^n}(J_n^0)\cap Y$ for all 
$i\ge 0$ and all $n\ge 0$.
\end{enumerate}
\end{lem}

\begin{proof}
Since $f|_Y$ is transitive, there exists $y_0\in Y$ such that 
$\omega(y_0,h)=Y$ by Proposition~\ref{prop:transitive-dense-orbit}; 
in particular, the set $Y':=\CO_h(y_0)$
is dense in $Y$ and $y\in\omega(y,h)$ for all $y\in Y'$.
Note that $Y'$ is infinite, otherwise $Y$ would be a finite set and would
contain isolated points.

Let $n\ge 0$. By Lemma~\ref{lem:wiggins-LY-summary2}(iii),
$h^{2^{n+1}}([1,1+x_n])=I_{n+1}^1\cup [1,1+x_{n+1}]$. Thus
Lemma~\ref{lem:wiggins-LY-summary2}(v) implies
that, for all $k\in\IN$,
$h^{k2^{n+1}}([1,1+x_n])\subset I_{n+1}^1\cup [1,1+x_{n+1}]$
and $h^i([1,1+x_n])\subset [0,1]$ for all $i>2^{n+1}$ such that 
$i\notin 2^{n+1}\IN$. This implies that 
$$
h^i((1+x_{n+1},1+x_n])\subset [0,1+x_{n+1}]
\quad\text{for all }i\ge 2^{n+1}.
$$ 
Consequently, there is no $y\in (1,3/2]=\bigcup_{n\ge 0}
(1+x_{n+1},1+x_n]$ such that $y$ is in $\omega(y,h)$. So
$Y'\cap (1,3/2]=\emptyset$, and thus $Y\cap (1,3/2]=\emptyset$
because $Y'$ is dense in $Y$.

Since $h^{2^n-1}(0)=a_{2n}$ by 
Lemma~\ref{lem:wiggins-LY-summary-Jn0}(ii)+(iii),
the point $0$ is not periodic, so
$h^k(0)\notin [a,1]$ for all $k\ge 1$.
If $y\in (a,1)$, then $h(y)=0$ and $h^k(y)\notin [a,1]$ for all
$k\ge 1$, which implies that $y\notin \omega(y,h)$. Consequently,
$Y\cap (a,1)=\emptyset$. We have shown that $Y\subset [0,a]\cup\{1\}$;
in addition, $1\notin Y$ because $Y$ has no isolated point; this proves~(i).

\medskip
Let $n\ge 1$. Since $\min L_n=\max I_n^0$ and $\max L_n=\min I_{n+1}^0$,
it follows that
$h(\min L_n)=\max J_n^1$ and $h(\max L_n)=\min J_{n+1}^1$ according to
Lemma~\ref{lem:wiggins-LY-summary-Jn0}(i), and thus
$h(\max L_n)<h(\min L_n)$.
Moreover, $h|_{[\min L_n,\middle(L_n)]}$ is $\uparrow$ and 
$h|_{[\middle(L_n),\max L_n]}$ is linear $\downarrow$, so there exists
$c_n$ in $[\middle(L_n),\max L_n]$ such that $h(c_n)=h(\min L_n)$.

Since $h([c_n,\max L_n])=[\min J_{n+1}^1,\max J_n^1]$ is included
in the interval $J_{n-1}^0$,
the map $h^{2^{n-1}}|_{[c_n,\max L_n]}$ is linear $\downarrow$ by
Lemma~\ref{lem:wiggins-LY-summary-Jn0}(ii). Moreover, $M_n$ is included
in $h([c_n,\max L_n])$, so $h^{2^{n-1}}([c_n,\max L_n])$ contains $L_n$ by
Lemma~\ref{lem:wiggins-LY-summary-Jn0}(vii). Thus
there exists a point $z_n$ in the interval $[c_n,\max L_n]$ such that 
$h^{2^{n-1}}(z_n)=z_n$ (by Lemma~\ref{lem:fixed-point}) and we have
$\slope(h^{2^{n-1}}|_{[c_n,\max L_n]})\le -2$.
Then for every $x\in [c_n,\max L_n]$ with
$x\ne z_n$, there exists $k\ge 1$ such that $h^{k2^{n-1}}(x)\notin
[c_n,\max L_n]$. By Lemma~\ref{lem:wiggins-LY-summary2}(v), we have
$h^{2^{n-1}}(I_{n-1}^1\cup [1,1+x_{n-1}])\subset I_{n-1}^1\cup [1,1+x_{n-1}]$,
which implies that
\begin{eqnarray}  
\lefteqn{\forall x\in [c_n,\max L_n],\ x\ne z_n,\ \exists k\ge 1,}\label{eq:out-of-[cn,maxLn]}\\
&\qquad\qquad&
h^{k2^{n-1}}(x)\in I_n^0\cup [\min L_n, c_n]\cup I_n^1\cup [1,1+x_{n-1}].
\nonumber
\end{eqnarray}

We show by induction on $n$ that 
\begin{equation}\label{eq:Y-In1}
\forall n\ge 0,\quad Y'\cap I_n^1\ne \emptyset.
\end{equation}
This is true for $n=0$ because $Y\subset [0,1]=I_0^1$ by (i).
Suppose that there exists $y\in Y'\cap I_{n-1}^1$ for some $n\ge 1$. We write
$I_{n-1}^1=I_n^0\cup L_n\cup I_n^1$; to prove that
$Y'\cap I_n^1\ne \emptyset$, we split into four cases according to the
position of $y$.\\
$\bullet$ If $y\in I_n^1$, there is nothing to do.\\
$\bullet$ If $y\in I_n^0$, then $h^{2^{n-1}}(y)\in I_n^1$ by
Lemma~\ref{lem:wiggins-LY-summary-Jn0}(vi) and $h^{2^{n-1}}(y)\in Y'$.\\
$\bullet$ If $y\in [\min L_n,c_n]$, then $h(y)\in h([\min L_n,\middle(L_n)]$
and $h^{2^{n-1}}(y)\in [1,1+x_n]$ by Lemma~\ref{lem:wiggins-LY-summary2}(ii),
which is impossible because $Y\subset [0,a]$ by (i).\\
$\bullet$ If $y\in [c_n,\max L_n]$, then $y\ne z_n$ because $Y'$ is infinite.
In addition $h^j(y)\in [0,1]$ for all $j\ge 0$ according to (i). Then
\eqref{eq:out-of-[cn,maxLn]} states that there exists $j\ge 1$
such that $h^j(y)$ belongs to $I_n^0\cup [\min L_n,c_n]\cup I_n^1$ and one of 
the first three cases applies with $y':=h^j(y)\in Y'$.

\medskip
We have $h(I_n^1)\subset J_n^0\cup [\min M_n,\middle(M_n)]$ by
Lemma~\ref{lem:wiggins-LY-summary2}(iv), and also
$$
h^{2^n-1}([\min M_n,\middle(M_n)])=h^{2^{n-1}}([\min L_n,\middle(L_n)])
=[1,1+x_n]
$$
by Lemmas \ref{lem:wiggins-LY-summary-Jn0}(vii) and
\ref{lem:wiggins-LY-summary2}(ii). Combined with (i) and the $f$-invariance
of $Y$, this implies that
\begin{equation}\label{eq:Y-Jn0}
h(Y\cap I_n^1)\subset J_n^0.
\end{equation}
We have $Y\subset \CO_h(I_n^1)$ by \eqref{eq:Y-In1}. Combined with
\eqref{eq:Y-Jn0} and Lemma~\ref{lem:wiggins-LY-summary-Jn0}(i)+(iii),
we get
$$
Y\subset \bigcup_{i=0}^{2^n-1}h^i(J_n^0)\quad \text{for all }n\ge 1;
$$
this is (ii). 

Furthermore,
$Y\cap h^i(J_n^0)=Y\cap h^{i\bmod 2^n}(J_n^0)$ for all $i\ge 0$.
Since $h(Y)=Y$, it is clear that $h^i(J_n^0\cap Y)\subset 
h^i(J_n^0)\cap Y$ and that $h^{2^n}(h^i(J_n^0)\cap Y)\subset 
h^{2^n+i}(J_n^0)\cap Y$. Thus 
$$
h^i(J_n^0\cap Y)=h^i(J_n^0)\cap Y=h^{i\bmod 2^n}(J_n^0)\cap Y\quad 
\text{for all }i\ge 0,
$$
which is (iii).
\end{proof}

The next lemma is the key tool in the proof of 
Proposition~\ref{prop:wiggins-LY-g-not-wiggins-chaotic}. 
It relies on the knowledge of the precise location of $h^i(J_n^0)$ in
$\bigcup_{k=1}^nI_n^0$.

\begin{lem}\label{lem:slope}
Let $h$ be the map defined above. Then
$\slope\left(h^{2^n-1-k}|_{h^k(J_n^0)}\right)\ge 1$
for all $n\ge 1$ and all $k\in\Lbrack 0,2^n-1\Rbrack$.
\end{lem}

\begin{proof}
A \emph{(finite) word}\index{word}\index{finite word} $B$ is an element of 
$\IN^n$ for some $n\in \IN$. If $B\in \IN^n$, the length of $B$ is 
$|B|:=n$.\index{length of a word}
\label{notation:lengthword}
If $B=b_1\ldots b_n$ and $B'=b_1'\ldots b_m'$ are two words, then
$BB'$ denotes the word obtained by concatenation, that is,\index{concatenation of words}
$$
BB':=b_1\ldots b_nb_1'\ldots b_m'\in\IN^{m+n}.
$$
An \emph{infinite word} is an element of $\IN^{\IN}$.\index{infinite word}

We define inductively a sequence of words $(B_n)_{n\ge 1}$ by:
\begin{itemize}
\item $B_1:=1$,
\item $B_n:=nB_1B_2\ldots B_{n-1}$,
\end{itemize}
and we define the infinite word $\bar\alpha=(\alpha(i))_{i\ge 1}$ by
concatenating the $B_n$'s: 
$$
\bar\alpha:=B_1B_2B_3\ldots B_n\ldots.
$$
A straightforward induction shows that
$|B_n|=2^{n-1}$; thus $|B_1|+|B_2|+\cdots+|B_k|=2^k-1$ and, in $\bar\alpha$,
the word $B_{k+1}$ starts at the index $2^k$, which gives
\begin{gather}
\label{eq:omega(2k)}
\alpha(2^k)=k+1,\\
\label{eq:omega-repeat}
\alpha(2^k+1)\ldots\alpha(2^{k+1}-1)=B_1\ldots B_k=
\alpha(1)\ldots\alpha(2^k-1).
\end{gather}

We prove by induction on $k\ge 1$ that
\begin{equation}
\label{eq:gi(Jn)-omega}
h^{i-1}(J_n^0)\subset I_{\alpha(i)}^0\quad\text{for all } n\ge k
\text{ and all } i\in\Lbrack 1,2^k-1\Rbrack.
\end{equation}

\noindent $\bullet$
Case $k=1$: $J_n^0\subset I_1^0=I_{\alpha(1)}^0$ for all $n\ge 1$.

\noindent $\bullet$
Suppose that \eqref{eq:gi(Jn)-omega} holds for $k$ and let
$n\ge k+1$. Since  $J_n^0\subset J_{k+1}^0$, 
Lemma~\ref{lem:wiggins-LY-summary-Jn0}(iii) implies that
$h^{2^k-1}(J_n^0)\subset I_{k+1}^0$, and thus
$h^{2^k}(J_n^0)\subset J_{k+1}^1\subset J_k^0$ by
Lemma~\ref{lem:wiggins-LY-summary-Jn0}(i). According to the
induction hypothesis, we have $h^{i-1}(J_k^0)\subset I_{\alpha(i)}^0$ for all
$i\in\Lbrack 1,2^k-1\Rbrack$, and \eqref{eq:omega-repeat} yields
$\alpha(i)=\alpha(2^k+i)$ for all $i\in\Lbrack 1,2^k-1\Rbrack$. Consequently,
$h^{2^k+i-1}(J_n^0)\subset I_{\alpha(2^k+i)}^0$ for all $i\in\Lbrack 1,
2^k-1\Rbrack$. Together with the induction hypothesis, this gives
\eqref{eq:gi(Jn)-omega} for $k+1$.

\medskip
Let $\mu_n:=\slope(h|_{I_n^0})$. By definition of $h$, we have
$$
\mu_n=\frac{\slope(\vfi_n)}{\prod_{i=1}^{n-1}\slope(\vfi_i)}.
$$
It is straightforward from \eqref{eq:gi(Jn)-omega} that
\begin{equation}\label{eq:slope-gk}
\forall k\in\Lbrack 2,2^n-1\Rbrack,\
\slope (h^{k-1}|_{J_n^0})=\prod_{i=1}^{k-1} \mu_{\alpha(i)}.
\end{equation}
By Lemma~\ref{lem:wiggins-LY-summary-Jn0}(ii)+(iii), 
the map $h^{2^n-1}|_{J_n^0}$
is linear and $h^{2^n-1}(J_n^0)=I_n^1$. Thus
$$
\slope(h^{2^n-1}|_{J_n^0})=\frac{|I_n^1|}{|J_n^0|}
=\prod_{i=1}^n\frac{1-\frac{2}{3^i}}{\frac{1}{3^i}}.
$$
Since $\slope(\vfi_i)=\frac{|I_i^1|}{|I_i^0|}=\frac{1-\frac{2}{3^i}}{\frac{1}{3^i}}$, we
get
\begin{equation}
\label{eq:slope-g2n-1}
\slope(h^{2^n-1}|_{J_n^0})=\prod_{i=1}^{2^n-1}\mu_{\alpha(i)}=
\prod_{i=1}^n \slope(\vfi_i).
\end{equation}

We show by induction on $n\ge 1$ that for all $k\in\Lbrack 1,2^n-1\Rbrack$
\begin{equation}\label{eq:prod-mui}
\prod_{i=1}^{k}\mu_{\alpha(i)}=\prod_{i=1}^n\slope(\vfi_i)^{\eps(i,k,n)}
\quad\text{for some }\eps(i,k,n)\in\{0,1\}.
\end{equation}

\noindent $\bullet$
$\mu_{\alpha(1)}=\mu_1=\slope(\vfi_1)$; this gives the case $n=1$.

\noindent $\bullet$
Suppose that \eqref{eq:prod-mui} holds for some $n\ge 1$. 
Since $\mu_{\alpha(2^n)}=\mu_{n+1}$ by \eqref{eq:omega(2k)}, we have
\begin{eqnarray*}
\prod_{i=1}^{2^n}\mu_{\alpha(i)}&=&\prod_{i=1}^{2^n-1}\mu_{\alpha(i)}\cdot
\mu_{n+1}\\
&=&\disp \prod_{i=1}^n\slope(\vfi_i)\cdot
\frac{\slope(\vfi_{n+1})}{\prod_{i=1}^n\slope(\vfi_i)} 
\quad\text{by \eqref{eq:slope-g2n-1}}\\
&=&\slope(\vfi_{n+1}).
\end{eqnarray*}
This is \eqref{eq:prod-mui} for $n+1$ and $k=2^n$ with 
$\eps(i,k,n+1)=0$ for all $i\in\Lbrack 1,n\Rbrack$ and $\eps(n+1,k,n+1)=1$.

Next, $\alpha(2^n+1)\ldots\alpha(2^{n+1}-1)=\alpha(1)\ldots
\alpha(2^n-1)$ by \eqref{eq:omega-repeat}; so, if
$k$ is in $\Lbrack 2^n+1,2^{n+1}-1\Rbrack$, then
\begin{eqnarray*}
\prod_{i=1}^k\mu_{\alpha(i)}&=&
\prod_{i=1}^{2^n}\mu_{\alpha(i)}\prod_{i=2^n+1}^k\mu_{\alpha(i)}
=\slope(\vfi_{n+1})\prod_{i=1}^{k-2^n}\mu_{\alpha(i)}\\
&=&\slope(\vfi_{n+1})\prod_{i=1}^n\slope(\vfi_i)^{\eps(i,k-2^n,n)}.
\end{eqnarray*}
That is, \eqref{eq:prod-mui} holds with 
$\eps(i,k,n+1)=\eps(i,k-2^n,n)$ for all $i\in\Lbrack 1,n\Rbrack$ and 
$\eps(n+1,k,n+1)=1$. This concludes the induction.

Equations \eqref{eq:slope-gk}, \eqref{eq:slope-g2n-1}  and 
\eqref{eq:prod-mui} imply that, for all
$k\in\Lbrack 1,2^n-1\Rbrack$,
\begin{equation}
\label{eq:slope-gk-bis}
\slope(h^k|_{J_n^0})=\prod_{i=1}^{k}\mu_{\alpha(i)}=
\prod_{i=1}^{n}\slope(\vfi_i)^{\eps(i,k,n)}\quad\text{for some }
\eps(i,k,n)\in\{0,1\}.
\end{equation}
Since 
$$
\slope\left(h^{2^n-1-k}|_{h^k(J_n^0)}\right)=
\frac{\slope(h^{2^n-1}|_{J_n^0})}{\slope(h^k|_{J_n^0})},
$$
\eqref{eq:slope-g2n-1} and \eqref{eq:slope-gk-bis}
imply that $\slope\left(h^{2^n-1-k}|_{h^k(J_n)}\right)$ is a
product of at most $n$ terms of the form $\slope(\vfi_i)$. This 
concludes the proof of the
lemma because $\slope(\vfi_i)\ge 1$ for all $i\ge 1$.
\end{proof}

\begin{prop}\label{prop:wiggins-LY-g-not-wiggins-chaotic}
Let $h$ be the map defined above. Then there exists no invariant 
set $Y$ such that $f|_Y$ is transitive and sensitive to
initial conditions.
\end{prop}

\begin{proof}
Let $Y$ be an invariant set such that $h|_Y$ is
transitive. If $Y$ has an isolated point, it is easy to see that $f|_Y$
is not sensitive to initial conditions.
We assume that $Y$ has no isolated point. 

The sets $\left(h^i(J_n^0\cap Y)\right)_{0\le i\le 2^n-1}$ are closed and,
by Lemma~\ref{lem:wiggins-LY-summary-Jn0}(v),
they are pairwise disjoint; let $\delta_n>0$
be the minimal distance between two of these sets.
If $x,x'\in Y$ and $|x-x'|<\delta_n$, then there is $i\in\Lbrack 0, 2^n-1\Rbrack$ such
that $x,x'\in h^i(J_n^0)$ and 
$h^k(x),h^k(x')\in h^{i+k\bmod{2^n}}(J_n^0)$ for all $k\ge 0$ by
Lemma~\ref{lem:wiggins-LY-transitive-set}(ii)+(iii). We set 
$$
\delta_n:=\max\{\diam (h^i(J_n^0)\cap Y)\mid i\in\Lbrack 0,2^n-1\Rbrack\}.
$$
Lemma~\ref{lem:slope} implies that $\diam (h^k(J_n^0)\cap Y)
\le \diam(h^{2^n-1}(J_n^0)\cap Y)$ for all integers $k$ in $\Lbrack 0,2^n-1\Rbrack$.
By Lemma~\ref{lem:wiggins-LY-summary-Jn0}(iii),
$ h^{2^n-1}(J_n^0)=I_n^1$; and by
Lemma~\ref{lem:wiggins-LY-transitive-set}(i),
$I_n^1\cap Y\subset [a_{2n},a]$. Thus $\delta_n\le \diam(I_n^1\cap Y)\le
a-a_{2n}$. This implies that 
$$
\lim_{n\to+\infty}\delta_n=0.
$$ 
This shows that $h|_Y$ is not sensitive to initial conditions. 
\end{proof}

Propositions \ref{prop:wiggins-LY-g-LY-chaotic} and
\ref{prop:wiggins-LY-g-not-wiggins-chaotic} show that the map $h$ is chaotic
in the sense of Li-Yorke but has no transitive sensitive subsystem.
At last this example is completed.

%************************************************************************
%Appendix
\chapter{Appendix: Some background in topology}\label{appendix}

The aim of this appendix is to recall succinctly some 
definitions and results in topology.
For details, one can refer to \cite{Bou3-f, Bou3-1, Kur1, Mun, Oxt}
(and also \cite{Rud2} for topological notions related to analysis).

\section{Complement of a set, product of sets}

\begin{defi}[complement of a set]\index{complement of a set}
\label{notation:complement}
Let $X$ be a set and $Y\subset X$. The \emph{complement} of $Y$ in $X$ is
$X\setminus Y:=\{x\in X\mid x\notin Y\}$.
\end{defi}

\begin{lem}\label{lem:compl-cup-cap}
Let $X$ be a set and $A,B\subset X$.
\begin{itemize}
\item $X\setminus (A\cup B)=(X\setminus A)\cap (X\setminus B)$,
\item $X\setminus (A\cap B)=(X\setminus A)\cup (X\setminus B)$.
\end{itemize}
These two properties remain valid for a countable union or intersection.
\end{lem}

\begin{defi}[product of sets]\index{product of sets}\index{Cartesian product of sets}
\label{notation:XtimesY}\label{notation:productXn}
Let $X_1,X_2$ be two sets. The (Cartesian) product of $X_1$ and $X_2$
is the set
$
X_1\times X_2:=\{(x_1,x_2)\mid x_1\in X_1,x_2\in X_2\}.
$
One can define similarly the product $X_1\times X_2\times\cdots \times X_n$.
When $X_1=X_2=\cdots=X_n=X$, let
$X^n$ denote $\underbrace{X\times\cdots\times X}_{n\rm\ times}$.

The set $X^{\IZ^+}$ is the countable product of copies of $X$, that is,
$$
X^{\IZ^+}:=\{(x_n)_{n\ge 0}\mid \forall n\in\IZ^+, x_n\in X\}.
$$
\end{defi}

%%%%%%%%%%%%%%%%%%%
\section{Definitions in topology}

\subsection{Distance, limit}

\begin{defi}[distance, metric space]\index{distance}\index{metric space}\index{triangular inequality}\index{d(x,y)@$d(x,y)$}
\label{notation:dxy}
Let $X$ be a set. A \emph{distance} on $X$ is a map $d\colon X\times X\to [0,+\infty)$ such that, for all $x,y,z\in X$:
\begin{itemize}
\item $d(x,y)=d(y,x)$,
\item $d(x,y)=0\Leftrightarrow x=y$,
\item $d(x,z)\le d(x,y)+d(y,z)$ \emph{(triangular inequality)}.
\end{itemize}
The set $X$ endowed with a distance is called a \emph{metric space}.
\end{defi}

The distance will be denoted by $d$ in any metric space,
except when several distances are involved.

\begin{ex}
In $\IR$, the usual distance is given by $d(x,y):=|y-x|$.

In $\IR^n$ ($n\ge 2$), there are several usual distances.
If $x=(x_1,\ldots, x_n)$ and $y=(y_1,\ldots, y_n)$ are
elements of $\IR^n$,
\begin{gather*}
d_{\infty}(x,y):=\max \{|y_i-x_i| \mid i\in\Lbrack 1,n\Rbrack\},\\
d_1(x,y):=\sum_{i=1}^n |y_i-x_i|,\\
d_2(x,y):=\sqrt{\sum_{i=1}^n (y_i-x_i)^2}\quad\text{(Euclidean distance)}.
\end{gather*}
$d_{\infty}, d_1$ and $d_2$ are three distances in $\IR^n$. They are
said to be \emph{equivalent}\index{equivalent distances} because,
for all $i,j\in\{1,2,\infty\}$, there exist positive real numbers $m,M$ 
such that
$$
\forall x,y\in\IR^n,\ m d_i(x,y)\le d_j(x,y)\le M d_i(x,y).
$$
\end{ex}

\begin{defi}[limit]\index{limit in a metric space}
Let $X$ be a metric space. A sequence $(x_n)_{n\ge 0}$ of points of $X$
converges (or tends) to $x\in X$ if $\disp\lim_{n\to+\infty}
d(x_n,x)=0$, that is,
$$
\forall \eps>0,\ \exists N\in\IN,\ \forall n\ge N,\ d(x_n,x)<\eps.
$$
Then $x$ is called the limit of  $(x_n)_{n\ge 0}$, and 
one writes $\disp\lim_{n\to+\infty}x_n=x$.
\end{defi}

%***********
\subsection{Open and closed sets, topology; limit point of a set}

\begin{defi}[open and closed balls]
\label{notation:balls}
\index{B(x,r)@$B(x,r),\overline{B}(x,r)$}\index{open ball}\index{closed ball}
Let $X$ be a metric space. If $x\in X$ and $r>0$, the open ball of center
$x$ and radius $r$ is $B(x,r):=\{y\in X\mid d(x,y)<r\}$, and the closed ball of center
$x$ and radius $r$ is $\overline{B}(x,r):=\{y\in X\mid d(x,y)\le r\}$.
\end{defi}

\begin{defi}[open and closed sets]\index{open set}\index{closed set}\index{topology}
Let $X$ be a metric space and $Y\subset X$. The set $Y$ is \emph{open}
if
$$
\forall x\in Y,\ \exists r>0,\ B(x,r)\subset Y.
$$
The set $Y$ is \emph{closed} if $X\setminus Y$ is open.

%Notice that $X$ and $\emptyset$ are both open and closed.

The family of all open sets of $X$ defines the \emph{topology} of $X$.
\end{defi}

\begin{ex}
In $\IR^n$, the three distances $d_{\infty}, d_1, d_2$ define the same
topology, that is, the same open and closed sets. The notion of
convergence of a sequence of points is also the same for these
three distances.
\end{ex}

\begin{defi}[discrete topology]\index{discrete topology}
Let $E$ be a set endowed with the distance:
$$
\forall x,y\in E,\ d(x,y):=\left\{\begin{array}{l}
1\text{ if }x\ne y,\\ 
0\text{ if }x=0.
\end{array}\right.
$$
The topology corresponding to this distance is called the \emph{discrete
topology}.
This topology is the usual topology for finite or countable sets
(e.g. $\{0,1\}$ or $\IZ$).
For the discrete topology, every singleton $\{x\}$ is both open and closed. 
\end{defi}

\begin{prop}
Let $X$ be a metric space.
\begin{itemize}
\item Any (finite or not) union of open sets is open. 
\item Any finite intersection of open sets is open.
\item Any (finite or not) intersection of closed
sets is closed. 
\item Any finite union of closed sets is closed.
\end{itemize}
\end{prop}

\begin{defi}[limit point of a set]\index{limit point of a set}
Let $X$ be a metric space and $Y\subset X$. A point $x\in X$ is
a \emph{limit point} of $Y$ if for every $r>0$, $B(x,r)\cap Y\ne \emptyset$.
Equivalently, $x$ is a limit point of $Y$ if there exists a
sequence of points of $Y$ that converges to~$x$.
\end{defi}

\begin{prop}
Let $X$ be a metric space and $Y\subset X$. The following assertions
are equivalent:
\begin{itemize}
\item the set $Y$ is closed,
\item all limit points of $Y$ belong to $Y$,
\item for every sequence $(y_n)_{n\ge 0}$ of points of $Y$, if there
exists $x\in X$ such that $\disp\lim_{n\to +\infty} y_n=x$,
then $x\in Y$.
\end{itemize}
\end{prop}

%***********
\subsection{Neighborhoods; interior, closure and boundary of a set}

\begin{defi}[neighborhood]\index{neighborhood}
Let $X$ be a metric space and $x\in X$. A \emph{neighborhood} of $x$ is
a set $U$ containing an open set $V$ such that $x\in V$. Equivalently,
$U$ is a neighborhood of $x$ if there exists $r>0$ such that
$B(x,r)\subset U$.
\end{defi}

\begin{defi}[interior, closure, boundary of a set]
\label{notation:interior}\label{notation:closure}\label{notation:boundary}
\index{closure of a set}\index{Int@$\Int{Y}$}\index{interior of a set}
\index{Bd(Y)@$\Bd{Y}$}\index{boundary of a set}
Let $X$ be a metric space and $Y\subset X$. 
\begin{itemize}
\item The \emph{interior} of $Y$,
denoted by $\Int{Y}$, is the set of points $x$ such that there exists
a neighborhood of $x$ included in $Y$.
It is the largest open set contained in $Y$.
\item
The \emph{closure} of $Y$, denoted by $\overline{Y}$, is the set of points
$x$ such that every neighborhood of $x$ meets $Y$.
It is the smallest closed set containing $Y$. Equivalently, 
$\overline{Y}$ is the set of all limit points of $Y$.
\item
The \emph{boundary} of $Y$ is $\Bd{Y}:=\overline{Y}\setminus \Int{Y}$.
\end{itemize}
\end{defi}

\begin{prop}
Let $X$ be a metric space and $A,B\subset X$ such that $A\subset B$.
Then $\Int{A}\subset \Int{B}$ and $\overline{A}\subset\overline{B}$.
\end{prop}

\begin{prop}\label{prop:complement}
Let $X$ be a metric space and $Y\subset X$.
\begin{itemize}
\item $X\setminus \overline{Y}=\Int{X\setminus Y}$,
\item $X\setminus \Int{Y}=\overline{X\setminus Y}$.
\end{itemize}
\end{prop}

%% \begin{defi}[topological space, open and closed sets]\index{topology}\index{topological space}\index{open set}\index{closed set}
%% Let $X$ be a set. A \emph{topology} on $X$ is a family $\CT$ of subsets of $X$
%% such that:
%% \begin{itemize}
%% \item if $U_i\in \CT$ for all $i\in I$ (where $I$ is any set of indices), 
%% then $\bigcup_{i\in I} U_i\in \CT$,
%% \item if $U_i\in \CT$ for all $i\in\Lbrack 1,n\Rbrack$ (where $n\in\IN$), then
%% $\bigcap_{i=1}^n U_i\in \CT$,
%% \item $\emptyset\in \CT$.
%% \end{itemize}
%% The set $X$ endowed with a topology is called a \emph{topological space}.
%% A set $U$ in $\CT$ is called an \emph{open} set (i.e., the topology $\CT$ is the
%% family of the open sets of $X$). A set $F:=X\setminus U$ with
%% $U\in\CT$ is called a \emph{closed} set.
%% \end{defi}

%***********
\subsection{Basis of open sets}

\begin{defi}[basis of open sets]\index{basis of open sets}\index{basis of the topology}
Let $X$ be a metric space. A \emph{basis of open sets} of $X$ is a family
$\CB$ of nonempty open sets of $X$ such that every open set
can be written as a (finite or not) union of elements of $\CB$. It
is also called a \emph{basis of the topology} of $X$.
\end{defi}

\begin{ex} ~

\noindent$\bullet$
In a metric space $X$, the open balls form a basis of open sets.

\noindent$\bullet$
In $\IR$, the family $\{(a,b)\mid a,b\in\IQ, a<b\}$ is a countable basis of
open sets.

\noindent$\bullet$
In a set $E$ endowed with the discrete topology, the family $(\{x\})_{x\in E}$
is a basis of open sets.
\end{ex}

%***********
\subsection{Distance between two sets, diameter}

\begin{defi}[distance between two sets]\index{distance between two sets}\index{d(x,y)@$d(A,B)$}
\label{notation:dAB}
Let $X$ be a metric space and $A,B\subset X$. The \emph{distance}
between the sets $A$ and $B$ is
$$
d(A,B):=\inf\{d(a,b)\mid a\in A, b\in B\}.
$$
\end{defi}

\begin{defi}[diameter, bounded set]
\label{notation:diam}
\index{diam(Y)@$\diam(Y)$}\index{diameter of a set}\index{bounded set}
Let $X$ be a metric space and $Y$ a nonempty subset of $X$. 
The \emph{diameter} of $Y$ is
$\diam(Y):=\sup\{d(x,y)\mid x,y\in Y\}$. 

The set $Y$ is \emph{bounded} if
there exist $x\in X$ and $r>0$ such that $Y\subset \overline{B}(x,r)$.
Equivalently, $Y$ is bounded if $\diam(Y)<+\infty$.
\end{defi}

%***********
\subsection{Dense sets, $G_{\delta}$-sets}

\begin{defi}[dense set]\index{dense set}
Let $X$ be a metric space. A set $Y\subset X$ is \emph{dense} in $X$ if
$\overline{Y}=X$. Equivalently, $Y$ is dense in $X$ if
$$
\forall x\in X,\ \forall \eps>0,\ \exists y\in Y,\ d(x,y)\le \eps.
$$
\end{defi}

\begin{defi}[$G_\delta$-set]\index{Gdelta set@$G_\delta$-set}
A \emph{$G_\delta$-set} is a countable intersection of open sets.
\end{defi}

\begin{prop}\label{prop:unionGdelta}
A countable union of $G_\delta$-sets is a $G_\delta$-set.
\end{prop}

%***********
\subsection{Borel sets}

\begin{defi}[$\sigma$-algebra]
A $\sigma$-algebra\index{$\alpha$ sc@$\sigma$-algebra}
of a set $X$ is a collection $\CA$ of subsets of $X$
such that:
\begin{itemize}
\item $\emptyset \in\CA$,
\item if $A\in \CA$, then $X\setminus A\in\CA$,
\item if $A_n\in\CA$ for all $n\ge 0$, then 
$$
\bigcup_{n\ge 0}A_n\in\CA\quad\text{and}\bigcap_{n\ge 0}A_n\in\CA.
$$
\end{itemize}
\end{defi}

\begin{defi}[Borel set]
Let $X$ be a metric space. A \emph{Borel set}
\index{Borel set}%\index{$\sigma$-algebra of Borel sets}
is any subset of $X$ that can be
formed from open sets (or, equivalently, from closed sets) through the
operations of countable unions, countable intersection and complement.
Equivalently, the family of all Borel sets is the smallest $\sigma$-algebra
containing all open sets of $X$.
\end{defi}

%%%%%%%%%%%%%%%%%%%
\section{Topology derived from the topology on $X$}

\begin{defi}[induced topology]\index{induced topology}
Let $X$ be a metric space and  $Y\subset X$. The 
restriction of the distance $d$ to $Y\times Y$ is a distance on $Y$, and
the topology given by this distance is called the \emph{induced topology} 
on $Y$. Equivalently, a set $A\subset Y$ is open (resp. closed) for the
induced topology on $Y$ if there exists an open (resp. closed) set 
$A'\subset X$ such that $A=A'\cap Y$.
\end{defi}

\begin{defi}[product topology]\index{product topology}\label{defi:producttopology}
Let $X_1,X_2$ be two metric spaces. The \emph{product topology} on
$X_1\times X_2$ is generated by the basis of open sets of the form
$U_1\times U_2$, where $U_i$ is a nonempty open set of $X_i$ for $i\in\{1,2\}$.

If the distances in $X_1, X_2$ are respectively $d_1, d_2$,
one can define a distance $d_{\infty}$ on $X_1\times X_2$ by
$$
d_{\infty}((x_1,x_2), (y_1,y_2))=\max(d_1(x_1,y_1),d_2(x_2,y_2)),
$$
and the product topology on $X_1\times X_2$ is 
the topology given by this distance.

One can define similarly the product topology on $X_1\times X_2\times\cdots
\times X_n$.
\end{defi}

\begin{defi}[product topology on $X^{\IZ^+}$]\index{product topology}
Let $X$ be a metric space. The \emph{product topology} on
$X^{\IZ^+}$ is generated by the basis of open sets of the form
$$
U_0\times U_1\times\cdots \times U_{k-1}\times X^{n\ge k}:=
\{(x_n)_{n\ge 0}\in X^{\IZ^+}\mid \forall n\in\Lbrack 0, k-1\Rbrack, x_n\in U_n\},
$$
where $U_0,\ldots, U_k$ are nonempty open sets of $X$ and $k\in\IN$.
If $d_X$ denotes the distance on $X$, one can define a distance $d$
on $X^{\IZ^+}$ by
$$
d((x_n)_{n\ge 0}, (y_n)_{n\ge 0}):=\sum_{n=0}^{+\infty}
\frac{d_X(x_n,y_n)}{2^n}
$$
(if $\diam(X)=+\infty$, one should replace $d_X(x_n,y_n)$ by
$\min(d_X(x_n,y_n),1)$).

The product topology on $X^{\IZ^+}$ is the topology given by this distance.
\end{defi}

\begin{ex}
Let $E:=\{0,1\}$ endowed with the discrete topology. The set
$\{0,1\}^{\IZ^+}$ is the set of all infinite sequences of $0$ and $1$.
This is a metric space. The family
$$
\{(x_n)_{n\ge 0}\in\{0,1\}^{\IZ^+}\mid \forall n\in\Lbrack 0,k-1\Rbrack, x_n=a_n\},
\text{ where } k\in\IN, a_0,\ldots,a_{k-1}\in\{0,1\},
$$
is a countable basis of open sets of $\{0,1\}^{\IZ^+}$.
\end{ex}

%%%%%%%%%%%%%%%%%%%
\section{Connectedness, intervals}

\begin{defi}[connected set]\index{connected set}\index{disconnected set}
Let $X$ be a metric space. A set $Y\subset X$ is \emph{disconnected}
if there exist disjoint open sets $U,V\subset X$ such that $Y\subset U\cup V$,
$Y\cap U\ne\emptyset$ and $Y\cap V\ne\emptyset$.
Otherwise $Y$ is called \emph{connected}.
\end{defi}

\begin{prop}\label{prop:union-connected}
Let $X$ be a metric space and let $(C_i)_{i\in \CI}$ be a (finite or infinite) 
family of connected sets in $X$.
If there exists a point $x$ such that $x\in C_i$ for all
$i\in \CI$, then  $\bigcup_{i\in \CI} C_i$ is a connected set.
\end{prop}

\begin{defi}[connected component]\index{connected component}
Let $X$ be metric space, $Y\subset X$ and $y\in Y$. 
The \emph{connected component} of $y$ in $Y$ is the largest (for inclusion) 
connected set $C$ containing $y$ such that $C\subset Y$. 
The connected components of two
points are either equal or disjoint. The connected components
of all points of $Y$ are called the \emph{connected components of $Y$}.
\end{defi}

\begin{defi}[interval]\index{interval}\index{real interval}\index{subinterval}
A \emph{(real) interval} $I$ is a subset of $\IR$ of one of the following forms:
\begin{itemize}
\item $I=[a,b]=\{x\in\IR\mid a\le x\le b\}$ with $a,b\in\IR$, $a\le b$ 
(if $a=b$, then $I=\{a\}$),
\item $I=(a,b)=\{x\in\IR\mid a<x<b\}$ with $a\in\IR\cup\{-\infty\},b\in\IR\cup\{+\infty\}$, $a\le b$ (if $a=-\infty$ and $b=+\infty$, then $I=\IR$; if $a=b$, then $I=\emptyset$),
\item $I=[a,b)=\{x\in\IR\mid a\le x< b\}$ with $a\in\IR,b\in\IR\cup\{+\infty\}$, $a<b$,
\item $I=(a,b]=\{x\in \IR\mid a<x\le b\}$ with $a\in\IR\cup\{-\infty\},b\in\IR$, $a<b$.
\end{itemize}
If $I$ is an interval, a \emph{subinterval} of $I$ is an interval included
in $I$.
\end{defi}

\begin{theo}
When $X$ is a real interval, the connected sets in $X$
are exactly the subintervals of $X$. In particular, the connected sets
in $\IR$ are exactly the intervals.
\end{theo}

\begin{prop}
Let $(I_n)_{n\ge 0}$ be a (finite or infinite) sequence of intervals in
$\IR$. Then 
\begin{itemize}
\item $\bigcap_{n\ge 0} I_n$ is an interval (maybe empty).
\item If there exists a point $x$ such that $x\in I_n$ for all
$n\ge 0$, then  $\bigcup_{n\ge 0} I_n$ is an interval containing $x$
(this is a particular case of Proposition~\ref{prop:union-connected}).
\end{itemize}
\end{prop}

%%%%%%%%%%%%%%%%%%%
\section{Compactness}

%%%%
\subsection{Definition and equivalent conditions}

\begin{defi}[open cover]\index{open cover}
An \emph{open cover} of a metric space $X$ is a family of open sets
$(U_i)_{i\in \CI}$ such that $X=\bigcup_{i\in \CI} U_i$.
\end{defi}

Notice that the set of indices is arbitrary in the previous definition.
For example, if $r>0$, $(B(x,r))_{x\in X}$ is an open cover
of $X$.

\begin{defi}[compact set]\index{compact set}
A metric space $X$ is \emph{compact} if every open cover $(U_i)_{i\in \CI}$
of $X$ admits a finite subcover, that is, there is a finite
set of indices $J\subset \CI$ such that $X=\bigcup_{i\in J}U_i$.

A subset $Y\subset X$ is \emph{compact} if $Y$ is compact for the induced
topology on $Y$.
\end{defi}

\begin{defi}[subsequence, limit point of a sequence]\index{subsequence}\index{limit point of a sequence}
Let $X$ be a metric space and $(x_n)_{n\ge 0}$ a sequence of points of
$X$. A \emph{subsequence} of $(x_n)_{n\ge 0}$ is a sequence
of the form $(x_{n_i})_{i\ge 0}$, where $(n_i)_{i\ge 0}$ is an increasing
sequence of non negative numbers.
A point $x\in X$ is a \emph{limit point} of $(x_n)_{n\ge 0}$
if there exist a subsequence of $(x_n)_{n\ge 0}$ that converges to $x$.
\end{defi}

\begin{theo}[Bolzano-Weierstrass theorem]\index{Bolzano-Weierstrass theorem}
A metric space $X$ is compact if and only if every sequence
$(x_n)_{n\ge 0}$ of points of $X$ admits a limit point.
\end{theo}

\begin{theo}\label{theo:closedsubsetcompact}
Let $X$ be a compact metric space. A subset $Y\subset X$ is compact if
and only if $Y$ is closed in $X$.
\end{theo}

\begin{prop}
A set $X\subset \IR^n$ ($n\in\IN$) is compact if and only if $X$ is
closed and bounded for the distance $d_{\infty}$ (or
equivalently for $d_1$ or $d_2$).
\end{prop}

%%%
\subsection{product, intersection of compact sets}

\begin{theo}\label{theo:product-of-compacts}
Let $X_1,X_2$ be compact metric spaces. Then $X_1\times X_2$ is a compact
metric space.
\end{theo}

\begin{prop}\label{prop:cap-compact}
Let $X$ be a metric space.
Let $(Y_n)_{n\ge 0}$ be a sequence of nonempty compact subsets of $X$ such that
$Y_{n+1}\subset Y_n$ for all $n\ge 0$. Then
$\bigcap_{n=0}^{+\infty} Y_n$ is a nonempty compact set.
If in addition $\lim_{n\to +\infty} \diam(Y_n)=0$, then 
$\bigcap_{n=0}^{+\infty} Y_n$ is a singleton (i.e., it contains a single point).
\end{prop}

%%%
\subsection{Cauchy sequence, complete space}

\begin{defi}[Cauchy sequence]\index{Cauchy sequence}
Let $X$ be a metric space and $(x_n)_{n\ge 0}$ be a sequence of points
of $X$. Then $(x_n)_{n\ge 0}$ is a \emph{Cauchy sequence} if
$$
\forall \eps>0,\ \exists N\ge 0,\ \forall n>m\ge N,\ d(x_n,x_m)<\eps.
$$
\end{defi}

\begin{prop}
Let $X$ be a metric space. If $(x_n)_{n\ge 0}$ is a sequence of points
of $X$ that converges, then it is a Cauchy sequence.
\end{prop}

\begin{defi}[complete space]\index{complete space}
A metric space $X$ is \emph{complete} if every Cauchy sequence converges.
\end{defi}

\begin{prop}
A compact metric space is complete.
\end{prop}

%%%
\subsection{Countable basis of open sets}

\begin{prop}\index{countable basis of open sets}
A compact metric space admits a countable basis of open sets, that is, there 
exists a family $(U_n)_{n\in\IN}$ of nonempty open sets of $X$ such that, 
for every open set
$U\subset X$, there exists $\CI\subset \IN$ such that $U=\bigcup_{n\in \CI}U_n$.
The sets $U_n$ can be chosen to be open balls.
\end{prop}

%%%
\subsection{Lebesgue number}

\begin{prop}[Lebesgue's number Lemma]\index{Lebesgue number}
Let $X$ be a compact metric space and $(U_i)_{i\in \CI}$ an open cover of
$X$. There exists $\delta>0$ such that
$$
\forall x\in X,\ \exists i\in \CI,\ B(x,\delta)\subset U_i.
$$
Such a number $\delta$ is called a \emph{Lebesgue number} of this cover.
\end{prop}

%%%
\subsection{Baire category theorem}

\begin{theo}[Baire category theorem]\label{theo:baire}\index{Baire category theorem}
Let $X$ be a complete metric space and $(U_n)_{n\ge 0}$ a sequence of
dense open sets. Then $\bigcap_{n\ge 0} U_n$ is a dense
$G_\delta$-set.
\end{theo}

\begin{cor}\label{cor:baire}
Let $X$ be a nonempty complete metric space and $(F_n)_{n\ge 0}$ a sequence of 
closed sets such that $X=\bigcup_{n\ge 0}F_n$. Then there exists an integer 
$n\ge 0$ such that $\Int{F_n}\ne\emptyset$.
\end{cor}

\begin{proof}
Suppose on the contrary that $\Int{F_n}=\emptyset$ for all $n\ge 0$.
We set $U_n:=X\setminus F_n$. By Lemma~\ref{lem:compl-cup-cap},
$$
X\setminus \bigcap_{n\ge 0}U_n=\bigcup_{n\ge 0} (X\setminus U_n)=
\bigcup_{n\ge 0} F_n=X,
$$
which implies that $\bigcap_{n\ge 0}U_n$ is empty.
On the other hand, $\overline{U_n}=X\setminus \Int{F_n}=X$ by
Proposition~\ref{prop:complement}, and thus
$U_n$ is a dense open set for every $n\ge 0$. Therefore
$\bigcap_{n\ge 0} U_n$ is dense according to
the Baire category Theorem~\ref{theo:baire}, which is a contradiction.
We conclude that there exists $n\ge 0$ such that $\Int{F_n}\ne\emptyset$.
\end{proof}

\begin{cor}\label{cor:baire2}
Let $X$ be a complete metric space and $(G_n)_{n\ge 0}$ a sequence of
dense $G_\delta$-sets. Then $\bigcap_{n\ge 0} G_n$ is a dense
$G_\delta$-set.
\end{cor}

\begin{proof}
For every $n\ge 0$, one can write $G_n=\bigcap_{k\ge 0}U_n^k$, where 
$U_n^k$ is an open set. Since $G_n$ is dense and $G_n\subset U_n^k$ for
all $k\ge 0$, the sets $(U_n^k)_{n,k\ge 0}$ are dense open sets.
The Baire category Theorem \ref{theo:baire} states
that $\bigcap_{n,k\ge 0}U_n^k$ is a dense $G_\delta$-set. 
Finally, we have $\bigcap_{n\ge 0} G_n=\bigcap_{n,k\ge 0}U_n^k$.
\end{proof}

%%%%%%%%%%%%%%%%%%%
\section{Cantor set}\label{sec:cantor}

%%%
\subsection{Definitions}

\begin{defi}[isolated point]\index{isolated point}
Let $X$ be a metric space and $Y\subset X$. A point $y\in Y$ is an
\emph{isolated point} in $Y$ if $B(y,r)\cap Y=\{y\}$ for some $r>0$.
If there is no such point in $Y$, one says that $Y$ has no isolated point.
\end{defi}

%% \begin{defi}[perfect set]\index{perfect set}
%% Let $X$ be a metric space and $Y\subset X$.
%% The set $Y$ is called a \emph{perfect set} if is closed and has no 
%% isolated point.
%% \end{defi}

\begin{defi}
Let $X$ be a metric space and $Y\subset X$. The set $Y$ is said to be
\emph{totally disconnected} if for every $y\in Y$, the connected component
of $y$ in $Y$ is reduced to $\{y\}$.
\end{defi}

\begin{defi}[Cantor set]\index{Cantor set}
Let $X$ be a metric space. The set $X$ is a \emph{Cantor
set} if it is nonempty, compact, totally disconnected and has no isolated
point.
\end{defi}

\begin{prop}
Let $X$ be a complete metric space and $Y\subset X$. If $Y$ is
nonempty, closed and has no isolated point, then $Y$ is uncountable.
In particular, a Cantor set is uncountable.
\end{prop}

%%%
\subsection{Examples of Cantor sets}

\begin{ex}
The set $\{0,1\}^{\IZ^+}$, endowed with the product topology given
by the discrete topology on $\{0,1\}$, is a Cantor set.
\end{ex}

\begin{ex}[triadic Cantor set]\label{ex:triadicCantor}
We are going to build by induction on $n$ a family of intervals
$(I_{\alpha_0\ldots\alpha_{n-1}})_
{(\alpha_0,\ldots,\alpha_{n-1})\in\{0,1\}^n}$ in $[0,1]$
such that, for all $n\in\IN$,
\begin{itemize}
\item $(I_{\alpha_0\ldots\alpha_{n-1}})_
{(\alpha_0\ldots\alpha_{n-1})\in\{0,1\}^n}$ are pairwise disjoint closed
intervals,
\item $|I_{\alpha_0\ldots\alpha_{n-1}}|=\frac{1}{3^n}$
for all $\alpha_0,\ldots,\alpha_{n-1}\in\{0,1\}$,
\item $I_{\alpha_0\ldots\alpha_{n-1}\alpha_n}\subset
I_{\alpha_0\ldots\alpha_{n-1}}$ for all $\alpha_0,\ldots,\alpha_n\in\{0,1\}$.
\end{itemize}
See Figure~\ref{fig:triadic-cantor} for the first steps of the construction.

\medskip
\noindent$\bullet$ At step $0$, we start with $I=[0,1]$.

\noindent$\bullet$ At step $1$, we cut $I$ in three equal parts and we remove
the open interval $(\frac13,\frac23)$ in the middle. There remain
two intervals $I_0:=[0,\frac13]$ and $I_1:=[\frac23,1]$.

$\bullet$ At step $n\ge 1$, we cut every interval
$I_{\alpha_0\ldots\alpha_{n-1}}$ into three equal parts and we remove
the open third in the middle. If $I_{\alpha_0\ldots\alpha_{n-1}}=[a,b]$,
there remain two intervals
$I_{\alpha_0\ldots\alpha_{n-1}0}:=[a,\frac{2a+b}3]$ and
$I_{\alpha_0\ldots\alpha_{n-1}1}:=[\frac{a+2b}3,b]$.
Trivially, these two intervals are disjoint, they are included in
$I_{\alpha_0\ldots\alpha_{n-1}}$, and their length is equal
to $\frac13|I_{\alpha_0\ldots\alpha_{n-1}}|$, and so their length is
$\frac{1}{3^n}$ by the induction hypothesis. 
Moreover, the sets
$(I_{\alpha_0\ldots\alpha_n})_{(\alpha_0,\ldots,\alpha_n)\in\{0,1\}^{n+1}}$
are pairwise disjoint by construction.

\begin{figure}[htb]
\centerline{\includegraphics{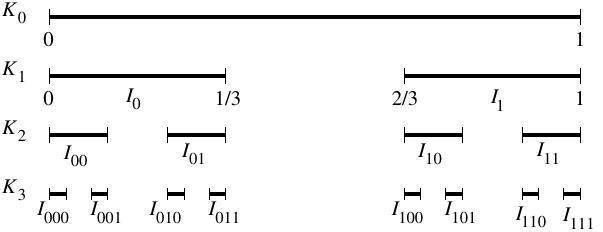}}
\caption{The first steps of the construction of the triadic Cantor set.}
\label{fig:triadic-cantor}
\end{figure}

\medskip
Let $K_0:=[0,1]$ and, for all $n\ge 1$,
let $K_n$ be the union of the intervals $I_{\alpha_0\ldots\alpha_{n-1}}$
for all $(\alpha_0,\ldots,\alpha_{n-1})\in\{0,1\}^n$. Then $K_n$ is a
compact set with $2^n$ connected components of length $\frac{1}{3^n}$ and
$K_{n+1}\subset K_n$. Let
$K:=\bigcap_{n\ge 0} K_n$. Then $K$ is a nonempty compact set
by Proposition~\ref{prop:cap-compact}. One can show that
$K$ is a Cantor set. 
The set $K$ is called the \emph{triadic Cantor set}.\index{triadic Cantor set}
\end{ex}

\begin{ex}
On can construct other sets in a similar way as in
Example~\ref{ex:triadicCantor},  by varying the size and/or the number
of the gaps. More precisely,

\noindent$\bullet$ At step 0, we start with a non degenerate compact
interval $K_0$.

\noindent$\bullet$ At step $n\ge 1$, for every connected component
$C$ of $K_{n-1}$, we choose $p$ disjoint non degenerate closed
subintervals of $C$ (with $p=p(C)\ge 2$) 
such that one contains $\min C$ and another one contains
$\max C$. We call $K_n$ the union of all these intervals.

Finally, $K:=\bigcap_{n\ge 0}K_n$ is a nonempty compact set with
uncountably many connected components and no isolated point.
All the Cantor sets included in $\IR$ can be obtained by this construction.
But notice that the sets obtained in this way are not all Cantor sets.
Indeed, let $\ell_n:=\max\{ |C|\mid C\text{ connected component of }K_n\}$;
then $K$ is a Cantor set if and only if $\lim_{n\to+\infty}\ell_n=0$,
otherwise $K$ has a non degenerate connected component (e.g., the set $K$
built in Example~\ref{ex:chaos-LY-htop0} is not a Cantor set;
see Lemma~\ref{lem:LY-htop0-non-separable-K}).

A classical family of Cantor sets in $\IR$ is obtained by fixing
a ratio $r\in(0,1/2)$ and constructing the intervals such that
all connected components of $K_n$ have the same length $\ell_n$
with $\ell_n=r\ell_{n-1}$ (e.g., $r=\frac{1}{2p+1}$ in 
Example~\ref{ex:sensitive-not-transitive}).
\end{ex}

\begin{prop}
Let $\IR$ be the ambient space and $X$ a Cantor set included in $\IR$.
Then $\Int{X}=\emptyset$.
\end{prop}

\begin{proof}
Every nonempty open set in $\IR$ contains a nonempty open interval. On 
the other hand, every connected component of $X$ is reduced to
a single point. This implies that $\Int{X}=\emptyset$.
\end{proof}

%%%
%\subsection{All Cantor sets are homeomorphic}

\begin{theo}
Every Cantor set $K$ is homeomorphic to $\{0,1\}^{\IZ^+}$, that is,
there exists a homeomorphism $\vfi\colon X\to \{0,1\}^{\IZ^+}$
(see Definition~\ref{defi:homeo} below for the definition of
homeomorphism).
\end{theo}

%%%%%%%%%%%%%%%%%%%
\section{Continuous maps}

\subsection{Definitions}

\begin{defi}[image, preimage of a set]\index{image of a set under a map}\index{preimage of a set under a map}\index{inverse image of a set under a map}
Let  $f\colon X\to Y$ be a map. The \emph{image} of a set $A\subset X$
under $f$ is $f(A):=\{f(x)\mid x\in A\}$. The \emph{preimage} 
(or \emph{inverse image}) of a set $B\subset Y$ under $f$
is $f^{-1}(B):=\{x\in X\mid f(x)\in B\}$.
\end{defi}

\begin{defi}[continuity]\index{continuous map}
Let $X,Y$ be metric spaces endowed with the distances $d_X,d_Y$ respectively. 
A map $f\colon X\to Y$ is \emph{continuous}
if one of the following equivalent assertions is satisfied:
\begin{enumerate}
\item 
$\disp
\forall x\in X, \forall \eps>0, \exists \delta>0,\ \forall x'\in X,
d_X(x,x')<\delta\Rightarrow d_Y(f(x),f(x'))<\eps,
$
\item for all open sets $U\subset Y$, $f^{-1}(U)$ is open,
\item for all closed sets $F\subset Y$, $f^{-1}(F)$ is closed,
\item for all $x\in X$ and all sequences $(x_n)_{n\ge 0}$ of points of $X$
converging to $x$, $\disp\lim_{n\to+\infty}f(x_n)=f(x)$.
\end{enumerate}
\end{defi}

\begin{rem}
Let $X,Y$ be metric spaces and $f\colon X\to Y$ a continuous map.
If $X'\subset X$, then $f|_{X'}\colon X'\to Y$ (restriction of $f$ to $X'$)
is also a continuous map.
\end{rem}

\begin{prop}
Let $X,Y,Z$ be metric spaces. If $f\colon X\to Y$ and $g\colon Y\to Z$
are continuous maps, then $g\circ f\colon X\to Z$ is a continuous map.
In particular, if $f\colon X\to X$ is a continuous map, then
$f^n\colon X\to X$ is a continuous map for every $n\in\IN$,
where $\disp
f^n:=\underbrace{f\circ f\circ\cdots\circ f}_{n \text{ \scriptsize
times}}.
$
\end{prop}

\begin{defi}[one-to-one and onto map, bijection]\index{one-to-one map}\index{injective map}\index{onto map}\index{surjective map}\index{bijection, bijective map}
Let $X,Y$ be metric spaces and let $f\colon X\to Y$ be a map.
\begin{itemize}
\item $f$ is \emph{one-to-one} (or \emph{injective}) if, for all $x,x'\in X$,
$x\ne x'\Rightarrow f(x)\ne f(x')$.
\item $f$ is \emph{onto} (or \emph{surjective}) if $f(X)=Y$.
\item $f$ is a \emph{bijection} (or a \emph{bijective map}) 
if it is one-to-one and onto. In this case, the \emph{inverse map} of $f$
is the map $f^{-1}\colon Y\to X$ satisfying $f(x)=y\Leftrightarrow x=f^{-1}(y)$.
\end{itemize}
\end{defi}

\begin{prop}
Let $f\colon I\to J$ be a continuous onto map, 
where $I,J$ are two nonempty real intervals.
Then $f$ is a bijection if and only if
\begin{itemize}
\item either $f$ is increasing, that is, $\forall x,y\in I$,
$x<y\Rightarrow f(x)<f(y)$,
\item or $f$ is decreasing, that is, $\forall x,y\in I$,
$x<y\Rightarrow f(x)>f(y)$.
\end{itemize}
\end{prop}

\begin{defi}[homeomorphism]\label{defi:homeo}\index{homeomorphism}
Let $X,Y$ be two metric spaces. A map $f\colon X\to Y$ is a \emph{homeomorphism}
if $f$ is continuous, bijective and $f^{-1}$ is continuous.
\end{defi}

\begin{defi}[uniform continuity]\index{uniformly continuous map}
Let $X,Y$ be metric spaces endowed with the distances $d_X,d_Y$. 
A map $f\colon X\to Y$ is \emph{uniformly continuous} if
$$
\forall \eps>0,\exists \delta>0, \forall x,x'\in X,
d_X(x,x')<\delta\Rightarrow d_Y(f(x),f(x'))<\eps.
$$
\end{defi}

%%%
\subsection{Inverse image of an intersection}

\begin{prop}\label{prop:f-1-cap}
Let $X,Y$ be metric spaces and let $f\colon X\to Y$ be a map.
If $(Y_n)_{n\ge 0}$ is a family of subsets of $Y$, then
$f^{-1}(\bigcap_{n\ge 0}Y_n)=\bigcap_{n\ge 0}f^{-1}(Y_n)$.
\end{prop}

\begin{proof}
A point $x$ belongs to $f^{-1}(\bigcap_{n\ge 0}Y_n)$ if and only if
$f(x)\in\bigcap_{n\ge 0}Y_n$, that is, $f(x)\in Y_n$ for all $n\ge 0$.
Since $f(x)\in Y_n\Leftrightarrow x\in f^{-1}(Y_n)$, we get
$$
x\in f^{-1}(\bigcap_{n\ge 0}Y_n)\Longleftrightarrow
x\in \bigcap_{n\ge 0}f^{-1}(Y_n).
$$
\end{proof}

\begin{prop}\label{prop:f-1-Gdelta}
Let $X,Y$ be metric spaces and let $f\colon X\to Y$ be a continuous map.
If $G\subset Y$ is a $G_\delta$-set, then $f^{-1}(G)$ is a $G_\delta$-set
in $X$.
\end{prop}

\begin{proof}
One can write $G=\bigcap_{n\ge 0} U_n$, where $U_n$ is an open set of
$Y$ for every $n\ge 0$. Then $f^{-1}(U_n)$ is an open set of $X$ because
$f$ is continuous, and $f^{-1}(G)=\bigcap_{n\ge 0}f^{-1}(U_n)$ by
Proposition~\ref{prop:f-1-cap}. Therefore, $f^{-1}(G)$ is a $G_\delta$-set.
\end{proof}

%%%
\subsection{Continuity and denseness}

\begin{theo}
Let $X, Y$ be metric spaces and $D$ a dense subset of $X$.
Let $f\colon X\to Y, g\colon X\to Y$ be two continuous maps.
If $f(x)=g(x)$ for all $x\in D$, then $f(x)=g(x)$ for all $x\in X$.
\end{theo}

%%%
\subsection{Continuity and connectedness}

\begin{theo}\label{theo:f(connected)}
Let $X, Y$ be metric spaces and $f\colon X\to Y$ a continuous map.
If $C\subset X$ is a connected set, then $f(C)$ is connected.
\end{theo}

The \emph{intermediate value theorem} is a corollary of 
Theorem~\ref{theo:f(connected)} for real maps. See Theorem~\ref{theo:ivt}
in Chapter~\ref{chap1}.

%%%
\subsection{Continuity and compactness}

\begin{theo}\label{theo:f(compact)}
Let $X, Y$ be metric spaces and $f\colon X\to Y$ a continuous map.
If $K\subset X$ is a compact set, then $f(K)$ is compact.
\end{theo}

\begin{prop}\label{prop:f(clA)}
Let $X, Y$ be metric spaces with $X$ compact, $f\colon X\to Y$ a continuous map
and $A\subset X$. Then $f(\overline{A})=\overline{f(A)}$.
\end{prop}

\begin{proof}
The set $\overline{A}$ is compact by Theorem~\ref{theo:closedsubsetcompact},
and thus $f(\overline{A})$ is compact by Theorem~\ref{theo:f(compact)}.
Trivially, $f(A)\subset f(\overline{A})$, which implies that
$\overline{f(A)}\subset f(\overline{A})$.

Let $x$ be a point in $\overline{A}$. Then there exists $(x_n)_{n\ge 0}$
a sequence of points of $A$ that converges to $x$. Since
$f$ is continuous, $\lim_{n\to+\infty}f(x_n)=f(x)$. This
implies that $f(x)\in \overline{f(A)}$, and hence
$f(\overline{A})\subset \overline{f(A)}$. We conclude that
$f(\overline{A})=\overline{f(A)}$.
\end{proof}

\begin{prop}
Let $X, Y$ be metric spaces and $f\colon X\to Y$ a continuous bijection.
If $X$ is compact, then $f$ is a homeomorphism.
\end{prop}

\begin{prop}
Let $X, Y$ be metric spaces and $f\colon X\to Y$ a continuous map. If
$X$ is compact, then $f$ is uniformly continuous.
\end{prop}

\begin{theo}\label{theo:max-on-compact}
Let $f\colon X\to \IR$ be a continuous map, where $X$ is a compact metric
space. Then $f$ admits a maximum and a minimum, that is,
\begin{gather*}
\exists x_M\in X,\ f(x_M)=\sup\{f(x)\mid x\in X\},\\
\exists x_m\in X,\ f(x_m)=\inf\{f(x)\mid x\in X\}.
\end{gather*}
\end{theo}

\begin{cor}
Let $X$ be a metric space and $A,B\subset X$. If $A,B$ are compact, then
$d(A,B)$ and $\diam(A)$ are reached, that is,
\begin{itemize}
\item there exist $a\in A,b\in B$ such that $d(A,B)=d(a,b)$,
\item there exist $a,a'\in A$ such that $\diam(A)=d(a,a')$.
\end{itemize}
\end{cor}

\begin{proof}
Let $f\colon X\times X\to\IR$ defined by $f(x,y):=d(x,y)$.
One can easily show that $f$ is continuous. 

The set $A\times B$ is compact by Theorem~\ref{theo:product-of-compacts}.
Thus $f|_{A\times B}$ admits a minimum by 
Theorem~\ref{theo:max-on-compact}, that is, there
exists a couple of points $(a,b)$ in $A\times B$ such that 
$d(a,b)=\inf\{d(x,y)\mid (x,y)\in A\times B\}$; the last
expression is the definition of $d(A,B)$.

Similarly, the set $A\times A$ is compact, and thus $f|_{A\times A}$
admits a maximum, that is, there exist $(a,a')\in A\times A$ such
that $d(a,a')=\sup\{d(x,y)\mid (x,y)\in A\times A\}$; the last
expression is the definition of $\diam(A)$.
\end{proof}

%%%%
\subsection{Uniform convergence of a sequence of real maps}

\begin{defi}\index{uniform distance}\index{uniform convergence of a sequence of maps}
Let $\CF$ be the space of all maps $f\colon [0,1]\to [0,1]$.
The \emph{uniform distance} on $\CF$ is defined by
$d_{\infty}(f,g):=\sup\{|g(x)-f(x)|\mid x\in [0,1]\}$, where $f,g\in\CF$.

Let $(f_n)_{n\ge 0}$ be a sequence of maps of $\CF$. Then $(f_n)_{n\ge 0}$
\emph{uniformly converges} to $f\in\CF$ if it converges to $f$ for the distance
$d_{\infty}$, that is,
$$
\forall \eps>0, \exists N\ge 0, \forall n\ge N, 
\forall x\in [0,1],\ |f_n(x)-f(x)|<\eps.
$$
\end{defi}

\begin{theo}
The space $\CF$ endowed with the distance $d_{\infty}$ defined above
is a complete space.
\end{theo}

%%%%%%%%%%%%%%%%%%%
\section{Zorn's Lemma}

A \emph{partially ordered} set is a set endowed with a binary relation
that indicates that, for certain pairs of elements, one of the 
elements precedes the other. Such a relation is called a \emph{partial order} 
to reflect the fact that not every pair of elements need be related, contrary
to a \emph{total order}.
The formal definitions are given below.

\begin{defi}[partial and total order]\index{partial order/partially ordered set}\index{total order/totally ordered set}
A \emph{partial order} on the set $E$ is a binary relation $\le$ such that,
for all $a,b,c\in E$,
\begin{itemize}
\item $a\le a$,
\item if $a\le b$ and $b\le a$, then $a=b$,
\item if $a\le b$ and $b\le c$, then $a\le c$.
\end{itemize}
Such a set $E$ is called \emph{partially ordered}.
If $a\le b$ or $b\le a$, the elements $a,b$ are said to be \emph{comparable}.

A \emph{total order} on $E$ is a partial order such that all pairs of
elements are comparable.
Such a set $E$ is called \emph{totally ordered}.
\end{defi}

\begin{defi}[lower and upper bound]\index{lower bound}\index{upper bound}
Let $E$ be a partially ordered set and $F\subset E$. 
An element $b\in E$ is a \emph{lower bound} (resp. \emph{upper bound})
of $F$ if $x\ge b$ (resp. $x\le b$) for all $x\in F$.
\end{defi}

\begin{defi}[minimal and maximal element]\index{minimal element}\index{maximal element}
Let $E$ be a partially ordered set. A \emph{minimal} (resp. \emph{maximal})
element of $E$ is an element $m\in E$ that is not 
greater (resp. smaller) than any other element in $E$, that is, if
$m\ge x$ (resp. $m\le x$) for some $x\in E$, then $m=x$.
\end{defi}

\begin{theo}[Zorn's Lemma]\index{Zorn's Lemma}
Let $E$ be a nonempty partially ordered set.
Suppose that every nonempty family of elements of $E$ that is totally 
ordered has a lower (resp. upper) bound in $E$. Then $E$ contains at 
least one minimal (resp. maximal) element.
\end{theo}

Zorn's Lemma is equivalent to the axiom of choice; it is a result of set 
theory (see e.g. \cite{Kur1}). 
However it can be used in topology by considering the partial
order given by the inclusion: the set $E$ is a family of
subsets of some space $X$, and $A\le B$ if $A\subset B$, where 
$A,B\in E$.

%**********************************************************************

\backmatter
\bibliographystyle{plain}

%\bibliography{../../../tex/biblio/biblio} 

%%%%%%%%%%%%%%%%%%%%%%%%%%
\printindex

%%%%%%%%%%%%%%%
%\chapter*{Notation}

\renewcommand{\indexname}{Notation}

\begin{theindex}

  \item $\Lbrack n,m\Rbrack$: interval of integers, \pageref{notation:intervalintegers}
  \item $\# E$: cardinality of a set, \pageref{notation:cardE}
  \item $f\vert _Y$: restriction of a map, \pageref{notation:restriction}, \pageref{notation:restrictionbis}
  \item $\vert J\vert$: length of an interval, \pageref{notation:lengthJ}
  \item $\langle a,b\rangle$: interval $[a,b]$ or $[b,a]$, 
		\pageref{notation:convexhull}
  \item $X<Y$, $X\le Y$: inequalities between subsets of $\IR$, 
		\pageref{notation:XlessY}
  \item $\to$ (e.g. $J\to K$): covering of intervals, \pageref{notation:coveringintervals}
  \item $\to$ (e.g. $u\to v$): arrow in a directed graph, \pageref{notation:arrow}
  \item $\Vert \cdot\Vert $: norm of a matrix, \pageref{notation:normmatrix}
  \item $\lhd$, $\rhd$, $\unlhd$, $\unrhd$: Sharkovsky's order, 
		\pageref{notation:sharkovskyorder}
  \item $2^\infty$: a type for Sharkovsky's order, \pageref{notation:2infty}
  \item $\vee$: refinement of covers, \pageref{notation:vee}
  \item $\prec$, $\succ$: $\CC\prec \CD$ if $\CD$ is finer that $\CC$, where $\CC, \CD$ are covers, \pageref{notation:finer}
  \item $\CU^n:=\CU\vee f^{-1}(\CU)\vee\cdots\vee f^{-(n-1)}(\CU)$, \pageref{notation:coverUn}
  \item $A\ge 0$: non negative matrix, \pageref{notation:nonnegativematrix}
  \item $A\le B$ $\Leftrightarrow B-A\ge 0$ (where $A,B$ are matrices),	\pageref{notation:AlessBmatrix}
  \item $A> 0$: positive matrix, \pageref{notation:positivematrix}
  \item $A<B$ $\Leftrightarrow B-A>0$ (where $A,B$ are matrices), \pageref{notation:AlessBmatrix}
  \item $f_P$: connect-the-dots map associated to $f\vert _P$, \pageref{notation:ctdfP}
  \item $k\vert n$: $k$ divides $n$, \pageref{notation:divides}
  \item $\uparrow$: increasing, \pageref{notation:increasing}
  \item $\downarrow$: decreasing, \pageref{notation:decreasing}
  \item $\vert B\vert$: length of a word, \pageref{notation:lengthword}

  \indexspace

  \item $\Sigma:=\{0,1\}^{\IZ^+}$, \pageref{notation:setSigma}
  \item $\sigma$: shift map on $\Sigma$, \pageref{notation:mapshift}
  \item $\omega(x,f)$: $\omega$-limit set of a point, \pageref{notation:omegax}
  \item $\omega(f)$: $\omega$-limit set of a map, \pageref{notation:omegaf}

  \indexspace

  \item $B_n(x,\eps)$: Bowen ball, \pageref{notation:Bowenball}
  \item $G(f\vert P)$: graph associated to $P$-intervals, \pageref{notation:Gfp}
  \item $G(f_P):=G(f_P\vert P)$, \pageref{notation:ctdGfP}
  \item $h_{top}(\CU,f)$: topological entropy of a cover, \pageref{notation:htopU}
  \item $h_{top}(f)$: topological entropy of a map, \pageref{notation:htopf}
  \item $h_A(f)$: topological sequence entropy of a map with respect to a sequence, \pageref{notation:hAf}
  \item $\LY(f,\delta)$, $\LY(f)$: set of Li-Yorke pairs, \pageref{notation:LY}
  \item $M(f\vert P)$: adjacency matrix of $G(f\vert P)$,  \pageref{notation:Mfp}
  \item $M(f_P):=M(f_P\vert P)$, \pageref{notation:cdtMfP}
  \item $\bmod\ n$: integer modulo $n$, \pageref{notation:integermodn}
  \item $N(\CU)$: minimal cardinality of a subcover of $\CU$, \pageref{notation:NU}
  \item $N_n(\CU,f):=N: \CU^n)$, \pageref{notation:NnU}
  \item $\CO_f(x), \CO_f(E)$: orbit of a point/set, \pageref{notation:orbit}
  \item $P_n(f)$: set of points s.t. $f^n(x)=x$, \pageref{notation:Pn}
  \item $R_{\alpha}$: rotation of angle $\alpha$, \pageref{notation:Ralpha}
  \item $r_n(f,\eps)$: minimal cardinality of an $(n,\eps)$-spanning set, \pageref{notation:rnfeps}
  \item $r_n(A,f,\eps)$: minimal cardinality of an $(A,n,\eps)$-spanning set, \pageref{notation:rnAfeps}
  \item $\slope(f)$: slope of a linear map, \pageref{notation:slope}
  \item $s_n(f,\eps)$: maximal cardinality of an $(n,\eps)$-separated set, \pageref{notation:snfeps}
  \item $s_n(A,f,\eps)$: maximal cardinality of an $(A,n,\eps)$-separated set, \pageref{notation:snAfeps}
  \item $U_{\eps}(f)$: set of $\eps$-unstable points, \pageref{notation:Ueps}
  \item $\CV(z)$: family of neighborhoods of $z$, \pageref{notation:Vz}
  \item $W^u(z,f^p)$: unstable manifold of $z$, \pageref{notation:Wuz}

  \indexspace

\item \textbf{Notation of topology}
  \item $\overline{Y}$: closure of a set, \pageref{notation:closure1}, \pageref{notation:closure}
  \item $\Int{Y}$: interior of a set, \pageref{notation:interior1}, \pageref{notation:interior}
  \item $\Bd{Y}$: boundary of a set, \pageref{notation:boundary1}, \pageref{notation:boundary}
  \item $\End{J}$: endpoints of an interval/graph, \pageref{notation:endpointsinterval}, \pageref{notation:endpointsgraph}
  \item $X\setminus Y$: complement of $Y$ in $X$, \pageref{notation:complement}
  \item $X_1\times X_2$: product of sets, \pageref{notation:XtimesY}
  \item $X^n:=X\times\cdots\times X$, \pageref{notation:productXn}

  \item $d(x,y)$: distance between 2 points, \pageref{notation:distance1}, \pageref{notation:dxy}
  \item $d(A,B)$: distance between 2 sets, \pageref{notation:dAB}
  \item $B(x,r),\overline{B}(x,r)$: open/closed ball of center $x$ and radius $r$, \pageref{notation:ball1}, \pageref{notation:balls}
  \item $\diam(Y)$: diameter of a set, \pageref{notation:diam1}, \pageref{notation:diam}

  \indexspace

\item \textbf{Sets of numbers}

  \item $\IC$: set of complex numbers, \pageref{notation:IC}
  \item $\IN$: set of natural integers, \pageref{notation:IN}
  \item $\IQ$: set of rational numbers, \pageref{notation:IQ}
  \item $\IR$: set of real numbers, \pageref{notation:IR}
  \item $\IZ$: set of integers, \pageref{notation:IZ}

\end{theindex}

%*********************************************************************

\end{document}